\RequirePackage{silence}
\documentclass[twoside,numbers=noenddot,headinclude,footinclude,cleardoublepage=empty,abstract=on,BCOR=5mm,paper=letter,fontsize=11pt]{scrreprt}
\PassOptionsToPackage{utf8}{inputenc}
  \usepackage{inputenc}

\PassOptionsToPackage{T1}{fontenc}
  \usepackage{fontenc}

\PassOptionsToPackage{
  drafting=false,    % print version information on the bottom of the pages
  tocaligned=false, % the left column of the toc will be aligned (no indentation)
  dottedtoc=true,  % page numbers in ToC flushed right
  eulerchapternumbers=true, % use AMS Euler for chapter font (otherwise Palatino)
  linedheaders=false,       % chaper headers will have line above and beneath
  floatperchapter=true,     % numbering per chapter for all floats (i.e., Figure 1.1)
  eulermath=false,  % use awesome Euler fonts for mathematical formulae (only with pdfLaTeX)
  beramono=true,    % toggle a nice monospaced font (w/ bold)
  palatino=true,    % deactivate standard font for loading another one, see the last section at the end of this file for suggestions
  parts=false,
  style=classicthesis
}{classicthesis}

\newcommand{\myfullTitle}{Moduli spaces and algebraic cycles in real algebraic geometry\xspace}

\newcommand{\myTitle}{Moduli spaces and\xspace}
\newcommand{\myTitletwo}{algebraic cycles in\xspace}
\newcommand{\myTitlethree}{real algebraic geometry\xspace}
\newcommand{\myName}{Olivier de Gaay Fortman\xspace}
\newcommand{\myDegree}{Doctor of Philosophy\xspace}
\newcommand{\myDepartment}{École normale supérieure de Paris\xspace}
\newcommand{\myUni}{École doctorale de Sciences Mathématiques de Paris Centre\xspace}

\usepackage{datetime}

\newcommand*{\img}[1]{%
    \raisebox{-.02\baselineskip}{%
        \includegraphics[
        height=\baselineskip,
        width=\baselineskip,
        keepaspectratio,
        ]{#1}%
    }%
}

\newcommand{\wh}[1]{{\widehat{#1}}}

\newcommand{\PSL}{\textnormal{PSL}}

\newcommand{\set}[1]{\left\{ #1 \right\}}
\newcommand{\va}[1]{\left| #1 \right|}

\newcommand{\equivcohom}{\rm H^{2k}_G(X(\CC), \ZZ(k))}

\newcommand{\Alb}{\textnormal{Alb}}

\newcommand{\Isom}{\textnormal{Isom}}

\newcommand{\etale}{\textnormal{\'et}}

\newcommand{\white}{\;\;\;}

\newcommand{\Grass}{\textnormal{Grass}}

\newcommand{\Jac}{\text{Jac}}

\newcommand{\NL}{\textnormal{NL}}

\newcommand{\sm}{\setminus}

\newcommand{\CC}{\mathbb{C}}

\newcommand{\BB}{\mathbb{B}}

\newcommand{\OO}{\mathcal{O}}

\newcommand{\FF}{\mathbb{F}}

\newcommand{\VV}{\mathbb{V}}
\newcommand{\HH}{\mathbb{H}}

\newcommand{\QQ}{\mathbb{Q}}
\newcommand{\PP}{\mathbb{P}}
\newcommand{\GG}{\mathbb{G}}

\newcommand{\PGL}{\textnormal{PGL}}

\newcommand{\Sym}{\textnormal{Sym}}

\newcommand{\et}{\textnormal{\'et}}

\newcommand{\ch}{\textnormal{ch}}
\newcommand{\Hdg}{\textnormal{Hdg}}
\newcommand{\GL}{\textnormal{GL}}

\newcommand{\SL}{\textnormal{SL}}

\newcommand{\RR}{\mathbb{R}}

\newcommand{\ZZ}{\mathbb{Z}}

\newcommand{\CH}{\textnormal{CH}}
\newcommand{\CCH}{\bb C H}
\newcommand{\RRH}{\bb R H}

\newcommand{\Pic}{\textnormal{Pic}}

\newcommand{\Coker}{\textnormal{Coker}}

\newcommand{\Ker}{\textnormal{Ker}}
\newcommand{\NS}{\textnormal{NS}}

\newcommand{\tors}{\textnormal{tors}}

\newcommand{\id}{\textnormal{id}}

\newcommand{\Hom}{\textnormal{Hom}}

\newcommand{\Lie}{\textnormal{Lie}}

\newcommand{\Gal}{\textnormal{Gal}}

\newcommand{\Ima}{\textnormal{Im}}
\newcommand{\End}{\textnormal{End}}

\newcommand{\Aut}{\textnormal{Aut}}

\newcommand{\Spec}{\textnormal{Spec}}

\newcommand{\Sp}{\textnormal{Sp}}

\newcommand{\ca}[1]{{\mathcal{#1}}}
\newcommand{\bb}[1]{{\mathbb{#1}}}
\newcommand{\msf}[1]{{\mathsf{#1}}}
\newcommand{\mr}[1]{{\mathscr{#1}}}
\newcommand{\mf}[1]{{\mathfrak{#1}}}
\newcommand{\tn}[1]{{\textnormal{#1}}}
\let\rm\relax 
\newcommand{\rm}[1]{{\mathrm{#1}}}

\DeclareSymbolFont{bbm}{U}{bbm}{m}{n}
\DeclareSymbolFontAlphabet{\mathbbm}{bbm}
\PassOptionsToPackage{english}{babel}
\usepackage{babel} % language support
\usepackage{enumitem} % for better itemize and enumerate
\usepackage{amsmath,amsthm,amssymb} % because math
\usepackage{thmtools} % for correct autorefs
\usepackage[onehalfspacing]{setspace} % as required by SGS
\usepackage{graphicx} % to include graphics
\usepackage{xspace} % to get the spacing after macros right
\usepackage{calc} % to allow adding lengths

\usepackage{booktabs} % for much better looking tables
\usepackage{array} % for better arrays (eg matrices) in maths
\usepackage{ifthen}
\usepackage{verbatim} % adds environment for commenting out blocks of text & for better verbatim
\usepackage{subfig} % make it possible to include more than one captioned figure/table in a single float
\usepackage{mathtools}
\usepackage{mathrsfs}
\usepackage[colorlinks, allcolors=CTlink]{hyperref}
\usepackage{appendix}
\usepackage{stmaryrd}
\usepackage{mathdots}
\usepackage[all,cmtip]{xy}
\usepackage[dvipsnames]{xcolor}
\usepackage{pdfpages}

\usepackage{classicthesis}

\usepackage[margin=1.5in]{geometry}

\hypersetup{%
  colorlinks=true, linktocpage=true, pdfstartpage=3, pdfstartview=FitV,%
  breaklinks=true, pageanchor=true,%
  pdfpagemode=UseNone, %
  plainpages=false, bookmarksnumbered, bookmarksopen=true, bookmarksopenlevel=1,%
  hypertexnames=true, pdfhighlight=/O,%nesting=true,%frenchlinks,%
  urlcolor=Black, linkcolor=CTlink, citecolor=CTlink, %CTcitation, %pagecolor=RoyalBlue,%
  pdftitle={\myTitle},%
  pdfauthor={\textcopyright\ \myName, \myDepartment, \myUni},%
  pdfsubject={},%
  pdfkeywords={},%
  pdfcreator={pdfLaTeX},%
  pdfproducer={LaTeX with hyperref and classicthesis}%
}

\makeatletter
\@ifpackageloaded{babel}%
  {%
    \addto\extrasenglish{%
    }%
    }{\relax}
\makeatother

\makeatletter
\ifthenelse{\boolean{ct@linedheaders}}%
{% lines above and below, number right
    \titleformat{\chapter}[display]%
    {\relax}{\raggedleft{\color{black}\chapterNumber\thechapter} \\ }{0pt}%
    {\titlerule\vspace*{.9\baselineskip}\raggedright\Large\color{CTtitle}\spacedallcaps}[\normalsize\vspace*{.8\baselineskip}\titlerule]%
}{% something like Bringhurst
    \titleformat{\chapter}[display]%
    {\relax}{\mbox{}\oldmarginpar{\vspace*{-3\baselineskip}\color{black}\chapterNumber\thechapter}}{0pt}%
    {\raggedright\huge\color{black}\spacedallcaps}[\normalsize\vspace*{.8\baselineskip}\titlerule]%
}
\makeatother

\newtheorem{theorem}{Theorem}[chapter]
\newtheorem{lemma}[theorem]{Lemma}

\newtheorem{corollary}[theorem]{Corollary}
\newtheorem{proposition}[theorem]{Proposition}
\newtheorem{condition}[theorem]{Condition}
\newtheorem{conditions}[theorem]{Conditions}
\newtheorem{theoreme}[theorem]{Th\'eor\`eme}
\newtheorem*{theoreme-non}{Th\'eor\`eme}
\newtheorem*{theorem-non}{Theorem}
\newtheorem*{condition-non}{Condition}

\theoremstyle{definition}
\newtheorem{definition}[theorem]{Definition}

\newtheorem{example}[theorem]{Example}
\newtheorem{examples}[theorem]{Examples}
\newtheorem{notation}[theorem]{Notation}
\newtheorem{conventions}[theorem]{Conventions}
\newtheorem{convention}[theorem]{Convention}

\theoremstyle{remark}
\newtheorem{remark}[theorem]{Remark}
\newtheorem{question}[theorem]{Question}

\newtheorem{remarks}[theorem]{Remarks}
\newtheorem{claim}[theorem]{Claim}

\newenvironment{itemize*}
  {\begin{itemize}[topsep=-\parskip+\jot,itemsep=-\parskip-\jot]}
  {\end{itemize}}
  
\newenvironment{enumerate*}
  {\begin{enumerate}[label=(\alph*),topsep=-\parskip+\jot,itemsep=-\parskip-\jot]}
  {\end{enumerate}}
  
\newenvironment{enumerate**}
  {\begin{enumerate}[label=(\roman*),topsep=-\parskip+\jot,itemsep=-\parskip-\jot]}
  {\end{enumerate}}
  
\newenvironment{enumerate***}
  {\begin{enumerate}[label=(\alph*'),topsep=-\parskip+\jot,itemsep=-\parskip-\jot]}
  {\end{enumerate}}

\begin{document}
\setboolean{@twoside}{false} 
% If you are wondering what \frenchspacing does, check out the Internet then decide if you want to use it.
\frenchspacing
\raggedbottom

\selectlanguage{english}

% Frontmatter
% Note: SGS requires that the abstract is on the second page. Do not move it further down.
% Roman page numbering is also required by SGS.
\pagenumbering{roman}
\pagestyle{plain}
%\include{frontbackmatter/titlepage}
%\cleardoublepage
% Do NOT edit the title info here. There is a central place to set Title, Author etc. in thesisconfig.tex
\begin{titlepage}
    \pdfbookmark[1]{\myfullTitle}{titlepage}

\includepdf[pages={1}]{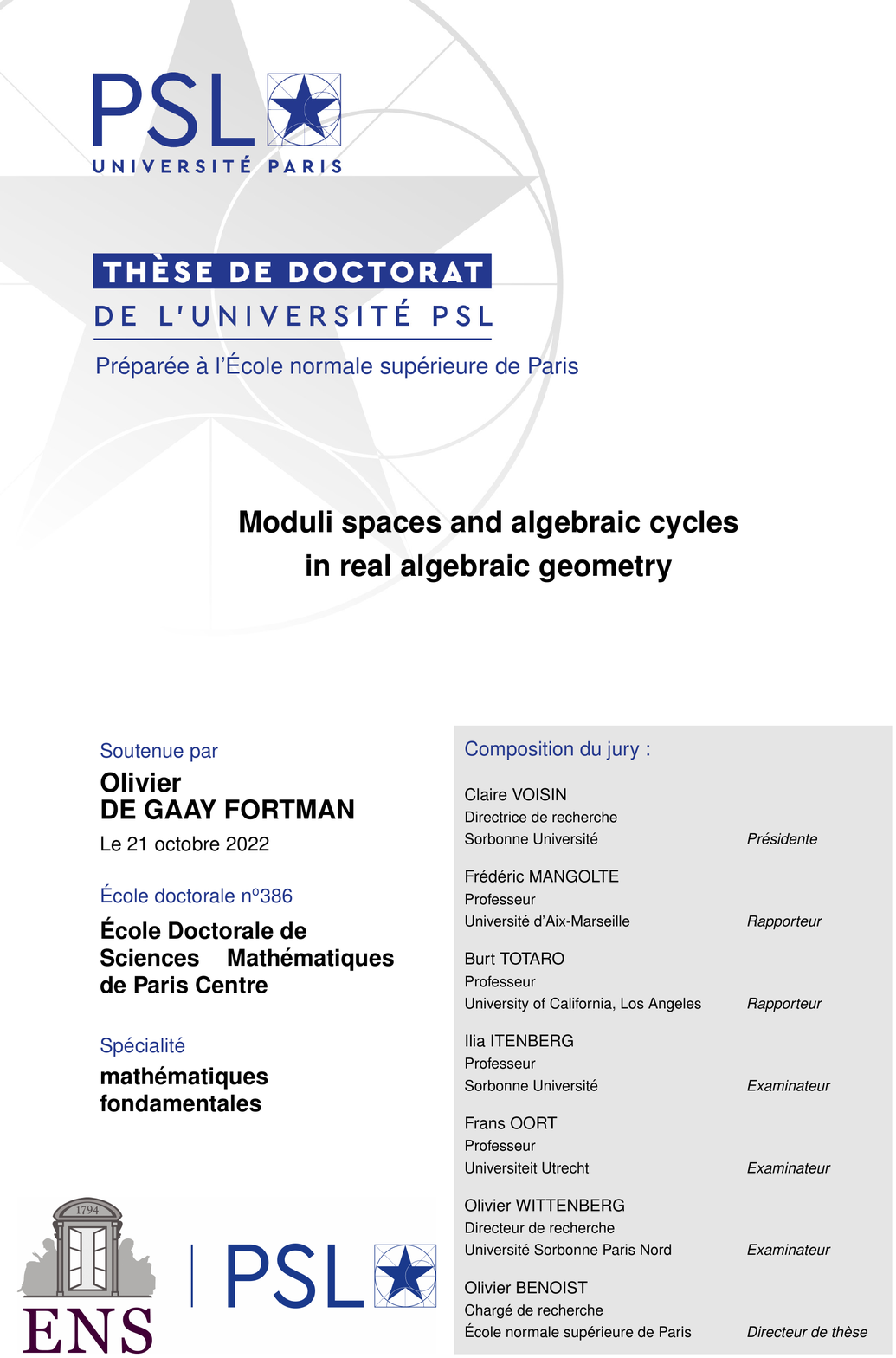}

\end{titlepage}

\cleardoublepage
%M\include{frontbackmatter/titlepagezero}
%\cleardoublepage
%\include{frontbackmatter/titlepagezerovero}
%\cleardoublepage
% Do NOT edit the title info here. There is a central place to set Title, Author etc. in thesisconfig.tex
\begin{titlepage}
%    \pdfbookmark[1]{\myTitle}{titlepage}

% Uncomment the addmargin environment when using Option 2 in the main file
% \begin{addmargin}[-0.6cm]{-3cm}
    \begin{center}
        \large

        \hfill

        \vfill

        \begingroup
            %\color{}
            \spacedallcaps{\large{\textbf{\myTitle}}} 
            
            \bigskip
            
            %\color{}
            \spacedallcaps{\large{\textbf{\myTitletwo}}}
            
            \bigskip
            
            %\color{}
            \spacedallcaps{\large{\textbf{\myTitlethree}}}
            
        \endgroup

		\vfill
		by\bigskip
		
        \spacedlowsmallcaps{\myName}

		\vfill
        
        A thesis submitted in conformity with\\
        the requirements for the degree of\bigskip
        
        \myDegree \bigskip
        
        \myDepartment\\
        %\myUni\bigskip
        
        %\textcopyright\ 
        %\myTime \myName
        October 21, 2022

    \end{center}
% \end{addmargin}
\end{titlepage}

\cleardoublepage
% Do NOT edit the title info here. There is a central place to set Title, Author etc. in thesisconfig.tex
\begin{titlepage}
%    \pdfbookmark[1]{\myTitle}{promotor}

% Uncomment the addmargin environment when using Option 2 in the main file
% \begin{addmargin}[-0.6cm]{-3cm}
%    \begin{center}
        \large

        \hfill

        \vfill

        \begingroup
            %\color{}
%            \spacedallcaps{\large{\textbf{\myTitle}}} 

Doctoral advisor:   \hspace{1cm}   \spacedlowsmallcaps{Olivier Benoist}
            
            \bigskip
            
            %\color{}
%            \spacedallcaps{\large{\textbf{\myTitletwo}}}
            
            \bigskip
            
            %\color{}
%            \spacedallcaps{\large{\textbf{\myTitlethree}}}
            
        \endgroup

		\vfill
%		\spacedlowsmallcaps{by}\bigskip
		
%        \spacedlowsmallcaps{\myName}

		\vfill
        
%        A thesis submitted in conformity with\\
  %      the requirements for the degree of
  \bigskip
  
    \bigskip
    \bigskip
    \bigskip
    \bigskip
    \bigskip
    \bigskip

     %   \myDegree
      \bigskip
        
       % \myDepartment
       
       % \myUni
       \bigskip
        
        %\textcopyright\ 
        %\myTime \myName
      %  \today

%    \end{center}
% \end{addmargin}
\end{titlepage}

\setboolean{@twoside}{true}
\cleardoublepage
\pdfbookmark[1]{Abstract}{Abstract}

\begingroup
\let\clearpage\relax
\let\cleardoublepage\relax
\let\cleardoublepage\relax

\chapter*{Abstract}

\doublespacing

% Do NOT edit the title info here. There is a central place to set Title, Author etc. in the thesisconfig.tex
%\begin{center}
%{\Large{\myfullTitle}}\\
%by \textsc{\myName}
%\myDegree\\
%\myDepartment\\
%\myUni\\
%\myTime
%\end{center}

%Consider a projective algebraic variety $X$, smooth over the field $\RR$ of real numbers. Real algebraic cycles on $X$ are integral linear combinations of algebraic subvarieties of $X$. 
\noindent
This thesis intends to make a contribution to the theories of algebraic cycles and moduli spaces over the real numbers. In the study of the subvarieties of a projective algebraic variety, smooth over the field of real numbers, the cycle class map between the Chow ring and the equivariant cohomology ring plays an important role. The image of the cycle class map remains difficult to describe in general; we study this group in detail in the case of real abelian varieties. To do so, we construct integral Fourier transforms on Chow rings of abelian varieties over any field. They allow us to prove the integral Hodge conjecture for one-cycles on complex Jacobian varieties, and the real integral Hodge conjecture modulo torsion for real abelian threefolds. 

For the theory of real algebraic cycles, and for several other purposes in real algebraic geometry, it is useful to have moduli spaces of real varieties to our disposal. Insight in the topology of a real moduli space provides insight in the geometry of a real variety that defines a point in it, and the other way around. In the moduli space of real abelian varieties, as well as in the Torelli locus contained in it, we prove density of the set of moduli points attached to abelian varieties containing an abelian subvariety of fixed dimension. Moreover, we provide the moduli space of stable real binary quintics with a hyperbolic orbifold structure, compatible with the period map on the locus of smooth quintics. This identifies the moduli space of stable real binary quintics with a non-arithmetic ball quotient. 

\endgroup
%\cleardoublepage
%\include{frontbackmatter/blank}
%\cleardoublepage
\pdfbookmark[1]{Acknowledgements}{acknowledgements}

\begingroup
\let\clearpage\relax
\let\cleardoublepage\relax

\chapter*{Acknowledgements}

%First and foremost, I would like to wholeheartedly thank my PhD adviser Olivier Benoist. It is difficult to express how much I have learned from him. The richness of his intuition has inspired me in several ways, and his continuous encouragement and enthusiasm have greatly motivated me during the past three years. Working under his guidance has been a great experience. 
%I am also very grateful for his continuous encouragement, enthusiasm and support. 
% Sa g ́en ́erosit ́e constante, la richesse de ses intuitions et sa grande disponibilit ́e ont rendu ces ann ́ees tr`es enrichissante

%\onehalfspacing

First and foremost, I would like to wholeheartedly thank my supervisor Olivier Benoist for guiding me so skillfully over the past three years. It has been truly great to work under his guidance. It is difficult to express how much I have learned from Benoist's incredible mathematical knowledge and intuition. Moreover, his continuous encouragement and enthusiasm turned my time at the ENS into an amazing period, and I am very thankful for this. 

I would like to heartily thank Fr\'ed\'eric Mangolte and Burt Totaro, the reviewers of my thesis, for the close attention that they have paid to my work. It is also a great honour to thank Olivier Benoist, Ilia Itenberg, Fr\'ed\'eric Mangolte, Frans Oort, Burt Totaro, Claire Voisin and Olivier Wittenberg for being part of my jury. 

%Je remercie  ́egalement Brendan Hassett et Laurent Manivel, qui ont consenti `a ˆetre rapporteurs de cette th`ese, pour l’attention qu’ils ont port ́ee `a mon travail.
I am grateful for having had the privilege of discussing regularly with Nicolas Tholozan, who inspired me in many ways. I appreciate the help that Tholozan offered me, as my mentor at the ENS. 
%It has also been great to have Nicolas as a mentor. 
% of hyperbolic geometry and Teichm\"uller space. 
%Our discussions have had a great influence on my work. 
I also thank all the other people at the ENS for being ready to discuss mathematics; in particular, I thank Nicolas Bergeron, Romain Branchereau, Samuel Bronstein, François Charles, Ga\"etan Chenevier, Antoine Ducros, Arnaud Eteve, Yusuke Kawamoto, Vadim Lebovici, Ariane Mézard, Léonard Pille-Schneider, Salim Tayou and Arnaud Vanhaecke for doing so. 

I am very thankful for having shared an office with Romain Branchereau, who soon became a great friend of mine. Our math discussions on our blackboard, and the fun that we always had, were invaluable to me. I am very grateful for the regular meetings with my friend Elyes Boughattas. Exchanging ideas with Elyes taught me very much, and it was incredible to organize a reading seminar together. Thanks to Ariane M\'ezard for the amazing time I had teaching the ENS course Complex Analysis with her. I learned a lot from working together with Ariane, and her heartfelt inquiries on the well-going of my PhD have meant much to me. 

I heartily thank Daniel Huybrechts for hosting me in Bonn during the months October -- December 2021. Discussing with Huybrechts and attending his lectures were a great stimulus for the development of my work. I thank the Hausdorff Center for Mathematics for their hospitality and pleasant working conditions. Moreover, it was very enriching to collaborate with and learn from Thorsten Beckmann. We had great fun at the University of Bonn; starting as colleagues, we soon became friends. 

I am indebted to Frans Oort for meeting me in Utrecht several times. It has been wonderful to experience his enthusiasm for the beauty of mathematics. 
 % and I admire his knowledge and intuition. 
%He listened, answered, and expressed his opinion in a way that has taught me a lot. 
Frans' work as a mathematician, and kindness as a person, are a great inspiration to me. I am grateful for his reading of an earlier draft of this thesis, and his valuable comments. Many thanks also to Fabrizio Catanese, for his invitation to spend several days in Bayreuth, to discuss mathematics with him and present my work at their seminar. This was a delightful visit, from which I learned a lot. Heartfelt thanks to Giuseppe Ancona, for proposing me to spend a research visit in Strasbourg and talk at their arithmetic and algebraic geometry seminar. I benefited much from the joyful discussions with both Giuseppe and Michele Ancona during my stay in Strasbourg. 

My sincere thanks to Christopher Lazda for guiding me during my master studies; this was a wonderful period which would eventually lead to this PhD.

I thank my parents Carl and Betteke for their eternal love and support. I am very proud of them, inspired by their work, and happy to acknowledge their vast influence on my development. They are and will always be an example for me. I thank my brother Nelson and my sister Duveke, their partners Emily and Mats, my friends Daan and Milan, my grandparents Bas and Ina, and all my other friends and family for their continuous affection and encouragement. 

Finally, I thank my fiancée Nola for her interest in the progress of my projects, for her understanding and support, and for always being there for me. 
\\
\\
%\vspace*{\fill}
I am grateful for the scholarship that I received from the FSMP, and for all the effort that the FSMP in general -- and Ariela Briani in particular -- has put into the development of my PhD. I am therefore happy to acknowledge that this project has received funding from the European Union's Horizon 2020 research and innovation programme under the Marie Sk\l{}odowska-Curie grant agreement N\textsuperscript{\underline{o}}754362. $\hfill \img{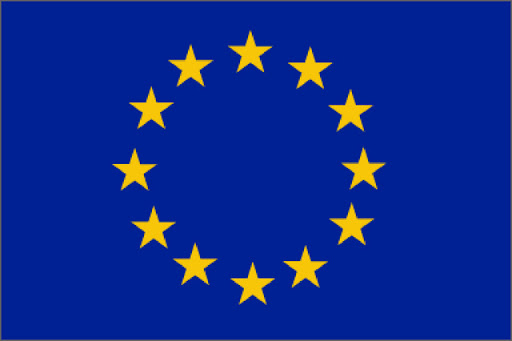}$ \\
I thank the Prins Bernhard Cultuurfonds for providing me with additional funding to cover the expenses of my studies, during these fantastic three years in Paris. 
%This project has received funding from the European Union's Horizon 2020 research and innovation programme under the Marie Sk\l{}odowska-Curie grant agreement 
%It has been one of the most inspiring 
%Meeting and discussing mathematics with Frans has been   
%the several meetings in Utrecht we had the past three years. From the moment we met, Frans has shown a lot of interest in my work, and was always keen to give his opinion, or answer a question. 

\endgroup
\clearpage
\begin{doublespace}
{\fontsize{12}{12}
\setlength{\cftbeforetoctitleskip}{100em}
\begingroup
\pagestyle{scrheadings}
\pdfbookmark[1]{\contentsname}{tableofcontents}
\setcounter{tocdepth}{1} % <-- 2 includes up to subsections in the ToC
\setcounter{secnumdepth}{3} % <-- 3 numbers up to subsubsections
\manualmark
\markboth{\spacedlowsmallcaps{\contentsname}}{\spacedlowsmallcaps{\contentsname}}
\tableofcontents
\automark[section]{chapter}
\renewcommand{\chaptermark}[1]{\markboth{\spacedlowsmallcaps{#1}}{\spacedlowsmallcaps{#1}}}
\renewcommand{\sectionmark}[1]{\markright{\textsc{\thesection}\enspace\spacedlowsmallcaps{#1}}}
\endgroup

}
\end{doublespace}
% The following command will make sure that publications are on the left and dedications are on the right.
%\cleardoubleevenpage
%\include{frontbackmatter/publications}
%\cleardoublepage
%\include{frontbackmatter/dedication}

% Add any chapters here
\cleardoublepage
\pagestyle{scrheadings}
\pagenumbering{arabic}
\cleardoublepage
\chapter{Introduction}
\label{ch:intro}

This thesis concerns a study of cycles and moduli spaces in real algebraic geometry. We intend to make a contribution towards understanding these concepts, individually as well as in light of one another, by looking at phenomena such as: 
%To do so, we consider some beautiful examples:
%comparison with one another.
 %individually and in comparison with the other. %both concepts individually as well as the relations between them. 
%We will investigate:
 %by means of investigating the following examples:
% concepts, by investigating several examples, such as: 

\begin{itemize}
\item The hyperbolic structure of the moduli space of real binary quintics.
\item The distribution of non-simple abelian varieties in a family of real abelian varieties. 
\item Algebraic curves on complex and real abelian varieties. 
\end{itemize}

%We will not so much develop a general theory, but rather look at several fine examples.

%With respect to these investigations, t
\noindent
The results as well as the methods of these investigations are closely intertwined, as we shall see. The general, overarching theory is \emph{equivariant Hodge theory}: the use of Hodge theory to study cohomology, moduli and cycles of real algebraic varieties. 
\\
\\
This introduction is structured as follows. %In the rest of this introduction, we would like to present the results that are obtained in this thesis. 
In Section \ref{stagesetting}, we set the stage by explaining what real algebraic geometry is about. We continue to introduce the main players in Section \ref{mainplayers}: these are moduli spaces on the one hand, and algebraic cycles on the other. Finally, in Section \ref{mainresults}, we discuss our most important results. 
%have led to important progress. 
%modern developments have led to important progress, making it an active field 
%, but is still a very active field of research. 
%is an old and fascinating topic, 
%lying at the roots of several theories in mathematics, 
%old topic and fascinating topic, which has recently 
 %of polynomials in $n +1 $ variables with real coefficients. %These sets are realized as the set of complex, respectively real points of a projective variety $X$ over $\RR$, embedded into the projective space $\PP^n_\RR$. 

\section{Setting the stage} \label{stagesetting}

Real algebraic geometry concerns the study of algebraic varieties over $\RR$, the field of real numbers. As such, the basic objects of the theory are sets of the form 
\begin{align} \label{algebraicsubset}
X(\CC) = \left\{  \alpha \in \PP^n(\CC) \mid F_i(\alpha) = 0 \; \forall i = 1, \dotsc, k  \right\}, \quad  X(\RR) = X(\CC) \cap \PP^n(\RR), \quad
%\left\{  \alpha \in \PP^n(\RR) \mid F_i(\alpha) = 0 \forall i \right\} \subset \PP^n(\CC), 
\end{align}
where the $F_i$ are homogeneous elements of the polynomial ring $\RR[X_0, \dotsc, X_n]$. Studying such polynomial equations is an old and fascinating topic, in which recent developments have led to important progress \cite{BCR87, mangolte-realvarieties}. 
%\cite{silholsurfaces, BCR87, itenbergenriques, mangolte-realvarieties}. 
%A modern point of view is that %It turns out that in some sense, 

\noindent
Nowadays, one often thinks of real algebraic geometry as \emph{$G$-equivariant complex algebraic geometry}. To explain this, we introduce the following:
%To ease notation, we introduce:
%as we shall now explain. 
\begin{definition} \label{thegaloisgroup}
We define, once and for all, $G$ as the Galois group of $\CC$ over $\RR$. 
\end{definition}
\noindent
Consider the algebraic subset $X(\CC) \subset \PP^n(\CC)$ defined in (\ref{algebraicsubset}). %algebraic subset defined by homogeneous polynomials $F_i$. 
Since the coefficients of the polynomials $F_i$ are real, complex conjugation $x \mapsto \bar x$ on the complex projective space $\PP^n(\CC)$ descends to an anti-holomorphic involution $\sigma \colon X(\CC) \to X(\CC)$. In this way, one obtains a functor
\[
X \mapsto \left(X_\CC, \; \sigma \colon X(\CC) \to X(\CC) \right)
\]
from the category of projective varieties over $\RR$ to the category of projective varieties over $\CC$ equipped with an anti-holomorphic $G$-action. This functor is an \emph{equivalence}. 

In particular, to study a real algebraic variety, one may view it either as an $\RR$-scheme $X$ and employ algebraic methods, or as a $G$-equivariant complex analytic space $X(\CC)$ and use topological and analytical techniques. A guiding principle in this thesis is that especially the combination of both approaches is very powerful.
\\
\\
Led by this principle, we intend to make a contribution to real algebraic geometry by exploring the following two themes:

\begin{enumerate}
\item \label{intro:one}Moduli spaces of certain classes of real algebraic varieties. 
\item \label{intro:two}Real algebraic subvarieties of a fixed ambient real algebraic variety. 
\end{enumerate}

\noindent
These two concepts will be the main players of this thesis. Let us introduce them properly, before we put them into action and explain our results. 

%since they will help to understand our results. 

%before we move on to the discussion of our results. 

%Let us try to get to know 
%Before we move on to the discussion of our results, let us consider these main players in a bit more detail. 

\section{Introduction to the main players} \label{mainplayers}

\subsection{Complex and real moduli spaces} \label{intro:sub:complexandrealmodulispaces}

To understand how the geometry of varieties in some interesting category behaves in family, 
%some class of varieties behaves in family, 
%In order to understand the behavior of the geometry of some class of algebraic varieties in family, 
it can be useful to study the topology of their moduli space. %topology of the moduli space. 
When the base field is $\CC$, constructing a moduli space %defining a natural topology on a moduli space 
is a well-understood problem. It can be attacked with analytic methods such as Teichm\"uller theory, or with algebraic methods such as stacks and GIT. %Well-known examples are the moduli spaces of abelian varieties and curves. In both cases, 
For example, if the moduli stack is algebraic and has finite inertia, then a coarse moduli space exists by \cite{K-M}. %taking complex points gives the desired complex analytic space. 
%corresponding moduli stack is separated and Deligne-Mumford, then the coarse moduli space exists \cite{K-M} and can be analytified. 
%can be uniformized in two different ways: as a GIT quotient of a suitable subscheme of a Hilbert scheme of some projective space by the group of projective $\ZZ$-linear transformations, and as a complex analytic quotient of a simply connected manifold by a discrete group of holomorphic automorphisms. %The period map identifies the analytification of the former with the latter, and the analytic and algebraic moduli theories are compatible. %moduli spaces are similarly useful. What is a real moduli space? %Similarly, moduli spaces n real algebraic geometry it is similarly useful to have moduli spaces to our disposal. Let us explain what we mean by this. 

In real algebraic geometry, it is less straightforward what a real moduli space should be, and how it can be constructed.

%In real algebraic geometry, moduli spaces are similarly useful to study the geometry of real algebraic varieties in family.  
%Similarly, one often encounters moduli problems in real algebraic geometry. %Analogues of such coarse moduli spaces in real algebraic geometry are similarly powerful. 
\begin{question}
What is a real moduli space?
\end{question}
\noindent
Intuitively, the real moduli space of a real variety should be a topological space that classifies isomorphism classes of all real varieties that have something in common with the given one (e.g. numerical invariants of some kind). One can also classify equivalence classes of other algebraic objects defined over $\RR$, such as sheaves on a real variety, morphisms, or combinations of these. More generally, we can consider sets $S(T)$ of families of selected objects over $\RR$-schemes $T$ modulo a certain equivalence relation $\sim$ and build a functor $F \colon (\rm{Sch}/\RR) \to (\rm{Set})$ by the rule $F(T) = S(T)/\sim$. %, then one says that $X$ is a fine moduli space for the moduli problem $F$, see [ref]. In this case,
One could define the real moduli space for the moduli problem $F$ as the topological space obtained by giving $F(\RR)$ the finest topology for which $f_\RR \colon T(\RR) \to F(\RR)$ is continuous %for %such that %equipped with a bijection $ F(\RR) \xrightarrow{\sim} \mr M_\RR$ such that 
for every $\RR$-scheme $T$ and every $f \in F(T)$. 

Generally, the set 
%Generally, however, 
$F(\RR)$ has more structure than just the structure of a topological space. For example, if the moduli problem $F$ is representable by a scheme $X$, then its real moduli space $X(\RR)$ is in fact a real-analytic space. Even in case $F$ is not representable, it is often possible to provide $F(\RR)$ with some additional structure.\begin{example}[Algebraic quotients] \label{algebraicquotients}
Let $F$ be a moduli problem over $\RR$ such that $F(T)$ is a certain set of isomorphism classes of families of polarized real varieties %that sends an $\RR$-scheme $T$ to the set of isomorphism classes of families of polarized real algebraic varieties %Consider a set of isomorphism classes $\mr M_\RR$ of polarized real algebraic varieties 
$(X, \ell)$ over $T$, where $\ell$ is a $G$-invariant ample class in the N\'eron-Severi group of $X_\CC$. Suppose there is a subscheme $H$ of $\text{Hilb}_{\PP^N}$ such that $F(\CC) \cong \PGL_{N+1}(\CC) \setminus H(\CC)$. If $\ell$ is induced by an ample line bundle on $X$ for every $[(X, \ell)] \in F(\RR)$, then there is a homeomorphism 
%we can define a topology on $F(\RR)$ via 
$F(\RR) \cong \PGL_{N+1}(\RR) \setminus H(\RR)$. For example, this works for hypersurfaces in $\PP^n_\RR$. Less trivially, this works for polarized real abelian varieties. 
%For example, this works for cubic surfaces. However, the above condition on $\ell$ is not always satisfied -- think of quartic hypersurfaces in the Brauer-Severi variety of dimension three over $\RR$. 
%think of sextic hypersurfaces in the Brauer-Severi variety $\{X^2 + Y^2 + Z^2 = 0\} \subset \PP^2_\RR$. %More generally, $\ell
%In general, the obstruction lies in the exactness of the sequence
%\[
%\Pic(X) \to \Pic(X_\CC)^G \to \textnormal{Br}(\RR) = \ZZ/2.
%\]
% and consider the stack $\ca M = [\PGL_{n+1} \setminus H]$. 
\end{example}

%if the real points of the complex moduli spaces do not give the real moduli space then there exist non-trivial automorphisms. In other words, 
%the existence of automorphisms might prevent the coarse moduli space to be fine, %athe problem is that the coarse moduli space is not fine. %the complex moduli space may not be fine, 
%Add enough structure to the objects in the classification, and the automorphism groups become trivial. 
%Changing the moduli problem by requiring the objects to have more structure might force the automorphism groups to become trivial. 
%This means to add enough structure to trivialize the automorphism groups in question. 

%Let $H$ be a subscheme of $\text{Hilb}_{\PP^n}$ for some $n$, and consider the stack $\ca M = [\PGL_{n+1} \setminus H]$. If $\ca M(\RR) = [\PGL_{n+1}(\RR) \setminus H(\RR)]$, then the answer to the above question is \emph{yes}. This happens for example in the case of: 
%However, this condition is certainly not satisfied. Consider for example the case of:

%For example, this works for hypersurfaces (or more generally completer intersections) in $\PP^n$, for abelian varieties and for curves. However, this condition is not always satisfied (think of an example). 
%\begin{example}
%\end{example}

\begin{example}[Analytic quotients] \label{introexample:grossharris}
Let $\ca A_g$ be the stack of principally polarized abelian varieties of dimension $g$, and let $\ca M_g$ be the stack of smooth curves of genus $g$. It is well-known that there exist complex analytic uniformizations 
\[
\ca A_g(\CC) \cong \Sp_{2g}(\ZZ) \setminus \bb H_g, \quad \textnormal{ and } \quad \ca M_g(\CC) \cong \Gamma_g \setminus \ca T_g.\] Here, $\bb H_g$ and $\ca T_g$ are complex manifolds, parametrizing complex abelian varieties and curves endowed with some additional structure, and the %covers % a moduli space for abelian varieties (resp. curves) equipped with some additional structure, 
covers $\bb H_g \to \ca A_g(\CC)$ and $\ca T_g \to \ca M_g(\CC)$ are obtained forgetting these structures. %by rigidifying the objects in $ \ca A_g(\CC)$ (resp. $\ca M_g(\CC)$) in an appropriate way. 
To obtain real moduli spaces for these moduli problems, one defines sets of anti-holomorphic involutions \[\{\sigma \colon \bb H_g \to \bb H_g\} \quad \textnormal{ and } \quad \{\tau \colon \ca T_g \to \ca T_g\}\] in such a way that every principally polarized real abelian variety (resp. real curve) lies in $\bb H_g^\sigma$ (resp. $\ca T_g^\tau$) for some $\sigma$ (resp. $\tau$). For suitably defined subgroups $\Sp_{2g}(\ZZ)(\sigma) \subset \Sp_{2g}(\ZZ)$ and $ \Gamma_g(\tau)  \subset \Gamma_g$, %compatible with the real structures, %For a complex manifold $X$ with an anti-holomorphic involution $\sigma \colon X \to X$, and any subgroup $G \subset \Aut(X)$ of holomorphic automorphisms, define $G(\sigma) \subset G$ to be the group of automorphisms in $G$ that commute with $\sigma$. %Defining $ \Sp_{2g}(\ZZ)(\sigma) $ and $\Gamma_g(\tau)$ to be the normalizer subgroups of the involutions $\sigma$ and $\tau$, 
this gives bijections \cite{grossharris, silholsurfaces, seppalasilhol2}
\begin{align*}%\label{analytictopology:def}
\ca A_g(\RR) \cong \bigsqcup_\sigma \Sp_{2g}(\ZZ)(\sigma) \setminus \bb H_g^\sigma \quad \textnormal{ and } \quad \ca M_g(\RR) \cong \bigsqcup_\tau \Gamma_g(\tau) \setminus \ca T_g^\tau.
%the moduli stacks of abelian varieties and curves, the topology on $\mr X(\RR)$ is compatible with those arising from (\ref{analytictopology:def}) in Example \ref{introexample:grossharris}. 
%and a continuous map $\ca M_g(\RR) \to \ca A_g(\RR)$. See \cite{grossharris}, \cite{silholsurfaces} and \cite{seppalasilhol2} for details. 
\end{align*}
We verify in Theorems \ref{th:homeomorphismmoduli} and \ref{th:homeomorphismmoduli2} that these bijections are homeomorphisms. 
% for the bijections (\ref{analytictopology:def}) are actually homeomorphisms. 
\end{example}
\noindent
Examples \ref{algebraicquotients} and \ref{introexample:grossharris} suggest that if $F$ arises as the functor attached to an algebraic stack over $\RR$, then $F(\RR)$ should carry the structure of a real-analytic orbifold. In Chapter \ref{ch:realmodulispaces}, we show that this indeed the case. 
%we generalize these approaches in the following way. 
We define a functorial orbifold structure on the %set of isomorphism classes $\va{\mr X(\RR)}$ of the 
real locus $\mr X(\RR)$ of a smooth, separated Deligne-Mumford stack $\mr X$ over $\RR$. One retrieves the real-analytic space $\mr X(\RR)$ when $\mr X$ is a scheme. %This answers Question \ref{howtomakerealmodspace} for moduli problems that are representable by an algebraic stack. 

\subsection{Complex and real period maps} \label{intro:sub:periodmaps}

We conclude that in complex and in real algebraic geometry, moduli spaces arise naturally via algebraic moduli stacks. 
%In the complex case, one analytifies the coarse moduli space of the stack; in the real case, one proceeds as above. 
As a next step, one may want to equip the moduli space with some additional structure. Hodge theory provides an excellent method for this goal, as we shall now explain.
%which we will explain next. 
%If a complex variety is determined by its Hodge structure, its moduli space uniformizes as (an open subset of) the quotient of a period domain by a discrete group of automorphisms. Moreover, if the variety has a real structure, then the attached period domain will have a real structure, and the period map will be defined over $\RR$. Let us explain all this in a bit more detail. 
\\
\\
Consider a projective holomorphic submersion of complex manifolds 
\begin{equation} \label{introeq:family}
\begin{split}
    \xymatrix{
    & \PP^N(\CC) \times B \ar[dr]& \\
     X \ar[rr]^\pi\ar@{^{(}->}[ur]&& B,
    }
\end{split}
\end{equation} 
where the manifold $B$ is connected and equipped with a base point $0 \in B$. %such that each fiber $X_t \coloneqq \pi^{-1}(t)$ has constant dimension $n$. Fix a point $0 \in B$. 
Possibly after replacing $B$ by an open subset around $0$, for each $t \in B$ there is an isomorphism $$\rm H^k(X_t, \ZZ) \cong \rm H^k(X_0, \ZZ).$$ %and the Hodge numbers $h^{p,q}(X_t) = \dim_\CC H^{p,q}(X_t)$ are constant on $B$. This gives 
One obtains a family of Hodge structures on $\rm H^k(X_0, \ZZ)$, parametrized by $B$. %A fundamental theorem of Griffiths says that this family is holomorphic: %by a fundamental theorem of Griffiths: %says that the Hodge structures on $H^k(X_0, \ZZ)$ vary holomorphically with $t \in B$. To be precise, %Since the fibers $X_t$ vary holomorphically, one may wonder if the Hodge structures on $H^k(X_0, \ZZ)$ vary holomorphically as well. This interesting result appears to be true by a theorem of Griffiths, and is more precisely stated in the following way: 
%for each $t \in B$, the Hodge structure on $H^k(X_t, \ZZ)$ corresponds to a filtration 
%induce 
%$F^\bullet_t$ on $H^k(X_t, \CC)$, 
The Hodge filtrations on $F^\bullet_t$ on $\rm H^k(X_0, \CC)$ induce a map from $B$ into a flag variety $\ca F$ %Fundamental theorems of Griffiths say that not only this map is holomorphic, but %Moreover, the set of filtrations in $\ca F$ that arise from geometry in this way satisfy certain conditions defining 
%one also knows that its 
and by the factorization of $\pi$ in (\ref{introeq:family}) %and by \emph{Griffiths transversality}, 
the image is contained in a certain locally closed subset %$\ca D $
%one of which open and one of which closed on $\ca F$, 
%condition called Griffith's transversality, and that this condition defines an open subset 
$\ca D \subset \ca F$, the \textit{period domain} \cite[\S10.1.3]{voisin}. The induced map
\begin{align*} %\label{introeq:periodmap}
P:    B \to \ca D
\end{align*}
is called a \textit{period map}, and turns out to be holomorphic. 

This procedure can be globalized, by taking a suitable cover $\widetilde B \to B$ instead of restricting to simply connected opens to trivialize the local system with stalks $\rm H^k(X_t,\ZZ)$. One obtains a morphism $\widetilde P \colon \widetilde B \to \ca D$ that induces a commutative diagram
%commutative diagram 
%This gives a commutative diagram of the form 
\begin{align}\label{perioddiagram}
%\xymatrixcolsep{5pc}
\begin{split}
\xymatrix{
\widetilde B \ar[d] \ar[r]^{\widetilde P} & \ca D \ar[d]& \\
B \ar[r]^{P \hspace{1mm}} & \Gamma \setminus B, & \Gamma = \Aut\left(\rm H^k(X_0, \ZZ)\right).
} 
\end{split}
\end{align} 
%In fact, there are several very moduli spaces of complex varieties admitting a period map whose target, the period domain, has the same dimension as the moduli space. 
%A global period map for the family (\ref{introeq:family}) is a morphism $P \colon B \to \Gamma \setminus \ca D$, where $\Gamma$ is a discrete group acting properly discontinuously on the period domain $\ca D$. 
%Thus, the global period map of a family such as (\ref{introeq:family}), is a morphism $P \colon B \to \Gamma \setminus \ca D$ with $\Gamma$ a discrete group acting properly discontinuously on the period domain $\ca D$. 
\noindent
The above construction applies to the complex locus $\msf M(\CC)$ of the coarse moduli space $\msf M$ of any separated Deligne-Mumford moduli stack $\ca M$ which is smooth over $\CC$. %To apply this construction to the case where $B = \mr M_\CC$ is a coarse moduli space of complex varieties, 
Indeed, although the complex analytic space $\msf M(\CC)$ might be singular and lacking a universal family, %When the corresponding moduli stack $\ca M$ is smooth, 
for a suitable \'etale cover $U \to \ca M$ there is a period map $\widetilde P \colon U(\CC) \to \ca D$ that descends %a suitably defined period morphism of complex orbifolds $\ca M(\CC) \to [\Gamma \setminus \ca D]$ descends 
to a morphism of analytic spaces $P \colon \msf M(\CC) \to \Gamma \setminus \ca D$. 

For such a smooth moduli stack $\ca M$, one may wonder if the map $\widetilde P \colon U(\CC) \to \ca D$ is immersive (infinitesimal Torelli), if the map $P \colon \msf M(\CC) \to \Gamma \setminus \ca D$ is injective (global Torelli), and if one can explicitly describe the image of $P$. 
%or surjective onto $\Gamma \setminus U$ for an explicit open subset $U \subset \ca D$. 
%If so, the resulting isomorphism $\mr M_\CC \cong \Gamma \setminus U$ can be useful to describe the geometry of $\mr M_\CC$. 
%the above procedure generalizes to define a period map $\ca P \colon \ca M(\CC) \to \Gamma \setminus \ca D$ for the universal family over the complex orbifold $\ca M(\CC)$. %Alternatively, adding some level structure defines a finite \'etale cover of $\ca M$, whose period map factors through a map of complex analytic spaces $P \colon \mr M_\CC \to \Gamma \setminus \ca D$. 
%smooth variety $S$ and a finite \'etale cover $S \to \ca M$. 
%The period map will $S \to \Gamma \setminus \ca D$ factors through a morphism of analytic spaces $\mr M_\CC \to \Gamma \setminus \ca D$. 
%finite cover $S \to \mr M_\CC$ as well as a universal family $\mr X \to S$, and the period map $S \to \Gamma \setminus \ca D$ factors through a morphism of analytic spaces $\mr M_\CC \to \Gamma \setminus \ca D$. 
%these are often not smooth, and do not admit universal families. 
%provides a useful description of $\mr M_\CC$ in terms of the period domain. 
%reveal interesting geometric properties. 
%can be used to study the geometry of $\mr M_\CC$. 
%For this to hold, the dimensions of the moduli space and the period domain should the same; if it happens, 
%If this happens, 
%this gives a useful way of describing 
%If so, the period map allows one to describe the moduli space in terms of the geometry of the period domain.
\begin{examples} \label{AVK3examples}
%Such a \emph{complex uniformization} result holds for 
These statements are true for the moduli stack $\ca A_g$ of $g$-dimensional principally polarized abelian varieties \cite{birkenhake}, as well as for the moduli stack $\ca M_d$ of degree $2d$-polarized K3 surfaces \cite{Palaiseau, HuybrechtsK3}. 
% holds for abelian varieties \cite{birkenhake} and K3 surfaces . % (after a slight adjustment, allowing $B$ to be an orbifold). 
\end{examples}
%Moreover, if the Hodge structure does not vary, one may take covers of projective space ramified along a fiber of \ref{introeq:family} to define an \emph{occult period map} \cite{ACTsurfaces, MR2789835, kudlarappoortocculte}. 
\noindent
 %Let us return to our general family (\ref{introeq:family}). As opposed to Examples \ref{AVK3examples}, 
%Returning to our general family (\ref{introeq:family}), 
In contrast with Examples \ref{AVK3examples}, it can occur that for the family (\ref{introeq:family}), the period map $P$ in diagram (\ref{perioddiagram}) is actually constant. %the Hodge structures on $\rm H^k(X_0, \ZZ)$ do not vary with $t \in B$ at all. %even though (\ref{introeq:family}) is not isotrivial, % -- think of the universal family of cubic surfaces. %The resulting period map is then constant. 
It may also happen that the period map is an embedding but $\dim(B) < \dim(\ca D)$. %whereas the period map embeds $B$ into $\Gamma \setminus \ca D$, the dimension of $\ca D$ is larger than the dimension of $B$. 
%\emph{Occult period maps} can be useful in such cases. To explain what these are: 
In these cases, it is sometimes useful to first associate an auxiliary variety $Y_t$ to the fiber $X_t$ above $t \in B$, carrying some additional structure (such as an embedding $H \subset \Aut(Y_t)$ for some fixed finite group $H$), and then take periods, namely those of the variety $Y_t$. %One constructs $Y_t$  %$H \subset \Aut(Y_t)$ for some finite group $H$. 
For example: 
\begin{enumerate}
\item If $X_t \subset \PP^3_\CC$ is a cubic surface, one can take $Y_t$ to be the triple cover of $\PP^3_\CC$ ramified along $X_t$ \cite{ACTsurfaces}. %and then consider the Hodge structure on the third primitive cohomology group of $Y_t$ with action of $\ZZ[\zeta_3]$
%is the lattice of an abelian variety with action of $\ZZ[\zeta_3]$. 
\item If $X_t = \{F_t = 0 \} \subset \PP^1_\CC$ is a binary quantic \cite[Chapter 4]{GIT}, one can take $Y_t$ to be a finite cover of $\PP^1_\CC$ ramified along $X_t$ \cite{DeligneMostow}, \cite[Section 5.3]{Moonen2011TheTL}. 
\end{enumerate}
%The varieties $\{Y_t\}_{t\in B}$ fit into a family
%, so that %its group of automorphisms contains it comes equipped with additional endomorphisms, 
%The period domain of the newly constructed family 
%\begin{align*} %\label{newfamily}
In this way, one obtains a family $\rho \colon Y \to B$ whose period domain $\ca D$ contains a sub-domain $\ca D' \subset \ca D$ that parametrizes Hodge structures with additional structure (see \cite[\S17.3]{periodmappings}). The period map of $\rho$ factors through a map $P \colon B \to \Gamma' \setminus \ca D'$ that might be closer to being an isomorphism than the original period map was. %even if the original period map was, or even satisfy the equality $\dim(B) = \dim(P(B))$. 
Such a map $P$ is called an \emph{occult period map}, due to its hidden nature \cite{kudlarappoortocculte}. 
% and are useful for studying complex moduli. 
%Occult period maps turn out to be very useful for studying complex moduli:
%The map (\ref{occultmap}) is is called an
 %that takes into account this additional structure. 
%If one is lucky, one will now have $\dim(B) = \dim(\ca D')$, and may start to try proving Torelli theorems for (\ref{occultmap}). 
%Sometimes, a moduli space can be described 

\begin{examples} \label{occultexamples}
The idea of considering the periods of a branched cover of the projective line goes back to Picard \cite{picardperiods}; Shimura studied the moduli of such curves in \cite{shimuratranscendental}. Deligne, Mostow and Thurston developed on the work of Picard, determining for which stable configurations of points on $\PP^1(\CC)$ the occult period map is an isomorphism \cite{DeligneMostow, mostowgeneralized, Thurston1998ShapesOP}. 
 %The moduli space of a certain configuration of points on $\PP^1(\CC)$ admits a natural occult period map to a ball quotient;  listed those  configurations for which this map is an %spaces of stable configurations of points on $\PP^1(\CC)$ for which the period map to a ball quotient is an 
%by complex reflection groups  
%the name \emph{occult period map} was introduced in \céite{kudlarappoortocculte}). 
Allcock--Carlson--Toledo identified the complex moduli space of stable cubic surfaces with a four-dimensional ball-quotient \cite{ACTsurfaces, beauvillecubicsurfaces, periodmappings}, extending their construction to cubic threefolds in \cite{ACTthreefolds} (c.f. \cite{looijenga_swierstra_2007}). Likewise, Kondō studied complex moduli of curves of genus three and four %used occult period maps to study moduli of curves of genus three and four 
\cite{kondogenus3, kondogenus4}, and then genus six with Artebani \cite{genussixkondo}. 
%them to morphisms of moduli stacks. 
%They descended the latter over their natural fields of definition \cite{kudlarappoortocculte}, after which Achter descended them to morphisms of integral models. 
%cubic surfaces \cite{ACTsurfaces} and cubic threefolds \cite{ACTthreefolds}.\cite{MR2789835}ACTsurfaces, 
\end{examples}
\begin{remark}
Kudla and Rapoport %came up with the name \emph{occult period maps}. They 
observed that occult period maps are compatible with families, extending them to morphisms of stacks defined over natural fields of definition \cite{kudlarappoortocculte}. The arithmetic these morphisms was studied by Achter \cite{achteroccult}.
%In Examples \ref{occultexamples}, the auxiliary object of which the period is taken is always an abelian variety or a K3 surface; see also
\end{remark}
\noindent
What happens over the real numbers? %One may wonder if in such favorable cases, something holds over $\RR$. 
Let us assume that family (\ref{introeq:family}) comes equipped with anti-holomorphic involutions \[
\left(\tau: X \to X, \sigma: B \to B \mid \sigma(0) = 0\right).\]
After replacing $B$ by a $G$-stable contractible open subset around $0$, we can trivialize the local system with stalks $\rm H^k(X_t, \ZZ)$ in a $G$-equivariant way. %the fixed-point set $B^G = B^\sigma$ is connected. By Ehresmann's theorem, the diffeomorphism type of the real locus $X_t(\RR) = (X_t)^\tau$ remains constant for $t \in B^G$, and 
The involution $$
\tau^\ast: \rm H^k(X_0, \ZZ) \to \rm H^k(X_0, \ZZ)
 $$ can then be used to define a natural anti-holomorphic involution on $\ca D$, making the period map %$\ca P: B \to \ca D$ 
 $G$-equivariant. %with respect to the anti-holomorphic $G$-actions. 
Taking fixed points results in a smooth \textit{real period map} $$\ca P^G \colon B^G \to \ca D^G.$$ 
As in the complex case, one may, instead of restricting to small opens, use $G$-equivariant covers to trivialize the monodromy. Such covers can also be used to define real period maps on real moduli spaces. This leads to interesting results. 
%Applying this to connected components of real moduli spaces, this sometimes leads to interesting uniformization results. 
%real uniformization of the \emph{connected components} of the moduli space of smooth varieties $\ca M(\RR)$. 

\begin{examples} \label{introexamples:realperiods}
%To uniformize connected components of a real moduli space of smooth varieties is relatively easy when the corresponding complex period map is an isomorphism. r in-
Tate identified the real moduli space of elliptic curves with the disjoint union of two copies of $\RR^\ast$ \cite[Chapter 1, Proposition 4.3.1]{grossarithmeticcomplexmultiplication}, which was generalized to abelian varieties by Gross--Harris and Silhol \cite{grossharris, silholsurfaces}. The case of real K3 surfaces was treated by Nikulin \cite{Nikulin1980,itenbergenriques} (see also \cite{KHARLAMOVK3, silholsurfaces}). % for a classification of the diffeomorphism type of the real locus of a degree four surface in $\PP^3_\RR$. 
%As for K3 surfaces, Kharlamov \cite[Ch.VIII,3.4]{silholsurfaces} lists all compact $\ca C^\infty$-surfaces that occur as real locus of a K3 surface over $\RR$, and Nikulin describes in \cite{Nikulin1980} the complete set of components of the moduli space of degree $2d$ polarized K3 surfaces over $\RR$. 
%In the ball quotient cases, each component of ${\ca M}(\RR)$ becomes an open subset of a quotient of real hyperbolic space. 
\end{examples}
\noindent
The occult period map construction carries over to the reals as well. %The construction of the occult period map (\ref{occultmap}) of family (\ref{introeq:family}) is natural. 
%Occult period maps \cite{kudlarappoortocculte} (see Examples \ref{occultexamples}) are constructed naturally, which implies that 
%Indeed, if family (\ref{introeq:family}) is defined over $\RR$, then family (\ref{newfamily}) will be defined over $\RR$. 
%Consequently, %the occult period map construction carries over to the reals, and may applied to moduli spaces in real algebraic geometry. 

\begin{examples} \label{realoccultexamples}
%To uniformize connected components of a real moduli space of smooth varieties is relatively easy when the corresponding complex period map is an isomorphism. r in-
%Tate identified the real moduli space of elliptic curves with the sum of two copies of $\RR^\ast$ \cite[Ch.1, Prop.4.3.1]{grossarithmeticcomplexmultiplication}, which was generalized to abelian varieties by Gross--Harris and Silhol \cite{grossharris, silholsurfaces}. 
\emph{Equivariant occult period maps} are used to uniformize the connected components of the real moduli spaces of cubic surfaces \cite{realACTsurfaces}, binary sextics \cite{realACTnonarithmetic}, curves of genus three \cite{riekenthesis, heckman2016hyperbolic}, and binary octics \cite{chuthesis, chu2011octics}.  
%As for K3 surfaces, Kharlamov \cite[Ch.VIII,3.4]{silholsurfaces} lists all compact $\ca C^\infty$-surfaces that occur as real locus of a K3 surface over $\RR$, and Nikulin describes in \cite{Nikulin1980} the complete set of components of the moduli space of degree $2d$ polarized K3 surfaces over $\RR$. 
%In the ball quotient cases, each component of ${\ca M}(\RR)$ becomes an open subset of a quotient of real hyperbolic space. 
\end{examples}

\noindent
This concludes our discussion of the basics of complex and real moduli theory. It is time to meet the other main players:

%Let us introduce the other main players:

%now turn to the other main players of this thesis: 
\subsection{Complex and real algebraic cycles}\label{intro:sub:complexandrealalgebraiccycles}

Any algebraic subvariety of a smooth projective variety induces a class in a suitably defined cohomology group. 
%An algebraic subvariety is in particular a topological subspace and as such, it induces a class in cohomology. 
Important conjectures predict that one can understand the subgroup of algebraic classes in the cohomology of this variety via structures that are \emph{a priori} not directly related to the algebraic cycles themselves (i.e. Hodge theory, Galois representations, etc.). In this section, we explain this in more detail when the base field is either $\CC$ or $\RR$. Since the situation over $\CC$ is better understood than the situation over $\RR$, we start our discussion in the complex direction. 
% (in fact, the \emph{real integral Hodge conjecture} has only been introduced very recently \cite{BW20}). 
%We will accordingly start our discussion in the complex direction, before we move on to the real case. 
%In an ideal world, one would like to understand 
%Comparable with the discussion in Sections 1.2.1 and 1.2.2, we shall first explain 
\\
\\
Let $X$ be a smooth projective variety over $\CC$. % one can consider the singular cohomology groups $\rm H^{k}(X(\CC), \ZZ)$. 
A fundamental theorem of Hodge \cite{MR139613} says that for each integer $k$, there is a functorial \emph{Hodge decomposition} 
\begin{align*}%\label{hodgedecomposition}
\rm H^k(X(\CC), \CC) = \rm H^k(X(\CC), \ZZ) \otimes \CC = \bigoplus_{p+q = k} \rm H^{p,q}(X) \;\; \textnormal{such that} \;\; \overline{\rm H^{p,q}(X)} = \rm H^{q,p}(X). 
\end{align*}Define the group of \emph{integral Hodge classes} of degree $2k$ as 
\[
\Hdg^{2k}(X(\CC), \ZZ) = \set{\alpha\in\rm H^{2k}(X(\CC), \ZZ) \mid \alpha_\CC \in \rm H^{k,k}(X) \subset \rm H^{2k}(X(\CC), \CC)}.
\]
These Hodge classes derive their main interest from the following fact. Let $Z \subset X$ be an irreducible subvariety of dimension $i$, and $\widetilde Z \to Z$ a resolution of singularities \cite{hironaka}. The composition $f\colon \widetilde Z \to Z \hookrightarrow X$ induces a Gysin homomorphism 
%\cite[page 169]{voisin}
\begin{align*}%\label{intro:fundamentalZ}
%\ZZ\cdot[Z] = 
f_\ast \colon \rm H^{0}(\widetilde Z(\CC), \ZZ) \to \rm H^{2k}(X(\CC), \ZZ) \quad \quad (k = \dim(X) - i)
\end{align*}
which is compatible with the Hodge decomposition. Thus, we obtain a Hodge class $[Z] = f_\ast(1) \in \rm H^{2k}(X(\CC), \ZZ)$. This construction induces a homomorphism
\begin{align} \label{intro:IHC}
\CH_i(X) \to \Hdg^{2k}(X(\CC), \ZZ) \quad \quad (k = \dim(X) - i). 
\end{align}

\begin{definition}
For a smooth projective variety $X$ over $\CC$, the \emph{integral Hodge conjecture for $i$-cycles} (IHC$_i$) is the property that the homomorphism (\ref{intro:IHC}) is surjective. Thus, $X$ satisfies IHC$_i$ if (\ref{intro:IHC}) is surjective, and $X$ fails IHC$_i$ if (\ref{intro:IHC}) is not surjective. 
% refers to the property that %We say that $X$ satisfies the \emph{integral Hodge conjecture for $i$-cycles} (IHC$_i$) if 
%the homomorphism (\ref{intro:IHC}) is surjective. 
\end{definition}
\noindent
%Despite its name, %the integral Hodge conjecture for $i$-cycles is a property rather than a conjecture: 
It is known since Atiyah and Hirzebruch \cite{atiyahintegralhodge} that there are varieties $X$ for which the property IHC$_i$ does not hold. %These are counterexamples for $k = n-2$ and $n \geq 7$. 
In some interesting cases, however, the integral Hodge conjecture for $i$-cycles turns out to be satisfied. %Before stating our main result in this direction, 
In Section \ref{intro:sub:theintegralhodgeconjectureforjacobianvarieties} below, we will indicate more precisely what is and what is not known in this direction. 
%\subsection{Real algebraic cycles}
\\
\\
Let us carry the discussion over to the real setting. Let $X$ denote a smooth projective algebraic variety over $\RR$. Over the years, much work has been done to establish the right analogue of the cycle class map in real algebraic geometry. Borel and Haefliger \cite{borelhaefliger} constructed a cycle class with values in $\rm H^{k}(X(\RR),\ZZ/2)$ for any real algebraic subvariety $Z \subset X$ of codimension $k$, by considering the embedding of real loci $Z(\RR) \subset X(\RR)$. %and then proceeding as above. %, but considering
%$Z(\RR)$ defines a $k$-cycle with coefficients in $\ZZ/2$, and thus a class in $H_{k}(X(\RR),\ZZ/2) \cong H^{k}(X(\RR), \ZZ/2)$. 
The study of the subgroup 
\begin{align*}% \label{intro:realcohomology}
\rm H_{\text{alg}}^{k}(X(\RR), \ZZ/2) \subset \rm H^{k}(X(\RR), \ZZ/2)
\end{align*} formed by these classes is a classical topic in real algebraic geometry \cite{silholsurfaces, BK98, mangolte}, related to the problem of $\ca C^{\infty}$ approximation of submanifolds of $X(\RR)$ by algebraic subvarieties \cite{BCR87,cinftyapproxbenoist}, \cite[\S 6.2]{BW21}.  

%As explained above, the image of the class of a real subvariety $Z \subset X$ in the singular cohomology ring $\rm H^{\bullet}(X(\CC), \ZZ)$ is contained in the ring of integral Hodge classes; this observation provides an obstruction for any cohomology class to be algebraic. Do we have similar obstructions for classes in $\rm H^{k}(X(\RR), \ZZ/2)$ to be algebraic? %The answer turns out to be \emph{yes}. To explain why this is so, we need a more refined cohomology theory -- this makes sense since %$\ZZ$-module $\ZZ$ on which the non-trivial element $\sigma \in G$ acts by sending $1$ to $-1$. %$G$-module $(\sqrt{-1})^q\cdot \ZZ$ on which $G$-acts by complex conjugation.  
%The group $H^i_G(X(\CC), \ZZ(q))$ may also be defined as follows: the functor $\Gamma^G: F \mapsto \Gamma(X(\CC), F)^G$ on the abelian category of abelian $G$-sheaves on $X(\CC)$ is left-exact, hence can be derived and one obtains $H^i_G(X(\CC), \ZZ(q)) = R^i\Gamma^G(X(\CC), \ZZ(q))$ [grothendieck]. 
%In particular, Krasnov defines a canonical cycle class map 
%\[
%\CH_d(X) \to \rm H^{2k}_G(X(\CC), \ZZ(k))\quad \quad (k = \dim(X) - d). 
%\]
To study the algebraic cycles on $X$, it is natural to consider a more refined cohomology theory than $\rm H^{\bullet}(X(\RR), \ZZ/2)$, forasmuch as the real locus $X(\RR) = X(\CC)^G$ constitutes only a small part of the topological $G$-structure of $X(\CC)$. The right choice seems to be provided by $G$-equivariant cohomology $\rm H^q_G(X(\CC), \ZZ(j))$ in the sense of Borel (see \cite{tohoku} or \cite{ATIYAH19841}), where $\ZZ(j)$ is the abelian group $\ZZ$ turned into a $G$-module by declaring that $\sigma(1)  = (-1)^j$ for the generator $\sigma \in G$. 

The study of equivariant cohomology in real algebraic geometry was initiated by Krasnov \cite{krasnovcharact}, and continued by Van Hamel in his thesis \cite{vanhamel}, who used extensively the formalism of equivariant sheaves and derived categories. As was noted by both \cite{krasnovgroth,vanhamel}, there is an equivariant cycle class map $\CH_i(X) \to \rm H^{2k}_G(X(\CC), \ZZ(k))$ that fits in a commutative diagram
\begin{equation*}% \label{eq:commutativecycleclass}
\xymatrixcolsep{1.5pc}
\xymatrix{
&\CH_i(X) \ar[dl]\ar[d]\ar[dr]& (k  = \dim(X) - i) \\
\rm H^{k}(X(\RR), \ZZ/2) & \ar[l] \rm H^{2k}_G(X(\CC), \ZZ(k)) \ar[r]& \rm H^{2k}(X(\CC), \ZZ(k)).
}
%H^{2k}_G(X(\CC), \ZZ(q)) \to H^{k}(X(\RR), \ZZ/2) 
\end{equation*}
%which sends the class of a $k$-cycle to its Borel–Haefliger cycle class. 
%\begin{question} \label{introquestion:BW}
%\end{question}
%What should then be the right set of conditions on classes in $\rm H^{2k}_G(X(\CC), \ZZ(k))$ to consider, so that all algebraic classes satisfy these conditions, and that in favorable situations, these conditions are actually sufficient to distinguish algebraic classes from non-algebraic ones? 
%What should then be the appropriate analogue of $\Hdg^{2k}(X(\CC), \ZZ(k))$ in $G$-equivariant cohomology? %What should be the smallest functorially defined subgroup of the $G$-equivariant cohomology group $\rm H^{2k}_G(X(\CC), \ZZ(k))$, such that the cycle class map factors through it, but whose definition only depends on the topology of $X(\RR)$ and the $G$-analytic structure of $X(\CC)$ and \emph{not} on the algebraic cycles of $X$? 
%This question was answered by Benoist and Wittenberg in \cite{BW20}: %the following way. 
 %where they formulate, for a smooth projective variety $X$ over $\RR$, the \textit{real integral Hodge conjecture for $k$-cycles on $X$}. 
%There is an obvious obstruction for a class in $\rm H^{2k}_G(X(\CC), \ZZ(k))$ to be algebraic: its image in $\rm H^{2k}(X(\CC), \ZZ(k))$ might not be of type $(k, k)$. 
%Benoist and Wittenberg then concluded in \cite{BW20} that 
It turns out that algebraic classes are contained in a certain subgroup 
\[
\textnormal{Hdg}^{2k}_G(X(\CC), \ZZ(k))_0  = \Hdg^{2k}_G(X(\CC), \ZZ(k)) \cap \rm H^{2k}_G(X(\CC), \ZZ(k))_0 \subset \rm H^{2k}_G(X(\CC), \ZZ(k))
\]
that takes into account the Steenrod operations on $\rm H^\bullet(X(\RR), \ZZ/2)$ as well as the Hodge structure on $\rm H^{2k}(X(\CC), \ZZ(k))$. See Section \ref{sec:realintegralhodge} for details. This leads to:
%This leads to:
%It follows that any algebraic class $\alpha \in \rm H^k(X(\RR), \ZZ/2)$ lies in the image of the natural map $\textnormal{Hdg}^{2k}_G(X(\CC), \ZZ(k))_0  \to \rm H^k(X(\RR), \ZZ/2)$, which answers the question above. =
%\cite[Theorem 1.18]{BW20} (which builds upon \cite{kahn} and \cite{krasnovgroth}).
%satisfy a certain topological condition \hypertarget{star}{$(\ast)$}, specific to real algebraic geometry. This was discovered by Kahn \cite{kahn} and Krasnov \cite{krasnovgroth}. See \cite[Theorem 1.18]{BW20} for a precise formulation of condition \hyperlink{star}{$(\ast)$}. Denote by $\Hdg^{2k}_G(X(\CC), \ZZ(k))$ the subgroup of classes $\alpha \in \rm H^{2k}_G(X(\CC), \ZZ(k)) $ such that the image of $\alpha$ in $\rm H^{2k}(X(\CC), \ZZ(k))$ is Hodge. Next, define $\rm H^{2k}_G(X(\CC), \ZZ(k))_0$ as the subgroup of classes in $\rm H^{2k}_G(X(\CC), \ZZ(k)) $ that satisfy the topological condition \hyperlink{star}{$(\ast)$} referred to above. Then the image of the cycle class map is contained in
%\begin{equation}
%\textnormal{Hdg}^{2k}_G(X(\CC), \ZZ(k))_0 \subset \rm H^{2k}_G(X(\CC), \ZZ(k))    
%\end{equation}
%the subgroup of classes $\alpha \in \rm H^{2k}_G(X(\CC), \ZZ(k))$ such that: 
%\begin{enumerate}xr
%\item the image of $\alpha$ in $\rm H^{2k}(X(\CC), \ZZ(k))$ is Hodge, and 
%\item $\alpha$ satisfies the topological condition \hyperlink{star}{$(\ast)$} referred to above. 
 %, see \cite[Definition 1.19]{BW20}. 
%\end{enumerate}

\begin{definition}[Benoist--Wittenberg]
A smooth projective variety $X_{/\RR}$ satisfies the \textit{real integral Hodge conjecture for $i$-cycles} ($\RR$IHC$_i$) if the following map is surjective:
\begin{equation*}
\CH_i(X) \to \textnormal{Hdg}^{2k}_G(X(\CC), \ZZ(k))_0 \quad \quad (k = \dim(X) - i). 
\end{equation*}
\end{definition}
\noindent
Thus, in spite of its name, $\RR$IHC$_i$ is a property that a smooth projective variety $X$ can satisfy. There are varieties for which it holds, and there are varieties for which it fails. %the real integral Hodge conjecture for $d$-cycles  $\RR$IHC$_d$ is a property that a smooth projective variety $X$ over $\RR$ can satisfy. There are examples for which it holds, and examples for which it fails. %holds for smooth projective varieties over $\RR$, and fails for others. %not really a conjecture: in some cases it holds, in some cases it fails. 
%In fact, (as in the complex case) it is a very interesting problem to find examples and counterexamples to the real integral Hodge conjecture. 
%$\RR$IHC$_d$. %In fact, for a fixed smooth projective variety $X$ over $\RR$, it is a very interesting problem to find out for which $d$ it satisfies $\RR$IHC$_d$, and for which $d$ that fails. 
%Indeed, the real integral Hodge conjecture 
In some sense, the property -- if satisfied -- reveils deep links between the equivariant complex structure of $X(\CC)$, the topology of $X(\RR)$, and the real subvarieties of $X$. For instance, if $\textnormal{Hdg}^{2k}_G(X(\CC), \ZZ(k))_0$ is large, then there should be many subvarieties even though it may be difficult to write down explicit equations. 
%deep insight in the algebraic subvarieties of $X$. %As such, it tells us much about the geometry of $X$ and of its subvarieties. 
%For instance, it may be difficult to write down explicit equations for subvarieties $Z \subset X$. If, at the same time, the group $\textnormal{Hdg}^{2k}_G(X(\CC), \ZZ(k))_0$ is large, then $\RR$IHC$_k$ predicts that in spite of this difficulty there should be many of them. 
%since it holds for some varieties but fails for others. 
\begin{comment}
Moreover, the real integral Hodge conjecture can reveil important information about the subgroup (\ref{intro:realcohomology}). 
%$$\rm H_{\text{alg}}^{k}(X(\RR), \ZZ/2) \subset \rm H^{k}(X(\RR), \ZZ/2).$$ 
Indeed, if $r \colon \textnormal{Hdg}^{2k}_G(X(\CC), \ZZ(i))_0 \to \rm H^{k}(X(\RR), \ZZ/2)$ is not surjective, then by the commutativity of (\ref{eq:commutativecycleclass}), the inclusion (\ref{intro:realcohomology}) is strict. 
%we have $$\rm H_{\text{alg}}^{k}(X(\RR), \ZZ/2) \neq \rm H^{k}(X(\RR), \ZZ/2).$$ 
This implication explains many of the known examples of real varieties $X$ for which the inclusion (\ref{intro:realcohomology}) is strict for some $k$ %that $\rm H_{\text{alg}}^{k}(X(\RR), \ZZ/2) \neq H^{k}(X(\RR), \ZZ/2)$ for some k 
\cite[Remark 2.7(ii)]{BW20}. On the other hand, if $r$ is surjective, then the real integral Hodge conjecture implies equality in (\ref{intro:realcohomology}). 
\end{comment}
\\
\\
%Hopefully, the discussion above has provided some insight into the general theories of moduli spaces and algebraic cycles over the real numbers, at least enough to understand our main
This concludes our introduction to algebraic cycles and cohomology of complex and real algebraic varieties. Now that the stage is set, and its main players are met, %Having set the stage, at Now that we have introduced generalities on moduli and algebraic cycles over the real numbers, 
it is time to present our results. 
%it is time to state the results that we have obtained in these directions. 

\newpage

\section{Discussion of the main results} \label{mainresults}

%Let us discuss the main results of this thesis. 
%\subsection{The integral Hodge conjecture for one-cycles on Jacobian varieties \\ \textcolor{white}{.} \hfill (with \textsc{Thorsten Beckmann})}
\subsection{Noether-Lefschetz loci for real abelian varieties \hfill (based on \cite{degaayfortmanrealmodulispaces})}
\label{intro:sub:realabeliansubvarietiesinfamily}

Our first result concerns the behaviour of specific algebraic cycles on real abelian varieties in a family. Namely, we have considered the following:
%Namely, we have analyzed the following: 
\begin{question} \label{intro:questionabelianscheme}
%\emph{Question:} $\;$ 
Let $A \to B$ be a family of polarized real abelian varieties, and let $k$ be a positive integer. Let $S_k(B(\RR)) \subset B(\RR)$ be the subset of points $t \in B(\RR)$ such that $A_t$ contains a real abelian subvariety of dimension $k$. Does there exist a natural criterion for the density of $S_k(B(\RR))$ in $B(\RR)$?
%When is $S_k(B(\RR))$ dense in $B(\RR)$? 
\end{question}
%How do the non-simple abelian varieties distribute in $B(\RR)$? 
%Let $S_k(\RR)$ be the set of points $t \in B(\RR)$ such that $A_t$ contains an abelian subvariety of dimension $k$. How does $S_k(
%he \emph{Noether-Lefschetz locus} $\tn{NL}(\mr B)$ is the locus where the Picard group is not generated by a hyperplane section -- in other words, the set of $t \in \mr B(\CC)$ such that $\mr S_t$ contains a curve which is not a hypersurface. The main theorem of \emph{loc.cit.} says that even though $\tn{NL}(\mr B)$ is a countable union of proper closed algebraic subvarieties \cite{MR0033557}, the union of the general components of $\tn{NL}(\mr B)$ is dense in $\mr B$, for the euclidean topology.  
\noindent
Viewing the above family as an equivariant variation of Hodge structure, Question \ref{intro:questionabelianscheme} concerns certain \emph{special loci} in $B(\RR)$. Indeed, we consider loci of $t \in B(\RR)$ such that the $G$-equivariant Hodge structure on the cohomology of $A_t(\CC)$ is "special" in the sense that it contains a $G$-stable sub-Hodge structure of rank $k$. What we ask for, is a density criterion for these special loci. Phrased like this, Question \ref{intro:questionabelianscheme} might remind of the Green--Voisin criterion for density of Noether-Lefschetz loci \cite[Proposition 17.20]{voisin}. Before we start looking for an answer, let us consider the analogous question for surfaces in $\PP^3$. 

%This criterion implies density of the locus where the fiber contains many Hodge classes for a variation of Hodge structure of weight $2$. 
Let $d \geq 4$ be an integer, and consider the universal degree $d$ smooth hypersurface $$\PP^3 \times \mr B \supset \mr S \to \mr B.$$% in $\bb P^3_\CC$. 
The Noether-Lefschetz locus %$\text{NL}(\mr B(\CC))$ 
is the set of $t \in \mr B(\CC)$ such that $\mr S_t$ contains a curve which is not a complete intersection. By the main result of \cite{NLlocus}, %is that %$\text{NL}(\mr B(\CC))$ 
the Noether-Lefschetz locus is euclidean dense in $\mr B(\CC)$, despite having empty interior \cite{MR0033557}. 

Do we have a similar result over the reals? This question makes sense, since the above universal surface $\mr S \to \mr B$ is naturally defined over $\RR$. %the Green-Voisin Density Criterion cannot be adapted to the reals without altering the hypothesis. %Going back to the universal family $\mr S \to \mr B$ of degree $d\geq 4$ surfaces in $\bb P^3_{{\bb C}}$, %one observes that this family has a real structure, so that 
In analogy with the above, one may thus define the \textit{real Noether-Lefschetz locus} %$$\NL(\mr B({\bb R})) \subset \mr B({\bb R})$$ 
as the set of real surfaces $S$ in $\bb P^3_{{\bb R}}$ with $\Pic(S) \not \cong {\bb Z}$. %By the above, the Green-Voisin Density Criterion is fulfilled hence $\NL(\mr B)$ is dense in $\mr B$, whereas 
%Density of $\NL(\mr B({\bb R}))$ in $\mr B(\bb R)$ may fail: 
The situation over $\RR$ turns out to be more delicate than the situation over $\CC$. Indeed, there exists a component $K$ of the real moduli space of degree four surfaces in $\PP^3_\RR$, such that every $S$ in $K$ satisfies $\Pic(S) \cong \ZZ$. %there exist degree four surfaces in $\bb P^3_{{\bb R}}$ whose real locus is a union of $10$ spheres, and for any such a surface $S$, one has $\Pic(S) = {\bb Z}$. Therefore, $\NL(\mr B({\bb R}))  \cap K = \emptyset$ for any connected component $K$ of surfaces of such a topological type \cite[Rem.1.5]{benoistttt}. 
There is a Green--Voisin criterion over the reals \cite[Proposition 1.1]{benoistttt}, but its hypothesis is more complicated and only applies to one component of $\mr B(\RR)$ at a time. 
% complicated hypothesis, applying to only one component %The remedy is to change the Green-Voisin criterion \cite[Prop.1.1]{benoistttt}, but the new criterion is complicated, and only applies to one component 
%can only applied to one connected component of $\mr B(\RR)$ at a time. 

Let us return to abelian varieties. 
%Because of the difficulties for surfaces in $\PP^3$, it came as a surprise that 
Question \ref{intro:questionabelianscheme} turns out to have a response which, in light of the above difficulties, is remarkably simple. %Indeed, in the case of density of non-simple abelian varieties, the criteria over $\CC$ and over $\RR$ turn out to be \emph{exactly the same}. Indeed, 
In \cite{Colombo1990}, Colombo and Pirola prove the following theorem. Let $
A \to B
$ be a holomorphic family of polarized complex abelian varieties over a connected manifold $B$, and let $S_k(B) \subset B$ be the locus of abelian varieties containing a $k$-dimensional abelian subvariety. If Condition \ref{criterion} in Chapter \ref{ch:density} holds, then %the set $S_k(B) \subset B$, of abelian varieties containing a $k$-dimensional abelian subvariety, 
$S_k(B)$ is euclidean dense in $B$. 

Over $\RR$, the criterion remains the same:
%When considering the analogous situation over $\RR$, the density criterion remains the same:
%We prove: %analogous result over $\RR$:

%Surprisingly, over the real numbers the above criterion remains the same. 
%Over the real numbers, we obtain:
%The following theorem is the real analogue of this result. 

\begin{theorem}[Theorem \ref{th:maindensitytheorem}]  \label{intro:NLtheorem}
Let $A \to B$ be a family of polarized real abelian varieties. If $B$ is connected and Condition \ref{criterion} holds, then $S_k(B(\RR))$ is euclidean dense in $B(\RR)$. 
\end{theorem}

\begin{comment}
as we shall now explain. 

Let \[\left(\phi \colon A \to B, E \in \rm R^2\phi_\ast\ZZ\right)\] be a polarized holomorphic family of complex abelian varieties over a connected manifold $B$. For $t \in B$, the polarization gives an isomorphism $H^{0,1}(A_t) \cong H^{1,0}(A_t)^\vee$. Moreover, the dual of the differential of the period map induces a symmetric bilinear form 
$
q: H^{1,0}(A_t) \otimes H^{1,0}(A_t) \to (T_tB)^\vee. 
$
%The following definition will lead to an interesting density criterion.
%The right density criterion of non-simple abelian varieties in the complex case. 
%We will consider the following:

\begin{condition}  \label{colombocriterion} 
There exists an element $t \in B$ and a $k$-dimensional complex subspace $W \subset \rm H^{1,0}(A_t)$ such that the complex $0 \to  \bigwedge^2 W \to W \otimes \rm H^{1,0}(A_t) \to (T_tB)^\vee $ is exact.
\end{condition}
\end{comment}

%See the introduction to Chapter \ref{ch:density} for what we mean by \emph{of real abelian varieties}. %(it is the analytic analogue of a polarized abelian scheme over $\RR$). 
\noindent
To give some applications of Theorem \ref{intro:NLtheorem}, consider $\ca A_g(\RR)$, the moduli space of principally polarized abelian varieties of dimension $g$ over ${\bb R}$, and $\ca M_g(\RR)$, the moduli space of smooth, proper and geometrically connected curves of genus $g$ over $\RR$. These moduli spaces carry natural topologies for which the Torelli map $t \colon \ca M_g(\RR) \to \ca A_g(\RR)$ is continuous, see Section \ref{intro:sub:complexandrealmodulispaces} above. 
%(see also Section .). 
%that compare with the classically defined topologies (Example \ref{introexample:grossharris}) via Theorems \ref{th:homeomorphismmoduli} and \ref{th:homeomorphismmoduli2}. 
%. Finally, let $\ca T_g(\RR) \subset \ca A_g(\RR)$ be the Torelli locus. This is the image of the Torelli morphism $\ca M_g(\RR) \to \ca A_g(\RR)$, see Example \ref{introexample:grossharris}. 

\begin{theorem}[Theorem \ref{theorem2}]  \label{abeliancyclecorollari}
\begin{enumerate}
\item
Let $k<g$ be positive integers, and consider the set $S_k(\ca A_g(\RR)) \subset \ca A_g(\RR)$ of moduli points of real abelian varieties containing a real abelian subvariety of dimension $k$. Then $S_k(\ca A_g(\RR))$ is dense in $\ca A_g(\RR)$. %Moreover, % are dense in $\ca A_g(\RR)$. For 
\item Suppose in addition that $k \leq 3 \leq g$, and let $\ca T_g(\RR) = t(\ca M_g(\RR))$ be the Torelli locus in $\ca A_g(\RR)$. The set $S_k(\ca A_g(\RR)) \cap \ca T_g(\RR)$ is dense in $\ca T_g(\RR)$. 
%Let $S_k(\ca M_g(\RR)) \subset \ca M_g(\RR)$ be the set of real algebraic curves $C$ admitting a map $\varphi: C \to A$ with $A$ a $k$-dimensional abelian variety over ${\bb R}$ such that $\varphi(C(\CC))$ generates $A(\CC)$. If $g \geq 3$ and $k \in \{1,2,3\}$, then $S_k(\ca M_g(\RR))$ is dense in $\ca M_g(\RR)$. 
\item Let $V \subset \bb P\rm H^0(\bb P^2_{{\bb R}}, \OO_{\bb P^2_{{\bb R}}}(d))$ be the moduli space of smooth degree $d \geq 3$ real plane curves. Let $S_1(V)$ be the set of $t \in V$ such that the corresponding curve $C_t$ admits a non-constant map $C_t \to E$ to a real elliptic curve $E$. Then $S_1(V)$ is dense in $V$. 
\end{enumerate}
\end{theorem}
\noindent
The proof of Theorem \ref{intro:NLtheorem} consists of a detailed study of the differential of the period map for abelian varieties, and its compatibility with several Galois actions (see diagram (\ref{Skdiagram})). 
% which turns out to behave very well with respect to the $G$-actions that arise when a family of abelian varieties is defined over $\RR$ (remember that $G = \Gal(\CC/\RR)$). 
Theorem \ref{abeliancyclecorollari} arises quite naturally as a corollary, due to the results in \cite{Colombo1990} and the fact that the density criteria over $\RR$ and $\CC$ are the same. %we verify the criterion provided by Theorem \ref{intro:NLtheorem}, which was essentially already done in \cite{Colombo1990}. 
%Note the way in which moduli and cycles interact with each other in Theorems \ref{intro:NLtheorem} and \ref{abeliancyclecorollari}. We will see more examples of this in Sections 1.3.3 and 1.3.4 below. 
%which was essentially already done by Colombo and Pirola: their criterion is the same. 
%We see how tightly the theories of moduli and real algebraic cycles are related sometimes. 
% thus harder to fulfill, and only implies density of $\NL(\mr B({\bb R}))$ in one component of $\mr B(\bb R)$ at a time.
%Theorem \ref{colpol} adapts this result to a polarized variation of Hodge structure of weight $1$ (which is nothing but a  of complex abelian varieties). The result is a criterion for the density of the locus where the fiber admits a sub-Hodge structure of dimension $k$. 
%To be more precise, recall that for a complex manifold $U$ and a rational weight $2$ variation of Hodge structure $(H_{\bb Q}^2, \ca H, F^1, \nabla)$ on $U$, the \textit{Noether-Lefschetz locus} $\tn{NL}(U) \subset U$ is the locus where the rank of the vector space of Hodge classes is bigger than the general value. If $H_{\bb Q}^2$ is polarizable then $\tn{NL}(U)$ is a countable union of closed algebraic subvarieties of $U$ \cite{MR1273413}. The Green-Voisin Density Criterion referred to above decides whether $\tn{NL}(U)$ is dense in $U$. I

\subsection{Real moduli of five points on the line\hfill (based on \cite{degaay-glueing})}
\label{intro:sub:fivepoints}

%A crucial difference between the complex and the real situation stands out. Indeed, 

%Although this question will be phrased rather vaguely, our second main result will provide a precise answer to it. %to which our second main result will provide an explicBefore stating our second main result, let us consider a question that will be more general
Consider the following question in real moduli theory, phrased informally but whose idea should be clear. Consider a smooth moduli stack $\ca M$ of smooth projective varieties over $\RR$. On the one hand, even if $\ca M(\CC)$ is connected, this may not be true for $\ca M(\RR)$. In general, the $G$-equivariant diffeomorphism type of $X(\CC)$ of real varieties $X$ in one component of $\ca M(\RR)$ remains the same by Ehresmann's theorem, 
%where the complex moduli space is connected, this is often not true for the real moduli space. 
whereas varieties in different components of $\ca M(\RR)$ may have non-homeomorphic real loci. For example, the moduli space $\ca A_1(\RR)$ has two connected components $K_i = \RR^\ast$ $(i \in \{1, 2\})$ such that $K_i$ parametrizes elliptic curves whose real locus is the disjoint union of $i$ circles. If $\ca M(\RR)$ is not connected, real period maps have to be defined on each component of $\ca M(\RR)$ separately (see Section \ref{intro:sub:periodmaps} above). %To get a connected moduli space, one must often allow singularities; 

On the other hand, it is often possible to define a slightly larger moduli stack $\overline{\ca M} \supset \ca M$ by allowing mild singularities, whose real locus $\overline{\ca M}(\RR)$ \emph{is} connected. In this case, $\overline{\ca M}(\RR)$ glues together the various components of $\ca M(\RR)$ in a natural way.\begin{question}\label{ACTquestion} Can the real period domains and real period maps be glued as well?
\end{question}
%To get a connected moduli space, we can allow some singularities . This embeds $\ca M(\RR)$ inside some larger moduli space $\overline{\ca M}(\RR)$, which may in fact happen to be connected.   
%correspond to the different topological types of real algebraic varieties. 
%In the elliptic curve cases, for example, there is one connected component parametrizing elliptic curves $E_{/\RR}$ whose real locus $E(\RR)$ is diffeomorphic to the circle $\bb S^1$, while the other parametrizes curves $E$ with $E(\RR) \cong \bb S^1 \sqcup \bb S^1$.  %the other component parametrizes curves with real locus the disjoint union of two circles. 
%More generally, for principally polarized abelian varieties of dimension $g$ there are $\left \lfloor{(3g+2)/2}\right \rfloor $ connected components of the moduli space, see \cite{grossharris}. 
%Let $\overline{\ca M}$ be a moduli stack over $\RR$, and suppose that the period map identifies $\overline{\ca M}(\CC)$ with the quotient $\Gamma \setminus \ca D$ of a period domain $\ca D$ by a discrete group of automorphisms $\Gamma$. Using $G$-equivariant Torelli theorems, one may wonder if one can uniformize the real moduli space $\overline{\ca M}(\RR)$ as well. 
\noindent
The goal of this section is to explain why, in the case of binary quintics, the answer is \emph{yes}. As for smooth quintics, we have the following (c.f. Examples \ref{realoccultexamples}):  
%We will first consider the moduli space of smooth binary quintics over $\RR$. 
%Before we explain more precisely our results in this direction, let us first recall the classical notions of period maps and period domains, and then introduce the notion of real period maps and real period domains. In the case of binary quintics, we prove:
%In the case of smooth binary quintics, we have:

\begin{theorem}[Theorem \ref{th:theorem01}] \label{INTROtheorem:smoothquintics}
Let $\ca M_0(\RR)$ be the real moduli space of smooth binary quintics. For $i \in \{ 0,1,2\}$, let $\mr M_{i}$ be the component of $\ca M_0(\RR)$ parametrizing quintics with $2i$ complex and $5 - 2i$ real points. For each $i$, the real period map induces an isomorphism of real analytic orbifolds $\mr M_i \cong P \Gamma_i \setminus \left(\RR H^2 - \mr H_i \right)$. Here $\RR H^2$ is the real hyperbolic plane, $\mr H_i$ a union of geodesic subspaces in $ \RR H^2$ and $P\Gamma_i$ an arithmetic lattice in $\textnormal{PO}(2,1)$. %Moreover, 
\end{theorem}
\noindent
The period map for smooth binary quintics referred to Theorem \ref{INTROtheorem:smoothquintics} is actually an occult period map (see Section \ref{intro:sub:periodmaps}). It is defined by sending a binary quintic $X = \{F = 0 \} \subset \PP^1_\RR$ to the $G$-equivariant Hodge structure with $\ZZ[\zeta_5]$-action $\rm H^1(C(\CC), \ZZ)$ attached to the degree five cover $C \to \PP^1_\RR$ ramified along $X$.

To explain the generalization of Theorem \ref{INTROtheorem:smoothquintics} to moduli of stable quintics, let us consider again a general inclusion of moduli stacks $\ca M \subset \overline{\ca M}$ over $\RR$ as above. %The moduli space $\overline{\ca M}(\RR)$ of real varieties with at most mild singularities may happen to have a well-behaved counterpart $\overline{\ca M}(\CC)$ over the complex numbers, in the sense that the
Suppose that a suitable (occult) period map on ${\ca M}(\CC)$ extends to $\overline{\ca M}(\CC)$ identifying the latter with an arithmetic quotient of a hermitian symmetric domain. %and for the above reason, connected over the real numbers. 
Examples are given by cubic surfaces, configurations of points on $\PP^1$, and K3 surfaces \cite{ACTsurfaces, DeligneMostow, engelcompactmoduli}. If this happens, then on the one hand, $\overline{\ca M}(\CC)$ is a locally symmetric variety, and on the other, the topology of $\overline{\ca M}(\RR)$ reveals how one type of real variety deforms into another when it crosses the boundary between two connected components of ${\ca M}(\RR)$. An optimistic approach to positively answer Question \ref{ACTquestion} would be to try to prove that the metric induced on $\ca M(\RR)$ via the period map extends to a path metric on the larger moduli space $\overline{\ca M}(\RR)$, complete in case $\overline{\ca M}(\CC)$ is complete. This may seem like a highly non-trivial thing to do. 

A beautiful theorem by Allcock, Carlson and Toledo \cite{realACTsurfaces} says that for cubic surfaces, this can in fact be done. %says that if $\ca M$ (resp. $\overline{\ca M}$) is the stack of smooth (resp. stable) cubic sufaces, such an extension to $\overline{\ca M}(\RR)$ of the metric on $\ca M(\RR)$ does indeed exist \cite{realACTsurfaces}. 
As a consequence, the real moduli space of stable cubic surfaces is homeomorphic to a hyperbolic quotient space $P\Gamma_\RR \setminus \RR H^4$. Here $\RR H^4$ denotes hyperbolic $4$-space, $P\Gamma_\RR \subset \text{PO}(4,1)$ is a discrete subgroup of isometries, and $P\Gamma_\RR \setminus \RR H^4$ contains the disjoint union of the connected components $P\Gamma_i \setminus \left( \RR H^4 - \mr H_i \right)$ of the moduli space of real smooth cubic surfaces in its interior. In \cite{realACTnonarithmetic}, they establish the analogous result for moduli of stable binary sextics. 
Apart from these two examples, no other real moduli stacks where known to admit a positive answer to Question \ref{ACTquestion}, before we proved our next result. 
%Before our next result, other examples of this phenomenon had not been known. 

Let $\ca M_s(\RR)$ be the real moduli space of stable binary quintics. We prove (see Theorem \ref{th:theorem02}) that there exists a complete hyperbolic metric on $\ca M_s(\RR)$ that restricts to the metrics on $\mr M_i$ induced by Theorem \ref{INTROtheorem:smoothquintics}. With respect to it, $\ca M_s(\RR)$ is isometric to the hyperbolic triangle of angles $\pi/3, \pi /5, \pi/10$. 

To put it differently: if we define \[P\Gamma_{3,5,10} =  \langle \alpha_1, \alpha_2, \alpha_3 | \alpha_i^2 = (\alpha_1\alpha_2)^3 = (\alpha_1\alpha_3)^5 = (\alpha_2\alpha_3)^{10} = 1 \rangle,\]
then the following holds true. (See Figure \ref{fig:triangle} for a visualization of this theorem.)
% visualization of $\overline{\mr{M}}_\RR$ as a hyperbolic triangle, see 
\begin{theorem}[Theorem \ref{th:theorem02}] \label{introth:theorem02}
There is an open embedding $\coprod_i P \Gamma_i \setminus \left(\RR H^2 - \mr H_i \right) \subset P\Gamma_{3,5,10} \setminus \RR H^2$ of hyperbolic orbifolds and a homeomorphism $\ca M_s(\RR) \cong P\Gamma_{3,5,10} \setminus \RR H^2$ that extend the orbifold isomorphism $\ca M_0(\RR) \cong \coprod_i P \Gamma_i \setminus \left(\RR H^2 - \mr H_i \right) $ of %three orbifold isomorphisms $\mr M_i \cong P \Gamma_i \setminus \left(\RR H^2 - \mr H_i \right)$ of 
Theorem \ref{INTROtheorem:smoothquintics}. 
 %where $P\Gamma_\RR$ is isomorphic to the group $\langle 
\end{theorem}
\noindent
The proof of Theorem \ref{introth:theorem02} is inspired by the proof of the analogous theorem for cubic surfaces in \cite{realACTsurfaces}. In fact, to prove their real uniformization result cited above, Allcock--Carlson--Toledo use their previous result that the complex moduli space of stable cubic surfaces is isomorphic to a four-dimensional ball quotient \cite{ACTsurfaces}. Likewise, we have to prove (see Section \ref{complexball}) that the complex moduli space of stable binary quintics is isomorphic to a two-dimensional ball quotient. Then our proof diverges. Allcock--Carlson--Toledo carry out a local calculation in the real moduli space, to prove that the metrics on the components glue along the discriminant to a complete metric on the space of stable surfaces. Instead, we carry out this local calculation on the ball quotient side. %we observe that their glueing construction is actually a particular case of a more general phenomenon. 

To do so, it seemed natural not to restrict our attention to this two-dimensional ball quotient, but to consider general unitary Shimura varieties instead. In Chapter \ref{ch:glueing}, we develop a method of glueing real hyperbolic quotient spaces in this context. 
% in the more general framework of unitary . These are ball quotients that arise via hermitian lattices
%In fact, it turned out that 
%all this fits very nicely in the more general context of unitary Shimura varieties. 
%The result is , we develop a method of glueing real hyperbolic quotient spaces in Chapter \ref{ch:realperiods}.  a method that fits in the more general framework of unitary Shimura varieties. We develop a method of glueing hyperbolic quotient spaces. 
The input is a hermitian lattice $\Lambda$ of hyperbolic signature over the ring of integers of a CM field; the procedure glues together real ball quotients arising from anti-isometric involutions on $\Lambda$ along a hyperplane arrangement; the output is again a real ball quotient (or a disjoint union of those), assembling the different pieces in a sometimes non-arithmetic way. See Theorem \ref{th:theorem03} for details. %his gives a general method of gluing real arithmetic ball quotients along a hyperplane arrangement, see Chapter \ref{ch:realperiods}; the result is again a real ball quotient (or a disjoint union of those), assembling the different pieces in a sometimes non-arithmetic way. See Theorem \ref{th:theorem03}. 
%The result carries a complete hyperbolic metric, thus is a disjoint union of hyperbolic quotient spaces. 
%, generalizing the work of Allcock, Carlson and Toledo cited above. 
%\cite{realACTsurfaces}, %\cite{realACTbinarysextics}, %\cite{realACTnonarithmetic}. 
%The resulting space is a disjoint union of real ball quotients, see Theorem \ref{th:theorem03}. 
For the right choice of $\Lambda$ (cohomology of a ramified cover of projective space, c.f. Section \ref{intro:sub:periodmaps}), one %For some moduli stacks of hypersurfaces, one can then apply this glueing construction (Theorem \ref{th:theorem03}) to the hermitian lattice $\Lambda$ that arises as the cohomology of the cover of projective space ramified along a member of the moduli space. One 
retrieves the main theorems of \cite{realACTnonarithmetic, realACTsurfaces} and obtains the new Theorem \ref{introth:theorem02}. 

\subsection{The integral Hodge conjecture for one-cycles on Jacobian varieties  \hfill (c.f. \cite{degaay-thorsten})
 \\ \textcolor{white}{.} \hfill (with \textsc{Thorsten Beckmann})}\label{intro:sub:theintegralhodgeconjectureforjacobianvarieties}

%In our next result, we somewhat 
%change directions, and consider curves on complex abelian varieties. %Our third result concerns curves on complex abelian varieties. 

Next, we consider curves on complex abelian varieties. At first glance it may seem that we are changing directions: the goal is to study conditions under which a complex abelian variety $A$ of dimension $g$ contains sufficiently many curves to generate its group of integral Hodge classes of degree $2g-2$. As such, the nature of this project is discrete rather than continuous; the members of any irreducible family of curves in $A$ all give the same cohomology class in $\rm H^{2g-2}(A(\CC), \ZZ)$. However, it is only a study of the analytic structure of $\ca A_g(\CC)$ that allows us to prove density of abelian varieties satisfying the above conditions. 

Before getting into details, let us discuss what is known about the integral Hodge conjecture in general. IHC$_i$ holds for every smooth projective $X$ if $i \in \{0, \dim(X)\}$ (trivial), or $i = \dim(X)-1$ (Lefschetz $(1, 1)$). %For one-cycles, it is a birational invariant. 
The first counterexamples were provided by torsion classes; the obstructions were topological. Non-torsion non-algebraic Hodge classes can be found on very general hypersurfaces of degree $d$ in $\PP^4_\CC$ for certain $d$. %Conjecturally, however, this bound is not sharp.  
%that any $d \geq 6$ should yield a counterexample 
%\cite[32]{griffithsharrisnoetherlefsch}. %: indeed,
%Griffiths and Harris 
% conjecture that if δ ≥ 6, the degree of any curve
%on a very general hypersurface of degree δ in P 4 C is a multiple of δ.
%By blowing up P d C along a very general hypersurface of degree δ in P 4 C ⊂ P d C , onesoulevoisin
IHC$_{n-2}$ can fail for %rational varieties if $k \in \{2, \dotsc, n-3\}$, and for griffithsharrisnoetherlefschintegralhodgetate
rationally connected varieties of dimension $n$, already for $n = 4$. Conjecturally, these varieties satisfy do satisfy IHC$_1$, and this is true in dimension three \cite{voisin_someaspectsofthehodgeconjecture, atiyahintegralhodge, totarocobordism, trento, colliotthelenevoisin, stabalyirrationalschreieder, voisin}. 

We conclude that IHC may hold if one imposes restrictions on the geometry of the variety $X$. Along the same lines, it is interesting to consider smooth projective threefolds $X$ of Kodaira dimension zero. If such a threefold $X$ satisfies the condition $h^0(X, K_X) > 0$, then $X$ satisfies the integral Hodge conjecture by work of Grabowski, Totaro and Voisin \cite{grabowski,voisinintegralhodge, totaroIHCthreefolds}; this condition on $h^0(X, K_X)$ is necessary by Benoist--Ottem \cite{benoistottem}. About the integral Hodge conjecture for one-cycles on higher-dimensional varieties $X$ with $K_X = 0$, not much seems to have been known. 
\newpage
\noindent
This leads us to the starting point of this project, carried out jointly with Thorsten Beckmann. What can be said about one-cycles on abelian varieties? In our investigation of this question, we build on some results of Grabowski, who proved in his thesis \cite{grabowski} that for a positive integer $g$ the following are equivalent:
%
% To investigate such a problem, it is natural to start with abelian varieties. We this project together with Thorsten Beckmann. 
\begin{enumerate}
\item \label{grabowski1} Every complex abelian variety of dimension $g$ satisfies the integral Hodge conjecture for one-cycles. 
\item \label{grabowski2}For every principally polarized complex abelian variety $(A, \theta)$ of dimension $g$, the integral cohomology class $\theta^{g-1}/(g-1)! $ is algebraic. 
\end{enumerate}
\noindent
Having established this, Grabowski proves the integral Hodge conjecture for complex abelian threefolds by noting that every principally polarized complex abelian threefold is a product of Jacobians ($\dim(\ca M_g) = \dim(\ca A_g)$ for $g = 3$). For $g > 3$, we have $\dim(\ca M_g) < \dim(\ca A_g)$, %The latter no longer holds in higher dimensions, 
making it hard to generalize his proof. 

In fact, the problem is that condition \ref{grabowski2} above is a condition on \emph{every} abelian variety in the moduli space $\ca A_g(\CC)$. We take the different approach of fixing \emph{one} abelian variety $A$, and study the integral Hodge conjecture for one-cycles on $A$. 
%\begin{question} \label{introquestion:IHC}
%Let $A$ be a principally polarized complex abelian variety. Under which condition does $A$ satisfy the integral Hodge conjecture for one-cycles? 
%\end{question}
%Before we try to formulate an answer to this question, let us shortly discuss what is known about the integral Hodge conjecture in general. 
%, and what are the main conjectures. 
%it fails in general. %For all other values of $k$ and $n$, it fails in 
%We established the integral Hodge conjecture for one-cycles for a large class of abelian varieties. %It turns out that the above question has a positive response for the principally polarized abelian variety $(A, \theta)$ under consideration, provided that its minimal cohomology class is algebraic. 

\begin{theorem}[Theorem \ref{maintheorem}, with T. Beckmann] \label{introth:onecycles}
Let $(A, \theta)$ be a principally polarized complex abelian variety of dimension $g$, and suppose that the Hodge class 
\[
\frac{\theta^{g-1}}{(g-1)!} \in \rm H^{2g-2}(A(\CC), \ZZ)
\]
is algebraic. Then $A$ satisfies the integral Hodge conjecture for one-cycles. 
\end{theorem}
\noindent
The proof of Theorem \ref{introth:onecycles} is inspired by Grabowski's approach to use Fourier transforms. The idea is quite simple, and goes as follows. On cohomology, the Fourier transform of an abelian variety preserves integral classes, so it is natural to ask whether the same holds for the Fourier transform on Chow groups. Suppose that, for some complex abelian variety $A$, this is indeed the case. That is, there exists a homomorphism $\ca F_{\wh A} \colon \CH(\wh A) \to \CH(A)$ that commutes with the cycle class map and the Fourier transform on integral cohomology. Because the latter is an isomorphism by \cite{beauvillefourier}, one may lift any degree $2g-2$ Hodge class $\alpha$ on $ A$ to a degree two Hodge class on $\wh A$. By Lefschetz (1,1), the latter is the class of a line bundle, which $\ca F_{\wh A}$ maps to a one-cycle in $\CH(A)$ lying above $\alpha$. 

The assumption on which this argument rests, is a compatibility of Fourier transforms with integral Chow groups of abelian varieties. We study this compatibility in detail in Chapter \ref{ch:integralfourier}, developing a theory of \emph{integral Fourier transforms}. The research in Chapters \ref{ch:integralfourier} and \ref{ch:onecycles} was carried out jointly with Thorsten Beckmann. 
%The framework of \emph{integral Fourier transforms} we develop works over any field, which is why it gives  
%, by \emph{integral Fourier transforms}. 

As a first corollary of Theorem \ref{introth:onecycles}, note that it applies to Jacobian varieties, for which the minimal class $\theta^{g-1}/(g-1)!$ is well known to be algebraic: 

\begin{theorem}[Theorem \ref{introth:IHCforjacobians}, with T. Beckmann]
Let $C_1, \dotsc, C_n$ be smooth projective curves over $\CC$. Then $A = \prod_{i = 1}^n J(C_i)$ satisfies the integral Hoge conjecture for one-cycles. 
\end{theorem}

%The fact that this answers the above question for a large subset of the moduli space $\ca A_g(\CC)$, is a consequence of the following corollary of Theorem \ref{introth:onecycles}. 
\noindent
It is also classical that \emph{Hecke orbits} are dense for the euclidean topology of $\ca A_g(\CC)$. These are isogeny orbits of polarized abelian varieties in their moduli space (see e.g.\ \cite{chaiordinaryhecke, oorthecke}). Since one can control the degree of these isogenies, one obtains:

\begin{theorem}[Theorem \ref{introth:density}, with T. Beckmann]
%Let $\overline{\ca T_g(\CC)} \subset \ca A_g(\CC)$ be the closed Torelli locus. Then $\textnormal{IHC}_1$ holds for every principally polarized complex abelian variety in $\overline{\ca T_g(\CC)}$. Moreover, 
Principally polarized abelian varieties satisfying the integral Hodge conjecture for one-cycles are euclidean dense in $\ca A_g(\CC)$. 
%Let $C_1, \dotsc, C_n$ be smooth projective curves over $\CC$. The product of Jacobians $A = \prod_{i = 1}^n J(C_i)$ satisfies the integral Hodge conjecture for one-cycles. 
\end{theorem}
\noindent
Our machinery on integral Fourier transforms (Chapter \ref{ch:integralfourier}) has similar implications for cycles on abelian varieties in positive characteristic, see Theorems \ref{introth:integraltate} and \ref{th:chinglichai}. Naturally, the next thing to do was to consider cycles on abelian varieties over $\RR$. 

\subsection{The real integral Hodge conjecture for abelian threefolds\hfill (based on \cite{degaay-RIHC})}
\label{intro:sub:therealintegralhodgeconjectureforabelianthreefolds}

Let us have a closer look at Grabowski's elegant proof of the integral Hodge conjecture for complex abelian threefolds. For such a threefold $A$, the Fourier transform $\mr F_A$ defines a natural isomorphism $\Hdg^2(A(\CC), \ZZ) \cong \Hdg^4(\wh A(\CC), \ZZ)$. By Lefschetz $(1,1)$, one is reduced to checking whether $\mr F_A$ sends line bundles to classes of one-cycles -- for this, one reduces to the line bunde $\ca L = \ca O_{J(C)}(\Theta)$ attached to the theta divisor $\Theta$ on a Jacobian variety $J(C)$. Since the Fourier transform $\mr F_{J(C)}$ sends $c_1(\ca L)$ to the class of the curve $C$ inside $J(C)$, this is enough to conclude. 

It is natural to ask whether the same line of arguments works over $\RR$. 

\begin{question} \label{introquestion:realHC}
Do real abelian threefolds satisfy the real integral Hodge conjecture? 
\end{question}
\noindent
We remark that the real integral Hodge conjecture for $k$-cycles holds for every real variety $X$ for $k \in \{0, \dim(X)\}$ by \cite{BW20}, and for $k = \dim(X) - 1$ by \cite{krasnovcharact, vanhamel, mangoltehamel}. Thus, to answer Question \ref{introquestion:realHC}, one needs only to consider one-cycles. In general, it is an open problem whether uniruled or Calabi-Yau threefolds satisfy $\RR$IHC$_1$ \cite[Question 2.16]{BW20}. There are partial results in the uniruled case \cite{BW21}, but for real Calabi-Yau threefolds, not much had been known.   %Although there are several steps in Grabowski's proof that cannot be adapted over $\RR$, we managed to obtain several results seem to indicate a positive answer to this question. 

We prove that at least modulo torsion, the answer to Question \ref{introquestion:realHC} is \emph{yes}. 
%such a variety does indeed satisfy the real integral Hodge conjecture. 

\begin{theorem}[Theorem \ref{theorem1} and Corollary \ref{IHCforconnected}] \label{introth:modtors}
Let $A$ be an abelian threefold over $\RR$. The following map  surjective:
\[
\CH_1(A) \to \Hdg_G^{4}(A(\CC), \ZZ(2))_0/(\textnormal{torsion}) = \Hdg^{4}(A(\CC), \ZZ(2))^G.
\]
In particular, if $A(\RR)$ is connected, then $A$ satisfies the real integral Hodge conjecture. 
%is surjective. %In other words, $A$ satisfies the real integral Hodge conjecture mo
\end{theorem}
\noindent
The proof of Theorem \ref{introth:modtors} uses a real analogue of the equivalence of statements \ref{grabowski1} and \ref{grabowski2} in Section \ref{intro:sub:theintegralhodgeconjectureforjacobianvarieties} (see Theorem \ref{grabowski}), and a reduction to the Jacobian of a curve with non-empty real locus. %$J(C)$ of a real curve $C$ with $C(\RR) \neq \emptyset$. 
The latter rests crucially on density of certain Hecke orbits in the moduli space $\ca A_g(\RR)$, a result which we prove in Theorem \ref{density}. The same density statement implies that Question \ref{introquestion:realHC} in the principally polarized case reduces to the case of Jacobians (see Theorem \ref{reduction}). %We also prove that $\RR$IHC for principally polarized real abelian threefolds would follow from $\RR$IHC for Jacobians of genus three curves with non-empty real locus. 

Finally, we manage to say something about torsion cohomology classes of degree four on real abelian varieties of any dimension $g$. The real integral Hodge conjecture for codimension-two cycles predicts that these classes are algebraic, provided that they satisfy a certain topological condition (\ref{topologicalcondition}) introduced in Section \ref{sec:realintegralhodge}. 

To state this result, we need the following fact: for any projective variety $X$, smooth over $\RR$, there exists a canonical filtration $F^\bullet$ on $\Hdg^{4}_G(X(\CC),\ZZ(2))_0$ coming from a certain spectral sequence, the \emph{Hochschild-Serre} spectral sequence (see (\ref{hochschild})). %In the case of abelian varieties, we obtain:

\begin{theorem}[Theorem \ref{fourierreduction}]
Let $A$ be an abelian variety over $\RR$. The group \\ $F^3\Hdg^4_G(A(\CC), \ZZ(2))_0$ is zero and the group $F^2\Hdg^4_G(A(\CC), \ZZ(2))_0$ is algebraic. 
\end{theorem}

Combined with Theorem \ref{introth:modtors}, this theorem has two notable consequences: 
\begin{enumerate}
\item (Theorem \ref{reduction} and Corollary \ref{questionreductioncorollary}). The real integral Hodge conjecture for principally polarized abelian threefolds is equivalent to the surjectivity of the Abel-Jacobi map
\begin{align*} %\label{abeljacobi-reduction}
\CH_1(J(C))_{\textnormal{hom}} \to \rm H^1(G, \rm H^3(J(C)(\CC), \ZZ(2)))_0
\end{align*}
for the Jacobian $J(C)$ of every real algebraic curve $C$ of genus three whose real locus $C(\RR)$ is non-empty. 
\item (Corollary \ref{abeliansurface}). If $A = B \times E$ is the product of a real abelian surface $B$ and a real elliptic curve $E$ whose real locus $E(\RR)$ is connected, then $A$ satisfies the real integral Hodge conjecture.  
\end{enumerate}

\cleardoublepage

\part{Real moduli spaces}
\cleardoublepage
% Note that depending on your settings in the table of contents, subsections and subsubsections might appear virtually identical.
% Make sure to set the ToC depth and the numbering depth in the ToC the way you want.
\chapter{Real moduli spaces}\label{ch:realmodulispaces}

\section{Introduction}

%\textcolor{red}{Change this introduction accordingly.}
%In the second part of this paper we put our density results in some perspective. 

%\begin{comment}
%In order to understand the behavior of the geometry of algebraic varieties in family, it can be useful to understand the topology of their moduli space. When the base field is the field of complex numbers, defining a natural topology on a moduli space is a well-understood problem. It can be attacked either with analytic methods such as Teichm\"uller theory, or with algebraic methods such as stacks and GIT. Well-known examples are the moduli spaces of abelian varieties and curves. In both cases, the moduli space is uniformized in two different ways: as a complex analytic quotient of a simply connected manifold by a discrete group of holomorphic automorphisms, and as a GIT quotient of a suitable subscheme of a Hilbert scheme of some projective space by the group of projective $\ZZ$-linear transformations. The period map identifies the analytification of the former with the latter quotient and the two theories are coherent. 
%\end{comment}
%Over the real numbers, the situation for abelian varieties and curves is as follows. 

Let $\va{\ca A_g(\RR)}$ be the set of isomorphism classes of principally polarized abelian varieties of dimension $g$ over $\RR$. Similarly, define $\va{\ca M_g(\RR)}$ as the set of isomorphism classes of proper, smooth and geometrically connected curves of genus $g$ over $\RR$. In \cite{grossharris, seppalasilhol2}, Gross, Harris, Sepp\"al\"a and Silhol provide these spaces with a real semi-analytic structure\footnote{These structures are \emph{not} real-analytic, as opposed to what is claimed in \emph{loc.cit.} -- see \cite{MR1705976}. Moreover, one has to be aware of the error in the definition of $\Gamma_H$ on page $180$ of \cite{grossharris}.}. 
%The $\va{\ca A_g(\RR)}$  has a real semi-analytic structure, as proven by Silhol \cite[Ch.IV, Section 4]{silholsurfaces}
%Similarly, in \cite{seppalasilhol2}, Sepp\"al\"a and Silhol provide $\va{\ca M_g(\RR)}$with a real semi-analytic structure. In fact, the construction shows that 
In both cases, the real moduli space is a disjoint union of quotients $\Gamma \setminus M$ of a real-analytic manifold $M$ by a properly discontinuous action of a discrete group $\Gamma$. Thus, the sets $\va{\ca A_g(\RR)}$ and $\va{\ca M_g(\RR)}$ carry the structure of topological space, semi-analytic variety, and real-analytic orbifold in a natural way. 
%these topological spaces underlie an orbifold structure, whose stabilizers correspond to the automorphism groups 

This seems the right analytic approach to construct the real moduli spaces of abelian varieties and curves. %However, due to the fact that some choices have been made, it may not be the most natural way to do so. 
Is there a general approach to construct a suitable real moduli space, starting from an algebraic moduli stack over $\RR$? If so, how does it compare to the above analytic construction of the real moduli spaces of abelian varieties and curves? The goal of this chapter is two answer these two questions. 
%Studying the topology of a moduli space can be useful to understand 
%can be studied to teach us something about the geometry of real abelian varieties and curves. 
%above-mentioned works of Gross, Harris, Sepp\"al\"a and Silhol imply in particular
%the set $\va{\ca N_g(\RR)}$ of real genus $g$ algebraic curves with a topology as follows
%Gross-Harris and Silhol have provided
%In Gross-Harris \cite[Section 9]{grossharris} and Silhol \cite[Ch.IV, Section 4]{silholsurfaces}.

%It turns out that the above topologies on $\mr A_g^{\RR}$ and $\mr M_g^{\RR}$ are a particular instance of a more general construction. % a functorial topology on the set of isomorphism classes of real points of any real algebraic stack. 
We will consider an algebraic stack $\mr X$ locally of finite type over $\RR$, and define a topology on the set $|\mr X(\RR)|$ of isomorphism classes of $\mr X(\RR)$ which agrees the euclidean topology on $\mr X(\RR)$ when $\mr X$ is a scheme. If $\mr X$ admits a coarse moduli space $\mr X \to M$, then $|\mr X(\RR)|$ should be thought of as the real analogue of $M(\CC)$ with its euclidean topology. The point is that we cannot use $M(\RR)$, since this set may not be in bijection with $|\mr X(\RR)|$. We show that the space $|\mr X(\RR)|$ underlies an orbifold structure if $\mr X$ is a smooth separated Deligne-Mumford stack. We finish the chapter by verifying that if $\mr X$ is the moduli stack of abelian varieties $\ca A_g$, or the stack of curves $\ca M_g$, then the topology on $|\mr X(\RR)|$ coincides with the topology arising via the analytic approach of Gross--Harris--Sepp\"al\"a--Silhol mentioned above. 
%of abelian varieties and curves, the canonical bijections $|\ca A_g(\RR)| \cong \mr A_g^{\RR}$ and $|\ca M_g(\RR)| \cong \mr M_g^{\RR}$ turn out to be homeomorphisms. %, see Theorems \ref{th:homeomorphismmoduli} and \ref{th:homeomorphismmoduli2}.

\section{Real algebraic stacks} \label{sec:coarsemoduli}

In Section \ref{sec:coarsemoduli}, %analyze the topology of the real locus $\mr X(\RR)$ of a real algebraic stack $\mr X$. We 
we equip the set of isomorphism classes $|\mr X(\RR)|$ of the real locus $\mr X(\RR)$ of a real algebraic stack with a topology, using a scheme $U$ and a morphism $U \to \mr X$ which is smooth, surjective and essentially surjective on $\RR$-points. For smooth Deligne-Mumford stacks, this space $|\mr X(\RR)|$ admits an orbifold structure. 

\begin{conventions} \label{notnot}
%In this subsection, we fix a field $F$ which is either real closed or finite. 
For a scheme $S$, an algebraic stack $\mr X$ over $S$ and a scheme $T$ over $S$, we let $|\mr X(T)|$ denote the set of isomorphism classes of the groupoid $\mr X(T)$. A smooth (resp. \'etale) \emph{presentation} of an algebraic stack $\mr X$ is a smooth (resp. \'etale) surjective $S$-morphism $U \to \mr X$, where $U$ is a scheme over $S$. A \emph{real algebraic stack} is an algebraic stack locally of finite type over $\RR$. By the topology on $X(\RR)$ for a scheme $X$ locally of finite type over $\RR$, we mean the real-analytic topology.
\end{conventions}

%there is a canonical orbifold structure on 
%In Section \ref{sec:coarsemoduli}, we consider an algebraic stack $\mr X$ locally of finite type over $\RR$, pick an $\bb R$-surjective smooth presentation $\phi: X \to \mr X$, endow $|\mr X(\bb R)|$ with the quotient topology $\tau$ induced by $X(\bb R) \to |\mr X(\bb R)|$, and prove that $\tau$ does not depend on $\phi$. We also show (see Theorem \ref{rfonto2}) that if $\mr X$ is Deligne-Mumford, then there exists an $\bb R$-surjective \emph{\'etale} presentation $\phi$, which may be used to define $\tau$.

\subsection{The topology of the real locus of an algebraic stack} \label{topology-on-real}

The goal of Section \ref{topology-on-real} is to define a topology on $\va{\mr X(\RR)}$ for any algebraic stack $\mr X$ locally of finite type over $\RR$, in a way that is functorial in $\mr X$, and that generalizes the real-analytic topology on $\mr X(\RR)$ when $\mr X$ is a scheme. %For this, we need a definition and a theorem. 

\begin{definition}[c.f. \cite{Sakellaridis2016TheSS, 2019}]\label{Rontodefinition}
Let $\mr X$ be an algebraic stack over a scheme $S$, and let $K$ be a field equipped with a morphism $\Spec(K) \to S$. A presentation $X \to \mr X$ by an $S$-scheme $X$ is \emph{$K$-surjective} if the map $X(K) \to |\mr X(K)|$ is surjective. 
\end{definition}

\begin{theorem}[Laumon--Moret-Bailly]\label{th:LMB}
Let $\mr X$ be an algebraic stack over a scheme $S$, and let $K$ be a field equipped with a morphism $\Spec(K) \to S$. For any $S$-morphism $x \colon \Spec(K) \to \mr X$, there exists an affine scheme $X$, a smooth morphism $P \colon X \to \mr X$, and an $S$-morphism $z \colon \Spec(K) \to X$ such that $P \circ z \cong x$. 
\end{theorem}

\begin{proof}
See \cite[Théorème (6.3)]{LM-B}. 
\end{proof}

\begin{corollary}\label{cor:LMB}
Let $\mr X$ be an algebraic stack over a scheme $S$, and let $K$ be a field equipped with a morphism $\Spec(K) \to S$. 
%Under the hypothesis of Theorem \ref{th:LMB}, 
There exists a $K$-surjective smooth presentation $U \to \mr X$. 
%such that $U(K) \to \va{\mr X(K)}$ is surjective. 
\end{corollary}
\begin{proof}
Let $U' \to \mr X$ be any smooth presentation. For every $[x] \in |\mr X(K)|$, choose a morphism $x \colon \Spec(K) \to \mr X$ that represents $[x]$, and define $X_x \to \mr X$ as in Theorem \ref{th:LMB}. The morphism $U = \left(\bigsqcup_{[x] \in \va{\mr X(K)}} X_x  \right) \bigsqcup U' \to \mr X$ satisfies the requirements. 
\end{proof}

\begin{proposition} \label{th1}
%Consider two $\bb R$-surjective smooth presentations $P \colon U \to \mr X$ and $Q \colon V \to \mr X$, where $\mr X$ is a real algebraic stack. 
Let $\mr X$ be a real algebraic stack and let $P_i \colon U_i \to \mr X$ for $i  \in \{1,2\}$ be two $\RR$-surjective smooth presentations. Consider the two quotient topologies on $\va{\mr X(\RR)}$ given by the two surjections $P_{i,\RR} \colon U_i(\RR) \to \va{\mr X(\RR)}$. These two topologies are the same. 
%ny other $\RR$-surjective smooth presentation $Q \colon U \to \mr X$ induces the same topology on $\va{\mr X(\RR)}$. 
%The real-analytic topology on $|\mr X(\bb R)|$ does not depend on the choice of $\bb R$-surjective smooth presentation $U \to \mr X$. 
\end{proposition}
\begin{proof}
%Consider $\bb R$-surjective presentations $P \colon U \to \mr X$ and $Q \colon V \to \mr X$. 
The stack $U_1 \times_{\mr X} U_2$ is equivalent to an $\bb R$-scheme $W$, and the projections $\pi_2: W \to U_2$ and $ \pi_1: W \to U_1$ are smooth. The morphisms $\pi_{2, \RR}: W(\bb R) \to U_2(\bb R)$ and $\pi_{1,\RR}: W(\bb R) \to U_1(\bb R)$ are surjective: if $x \in U_2(\bb R)$ then $P(x) \in \mr X(\bb R)$, so by $\bb R$-surjectivity there is a $y \in U_1(\bb R)$ and an isomorphism $\psi: Q(y) \cong P(x)$ in $\mr X(\bb R)$; then $(x,y, \psi)$ defines an object in $(U_1 \times_{\mr X} U_2)(\bb R)$ for which $\pi_{2, \RR}(x,y, \psi) = x$ in $U_2(\RR)$. Since $\pi_{2,\RR}$ and $\pi_{1,\RR}$ are open by Lemma \ref{2} below, the result follows. 
\end{proof}

\begin{lemma} \label{2}
Let $g: X \to Y$ be a morphism of schemes which are locally of finite type over $\bb R$. Let $g_\RR: X(\bb R) \to Y(\bb R)$ be the induced map of real-analytic spaces. The morphism of analytic spaces $g_\RR$ is open if the morphism of schemes $g$ is smooth. 
\end{lemma}
\begin{proof}
We can work locally on $X(\bb R)$; let $U \subset X$ and $V \subset Y$ be affine open subschemes with $g(U)  \subset V$ such that $g|_U = \pi_2 \circ f$, where $f: U \to \bb A^d_V$ is \'etale for some integer $d \geq 0$, and $\pi_2: \bb A^d_V \to V$ is the projection on the right factor \cite[\href{https://stacks.math.columbia.edu/tag/054L}{Tag 054L}]{stacks-project}. Because $X(\bb C) \to X$ is continuous \cite[\S XII, 1.1]{SGA1}, the set $U(\CC)$ is open in $X(\CC)$, thus $U(\RR) = U(\CC) \cap X(\RR)$ is open in $X(\RR)$. Similarly, $V(\RR)$ is open in $Y(\RR)$. 
 % hence $X(\bb R) \to X(\bb C) \to X$ is continuous. 
Finally, the map $U(\bb R) \to V(\bb R)$ is open because it factors as the composition $$U(\bb R) \xrightarrow{f_\RR} \bb R^d \times V(\bb R) \xrightarrow{\pi_{2,\RR}} V(\bb R),$$ where $\pi_{2,\RR}$ is open and $f_\RR$ a local homeomorphism. The lemma follows. 
\end{proof}

\begin{definition}\label{deffer}
Let $\mr X$ be an algebraic stack locally of finite type over $\RR$. The \emph{real-analytic topology} on $|\mr X(\bb R)|$ is the quotient topology on $|\mr X(\bb R)|$ given by the surjection $U(\RR) \to \va{\mr X(\RR)}$ and the real-analytic topology on $U(\bb R)$, for any $\RR$-surjective smooth presentation $U \to \mr X$ (see Definition \ref{Rontodefinition} and Corollary \ref{cor:LMB}).
\end{definition}
\noindent
In the sequel, $\va{\mr X(\RR)}$ will denote the topological space defined in Definition \ref{deffer}. 

\begin{corollary} \label{stackycorollary}
\begin{enumerate}
\item \label{st:one}
The assignment of a topological space to an algebraic stack locally of finite type over the real numbers is functorial.
\item \label{stackyopen} If $\rho \colon \mr X \to \mr Y$ is a smooth morphism \cite[\href{https://stacks.math.columbia.edu/tag/06FM}{Tag 06FM}]{stacks-project} of algebraic stacks locally of finite type over $\RR$ , then the map $\rho_\RR \colon |\mr X(\RR)| \to |\mr Y(\RR)|$ is open. In particular, if $\mr U \subset \mr X$ is an open substack, then $|\mr U(\RR)| \subset |\mr X(\RR)|$ is an open subset. 
\item \label{stacky:coarsecont}
If a real algebraic stack $\mr X$ admits a coarse moduli space $\pi: \mr X \to M$, then the natural map $|\mr X(\bb R)| \to M(\CC)$ is continuous. 
\end{enumerate}
\end{corollary}
\begin{proof}
1. Let $\mr X \to \mr Y$ be a morphism of real algebraic stacks, $V \to \mr Y$ an $\RR$-surjective smooth presentation, and $\mr Z = \mr X \times_{\mr Y} V$. Then $\va{\mr Z(\RR)} \to \va{\mr X(\RR)}$ is surjective. If $U \to \mr Z$ be an $\RR$-surjective smooth presentation, then the composition $U \to \mr Z \to \mr X$ is an $\RR$-surjective smooth presentation of $\mr X$. Statement \ref{st:one} follows. 

%it suffices to note that the pull-back of an $\RR$-surjective \'etale presentation $V \to \mr X$ along any morphism of real algebraic stacks $\mr X \to \mr Y $ gives an $\RR$-surjective \'etale presentation 
2. Consider an open subset $B \subset |\mr X(\RR)|$. Let $\pi \colon V \to \mr Y$ %and $U \to \mr X\times_{\mr Y}V$ 
be an $\RR$-surjective smooth presentation and denote the projection $ \mr X\times_{\mr Y}V \to \mr X$ by $f$. Suppose that $\rho \colon \mr X \to \mr Y$ is the smooth morphism under consideration, and let $\tilde \rho \colon \mr X\times_{\mr Y}V \to V$ be its base-change. We have an equality $$\tilde \rho_\RR(f_\RR^{-1}(B)) = \pi_\RR^{-1}(\rho_\RR(B)) \subset V(\RR).$$
Indeed, this can be deduced from the following commutative diagram of sets:
\[
\xymatrix{
\va{\left( \mr X \times_{\mr Y} V \right)(\RR)}  \ar@{->>}[r] \ar@/^2pc/[rr]_{\tilde \rho_\RR} \ar@{->>}[dr]^{f_\RR} &\va{\mr X(\RR)}  \times_{\va{\mr Y(\RR)}} V(\RR) \ar@{->>}[d] \ar[r]& V(\RR) \ar@{->>}[d]^{\pi_\RR} \\
&\va{\mr X(\RR)} \ar[r]^{\rho_\RR} &  \va{\mr Y(\RR)}. 
} 
\]
Therefore, to prove statement \ref{stackyopen}, it suffices to treat the case $\mr Y = V$ is a scheme. This is clear: if $U \to \mr X$ is any $\RR$-surjective smooth presentation, then the composition $U \to \mr X \to V$ is smooth, hence $U(\RR) \to V(\RR)$ is open by Lemma \ref{2}. 

3. This follows from statement \ref{st:one} and the continuity of $M(\RR) \to M(\CC)$.\end{proof}

\subsection{Pointwise surjective presentations of Deligne-Mumford stacks}  \label{pointwise}

\begin{comment}
The goal of this subsection is to prove that every Deligne-Mumford stack $\mr X$ of finite type over $F$ admits an \'etale presentation by an $F$-scheme $X$ such that the induced map $X(F) \to |\mr X(F)|$ is surjective. See Definition \ref{Rontodefinition} \& Theorem \ref{rfonto2} below. This statement is the \'etale analogue of Theorem A in \cite{2019}. In this article, Aizenbud and Avni prove that any algebraic stack $\mr X$ of finite type over a noetherian scheme $S$ admits a smooth presentation $\phi: X \to \mr X$ by an $S$-scheme $X$ such that for every morphism $\Spec(F) \to S$, the map $X(F) \to |\mr X(F)|$ is surjective. We use their result to extend it in the following way: if $\mr X$ is Deligne-Mumford over $S  = \Spec(F)$, then $\phi$ can be chosen \'etale.

\begin{theorem}[\cite{2019}] \label{smoothonto}
Any algebraic stack $\mr X$ of finite type over $F$ admits an $F$-surjective smooth presentation $\phi: X \to \mr X$ by a scheme $X$ over $F$. $\hfill \qed$
\end{theorem}

\end{comment}

The goal of this subsection is to prove the following (compare Theorem \ref{th:LMB}):

\begin{theorem} \label{rfonto2new}
Let $K$ be a field, and let $\mr X$ be a Deligne-Mumford stack over $K$. For any $K$-morphism $x \colon \Spec(K) \to \mr X$, there exists an affine scheme $U$, an \'etale morphism $\Psi \colon U \to \mr X$, and a $K$-morphism $y \colon \Spec(K) \to U$ such that $\Psi \circ y \cong x$. 
%There exists a scheme $U$ and a $K$-surjective morphism 
%Let $S$ be a scheme, $K$ a field with a morphism $\Spec(K) \to S$, and 
%Any Deligne-Mumford stack $\mr X$ of finite type over $F$ admits an $F$-surjective \'etale presentation $\phi: X \to \mr X$ by a scheme $X$ over $F$. 
\end{theorem}
\noindent
We will deduce Theorem \ref{rfonto2new} from Theorem \ref{th:LMB}. To do so, we will need: % the a lemma. Let $\mr X$ be a Deligne-Mumford stack. 

\begin{lemma}[Laumon--Moret-Bailly] \label{lemma:LMMBB}
Let $\mr X$ be a Deligne-Mumford stack over a scheme $S$, let $X$ be an affine $S$-scheme, and let $P \colon X \to \mr X$ be a smooth morphism. Define $\Omega_{X/\mr X}$ as the conormal sheaf of the diagonal $X \to X \times_{\mr X}X$. The canonical morphism $\rho \colon \Omega_{X/S} \to \Omega_{X/\mr X}$ is surjective and $\Omega_{X/\mr X}$ is an $\ca O_X$-module locally free of finite rank. 
\end{lemma}

%This deduction is inspired by the proof of Theorem 8.1 in \cite{LM-B}. 

\begin{proof}
See \cite[(8.2.2) - (8.2.3.2)]{LM-B}. 
 \end{proof}
 
 \begin{proof}[Proof of Theorem \ref{rfonto2new}]
By Theorem \ref{th:LMB}, there exists an affine scheme $X$, a smooth morphism $P \colon X \to \mr X$, and a $K$-morphism $z \colon \Spec(K) \to X$ such that $P \circ z \cong x$. By Lemma \ref{lemma:LMMBB}, the $\ca O_X$-module $\Omega_{X/\mr X}$ is locally free of finite rank; let $r$ be the rank of $\Omega_{X/\mr X}$ in a neighborhood around $z \in X$. Consider the sequence (c.f. Lemma \ref{lemma:LMMBB}): 
\begin{align}\label{surjectivityofrho}
\ca O_X \xrightarrow{d} \Omega_{X/K} \xrightarrow{\rho} \Omega_{X / \mr X}. 
\end{align}
\emph{\hypertarget{claim1}{Claim 1}}: There are sections $f_1, \dotsc, f_r \in R$ whose images $\rho d(f_i)(z) $ in $\Omega_{X/\mr X} \otimes k(z)$ form a basis for the $k(z)$-vector space $\Omega_{X/\mr X} \otimes k(z)$. (Compare \cite[(8.2.4)]{LM-B}.)

Indeed, let $R = \ca O_X(X)$; then $ \Omega_{X/K}  \otimes k(z)$ is generated by $\{d(f)(z) \mid f \in R\}$, where $d(f) \in \Omega_{X/k}$ is the differential of a section $f \in R$, and $d(f)(z)$ is the image of $d(f)$ in $\Omega_{X/k} \otimes k(z)$. The claim follows from the surjectivity of $\rho$ in (\ref{surjectivityofrho}).

We obtain representable $K$-morphisms 
\[
f \coloneqq \left(f_1, \dotsc, f_r\right) \colon X \to \bb A^r_K, \quad \textnormal{ and } \quad (P,f) \colon X \to \mr X \times_K \bb A^r_K. 
\]
\emph{\hypertarget{claim2}{Claim 2}}: $(P,f)$ is \'etale on a neighorhood $X' \subset X$ of $z$. (Compare \cite[(8.2.4.2)]{LM-B}.)

To prove Claim \hyperlink{claim2}{2}, let $Y $ be a scheme with surjective \'etale morphism $Y \to \mr X$, and let $W = X \times_{\mr X}Y$. By \cite[\href{https://stacks.math.columbia.edu/tag/01UV}{Tag 01UV}]{stacks-project} and \cite[\href{https://stacks.math.columbia.edu/tag/01UX}{Tag 01UX}]{stacks-project}, the composition $W \xrightarrow{h} Y \times_K\bb A^r_K \xrightarrow{\pi} Y$ induces a canonical exact sequence
\begin{align*}%\label{leftexctonly}
h^\ast \pi^\ast\Omega_{\bb A^r_K/K} \to \Omega_{W/Y} \to \Omega_{W/ Y \times \bb A^r} \to 0.
\end{align*}
By Claim \hyperlink{claim1}{1}, $h^\ast \pi^\ast\Omega_{\bb A^r_K/K}\otimes k(w) \to \Omega_{W/Y} \otimes k(w)$ is an isomorphism. Thus, by \cite[IV, \S4, 17.11.2]{EGA}, $h$ is \'etale at $w \in W$. Since $W \to X$ is \'etale, Claim \hyperlink{claim2}{2} follows.
%that $(P,f)$ is \'etale on a neighbourhood $X' \subset X$ of $z \in X$ as desired. 

Define a closed subscheme $j: U \hookrightarrow X'$ as the following fibre product:
\begin{align*}
\xymatrixcolsep{3pc}
%\xymatrixrowsep{1pc}
\xymatrix{
U \ar[r]^{\Psi} \ar[d]^j & \mr X \ar[d] \ar[r] & \Spec(K) \ar[d]^{f(z)}
\\
X' \ar[r]^{(P,f) \hspace{3mm}} & \mr X \times_{K} \bb A^r_{K} \ar[r] & \bb A^r_K.
}
\end{align*}
The morphism $\Psi \colon U \to \mr X$ is \'etale since $(P,f)$ is \'etale. Let $F $ be the morphism $\mr X \to \mr X \times_K \bb A^r_K$, and choose an isomorphism $\phi \colon P(z) \cong y$ in $\mr X(K)$. Then $F(\phi)$ is an isomorphism $F(P(z)) = (P(z), f(z)) \cong (y, f(z))$ in $\left(\mr X \times_K \bb A^r_K\right)(K)$. By definition of the fibre product of $X'$ and $\mr X$ over $\mr X \times_K \bb A^r_K$ \cite[(2.2.2)]{LM-B}, the triple $$\left(z \in X'(K),\;\; x \in \mr X(K), \;\; F(\phi) \colon (P,f)(z) \cong (y, f(z)) \right)$$
induces $y \colon \Spec(K) \to U$ such that $j \circ y = z \in X'(K)$ and $\Psi \circ y  \cong x \in \mr X(K)$. 
 \end{proof}
 \begin{corollary}\label{rfontofour}
Let $K$ be a field, and let $\mr X$ be a Deligne-Mumford stack over $K$. There exists a $K$-surjective \'etale presentation $U \to \mr X$. \hfill \qed
%scheme $U$, locally of finite type over $K$, and a 
\end{corollary}

\subsection{The orbifold structure of the real locus of a Deligne-Mumford stack} \label{orbifolds}

%We will now prove that if $\mr X$ is a smooth separated Deligne-Mumford stack over $\RR$, then the topological space $\va{\mr X(\RR)}$ defined in Definition \ref{deffer} admits the structure of a real-analytic orbifold. Moreover 
%$\mr X(\RR)$ can functorially be provided with the structure of a real-analytic orbifold. 

Let $\mr X$ be a smooth separated Deligne-Mumford stack over $\RR$. The goal of this section is to use the results of Section \ref{pointwise} to show that the topological space $\va{\mr X(\RR)}$ (see Definition \ref{deffer}) admits a canonical orbifold structure, agreeing with smooth manifold structure of $\mr X(\RR)$ in case $\mr X$ is a scheme. Let $P \colon U \to \mr X$ be an $\RR$-surjective \'etale presentation; $P$ exists by Corollary \ref{rfontofour}. 
 %Recall the notion of \'etale groupoid scheme over $\RR. 
%The map $P$ induces a presentation for $\mr X$ a groupoid scheme %(c.f. \cite[(2.4.3)]{LM-B} or \cite[\href{https://stacks.math.columbia.edu/tag/0230}{Tag 0230}]{stacks-project}) 
%as follows. 
For $R = U \times_{\mr X}U$, there are maps $s, t \colon R \to U$ and $c \colon R \times_{s,U,t} R \to R$ turning $
(U, R, s, t, c) = \left(R \rightrightarrows U \right)
$
into a groupoid scheme over $\RR$ \cite[Proposition (3.8)]{LM-B}. %see \cite[\href{https://stacks.math.columbia.edu/tag/04T4}{Tag 04T4}]{stacks-project}. 
%Let $[U/R]$ the quotient stack of $(U, R, s,t,c)$ (see \cite[(2.4.3), (3.4.3)]{LM-B}). %see \cite[\href{https://stacks.math.columbia.edu/tag/044Q}{Tag 044Q}]{stacks-project}. Then 
The morphisms $s, t \colon R \to U$ are \'etale, the morphism $[U/R] \to \mr X$ is an equivalence \cite[(4.3)]{LM-B}, %\cite[\href{https://stacks.math.columbia.edu/tag/04T5}{Tag 04T5}]{stacks-project}. 
and the morphism $j = (t,s) \colon R \to U \times_\RR U$ is proper since $\mr X$ is separated over $\RR$.  
%Since $\mr X$ is separated over $\RR$, its diagonal morphism is proper, and therefore the map $j \colon R \to U \times_\RR U$ is proper. 

The schemes $U$ and $R$ are smooth over $\RR$, thus the real-analytic spaces $U(\RR)$ and $R(\RR)$ are real-analytic manifolds. Consequently, the resulting five-tuple 
\[
\mr X(\RR)_P \coloneqq (U(\RR), R(\RR), s_\RR, t_\RR, c_\RR) = \left(R(\RR) \rightrightarrows U(\RR) \right)
\]
is a proper \'etale Lie groupoid, see \cite[\S1]{moerdijk-orbifoldsasgroupoids}, \cite[Chapter 5]{moerdijk-mrcun} or \cite[Chapter 1]{ruan-stringy}. Indeed, $j_\RR \colon R(\RR) \to U(\RR) \times U(\RR)$ is proper by \cite[Chapitre 1, \S10, Th\'eor\`eme 1]{bourbaki-topologie} because $j_\CC \colon R(\CC) \to U(\CC) \times U(\CC)$ is proper \cite[XII, Proposition 3.2]{SGA1}; the maps $s_\RR$ and $t_\RR$ are local diffeomorphisms because $s$ and $t$ are \'etale. 

%By [...], for any proper \'etale Lie groupoid $\ca G = \left(G_1 \rightrightarrows G_0\right)$, there is a canonical orbifold structure on the orbit space $\va{\ca G} = G_0/G_1$ of $\ca G$, where the definition of \emph{orbifold} is taken from [Ruan, Definition 1.48]. 
Since $P_\RR \colon U(\RR) \to \va{X(\RR)}$ induces a homeomorphism $U(\RR)/R(\RR)= \va{\mr X(\RR)}$, we obtain an orbifold structure \cite[\S3.2]{moerdijk-orbifoldsasgroupoids} %\cite[Definition 1.47]{ruan-stringy} 
on the topological space $|\mr X(\RR)|$. Let $\rm O_\RR(\mr X)$ denote the resulting orbifold \cite[\S3.3]{moerdijk-orbifoldsasgroupoids}. 
%\cite[Definition 1.48]{ruan-stringy}. 

\begin{proposition} \label{from-stack-to-orbifold}
Let $\mr X$ and $\mr Y$ be smooth separated Deligne-Mumford stacks over $\RR$. The orbifold $\rm O_\RR(\mr X)$ %on $\va{\mr X(\RR)}$ 
does not depend on the choice of $\RR$-surjective \'etale presentation $U \to \mr X$. A morphism %of smooth and separated Deligne-Mumford stacks 
of stacks $\mr X \to \mr Y$ induces a morphism of orbifolds $\rm O_\RR(\mr X) \to \rm O_\RR(\mr Y)$.  
\end{proposition}

\begin{proof}
From two $\RR$-surjective \'etale presentations $P \colon U \to \mr X$ and $P' \colon U' \to \mr X$ we obtain a third one that covers both, namely $P^{''} \colon U^{''} = U \times_{\mr X}U' \to \mr X$. This establishes a Morita equivalence 
$
\mr X(\RR)_P \leftarrow \mr X(\RR)_{P^{''}} \to \mr X(\RR)_{P'}$ 
%between $\mr X(\RR)_P$ and $\mr X(\RR)_{P''}$ 
\cite[\S2.4]{moerdijk-orbifoldsasgroupoids}. Thus, $\mr X(\RR)_P$ and $\mr X(\RR)_{P'}$ induce equivalent orbifold structures on $\va{\mr X(\RR)}$. Next, let $Q \colon V \to \mr Y$ and $U \to \mr X \times_{\mr Y} V$ be $\RR$-surjective \'etale presentations, and let $P$ be the $\RR$-surjective \'etale presentation $U \to  \mr X \times_{\mr Y} V \to \mr X$. The map of groupoids $\mr X(\RR)_P \to \mr Y(\RR)_Q$ gives a map of orbifolds $\rm O_\RR(\mr X) \to \rm O_\RR(\mr Y)$ \cite[\S3.3]{moerdijk-orbifoldsasgroupoids}. 
\end{proof}
%The second statement holds because the pull-back of an $\RR$-surjective \'etale presentation $U \to \mr Y$ along $\mr X \to \mr Y$ defines an $\RR$-surjective \'etale presentation $V \to \mr Y$. 
%\noindent
%It is also possible to equip $\va{\mr X(\RR)}$ with an orbifold structure in the classical sense (see \cite[Section 2.4]{moerdijk-mrcun} or \cite[Definition 1.2]{ruan-stringy}), via \cite[Corollary 5.31]{moerdijk-mrcun}. 
%Moreover, to any Lie groupoid $\ca G$ one can associate an effective Lie groupoid $\rm{Eff}(\ca G)$ and 
%The following proposition is a corollary of the results in the previous sections. 
%\begin{proposition}
%\begin{enumerate}
%\item \label{DM-orbi}
%Let $\mr X$ be a smooth separated Deligne-Mumford stack over $\RR$. The topological space $|\mr X(\RR)|$ defined in Definition \ref{deffer} canonically admits the structure of a real-analytic orbifold. 
%\end{enumerate}%
%By abuse of notation, we let $\mr X(\RR)$ denote the orbifold defined in Part \ref{DM-orbi}. 
%\begin{enumerate}
%\item 
%If $\mr X \to \mr Y$ is a smooth morphism of smooth separated Deligne-Mumford stacks over $\RR$, then the continuous map $\va{\mr X(\RR)} \to \va{ \mr Y(\RR)}$ is real-analytic for the orbifold structures on $|\mr X(\RR)|$ and $\va{\mr Y(\RR)}$ defined in part \ref{DM-orbi}. 
%\end{enumerate}
%\end{proposition}

\section{Moduli of real abelian varieties and curves} \label{sec:realmoduliabelianvarieties}

In the final part of Chapter \ref{ch:realmodulispaces}, we specialize to the stacks $\ca A_g$ and $\ca M_g$, of principally polarized abelian varieties of dimension $g$, respectively smooth, proper and geometrically connected curves of genus $g$. By the results of Section \ref{sec:coarsemoduli}, there is a natural topology on the set of isomorphisms classes of the real locus of both stacks; by the results of Section \ref{orbifolds}, this topology underlies an orbifold structure. Gross--Harris and Sepp\"al\"a--Silhol also defined a topology (and even an orbifold structure) on $\va{\ca A_g(\RR)}$ and on $\va{\ca M_g(\RR)}$, see Examples \ref{introexample:grossharris}. Our goal in this Section \ref{sec:realmoduliabelianvarieties} will be to prove that their topology agrees with ours. Although our proof shows that also the two orbifold structures agree, we will not discuss this aspect here.
%that the two topologies on $\va{\mr A_g(\RR)}$ (resp. $\va{\mr M_g(\RR)}$) 
%defined by our approach (see Definition \ref{deffer}) and by the approach of Gross--Harris-Seppala--SIlhol, are the same. 

\subsection{Moduli of real abelian varieties} \label{modofrealAV}
The set $\va{\ca A_g(\RR)}$ of isomorphism classes of principally polarized real abelian varieties of dimension $g$ can be provided with a real semi-analytic space structure as follows. The result seems to have been proven independently by to Gross--Harris \cite[Section 9]{grossharris} and Silhol \cite[Chapter IV, Section 4]{silholsurfaces}, although the former made an error in the definition of the group $\Gamma_H$ that appears on page $180$ of \cite{grossharris}. 
\newpage
\noindent
Let us recall Silhol's approach. We only give a sketch here; the reader is referred to Section \ref{sec:moduliofabelianRRR} for a more detailed account. Following \cite[Chapter IV, Definition 4.4]{silholsurfaces}, define a subset $\mr T(g) \subset \bb Z^2$ as 
\begin{align*}
\mr T(g) = \big\{(r, \alpha) \in \bb Z^{2}:  1 \leq r \leq g, \;  \alpha \in \{1,2\}  \mid  \alpha = 1 \tn{ if } r \tn{ is odd}\big\} \cup \set{(0,1)}.  
\end{align*}
%$$I : = \left\{(\alpha, \lambda) \in \bb Z^{2}:  1 \leq \lambda \leq g  \tn{ and }  \alpha \in \{1,2\} \tn{ such that } \alpha = 1 \tn{ if } \lambda \tn{ is odd } \right\} \cup \left\{ (0,0) \right\}.$$ 
Attach to each $ \tau = (r, \alpha) \in \mr T(g)$ a matrix $M(\tau) \in M_{g}(\bb Z)$ as in \cite[Ch.IV, Theorem 4.1]{silholsurfaces} (see the matrices \ref{typeone} and \ref{typetwo} in Section \ref{sec:moduliofabelianRRR}). Then define 
\[
\GL_g^\tau(\ZZ) =  \{ T \in \GL_g(\bb Z): T^t \cdot M(\tau) \cdot T \equiv M(\tau) \bmod 2 \} \subset \GL_g(\bb Z). %\subset \Sp_{2g}(\bb Z),
\] %where the inclusion 
%$\GL_g(\bb Z) \hookrightarrow \Sp_{2g}(\bb Z)$ is defined as $  T \mapsto \big(\begin{smallmatrix}
%T^t &  0 \\
%0  & T^{-1}
%\end{smallmatrix}\big)$. 
Consider the genus $g$ Siegel space $\bb H_g$ and define an anti-holomorphic involution 
\begin{align} \label{vartheta}
\Sigma_\tau : \bb H_g \to \bb H_g, \quad \Sigma_\tau(Z) = M(\tau) - \overline{Z}.
\end{align}
Let $\bb H_g^{\Sigma_\tau}$ be the fixed locus of $\Sigma_\tau$. Then $\GL_g^\tau(\ZZ)$ acts on $\bb H_g^{\Sigma_\tau}$ via the inclusion 
    \[
f_\tau\colon    \GL_g(\RR) \hookrightarrow \Sp_{2g}(\RR), \quad T \mapsto \begin{pmatrix} T^t & \frac{1}{2}\left(M(\tau) \cdot T^{-1}-T^t\cdot M(\tau)\right) \\ 0 & T^{-1} \end{pmatrix}. 
    \]
By \cite[Proposition 9.3]{grossharris} or \cite[Ch.IV, Theorem 4.6]{silholsurfaces}, taking period matrices induces a bijection 
\begin{equation} \label{eq:topologyaggrossharris}
\va{\ca A_g(\RR)} \cong \bigsqcup_{\tau \in \mr T(g)} \GL_g^\tau(\ZZ) \backslash \bb H_g^{\Sigma_\tau}.
\end{equation}

\subsection{Moduli of real algebraic curves} \label{modofrealAC}

In this section, as well as in the rest of this thesis, we shall often use the following: 

\begin{definition} \label{definition:realalgebraiccurve}
A \emph{real algebraic curve} is a proper, smooth and geometrically connected curve over $\RR$. 
\end{definition}
\noindent
Sepp\"al\"a and Silhol provide the set $\va{\ca M_g(\RR)}$ of real genus $g$ algebraic curves with a topology as follows. Fix a compact oriented $\ca C^{\infty}$-surface $\Sigma$ of genus $g$ and let $\ca T_g$ be the Teichm\"uller space of the surface $\Sigma$ (see e.g. \cite{arbarelloteichmuller}). 
\begin{definition}
%Theorem 1.1. Let k >__0 be an integer. Assume that (g, k, e) is the topological type of a real algebraic curve. If ~= 1, then 1<_k<-g+ 1 and k---g+ 1 (mod 2). For e=0, we have 0__<k__<g.These are the only restrictions for the topological types of real algebraic curves of genus g
Let $\mr J(g) \subset \ZZ^2$ be the set of tuples $(k, \epsilon) \in \bb Z^2$ that satisfy the following conditions. We have $\epsilon \in \{0,1\}$. If $\epsilon = 1$, then $ 1 \leq k \leq g +1$ and $k \equiv g+1 \bmod 2$. If $\epsilon = 0$, then $0 \leq k \leq g$.  %Let $J \subset \bb Z^2$ be the set of tuples $(\epsilon, k) \in \bb Z^2$ with $\epsilon \in \{0,1\}$ such that $1 \leq k \leq g+1$ and $k \equiv g + 1 \mod 2$ when $\epsilon = 1$, and $0 \leq k \leq g$ otherwise. 
\end{definition}
\noindent
To every $j = (k(j), \epsilon(j)) \in \mr J(g)$ one can attach an orientation-reversing involution $\sigma_j: \Sigma \to \Sigma$ of \emph{type} $j \in \mr J(g)$ \cite[Definition 1.2]{seppalasilhol2}. This means that $k(j) = \#\pi_0(\Sigma^{\sigma_j})$ and that $\epsilon(j) = 0$ if and only if $\Sigma \setminus \Sigma^{\sigma_j}$ is connected. Moreover, every such involution $\sigma_j: \Sigma \to \Sigma$ induces an anti-holomorphic involution $\sigma_j: \ca T_g \to \ca T_g$ \cite[\S I.2]{seppalasilhol2}. Denote by $N_j = \{ g \in \Gamma_g: g \circ \sigma_j = \sigma_j \circ g \}$ the normalizer of $\sigma_j: \ca T_g \to \ca T_g$ in the mapping class group $\Gamma_g$ of $\Sigma$. There is a natural bijection \cite[Theorem 2.1, Definition 2.3]{seppalasilhol2} 
\begin{equation} \label{eq:topologymgseppalasilhol}
    \va{\ca M_g(\RR)} \cong \bigsqcup_{j \in J} N_j \backslash \ca T_g^{\sigma_j}.
\end{equation}

\subsection{Comparing the real moduli spaces}  \label{subsec:compare}

%Let $\ca A_g$ be the algebraic stack over $\RR$ of principally polarized abelian varieties of dimension $g$, and let $\ca M_g$ be the stack over $\RR$ of smooth, proper and geometrically connected curves of genus $g$. 
Consider the real-analytic topologies on $|\ca A_g(\RR)|$ and $|\ca M_g(\RR)|$, see Definition \ref{deffer}. Our goal in Section \ref{subsec:compare} is to prove that (\ref{eq:topologyaggrossharris}) and (\ref{eq:topologymgseppalasilhol}) are homeomorphisms. 
%$|\ca A_g(\RR)|$ is homeomorphic to $\va{\ca A_g(\RR)}$, and $|\ca M_g(\RR)|$ to $\va{\ca M_g(\RR)}$, where the topology on $\va{\ca A_g(\RR)}$ (resp. $\va{\ca M_g(\RR)}$) is given by Equation (\ref{eq:topologyaggrossharris}) (resp. Equation (\ref{eq:topologymgseppalasilhol})).

\begin{theorem} \label{th:homeomorphismmoduli}
The natural bijection (\ref{eq:topologyaggrossharris}) is a homeomorphism.
%$|\ca A_g(\bb R)| \to \va{\ca A_g(\RR)}$ 
\end{theorem}

\begin{proof}
Let $\ca S \to \ca A_{g}$ be an $\bb R$-surjective \'etale surjection by a scheme $\ca S$ over $\bb R$ (see Corollary \ref{rfontofour}). Recall that $\ca A_g$ is smooth over $\RR$ \cite[Remark 1.2.5]{Jong1993}, so that $\ca S(\bb C)$ is a complex manifold. The composition of $\ca S(\bb R) \to \va{\ca A_g(\RR)}$ with (\ref{eq:topologyaggrossharris}) % \cong \bigsqcup_{i \in I} \Gamma_i \setminus \bb H_g^{\Sigma_\tau}$ 
defines a surjective \emph{real period map} (c.f. Section \ref{intro:sub:periodmaps}):
\begin{equation*}% \label{realperiodmap}
\mr P: \ca S(\bb R) \to  \bigsqcup_{\tau \in \mr T(g)} \GL_g^\tau(\ZZ) \backslash \bb H_g^{\Sigma_\tau}.% \bigsqcup_{i \in I} \Gamma_i \setminus \bb H_g^{\Sigma_\tau}.
\end{equation*}\begin{claim} \label{abelianclaim}
The real period map $\mr P$ is continuous and open.
\end{claim}
%\noindentIt is clear that Claim \ref{abelianclaim} implies the theorem. 
\noindent
\begin{proof}[Proof of Claim \ref{abelianclaim}]
We may work locally on $\ca S(\RR)$. Fix a point $0 \in \ca S(\RR)$, recall that $G = \Gal(\CC/\RR)$ (Definition \ref{thegaloisgroup}), and consider a $G$-stable contractible open neighbourhood $B$ of $0$ in $\ca S(\bb C)$ such that $B(\bb R) = B \cap \ca S(\bb R)$ is connected. The universal family $\ca X_g \to \ca A_g$ induces a holomorphic family of complex abelian varieties $\phi: A \to B$, equipped with two anti-holomorphic involutions $$(\sigma: A \to A, \chi: B \to B)$$ that commute with $\phi$, polarized by some section $E \in R^2\phi_\ast\ZZ$, such that $$E_{\chi(t)}(\sigma_\ast(x), \sigma_\ast(y)) = -E_t(x,y)$$ for every $t \in B$ and $x,y \in \rm  H_1(A_t, \ZZ)$, where $\sigma_\ast \colon \rm  H_1(A_t, \ZZ) \to \rm  H_1(A_{\chi(t)}, \ZZ)$ is the push-forward of the anti-holomorphic map $\sigma \colon A_t \to A_{\chi(t)}$. In other words, $\underline{\phi} \coloneqq \left(\phi, \sigma, \chi, E \right)$ is a \emph{holomorphic family of polarized real abelian varieties}. 

Claim \ref{abelianclaim} will now follow from:%To prove Claim \ref{abelianclaim}, our strategy is as as follows:
\begin{claim} \label{localclaim}The family $\underline{\phi}$ admits a period map $P \colon B \to \bb H_g$, which is $G$-equivariant for one of the holomorphic involutions $\Sigma_\tau$ on $\bb H_g$ (see (\ref{vartheta})), whose differential $dP$ defines an isomorphism on tangent spaces at each point of $B$, and such that the induced map $ B(\RR) \to \GL_g^\tau(\ZZ) \setminus \bb H_g^{\Sigma_\tau}$ coincides with the restriction of $\mr P$ to $B(\RR)$. 
\end{claim}
\noindent
Before we prove Claim \ref{localclaim}, we show that it implies Claim \ref{abelianclaim}. Consider the period map $P \colon B \to \bb H_g$ in Claim \ref{localclaim}. The map $P^G \colon B(\RR) \to \bb H_g^{\Sigma_\tau}$ defines an isomorphism on tangent spaces at each point of $B(\RR)$, thus is open and continuous, hence the restriction of $\mr P$ to $B(\RR)$ is open and continuous, proving Claim \ref{abelianclaim}.
\end{proof}
\begin{proof}[Proof of Claim \ref{localclaim}]
Consider the real abelian variety $(A_0, \sigma: A_0 \to A_0)$. Define $\Lambda_0$ to be the lattice $\rm H_1(A_0, \ZZ)$ and consider the alternating form $E_0: \Lambda_0 \times \Lambda_0 \to \bb Z$. By \cite[Section 9]{grossharris}, there exists a symplectic basis $\underline m = \{m_1, \dotsc, m_g; n_1, \dotsc, n_g \}$ for $\Lambda_0$ and a unique element $\tau \in \mr T(g)$ such that the matrix corresponding to $\sigma_*$ with respect to $\underline m$ is the matrix
\begin{equation} \label{eq:T}
   T = \begin{pmatrix}
 I_g & M(\tau)\\
 0 & -I_g
\end{pmatrix} \in \GL_{2g}(\ZZ),
\end{equation}
where $M(\tau) \in M_g(\ZZ)$ corresponds to $\tau \in \mr T(g)$ as in Section \ref{modofrealAV}. 
%in \cite[Ch.IV, Definition 4.4]{silholsurfaces} 

The canonical trivialization $R^1\phi_*\bb Z \cong \rm H^1(A_0, \bb Z)$ is $G$-equivariant, and induces the following isomorphism (resp. abelian groups):
\begin{align*}\label{trivialization2}
&s: R^1\phi_*\bb Z\cong \rm H^{1}(A_0, \bb Z) \xrightarrow{E_0} \rm H_1(A_0, \bb Z) = \Lambda_0, \\
& r_t: \rm H_1(A_t, \ZZ) \xrightarrow{E_t} \rm H^1(A_t,\ZZ) \xrightarrow{s_t} \Lambda_0, \quad \forall t \in B.
\end{align*}
Thus, for every $t \in B$, the lattice $\rm H_1(A_t, \bb Z)$ is equipped with the symplectic basis $$r_t^{-1}(\underline m) = \{m(t)_1, \dotsc, m(t)_g; n(t)_1, \dotsc, n(t)_g\}.$$
\noindent
We claim that these bases are compatible with the real structure of $\phi: A \to B$, in the sense that for every $t \in B$, the pushforward
$$
\sigma_*: \rm H_1(A_t, \bb Z) \to \rm H_1(A_{\chi(t)}, \bb Z)
$$
is given by matrix $T$ of equation (\ref{eq:T}) with respect to the bases $r_t^{-1}(\underline m)$ and $r_{\sigma(t)}^{-1}(\underline m)$. Indeed, this follows from the fact that the following diagram commutes:
\begin{equation} \label{comofcom}
\xymatrixcolsep{2pc}
\xymatrix{
\Lambda_0\ar[d]^{\sigma_*} &\rm H^1(A_0, \bb Z)\ar[d]^{-\tau*} \ar[l]&\rm H^1(A_t, \bb Z) \ar[d]^{-\sigma_*}\ar[l] & \rm H_1(A_t, \bb Z) \ar[l]\ar[d]^{\sigma_*} \\ 
\Lambda_0 &\rm H^1(A_0, \bb Z)\ar[l] &\rm H^1(A_{\chi(t)}, \bb Z) \ar[l] & \rm H_1(A_{\chi(t)}, \bb Z).\ar[l] 
}
\end{equation}
The first and the third square commute because if $\sigma: A_t \to A_{\sigma(t)}$ is the anti-holomorphic restriction of $\tau$ to the fibre $A_t$, then $E_t(x,y) = - E_{\chi(t)}(\sigma_\ast(x), \sigma_\ast(y))$ for every $x,y \in \rm H_1(A_t, \ZZ)$. The second diagram commutes because the trivialization $R^1\phi_*\bb Z \cong \rm H^1(A_0,\ZZ)$ is $G$-equivariant. 

The family of symplectic bases $\{r_t^{-1}(\underline m)\subset \rm H_1(A_t, \bb Z) \}_{t \in B}$ induces a holomorphic period map $P: B \to \bb H_g$, which is defined by sending an element $t \in B$ to the period matrix of the abelian variety $A_t$ with respect to the symplectic basis $r_t^{-1}(\underline m)$. Note that the differential $dP$ defines an isomorphism on each tangent space, because the morphism $\ca S \to \ca A_g$ is \'etale. 

Consider the anti-holomorphic involution $\Sigma_\tau: \bb H_g \to \bb H_g$ defined in (\ref{vartheta}). %by $\Sigma_\tau(Z) = M(\tau) - \bar Z$, see Section \ref{sec:densityinag}. 
We claim that: 
\begin{enumerate}
\item The composition 
\begin{equation*} %\label{eq:composition}
    B(\RR) \xrightarrow{P^G} \bb H_g^{\Sigma_\tau} \to \GL_g^\tau(\ZZ) \setminus \bb H_g^{\Sigma_\tau}
\end{equation*}
equals the restriction of $\mr P$ to the open neighborhood $B(\RR) \subset \ca S(\RR)$ of $0 \in B(\RR)$. 
\item \label{grosharrisitem:two}The period map $P$ is Galois-equivariant, in the sense that $\Sigma_\tau \circ P = P \circ \sigma$. 
\end{enumerate} 
The first part follows by definition of the bijection (\ref{eq:topologyaggrossharris}). So let us prove \ref{grosharrisitem:two}.  Fix an element $t \in B$; we need to show that 
\begin{align}\label{needtoshow}
\Sigma_\tau(P(t)) = P(\chi(t)) \in \bb H_g.
\end{align} For this, consider the symplectic basis $r_t^{-1}(\underline m) = \{m(t)_i; n(t)_j\}_{ij}
\subset \rm H_1(A_t, \ZZ).$ Let $\{\omega(t)_1, \dotsc, \omega(t)_g\} \subset \rm H^0(A_t, \Omega_{A_t}^1)$ be the basis dual to $\{m(t)_1, \dotsc, m(t)_g\}$ under the natural pairing
\begin{equation} \label{eq:pairing}
    \rm H_1(A_t, \ZZ) \times \rm H^0(A_t, \Omega_{A_t}^1) \to \CC,\white (\gamma, \omega) \mapsto \int_\gamma\omega.
\end{equation}
Then 
$
P(t)= \left(\int_{n(t)_i} \omega(t)_j\right)_{ij} \in \bb H_g. 
$
Moreover, the following diagram commutes by Lemma \ref{lemmatje} in Chapter \ref{ch:density}:
\begin{equation} \label{comofcom}
\xymatrixcolsep{2pc}
\xymatrix{
 \rm H_1(A_t, \bb Z) \ar[d]^{\tau_*}  \ar[r] & \rm H^0(A_t, \Omega^1_{A_t})^\vee \ar[d]^{F_{dR}} \\
\rm H_1(A_{\sigma(t)}, \bb Z)  \ar[r]& \rm H^0(A_{\sigma(t)}, \Omega^1_{A_{\sigma(t)}})^\vee.
}
\end{equation}
Here, $F_{dR}$ is the anti-linear map induced by the differential $d\sigma: T_eA_t \to T_eA_{\chi(t)}$ of $\sigma: A_t \to A_{\chi(t)}$ and the canonical identification $T_eX = \rm H^0(X, \Omega^1_{X})^\vee$ for any complex abelian variety $X$, where $e \in X$ is the origin. See also Lemma \ref{lemmatje}. This implies that (\ref{eq:pairing}) is compatible with $\sigma_\ast$ and $F_{dR}$, and that $F_{dR}\left(\omega(t)_j \right) = \omega(\chi(t))_j$. Therefore,
\begin{align*}
\Sigma_\tau( P(t)) &= M(\tau) - \overline{\left(\int_{n(t)_i} \omega(t)_j\right)_{ij}} \\
&= M(\tau) - \left(\int_{\sigma_\ast(n(t)_i)} F_{dR}\left(\omega(t)_j \right)\right)_{ij} \\
&= M(\tau) - 
\left(\int_{\sigma_\ast(n(t)_i)} \omega(\chi(t))_j\right)_{ij}. 
\end{align*}
If $M(\tau) = \left( x_{ij} \right)_{ij}$, and $\psi_t: \rm H_1(A_t, \ZZ) \cong \ZZ^{2g}
$ is the isomorphism identifying $r_t^{-1}(\underline m)$ with the standard symplectic basis $\{ e_1, \dotsc, e_g ; f_1, \dotsc , f_g\} \subset \ZZ^{2g}$, then, for $T \in \GL_{2g}(\ZZ)$ as in (\ref{eq:T}), we have:
\begin{align*}
\sigma_\ast(n(t)_i) &= \psi_{\chi(t)}^{-1}\left( T\cdot f_i \right) = \psi_{\chi(t)}^{-1}\left(  \sum_{k = 1}^g x_{ki} e_k - f_i \right) \\
&= \sum_{k = 1}^g x_{ki} m(\chi(t))_k  - n(\chi(t))_i \in \rm H_1(A_{\chi(t)}, \ZZ). 
\end{align*}
Therefore, we obtain $$\int_{\sigma_\ast(n(t)_i)} \omega(\chi(t))_j = \int_{\sum_{k = 1}^g x_{ki} m(\chi(t))_k  - n(\chi(t))_i} \omega(\chi(t))_j,$$ 
which further implies that 
\begin{align*}
\left(\int_{\sigma_\ast(n(t)_i)} \omega(\chi(t))_j\right)_{ij} 
&= \left(\int_{x_{ji} m(\chi(t))_j} \omega(\chi(t))_j\right)_{ij} - 
\left(\int_{n(\chi(t))_i} \omega(\chi(t))_j\right)_{ij} \\
&= (x_{ji})_{ij} - P\left(\chi(t)\right). 
\end{align*}
Since $\left( x_{ji}\right)_{ij} = \left( x_{ij}\right)_{ij} = M(\tau)$ because $M(\tau)$ is symmetric, the equality (\ref{needtoshow}) follows. This finishes the proof of Claim \ref{localclaim}, and thereby the proof of Claim \ref{abelianclaim}.
%our claim $\Sigma_\tau\left(P(t)\right) = P(\sigma(t))$ follows. 
\end{proof}
\noindent
Theorem \ref{th:homeomorphismmoduli} follows readily from Claim \ref{abelianclaim}.
\end{proof}

\begin{theorem} \label{th:homeomorphismmoduli2}
The natural bijection (\ref{eq:topologymgseppalasilhol}) is a homeomorphism. 
\end{theorem}

\begin{proof}
The strategy is similar to our strategy in the proof of Theorem \ref{th:homeomorphismmoduli}. Let $\ca S \to \ca M_{g}$ be an $\bb R$-surjective \'etale presentation by a scheme $\ca S$ over $\bb R$, and let $\ca C \to \ca S$ be the corresponding family of real algebraic curves of genus $g$. The composition $$\mr P: \ca S(\bb R) \to \va{\ca M_g(\RR)} \cong \bigsqcup_{j \in \mr J(g)} N_j \setminus \ca T_g^{\sigma_j}$$ is surjective. As before, it is enough to prove that $\mr P$ is continuous and open. 

For this, note that $\ca S(\CC)$ is a complex manifold because ${\ca M_g}$ is smooth \cite[Theorem 5.2]{DM69}. Fix $0 \in \ca S(\RR)$ and consider a $G$-stable contractible open neighborhood $B \subset \ca S(\CC)$ of $0$ such that $B(\RR) = B \cap \ca S(\RR)$ is connected. Let $\phi: C \to B$ be the induced family of Riemann surfaces over $B$, with anti-holomorphic involutions $(\tau \colon C \to C, \sigma\colon B \to B)$ that commute with $\phi$. By \cite{seppalasilhol2}, there exists a Teichm\"uller structure $[f_0], f_0: C_0 \xrightarrow{\sim} \Sigma$ and a unique $j \in \mr J(g)$ such that the composition
$$
\Sigma \xrightarrow{f_0^{-1}} C_0 \xrightarrow{\tau} C_0 \xrightarrow{f_0} \Sigma 
$$
is isotopic to the map $\sigma_j: \Sigma \to \Sigma$ -- in other words, such that $[f_0 \circ \tau \circ f_0^{-1} ] = [\sigma_j]$. 

Since the Kodaira-Spencer morphism $\rho: T_0B \to \rm H^1(C_0, T_{C_0})$ is an isomorphism, the family $\phi: C \to B$ is a Kuranishi family for the fiber $C_0$. The family $\phi: C \to B$ is topologically trivial because $B$ is contractible, thus $\phi$ can be endowed with a unique Teichm\"uller structure 
$$\set{[f_t], f_t: C_t \to \Sigma}_{t \in B}$$
 extending the Teichm\"uller structure $[f_0]$ on $C_0$. 
 
 For $t \in B$, consider the anti-holomorphic map $\tau: C_t \to C_{\sigma(t)}$. We claim that the two Teichm\"uller structures $[f_t]$ and $[\sigma_{j} \circ f_{\sigma(t)} \circ \tau]$ on $C_t$ agree. Indeed, the family $$\left\{ [\sigma_{j} \circ f_{\sigma(t)} \circ \tau] \right\}_{t \in B}$$ defines a new Teichm\"uller structure on $\phi: C \to B$, agreeing with $\left\{ [f_t] \right\}_{t \in B}$ at the point $0 \in B$, therefore agreeing with $\left\{ [f_t] \right\}_{t \in B}$ everywhere on $B$ \cite[XV, \S 2]{curvesII}.
 
Consequently, the holomorphic map $P: B \to \ca T_g$ defined as $P(t) = [C_t, [f_t]]$ can be shown to be $G$-equivariant as follows. The involution $\sigma_j: \ca T_g \to \ca T_g$ is defined by sending the class $[C, [f]]$ of a curve $C$ with Teichm\"uller structure $[f]$ to $$[C^\sigma, [C^\sigma \xrightarrow{\text{can}} C \xrightarrow{f} \Sigma \xrightarrow{\sigma_j}\Sigma ]] \in \ca T_g,$$ where $C^\sigma$ is the complex conjugate curve of $C$ and $\text{can}: C^\sigma \to C$ is the canonical anti-holomorphic map (see \cite[Chapter I, \S1]{silholsurfaces} for these notions). The family $\phi^\sigma: C^\sigma \to B^\sigma$ is isomorphic to the family $\phi: C \to B$ via the maps $\sigma \circ \text{can}: B^\sigma \to B$ and $\tau \circ \text{can}: C^\sigma \to C$, and the composition $\tau \circ \text{can}: C^\sigma \cong C$ restricts to an isomorphism $C_t^\sigma \cong C_{\sigma(t)}$ for each $t \in B$. For $t \in B$, we obtain 
$$
\sigma_j ( P(t)  ) = \sigma_j ( [C_t, [f_t]]  ) = [C_{\sigma(t)}, [C_{\sigma(t)} \xrightarrow{\tau} C_t \xrightarrow{f_t} \Sigma \xrightarrow{\sigma_j} \Sigma]] = [C_{\sigma(t)}, [f_{\sigma(t)}]] = P(\sigma(t)).$$ 
We conclude that $\sigma_j \circ P = P \circ \sigma$ as desired. 

Let $P^G \colon B(\RR) \to \ca T_g^{\sigma_j}$ be the induced morphism on real loci. Since $P$ induces isomorphisms on tangent spaces, the same holds for $P^G$. With respect to the projection $\pi: \ca T_g^{\sigma_j} \to N_j \setminus \ca T_g^{\sigma_j}$, one has $$\pi \circ P^G = \mr P|_{B(\RR)} \colon B(\RR) \to N_j \setminus \ca T_g^{\sigma_j}.$$ Therefore, $\mr P$ is continuous and open around the point $0 \in \ca S(\RR)$.
\end{proof}

\cleardoublepage
% Note that depending on your settings in the table of contents, subsections and subsubsections might appear virtually identical.
% Make sure to set the ToC depth and the numbering depth in the ToC the way you want.
\chapter{Noether-Lefschetz loci for real abelian varieties}\label{ch:density}

\section{Introduction} \label{densityintroduction}

In the previous Chapter \ref{ch:realmodulispaces}, we saw how one can obtain a real moduli space from a real algebraic moduli stack. This generalized the classical approaches to construct the real moduli spaces of abelian varieties and curves. In this chapter, we will continue to consider families of real abelian varieties $\ca A \to B$. This time, we will look at the distribution of certain Noether-Lefschetz loci in $B(\RR)$. More precisely, we consider the locus $R_k$ consisting of $t \in B(\RR)$ such that the fiber $A_t$ contains a real abelian subvariety of fixed dimension. We prove that -- even though such a set $R_k$ often has empty interior -- in favourable situations, $R_k$ is dense in $B(\RR)$. 

The goal of Chapter \ref{ch:density} will be to prove two results (Theorems \ref{th:maindensitytheorem} and \ref{theorem2}). Let us set up the context. Fix an integer $g \geq 1$. Let $\ca A$ and $B$ be complex manifolds, and let
\begin{equation} \label{family}
\left(\psi:  \ca A \to B, \white s: B \to \ca A, \white E \in R^2\psi_*\bb Z\right)
\end{equation} be a polarized holomorphic family of $g$-dimensional complex abelian varieties. The map $\psi$ is a proper holomorphic submersion, and $s$ is a holomorphic section of $\psi$. Moreover, and $A_t \coloneqq \psi^{-1}(t) $ is a complex abelian variety of dimension $g$ with origin $s(t)$, polarized by $E_t \in \rm H^2(A_t, \bb Z)$, for $t \in B$. Assume that $\psi$ admits a real structure in the following sense: $\ca A$ and $B$ are equipped with anti-holomorphic involutions $\tau$ and $\sigma$, commuting with $\psi$ and $s$ and compatible with the polarization, in the sense $\tau^*(E) = -E$. For example, this is true when $\psi$ is real algebraic, i.e. induced by a polarized abelian scheme over a smooth $\bb R$-scheme.

Let $B({\bb R})$ be the set of fixed points under the involution $\sigma: B \to B$. If $t \in B({\bb R})$, then $A_t$ is equipped with an anti-holomorphic involution $\tau$ preserving the group law. We shall not distinguish between the category of abelian varieties over ${\bb R}$ and the category of complex abelian varieties equipped with an anti-holomorphic involution preserving the group law. Thus, if $t \in B({\bb R})$, then $A_t$ is an abelian variety over $\bb R$. Define a set $R_k \subset B({\bb R})$ as follows:
\begin{equation}\label{rk}
R_k  = \left\{ t \in B({\bb R}): A_t \text{\emph{ contains an abelian subvariety over ${\bb R}$ of dimension }} k \right\}.
\end{equation}
For $t \in B$, the polarization gives an isomorphism $H^{0,1}(A_t) \cong \rm H^{1,0}(A_t)^\vee$; using the dual of the differential of the period map we obtain a symmetric bilinear form 
\begin{equation} \label{symbil}
q: \rm H^{1,0}(A_t) \otimes \rm H^{1,0}(A_t) \to (T_tB)^\vee. 
\end{equation}

%\blfootnote{\textcolor{CTlink}{$^\ast$}\footnotesize{\'Ecole normale sup\'erieure de Paris, 45 Rue d'Ulm, DMA, 75005 Paris, \href{mailto:olivier.de.gaay.fortman@ens.fr}{olivier.de.gaay.fortman@ens.fr}.}} 
%\blfootnote{\textcolor{CTlink}{$^\ast$}November 13, 2021. This project has received funding from the European Union's Horizon 2020}
%\blfootnote{\footnotesize{research and innovation programme under the Marie Sk\l{}odowska-Curie grant agreement N\textsuperscript{\underline{o}} 754362.$\;$\img{EU}}}

%\restoregeometry
%\pagestyle{plain}

\begin{condition}  \label{criterion} 
There exists an element $t \in B$ and a $k$-dimensional complex subspace $W \subset \rm H^{1,0}(A_t)$ such that the complex $0 \to  \bigwedge^2 W \to W \otimes \rm H^{1,0}(A_t) \to (T_tB)^\vee $ is exact.
\end{condition}
\noindent
The first main theorem of Chapter \ref{ch:density} is the following.
\begin{theorem}  \label{th:maindensitytheorem} 
If $B$ is connected and if Condition \ref{criterion} holds, then $R_k$ is dense in $B({\bb R})$. 
\end{theorem}
\noindent
The second main theorem of Chapter \ref{ch:density} is a consequence of Theorem \ref{th:maindensitytheorem}. It consists of the following three applications of Theorem \ref{th:maindensitytheorem}. See Definition \ref{definition:realalgebraiccurve} for the definition of \emph{real algebraic curve}. 
%A \textit{real algebraic curve} will be a proper, smooth, geometrically connected curve over ${\bb R}$. 
%Let $\mr M_g^{\bb R}$ be the set of isomorphism classes of real algebraic curves of genus $g$, and $\mr A_g^{{\bb R}}$ the set of isomorphism classes of real principally polarized abelian varieties of dimension $g$. The sets $\mr A_g^{{\bb R}}$ and $\mr M_g^{\bb R}$ carry natural real semi-analytic structures by work of Gross-Harris \cite{grossharris} and Sepp\"{a}l\"{a}-Silhol \cite{seppalasilhol2}.

\begin{theorem} \label{theorem2}
\begin{enumerate}
\item[\hypertarget{theoremA}{\textbf{A.}}] 
Given a positive integer $k < g$, moduli points of abelian varieties over $\bb R$ containing a $k$-dimensional abelian subvariety over ${\bb R}$ are dense in the moduli space $\va{\ca A_g(\RR)}$ of principally polarized abelian varieties of dimension $g$ over ${\bb R}$. 
\item[\hypertarget{theoremB}{\textbf{B.}}] For $1 \leq k \leq 3 \leq g$, moduli of real algebraic curves $C$ admitting a map $\varphi: C \to A$ to a $k$-dimensional abelian variety $A$ over ${\bb R}$ such that $\varphi(C(\CC))$ generates the group $A(\CC)$ are dense in the moduli space $\va{\ca M_g(\RR)}$ of real algebraic curves of genus $g$. 
\item[\hypertarget{theoremC}{\textbf{C.}}] Let $V \subset \bb P\rm H^0(\bb P^2_{{\bb R}}, \OO_{\bb P^2_{{\bb R}}}(d))$ be the moduli space of smooth degree $d \geq 3$ real plane curves. Let $R_1(V)$ be the set of $t \in V$ such that the corresponding curve $C_t$ admits a non-constant map $C_t \to E$ to a real elliptic curve $E$. Then $R_1(V)$ is dense in $V$. 
%If $V \subset \bb P\rm H^0(\bb P^2_{{\bb R}}, \OO_{\bb P^2_{{\bb R}}}(d))$ is the real algebraic set of degree $d$ smooth plane curves over ${\bb R}$, then the subset of $V$ corresponding to those curves that map non-trivially to elliptic curves over $\bb R$ is dense in $V$. 
\end{enumerate}
\end{theorem}

 \begin{remarks}
 \begin{enumerate}
 \item
The topology on the moduli space $\va{\ca A_g(\RR)}$ in Theorem \ref{theorem2}.\hyperlink{theoremA}{A} (resp.  $\va{\ca M_g(\RR)}$ in Theorem \ref{theorem2}.\hyperlink{theoremB}{B}) is the real-analytic topology, see Definition \ref{deffer}. It coincides with the topology constructed by Gross--Harris \cite{grossharris} (resp. Sepp\"{a}l\"{a}--Silhol \cite{seppalasilhol2}), see Theorem \ref{th:homeomorphismmoduli} (resp. Theorem \ref{th:homeomorphismmoduli2}). %underlying the real semi-analytic structure. We proved in Section \ref{sec:realmoduliabelianvarieties} that these topologies have other more intrinsic incarnations. 
\item
Although well-known in the complex case, Theorem \ref{theorem2}.\hyperlink{theoremA}{A} is new in the real case. 
\end{enumerate}
\end{remarks}
\noindent Our proofs rely on results in the complex setting proved by Colombo and Pirola in \cite{Colombo1990}. Indeed, Theorem \ref{th:maindensitytheorem} is the analogue over $\bb R$ (with \textit{unchanged} hypothesis) of the following theorem. Define $S_k \subset B$ to be the set of those $t$ in $B$ for which the complex abelian variety $A_t$ contains a complex abelian subvariety of dimension $k$.
\begin{theorem}[Colombo-Pirola \cite{Colombo1990}] \label{colpol}
If $B$ is connected and if Condition \ref{criterion} holds, then $S_k$ is dense in $B$. $\hfill \qed$
\end{theorem}
\noindent
Colombo and Pirola in turn were inspired by the Green-Voisin Density Criterion \cite[Proposition 17.20]{voisin}. Indeed, the latter gives a criterion for density of the locus where the fiber contains many Hodge classes for a variation of Hodge structure of weight $2$. Theorem \ref{colpol} adapts this result to a polarized variation of Hodge structure of weight $1$ (which is nothing but a polarized family of complex abelian varieties). The result is a criterion for the density of the locus where the fiber admits a sub-Hodge structure of rank $k$. 

To be a bit more precise, recall that for a complex manifold $U$ and a rational weight $2$ variation of Hodge structure $(H_{\bb Q}^2, \ca H, F^1, \nabla)$ on $U$, the \textit{Noether-Lefschetz locus} $\tn{NL}(U) \subset U$ is the locus where the rank of the vector space of Hodge classes is bigger than the general value. If $H_{\bb Q}^2$ is polarizable then $\tn{NL}(U)$ is a countable union of closed algebraic subvarieties of $U$ \cite{MR1273413}. The Green-Voisin Density Criterion referred to above decides whether $\tn{NL}(U)$ is dense in $U$. It was first stated in \cite{NLlocus} and applied to the universal degree $d\geq4$ surface $\mr S \to \mr B$ in $\bb P^3$. In this case, $\tn{NL}(\mr B)$ is the locus where the Picard group is not generated by a hyperplane section, and the union of the general components of $\tn{NL}(\mr B)$ is dense in $\mr B$ [\textit{loc. cit.}]. Analogously, $S_k$ is a countable union of components of $\NL(B)$ \cite{laszlodebarre}, and Condition \ref{criterion} implies that this union is dense in $B$. 

Let us carry the discussion over to the real setting. Unfortunately, the Green-Voisin Density Criterion cannot be adapted to the reals without altering the hypothesis. Going back to the universal family $\mr S \to \mr B$ of degree $d\geq 4$ surfaces in $\bb P^3_{{\bb C}}$, one observes that this family has a real structure, so that we can define the \textit{real Noether-Lefschetz locus} $\NL(\mr B({\bb R})) \subset \mr B({\bb R})$ as the locus of real surfaces $S$ in $\bb P^3_{{\bb R}}$ with $\Pic(S) \neq {\bb Z}$. By the above, the Green-Voisin Density Criterion is fulfilled hence $\NL(\mr B)$ is dense in $\mr B$, whereas density of $\NL(\mr B({\bb R}))$ in $\mr B(\bb R)$ may fail: for every degree $4$ surface in $\bb P^3_{{\bb R}}$ whose real locus is a union of $10$ spheres, $\Pic(S) = {\bb Z}$, and so $\NL(\mr B({\bb R}))  \cap K = \emptyset$ for any connected component $K$ of surfaces of such a topological type \cite[Remark 1.5]{benoistttt}. There is an alternate criterion \cite[Proposition 1.1]{benoistttt}, but the hypothesis is more complicated thus harder to fulfill, and only implies density of $\NL(\mr B({\bb R}))$ in one component of $\mr B(\bb R)$ at a time. It is therefore remarkable that for the real analogue of density of $S_k$ in $B$, none of these problems occur. Theorem \ref{th:maindensitytheorem} shows that the complex density criterion can be carried over to the reals \textit{without changing it}. Condition \ref{criterion} does not involve the real structures at all, applying to any real structure on the family. The result is density of $S_k \subset B$ \textit{and} $R_k \subset B({\bb R})$. It is for this reason that the applications of Theorem \ref{th:maindensitytheorem} are generous: the statements in Theorem \ref{theorem2}, as well as their proofs, are direct analogues of some applications of Theorem \ref{colpol} in \cite{Colombo1990}. 

Finally, we remark that Colombo-Pirola's Theorem \ref{colpol} has been generalized by Ching-Li Chai \cite{Chai1998DensityOM}, who considers a variation of rational Hodge structures over a complex analytic variety and rephrases and answers the following question in the context of Shimura varieties: when do the points corresponding to members having extra Hodge cycles of a given type form a dense subset of the base? It should be interesting to investigate whether such a generalization can be carried over to the real numbers as well.

\section{Abelian subvarieties in a family} \label{realfamily}

In this section, we shall prove Theorem \ref{th:maindensitytheorem}. %which we present hereafter, in Section \ref{densityproofsection}. 
\\
\\
Let
\begin{equation*} %\label{family}
\left(\psi:  \ca A \to B, \white s: B \to \ca A, \white E \in R^2\psi_*\bb Z\right)
\end{equation*} 
be as in Section \ref{densityintroduction}. Let $\bb V_{{\bb Z}} =  R^1\psi_{*}{\bb Z}$ be the ${\bb Z}$-local system attached to $\psi$. The holomorphic vector bundle $\ca H  =  \bb V_{{\bb Z}}  \otimes_{{\bb Z}} \OO_{B}$ is endowed with a filtration by the holomorphic subbundle $F^1\ca H = \ca H^{1,0} \subset \ca H$, of fiber $$(\ca H^{1,0})_t = \rm H^{1,0}(A_t) \subset \rm H^1(A_t, {\bb C})= \ca H_t.$$ Denote by $\ca H_{{\bb R}}$ the real $C^{\infty}$-subbundle of $\ca H$ whose fibers are $(\ca H_{{\bb R}})_t = \rm H^1( A_t, {\bb R})$. Define $G = \Gal({\bb C}/{\bb R})$. Recall that $\ca A $ and $B$ are endowed with anti-holomorphic involutions $\tau$ and $\sigma$ such that $\psi \circ \tau = \sigma \circ \psi$. The map $\tau$ induces, for each $t \in B$, an anti-holomorphic isomorphism $\tau: A_t \cong A_{\sigma(t)}$. The pullback of $\tau$ gives an isomorphism $\tau^*: \ca H_{\sigma(t)} \to \ca H_{t}$ inducing an involution of differentiable manifolds
$$
F_{\infty}: \ca H \to \ca H 
$$
over the involution $\sigma: B \to B$. Composing $F_{\infty}$ fiberwise with complex conjugation provides an involution of differentiable bundles 
$$
F_{dR}: \ca H \to \ca H
$$ over $\sigma \colon B \to B$ which respects the Hodge decomposition by \cite[I, Lemma 2.4]{silholsurfaces}. 

Let $\mr G(k,\ca H^{1,0})$ be the complex Grassmannian bundle of complex $k$-planes in $\ca H^{1,0}$ over $B$, and $\mr G(2k, \ca H_{{\bb R}})$ the real Grassmannian bundle of real $2k$-planes in $\ca H_{{\bb R}}$ over $B$. We see that $G$ acts on these bundles via the morphisms $F_{dR}: \ca H^{1,0} \to \ca H^{1,0}$ and $\tau^\ast: \ca H_{{\bb R}} \to \ca H_{{\bb R}}$. Since the diffeomorphism $\ca H_{{\bb R}} \to \ca H^{1,0}$, defined as the composition of morphisms
$$
\ca H_{{\bb R}} \hookrightarrow \ca H = \ca H^{1,0} \oplus \ca H^{0,1} \to \ca H^{1,0},
$$
is $G$-equivariant, it induces a $G$-equivariant morphism of differentiable manifolds
\[
\mr G(k,\ca H^{1,0}) \to \mr G(2k, \ca H_{{\bb R}}).
\]
Let $0 \in B({\bb R})$ and choose a $G$-stable contractible neighbourhood $U$ of $0$ in $B$. Trivialize $\bb V_{\bb Z}$ over $U$, which trivializes $\mr G(2k, \ca H_{{\bb R}})$ over $U$, and consider the morphism $\Phi$ defined as the composition
\begin{align*}
\mr G(k,\ca H^{1,0})|_U &\to \mr G(2k, \ca H_{{\bb R}})|_U \\
&\cong U \times \tn{Grass}_{{\bb R}}(2k, \rm H^{1}(A_0, {\bb R})) \to \tn{Grass}_{{\bb R}}(2k, \rm H^{1}(A_0, {\bb R})).
\end{align*} Let $j$ be the canonical map $\tn{Grass}_{{\bb Q}}(2k, \rm H^{1}(A_0, {\bb Q})) \to \tn{Grass}_{{\bb R}}(2k, \rm H^{1}(A_0, {\bb R}))$. We obtain the following diagram: 
\begin{equation} \label{Skdiagram}
\xymatrixcolsep{5pc}
\xymatrix{
\mr G(k,\ca H^{1,0})|_U\ar[d]^f \ar[r]^{\Phi \hspace{1cm}} & \tn{Grass}_{{\bb R}}(2k, \rm H^{1}(A_0, {\bb R})) \\
U & \tn{Grass}_{{\bb Q}}(2k, \rm H^{1}(A_0, {\bb Q})). \ar[u]_j
}
\end{equation}
\noindent
Write $U(\bb R) = U \cap B(\bb R)$. Diagram (\ref{Skdiagram}) provides a parametrization of polarized real abelian varieties containing a $k$-dimensional real abelian subvariety:

\begin{proposition} \label{gequivprop00}
\begin{enumerate}  \label{gequivprop0}
\item  \label{gequivprop1}
The morphisms 
$
f$, $\Phi$ and $j$ in diagram (\ref{Skdiagram}) are $G$-equivariant. 
\item \label{gequivprop2} 
We have 
$$
f(\Phi^{-1}(j(\tn{Grass}_{{\bb Q}}(2k, \rm H^{1}(A_0, \bb Q))^G))) = R_k \cap U({\bb R}),
$$
where $R_k \subset B(\RR)$ is the set defined in (\ref{rk}). 
% \[
%R_k = \{ t \in B({\bb R}): A_t \text{\textit{ contains an real abelian subvariety of dimension }} k \}.\] 
\end{enumerate}
\end{proposition}
\begin{proof} 
1. The fact that $f$ and $j$ are $G$-equivariant is immediate from the description of $F_{dR}$. As for the morphism $\Phi$, it suffices to show that the trivialization $$\ca H_{\bb R} \cong U \times \rm H^1(A_0, \bb R)$$ is $G$-equivariant. This map is induced by the restriction $r: \bb V_{\bb R}|_U \to (\bb V_{\bb R})_0$, which is an isomorphism of local systems; but $r$ is unique if we require that $r$ induces the identity on $(\bb V_{\bb R})_0$. 

2. Let $t \in U({\bb R})$ and consider the polarized real abelian variety $(A_t, \tau)$. We have $A_t \cong V / \Lambda$, where $V \cong \rm H^{1,0}(A_t)^{*}$ and $\Lambda \cong \rm H_1(A_t, {\bb Z})$. It follows that $A_t$ contains a complex abelian subvariety $X$ of dimension $k$ if and only if there exists a $k$-dimensional ${\bb C}$-vector subspace $W_1 \subset \rm H^{1,0}(A_t)$ and a $k$-dimensional $\bb Q$-vector subspace $W_2 \subset \rm H^1(A_t, \bb Q)$ such that, under the canonical real isomorphism $\rm H^{1,0}(A_t) \cong \rm H^1(A_t, {\bb R})$, the space $W_1$ is identified with $W_2 \otimes {\bb R}$. In this case, $L = W_2^* \cap \Lambda$ is a lattice in $W_1^*$, and $X = W_1^{*} / L$. Then observe that the $k$-dimensional complex abelian subvariety $X \subset A_t$ is a $k$-dimensional real abelian subvariety $(X, \tau|_X)$ of $(A_t, \tau)$ if and only if $\tau(X)  = \tau(W_1^{*} / L) = W_1^{*} / L = X$. The latter is equivalent to $F_{dR}(W_2) = W_2$ by Lemma \ref{lemmatje} below, and we are done.
\end{proof}
\begin{lemma} \label{lemmatje}
Let $A$ be a complex torus, let $\Lambda = \rm H_1(A, \ZZ)$ and consider the Hodge decomposition $\Lambda_\CC = V^{-1,0} \oplus V^{0,-1}$. For an anti-holomorphic involution $\sigma: A \to A$ such that $\sigma(e) = e$, let $F_\infty(\sigma): \Lambda \to \Lambda$ be the pushforward of $\sigma$, and let $F_{dR}(\sigma): V^{-1,0}  \to V^{-1,0} $ correspond to the differential $d\sigma_e: T_eA \to T_eA$. This defines a bijection between:
\begin{enumerate}
    \item[(i)] The set of real structures $\sigma: A \to A$.
    \item[(ii)] The set of involutions $F_\infty: \Lambda \to \Lambda$ such that $F_{\infty, \CC} \left( V^{-1,0}  \right) = V^{0, -1} $.
    \item[(iii)] The set of anti-linear involutions $F_{dR}: V^{-1,0}  \to V^{-1,0} $ such that $F_{dR} \left( \Lambda  \right) = \Lambda $.
\end{enumerate}
\end{lemma}

\begin{proof}
If $\sigma: A \to A$ is as in $(i)$, then $F_\infty(\sigma)_\CC$ interchanges the factors of the Hodge decomposition by \cite[Chapter I, Lemma (2.4)]{silholsurfaces}, so $\sigma \mapsto F_\infty(\sigma)$ is a well-defined map $(i) \to (ii)$. For $F_\infty: \Lambda \to \Lambda$ as in $(ii)$, the restriction to $V^{-1,0}$ of the composition of $F_{\infty, \CC} $ with complex conjugation defines a map $F_{dR} = \text{conj} \circ F_{\infty, \CC} : V^{-1,0} \to V^{-1,0}$ as in $(iii)$, which equals $F_{dR}(\sigma)$ when $F_\infty=F_\infty(\sigma)$. Then $F_\infty \mapsto \text{conj} \circ F_{\infty, \CC}$ determines the bijection between $(ii)$ and $(iii)$. Finally, the map $\sigma \mapsto F_{dR}(\sigma)$, $(i) \to (iii)$ is the map giving the natural correspondence between anti-holomorphic involutions on $A$ preserving $e$ and anti-holomorphic involutions on its universal cover $T_eA$ that preserve $0$ and are compatible with $\pi_1(A,e) = \rm H_1(A, \ZZ)$-orbits.
\end{proof}

\subsection{Proving the density theorem}

For a smooth manifold $W$ on which a compact Lie group $H$ acts by diffeomorphisms, the set of fixed points $W^H$ has a natural manifold structure that makes it a submanifold of $W$ \cite[Corollary I.2.3]{audin}. $\Phi$ in Diagram (\ref{Skdiagram}) is $G$-equivariant by Proposition \ref{gequivprop00}.\ref{gequivprop1}; denote by $\Phi^G$ the induced morphism on fixed spaces. We obtain a diagram of $\ca C^{\infty}$-manifolds:

\begin{equation} \label{Skdiagram2}
\xymatrixcolsep{5pc}
\xymatrix{
\mr G(k,\ca H^{1,0})^G|_{U(\bb R)} \ar[d]^f \ar[r]^{\Phi^G\hspace{1cm}} & \tn{Grass}_{{\bb R}}(2k, \rm H^{1}(A_0, {\bb R}))^G \\
U({\bb R}) &  \tn{Grass}_{{\bb Q}}(2k, \rm H^{1}(A_0, {\bb Q}))^G. \ar[u]_j
}
\end{equation}
\noindent
Recall that we obtained the equality $$f(\Phi^{-1}(j(\tn{Grass}_{{\bb Q}}(2k, \rm H^{1}(A_0, \bb Q))^G))) = R_k \cap U({\bb R})$$ in Proposition \ref{gequivprop00}.\ref{gequivprop2}. Recall also the symmetric bilinear form $$
q: \rm H^{1,0}(A_t) \otimes \rm H^{1,0}(A_t) \to (T_tU)^\vee
$$ from Section \ref{densityintroduction}, given by the differential of the period map and the isomorphism $H^{0,1}(A_t) \cong \rm H^{1,0}(A_t)^\vee$ which the polarization induces. Finally, recall the notation (Equation (\ref{rk}), \S \ref{densityintroduction}) $$ R_k  = \{ t \in B({\bb R}): A_t \text{\textit{ contains an abelian subvariety over ${\bb R}$ of dimension }} k \}.$$
Let us fix notation and introduce the aim of this section. 
\begin{notation}
For $t \in B$ and $W \in \Grass_{\bb C}(k, \rm H^{1,0}(A_t))$, let $W^\perp$ denote the orthogonal complement of $W$ in $\rm H^{1,0}(A_t)$ with respect to the polarization $E_t$. Define the sets $\mr E_k, \mr F_k \subset \mr G(k, \ca H^{1,0})$ and $\mr E_k(\bb R), \mr F_k(\bb R),\mr R_{k,U} \subset \mr G(k, \ca H^{1,0})^G$ as follows:
\begin{equation*}
\mr E_k = \{(t, W) \in \mr G(k, \ca H^{1,0}) \tn{ : } 0 \to W \otimes W^\perp \to (T_tB)^\vee \ \tn{ is exact} \}
\end{equation*}
\begin{equation*}
\mr F_k = \{(t, W) \in \mr G(k, \ca H^{1,0}) \tn{ : } 0 \to  \wedge^2 W\to W \otimes \rm H^{1,0}(A_t) \to (T_tB)^\vee  \tn{ is exact} \} 
\end{equation*}
\begin{equation*}
\mr E_k(\bb R) = \mr E_k \cap \mr G(k, \ca H^{1,0})^G
\end{equation*}
\begin{equation*}
\mr F_k(\bb R) = \mr F_k \cap \mr G(k, \ca H^{1,0})^G
\end{equation*}
\begin{equation*}
\mr R_{k,U} = \Phi^{-1}(j (\tn{Grass}_{\bb Q}(2k, \rm H^{1}(A_0, \bb Q))^G)).
\end{equation*}
\end{notation}
\noindent
Note that $f(\mr R_{k,U}) = R_k \cap U(\bb R)$. Our strategy to prove Theorem \ref{th:maindensitytheorem} is the following:\begin{proposition} \label{propper}
Let $\mr E_k, \mr F_k , \mr R_{k,U} \subset \mr G(k, \ca H^{1,0})$ be as above. One has inclusions
\begin{align*}
\xymatrixcolsep{0.8pc}
\xymatrix{
\mr F_k(\bb R)|_{U(\bb R)} \white \ar@{^{(}->}[r] \white &\white \mr E_k(\bb R)|_{U(\bb R)} \white \ar@{^{(}->}[r]\white&\white \overline{\mr R_{k,U}} \white\ar@{^{(}->}[r]\white& \white \mr G(k, \ca H^{1,0})^G|_{U(\bb R)} .
}
\end{align*}
If $\mr F_k$ is not empty then $\mr F_k(\bb R)$ is dense in $\mr G(k, \ca H^{1,0})^G$. Consequently, if $\mr F_k \neq \emptyset$, then $\mr R_{k,U}$ is dense in $\mr G(k, \ca H^{1,0})^G|_{U(\bb R)}$; in particular, then $R_k \cap U(\bb R)$ is dense in $U(\bb R)$.
\end{proposition}
\noindent
Before proving Proposition \ref{propper} we remark that it implies Theorem \ref{th:maindensitytheorem}. 

\begin{proof}[Proof of Theorem \ref{th:maindensitytheorem}]
It suffices to show that for each $x \in B(\bb R)$ and any $G$-stable contractible open neighborhood $x \in U \subset B$, the set $R_k \cap U \cap B(\bb R)$ is dense in $U \cap B(\bb R)$. Let $x \in U \subset B$ be such a real point and open neighborhood. Condition \ref{criterion} is satisfied if and only if $\mr F_k$ is non-empty; by Proposition \ref{propper}, we are done.
\end{proof}

\begin{proof}[Proof of Proposition \ref{propper}]
We first prove the inclusions. We have $\mr F_k \subset \mr E_k$: if $(t, W) \in \mr F_k$ then $\Ker( W \otimes \rm H^{1,0}(A_t) \to (T_tB)^\vee) \subset W \otimes W$. Hence $\mr F_k(\bb R)|_{U(\bb R)} \subset \mr E_k(\bb R)|_{U(\bb R)} $.

For the second inclusion, consider again the map 
$$
\Phi: \mr G(k,\ca H^{1,0})|_U \to \Grass_{\bb R}(2k, \rm H^1(A_0, \bb R))
$$
as in diagram (\ref{Skdiagram}). Define a set $Y \subset \mr G(k,\ca H^{1,0})$ as follows:
$$
Y = \{ x \in \mr G(k,\ca H^{1,0})|_U \mid \tn{ \textit{the rank of} } d\Phi \tn{ \textit{is maximal at} }x  \}  \subset \mr G(k, \ca H^{1,0})|_U.
$$
Then $ Y = \mr E_k|_U$ by \cite[\S 1]{Colombo1990}. Let $Z$ be the set of points $x \in \mr G(k,\ca H^{1,0})^G|_{U(\bb R)}$ such that the differential $$d(\Phi^G): T_x\mr G(k,\ca H^{1,0})^G|_{U(\bb R)}  \to T_{\Phi(x)}\tn{Grass}_{{\bb R}}(2k, \rm H^{1}(A_0, {\bb R}))^G$$ 
of $\Phi^G$ at $x$ has maximal rank. We claim that $$
Z = Y \cap \mr G(k,\ca H^{1,0})^G|_{U(\bb R)}.$$ Indeed, for a point $x \in \mr G(k, \ca H^{1,0})^G|_{U(\bb R)}$, the rank of $d(\Phi^G)$ is maximal at $x$ if and only if the rank of $d \Phi$ is maximal at $x$. The latter is true because $$T_x\mr G(k, \ca H^{1,0})  = (T_x\mr G(k, \ca H^{1,0}))^G \otimes_{{\bb R}} {\bb C} = T_x\mr G(k, \ca H^{1,0})^G \otimes_{{\bb R}} {\bb C},$$ and similarly for $T_{\Phi(x)} \tn{Grass}_{{\bb R}}(2k,\rm H^1(A_0, {\bb R}))$, and because the differential $d\Phi$ at the point $x \in \mr G(k,\ca H^{1,0})^G|_{U(\bb R)}$ is the complexification of the differential $d( \Phi^G)$ at $x$.

Consequently, 
\[
Z = Y \cap \mr G(k,\ca H^{1,0})^G|_{U(\bb R)} = \mr E_k \cap \mr G(k, \ca H^{1,0})^G|_{U(\bb R)} = \mr E_k(\bb R)|_{U(\bb R)}.
\]
It follows that the morphism
$$
\Phi^G: \mr G(k,\ca H^{1,0})^G|_{U(\bb R)} \to \tn{Grass}_{{\bb R}}(2k, \rm H^{1}(A_0, {\bb R}))^G
$$
is open when restricted to the set $\mr E_k(\bb R)|_{U(\bb R)}$. Since $j(\tn{Grass}_{{\bb Q}}(2k, \rm H^{1}(A_0, {\bb Q}))^G)$ is dense in $\tn{Grass}_{{\bb R}}(2k, \rm H^{1}(A_0, {\bb R}))^G$ by Lemma \ref{lemmabetween} below, the set $$\Phi^{-1}( j(\tn{Grass}_{{\bb Q}}(2k, \rm H^{1}(A_0, {\bb Q}))^G)) \cap \mr E_k(\bb R)|_{U(\bb R)}$$ is dense in $\mr E_k(\bb R)|_{U(\bb R)}$. The equality $$\Phi^{-1}( j(\tn{Grass}_{{\bb Q}}(2k, \rm H^{1}(A_0, {\bb Q}))^G)) \cap \mr E_k(\bb R)|_{U(\bb R)} = \mr R_{k,U} \cap  \mr E_k(\bb R)|_{U(\bb R)}$$ implies that $\mr R_{k,U} \cap  \mr E_k(\bb R)|_{U(\bb R)}$ is dense in $ \mr E_k(\bb R)|_{U(\bb R)}$. In particular, $ \mr E_k(\bb R)|_{U(\bb R)}$ is contained in the closure of $\mr R_{k,U}$ in $ \mr G(k,\ca H^{1,0})^G|_{U(\bb R)} $. The inclusions are thus proved.
\\
\\
It remains to prove the assertion that non-emptiness of $\mr F_k$ implies density of $\mr F_k(\bb R)$ in $\mr G(k, \ca H^{1,0})^G$. Write $\mr G = \mr G(k, \ca H^{1,0})$ and $\mr G(\RR) = \mr G^G$. Let $Z_k \subset \mr G $ be the complement of $\mr F_k$ in $\mr G$. Observe that $Z_k$ equals the set of pairs $(t,W)\in \mr G $ such that the injection $$\iota \colon \bigwedge^2 W \to \Ker ( W  \otimes \rm H^{1,0}(A_t) \to (T_tB)^\vee)$$ is not an isomorphism. This latter map $\iota$ varies holomorphically; therefore, locally on $\mr G$, the set of points where the rank is not maximal is determined by the vanishing of minors in a matrix with holomorphic coefficients. Since $\mr G$ is connected, it follows that $Z_k$ is nowhere dense in $\mr G$. 

If $n = \dim_\CC(\mr G)$, then $\mr G(\RR)$ is a closed real submanifold of $\mr G$ of real dimension $n$, and we claim that $Z_k(\RR)$ is nowhere dense in $\mr G(\RR)$. Suppose for contradiction that $Z_k(\RR)$ contains a non-empty open set $V \subset \mr G(\RR)$. Locally around $t\in V$, $t \in V \subset \mr G(\RR)$ is the inclusion of the real locus $ 0 \in B(\RR) \subset \RR^n$ of a euclidean open ball $0 \in B \subset \CC^n$, such that, for some holomorphic functions $f_i: B \to \CC$, one has 
$$B \cap Z_k = \{f_1 =  \dotsc = f_m = 0 \}.$$ Since the $f_i$ vanish on $B(\RR)$, they vanish on $B$, hence $B \subset Z_k$ which is absurd.
\end{proof}

\begin{lemma}\label{lemmabetween} Let $n, k \in \bb Z_{\geq 0}$. Let $V$ be an $n$-dimensional vector space over $\bb Q$ and consider a linear transformation $F \in \textnormal{GL}(V)_{\bb Q}$ which is diagonalizable over $\bb Q$. For $L \in \{\QQ, \RR\}$, denote by $\bb G(k, V)^F(L)$ the set of $k$-dimensional $L$-subvector spaces $W \subset V \otimes_{\bb Q} L$ for which $F(W) = W$. Then $\bb G(k, V)^F(\bb Q)$ is dense in $\bb G(k, V)^F(\bb R)$. 
\end{lemma}

\begin{proof} 
For $F = \id$ this is an elementary fact. In order to deduce the general case from this, let $\lambda_1, \dotsc, \lambda_r \in \bb Q^*$ be the eigenvalues of $F$, and denote by $$V_i \coloneqq V^{F = \lambda_i} \subset V, \quad i \in I \coloneqq  \{1, \dotsc, r \}$$ the corresponding eigenspaces. Eigenspaces are preserved under scalar extension, so $(V \otimes_{\bb Q} \bb R)^{F_{\bb R} = \lambda_i} = V_i \otimes_{\bb Q} \bb R$. But a $k$-dimensional $\bb R$-subvector space $W \subset V_{\bb R}$ satisfies $F(W) = W$ if and only if $W = \bigoplus_{i \in I}W_i$ with $W_i$ a $\bb R$-subvector space of $V_i\otimes_{\bb Q} \bb R$ for each $i \in I$ and $\sum_{i \in I} \dim W_i = k$, and $W$ is defined over $\bb Q$ if and only if each $W_i$ is. This means that, under the canonical diffeomorphism 
\begin{equation}
\bb G(k, V)^F(\bb R)  \cong \underset{\sum k_i = k}{\bigsqcup} \prod_{i \in I} \bb G(k_i, V_i)(\bb R),
\end{equation}
the rational subspace $ \bb G(k, V)^F(\bb Q) $ is identified with $\bigsqcup_{k_i} \prod_{i \in I} \bb G(k_i, V_i)(\bb Q)$. 
\end{proof}

\section{Real deformation spaces}

\subsection{Density in the complex deformation space}

Let  
$
\psi$ be a holomorphic family $\psi: \ca A \to B$ of complex abelian varieties of dimension $g$, polarized by a section $E \in R^2\psi_*\bb Z$. For $t \in B$, denote by $\omega_t \in \rm H^1(A_t, \Omega^1_{A_t})$ the K\"ahler class corresponding to the polarization $E_t \in \rm H^2(A_t, \bb Z)$. Let $t \in B$ and define $X = \psi^{-1}(t)$. Suppose that $\psi$ is a universal local deformation of $(X, \omega_t)$; in other words, that the Kodaira-Spencer map \begin{align}\label{kodairaspencer}
\rho: T_tB \to \rm H^1(X, T_X)_{\omega}
\end{align}
is an \emph{isomorphism}. Here, $\rm H^1(X, T_X)_{\omega}$ denotes the kernel of the composition 
$$\rm H^1(X, T_X) \to \rm H^2(X, T_X \otimes \Omega^1_X) \to \rm H^2(X, \OO_X)$$ 
of cup-product with $\omega_t$ with the map on cohomology induced by $T_X \otimes \Omega^1_X \to \OO_X$. 

\begin{proposition} \label{satisfycondition}
Condition \ref{criterion} is satisfied for the point $t \in B$ and any element $$W \in \textnormal{Grass}_\CC(k, \rm H^{1,0}(A_t)).$$ 
\end{proposition}
\begin{proof}

By a theorem of Griffiths \cite[Th\'eor\`eme 17.7]{voisin}, the dual $q^\vee$ of the bilinear form $$q: \rm H^{1}(X, \OO_X)^\vee \otimes \rm H^{1}(X, \OO_X)^\vee = \rm H^{1,0}(X) \otimes H^{1,0}(X)  \to (T_tB)^\vee$$ (see equation (\ref{symbil}) in \S \ref{densityintroduction}) factors as
\begin{align*}
T_tB \to \rm H^1(X, T_X) &\to \Hom(\rm H^{0}(X, \Omega^1_X) , \rm H^{1}(X, \OO_X)) \xrightarrow{\sim} \rm H^{1}(X, \OO_X) \otimes \rm H^{1}(X, \OO_X). 
\end{align*}The second arrow is an isomorphism. The third arrow is induced by the polarization. Remark that the following diagram commutes, and that the first row is exact:
\begin{equation}\label{com}
\xymatrixcolsep{1pc}
\xymatrix{
0 \ar[r]  & \rm H^1(X, T_X)_{\omega} \ar[r]  &  \rm H^1(X, T_X)  \ar[r] \ar[d]^{\rotatebox{90}{$\sim$}} & \rm H^2(X, \OO_X) \ar[r]  & 0 \\
0 \ar[r]& T_tB \ar[u]_{\rho}\ar[r]^{q^\vee \hspace{2cm}} & \rm H^{1}(X, \OO_X) \otimes \rm H^{1}(X, \OO_X) \ar[r] & \bigwedge^2\rm H^{1}(X, \OO_X) \ar[r] \ar[u]_{\rotatebox{90}{$\sim$}}  & 0.
}
\end{equation}
We conclude that the Kodaira-Spencer map $\rho$ in (\ref{kodairaspencer}) is an isomorphism if and only if the second row in (\ref{com}) is exact if and only $q^\vee$ induces an isomorphism 
\begin{equation} %\label{analytickodaira}
q^\vee \colon T_tB \xrightarrow{\sim} \Sym^2\rm H^1(X, \OO_X).
\end{equation}
Identify $(T_tB)^\vee$ with $\Sym^2\rm H^{1,0}(X)$ and $q$ with $\rm H^{1,0}(X) \otimes \rm H^{1,0}(X) \to \Sym^2\rm H^{1,0}(X)$, and observe that for any complex vector space $V$ and any $k$-dimensional subspace $W \subset V$, the natural sequence $$0 \to\bigwedge^2W \to W \otimes V \to \text{Sym}^2(V)$$ is exact.
\end{proof}

\subsection{Real structures on deformation spaces}

The goal of this section is to prove that for a universal local deformation of a complex manifold $X$, any real structure on $X$ extends uniquely to a real structure on the local deformation. 
\\
\\
Let $X$ be a compact complex manifold, possibly polarized by the first Chern class $\omega = c_1(\ca L) \in H^2(X, \ZZ)$ of an ample line bundle $\ca L$, and let $$\pi: \mr X \to B \ni 0$$
be a universal local deformation of $X$ (resp. of $(X, \omega)$), where $B$ is a complex analytic space. Let $\sigma: X \to X$ be an anti-holomorphic involution, compatible with $\omega$ in case $X$ is polarized. In the following proposition and its proof, only the germ of $\pi$ in $0 \in B$ plays a role and all statements should be read in this sense.

\begin{proposition}[compare \cite{catanese_frediani_2003}, Section 4] \label{realdef}The real structure $\sigma : X \to X$ extends uniquely to a real structure on the universal local deformation $\pi: \mr X \to B$. In other words, possibly after restricting $(B,0)$, there is a unique couple of anti-holomorphic involutions $(\tau: B \to B,\mr T: \mr X \to \mr X)$ such that $\pi \circ \mr T = \tau \circ \pi$, $\tau(0) = 0$ and $\mr T|_X = \sigma: X \to X$.  
\end{proposition}

\begin{proof}
Consider the complex conjugate analytic spaces $X^\sigma$, $B^\sigma$ and $\mr X^\sigma$ of $X, B$ and $\mr X$ respectively (see \cite[Chapter I, Definition 1.1]{silholsurfaces} for the definition of complex conjugate analytic variety; the definition for general analytic spaces is similar). There is an induced local deformation $$\pi^\sigma\colon \mr X^\sigma \to B^\sigma\ni 0$$ of the manifold $\left(\mr X^\sigma \right)_0 = \left(\mr X_0 \right)^\sigma = X^\sigma$. This local deformation is universal. 
%nd $\pi^\sigma$ is a universal local deformation of $X^\sigma$. 

The anti-holomorphic involution $\sigma: X \to X$ induces a biholomorphic function $\phi: X^\sigma \to X$ such that $ \phi \circ \phi^\sigma = \id: X \to X$, hence the fibers of $\pi$ and $\pi^\sigma$ above $0$ are isomorphic via $\phi$. By the universal properties of $\pi$ and $\pi^\sigma$, this means that, possibly after restricting $B$ around $0$, there is a unique pair of biholomorphisms $(f: B^\sigma \to B, g: \mr X^\sigma \to \mr X)$ making the following diagram cartesian: 
\begin{equation}
\xymatrix{
\mr X^\sigma \ar[r]^{g} \ar[d]^{\pi^\sigma} & \mr X \ar[d]^\pi \\
B^\sigma \ar[r]^{f} & B.
}
\end{equation}
Moreover, $f(0) = 0$ and $g|_{X^\sigma}: X^\sigma \to X$ is the map $\phi$. Applying the functor $(-)^\sigma$ to this diagram, we obtain a pair of cartesian diagrams
\begin{equation} \label{getadiagram}
\xymatrix{
\mr X \ar[d]^\pi\ar[r]^{g^\sigma } & \mr X^\sigma \ar[r]^{g} \ar[d]^{\pi^\sigma} & \mr X \ar[d]^\pi \\
B \ar[r]^{f^\sigma } & B^\sigma \ar[r]^{f} & B
}
\end{equation}such that $(f \circ f^\sigma)(0) = 0$, and such that the map $$(g \circ g^\sigma)|_{X}\colon X \to X^\sigma \to X$$ equals $\phi \circ \phi^\sigma = \id$. We must have $f \circ f^\sigma = \id$ and $g \circ g^\sigma = \id$. By composing $f^\sigma: B \to B^\sigma$ (resp. $g^\sigma: \mr X \to \mr X^\sigma$) with the canonical anti-holomorphic map $B^\sigma \to B$ (resp. $\mr X^\sigma \to \mr X$), we obtain the desired involutions $\tau$ and $\mr T$.
%$\tau: B \to B$ and $\mr T: \mr X \to \mr X$. 
\end{proof}

%\section{Real deformation spaces}

\section{Analytic families of Jacobians}

The goal of this section is to prove that a real structure on a family of curves extends canonically to a real structure on the corresponding family of Jacobians.
\\
\\
Let $B$ be a simply connected complex manifold and let $\pi: \mr X \to B$ be a family of compact Riemann surfaces of genus $g$. The relative Jacobian
\begin{equation} \label{jacobian} 
\psi \colon J_{\mr X} =  \underline{\Pic^0}(\mr X/B) \to B
\end{equation}
of $\pi$ arises from following exact sequence of sheaves on $\mr X$:
\begin{equation} \label{gsheaves}
0 \longrightarrow \bb Z \longrightarrow \OO_{\mr X} \xrightarrow{\text{exp}(2\pi iz)} \OO_{\mr X}^* \longrightarrow 0.
\end{equation}
Indeed, sequence (\ref{gsheaves}) induces an inclusion $R^1 \pi _* \bb Z \to R^1 \pi _* \OO_{\mr X}$, where we identify $R^1 \pi _* \bb Z$ with its \'etal\'e space and $R^1 \pi _* \OO_{\mr X}$ with its corresponding holomorphic vector bundle; define $J_{\mr X} = R^1 \pi _* \OO_{\mr X}/ R^1 \pi _* \bb Z$. 

We claim that family (\ref{jacobian}) is a polarized holomorphic family of $g$-dimensional complex abelian varieties. Indeed, $\psi: J_{\mr X} \to B$ admits a section $s: B \to J_{\mr X}$, corresponding to the line bundle $\OO_{J_{\mr X}} \in \Pic^0({\mr X})$ (recall that $J_{\mr X}$ represents the relative Picard functor of degree $0$). For $0 \in B$, we have the Riemann form $$E_{0} = - \langle \;\; , \;\; \rangle: \bigwedge^2\rm H_1({\mr X}_{0}, \bb Z) \to \bb Z.$$ The trivialization $R^2\pi_*\bb Z \cong \rm H^2(J_{{\mr X}_{0}}, \bb Z)$ implies that 
\begin{align*}
\Hom_{\bb Z}(\wedge^2  \rm H_1({\mr X}_{0}, \bb Z), \bb Z) &= \Hom_{\bb Z}( \wedge^2  \rm H_1(\Jac({\mr X}_{0}), \bb Z), \bb Z)  \\
&=\rm H^2({J_{\mr X}}_{0}, \bb Z)  \cong \rm H^0(B, R^2\pi_*\bb Z).
\end{align*}
In this way, $E_{0}$ extends to a polarization on the Jacobian family (\ref{jacobian}).

\begin{lemma} \label{realjacobianstructure}
Consider the polarized family of Jacobian varieties $$(\psi: J_{\mr X} \to B, s: B \to J_{\mr X}, E \in R^2\psi_*\bb Z)$$ defined above. If the relative curve $\pi$ admits a real structure $(\sigma: B \to B, \sigma': \mr X \to \mr X)$, then the polarized relative Jacobian $\psi: J_{\mr X} \to B$ admits a real structure compatible with $\sigma$, in the sense that there exists an anti-holomorphic involution $\Sigma: J_{\mr X} \to J_{\mr X}$ such that
 $$(i) \white \psi \circ \Sigma = \sigma \circ \psi, \white\white (ii) \white \Sigma \circ s = s \circ \sigma, \white\white (iii)\white \Sigma^\ast(E) = -E.$$
\end{lemma}
\begin{proof}
Write $\bb V_{{\bb Z}} = R^1 \pi _* {\bb Z}$ and $\ca H = \bb V_{{\bb Z}}  \otimes_{{\bb Z}} \OO_{B}$. Then $\ca H$ is endowed with a filtration by the holomorphic subbundle $F^1\ca H  \subset \ca H$, and we have $$\ca H / F^1\ca H = \ca H^{0,1} = R^1 \pi _* \OO_{\mr X}.$$ By the strategy in Section \ref{realfamily}, the real structure $(\sigma, \sigma')$ on $\pi: \mr X \to B$ induces an anti-holomorphic involution 
$
F_{dR}: \ca H \to \ca H$ compatible with $\sigma$ and preserving $\ca H^{0,1}$ and $\bb V_{{\bb Z}}$. Since $J_{\mr X} =  \ca H^{0,1} / \bb V_{{\bb Z}} $, the involution $F_{dR}$ induces an anti-holomorphic involution 
$
\Sigma: J_{\mr X} \to J_{\mr X}.  
$
By construction of $F_{dR}$ in Section \ref{realfamily}, we have $\psi \circ \Sigma = \sigma \circ \psi$. For $t \in B$, $\Sigma$ induces an anti-holomorphic map $\Sigma: \Jac(\mr X_{\sigma(t)}) \to \Jac(\mr X_{t})$. We need to prove that 
$
\Sigma^\ast(E_t) = - E_{\sigma(t)}$ for the map $$ \Sigma^\ast \colon \rm H^2( \Jac({\mr X}_t) , \bb Z) \to \rm H^2( \Jac(\mr X_{\sigma(t)}) , \bb Z).$$ This follows from the commutativity of the following diagram:
$$
\xymatrixcolsep{5pc}
\xymatrix{
\rm H^1({\mr X}_t, {\bb R}) \otimes \rm H^1({\mr X}_t, {\bb R}) \ar[d]^{(\sigma')^\ast \otimes (\sigma')^\ast} \ar[r]^{\hspace{1cm}\cup} & \rm H^2({\mr X}_t, {\bb R}) \ar[d]^{(\sigma')^\ast} \ar[r]^{\sim} & {\bb R} \ar[d]^{-1} \\
  \rm H^1(\mr X_{\sigma(t)}, {\bb R}) \otimes   \rm H^1(\mr X_{\sigma(t)}, {\bb R})  \ar[r]^{\hspace{1cm}\cup} &  \rm  H^2(\mr X_{\sigma(t)}, {\bb R}) \ar[r]^{\sim} &  {\bb R}.  } 
$$
The square on the left commutes because pullback commutes with cup-product, and the square on the right commutes because $\sigma' \colon \mr X_{\sigma(t)} \xrightarrow{\sim} \mr X_{t}$ is anti-holomorphic thus reverses the orientation since $\dim(\mr X_t) = \dim(\mr X_{\sigma(t)}) = 1$ is odd. 
\end{proof}

\section{Density in real moduli spaces}\label{provetheorem2}

The goal of this section is to prove Theorem \ref{theorem2} by applying the results of the previous sections. %The structure of Section \ref{provetheorem2} is as follows. %We start by recalling the construction of the topology on $\mr A_g^\RR$ as in \cite{grossharris} and \cite{silholsurfaces}, and once that is done, we prove Theorem \ref{theorem2}.\hyperlink{theoremA}{A}. We proceed by recalling the construction of the topology on $\mr M_g^\RR$ as in \cite{seppalasilhol2} and then prove Theorem \ref{theorem2}.\hyperlink{theoremB}{B}. Finally, we finish Section \ref{provetheorem2} by proving Theorem \ref{theorem2}.\hyperlink{theoremC}{C}.

\subsection{Density in $\va{\ca A_g(\RR)}$} \label{sec:densityinag}

\begin{proof}[Proof of Theorem \ref{theorem2}.\textcolor{CTlink}{A}]
Let $\tau \in \mr T(g)$ and $Z \in \bb H_g^{\Sigma_\tau}$ (see Section \ref{modofrealAV}). Let $$(X, \omega\in \rm H^2(X, \ZZ))$$ be a principally polarized complex abelian variety with symplectic basis with period matrix $Z$. Then $(X, \omega)$ admits a unique real structure $\sigma: X \to X$, compatible with $\omega$ and the symplectic basis \cite[Section 9]{grossharris}. There exists a $\Sigma_\tau$-invariant connected open neighborhood $B \subset \bb H_g$ of $Z \in \bb H_g$, and a universal local deformation $$\pi: \mr X \to B \ni Z$$ of the polarized complex abelian variety $(X, \omega)$. By Proposition \ref{realdef}, possibly after restricting $B$ around $Z$, the real structure $\sigma: X \to X$ extends uniquely to a real structure on the polarized family $\pi$, which, by uniqueness, is compatible with the real structure $\Sigma_\tau \colon B \to B.$ By Proposition \ref{satisfycondition}, Condition \ref{criterion} is satisfied. By Theorem \ref{th:maindensitytheorem}, the subset 
\[R_k \cap B^{\Sigma_\tau}  \subset B^{\Sigma_\tau}\]
is dense in $B^{\Sigma_\tau}$. It follows that $R_k$ is dense in $\bb H_g^{\Sigma_\tau}$. 
\end{proof}

\subsection{Density in $\va{\ca M_g(\RR)}$}

\begin{proof}[Proof of Theorem \ref{theorem2}.\textcolor{CTlink}{B}] 

Suppose that $g \geq 3$, let $j \in \mr J(g)$ and consider a point $0 \in \ca T_t^{\sigma_j}$ (see Section \ref{modofrealAC}). Let $(X, [f])$ be a complex Teichm\"uller curve of genus $g$ that gives rise to the point $0$. By \cite{seppalasilhol2}, there is a unique real structure $\sigma: X\to X$ which is compatible with the Teichm\"uller structure $[f]$ and the involution $\sigma_j: \Sigma \to \Sigma$. Moreover, there exists a $\sigma_j$-invariant simply connected open subset $B \subset \ca T_g$ of $0$ in the Teichm\"uller space $\ca T_g$, and a Kuranishi family $$\pi: \mr X \to B \ni 0$$ of the Riemann surface $X$. By Proposition \ref{realdef}, up to restricting $B$ around $0$, the real structure $\sigma: X\to X$ extends uniquely to a real structure 
\[
(\tau \colon B \to B , \ca T \colon \mr X \to \mr X )
\]
on the Kuranishi family $\pi$ such that $\tau(0) = 0$. By uniqueness, $\tau: B \to B$ coincides with $\sigma_j$. By Lemma \ref{realjacobianstructure}, the real structure $(\tau, \mr T)$ induces a real structure 
\[
(\tau \colon B \to B, \Sigma \colon J_{\mr X} \to J_{\mr X})
\] on the Jacobian $J_{\mr X} \to B$ of the curve $\pi: \mr X \to B$. 

Let $k \in \{1,2,3\}$. Observe that, by Theorem \ref{th:maindensitytheorem}, it suffices to prove that Condition \ref{criterion} holds in $B$. That is, we need to show that there exists an element $t \in B$ and a $k$-dimensional complex subspace $$W \subset \rm H^{1,0}(\Jac(\mr X_t))  = \rm H^{1,0}(\mr X_t)$$ such that the sequence 
$$
0 \to \bigwedge^2 W \to W \otimes H^{1,0}(\mr X_t) \to (T_tB)^\vee
$$ is exact. The family
$
\pi: \mr X \to B
$
is a universal local deformation of $\mr X_t$ for each $t \in B$, hence $ T_tB \cong \rm H^1(\mr X_t, T_{\mr X_t})$. By \cite[Lemme 10.22]{voisin}, the dual of $$q: H^{1,0}(\mr X_t) \otimes H^{1,0}(\mr X_t) \to (T_t{B})^\vee$$ is nothing but the cup-product $$
\rm H^0(K_{\mr X_t}) \otimes \rm H^0(K_{\mr X_t}) \to \rm H^0(K_{\mr X_t}^{\otimes 2}).$$
We are reduced to the claim that for each $k \in \{1,2,3\}$, there exists $t \in B$ and a $k$-dimensional subspace $W \subset \rm H^0(K_{\mr X_t})$ such that the following sequence is exact:
\begin{equation} \label{exseq}
0 \to \wedge^2 W \to W \otimes \rm H^0(K_{\mr X_t}) \to \rm H^0(K_{\mr X_t}^{\otimes 2}).
\end{equation}
In \cite[\emph{Proof of Theorem (3)}]{Colombo1990}, Colombo and Pirola consider the moduli space of complex genus $g \geq 3$ curves $$\ca M_g(\CC) = \Gamma_g \sm \ca T_g$$ to prove the complex analogue of Theorem \ref{theorem2}.\hyperlink{theoremB}{B}. They show that there exists a moduli point $p = [C] \in \ca M_g(\CC)$ and a $k$-dimensional complex subspace $W \subset \rm H^0(K_C)$ such that (\ref{exseq}) is exact. This implies that Condition \ref{criterion} is satisfied for some point $t \in \ca T_g$. Since Condition \ref{criterion} is open for the Zariski topology on $\ca T_g$, it is dense for the euclidean topology, hence Condition \ref{criterion} holds for some $t \in B$.  
\end{proof}

\subsection{Real plane curves covering an elliptic curve}

\begin{proof}[Proof of Theorem \ref{theorem2}\textcolor{CTlink}{.C}] 
Fix $d \in \bb Z_{\geq 3}$. Define $N = {d+2 \choose 2}$ and let $$\mr B({\bb C}) \subset \rm H^0(\bb P^2_{\bb C}, \OO_{\bb P^2_{\bb C}}(d)) \cong \bb C^N$$ be the Zariski open subset of non-zero degree $d$ homogeneous polynomials $F$ that define smooth plane curves $\{F = 0 \} \subset \bb P^2({\bb C})$. Consider the universal plane curve $$\mr B({\bb C}) \times \bb P^2({\bb C}) \supset \mr S(\CC) \xrightarrow{\pi} \mr B(\CC).$$ %The complex vector space $\rm H^0(\bb P^2_{\bb C}, \OO_{\bb P^2_{\bb C}}(d))$ has a real structure, i.e. $\Gal({\bb C} / {\bb R})$ acts anti-linearly on it and this action preserves the space $\mr B({\bb C})$. The induced action on $\mr B({\bb C}) \times \bb P^2({\bb C})$ preserves in turn $\mr S({\bb C})$, and the morphism $\pi : \mr S({\bb C}) \to \mr B({\bb C})$ is Galois equivariant. 
The map $\pi$ is induced by a morphism of real varieties $\mr S \to \mr B$. %By the above, $\mr B$, $\mr S$ and $\mu$ are actually defined over $\bb R$. 
For a projective, flat morphism of locally Noetherian schemes with integral geometric fibers, the relative Picard scheme exists \cite[\S V, 3.1]{FGA}. We obtain an abelian scheme 
$$
\tn{\underline{Pic}}_{\mr S/\mr B}^0 \to \mr B
$$ of relative dimension $g = (d-1)(d-2)/2$ over $\bb R$. By Theorem \ref{th:maindensitytheorem}, it suffices to %satisfy Condition \ref{criterion}. In other words, we need to 
prove the existence of an element $t \in \mr B({\bb C})$ for which there exists a non-zero $v \in \rm H^{1,0}( \tn{Jac}(\mr S_t({\bb C})))$  such that the map $$\langle v \rangle \otimes \rm H^{1,0}(\tn{Jac}(\mr S_t({\bb C}))) \to (T_t\mr B({\bb C}))^\vee$$ is injective. This is done in \cite[\textit{Proof of Proposition (6)}]{Colombo1990}. %where Colombo and Pirola prove the complex analogue of Theorem \ref{theorem2}.\hyperlink{theoremC}{C}. 
In fact, the element $t \in \mr B(\CC)$ attached to the degree $d$ Fermat curve $F = X_0^d + X_1^d + X_2^d$ satisfies this criterion (compare \cite[Proposition 3]{kim}).
\end{proof}

\cleardoublepage
% Note that depending on your settings in the table of contents, subsections and subsubsections might appear virtually identical.
% Make sure to set the ToC depth and the numbering depth in the ToC the way you want.
\chapter{Glueing real ball quotients}\label{ch:glueing}

\section{Introduction}

In the previous Chapters \ref{ch:realmodulispaces} and \ref{ch:density}, we often encountered (and made use of) the fact that for each connected component of the moduli space of real abelian varieties, there is a natural morphism that identifies this component with the quotient $\Gamma \setminus M$ of a real-analytic manifold by a properly discontinuous group action. The same goes for the moduli space of real algebraic curves. In the next Chapter \ref{ch:binaryquintics}, we will see that something similar is true for the real moduli space $\ca M_0(\RR)$ of smooth binary quintics. One difference stands out: each component $\mr M_i$ of $\ca M_0(\RR)$ identifies only with an open subset $\Gamma_i \setminus (M - \mr H_i)$ of a certain quotient $\Gamma_i \setminus M$, instead of the entire quotient space. In this case, $M = \RR H^2$ is the real hyperbolic plane, and the open sets are obtained by removing unions $\mr H_i$ of lower-dimensional geodesic subspaces. 

This difference actually works in our favour. The larger moduli space $\ca M_s(\RR)$ of stable real binary quintics is connected, which raises the dream that: 
\begin{enumerate}
\item  \label{gluestepone} One can, one way or another, glue the quotient spaces $\Gamma_i \setminus \RR H^2$ into some larger hyperbolic quotient $\Gamma \setminus \RR H^2$; and that
\item \label{gluesteptwo} The various isomorphisms $\mr M_i \cong \Gamma_i \setminus (\RR H^2 - \mr H_i)$ extend to an isomorphism $\ca M_s(\RR) \cong \Gamma \sm \RR H^2$. 
\end{enumerate}
%, and that the various isomorphisms $\mr M_i \cong \Gamma_i \setminus (\RR H^2 - \mr H_i)$ extend to an isomorphism $\overline{\mr M}_\RR \cong \Gamma \sm \RR H^2$. 
\noindent
It turns out that this dream can be realized. 
%\begin{enumerate}
%\item \label{gluestepone} Glue the hyperbolic quotient spaces $\Gamma_i \setminus \RR H^2$ into one large space $\Gamma \setminus \RR H^2$, that contains all the copies $ \Gamma_i \setminus (\RR H^2 - \mr H_i)$. 
%\item \label{gluesteptwo} Proof that the period maps $\mr M_i \cong \Gamma_i \setminus (\RR H^2 - \mr H_i)$ extend to an isomorphism $\overline{\mr M}_\RR \cong \Gamma \sm \RR H^2$. 
%\end{enumerate}
%\noindent
Chapter \ref{ch:binaryquintics} will be devoted to Step \ref{gluesteptwo}. In the current Chapter \ref{ch:glueing}, we focus on Step \ref{gluestepone}. 
\\
\\
To carry out Step \ref{gluestepone}, it seemed natural not to restrict our attention to the two-dimensional case, but to glue real ball quotients in any dimension. Unitary Shimura varieties turned out to provide a suitable framework for doing so. %In this context, we construct a method of gluing real arithmetic ball quotients along a hyperplane arrangement, 
We build upon work of Allcock, Carlson and Toledo \cite{realACTsurfaces}. Let us outline the construction. 

\noindent
Let $K$ be a CM field of degree $2g$ over $\QQ$ with ring of integers $\OO_K$, and let $\Lambda$ be a finite free $\OO_K$-module equipped with a hermitian form $h: \Lambda \times \Lambda \to \OO_K$. Suppose that $h$ has signature $(n,1)$ with respect to an embedding $\tau: K \to \CC$ and is definite other infinite places of $K$. Let $\CC H^n$ be the space of negative lines in $\Lambda \otimes_{\OO_K, \tau} \CC$ and $P\Gamma = \Aut(\Lambda, h) / \mu_K$ where $\mu_K \subset \OO_K^\ast$ is the group of finite units in $\OO_K$. Let $P\mr A$ be the quotient of the set of anti-unitary involutions $\alpha: \Lambda \to \Lambda$ by $\mu_K$. %, and let $P\mr A$ be the quotient of $\mr A$ . %We let $\widetilde Y = \coprod_{\alpha \in P\mr A} \RR H^n_\alpha$, where $\RR H^n_\alpha = \left(\CC H^n\right)^\alpha$. 

Consider the hyperplane arrangement $\mr H = \cup_{h(r,r) = 1} \langle r_\CC \rangle ^\perp \subset \CC H^n$ and assume
\begin{condition} \label{conditionast}
% $(\textcolor{blue}{\ast})$: 
Different hyperplanes intersect orthogonally or not at all, c.f. \cite{orthogonalarrangements}. 
\end{condition}
%assume Condition \ref{orthogonal}: the intersection of different hyperplanes is empty or orthogonal. %saying that the intersection of any two non-disjoint but different hyperplanes in $\mr H$ is orthogonal.
\noindent
For example, this holds %under %$h$ is definite at all places different from $\tau$ and if 
%when the different ideal is generated by in element 
under a condition on the CM field $K$ (see Theorem \ref{th:conditionsimplyhypothesis}) 
%depending on the CM type and 
%when the CM type is primitive, $ and 
%nder the assumption that any two norm $1$ vectors $r,t \in \Lambda$ such that $H_r \neq H_t$ and $H_r \cap H_t \neq \emptyset$ are orthogonal
%the different ideal $\mf D_K \subset \OO_K$ is generated by an $\eta \in \OO_K$ such that $\overline{\tau(\eta)} = - \tau(\eta)$, see Theorem \ref{th:conditionsimplyhypothesis}. The condition on $\mf D_K$ is 
satisfied when $K$ is cyclotomic or quadratic (see Lemma \ref{lemma:discr}). (In fact, condition \ref{conditionast} is \textit{always} satisfied if one is willing to adapt the definition of $\mr H$, see Remark \ref{remark:avoidcondition}.)
%
%the hyperplanes $H_r \subset \CC H^n$ of the form $H_r = \langle r \rangle^\perp$ for a norm one vectors $r \in \Lambda$ either intersect orthogonally or do not intersect at all,

We claim that there is a canonical way to glue the different copies 
\[
\RR H^n_\alpha: = \left(\CC H^n\right)^\alpha \subset \CC H^n, \quad \alpha \in P\mr A
\] of the real hyperbolic space $\RR H^n$ along the hyperplane arrangement $\mr H$. See Remark \ref{rem:gluing} for the precise formulation of the equivalence relation. This gives a topological space which we denote by $Y$, acted upon by $P\Gamma$. Define $P\Gamma_\alpha \subset P\Gamma$ to be stabilizer of $\RR H^n_\alpha$. The precise goal of Chapter \ref{ch:glueing} is to prove the following theorem. 
%With these definitions in place, we can state:
\begin{theorem} \label{th:theorem03}
The topological space $P\Gamma \setminus Y$ admits a metric that makes it a complete path metric space. With respect it, the natural map $P\Gamma \setminus Y \to P\Gamma \setminus \CC H^n$ is a local isometry. %Each point $x \in P\Gamma \setminus K$ admits an open neighborhood $U \subset P\Gamma \setminus K$ which is isometric to the quotient of an open subset $V \subset H^n$ by a finite group of isometries. 
In fact, the metric underlies a real hyperbolic orbifold structure on $P\Gamma \setminus Y$, such that \[ \coprod_{\alpha \in P\Gamma \setminus P\mr A}  [P\Gamma_\alpha \setminus \left(\RR H^n_\alpha - \mr H \right)]  \subset P\Gamma \setminus Y\] is an open suborbifold and such that for each connected component $C \subset P\Gamma \setminus Y$ there is a lattice $P\Gamma_C \subset \textnormal{PO}(n,1)$ and an isomorphism of real hyperbolic orbifolds $C \cong [P\Gamma_C \setminus  \RR H^n]$. %\item We have $\coprod_{\alpha \in C\mr A}  [P\Gamma_\alpha \setminus \left(H^n_\alpha - \mr H \right)]  isomorphism $P\Gamma \setminus K \cong \coprod_{\pi_0(P\Gamma \setminus K)}[P\Gamma_X \setminus H^n]$ restricts to an isomorphism 
%There is an open embedding $\coprod_{\alpha \in C\mr A}  [P\Gamma_\alpha \setminus \left(H^n_\alpha - \mr H \right)] \hookrightarrow P\Gamma \setminus K$ with dense image. % of real hyperbolic orbifolds. 
%Let $X_\RR^0 := \coprod_{\alpha \in C\mr A}  [P\Gamma_\alpha \setminus \left(H^n_\alpha - \mr H \right)]$. Then $X_\RR^
\end{theorem}
%It may happen that the topological space $P\Gamma \setminus Y$ in Theorem \ref{th:theorem03} is connected, and 
%see Remark. 
% is the standard hermitian form of signature $(n,1)$ on $\OO_K^4$  or $\OO_K^4$ , 
%$h = \textnormal{diag}(1,1,1)$ or $h = \textnormal{diag}(1,1,1,1)$ on the lattice $\Lambda$
\noindent
For some moduli stacks of smooth hypersurfaces $\ca M_0$ one can apply Theorem \ref{th:theorem03} to the hermitian lattice $\Lambda$ that arises as the cohomology of the cover of projective space ramified along a member of the moduli space. Let $\ca M_s$ be the stack of GIT stable hypersurfaces. If the discriminant $\Delta  = \ca M_s(\CC) - \ca M_0(\CC)$ is a normal crossings divisor and the period map induces an isomorphism of analytic spaces $$\ca M_s(\CC) \cong P\Gamma \setminus \CC H^n$$ identifying $\Delta$ with $P\Gamma \setminus \mr H$, then there is a real period homeomorphism $$\ca M_s(\RR) \cong P\Gamma \setminus Y.$$ For cubic surfaces and binary sextics, this is the content of \cite{realACTsurfaces, realACTnonarithmetic}. For binary quintics, this yields the main Theorem \ref{th:theorem02} of Chapter \ref{ch:binaryquintics} (via Theorem \ref{th:realstableperiod}).

\begin{remarks}
\begin{enumerate}
\item 
The lattice $P\Gamma_C$ attached to a component $C \subset P\Gamma \setminus Y$ can be non-arithmetic. Indeed, such is the case for $K = \QQ(\zeta_5)$ and $h = \textnormal{diag}(1, 1, \frac{1 - \sqrt{5}}{2})$ by Remark \ref{remark:takeuchi} and Theorem \ref{th:calculatemonodromyshimura}, and for $K = \QQ(\zeta_3)$ and $h = \textnormal{diag}(1, \dotsc, 1, -1)$ for $n = 3$ \cite{realACTnonarithmetic} and $n = 4$ \cite{realACTsurfaces}. %Remark also that the arrangement $\mr H \subset \CC H^n$ is automatically orthogonal if one defines it using certain fractional ideals $\mf a$ and maps $\mf a \to \Lambda$ instead of norm $1$ vectors $r \in \Lambda$, and that the gluing construction as well as Theorem \ref{th:theorem03} generalize to this setting, see Remark \ref{remark:avoidcondition}.
\item Our glueing construction relies on Condition \ref{conditionast}, saying that the hyperplane arrangement $\mr H \subset \CC H^n$ is an \textit{orthogonal arrangement} in the sense of \cite{orthogonalarrangements}. 
%Such arrangements are interesting in their own right. Indeed, if $n > 1$ then the orbifold fundamental $\pi_1^\textnormal{orb} \left( P\Gamma \setminus \left( \CC H^n - \mr H \right) \right)$ is not a lattice in any Lie group with finitely many connected components [\emph{loc.cit.}, Theorem 1.2]. In particular, neither $\pi_1 \left( P_0 / \mf S_5  \right)$ nor $\pi_1^\textnormal{orb}\left( \mr M_\CC \right)$ is a lattice in any Lie group with finitely many connected components. 
Condition \ref{conditionast} %The hyperplane arrangement $\mr H \subset \CC H^n$ is an orthogonal arrangement 
is in turn implied by the following condition, satisfied by quadratic and cyclotomic CM fields (see Section \ref{unitaryshimura}):
\end{enumerate}
\end{remarks}
\begin{condition}\label{conditionastast}
The different ideal $\mf D_K \subset \OO_K$ (see e.g. \cite[Chapter III]{Neukirch}) is generated by an element $\eta \in \OO_K - \OO_F$ such that $\eta^2 \in \OO_F$. 
\end{condition}
\begin{remarks}
\begin{enumerate}
\item 
In fact, there is always a canonical orthogonal arrangement $\ca H \subset \CC H^n$ attached to $h$ in such a way that $\ca H = \mr H$ when Condition \ref{conditionastast} holds, see Remark \ref{remark:avoidcondition}. Moreover, one can glue the different copies $\RRH^n_\alpha$ of real hyperbolic $n$-space along the hyperplane arrangement $\ca H$ 
obtaining a complete hyperbolic orbifold as in Theorem \ref{th:theorem03}, but we will not prove this. 
\item 
%Also our gluing construction (the equivalence relation on the disjoint union Definition \ref{def:conditions} and the orbifold Theorem \ref{th:theorem03}) generalizes to a gluing construction using the arrangement $\ca H \subset \CC H^n$, but we will not prove this. 
%In fact, if 
% and can in fact be avoided altogether if one is willing to slightly alter the definition $\mr H$
%Remark that Theorem \ref{th:conditionsimplyhypothesis} gives many examples of orthogonal arrangements. 
%The quotient $V(\CC)  = G \setminus X$ of a bounded symmetric domain $X$ by an arithmetic group $G$ often has a model defined over a number field. 
In \cite{shimurarealpoints}, Shimura studied real points of an arithmetic quotient of a bounded symmetric domain. Gromov and Piatetski-Shapiro developed a construction of glueing hyperbolic quotients in \cite{gromovshapiro}. The difference between their glueing method and ours is explained in \cite[Section 13, Remark (1)]{realACTsurfaces}. 
\end{enumerate}
%On the Real Points of an Arithmetic Quotientx
%of a Bounded Symmetric Domain
%Shimura assumes this number field to be totally real and studies the real locus $V(\RR)$. However,  %apart from its number-theoretical nature, 
%that $V$ is defined over a totally real number field %However, such a study is quite different from ours since it concerns 
%and 
%though $V(\RR)$ is an arithmetic subquotient of $V(\CC)$. % whereas $P\Gamma_C \setminus \RR H^n \to P\Gamma \setminus \CC H^n$ is finite-to-one and $P\Gamma_C$ is not always arithmetic. 
%where a non-arithmetic quotient is obtained by gluing non-commensurable arithmetic pieces. 
%though we have not investigated any precise relations. 
%the spaces $P\Gamma_C \setminus \RR H^n$ being assembled from some incommensurable arithmetic pieces as in \textit{loc. cit.} See also 
% and also because the components of $V(\RR)$ are generally arithmetic quotient spaces. 
\end{remarks}

\section{The glueing construction} \label{gluing}

Consider a hermitian form $h$ on a finite free module over the ring of integers of a CM field $K$ with hyperbolic signature for some embedding of $K$ into $\CC$. As is well-known, there is a complex ball quotient $P\Gamma \setminus \CC H^n$ attached to $h$ in a canonical way. We prove in this section that there is also a natural real ball quotient $P\Gamma_\RR \setminus \RR H^n$ attached to $h$ (or a disjoint union of those). As in the complex case, $P\Gamma_\RR \setminus \RR H^n$ is sometimes a moduli space for real varieties. To define this space, one considers the real hyperbolic spaces $\left( \CC H^n \right)^\alpha$ attached to anti-unitary involutions $\alpha: \Lambda \to \Lambda$, glues them along a hyperplane arrangement $\mr H \subset \CC H^n$, takes the quotient by $P\Gamma$, and defines a complete real hyperbolic orbifold structure on the result. 
% defined by gluing the real ball quotients $\textnormal{Stab}$
%In contrast to the complex case, it may happen that the lattice $P\Gamma_\RR \subset \textnormal{PO}(n,1)$ is non-arithmetic. %If this is the case, then $P\Gamma_\RR \setminus \RR H^n$ is not the quotient of the fixed locus of an anti-holomorphic involution $\alpha: \CC H^n \to \CC H^n$ by its stabilizer in $P\Gamma$ (though $P\Gamma_\RR \setminus \RR H^n$ contains a dense open subset of such a quotient in its interior for every real structure $\alpha$ defined over the lattice).  
%It contains an open dense subset of the real ball quotient attached to any real structure defined over the lattice in its interior, but its structure is more sophisticated: 
%there is also a natural real ball quotient attached to 
%we attach to such a form $h$ 
%suppose that the signature of $h$ is hyperbolic with respect to a complex embedding, and such that any two hyperplanes defined by norm $1$ vectors are disjoint, equal, or intersect orthogonally. 

\subsection{The set-up} \label{set-up}

Let $g$ and $n$ be positive integers. Let $K$ be a CM field of degree $2g$ over $\QQ$ and let $F \subset K$ be its totally real subfield. Let $\OO_K$ (resp. $\OO_F$) be the ring of integers of $K$ (resp. $F$) and let $\sigma \in \Gal(K/F)$ be the non-trivial element. We will often write 
\[
\sigma \colon K \to K, \quad \sigma(x) = \overline{x}. 
\]
Fix a set of embeddings 
\begin{align} \label{setofembeddings}
\Psi = \set{\tau_i \colon K \to \CC}_{1 \leq i \leq g} \quad \mid \quad \Psi \cup \Psi \sigma = \{\tau_i, \tau_i \sigma\}_{1 \leq i \leq g} = \Hom(K, \CC).
\end{align}
Let $\Lambda$ be a free $\OO_K$-module of rank $n+1$ equipped with a hermitian form 
\[
h \colon  \Lambda \times \Lambda \to \OO_K
\]
of signature $(r_i,s_i)$ with respect to $\tau_i$. In other words, $h$ is linear in its first argument and $\sigma$-linear in its second, and the complex vector space $\Lambda \otimes_{\OO_K, \tau_i} \CC$ admits a basis $\{e_i\}$ such that $(h^{\tau_i}(e_i,e_j))_{ij}$ is a diagonal matrix with $r_i$ diagonal entries equal to $1$ and $s_i$ diagonal entries equal to $-1$. Here \[
h^{\tau_i} \colon \Lambda \otimes_{\OO_K, \tau_i} \CC \times \Lambda \otimes_{\OO_K, \tau_i} \CC \to \CC\] is the hermitian form attached to $h$ and the embedding $\tau_i$. Define 
\[
\tau = \tau_1: K \to \CC, \quad \tn{ and } \quad V = \Lambda \otimes_{\OO_K, \tau} \CC,\]\[ \tn{ and assume that} \quad\quad\quad (r_i, s_i )\quad =\quad \begin{cases}
(n,1) \quad &\tn{ if } i = 1,\\
(n+1, 0) \quad &\tn{ if $2 \leq i \leq g$.}
\end{cases}
\]
 %In the sequel, whenever an explicit embedding $K \to \CC$ is not mentioned then it means we are using $\tau$. 
 %We call $\phi_r^i$  \textit{$\zeta^i$-reflection} of the root $r$; we write $h_r = h_r^1$. 
%and a norm $2$ vector of $\Lambda$ is a \textit{long root}. 
%then for $i \in (\ZZ/m)^\ast$, the $\zeta^i$-reflections in short roots are isometries $\Lambda$. %, as are $-1$-reflections in long roots. 
Let $m$ be the largest positive integer for which the $m$-th cyclotomic field $\QQ(\zeta_m)$ can be embedded in $K$, where $\zeta_m = e^{2 \pi i /m} \in \CC$. Let $\zeta \in K$ be a primitive $m$-th root of unity in $K$, and define
\[
\mu_K = \langle \zeta \rangle \subset \ca O_K^\ast \subset \OO_K. 
\]
Moreover, define $\Gamma$ to be the unitary group of $\Lambda$, and $P\Gamma$ as its quotient by $\mu_K$:
\[
\Gamma = U(\Lambda)(\OO_K) = \Aut_{\OO_K}(\Lambda, h) \quad \tn{ and } \quad  P\Gamma = \Gamma/\mu_K.
\]
A norm one vector $r \in \Lambda$ is called a \textit{short root}. Let $\mr R \subset \Lambda$ be the set of short roots. For $r \in \mr R$, define isometries $\phi_r^i \colon V \to V$ as follows:
\begin{equation*}
\phi_r(x) = x-(1-\zeta)h(x,r) \cdot r, \quad \phi_r^i(x) = x-(1-\zeta^i) h(x,r) \cdot r, \quad i \in (\ZZ/m)^\ast. 
\end{equation*}
Note that $\phi^i_r \in \Gamma$ for $r \in \mr R$, and that $\phi_r^i = \phi_r \circ \cdots \circ \phi_r$ ($i$ times). In particular, $\phi_r^m = \id$. Let $\PP(V)$ be the projective space of lines in $V$, and let
\[
\CC H^n = \set{ \ell = [v] \in \PP(V) \mid h(v,v) < 0} \subset \PP(V)
\] be the space of negative lines in $V$. Define
\begin{equation*}
H_r = \{x \in \CC H^n: h(x,r) = 0 \} \; \tn{ for } \; r \in \mr R, \quad \tn{ and }\quad \mr H = \bigcup_{r \in \mr R} H_r \white
 \subset \white \CC H^n.  
\end{equation*}

\begin{lemma}
The family of hyperplanes $(H_r)_{r \in \mr R}$ is locally finite, so that the hyperplane arrangement $\mr H \subset \CC H^n$ is a divisor of $\CC H^n$.
\end{lemma}
\begin{proof}
See \cite[Lemma 5.3]{beauvillecubicsurfaces}. 
\end{proof}
\noindent
Define an $\ca O_F$-linear map $\alpha: \Lambda \to \Lambda$ to be \textit{anti-unitary} if for all $x,y \in \Lambda$ and $\lambda \in \OO_K$, one has 
$
\alpha(\lambda x) = \sigma(\lambda) \cdot \alpha(x)$ and $ h(\alpha(x), \alpha(y)) = \sigma(h(x,y)) \in \OO_K.
$ Define $\Gamma'$ to be the group of unitary and anti-unitary $\OO_F$-linear bijections $\Lambda \xrightarrow{\sim} \Lambda$. Let $\mr A \subset \Gamma'$ be the set of anti-unitary involutions $\alpha \colon \Lambda \to \Lambda$. %For $\alpha \in \mr A$, $\lambda \in \mu_K$ and $x,y \in \Lambda$, one has $h(\lambda\alpha(x), \lambda \alpha(y)) = \sigma(h(x,y))$. Because $\lambda \alpha = \alpha \sigma(\lambda)$, $(\lambda \alpha)^2 = \id$ hence 
Then 
\begin{align} \label{anti-iso-group}
\mu_K \subset \Gamma \subset \Gamma' \quad - \quad \tn{ define } \quad P\Gamma' = \Gamma'/\mu_K. 
\end{align}
Let $\lambda \in K^\ast$. Observe that 
\begin{align}\label{crucialhypo}
\left(
\lambda \in \OO_K^\ast \tn{ and }
\va{\lambda}^2 = \lambda \cdot \sigma(\lambda)  = 1\right) \quad \iff \quad \lambda \in \mu_K.
\end{align}
Indeed, we have, for any embedding $\varphi \colon K \to \CC$, that
\[
\va{\varphi(\lambda)}^2 = \varphi(\lambda) \cdot \overline{\varphi(\lambda)} = \varphi(\lambda) \cdot \varphi(\sigma(\lambda)) = \varphi(\lambda \cdot \sigma(\lambda)),
\]
where $\overline{\varphi(\lambda)} = \varphi(\sigma(\lambda))$ by \cite[Proposition 1.4]{milneCM}. %Thus, under the hypothesis on the left of (\ref{crucialhypo}), 
Moreover, we have $\va{\varphi(\lambda)} = 1$ for each $\varphi \colon K \to \CC$ if and only if $\lambda$ is a root of $1$, see \cite[Corollary 5.6]{milneANT}.  
\begin{lemma} \label{Gammaembedding}
Let $\Isom(\CCH^n)$ be the group of isometries $f \colon \CC H^n \xrightarrow{\sim} \CC H^n$. The natural homomorphism $P\Gamma' \to \Isom(\CCH^n)$ is injective.
\end{lemma}
\begin{proof}
Let $g \in \Gamma'$ be an element that induces the identity on $\CC H^n$. Let $\{e_i\}_{i = 1, \dotsc, n+1}$ be a basis for $\Lambda$. Consider $\{e_i\}$ as a basis of $V$. There exist $\lambda_i \in \CC^\ast$ such that $g(e_i) = \lambda_i \cdot e_i$. We must have $\lambda_1 = \dotsc = \lambda_{n+1}$. 
If $g \in \Gamma$, this means that $g = \lambda \in \CC^\ast \subset \GL(V)$. The basis $\{e_i\}$ induces a $\CC$-linear isomorphism $V \cong \CC^{n+1}$, hence $\lambda$ acts on $\OO_K^{n+1} \subset \CC^{n+1}$ by multiplication, thus $\lambda \in \OO_K^\ast$. For $r \in \mr R$, we have $1 = h(r,r) = h(g(r), g(r)) = \va{\lambda}^2$, thus $\lambda \in \set{ x \in \OO_K^\ast \mid \va{x} = 1} = \mu_K$ (see (\ref{crucialhypo})). 

Suppose that $g \not \in \Gamma$. Then $g$ induces the map $\sum_i\mu_ie_i \mapsto \lambda \cdot  \sum_i\overline{\mu_i}e_i$ on $V$, which is absurd since then $g \neq \id \in \Isom(\CCH^n)$. 
%with respect to the above basis $\{e_i\}$, but this map does not induce the identity on $\CCH^n$. 
% $V \cong \CC^{n+1}$ sh
%consider its composition $g \circ \alpha$ with any anti-isometric involution $\alpha \colon V \to V$. 
%so $\lambda$ acts on $\OO_K^{n+1}$ by multiplication, we see that 
%the equality $g(e_1 + e_2) = \lambda_{12}\cdot (e_1 + e_2) = \lambda_1e_1 + \lambda_2 e_2$ shows that $\lambda_1 = \lambda_2$. 
% \cite{milneANT}. 
%which implies that $g \in \mu_K\subset \Gamma'$. 
\end{proof}
\noindent
The group $\mu_K$ acts on $\mr A$ by multiplication; define 
\[
P\mr A = \mu_K \setminus \mr A, \quad \tn{ and } \quad C\mr A = P\Gamma \setminus P\mr A,
\]
where $P\Gamma$ acts on $P\mr A$ by conjugation. Any $\alpha \in P\mr A$ defines an anti-holomorphic involution 
\[
\alpha: \CC H^n \to \CC H^n; \quad \tn{ define } \quad  \RR H^n_{\alpha} = (\CC H^n)^\alpha \subset \CC H^n.
\]
For any element $\alpha \in \mr A$, the quadratic form $h|_{V^\alpha}$ on the real vector space $V^\alpha = \Lambda^\alpha\otimes_{\OO_F, \tau|_F}\CC$ has hyperbolic signature. The following lemma is readily proved:\begin{lemma} \label{hyperbolic}
For $\alpha \in \mr A$, let $\PP(V^\alpha)$ be the real projective space of lines in $V^\alpha$, and let $\RRH(V^\alpha) \subset \PP(V^\alpha)$ be the space of negative lines in $V^\alpha$. The canonical isomorphism $\PP(V^\alpha) \cong \PP(V)^\alpha$ restricts to an isomorphism $\RRH(V^\alpha) \cong \RRH^n_\alpha$. \hfill \qed
%The surjection $V^\alpha - \{0\} \to \PP(V^\alpha)$ induces an isomorphism $\RRH(V^\alpha) \cong \RRH^n_\alpha$. 
%the subspace $\RR H^n_\alpha \subset \CCH^n$ is canonically identified with the space of negative real lines in $V^\alpha= \Lambda^\alpha \otimes_{\OO_F, \tau|_F} \RR \subset \Lambda \otimes_{\OO_K, \tau} \CC = V$. \hfill \qed
\end{lemma}
\noindent
We conclude that $\RR H^n_\alpha \subset \CC H^n$ is isometric to the real hyperbolic space of dimension $n$. Finally, we define \[P\Gamma_\alpha = \textnormal{Stab}_{P\Gamma}(\RR H^n_\alpha) \subset P\Gamma\quad \quad \tn{(the stabilizer of $\RR H^n_\alpha$ in $P\Gamma$).}\] %and define $N_{P\Gamma}(\alpha) = \{\phi \in P\Gamma: \phi \circ \alpha = \alpha \circ \phi\}$ to be the normalizer of $\alpha$ in $P\Gamma$.

\subsection{Preliminary results}

%For each $\alpha \in P\mr A$, let $\phi_\alpha: \RR H^n_\alpha \to \CC H^n$ be the canonical inclusion. 
We assume, in the entire Section \ref{gluing}, that the following condition is satisfied: 

\begin{condition} \label{orthogonal}
If $r, t \in \mr R$ are such that $H_r \neq H_t$ and $H_r \cap H_t \neq \emptyset$, then $h(r,t) = 0$. 
\end{condition}

\begin{example}
Theorem \ref{th:conditionsimplyhypothesis} of Section \ref{unitaryshimura} shows that Condition \ref{orthogonal} if the following conditions are satisfied: 
\begin{enumerate}
\item \label{orthogonal:one}The different ideal $\mf D_K \subset \OO_K$ is generated by a single element $\eta \in \OO_K$ such that $\sigma(\eta) = -\eta$ and $\Im(\tau_i(\eta)) > 0$ for every $i$. 
%\item \label{orthogonal:two}We have $(r_i,s_i) = (n+1,0)$ for $2 \leq i \leq g$. 
\item \label{orthogonal:three}The CM type $(K, \Psi)$ is primitive. 
\end{enumerate}
%Consider Conditions \ref{condition:one} and \ref{condition:two} in Section \ref{unitaryshimura}. If these conditions are satisfied, 
We will see that condition \ref{orthogonal:one} is automatically satisfied for quadratic and cyclotomic CM fields $K$, see Lemma \ref{lemma:discr}. 
\end{example}
%\begin{comment}
%Let 
%\[
%\textbf r = (r_1, \dotsc, r_k) \in \mr R^k \;\; \tn{ such that } \;\; x \in H_{r_1} \cap \cdots \cap H_{r_k}, \;\;  H_{r_i} \neq H_{r_j} \;\; \tn{ for all } \;\; r_i \neq r_j. 
%\] 
%The hyperplanes $H_{r_i}$ are orthogonal by Condition \ref{orthogonal}, hence 
%\[
%H_{r_1} \cap \cdots \cap H_{r_k} \subset \CC H^n
%\]
%is a totally geodesic subspace of codimension $k$ in $\CC H^n$. 
%\end{comment}
\noindent
Note that Condition \ref{orthogonal} implies that if $H_{r_1}, \dotsc, H_{r_k}$ for $r_i \in \mr R$ are mutually distinct, and if their common intersection is non-empty, then $\cap_{i = 1}^k H_{r_i} \subset \CC H^n$ is a totally geodesic subspace of codimension $k$. Note also that for any $r \in \mr R$, the element $\phi_r \in \Gamma$ generates a finite subgroup $\langle \phi_r \rangle \subset \Gamma$ of order $m$, and that the restriction of the quotient map $\pi \colon \Gamma \to P\Gamma$ to this subgroup $\langle \phi_r \rangle \subset \Gamma$ is injective. We will abuse notation, by letting $\phi_r \in P\Gamma$ denote the image of $\phi_r \in \Gamma$ in $P\Gamma$. 

\begin{definition} \label{fullset}
Let $\ca H = \{H_r \mid r \in \mr R\}$. For $x \in \CC H^n$, define 
\[
\ca H(x) = \set{H \in \ca H \mid x \in H}, \quad G(x) = \set{ \phi_r^i \in P\Gamma \tn{ for }r \in \mr R, i \in \ZZ/m \mid x \in H_r }.
\]
The hyperplanes $H \in \ca H(x)$ are called the \emph{nodes} of $x$. We say that \emph{$x$ has $k$ nodes} if the cardinality of $\ca H(x)$ is $k$.  
%We say that \textit{$x \in \CC H^n$ has $k$ nodes}, and call the $H_{r_i}$ the \textit{nodes of $f$}. 
\end{definition}
\begin{lemma} \label{Gxlemma}
Let $x \in \CCH^n$ and suppose that $x$ has $k$ nodes. Then $G(x) \cong (\ZZ/m)^k$. 
\end{lemma}
\begin{proof}
Let $r, t \in \mr R$. Then, for $z \in \Lambda$, one has 
\begin{align} \label{reflectionscommute}
\begin{split}
\phi_r^i(\phi_t^j(z)) 
=& \; \phi_r^i \left( z - (1-\zeta^j)h(z,t)\cdot t \right) \\
=& \;z - (1-\zeta^j)h(z,t)\cdot t  - (1- \zeta^i) h\left( \left(z - (1-\zeta^j)h(z,t)\cdot t \right), r \right) \cdot r \\
=& \;z - (1-\zeta^j)h(z,t)\cdot t  - (1 - \zeta^i) h(z, r) \cdot r + (1-\zeta^i)(1- \zeta^j) h(z,t)h(t,r) \cdot r.
\end{split}
\end{align} 
Now suppose that $H_r, H_t \in \ca H(x)$, with $H_r \neq H_t$. By Condition \ref{orthogonal}, we have $h(r,t) = 0$; by (\ref{reflectionscommute}), this implies that $\phi_r^i \circ \phi_t^j = \phi_t^j \circ \phi_r^i $ for each $i,j \in \ZZ/m$. We conclude that the group $G(x)$ is abelian. 

Next, suppose that $H_r = H_t \in \ca H(x)$. To finish the proof, it suffices to show that $\phi_t = \phi_r^i$ for some $i \in \ZZ/m$. To prove this, we introduce Lemma \ref{finitereflectionorders} below. 
\end{proof}
\begin{lemma} \label{finitereflectionorders}
Let $r \in \mr R$. Let $\phi: \CC H^n \to \CC H^n$ be an isometry such that $\phi^m = \id$ and such that $\phi$ restricts to the identity on $H_r \subset \CC H^n$. Then $\phi = \phi_r^i$ for some $i \in \ZZ/m$. 
\end{lemma}
\begin{proof}
%It suffices to show that $\phi \in G(\textbf r)$. %By the above, we have $H_{t_i} = H_{r_{\sigma(i)}}$; to ease notation, write $t = t_i$ and $r =  r_{\sigma(i)}$. We then have two isometries $h_r, h_t \in \Gamma$ % \subset \text{Isom}(\CC H^n)$ 
%that both fix $H_r \subset \CC H^n$ pointwise. %Consider a decomposition $V = \langle r \rangle \oplus \langle r \rangle^\perp$ and choose an isomorphism $V \cong \CC^{n,1} = \langle e_1, \dotsc, e_{n+1} \rangle $ mapping $r$ to $e_1$; this gives 
%Choose an isomorphism $\phi: \CC H^n \cong H^n$ with $\phi(H_r) = {\bb B} H^{n-1}_\CC$, and write $f_r = \phi h_r \phi^{-1}$ and $f_t = \phi h_t \phi^{-1}$. 
Let $\HH^n_\CC$ be the hyperbolic space attached to the standard hermitian space $\CC^{n,1}$ of dimension $n+1$. It is classical that
$$
\text{Stab}_{U(n,1)}(\HH^{n-1}_\CC) = U(n) \times U(1).
$$
Thus, any $\phi \in U(n,1)$ that fixes $\HH^{n-1}_\CC$ pointwise lies in $U(1) = \{z \in \CC^\ast : \va{z}^2 = 1\}$. If $\phi^m = \id$, then $\phi$ lies in the unique subgroup of $U(1)$ that has order $m$.  
%In particular, $\phi \in U(1)$ is contained in the finite order $m$ subgroup of $U(1)$ generated by the isometry $f_r$ that corresponds to $h_r$ under $\CC H^n \cong \HH^n_\CC$. %Therefore $\langle f_r \rangle = \langle g_r \rangle \subset U(1)$. 
%We conclude that $ = f_r^i$ for some $i$, hence $h_t = \phi_r^i$.
\end{proof}
%We say that a set of hyperplanes $\{H_{r_1}, \dotsc, H_{r_k} \}$ attached to roots $r_i \in \mr R$ is a \textit{maximal set of hyperplanes through $x$} if $x \in \cap_{i = 1}^k H_{r_i}$, $H_{r_i} \neq H_{r_j}$ for $i \neq j$ and if whenever $t \in \mr R$ is such that $x \in H_t$, then $H_r = H_{r_i}$ for some $i$. %We call $k$ the \textit{order} of $x$; in other words, $k$ denotes the maximal number of different hyperplanes passing through $x$. 
\noindent
We will also consider $G(x)$ as a subgroup of $\Gamma$ sometimes, which we may do, since for each $[g] \in G(x)$ there is a unique $g \in \Gamma$ such that $\pi(g) = [g]$ for $\pi \colon \Gamma \to P\Gamma$. 
 %Note moreover that $G(x) \subset P\Gamma$ is a finite subgroup. In fact, if $x \in \CC H^n$ has $k$ nodes, then $G(x) \cong (\ZZ/m)^k$. This follows from the fact that the various isometries $\phi_r \colon \Lambda \to \Lambda$, attached to the short roots $r \in \mr R$ that satisfy $H_r \in \ca H(x)$, commute with each other. 
%Then $\{e_i \otimes 1 \mid i \in \{1, \dotsc, n+1\}\} \subset \Lambda \otimes_{\OO_K, \tau} \CC = V$ is a basis for $V$. 
%For each $i$, we have that $g(e_i) = \lambda_i \cdot e_i$ for some $\lambda_i \in \OO_K^\ast$. 
Lemma \ref{Gammaembedding} allows us to view $P\mr A$ as a subset of $\Isom(\CCH^n)$, and also to view the groups $G(x) \subset P\Gamma \subset P\Gamma'$ (for any $x \in \CC H^n$) as subgroups of $\Isom(\CC H^n)$. Define 
\[
\widetilde Y = \coprod_{\alpha \in P\mr A} \RR H^n_\alpha. 
\]
We will glue the different hyperbolic spaces $\RR H^n_\alpha$, by defining an equivalence relation $\sim$ on $\widetilde Y$. To define this equivalence relation, we need a couple of results.\begin{lemma}\label{anti-involution}
Let $\alpha \in \mr A$ and $r \in \mr R$. Then $\alpha \circ \phi_r^i = \phi_{\alpha(r)}^{-i} \circ \alpha$. %In particular, if $\alpha(r) = \lambda \cdot r$ for some $r \in \mu_K$, then $\alpha \circ \phi_r^i$ is an involution, and therefore $\alpha \circ \phi_r^i \in \mr A \subset \Gamma'$ in that case. 
\end{lemma}
\begin{proof}
Indeed, for $x \in \Lambda$, we have 
\begin{align*}
\alpha(\phi_r^i(x)) 
&= \alpha\left( x - (1-\zeta^i)h(x,r)\cdot r \right) \\
&= \alpha(x) - (1-\zeta^{-i}) \overline{h(x,r)}\cdot \alpha(r) \\
&= \alpha(x) - (1-\zeta^{-i}) h(\alpha(x),\alpha(r))\cdot \alpha(r) \\
&= \phi_{\alpha(r)}^{-i}(\alpha(x)).  
\end{align*}
%The second statement follows from the first, because $\phi^i_{\lambda \cdot r} = \phi^i_r$ for $\lambda \in \mu_K$. 
\end{proof}
\begin{lemma}\label{alphaswitch}
Let $x \in \RRH^n_\alpha$ and write $\ca H(x) = \set{H_{r_1}, \dotsc, H_{r_k}}$ for some $r_i \in \mr R$. Then for each $i \in \{1, \dotsc, k\}$ there is a unique $j \in \{1, \dotsc, k\}$ such that $
\alpha(H_{r_i}) = H_{\alpha(r_i)} =  H_{r_j}.$
\end{lemma}
\begin{proof}
Indeed, we have, for any $\beta \in \mr A$ and $r \in \mr R$, that 
\begin{align*}
\beta(H_r) = H_{\beta(r)}. 
\end{align*}
Since $x \in H_{r_i}$, we have $x = \alpha(x) \in \alpha(H_{r_i}) = H_{\alpha(r_i)}$ for every $i$. In particular, we have $H_{\alpha(r_i)} \in \ca H(x)$ (see Definition \ref{fullset}), so that $H_{\alpha(r_i)} = H_{r_j}$ for some $j$. 
\end{proof}
\begin{lemma} \label{ij-lemma}
Let $r, t \in \mr R$ and $i,j \in \ZZ/m - \{0\}$ be such that $\phi_r^i = \phi_t^j \in \Gamma$. Then $i = j$ and $a \cdot t = b\cdot r$ for some $a, b \in \OO_K - \{0\}$ such that $\va{a}^2 = \va{b}^2$. 
\end{lemma}
\begin{proof}
%Suppose for contradiction that there is no $\lambda \in \mu_K$ such that $t = \lambda\cdot r$. We claim that this implies that $r, t \in V$ are linearly independent over $K$. Indeed, if $\mu \in K$ is such that $\mu \cdot r = t$, then can write $\mu = a/b$ for $a,b \in \OO_K - \{0\}$. 
Suppose for contradiction that there are no such $a, b\in \OO_K - \{0\}$. Since $K = \tn{Frac}(\OO_K)$, this implies that $r, t \in V$ are linearly independent over $K$, and hence over $\CC$, which contradicts the equality 
\[
\zeta^i \cdot r = \phi_r^i(r) = \phi_t^j(r) = r - (1 - \zeta^j)h(r,t) \cdot t \in V.
\]
Thus, there exist $a,b \in \OO_K - \{0\}$ with $a \cdot t = b \cdot r.$ We have $\va{a}^2 = h(a\cdot t, a\cdot t) = h(b\cdot r, b\cdot r) = \va{b}^2$. Write $\lambda = b/a \in K^\ast$; then $t = \lambda \cdot r$ with $\va{\lambda}^2 = 1$, so that 
\[
\zeta^i \cdot r = \phi_r^i(r) = \phi_t^j(r) = \phi_{\lambda \cdot r}^j(r) = r - (1-\zeta^j)h(r, \lambda\cdot r) \cdot \lambda\cdot r= 
\zeta^j \cdot r \in V,
\]
from which we conclude that $ i = j$. 
\end{proof}
\begin{definition} \label{alphainvolution}
Let $\alpha \in P\mr A$ and $x \in \RRH^n_\alpha$. Write $\ca H(x) = \set{H_{r_1}, \dotsc, H_{r_k}}$, see Definition \ref{fullset}. By Lemma \ref{alphaswitch}, the involution $\alpha$ induces an involution on the set $\ca H(x)$. Define $\alpha \colon I \to I$ as the resulting involution on the set $I = \set{1, \dotsc, k}$. 
\end{definition}
\begin{proposition}\label{preliminaryproposition}
Let $\alpha \in P\mr A$ and $x \in \RRH^n_\alpha$. Write $\ca H(x) = \set{H_{r_1}, \dotsc, H_{r_k}}$ (Definition \ref{fullset}) and let $g = \phi_{r_1}^{i_1} \circ \cdots \circ \phi_{r_k}^{i_k} \in G(x)$ for some $i_\nu \in \ZZ/m$. The following are equivalent:
\begin{enumerate}
\item We have $g \circ \alpha \in P\mr A$. (In other words, $g\circ \alpha$ is an involution.)
\item For each $\nu \in I$, we have $i_{\nu} \equiv i_{\alpha(\nu)} \bmod m$. 
\end{enumerate}
\end{proposition}
\begin{proof}
%Since the composition of an unitary and an anti-unitary transformation is anti-unitary, it suffices to show that $g \circ \alpha$ is an involution. % and lift $x \in \RRH^n_\alpha$ to an element $y \in V^\alpha$ (this is possible by Lemma \ref{hyperbolic}). %We can write $\ca H(x) = \set{H_{r_1} , \dotsc, H_{r_k}}$ for some $r_i \in \mr R$. We have $g = \phi_{r_1}^{n_1} \circ \cdots   \circ \phi_{r_k}^{n_k}$ for some $n_i \in \ZZ/m$. 
Lift $\alpha \in P\mr A$ to an element $\alpha \in \mr A$. We claim that, for each $i \in I$, we have $$\phi_{\alpha(r_i)} = \phi_{r_{\alpha(i)}}.$$
%there is a unique $j_i \in I$ such that $\phi_{a(r_i)} = \phi_{r_{j_i}}$. 
Indeed, by Lemma \ref{alphaswitch}, for each $i \in I$, we have $\alpha(H_{r_i}) = H_{\alpha(r_i)} = H_{r_{j}} \in \ca H(x)$ for some $j \in I$. By definition of the involution $\alpha \colon I \to I$, we have $j = \alpha(i)$. Therefore, $\phi_{\alpha(r_i)} = \phi_{r_{\alpha(i)}}^{b_i}$ for some $b_i \in \ZZ/m$, see Lemma \ref{finitereflectionorders}. By Lemma \ref{ij-lemma}, we have $b_i = 1$. 

By Lemma \ref{anti-involution} and by the claim above, we obtain $
\phi_{r_\nu}^{i_\nu} \circ \alpha = \alpha \circ \phi_{\alpha(r_\nu)}^{-i_{\nu}} = \alpha  \circ \phi_{r_{\alpha(\nu)}}^{-i_{\nu}}$ for each $\nu \in I$, which implies that 
\[
\phi_{r_1}^{i_1} \circ \cdots   \circ \phi_{i_k}^{i_k} \circ \alpha = \alpha \circ 
\phi_{r_{\alpha(1)}}^{-i_1} \circ \cdots   \circ \phi_{i_{\alpha(k)}}^{-i_k}.  
\]
Therefore,
\begin{align*}
\left(\phi_{r_1}^{i_1} \circ \cdots   \circ \phi_{i_k}^{i_k} \circ \alpha\right)^2 
&= 
\phi_{r_1}^{i_1} \circ \cdots   \circ \phi_{i_k}^{i_k} \circ \phi_{r_{\alpha(1)}}^{-i_1} \circ \cdots   \circ \phi_{i_{\alpha(k)}}^{-i_k} 
\\
&= 
\phi_{r_{\alpha(1)}}^{i_{\alpha(1)}} \circ \cdots   \circ \phi_{i_{\alpha(k)}}^{i_{\alpha(k)}} \circ 
\phi_{r_{\alpha(1)}}^{-i_1} \circ \cdots   \circ \phi_{i_{\alpha(k)}}^{-i_k}  \\
&= 
\phi_{r_{\alpha(1)}}^{i_{\alpha(1)} - i_1 } \circ \cdots   \circ \phi_{i_{\alpha(k)}}^{i_{\alpha(k)} - i_k} \in G(x),
\end{align*}
and this is the identity in $G(x)$ if and only if $i_\nu \equiv i_{\alpha(\nu)} \bmod m$. 
\end{proof}

\subsection{The glueing construction}

\begin{definition} \label{def:conditions}
Define a relation $R \subset \widetilde Y \times \widetilde Y$ as follows. An element
 \[
(x_\alpha, y_\beta) \in \RR H^n_\alpha \times \RR H^n_\beta \subset \widetilde Y \times \widetilde Y
 \]
 is an element of $R$ if the following conditions are satisfied:
%$\phi_\alpha(x) = \phi_\beta(y) \in \CC H^n$ and either $\alpha = \beta$, or else $\alpha \neq \beta$ and one has: 
%the following conditions hold:  
%\begin{conditions} \label{conditionss}
\begin{enumerate} 
\item \label{item1} With respect to the inclusions $\RR H^n_\alpha \subset \CC H^n$ and $\RR H^n_\beta \subset \CC H^n$, we have $x_\alpha = y_\beta \in \CC H^n$. 
\item \label{item2} If $\alpha \neq \beta$, then 
\begin{enumerate}
\item $x_\alpha = y_\beta$ lies in $ \mr H$; and   % and $y \in \RR H^n_\alpha \cap \mr H$. 
% (equivalently,$y \in \RR H^n_\beta \cap \mr H$, or $\phi_\alpha(x) \in \mr H$). 
\item 
$\beta = g \circ \alpha \in P\mr A$ for some $g \in G(x_\alpha) = G(y_\beta)$ (c.f. Lemma \ref{Gammaembedding}). 
\end{enumerate}
% for some $\textbf r = (r_i)_i \in \mr R^k$ such that $x \in \cap_{i = 1}^k H_{r_i}$. 
%$\{H_{r_i}\}_{i =1}^k $ is a maximal set of hyperplanes through $x$. 
%such that $x \in \cap_{i =1}^k H_{r_i}$, where $k$ is the order of $x$ and 
%is the maximal number of different $H_i \subset \mr H$ such that $x \in \cap_{i = 1}^k H_i$. 
%of short roots $t_i$ such that 
%$H_{t_i} \neq H_{t_j}$ 
%, where $Y$ is the maximal positive integer for which there exists $\textbf r = (r_1, \dotsc, r_k) \in \mr \mr R^k$ 
%where $Y$ is the maximal positive integer such that $x \in H_{r_1} \cap \dotsc \cap H_{r_k}$ with $H_{r_i} \neq H_{r_j}$ for all $r_i \neq r_j$, and for all $t \in R$ such that $x \in H_t$ one has $H_t = H_{r_i}$ for some $i$. 
%Then there exist integers $j_1, \dotsc, j_k$ such that $\beta = h_{r_1}^{j_1} \circ \dotsc \circ h_{r_k}^{j_k} \circ \alpha$. 
% (such $r_i$ exist by 2.). 
%(Such $r_i$ exist by 2. and we thus let $Y$ be the largest number of non-proportional short roots such that the intersection of their associated hyperplanes contains $\phi_\alpha(x) = \phi_\beta(y)$.)
\end{enumerate}
%\end{conditions}
\end{definition}

\begin{remark} \label{rem:gluing}
%Clearly, 1. and 2. together say that $\phi_\alpha(x) = \phi_\beta(y) \in \RR H^n_\alpha \cap \RR H^n_\beta \cap \mr H \subset \CC H^n$. Hence these 
Conditions \ref{item1} and \ref{item2} in Definition \ref{def:conditions} say that we are identifying points of $\RR H^n_\alpha \cap \mr H$ and $\RR H^n_\beta \cap \mr H$ that have the same image in $\CC H^n$. But we do not glue all such points: the real structures $\alpha$ and $\beta$ are required to differ by complex reflections in the hyperplanes that pass through $x$. 
%Condition 3 says that we only glue them when the real structures $\alpha$ and $\beta$ are related in a certain way: their difference $\alpha \circ \beta^{-1}$ must be a product of $\zeta$-reflections in the hyperplanes that contain $x = y \in \CC H^n$. 
In fact, we will see below (see Lemma \ref{intersection}) that the glueing can be rephrased as follows: we glue $\RR H^n_\alpha$ and $\RR H^n_\beta$ along their intersection, provided that this intersection is contained in $\mr H$ in such a way that for some (equivalently, any) $x \in \RR H^n_\alpha \cap \RR H^n_\beta$, the real structures $\alpha$ and $\beta$ differ by reflections in hyperplanes that pass through $x$. 
 %this gives us the one rule to identify $x \in {\bb B} R\RR H^n_\alpha$ and $y \in {\bb B} R\RR H^n_\beta$ which have 
\end{remark}

\begin{lemma} \label{eqrel}
$R$ is an equivalence relation. 
\end{lemma}

\begin{proof}
Consider three elements $x_\alpha , y_\beta, z_\gamma \in \widetilde Y$. The fact that $x_\alpha\sim x_\alpha$ is clear. 

Suppose that $x_\alpha \sim y_\beta$. If $\alpha  = \beta$ then $x_\alpha = y_\beta \in \widetilde Y$ hence $y_\beta \sim x_\alpha$. If $\alpha \neq \beta$ then $x_\alpha = y_\beta \in \mr H \subset \CC H^n$, and $\beta = g \circ \alpha$ for $g \in G(x_\alpha) = G(y_\beta)$ as in Definition \ref{def:conditions}. %we take $\textbf r = (r_i)_i \in \mr R^k$ such that $x \in \cap H_{r_i}$ where the $H_{r_i}$ are different and $k$ is the order of $x$. But then $y = x \in \cap H_{r_i}$; 
Since $\alpha = g^{-1} \circ \beta$ with $g^{-1} \in G(x_\alpha)$, this shows that $y_\beta \sim x_\alpha$. 
%$\gamma = \psi \circ \beta$ with $\psi \in G(\textbf t)$ for some $\textbf t = (t_i)_i \in R^\ell$ such that $y \in \cap H_{t_i}$
%\textbf{We claim that} $G(x)= G(\textbf t)$. From the claim, it follows that $\alpha  = \phi^{-1} \circ \beta$ with $\phi^{-1} \in G(x)= G(\textbf t)$. 
%hence $y_\beta \sim x_\alpha$. \\

Suppose that $x_\alpha\sim y_\beta$ and $y_\beta \sim z_\gamma$; we claim that $x_\alpha\sim z_\gamma$. We may and do assume that $\alpha, \beta$ and $\gamma$ are different, which implies that $x_\alpha = y_\beta = z_\gamma \in \mr H$, that $\gamma = h \circ \beta$ for some $h \in G(y_\beta)$, and that $\beta = g \circ \alpha$ for some $g \in G(x_\alpha)$. We obtain $
\gamma = h \circ \beta = h \circ g \circ \alpha$ for $ h \circ g \in G(x_\alpha) = G(y_\beta) = G(z_\gamma)$. 
\end{proof}

\begin{definition} \label{gluedspace}
Define $Y$ to be the quotient of $\widetilde Y$ by the equivalence relation $R$, and equip it with the quotient topology. We shall prove (Lemma \ref{lemma:pgammaaction}) that the group $P\Gamma$ acts on $Y$. We call $P\Gamma \setminus Y$ the \textit{glued space} attached to the hermitian $\OO_K$-lattice $(\Lambda, h)$.  
\end{definition}

\begin{lemma} \label{lemma:pgammaaction}
The action of $P\Gamma$ on $\CC H^n$ induces an action of $P\Gamma$ on $\widetilde Y$, compatible with the equivalence relation $R$, so that $P\Gamma$ acts on $Y$. Moreover, $P\Gamma \setminus \widetilde Y =  \coprod_{\alpha \in C\mr A} P\Gamma \setminus \RR H^n_\alpha$. 
\end{lemma}

\begin{proof}
If $\phi \in P\Gamma$, then $
\phi \left( \RR H^n_\alpha \right) = \RR H^n_{\phi \alpha \phi^{-1}}$ hence $P\Gamma$ acts on $\widetilde Y = \coprod_{\alpha \in P\mr A} \RR H^n_\alpha$, and
$$
P\Gamma \setminus \widetilde Y = P\Gamma \setminus \coprod_{\alpha \in P\mr A} \RR H^n_\alpha = \coprod_{\alpha \in C\mr A} P\Gamma \setminus \RR H^n_\alpha. % = X. 
$$
Now suppose that $x_\alpha\sim y_\beta \in \widetilde Y$ and $f \in P\Gamma$. Then $f(x_\alpha) \in \RR H^n_{f \alpha f^{-1}}$ and $f(y_\beta) \in \RRH^n_{f\beta f^{-1}}$. We claim that 
\[
f(x_\alpha)_{f \alpha f^{-1}} \sim f(y_\beta)_{f\beta f^{-1}}. 
\]
For this, we may and do assume that $x_\alpha\neq y_\beta$, hence $x_\alpha = y_\beta \in \mr H$ and $\beta = g \circ \alpha$ for some $g \in G(x_\alpha)$ as in Definition \ref{def:conditions}. In particular, $f(x_\alpha) = f(y_\beta)$. Since $f \circ \phi_r^{i} \circ f^{-1} = \phi_{f(r)}^i$ for each $r \in \mr R$ and $i \in \ZZ/m$, and $h(x,r) = 0 $ if and only if $h(f(x), f(r)) = 0$, we have $fG(x)f^{-1} = G(f(x))$ for each $x \in \CCH^n$. We are done:
\[
f \beta f^{-1} = f( g \circ \alpha) f^{-1} = f g f^{-1}  \circ  (f\alpha f^{-1}), \quad f g f^{-1} \in G(f(x_\alpha)). 
\]
\end{proof}
%We have $f(H_{r_i}) = H_{f(r_i)}$ (see ...),  Therefore $f(x) \in \cap_i H_{f(r_i)}$, %The orders of $x$ and $f(x)$ are the same since $f$ is an isometry, 
%

\section{The hyperbolic orbifold structure}

We are now in position to state the main theorem of Chapter \ref{ch:glueing}: 

\begin{theorem} \label{glueingtheorem1}
\begin{enumerate}
\item \label{completepart}
The glued space $P\Gamma \setminus Y$ admits a metric that makes it a complete path metric space. The natural map $P\Gamma \setminus Y \to P\Gamma \setminus \CC H^n$ is a local isometry. 
\item 
Each point $x \in P\Gamma \setminus Y$ admits an open neighborhood $U \subset P\Gamma \setminus Y$ which is isometric to the quotient of an open subset $V \subset \RR H^n$ by a finite group of isometries. Therefore, the glued space $P\Gamma \setminus Y$ has a real hyperbolic orbifold structure. 
\item \label{oopensuborbifold} One has $ \coprod_{\alpha \in C\mr A}  [P\Gamma_\alpha \setminus \left(\RR H^n_\alpha - \mr H \right)]  \subset P\Gamma \setminus Y$ as an open suborbifold. 
\item The connected components of the real-hyperbolic orbifold $P\Gamma \setminus Y$ are uniformized by $\RR H^n$. This means that for each component $C \subset P\Gamma \setminus Y$ there exists a lattice $P\Gamma_C \subset \textnormal{PO}(n,1)$ and an isomorphism of real hyperbolic orbifolds $C \cong [P\Gamma_C \setminus \RR H^n]$. In other words, 
\begin{equation*}
    P\Gamma \setminus Y \cong \coprod_{C\in \pi_0(P\Gamma\setminus K)} \left[ P\Gamma_C \setminus \RR H^n \right].
\end{equation*}
%\item We have $\coprod_{\alpha \in C\mr A}  [P\Gamma_\alpha \setminus \left(\RR H^n_\alpha - \mr H \right)]  isomorphism $P\Gamma \setminus Y \cong \coprod_{\pi_0(P\Gamma \setminus Y)}[P\Gamma_X \setminus \RR H^n]$ restricts to an isomorphism 
%There is an open embedding $\coprod_{\alpha \in C\mr A}  [P\Gamma_\alpha \setminus \left(\RR H^n_\alpha - \mr H \right)] \hookrightarrow P\Gamma \setminus Y$ with dense image. % of real hyperbolic orbifolds. 
%Let $X_\RR^0 := \coprod_{\alpha \in C\mr A}  [P\Gamma_\alpha \setminus \left(\RR H^n_\alpha - \mr H \right)]$. Then $X_\RR^
\end{enumerate}
\end{theorem}
\noindent
The remaining part of Section \ref{gluing} is devoted to the proof of Theorem \ref{glueingtheorem1}. It can happen that $P\Gamma \setminus Y$ is connected: such is the case when $K = \QQ(\zeta_3)$ and $h = \text{diag}(1,1,1,1,-1)$, see \cite{realACTsurfaces}. When $K = \QQ(\zeta_3)$ and $h = \text{diag}(1,1,1,-1)$, then $P\Gamma \setminus Y$ has two components, see \cite[Remark 6]{realACTbinarysextics}. In Chapter \ref{ch:binaryquintics} we show that some $(d, n)$, there is a homeomorphism between $P\Gamma \setminus Y$ and the space of stable hypersurfaces of degree $d$ in $\PP^n_\RR$, restricting to an orbifold isomorphism between $\coprod_{\alpha} [P\Gamma_\alpha \setminus \left(\RR H^n_\alpha - \mr H \right)]$ and the locus of smooth hypersurfaces (Theorem \ref{th:realstableperiod}). 
% \cong \ca M_0^\RR$.  

\subsection{The path metric on the glued space} 		\label{prooftheorem}

We start with a lemma. We will need it in the proof of Lemma \ref{localisometry} below, which will in turn be used to define a path metric on $P\Gamma \setminus Y$ making it locally isometric to quotients of $\RR H^n$ by finite groups of isometries. It also serves as a sanity check: once there exists an element $x \in \RR H^n_\alpha \cap \RR H^n_\beta$ such that $x_\alpha \sim x_\beta$, then one glues the entire copy $\RR H^n_\alpha$ to the copy $\RR H^n_\beta$ along their intersection in $\CC H^n$. %Tthere is a natural continuous map $\widetilde Y \to \CC H^n$. 

\begin{lemma} \label{intersection}
\begin{enumerate}
\item \label{uno}Let $g = \prod_{\nu  = 1}^k \phi_{r_\nu}^{i_\nu} \in \Gamma$ for some set $\{r_\nu\} \subset \mr R$ of mutually orthogonal short roots $r_\nu$, where $i_\nu \not \equiv 0 \bmod m$ for each $\nu$. Then $\left(\CC H^n\right)^g = \cap_{\nu = 1}^k H_{r_\nu}$. 
\item \label{dos}
Let $\alpha, \beta \in P\mr A$ and $x \in \RR H^n_\alpha \cap \RR H^n_\beta$ such that $x_\alpha\sim x_\beta$. Then $y_\alpha \sim y_\beta$ for every $y \in \RR H^n_\alpha \cap \RR H^n_\beta$. 
\item \label{tres} The natural map $\widetilde Y \to \CC H^n$ descends to a $P\Gamma$-equivariant map $\mr P: Y \to \CC H^n$. 
\end{enumerate}
\end{lemma}
\begin{proof}
1. Let $y \in V$ be representing an element in $\left(\CC H^n\right)^\phi$. Since the $r_i$ are orthogonal, and $g(y) = \lambda$ for some $\lambda \in \CC^\ast$, we have
\begin{align}
g(y) = \prod_{\nu  = 1}^k \phi_{r_\nu}^{i_\nu}(y) = y - \sum_{\nu =  1}^k \left(1 - \zeta^{i_\nu}\right) h(y, r_\nu)r_\nu = \lambda y,
\end{align}
hence $(1-\lambda)y = \sum_{i = 1}^\ell \left(1 - \zeta^{i_\nu}\right) h(y, r_\nu)r_\nu  \in V$. But $y$ spans a negative definite subspace of $V$ while the $r_\nu$ span a positive definite subspace, so that we must have $1 - \lambda = 0 = \sum_{\nu = 1}^k \left(1 - \zeta^{i_\nu}\right) h(y, r_\nu)r_\nu$. Since the $r_\nu$ are mutually orthogonal, they are linearly independent; since $\zeta^{i_\nu} \neq 1$ we find $h(y, r_\nu) = 0$ for each $\nu$. Conversely, if $x \in \cap H_{r_\nu}$, then $\phi_{r_\nu}^{i_\nu}(x) = x$ for each $\nu$. 

2. Since $x_\alpha  \sim x_\beta$, there exists $g \in G(x)$ such that $\beta = g \circ \alpha$. Write $\ca H(x) = \set{H_{r_1}, \dotsc, H_{r_k}}$. Let $y \in \RR H^n_\alpha \cap \RR H^n_\beta$. Then $\alpha(y) = \beta(y) = y$ implies that $g(y) = y$. In particular, $y \in \cap_\nu H_{r_\nu}$ by part \ref{uno}, which implies that $\ca H(x) \subset \ca H(y)$, which in turn implies that $G(x) \subset G(y)$. We conclude that $g \in G(y)$. Hence $y_\alpha \sim y_\beta$. 
%$g \in G(y)$, see Definition \ref{fullset}. 
%$H_{r_\nu} \in \ca H(y)$ for each $\nu$, which implies that 
%Let $\textbf r = (r_1, \dotsc, r_k) \in \mr R^k$ such that $x \in \cap_i H_{r_i}$ for a maximal set of hyperplanes $\{H_{r_i}\}$ through $x$. We have $\beta = \phi \circ \alpha$ for some $\phi \in G(\textbf r)$. Let $I \subset \{1, \dotsc, k\}$ be a subset such that $\phi = \prod_{i \in I} h_{r_i}^{j_i}$ with $j_i \neq 0 \mod m$. If $I = \emptyset$ then there is nothing to prove, so suppose the contrary. Notice that %$\alpha(y) = \beta(y) = y$, so that 
%$\phi(y) = y$. 
%Now let $\textbf t = (t_1, \dotsc, t_\ell) \in \mr R^\ell$ such that $y \in \cap_i H_{t_i}$ as in Condition \ref{conditionss}.\ref{item3}. It follows from the Claim that there is a subset $I \subset \{1, \dotsc, k\}$ such that $y \in \cap_{i \in I} H_{r_i}$. 
%there is a subset $I \subset \{1, \dotsc, k\}$ such that $y \in \cap_{i \in I} H_{r_i}$. 
%Part 1 implies that $y \in \cap_{i \in I} H_{r_i}$. Now let $\textbf t = (t_1, \dotsc, t_\ell) \in \mr R^\ell$ such that $y \in \cap_j H_{t_j}$ as in Definition \ref{def:conditions}. Then for each $i$ there is a $j$ such that $H_{r_i} = H_{t_j}$. By Lemma \ref{lemma:finitereflectionorders}, we have $h_{r_i} \in G(t_j)$, so that $
%\phi \in G( \textbf{r} ) \subset G(\textbf t). 
%$
%which in turn implies that if $\phi = \prod_{j} h_{r_j}^{i_j}$, then 
%and $\beta \circ \alpha^{-1} = \prod_j h_{r_j}^{i_j}$ for some set of 
%\\
%\\

3. If $x_\alpha\sim y_\beta$, then $x = y \in \CC H^n$. 
\end{proof}
\noindent
By Lemma \ref{intersection}, we obtain continuous maps \[
\mr P : K \to \CC H^n, \quad \tn{ and } \quad \overline{\mr P}: P\Gamma \setminus Y \to P\Gamma \setminus \CC H^n.
\] Our next goal is to prove that each point $x \in Y$ has a neighbourhood $V \subset Y$ that maps homeomorphically onto a finite union $\cup_{i = 1}^\ell \RR H^n_{\alpha_i} \subset \CC H^n$. Hence $x$ has an open neighourhood $x \in U \subset V$ that identifies with an open set in a union of copies of $\RR H^n$ in $\CC H^n$ under the map $\mr P$. This allows us to define a metric on $Y$ by pulling back the metric on $\CC H^n$.
%has an open neighborhood $U \subset Y$ such that the restriction $U \to \mr P(U)$ is a homeomorphism. In fact, $U$ can be chosen in such a way that $

\begin{lemma} \label{lemma:compactfinite}
Each compact set $Z \subset \CC H^n$ meets only finitely many $\RR H^n_\alpha$, $\alpha \in P\mr A$. 
\end{lemma}
\begin{proof}
Recall the subgroup $P\Gamma' \subset \textnormal{Isom}(\CC H^n)$ (see (\ref{anti-iso-group}) and Lemma \ref{Gammaembedding}). %is the group of all $\mu_K$-orbits of isometries and anti-isometries $\phi: \Lambda \to \Lambda$, then
We have that $P\Gamma'$ acts properly discontinuously on $\CC H^n$. So if $S$ is the set of $\alpha \in P\mr A$ such that $\alpha Z \cap Z \neq \emptyset$, then $S$ is finite. In particular, $Z$ meets only finitely many $\RR H^n_\alpha$.
\end{proof}
\noindent
Fix a point $f \in Y$ and a point $x_\alpha \in \widetilde Y$ lying above $f$. Let $\alpha_1, \dotsc, \alpha_\ell$ be the elements in $P\mr A$ such that $x_{\alpha_i} \sim x_\alpha$ for each $i \in I \coloneqq \set{1, \dotsc, \ell}$ (since the group $G(x)$ is finite by Lemma \ref{Gxlemma}, these are finite in number). 

Let $p: \widetilde Y \to Y$ be the quotient map, and define
\begin{equation} \label{Kf}
Y_f = p\left(\coprod_{i = 1}^\ell \RR H^n_{\alpha_i}\right) \subset Y. 
\end{equation}
\noindent
We prove that $Y$ is locally isometric to opens in unions of real hyperbolic subspaces of $\CC H^n$. Indeed, we have the following:

\begin{lemma} \label{localisometry}
\begin{enumerate}
\item \label{un}The set $Y_f$ is closed in $Y$. 
\item \label{deux}We have $\mr P\left( Y_f \right) = \cup_{i = 1}^\ell \RR H^n_{\alpha_i} \subset \CC H^n$, and the map 
\begin{equation*}
\mr P_f: Y_f \to \cup_{i = 1}^\ell \RR H^n_{\alpha_i} 
\end{equation*}
induced by $\mr P$ is a homeomorphism. 
    \item \label{trois}The set $Y_f \subset Y$ contains an open neighborhood $U_f$ of $f$ in $Y$. 
\end{enumerate}

\end{lemma}

\begin{proof}
1. One has 
\begin{equation*}
p^{-1} \left( Y_f \right) = p^{-1} \left( p\left(\coprod_{i = 1}^\ell \RR H^n_{\alpha_i}\right) \right)  = \bigcup_{i = 1}^\ell p^{-1} \left( p\left( \RR H^n_{\alpha_i}\right) \right).
\end{equation*}
Therefore, it suffices to show that $p^{-1} \left( p\left( \RR H^n_{\alpha_i}\right) \right)$ is closed in $Y$. But notice that $x_\beta \in p^{-1} \left( p\left( \RR H^n_{\alpha}\right) \right)$ if and only if $x \in \RR H^n_\alpha$ and $x_\alpha\sim x_\beta$, which implies (Lemma \ref{intersection}) that $\RR H^n_\alpha \cap \RR H^n_\beta \subset p^{-1} \left( p\left( \RR H^n_{\alpha}\right) \right)$. Hence for any $\alpha \in P\mr A$, one has
$$
p^{-1} \left( p\left( \RR H^n_{\alpha}\right) \right) = \coprod_{\beta \sim \alpha} \RR H^n_\alpha \cap \RR H^n_\beta,
$$
where $\beta \sim \alpha$ if and only if there exists $x \in \RR H^n_\alpha \cap \RR H^n_\beta$ such that $x_\alpha\sim x_\beta$. It follows that $p^{-1} \left( p\left( \RR H^n_{\alpha}\right) \right) \cap \RR H^n_\beta$ is closed in $\RR H^n_\beta$ for every $\beta \in P\mr A$. But the $\RR H^n_\beta$ are open in $\widetilde Y$ and cover $\widetilde Y$, so that $p^{-1} \left( p\left( \RR H^n_{\alpha}\right) \right)$ is closed in $\widetilde Y$. 

2. We have 
\begin{equation*}
\mr P_f(Y_f) =  \mr P\left( p\left(\coprod_{i = 1}^\ell \RR H^n_{\alpha_i}\right)\right) = \widetilde{ \mr P} \left( \coprod_{i = 1}^\ell \RR H^n_{\alpha_i}\right) = \bigcup_{i = 1}^\ell \RR H^n_{\alpha_i} \subset \CC H^n. 
%\mr P_f(Y_f) = \mr P(Y_f) = \widetilde{ \mr P} \left( p^{-1} \left( Y_f \right) \right) = \cup_{i = 1}^\ell \RR H^n_{\alpha_i} \subset \CC H^n. 
\end{equation*}
%\textbf{so the first thing to show is that ?}
To prove injectivity, let $x_{\alpha_i}, y_{\alpha_j} \in \widetilde Y$ and suppose that $x = y \in \CC H^n$. Then indeed, $x_{\alpha_i} \sim y_{\alpha_j}$ because $\sim$ is an equivalence relation by Lemma \ref{eqrel}. 

Let $Z \subset \CC H^n$ be a compact set. Write 
\[
\widetilde{ \mr P} : \widetilde Y \to \CC H^n
\] 
for the canonical map. 
%\begin{lemma}
%Observe that if $A$ is any topological space, $B \subset A$ a compact subset and $\{C_\alpha\}_{\alpha \in I}$ a locally finite collection of subsets of $A$, then $B \cap C_\alpha \neq \emptyset$ for only finitely many $\alpha \in I$. 
%\begin{proof}
%Cover $Z$ by opens $U_i$ that meet only finitely many of the $Y_\alpha$. Since $Z$ is compact, it admits a finite subcover $U_1, \dotsc, U_n$. Then $Z \subset \cup_{i = 1}^nU_i$   only finitely many of the $Y_\alpha$. 
%\end{proof}
%Consequently, by Lemma \ref{lemma:locallyfinite}, $Z$ meets only finitely many of the $\RR H^n_\alpha$ and 
Remark that $Z$ meets only finitely many of the $\RR H^n_\alpha$ for $\alpha \in P\mr A$, see Lemma \ref{lemma:compactfinite}. Each $Z \cap \RR H^n_\alpha$ is closed in $Z$ since $\RR H^n_\alpha$ is closed in $\CC H^n$, so each $Z \cap \RR H^n_\alpha$ is compact. We conclude that ${\widetilde{ \mr P}}^{-1}(Z) = \coprod Z \cap \RR H^n_\alpha$ is compact. In particular, $\widetilde{\mr P}$ is closed \cite[Theorem A.57]{MR2954043}. 

Finally, we prove that $\mr P_f$ is closed. Let $Z \subset Y_f$ be a closed set. Then $Z$ is closed in $Y$ by part 1, hence $p^{-1}(Z)$ is closed in $\widetilde Y$, hence $\widetilde{ \mr P} \left( p^{-1}(Z) \right) )$ is closed in $\CC H^n$, so that 
\begin{equation*}
\mr P_f(Z) = \mr P(Z) = \widetilde{ \mr P} \left( p^{-1} \left( Z \right) \right) = \left(\widetilde{ \mr P} \left( p^{-1} \left( Z \right) \right)\right) \cap \left(\cup_{i = 1}^\ell \RR H^n_{\alpha_i} \right)
\end{equation*}
is closed in $\cup_{i = 1}^\ell \RR H^n_{\alpha_i} $.  

3. Let $x = \mr P(f) \in \CC H^n$. Since $\CC H^n$ is locally compact, there exists a compact set $Z \subset \CC H^n$ and an open set $U \subset \CC H^n$ with $x \in U \subset Z$. Since $Z$ is compact, it meets only finitely many of the $\RR H^n_\beta \subset \CC H^n$ (Lemma \ref{lemma:compactfinite}). Consequently, the same holds for $U$; define $V = \mr P^{-1}(U) \subset Y$. Define 
\[
\mr B = \{\beta \in P\mr A: U \cap \RR H^n_\beta \neq \emptyset \}.
\]
Also define, for $\beta \in P\mr A$, $Z_\beta = p\left(\RR H^n_\beta \right) \subset Y$. Then 
\begin{equation*}
    f \in V \subset \bigcup_{\beta \in \mr B} Z_\beta = \bigcup_{\substack{\beta \in \mr B \\ \beta(x) = x}} Z_\beta \bigcup_{\substack{\beta \in \mr B \\ \beta(x) \neq x}} Z_\beta.  
\end{equation*}
\noindent
Since each $Z_\beta$ is closed in $Y$ by the proof of part \ref{un}, there is an open $V' \subset V$ with 
\begin{equation*}
f \in  V' \subset \bigcup_{\substack{\beta \in \mr B \\ \beta(x) = x}} Z_\beta = 
\bigcup_{\substack{\beta \in \mr B \\ \beta(x) = x \\ x_\beta \sim x_\alpha}} Z_\beta
\bigcup_{\substack{\beta \in \mr B \\ \beta(x) = x \\ x_\beta \not \sim x_\alpha}} Z_\beta
\end{equation*}
\noindent
Hence again there exists an open subset $V'' \subset V'$ with
\begin{equation*}
    f \in  V'' \subset \bigcup_{\substack{\beta \in \mr B \\ \beta(x) = x \\ x_\beta \sim x_\alpha}} Z_\beta \subset 
    \bigcup_{\substack{\beta \in P \mr A \\ \beta(x) = x \\ x_\beta \sim x_\alpha}} Z_\beta = Y_f. 
\end{equation*}
Therefore, $U_f \coloneqq V'' \subset Y$ satisfies the requirements. 
\end{proof}
\noindent
We need one further lemma:

\begin{lemma}
The topological space $Y$ is Hausdorff. 
\end{lemma}
\begin{proof}
%Let $x_\alpha y_\beta \in \widetilde Y$ such that $x_\alpha \not \sim y_\beta$. 
Let $f , f' \in Y$ be elements such that $f \neq f'$. First suppose that $f \not \in Y_{f'}$. Since $Y_{f'}$ is closed in $Y$ by Lemma \ref{localisometry}, there is an open neighbourhood $U$ of $f$ such that $U \cap U_{f'}  \subset U \cap Y_{f'} = \emptyset$. 

Next, suppose that $f \in Y_{f'}$. Lift $f$ and $f'$ to elements $x_\alpha,  y_\beta \in \widetilde Y$. Assume first that $x = y$. This means that $\mr P(f) = \mr P(f')$. Since $\mr P: Y_{f'} \to \CC H^n$ is injective, this implies that $f = f'$, contradiction. So we have $x \neq y \in \CC H^n$. But $\CC H^n$ is Hausdorff, so there are open subsets $\left(U \subset \CC H^n, V \subset \CC H^n \right)$ such that $x \in U$, $y \in V$ and $U \cap V = \emptyset$. Then $\mr P^{-1}(U) \cap \mr P^{-1}(V) = \emptyset$. 
%Then the fact that $x_\alpha \in Y_g$ implies that $x_\alpha \sim $
%By Lemma \ref{localisometry}, the map $\mr P: Y \to \CC H^n$ identifies $Y_{f'}$ with a closed subset of $\CC H^n$. In particular, $Y_g$ is Hausdorff. This implies that there are open subset $U \subset \CC H^n$, $V \subset \CC H^n$ such that $f \in U$, $g \in V$ and $U \cap V \cap Y_g = \emptyset$. 
%Let $\alpha_i \in P \mr A$ be the elements such that $x_{\alpha_i} \sim x_\alpha$, and $\beta_j \in P \mr A$ such that $(y, \beta_j) \sim y_\beta$. 
%the respective images of $x_\alpha$ and $y_\beta$ in $Y$. 
%Suppose first that $x = y \in \CC H^n$. Consider the open neighbourhoods $f \in U_f$, $g \in U_g$ from Lemma \ref{localisometry}. Suppose that $U_f \cap U_g \neq \emptyset$. This implies that there exists some $z \in \RR H^n_{\alpha_i} \cap \RR H^n_{\beta_j}$ such that $(z,\alpha_i) \sim (z, \beta_j)$. 
\end{proof}

\noindent
We then obtain:

\begin{proposition} \label{metric}
$Y$ is naturally a path metric space which is piecewise isometric to $\RR H^n$.
\end{proposition}
\begin{proof}
From Lemma \ref{localisometry}, we deduce that for each $f \in Y$ there exists an open neighborhood $f \in U_f \subset Y$ such that $\mr P$ induces a homeomorphism $Y \supset U_f \xrightarrow{\sim} \mr P(U_f) \subset \CC H^n$. Pull back the metric on $\mr P(U_f)$ to obtain a metric on $U_f$. Then define a metric on $Y$ as the largest metric which is compatible with the metric on each open set $U_f$ and which preserves the lengths of paths.
\end{proof}

\begin{proposition} \label{prop:pathmetricquotient}
The path metric on $Y$ descends to a path metric on $P\Gamma \setminus Y$. 
\end{proposition}

\begin{proof}
The metric on $Y$ descends in any case to a pseudo-metric on $P\Gamma \setminus Y$, and by \cite[Chapter 1]{Gromov2007}, this is a metric if $P\Gamma$ acts by isometries on $Y$ with closed orbits. %(note that this is equivalent to $P\Gamma \setminus Y$ being $T_0$, $T_1$ or $T_2$). 
This is true: the fact that $P\Gamma$ acts isometrically on $Y$ comes from the $P\Gamma$-equivariance of $\mr P: Y \to \CC H^n$ (Lemma \ref{intersection}) together with the construction of the metric on $Y$ (Proposition \ref{metric}). To check that the $P\Gamma$-orbits are closed in $Y$, let $f \in Y$ with representative $x_\alpha\in \widetilde Y$. By equivariance of $p: \widetilde Y \to Y$, we have $p^{-1} \left( P\Gamma \cdot f \right) = P \Gamma \cdot \left(p^{-1}f \right)$, so since $p$ is a quotient map, it suffices to show that 
$$
P \Gamma \cdot \left(p^{-1}f \right) = P\Gamma \cdot \cup_{x_\beta \sim x_\alpha}x_\beta = \cup_{x_\beta \sim x_\alpha} P\Gamma \cdot x_\beta
$$
is closed in $\widetilde Y$, thus that each orbit $P\Gamma \cdot x_\beta$ is closed in $\widetilde Y$. %(indeed, there are only finitely many $\beta \in P\mr A$ such that $\beta(x) = x$ and $x_\alpha\sim x_\beta$). 
Since $P\Gamma$ is discrete, it suffices to show that $P\Gamma$ acts properly on $\widetilde Y$. So let $Z \subset \widetilde Y$ be any compact set: we claim that $\{g \in P\Gamma: g Z \cap Z \neq \emptyset \}$ is a finite set. Indeed, for each $g \in P\Gamma$, one has $\widetilde{ \mr P} \left( g Z \cap Z  \right) \subset g \widetilde{\mr P}(Z) \cap \widetilde{\mr P}(Z)$, and the latter is non-empty for only finitely many $g \in P\Gamma$, by properness of the action of $P\Gamma$ on $\CC H^n$. 
%$P\Gamma \cdot f = [P\Gamma \cdot x_\alpha]$, hence it suffices to show
%$ = p \left(P\Gamma \cdot x_\alpha\right)$. Since $p$ is closed by Lemma ..., it suffices to prove that $P\Gamma \cdot x_\alpha$ is closed in $\widetilde Y$. Since the latter is covered by the open sets $\RR H^n_\beta$, it suffices to show that $P\Gamma \cdot x \cap \RR H^n_\beta$ is closed in $\RR H^n_\beta$ for each $\beta \in P\mr A$. Since $\RR H^n_\beta$ is closed in $\CC H^n$, it suffices to show that $P\Gamma \cdot x$ is closed in $\CC H^n$, but $P\Gamma$ is a discrete group acting properly on $\CC H^n$, so this is clearly true. 

Since the metric on $Y$ is a path metric, so is the metric on $P\Gamma \setminus Y$ \cite{Gromov2007}.
\end{proof}

\subsection{The orbifold structure on the glued space} 

The next step is to prove that the glued space $P\Gamma \setminus Y$ (see Definition \ref{gluedspace}) is locally isometric to quotients of open sets in $\RR H^n$ by finite groups of isometries. 
\begin{definition}
Let $f \in Y$ with representative $ x_\alpha \in \widetilde Y$. Thus, $x$ is an element in $\CCH^n$, and $\alpha \in P\mr A$ is the class of an anti-unitary involution such that $\alpha(x) = x$. 
\begin{enumerate}
\item 
The \emph{nodes} of $f$ are by definition the nodes of $x_\alpha$ (see Definition \ref{fullset}). Thus, these are the hyperplanes $H \in \ca H(x)$, i.e. the hyperplanes $H_r \in \ca H$ defined by short roots $r \in \mr R$ such that $x \in H_r$ (equivalently, such that $h(x,r) = 0$). 
\item 
%Let $\textbf r = (r_1, \dotsc, r_k) \in \mr R^k$ be a vector of short roots such that $x \in \cap_{i = 1}^kH_{r_i}$, $H_{r_i} \neq H_{r_j}$ for $i \neq j$ and $k$ the order of $x$. 
The number of nodes of $f$ is the cardinality of $\ca H(x)$. %We say that \textit{$f$ has $k$ nodes}, and call the $H_{r_i}$ the \textit{nodes of $f$}. 
\item 
The anti-unitary involution $\alpha \in P\mr A$ induces an involution on the set $\ca H(x)$ by Lemma \ref{alphaswitch}. 
%because $\alpha(x) = x$, $k$ is the order of $x$ and $\alpha H_{r_i} = H_{\alpha(r_i)}$. % in other words, writing $I = \{1, \dotsc, k\}$, we can view $\alpha$ as a permutation $\alpha : I \to I$. 
Let $H \in \ca H(x)$ be a node. We call $H$ a \emph{real node} of $f$ if $\alpha(H) = H$. We call $(H, \alpha(H))$ a \emph{pair of complex conjugate nodes} of $f$ if $\alpha(H) \neq H$. %If $\alpha(H_{r_i}) =  H_{r_i}$ then we call $H_{r_i}$ a \textit{real node of $f$}. If $\alpha(H_{r_i}) = H_{r_j}$ for some $j$ with $H_{r_i} \neq H_{r_j}$ then we call the pair $(H_{r_i}, H_{r_j})$ a \textit{a pair of complex conjugate nodes of $f$}. 
\item 
If $k$ is the number of nodes of $f$, we will generally write $k = 2a + b$, with $a$ the number of pairs of complex conjugate nodes of $f$, and $b$ the number of real nodes of $f$. 
\end{enumerate}
\end{definition}
\noindent
Fix again a point $f \in Y$ and a point $x_\alpha \in \widetilde Y$ lying above $f$. Let $k = 2a + b$ be the number of nodes of $f$. Thus $x \in \RRH^n_\alpha$, and there exist $r_1, \dotsc, r_k \in \mr R$ such that 
\[
\ca H(x) = \set{H_{r_1}, \dotsc, H_{r_k}}, \quad G(x) = \langle \phi_{r_1}, \dotsc, \phi_{r_\ell} \rangle \cong (\ZZ/m)^k.
\]
For $\beta \in P\mr A$, observe that $x_\beta \sim x_\alpha$ if and only if $\alpha \circ \beta \in G(x)$. We relabel the $r_i$ so that they satisfy the following condition: 
\begin{align} \label{relabel}
\begin{split}
\alpha(H_{r_i}) &= H_{r_{i+1}} \tn{ for } i \tn{ odd and } i \leq 2a, \\
\alpha(H_{r_i}) &= H_{r_{i-1}} \tn{ for } i \tn{ even and } i \leq 2a, \tn{ and } \\ 
\alpha(H_{r_i}) &= H_{r_{i}} \tn{ for } i  \in \set{2a +1, \dotsc, k}.
\end{split}
\end{align}
%Let $k = 2a + b$ be the number of nodes of $f$. 
%Let $\alpha_1, \dotsc, \alpha_\ell$ be the elements in $P\mr A$ such that $x_{\alpha_i} \sim x_\alpha$ for each $i \in I \coloneqq \set{1, \dotsc, \ell}$. Thus, for each $i \in I$, $\alpha_i(x) = x$ and $\alpha_i \circ \alpha \in G(x)$. We can write 
%This implies that for each $i \in I$, there is a short root $r_i \in \mr R$ such that  
%\[
%\ca H(x) = \set{H_{r_1}, \dotsc, H_{r_\ell}}, \quad G(x) = \langle \phi_{r_1}, \dotsc, \phi_{r_\ell} \rangle \cong (\ZZ/m)^\ell.
%\
%Assume that $r_1, \dotsc, r_b \in I$ are the real nodes of $f$, that $t_1, \dotsc, $
%Write $\{1, \dotsc, k\} = I \cup J_1 \cup J_2$ where $H_{r_i}$ is a real node of $f$ for each $i \in I$ and a complex node of $f$ for each $i \in J = J_1 \cup J_2$, such that for all $i \in J_1$, $\alpha $. 
\noindent
In other words, $H_{r_i}$ is a real root if and only if $i > 2a$, and $\left(H_{r_i}, H_{r_{i+1}}\right)$ is a pair of complex conjugate roots if and only if $i < 2a$ is odd. 

\begin{lemma} \label{lemma:localcoordinates}
%Let $f \in Y$ with representative $x_\alpha\in \widetilde Y$ be as above. %Let $\alpha_1, \dotsc, \alpha_\ell $ be the elements of $P\mr A$ that fix $x$ and satisfy $x_{\alpha_i} \sim x_\alpha$. 
%Suppose that $f$ has $2a$ non-real and $b$ real nodes. Then
Continue with the notation from above.  
\begin{enumerate}
    \item \label{ein}Let $\beta \in P\mr A$ be such that $x_\beta \sim x_\alpha$. Then 
    \[
    \beta = \prod_{i = 1}^a \left( \phi_{r_{2i-1}} \circ \phi_{r_{2i}} \right)^{j_i} \circ \prod_{i = 2a+1}^k \phi_{r_i}^{j_i} \circ \alpha
    \]
    for some $j_1, \dotsc, j_a, j_{2a+1}, \dotsc, j_k \in \ZZ/m$. In particular, there are $m^{a + b}$ such $\beta$. 
    \item \label{zwei}There is an isometry $\CC H^n \xrightarrow{\sim} {\bb B}^n(\CC)$ identifying $x$ with the origin, $\phi_{r_i}$ with the map \begin{equation*}
        {\bb B}^n(\CC) \to {\bb B}^n(\CC), \white (t_1, \dotsc,t_i, \dotsc, t_n) \mapsto (t_1, \dotsc, \zeta t_i, \dotsc, t_n),
    \end{equation*}
    and $\alpha$ with the map defined by %on the coordinates $t_i$ of ${\bb B}^n(\CC)$ by
    \begin{equation} \label{eq:alpha}
        t_i \mapsto 
\begin{cases}
\bar t_{i+1} \white \textnormal{for $i$ odd and $i \leq 2a$}\\
\bar t_{i-1} \white \textnormal{for $i$ even and $i \leq 2a$}\\
\bar t_i \white\white \textnormal{for $i > 2a$}.
\end{cases}
    \end{equation}
\end{enumerate}
\end{lemma}
\begin{proof}
1. This follows readily from Proposition \ref{preliminaryproposition}. 

2. Since the $H_{r_i}$ are orthogonal by Condition \ref{orthogonal}, and their intersection contains $x$, we can find coordinates $t_1, \dotsc, t_{n+1}$ on $V$ that induce an identification 
%\begin{align} \label{hermit-zero}
$
\left(V, h \right) \cong \CC^{n,1} \coloneqq \left(\CC^{n+1}, H\right)$ with $ H(x,x) = \va{x_1}^2 + \cdots + \va{x_{n}}^2 - \va{x_{n+1}}^2$, 
in such a way that $H_{r_i} \subset V$ is identified with the hyperplane $\{t_i = 0\} \subset \CC^{n+1}$ and $x \in \cap_i H_{r_i}$ with the point $(0,0, \dotsc, 0, 1)$. We will do this in the following way. Define
\[
T =  \langle x \rangle \oplus \langle r_1 \rangle  \oplus \cdots \oplus \langle r_k \rangle  \subset V, 
\quad 
W = T^\perp = \{w \in V \mid h(w,t) = 0 \;\; \forall t \in T\}. 
\]
%Then 
%\[
%V = T \oplus W = \langle x \rangle \oplus \langle r_1 \rangle  \oplus \cdots \oplus \langle r_k \rangle  \oplus W.
%\]
%\[
%0 = \overline{h(r_i, w)} = h(\alpha(r_i), \alpha(w)) = h(\lambda_i \cdot r_{\alpha(i)}, \alpha(w)) =  \lambda_i \cdot \overline{u_{\alpha(i)}} \cdot h(r_{\alpha(i)}, r_{\alpha(i)}) = \lambda_i \cdot \overline{u_{\alpha(i)}}. 
%\]
For each $i \in I = \{1, \dotsc, k\}$, we have $\alpha(r_i) = \lambda_i \cdot r_{\alpha(i)}$ for some $\lambda_i \in K$ (see Lemma \ref{alphaswitch}, Definition \ref{alphainvolution}, and Lemmas \ref{finitereflectionorders} and \ref{Gxlemma}). Observe that $\alpha(W) = W$. Since $W \subset \langle x \rangle^\perp$, the hermitian space $(W, h|_{W})$ is positive definite. Let $\{w_1, \dotsc, w_{n-k}\} \subset W$ be an orthonormal basis such that $\alpha(w_i) = w_i$, which exists by the elementary
%This follows from the following  lemma, whose proof works by induction on $m$:
% because for any $w \in W$, we can write $\alpha(w) = u_x \cdot x + \sum_{i = 1}^k u_i r_i + w' \in V$. This gives $0 = \lambda_i \cdot \overline{u_{\alpha(i)}}$, so that $u_{\alpha(i)} = 0$ for each $i$, and similarly $u_x = 0$. 
%If $\ell \in \{0,1\}$, there is nothing to prove. Suppose 
\begin{lemma}
Let $(W,h)$ a non-degenerate hermitian vector space of dimension $n \geq 1$ and let $\alpha \colon W \to W$ be an anti-linear involution with $h(\alpha(x), \alpha(y)) = \overline{h(x,y)}$ for $x,y \in W$. For each positive integer $m \leq n$, there exists a linearly independent set $\{w_i\}_{i = 1}^m \subset W$ such that $h(w_i, w_j) = \pm \delta_{ij}$, %$h(w_i, w_i) = 1$ and $h(w_i, w_j) = 0$ for $i \neq j$, 
and such that $\alpha(w_i) = w_i$ for each $i = 1, \dotsc, m$. \hfill \qed
%such that $\alpha(w_i) = w_i$ for each $i$. 
\end{lemma}
%\begin{proof}
%We argue by induction on $m$. If $m = 1$ then there is nothing to prove. Assume that $\{w_i\}_{i = 1}^{m-1} \subset W$ is a subset satisfying the requirements. Define $W_{m-1} = \langle w_i \rangle_{i = 1}^{m-1}$, and $T = W_{m-1}^\perp$. We claim that $\alpha(T) = T$. For this, let $t \in T$, and write $\alpha(t) = x + t_0$ for $x \in W_{m-1}$. Then $0 = \overline{h(w_i, t)} = h(w_i, \alpha(t)) = h(w_i, x + t_0) = h(w_i, x)$. Thus, $h(w_i, x) = 0$ for every $i = 1, \dotsc, m-1$, which implies that $x  = 0$ as desired. Let $w_m \in T$ be any element with $\alpha(t) = t$. 
%\end{proof}
%Let $w_1 \in W$ be any element with $\alpha(w_1) = w_1$, and write $W = \langle w_1 \rangle \oplus W_2$ for $W_2 = \langle w_1 \rangle^{\perp}$. For any $w_2 \in W_2$, we may write $\alpha(w_2) = c_1 \cdot w_1+ c_2 \cdot w_2'$ with $w_2' \in W_2$. Then $0 = h(\alpha(w_1), \alpha(w_2)) = h(w_1, \alpha(w_2)) = \overline{c_1}$, so that $\alpha(w_2) \in W_2$. 
%h(w_i, w_j) = h(\alpha(w_i), \alpha(w_j)) = h(\sum_kc_k w_k, \sum_mc_m w_m) = 
%Since $W \subset \langle x \rangle^\perp$, the hermitian space $(W, h|_{W})$ is positive definite. 
%Choose an orthonormal basis $\{w_1, \dotsc, w_{n - k}\} \subset W$ with $\alpha(w_i) = w_i$ for each $i$. 
\noindent
Let $\{e_i\}_{i = 1}^{n+1}$ be the standard basis of $\CC^{n+1}$, and define a $\CC$-linear isomorphism 
\begin{align} \label{hermit-iso}
\Phi \colon V \xrightarrow{\sim} \CC^{n+1}, \quad 
\left(
\frac{x}{ h(x,x)} \mapsto e_{n+1}, \quad  r_i \mapsto e_i, \quad w_i \mapsto e_i
\right).
\end{align}
%Then (\ref{hermit-iso}) is an isomorphism $(V, h) \cong \CC^{n,1}$ as in (\ref{hermit-zero}). %ecause $\Phi$ identifies $h$ with $H$ by construction. 
By (\ref{relabel}), we have that $\alpha(r_i) = \lambda_i \cdot r_{i+1}$ for $i$ odd and $i \leq 2a$, that $\alpha(r_i) = \lambda_i \cdot r_{i-1}$ for $i$ even and $i \leq 2a$, and that $\alpha(r_i) = \lambda_i \cdot r_{i}$ for $i > 2a$. We conclude that the anti-linear involution on $\CC^{n+1}$ induced by $\alpha$ and (\ref{hermit-iso}) corresponds to the matrix
%where $W$ is the orthogonal complement of the left hand side is such that $\alpha \left(\langle x \rangle \oplus_{i = 1}^k H_{r_i} \right) = \left(\langle x \rangle \oplus_{i = 1}^k H_{r_i} \right)$ hence $\alpha W = W$ as well. 
%is in the coordinate system $t_i$ given by 
\begin{equation*}
\alpha = \begin{pmatrix}
0 & \alpha_1  & \ldots & 0&\ldots&&\ldots&0\\
\alpha_2  &  0 & 0 & 0&&\ldots&&\vdots\\
0  &  0 & 0 & \alpha_4&&&&\\
0  &  0 & \alpha_3 & 0&&&&\\
\vdots & 0 & \ddots & \vdots&&\ddots&&\vdots\\
\vdots & \vdots & 0 & \vdots&&&\alpha_{n}&0 \\
0  &   0       &\ldots & 0 &&\ldots&0&\alpha_{n+1}
\end{pmatrix}
\end{equation*}
where each $\alpha_i$ is an anti-linear involution $\CC \to \CC$, and $\alpha_i = \alpha_{i+1}$ for $i < 2a$ odd. If $\alpha_i(1) = \mu_i \in \CC^\ast$, then $\mu_i^{-1} \cdot \alpha_i  = \tn{conj} \colon \CC \to \CC$ (complex conjugation). Since $\va{\mu_i} = 1$, there exists $\rho_i \in \CC$ such that $\mu_i = \overline{\rho_i}/\rho_i$ and $\va{\rho_i} = 1$. This gives $\mu_i^{-1} \cdot \alpha_i = \rho_i \cdot \alpha_i \cdot \rho_i^{-1} = \tn{conj} \colon \CC \to \CC$. The composition %gives then an isomorphism $(V,h) \cong (\CC^{n+1}, H)$ as in (\ref{hermit-zero}), such that $\alpha$ becomes (\ref{eq:alpha}):
 $$V \xrightarrow{\Phi} \CC^{n+1} \xrightarrow{\tn{diag}(\rho_i)} \CC^{n+1}$$
 induces an isomorphism $\CCH^n \cong \bb B^n(\CC)$ with the required properties. 
\end{proof}
\begin{definition}
\begin{itemize}
\item
Define $A_f = \textnormal{Stab}_{P\Gamma}(f)$ to be the subgroup of $P\Gamma$ fixing $f \in Y$. This contains the group $G(x) \cong (\ZZ/m)^k$. 
\item 
Define $B_f$ as the subgroup of $G(x)$ generated by the order $m$ complex reflections associated to the real nodes of $f$, rather than all the nodes. %Hence if $I \subset \{1, \dotsc, k\}$ is such that $H_{r_i}$ is a real node for $i \in I$, then 
Hence $B_f  = \langle \phi_{r_i} \rangle_{i > 2a} \cong (\ZZ/m)^b$. 
\end{itemize}
\end{definition}
\noindent
Recall the quotient map $p\colon \widetilde Y \to Y$, the definition (\ref{Kf}) of $Y_f$, and Lemma \ref{localisometry}. 

\begin{lemma}
The stabilizer $A_f$ of $f \in Y$ preserves the subset $Y_f \subset Y$.
\end{lemma}

\begin{proof}
Let $\psi \in A_f$, with $f = p(x_\alpha) \in Y$, $x \in \cap_i H_{r_i}$. Then $\psi(x)_{\psi \alpha \psi^{-1}} \sim x_\alpha$. Now let $p(y_\beta) \in Y_f$. Then $\beta(x) = x$ and $x_\alpha\sim x_\beta$. Hence $x_\alpha \sim \psi \cdot x_\alpha\sim \psi \cdot x_\beta = \psi(x)_{\psi \beta \psi^{-1}}$. This implies that $\psi \beta \psi^{-1} \circ \alpha \in G(x)$, so that $p\left(\psi(y)_{\psi \beta \psi^{-1}} \right) \in Y_f$. 
%$\phi(x) = x$ and $\phi \alpha \phi^{-1} $
\end{proof}

%Let $A_f \subset P\Gamma_f$ be the subgroup $A_f = G( \textbf r)$. 
\noindent
We also need the following lemma. Write $m = 2^ak$ with $k \neq 0 \bmod 2$. 

\begin{lemma} \label{lemma:T/G}Let $T = \{t \in \CC: t^m \in \RR\}$. Then $G = \langle \zeta_m \rangle$ acts on $T$ by multiplication. %We have $\zeta_{2^{a+1}} \cdot \zeta_{2^a}^{2^{a-2}} = i$, 
Each element in $T/G$ has a unique representative of the form $\zeta_{2^{a+1}}^\epsilon \cdot r$, $r \geq 0$ and $\epsilon \in \{0,1\}$. 
\end{lemma}
\begin{proof}
%Since $m$ is the largest integer such that $\QQ(\zeta_m)$ embeds in $K$, we have $2|m$. Indeed, suppose the contrary: then $-\zeta_m = \zeta_{2m} \in K$, contradiction. 
Therefore, we have $a \geq 1$. 
%For $t = r e^{i \phi} \in \CC$, $r \in \RR$, we have $t = \bar t \zeta_m^j$ if and only if $e^{i\phi} = \pm \zeta_{2m}^j$ if and only if $t = \pm r\zeta_{2m}^j = r'\zeta_{2m}^j$, $r' \in \RR$. 
Next, observe that $t = r \zeta_{2m}^j$ for some $j \in \ZZ$ and $r \in \RR$ if and only if $t^m \in \RR$. %Indeed, if $t = r \zeta_{2m}^j$ then $t^m = r^m (-1)^{j} \in \RR$. Conversely, if $t = r e^{i\phi}$, then $t^m \in\RR$ if and only if $m\phi \in p \ZZ$, which implies that $t = r \zeta_{2m}^j$ for some $j \in \ZZ$. 
One easily shows that since $\gcd(2,k) = 1$, we have $\zeta_{2^{a+1}} \cdot \zeta_{2^ak} = (\zeta_{2^{a+1}k})^{k+2}$. Raising both sides to the power $b = (k+2)^{-1} \in (\ZZ/m)^\ast$ %, where $b \in (\ZZ/m)^\ast$ for $m = 2^ak$ is the inverse of $y+2 \in (\ZZ/m)^\ast$, we see that we can write 
gives $\zeta_{2m} = \zeta_{2^{a+1}}^b \cdot \zeta_m^b$. %For example, $-\zeta_5 = \zeta_{10}^7$ hence $-\zeta_5^3 = \zeta_{10}$.
Consequently, $t^m \in \RR$ 
%$t = \bar t \zeta_m^j$ 
if and only if $t = r \cdot  \zeta_{2^{a+1}}^{bj} \cdot  \zeta_{m}^{bj}$ for some $r \in \RR$. %When $j$ ranges between $0$ and $m = 2^ak$, $\zeta_{2^{a+1}}^{bj}$ %ranges between $1$ and $\zeta_{2^{a+1}}^{b(m)} = \pm 1$, and 
%takes all values $\zeta_{2^{a+1}}^\ell$, $1 \leq l \leq 2^{a+1}$. 
Finally, $\zeta_{2^{a+1}}^u \cdot \zeta_{2^{a}}^v = \zeta_{2^{a+1}}^{u + 2v}$ hence $\langle\zeta_{2^{a+1}}\rangle / \langle \zeta_{2^{a}} \rangle \cong \ZZ/2$.    
%\zeta_{2^{a+1}}
\end{proof}
\noindent
We obtain the key to Theorem \ref{glueingtheorem1}. 

\begin{proposition} \label{localmodel}
Keep the above notations, and consider the set $Y_f \subset Y$ (see (\ref{Kf})). \begin{enumerate}
    \item\label{caseone} If $f$ has no nodes, then $G(x) = B_f$ is trivial, and $Y_f = \RR H^n_\alpha \cong {\bb B}^n(\RR)$. 
    \item\label{casetwo} If $f$ has only real nodes, then $B_f \setminus Y_f$ is isometric to ${\bb B}^n(\RR)$. 
    \item\label{casethree} If $f$ has $a$ pairs of complex conjugate nodes ($k = 2a$), and no other nodes, then $B_f\setminus Y_f = Y_f$ is the union of $m^a$ copies of ${\bb B}^n(\RR)$, any two of which meet along a ${\bb B}^{2c}(\RR)$ for some integer $c$ with $0 \leq c \leq a$.
    \item\label{casefour} If $f$ has $2a$ complex conjugate nodes and $b$ real nodes, then there is an isometry between $B_f \setminus Y_f$ and the union of $m^a$ copies of ${\bb B}^n(\RR)$ identified along common ${\bb B}^{2c}(\RR)'s$, that is, the set $Y_f$ of case \ref{casethree} above.
    \item\label{casefive} In each case, $A_f$ acts transitively on the indicated copies of ${\bb B}^n(\RR)$. If ${\bb B}^n(\RR)$ is any one of them, and $\Gamma_f = (A_f/B_f)_{{\bb B}^n(\RR)}$ its stabilizer, then the natural map
    \begin{equation*}
        \Gamma_f \setminus {\bb B}^n(\RR) \to \left(A_f/B_f\right) \setminus \left(B_f \setminus Y_f \right) = A_f \setminus Y_f
    \end{equation*}
is an isometry of path metrics. 
\end{enumerate}
\end{proposition}
\begin{proof}
1. This is clear. 
%Moreover, for $x,y \in \ZZ_{>0}$, if $x$ does not divide $y$ and $y$ does not divide $x$, then the product of a $x$-th and a $y$-th primitive root of unity is a primitive $\text{lcm}(x,y)$-th root of unity. Now write $m = 2^u \cdot v$, with $u,v \in \ZZ_{\geq 0}$ and $v \neq 0 \mod 2$. Then $2m= 2^{u+1} \cdot v$, and the product of a primitive $2^{u+1}$-th and a primitive $m$-th root of unity is a primitive $2m$-th root of unity, and every primitive $2m$-th root of unity can be decomposed as such. 
%$\zeta_{2^{u+1}} \cdot \zeta_m = \zeta_{2m}$.
% and $\epsilon \in \{0, 1, \dotsc, 2^u\}$.  
%$e^{2i\phi} =e^{2p i\cdot j/m}= (e^{2p i\cdot j/2m})^2 = (\zeta_{2m}^j)^2$ if and only if $e^{i\phi} = \pm \zeta_{2m}^j$. 
%if and only if $\phi = \frac{p}{m}j \mod p\ZZ$ if and only if $t = re^{i p/m \cdot j} = re^{2i p/2m \cdot j} = r\zeta_{2m}^j$. 
%If $f$ has only real nodes, then 

2. Suppose then that $f$ has $k$ real nodes. Then in the local coordinates $t_i$ of Lemma \ref{lemma:localcoordinates}.\ref{zwei}, we have that $\alpha: {\bb B}^n(\CC) \to {\bb B}^n(\CC)$ is defined by $\alpha(t_i) = \bar t_i$. Part \ref{ein} of the same lemma shows that any $\beta \in P\mr A$ fixing $x$ such that $x_\alpha \sim x_\beta$ is of the form 
\begin{equation*}
{\bb B}^n(\CC) \to {\bb B}^n(\CC), \white (t_1, \dotsc,t_i, \dotsc, t_n) \mapsto (\bar t_1 \zeta^{j_1}, \dotsc, \bar t_k \zeta^{j_k}, \bar t_{k+1}, \dotsc, \bar t_n). 
\end{equation*}
Since $f$ has $k$ real nodes and no complex conjugate nodes, we have (writing $j = (j_1, \dotsc, j_k)$ and $\alpha_j = \prod_{i = 1}^k \phi_{r_i}^{j_i} \circ \alpha$): 
\begin{equation*}
Y_f \cong  \bigcup_{j_1, \dotsc, j_k = 1}^m \RR H^n_{\alpha_j} \cong \left\{ (t_1, \dotsc, t_n) \in {\bb B}^n(\CC): t_1^m, \dotsc, t_k^m, t_{k+1}, \dotsc, t_n \in \RR \right\}. 
\end{equation*}
Each of the $2^k$ subsets 
\begin{align*}
    K_{f, \epsilon_1, \dotsc, \epsilon_k} \coloneqq 
    \left\{ (t_1, \dotsc, t_n) \in {\bb B}^n(\CC): \zeta_{2^{a+1}}^{-\epsilon_1}t_1, \dotsc, \zeta_{2^{a+1}}^{-\epsilon_k}t_k \in \RR_{\geq 0} \textnormal{ and } t_{k+1}, \dotsc, t_n \in \RR    \right\},
\end{align*}
indexed by $\epsilon_1, \dotsc, \epsilon_k \in \{0,1\}$, is isometric to the closed region in ${\bb B}^n(\RR)$ bounded by $k$ mutually orthogonal hyperplanes. By Lemma \ref{lemma:T/G}, their union $U$ is a fundamental domain for $B_f$, in the sense that it maps homeomorphically and piecewise-isometrically onto $B_f \setminus Y_f$. Under its path metric, $U = \cup K_{f, \epsilon_1, \dotsc, \epsilon_k} $ is isometric to ${\bb B}^n(\RR)$ by the following map: 
\begin{equation*}
    U \to {\bb B}^n(\RR), \white (t_1, \dotsc, t_k) \mapsto \left((-\zeta_{2^{a+1}})^{-\epsilon_1}t_1, \dotsc, (-\zeta_{2^{a+1}})^{-\epsilon_k}t_k, t_{k+1}, \dotsc, t_n\right). 
\end{equation*}
This identifies $B_f \setminus Y_f$ with the standard ${\bb B}^n(\RR) \subset {\bb B}^n(\CC)$. 

3. Now suppose $f$ has $k = 2a$ nodes $H_{r_1}, \dotsc, H_{r_{2a}}$. There are now $m^{a}$ anti-isometric involutions $\alpha_{j_i}$ fixing $x$ and such that $x_{\alpha_{j_i}} \sim x_\alpha$: they are given in the coordinates $t_i$ as follows, taking $j = (j_1, \dotsc, j_a) \in (\ZZ/m)^a$: 
\begin{equation*}
    \alpha_j: (t_1, \dotsc, t_n) \mapsto (\bar t_2 \zeta^{j_1}, \bar t_1 \zeta^{j_1}, \dotsc, \bar t_{2a} \zeta^{j_{a}}, \bar t_{2a-1} \zeta^{j_a}, \bar t_{2a+1}, \dotsc, \bar t_n).
\end{equation*}
So any fixed-point set $\RR H^n_{\alpha_j}$ is identified with  
\begin{align*}
    {\bb B}^n(\RR)_{\alpha_j} &\coloneqq \\
    &\left\{ (t_1, \dotsc, t_n) \in {\bb B}^n(\CC): t_{i} = \bar t_{i-1} \zeta^{j_i} \textnormal{ for $1 \leq i \leq 2a $ even, } t_i \in \RR \textnormal{ for $i > 2a$} \right\}. 
\end{align*}
All these $m^a$ copies of ${\bb B}^n(\RR)$ meet at the origin of ${\bb B}^n(\CC)$; in fact, for $j \neq j'$, the space ${\bb B}^n(\RR)_{\alpha_{j}}$ meets the space ${\bb B}^n(\RR)_{\alpha_{j'}}$ in a ${\bb B}^{2c}(\RR)$ if $c$ is the number of pairs $(j_i, j'_i)$ %with $j_i$ appearing in $j$ and $j'_\ell$ appearing in $j'$ 
with $j_i = j'_i$.

4. Now we treat the general case. In the local coordinates $t_i$, any anti-unitary involutions fixing $x$ and equivalent to $\alpha$ is of the form 
\begin{align*}
    \alpha_j: &(t_1, \dotsc, t_n) \mapsto \\
    &(\bar t_2 \zeta^{j_1}, \bar t_1 \zeta^{j_1}, \dotsc, \bar t_{2a} \zeta^{j_{a}}, \bar t_{2a-1} \zeta^{j_a}, \bar t_{2a+1}\zeta^{j_{2a+1}}, \dotsc, \bar t_k \zeta^{j_{k}}, \bar t_{k+1}, \dotsc, \bar t_n)
\end{align*}
for some $j = (j_1, \dotsc, j_a, j_{2a+1}, \dotsc, j_{k}) \in (\ZZ/m)^{a+b}$. We now have $B_f \cong (\ZZ/m)^b$ acting by multiplying the $t_i$ for $2a+1 \leq i \leq k$ by powers of $\zeta$, and there are $m^{a+b}$ anti-unitary involutions $\alpha_j$. We have 
\begin{align*}
&Y_f \cong \bigcup_{j_1, \dotsc, j_k = 1}^m \RR H^n_{\alpha_j} \cong \\ 
&\left\{ (t_1, \dotsc, t_n) \in {\bb B}^n(\CC) \mid t_2^{m} = \bar t_1^{m}, \dotsc, t_{2a}^{m} = \bar t_{2a-1}^{m}, t_{2a+1}^{m}, \dotsc, t_k^{m}, t_{k+1}, \dotsc, t_n \in \RR \right\}.
\end{align*}
We look at subsets $K_{f, \epsilon_1, \dotsc, \epsilon_k} \subset Y_f$ again, this time defined as 
\begin{align*}
    K_{f, \epsilon} =& K_{f, \epsilon_1, \dotsc, \epsilon_k} \\
    =&\left\{ (t_1, \dotsc, t_n) \in {\bb B}^n(\CC)\mid \right. \\
    %t_2^{m} = \bar t_1^{m}, \dotsc, t_{2a}^{m} = \bar t_{2a-1}^{m}
    &\left.t_i^m = \bar t_{i-1}^m \textnormal{ $i \leq 2a$ even, }
    \zeta_{2^{a+1}}^{-\epsilon_i}t_i \in \RR_{\geq 0} \textnormal{ $2a<i\leq k$, }
    %\zeta_{2^{a+1}}^{-\epsilon_1}t_1, \dotsc, \zeta_{2^{a+1}}^{-\epsilon_k}t_k \in \RR_{\geq 0} 
     t_i \in \RR, i > k    \right\}.
\end{align*}
As before, we have that the natural map
$
U:= \bigcup_{\epsilon} K_{f, \epsilon} \to B_f \setminus Y_f
$ is an isometry. Define  
\begin{equation*}
\widetilde Y_f = 
\left\{ (t_1, \dotsc, t_n) \in {\bb B}^n(\CC):
    %t_2^{m} = \bar t_1^{m}, \dotsc, t_{2a}^{m} = \bar t_{2a-1}^{m}
    t_i^m = \bar t_{i-1}^m \textnormal{ for $i \leq 2a$ even, }
    %\zeta_{2^{a+1}}^{-\epsilon_i}t_i \in \RR_{\geq 0} \textnormal{ $2a<i\leq k$, }
    %\zeta_{2^{a+1}}^{-\epsilon_1}t_1, \dotsc, \zeta_{2^{a+1}}^{-\epsilon_k}t_k \in \RR_{\geq 0} 
     t_i \in \RR, \textnormal{ for } i > 2a    \right\}.
\end{equation*}
\noindent
Under its path metric, $U = \cup_\epsilon K_{f, \epsilon_1, \dotsc, \epsilon_k}$ is isometric to $\widetilde Y_f$ 
%${\bb B}^n(\RR)$ 
by the following map: 
\begin{align*}
    U \to \widetilde Y_f, \white 
&(t_1, \dotsc, t_k) \mapsto \\ 
&\left(t_1, \dotsc, t_{2a}, (-\zeta_{2^{a+1}})^{-\epsilon_1}t_{2a+1}, \dotsc, (-\zeta_{2^{a+1}})^{-\epsilon_k}t_k, t_{k+1}, \dotsc, t_n\right). 
\end{align*}
Hence $B_f \setminus Y_f \cong \widetilde Y_f$; but since $\widetilde Y_f$ is what $Y_f$ was in case \ref{casethree}, we are done. 
%which was needed to be shown, so $B_f \setminus Y_f$ is isometric 

5. The transitivity of $A_f$ on the copies of ${\bb B}^n(\RR)$ follows from the fact that $G(x)\subset A_f$ contains transformations multiplying $t_1, \dotsc, t_{2a}$ by powers of $\zeta$, hence $t_i \mapsto \zeta^ut_i, t_{i-1} \mapsto t_{i-1}$ maps those $t_{i-1}, t_i$ with $t_i = \bar t_{i-1}\zeta^{j_i}$ to those $t_{i-1}, t_i$ with $t_i = \bar t_{i-1}\zeta^{j_1+u}$. So if $B$ is any one of the copies of ${\bb B}^n(\RR)$, and $G = (A_f/B_f)_H$ is its stabilizer, then it remains to prove that $G \setminus B \to A_f \setminus Y_f$ is an isometry. Surjectivity follows from the transitivity of $A_f$ on the ${\bb B}^n(\RR)^{'}s$. It is a piecewise isometry so we only need to prove injectivity. This will follow from an elementary lemma. 
\begin{lemma}
Let $X$ be a set on which a group $G$ acts, let $Y$ and $I$ be sets, and let $\set{\phi_i \colon Y \hookrightarrow X}_{i \in I}$ be a set of embeddings. Write $Y_i = \phi_i(Y)$ and suppose that $X = \cup_i Y_i$. Fix $0 \in I$. Let $H \subset G$ be the stabilizer of $Y_0$. % we have $H = \{g \in G: gY_0 = Y_0 \}$. 
Suppose that for all $y \in X$, the stabilizer of $y$ in $G$ acts transitively on the sets $Y_i$ containing $y$. % in the sense that for each $i$ and $j$ there is a $g \in G$ such that $gY_i = Y_j$. 
Then $H\setminus Y_0 \to G \setminus X$ is injective.
\end{lemma} 
\begin{proof}
Let $x,y \in Y_0$ and $g \in G$ such that $g \cdot x = y$. Then $y = gx \in gY_0$. Since also $y \in Y_0$, there is an element $h \in \text{Stab}_G(y)$ such that $hgY_0 = Y_0$ and $hg(x) = h(y) = y$. Let $f = hg$; then $f \in H$ and $f \cdot x = y$, which proves what we want.
\end{proof} 
\noindent
Now let us use the lemma: suppose that $y \in B_f \setminus Y_f$. We need to prove that $\text{Stab}_{A_f/B_f}(y)$ acts transtivitely on the copies of ${\bb B}^n(\RR)$ containing $y$. There exists
\begin{equation*}
    j = (j_1, \dotsc, j_a, j_{2a+1}, \dotsc, j_k) \in (\ZZ/m)^{a+b}
\end{equation*} such that $y = (t_1, \dotsc, t_n)$ with $t_{i} = \bar t_{i-1} \zeta^{j_i}$ for $i \leq 2a$ even, $t_{i} = \bar t_{i-1} \zeta^{j_i}$ for $2a < i \leq k$, and $t_i \in \RR$ for $i > k$. If all $t_i$ are non-zero, then $y \in \cup_{j'} \RR H^n_{\alpha_{j'}}$ is only contained in $\RR H^n_{\alpha_j}$, so there is nothing to prove. Let us suppose that $t_1 = t_2 = 0$ and the other $t_i$ are non-zero. Then $y$ is contained in all the $\RR H^n_{\alpha_{j'}}$ with $j_i' = j_i$ for $i \geq 2$; there are $m$ of them. The stabilizer of $y$ multiplies $t_1$ and $t_2$ by powers of $\zeta$ and leaves the other $t_i$ invariant; it acts transitively on the $\RR H^n_{\alpha_{j'}}$ containing $y$ for if $t_2 = \bar t_1 \zeta^{j_1'}$ then $\zeta^{(j_1''-j_1')}t_2 = \bar t_1 \zeta^{j_1''}$. The general case is similar. 
\end{proof}
\noindent
As indicated above, we can now prove Theorem \ref{glueingtheorem1}.
\begin{proof}[Proof of Theorem \ref{glueingtheorem1}]
1. The path metric on $P\Gamma \setminus Y$ is given by Proposition \ref{prop:pathmetricquotient}. Note that the map $\mr P: Y \to \CC H^n$ is a local embedding by Lemma \ref{localisometry}, which was used to define the metric on $Y$ (Proposition \ref{metric}). Thus, almost by definition, $\mr P$ is a local isometry. For each $f \in Y$ we can find a $P\Gamma_f$-invariant open neighborhood $U_f \subset Y_f \subset Y$ such that $P\Gamma_f \setminus U_f \subset P\Gamma \setminus Y$, with $U_f$ mapping bijectively onto an open subset $V_f$ in the closed subset $\mr P(Y_f) = \cup_i\RR H^n_{\alpha_i} \subset \CC H^n$. By $P\Gamma$-equivariance of $\mr P$, the set $V_f$ is $P\Gamma_f$-invariant, and we have $P\Gamma_f \setminus V_f \subset P\Gamma \setminus \CC H^n$. % for if two points $x,y \in V_f$ are in the same orbit under $P\Gamma$, then the same holds for their preimages in $U_f$, hence they are in the same orbit.
Thus $$\overline{\mr P}: P\Gamma \setminus Y \to P\Gamma \setminus \CC H^n$$ is also a local isometry. % (which gives another proof of the fact that the pseudo-metric on $P\Gamma \setminus Y$ is a metric: it coincides with the largest path metric compatible with the metric on small enough subsets of $P\Gamma \setminus \CC H^n$). %the pullback of the metric on $P\Gamma \setminus Y$ by the local embedding $P\Gamma \setminus Y \to P\Gamma \setminus \CC H^n$, where `pullback' stands for the largest metric compatible with the metric on small enough subsets of $P\Gamma \setminus \CC H^n$ and preserving the lengths of paths). 

%To prove that $P\Gamma \setminus Y$ is complete, %consider $Y$ first. 
2. Note that the map $\mr P: Y \to \CC H^n$ is proper because any compact set in $\CC H^n$ meets only finitely many ${\RR H^n_\alpha}^{'}s$, $\alpha \in P\mr A$ (Lemma \ref{lemma:compactfinite}), and $\mr P$ carries each $H_\alpha = p\left(\RR H^n_\alpha \right)$ homeomorphically onto $\RR H^n_\alpha$. Since $P\Gamma \setminus \CC H^n$ is complete, the space $P\Gamma \setminus Y$ is complete as well. 
%then $P\Gamma \setminus Y$ must be complete: $\mr P$ is proper and $P\Gamma \setminus \CC H^n$ is complete. 

Finally, let $[f] \in P\Gamma \setminus Y$ be the image of $f \in Y$. Then $[f]$ has an open neighborhood isometric to the quotient of an open set $W$ in $\RR H^n$ by a finite group of isometries $\Gamma_f$. Indeed, take $Y_f \subset Y$ as in Equation (\ref{Kf}), and $f \in U_f \subset Y_f$ as in Lemma \ref{localisometry}.\ref{deux}. We let $A_f = P\Gamma_f$ be the stabilizer of $f$ in $P\Gamma$ as before, and take an $A_f$-equivariant open neighborhood $V_f \subset U_f$ such that $A_f \setminus V_f \subset P\Gamma \setminus Y$. By Proposition \ref{localmodel}.\ref{casefive}, we know that $A_f \setminus Y_f$ is isometric to $\Gamma_f \setminus \RR H^n$ for some finite group of isometries of $\RR H^n$. This implies that $A_f \setminus V_f$ is isometric to some open set $W'$ in $\Gamma_f \setminus \RR H^n$. Take $W \subset \RR H^n$ to be the preimage of $W'$. %We obtain an orbifold structure. Indeed, 

\emph{Claim:} For any path metric space $X$ locally isometric to quotients of $\RR H^n$ by finite groups of isometries, there is a unique real-hyperbolic orbifold structure on $X$ whose path metric is the given one. 
%under the natural projection. %The open neighborhood $A_f \setminus V_f$ of $[f]$ in $P\Gamma \setminus Y$ is thus isometric to $\Gamma_f \setminus W$.  
%\begin{lemma} \label{lemma:hyperbolic}

\emph{Proof of the Claim:} If $U$ and $U'$ are connected open subsets of $\RR H^n$ and $\Gamma$ and $\Gamma'$ finite groups of isometries of $\RR H^n$ preserving $U$ and $U'$ respectively, then any isometry $\bar \phi: \Gamma\setminus U \to \Gamma' \setminus U'$ extends to an isometry $\phi: \RR H^n \to \RR H^n$ such that $\phi(U) = U'$ and $\phi \Gamma{\phi}^{-1} = \Gamma' \subset \textnormal{Isom}(\RR H^n)$. %$\hfill \qed$
%\end{lemma}

We conclude that $P\Gamma \setminus Y$ is naturally a real hyperbolic orbifold. 

3. Let us show that 
\[
O \coloneqq \coprod_{\alpha \in C\mr A}  [P\Gamma_\alpha \setminus \left(\RR H^n_\alpha - \mr H \right)]  \subset P\Gamma \setminus Y
\]
as hyperbolic orbifolds. It suffices to show the following

\emph{Claim:} For those $f = p(x_\alpha) \in Y$ that have no nodes, the stabilizer $A_f = P\Gamma_f \subset P\Gamma$ of $f \in Y$ and the stabilizer $P\Gamma_{\alpha,x} \subset P\Gamma_{\alpha}$ of $x$ in $\RR H^n_\alpha$ agree as subgroups of $P\Gamma$. 

\emph{Proof of the Claim:} To prove that $A_f = P\Gamma_{\alpha,x}$, we first observe that $p: \widetilde Y \to Y$ induces an isomorphism between $P\Gamma_{x_\alpha}$, the stabilizer of $x_\alpha\in \widetilde Y$ and $P\Gamma_f$, the stabilizer of $f = [x, \alpha] \in Y$. So it suffices to show that $P\Gamma_{x_\alpha} = P\Gamma_{\alpha,x}$. For this we use that the normalizer $N_{P\Gamma}(\alpha)$ and the stabilizer $P\Gamma_\alpha \subset P\Gamma$ of $\alpha$ in $P\Gamma$ are equal, which implies that $    P\Gamma_{\alpha,x} = P\Gamma_{x_\alpha}$ because
%Lemma \ref{lemma:stabilizernormalizer}:
\begin{align*}
\left\{ g \in P\Gamma_\alpha : gx = x \right\}  &= \left\{ g \in N_{P\Gamma}(\alpha) : gx = x \right\} \\
&=       \left\{ g \in P\Gamma: g\cdot x_\alpha= (g(x), g\alpha g^{-1}) = x_\alpha\right\} . 
    \end{align*}
So the claim is proved. Part \ref{oopensuborbifold} of the theorem can be deduced from it as follows. Let $f = p(x_\alpha) \in Y$ have no nodes. We have $Y_f = \RR H^n_\alpha$, hence 
\[
A_f \setminus \RR H^n_\alpha= A_f \setminus Y_f = \Gamma_f \setminus \RR H^n \quad \tn{with} \quad \Gamma_f = A_f \setminus B_f = A_f.
\] 
By construction, an orbifold chart of the glued space $P\Gamma \setminus Y$ is given by $$W \to A_f \setminus W \subset P\Gamma_\alpha \setminus \RR H^n_\alpha \subset Y$$ for an invariant open subset $W$ of $\RR H^n_\alpha$ containing $x$. Because $A_f = P\Gamma_{\alpha,x}$ by the claim, this is also an orbifold chart for $O$ at the point $x_\alpha$. 

4. The real-hyperbolic orbifold $P\Gamma \setminus Y$ is complete by part \ref{completepart}, so the uniformization of the connected components of $P\Gamma \setminus Y$ follows from the Ehresmann--Thurston uniformization theorem for $(G,X)$-orbifolds, see \cite[Proposition 13.3.2]{Thurston80}. This concludes the proof of Theorem \ref{glueingtheorem1}, and thereby also of Theorem \ref{th:theorem03}. 
\end{proof}

\section{Unitary Shimura varieties} \label{unitaryshimura}

The goal of this section is to prove Proposition \ref{prop:canonicalbijection}, which describes (in case the signature of $h$ is hyperbolic for one place of $K$ and definite for all others) our ball quotient $P\Gamma \setminus \CC H^n$ in terms of moduli of abelian varieties with $\OO_K$-action of hyperbolic signature, and Proposition \ref{prop:HcorrespondsNonSimple}, which interprets the divisor $P\Gamma \setminus \mr H$ as the locus of abelian varieties $A$ that admit a homomorphism $\CC^g/\Psi(\OO_K) \to A$. % where $B$ is an abelian variety with complex multiplication. %the moduli space of abelian varieties in $P\Gamma \setminus \CC H^n$ that are products of lower-dimensional ones. This gives a geometric proof of the fact that when the different ideal is generated by an element $\eta \in \OO_K$ such that $\sigma(\eta) = - \eta$, the intersection of two hyperplanes $H_r \neq H_t \subset \mr H $ is either empty or orthogonal, see Proposition \ref{prop:conditionsimplyhypothesis}. \\
%Both statements use the assumption that the hermitian form $h: \Lambda \times \Lambda \to \OO_K$ is hyperbolic for one place of $K$ and definite  for all others. 
\noindent
This has two applications: 
\begin{enumerate}
\item Consider a relative uniform cyclic cover (see e.g. \cite{arsievistoli2004}) \[\mr X \to P \to S,\] 
where $P =  \PP^1_S$ (resp. $\PP^3_S$), the fibers of $\mr X \to S$ are curves (resp. threefolds with $\rm H^{0,3} = 0$) and the induced hermitian form on middle cohomology satisfies the above signature condition. 
%is hyperbolic for one place and definite for all others. 
Since the image $\mf I \subset P\Gamma \setminus \CC H^n$ of the period map $S(\CC) \to P\Gamma \setminus \CC H^n$ is contained in the locus of abelian varieties whose theta divisor is irreducible, %Proposition \ref{prop:HcorrespondsNonSimple} implies that 
one has $\mf I \subset P\Gamma \setminus \left( \CC H^n - \mr H \right)$. 
\item  %using the uniqueness of a decomposition of a polarized abelian variety into non-decomposable ones \cite{debarrepolarisations}, one can show that 
%Proposition \ref{prop:HcorrespondsNonSimple} and a certain condition on 
If the different ideal $\mf D_K \subset \OO_K$ is generated by some $\eta \in \OO_K - \OO_F$ with $\eta^2 \in \OO_F$, then the hyperplanes in the arrangement $\mr H \subset \CC H^n$ are orthogonal along their intersection (see Theorem \ref{th:conditionsimplyhypothesis}). 
\end{enumerate}

\subsection{Alternating and hermitian forms on the lattice}

%Recall that we are in the following situation. The field $K$ is a CM field of degree $2g$ over $\QQ$ and $F \subset K$ is its totally real subfield. The ring $\OO_K$ is the ring of integers of $K$. The non-trivial element in $\Gal(K/F)$ is denoted by $\sigma$. %We fixed a set of embeddings $\{\tau_i : K \to \CC\}$ such that $\{\tau_i, \tau_i \sigma\}$ are all the embeddings of $K$ into $\CC$. 

The goal of this subsection is to prove two lemmas. They will later be used to show that $P\Gamma \setminus \CC H^n$ is a moduli space of abelian varieties, %(considered that $h$ is definite at infinite places different from $\tau$) 
and to give a modular interpretation of the divisor $P\Gamma \setminus \mr H \subset P\Gamma \setminus \CC H^n$. 
%Finally, $\Lambda$ is a free $\OO_K$-module of rank $n+1$, which we view as a $\ZZ$-module equipped with a homomorphism $\iota: \OO_K \to \End_{\ZZ}(\Lambda)$. 
%if the signature of $h$ at other places than $\tau$ is definite
%moduli space of polariabelian varieties with $\OO_K$-action
%is twofold. Firstly, we aim to prove a correspondence between skew-hermitian forms $T$ on $\Lambda$ and alternating forms $E$ on $\Lambda$ compatible with the $\OO_K$-action $\iota$. Secondly, given such an alternating form $E: \Lambda \times \Lambda \to \ZZ$, for any embedding $\varphi: K \to \CC$ there are then two ways of defining a hermitian form on the complex vector space $\Lambda \otimes_{\OO_K, \varphi} \CC = \left(\Lambda \otimes_\ZZ \CC \right)_\varphi$: one as the restriction of $iE_\CC(x, \bar y)$, and one as $i T^\varphi$. 
%For the applications we have in mind (Proposition \ref{prop:canonicalbijection} and Proposition \ref{prop:HcorrespondsNonSimple}), 
%For the modular interpretation of $P\Gamma \setminus \CC H^n $ and $P\Gamma \setminus \mr H$, it is essential that these two hermitian forms agree, which we prove next.
%We claim that skew-hermitian and alternating forms on $\Lambda$ and related in the following way:
\\
\\
We continue with the notation of Section \ref{set-up}. In particular, $\Lambda$ is a free $\OO_K$-module of rank $n+1$. 

\begin{lemma} \label{lemma:equivalentforms}
The assignment $T \mapsto \textnormal{Tr}_{K/\QQ} \circ T$ defines a bijection between:
%following sets are in bijection:
\begin{enumerate}
    \item \label{item:hermitian} The set of skew-hermitian forms $T: \Lambda_{\QQ} \times \Lambda_{\QQ} \to K$. %such that $T(\Lambda, \Lambda) \subset \mf D_K^{-1}$. 
    \item \label{item:alternating} The set of alternating forms $E: \Lambda_{\QQ} \times \Lambda_{\QQ} \to \QQ$ such that $E(a \cdot x,y) = E( x, a^\sigma \cdot y)$.
\end{enumerate}
Under this correspondence, $T(\Lambda, \Lambda) \subset \mf D_K^{-1}$ if and only if $E(\Lambda, \Lambda) \subset \ZZ$. 
\end{lemma}

\begin{proof}
Let \[T\colon  \Lambda_{\QQ} \times \Lambda_{\QQ} \to K\] be as in \ref{item:hermitian}. Define $E_T = \text{Tr}_{K/\QQ} \circ T$. % where $\text{Tr}_{K/\QQ}: K \to \QQ$ is the trace map. 
Since $T$ is skew-hermitian, we have, for each $x, y \in \Lambda_\QQ$, that $$\text{Tr}_{K/\QQ}T(x,y) = - \text{Tr}_{K/\QQ}\overline{T(y,x)}.$$ Since $K/\QQ$ is separable, for any $x \in K$, we have \cite[(7-1)]{stevenhagen}:
$$
\text{Tr}_{K/\QQ}(x) = \sum_{1 \leq i \leq g} \left( \tau_i(x) + \tau_i\sigma(x) \right).
$$
Thus, we have $\text{Tr}_{K/\QQ}(\sigma(x)) = \text{Tr}_{K/\QQ}(x)$, so that $E_T(x,y) = - E_T(y,x)$ for any $x,y \in \Lambda_\QQ$. The property in \ref{item:alternating} is easily checked. %Moreover, we have $T(\iota(a)x,y) = aT(x,y) = T(x,i(a^\sigma)y)$ hence $E(\iota(a)x,y) = E(x,\iota(a^\sigma)y)$ for all $x,y \in \Lambda$. 

Conversely, let $E: \Lambda_\QQ \times \Lambda_\QQ \to \QQ$ be as in \ref{item:alternating}. Choose a basis $\{b_1, \dotsc, b_{n+1} \} \subset \Lambda$ for $\Lambda$ over $\OO_K$. %Let $\{e_i\}$ be the canonical basis of $\OO_K^{n+1}$.
% such that $\sigma(e_i) = e_i$ for every $i$. 
Define $Q$ to be the induced map $K^{n+1} \times K^{n+1} \to \QQ$ and consider the map
%$$
%Q: K^{n+1} \times K^{n+1} \cong \Lambda_\QQ \times \Lambda_\QQ \xrightarrow{E} \QQ. 
%$$
$K \to \QQ$, $a \mapsto Q(a \cdot e_i, e_j)$. Since the trace pairing 
\[
K \times K \to \QQ, \quad (x,y) \mapsto \text{Tr}_{K/\QQ}(xy)
\] 
is non-degenerate \cite[\href{https://stacks.math.columbia.edu/tag/0BIE}{Tag 0BIE}]{stacks-project}, %hence defines an isomorphism $K \cong \Hom_{\QQ}(K, \QQ)$. 
there is a unique $t_{ij} \in K$ such that $Q(a \cdot e_i, e_j) = \text{Tr}_{K/\QQ}(a \cdot t_{ij})$ for every $a \in K$. This gives a matrix $(t_{ij})_{ij} \in M_{n+1}(K)$ such that $\sigma(t_{ij}) = - t_{ji}$, and the basis $\{b_i\}$ induces a skew-hermitian form $T_E: \Lambda_\QQ \times \Lambda_\QQ \to K$. 

The last claim from the definition of $\mf D_K^{-1} \subset K$ as the trace dual of $\OO_K$, see \cite[Chapter III, \S3]{localfields}. 
\end{proof}
% and via our choice of basis a $\QQ$-linear form $T_E: \Lambda_\QQ \times \Lambda_\QQ \to K$ which is $K$-linear in the first variable and $K$-$\sigma$-linear in the second. 
%Now $T_E$ is skew-hermitian: for $a \in K$,  we have
%$$Q(a \cdote_j, e_i) =
% Q(e_j, \iota(a^\sigma)e_i) = %-Q(\iota(a^\sigma)e_i,e_j) = 
 %-\text{Tr}_{K/\QQ}(a^\sigma t_{ij}) = - \text{Tr}_{K/\QQ}(a \overline{t_{ij}}) = \text{Tr}_{K/\QQ}(a (-\overline{t_{ij}}). 
%$$
%Then $E \mapsto T_E$ defines the inverse of $T \mapsto \text{Tr}_{K/\QQ} \circ T$. 
%= \{x \in K: \text{Tr}_{K/\QQ}(x\OO_K) \subset \ZZ \}$. 
%hence $E_T(\Lambda, \Lambda) \subset \ZZ$ if and only if $T(\Lambda, \Lambda) \subset \mf D_K^{-1}$. 
%for every $a \in K$, hence $t_{ji} = -\overline{t_{ij}}$. It is clear that $E_{T_E} = E$, which also shows that $T(\Lambda, \Lambda) \subset \mf D_K^{-1}$. Starting from a form $T: \Lambda_{\QQ} \times \Lambda_{\QQ} \to K$ as in \ref{item:hermitian}, it gives rise to a skew-hermitian matrix $S = (s_{ij}) \in M_{n+1}(K)$ via our basis $\{e_i\}$.  Since $E(a \cdotb_i, b_j) = \text{Tr}_{K/\QQ} T(a \cdotb_i,b_j) = \text{Tr}_{K/\QQ} (a e_i S \sigma(b_j))  = \text{Tr}_{K/\QQ} (a e_i Se_j) = \text{Tr}_{K/\QQ} (a s_{ij})$, we have that $t_{ij} = s_{ij}$, hence $T_{E_T} = T$.  
%In particular, $Q(e_i, e_j) = \text{Tr}_{K/\QQ}(t_{ij})$. 

\begin{examples} \label{examplesunitary}
\begin{enumerate}
    \item \label{ex:unitone} Suppose $K = \QQ(\sqrt \Delta)$ is imaginary quadratic with discriminant $\Delta$ and non-trivial Galois automorphism $a \mapsto a^\sigma$. Let $E: \Lambda \times \Lambda \to \ZZ$ be an alternating form with $E(a \cdot x,y) = E(x, a^\sigma \cdot y)$. The form $T: \Lambda \times \Lambda \to \mf D_{K}^{-1} = (\sqrt{\Delta})^{-1}$ is defined as 
    \begin{equation*}
        T(x,y) = \frac{E( \sqrt{\Delta}\cdot  x,y) + E(x,y)\sqrt{\Delta} }{2\sqrt \Delta}. 
    \end{equation*}
    %Note that indeed, $\text{Tr}_{K/\QQ} \left(T(x,y) \right) = E(x,y)$. 
    \item \label{ex:unittwo} Let $K = \QQ(\zeta)$ where $\zeta = \zeta_p = e^{2 \pi i/p} \in \CC$ for some prime number $p > 2$. Let $E: \Lambda \times \Lambda \to \ZZ$ be an alternating form with $E(a \cdot x,y) = E(x, a^\sigma \cdot y)$. Then $\mf D_K = \left(p/(\zeta - \zeta^{-1})\right)$ and %form $T: \Lambda \times \Lambda \to \mf D_{K}^{-1} = \left((\zeta - \zeta^{-1})/p\right)$ is defined as
    %the equality $f(\mathfrak p|p) = 1$ implies that $e(\mathfrak p|p) = p-1$, because of $e \cdot f = [K:\QQ] = p-1$. Also in the other cases, we have that $\mf D_K = \left(\eta \right)$ for some $\eta \in \OO_K$ with $\sigma(\eta)  = - \eta$. Let $g: \OO_K \to \End(\Lambda)$ be the structural morphism. Define
\begin{equation*}
T: \Lambda \times \Lambda \to \mf D_K^{-1}, \white T(x,y) = \frac{1}{p}\sum_{j = 0}^{p-1}\zeta^jE\left( x, \zeta^j \cdot  y \right).
\end{equation*}
%Then $T_0$ is skew-hermitian and $\text{Tr}_{K/\QQ} (T_0(x,y)) = p E(x,y)$. Hence $T(x,y) = T_0(x,y)/p$ is the skew-hermitian form $T : \Lambda_\QQ \times \Lambda_\QQ \to K$ corresponding to $E$. %Notice that $g = (p-1)/2$. If $\tau_1, \dotsc, \tau_g$ are defined as $\tau_i(\zeta) = \zeta^i \in \CC$, then $\Im \tau_i\left(\zeta - \zeta^{-1}\right) > 0$ for every $i$. 
%Define $\eta = \zeta - \zeta^{-1}$, then $\mf D_K = (\eta)^{p-2}$; notice that $(\zeta - \zeta^{-1}) \mf D_K = \mf p^{p-1} = (p) \subset \ZZ$. Therefore, 
%    \item \label{ex:unitthree} %Let $K/\QQ$ be any CM field of degree $2g$ with totally real subfield $F$. 
%Let $\Psi$ be a set of embeddings $\varphi: K \to \CC$ such that $\Psi \cup \sigma\Psi = \Hom(K, \CC)$. Define $\Psi: \OO_K \to \CC^g$ by $\Psi(a) = (\varphi(a))_{\varphi \in \Psi}$. Let $\eta^{-1} \in \OO_K$ be an element such that $K = F(\eta^{-1})$, $\sigma(\eta^{-1}) = -\eta^{-1}$ and $\Im(\varphi(\eta^{-1})) > 0$ for every $\varphi \in \Psi$. Suppose moreover that $\eta \in \mf D_K^{-1} \subset K$. Then $E: \Psi(K) \times \Psi(K) \to \QQ$, $T(\Psi(x), \Psi(y)) = \text{Tr}_{K/\QQ} (\eta x \bar y)$ is an alternating form such that $E(\Psi(\OO_K), \Psi(\OO_K)) \subset \ZZ$. %a skew-hermitian form such that $T(\Psi(\OO_K), \Psi(\OO_K)) \subset \mf D_K^{-1}$. 
%    The corresponding skew-hermitian form $T: \Psi(\OO_K) \times \Psi(\OO_K) \to \mf D_K^{-1}$ is given by $T(\Psi(x), \Psi(y)) = \eta x \bar y$. 
\end{enumerate}
\end{examples}

\noindent
Now consider a corresponding pair 
\[
\left(E \colon \Lambda_\QQ \times \Lambda_\QQ \to \QQ, \quad T: \Lambda_{\QQ} \times \Lambda_{\QQ} \to K\right)
\]
as in Lemma \ref{lemma:equivalentforms}, and suppose that $E$ is non-degenerate. Let $\varphi: K \to \CC$ be an embedding. Define a skew-hermitian form $T^\varphi$ as
\begin{align*}
\begin{split}
& T^\varphi\coloneqq \Lambda \otimes_{\OO_K, \varphi} \CC \times \Lambda \otimes_{\OO_K, \varphi} \CC\to \CC, \\
 & T^\varphi( \sum_i x_i \otimes \lambda_i, \sum_j y_j \otimes \mu_j) = 
\sum_{ij} \lambda_i \overline{\mu_j} \cdot \varphi\left( T(x_i, y_j) \right). 
\end{split}
\end{align*}
On $\Lambda_\CC$, we also have the skew-hermitian form $A(x,y) = E_\CC(x,\bar y)$. The composition 
\[
\left(\Lambda \otimes_\ZZ \CC \right)_{\varphi} \to \Lambda \otimes_\ZZ \CC \to \Lambda \otimes_{\OO_K, \varphi} \CC
\]
is an isomorphism. Define $A^\varphi$ to be the restriction of $A$ to the subspace $\left(\Lambda \otimes_\ZZ \CC \right)_{\varphi} = \Lambda \otimes_{\OO_K, \varphi} \CC \subset \Lambda_\CC$. Note that 
\[
\Lambda \otimes_\ZZ \CC \cong \oplus_{\phi: K \to \CC}\left(\Lambda \otimes_\ZZ \CC \right)_{\phi}.\]
For $x \in \Lambda \otimes_\ZZ \CC$, let $x^\phi$ be the image of $x$ under $\Lambda \otimes_\ZZ \CC \to \left(\Lambda \otimes_\ZZ \CC \right)_{\phi}$. 

\begin{lemma} \label{lemma:agree}
%With the above notation and assumptions, 
Let $\varphi \colon K \to \CC$ be an embedding. We have an equality of skew-hermitian forms: 
\[
T^\varphi = A^\varphi \colon \left(\Lambda \otimes_\ZZ \CC \right)_{\varphi} \times \left(\Lambda \otimes_\ZZ \CC \right)_{\varphi} \to \CC.
\]
More precisely, we have $A(x,y) = \sum_{\phi: K \to \CC} T^\phi(x^\phi, y^\phi)$ for every $x,y \in \Lambda \otimes_\ZZ \CC$.
%Consider the bijection in Lemma \ref{lemma:equivalentforms}. 
%Let $\varphi: K \to \CC$ be an embedding. 
%Let $E: \Lambda \times \Lambda \to \ZZ$ and $T: \Lambda_{\QQ} \times \Lambda_{\QQ} \to K $ be as in Lemma \ref{lemma:equivalentforms} such that $E$ is non-degenerate. Let $\varphi: K \to \CC$ be an embedding. 
%. Under the canonical isomorphism $\Lambda \otimes_{\OO_K, \varphi} \CC  \cong \left(\Lambda \otimes_\ZZ \CC \right)_{\varphi}$, the skew-hermitian form $T^{\varphi}_\CC(x,y)$ on $\Lambda \otimes_{\OO_K, \varphi} \CC $ corresponds to the skew-hermitian form $E_{\CC}(x, \bar y)$ on $\left(\Lambda \otimes_\ZZ \CC \right)_{\varphi}$. % $\Lambda \otimes_{\OO_K, \varphi} \CC \times \Lambda \otimes_{\OO_K, \varphi} \CC \to \CC$.T_\CC
\end{lemma}
\begin{proof}
Write $V = \Lambda_\QQ$. %We claim that the following diagram commutes:
The lemma follows from the fact that the following diagram commutes:
$$
\xymatrixcolsep{5pc}
\xymatrix{
V \times V\ar@{^{(}->}[d] \ar[d] \ar[r]^T& K\ar@{^{(}->}[d] \ar[r]^{\textnormal{Tr}_{K/\QQ}} & \QQ \ar@{^{(}->}[dd] \\
V \otimes_\QQ \CC \times V \otimes_\QQ \CC \ar[r]^{T_\CC} \ar[drr]^{A(x, y)} \ar@{=}[d] & K \otimes_\QQ \CC\ar@{=}[d]  & \\
\oplus_\phi \left(V \otimes_\QQ \CC\right)_\phi  \times \left(V \otimes_\QQ \CC\right)_\phi  \ar[r]^{\white \white\white\white\white  \oplus T^\phi} & \oplus_\phi \CC_\phi \ar[r]_{\sum} & \CC.
}
%\left(V \otimes_\QQ \CC\right)_{\varphi} \times \left(V \otimes_\QQ \CC\right)_{\varphi} \ar@{^{(}->}[u]  \ar[r]^{\white \white\white\white\white T^\varphi} & \CC_{\varphi_0}. \ar@{^{(}->}[u]\ar@{=}[ur] & 
$$
Here, $\phi$ ranges over the set of embeddings $K \to \CC$, $\CC_\phi$ is the $K$-module $\CC$ where $K$ acts via $\phi$, and $$T_\CC: V \otimes_\QQ \CC \times V \otimes_\QQ \CC \to K \otimes_\QQ\CC$$ is the map that sends $(v \otimes \lambda, x \otimes \mu)$ to $\lambda \bar \mu T(v,w)$. 
\end{proof}

\subsection{Moduli of abelian varieties acted upon by the ring of integers of a CM field} \label{subsec:moduliabelianvarieties}

\begin{notation} \label{not:fixhermitian}
%In the rest of Section \ref{unitaryshimura}, we fix the following. Let $\eta \in K$ such that \textcolor{blue}{(1)} $K = F(\eta)$, \textcolor{blue}{(2)} $\sigma(\eta) = - \eta$, $\eta^2 \in F$, and \textcolor{blue}{(3)} $\eta^{-1} \in \mf D_K^{-1}$. Moreover, let $\Psi$ be a set of embeddings $\{\tau_i: K \to \CC\}_{1 \leq i \leq g}$ such that \textcolor{blue}{(4)} $\Psi \cup \sigma\Psi = \Hom(K, \CC)$ and \textcolor{blue}{(5)} $\Im \left(\tau_i(\eta) \right) > 0$ for $1 \leq i \leq g$. 
%and $\Im \tau(\eta) > 0$, $\Im \tau \sigma (\eta) < 0$, ...$\Im \tau_i(\eta) > 0$ for every $i$ such that $r_i = n+1$ and $\Im \tau_i(\eta) < 0$ for every $i$ such that $s_i = n+1$
%Note that such $\eta$ and $\Psi$ exist. 
In the rest of Section \ref{unitaryshimura}, we fix: 
\begin{enumerate}
\item a non-degenerate hermitian form $\mf h: \Lambda \times \Lambda \to \mf D_K^{-1}$; and 
\item an element $\xi \in \mf D_K^{-1}$ such that $\sigma(\xi) = - \xi$ and $\Im \left(\tau_i(\xi) \right) < 0$ for $1 \leq i \leq g$ and write $\eta = \xi^{-1}$. Here, the embeddings $\tau_i \colon K \to \CC$ are those introduced in (\ref{setofembeddings}). 
\end{enumerate}
These data define a skew-hermitian form 
\[
T\colon \Lambda \times \Lambda \to \mf D_K^{-1}, \quad  T \coloneqq \xi \cdot \mf h.
\]
The form $T$ is in turn attached to a symplectic form (see Lemma \ref{lemma:equivalentforms})
\[
E: \Lambda \times \Lambda \to \ZZ \; \tn{ such that } \; E(a x, y) = E(x, a^\sigma y)\; \tn{ for all }\; a \in \OO_K, \; x,y \in \Lambda. 
\]
Write $V_i = \Lambda_\QQ \otimes_{K, \tau_i} \CC$ and define 
\[
\mf h^{\tau_i}: V_i\times V_i \to \CC
\]
to be the hermitian form restricting to $\tau_i \circ \mf h $ on $\Lambda$. Let $(r_i,s_i)$ be the signature of the hermitian form $\mf h^{\tau_i}$.
\end{notation}
 \noindent
Let $A$ be a complex abelian variety, $\iota$ a homomorphism $\ca O_K \to \text{End}(A)$, and $\lambda$ a polarization $A \to A^\vee$, satisfying the following (c.f. \cite[Part I, \S2.1]{kudlarap-special-II}):
\begin{conditions} \label{KRconditions}
\begin{enumerate} 
    \item We have $\iota(a)^\dagger = i(a^\sigma)$ for the Rosati involution $$\dagger \colon \End(A)_{\QQ} \to\End(A)_{\QQ}, \quad \quad \text{ and }$$
    \item 
$
\textnormal{char}(t, \iota(a) | \textnormal{Lie}(A)) =
\prod_{\nu = 1}^g (t-a^{\tau_i})^{r_i}
\cdot 
(t-a^{\tau_i\sigma})^{s_i} \; \in \;  \CC[t]
%(T-\varphi(a^{\sigma_1}))^{r_1}(T-\varphi(a^{\sigma_2}))^{r_2}) \cdots (T-\varphi(a^{\sigma_g}))^{r_g} \cdot
%(T-\varphi(a^{\sigma_1\sigma }))^{s_1}(T-\varphi(a^{\sigma_2\sigma}))^{r_2}) 
$ \\
\phantom{the}
\hfill (the characteristic polynomial of $\iota(a)$). 
\end{enumerate}
\end{conditions}
\noindent
Note that $\dim A = g(n+1)$. 
%In other words, the representation $\iota: K \to \End_\CC(\Lie(A))$ should be equivalent to $\Psi$. 
%Write $A = \rm H^{-1,0}/L$ and 
Define $E_A : \rm H_1(A, \ZZ) \times \rm H_1(A, \ZZ) \to \ZZ$ to be the alternating form corresponding to $\lambda$. The condition on the Rosati involution implies that $E_A(\iota(a)x, y) = E_A(x, \iota(a^\sigma)y)$ for $x,y \in \rm H_1(A, \QQ)$. Define a hermitian form $\mf h_A$ on the $\OO_K$-module $\rm H_1(A, \ZZ)$ as follows:
\[
\mf h_A  = \eta T_A \colon \; \rm H_1(A, \ZZ) \times \rm H_1(A, \ZZ) \to \mf D_K^{-1}. 
\]
Here, $T_A \colon \rm H_1(A, \ZZ) \times \rm H_1(A, \ZZ) \to \mf D_K^{-1}$ is the skew-hermitian form attached to $E_A$ via Lemma \ref{lemma:equivalentforms}. 
\begin{definition} \begin{enumerate}
\item 
Let $\widetilde{\textnormal{Sh}}_{K}(\mf h)$ be the set of isomorphism classes of four-tuples $(A, i, \lambda, j)$, where $(A,i, \lambda)$ is as above and satisfies Conditions \ref{KRconditions}, and where $j \colon \rm H_1(A, \ZZ) \to \Lambda$ is a symplectic isomorphism of $\OO_K$-modules. % with %$\mf h_A(x,y) = \mf h(j(x),j(y))$ for every $x,y \in H_1(A, \ZZ)$ (this is equivalent to
%$E_A(x,y) = E(j(x),j(y))$ for all $x,y \in H_1(A, \ZZ)$. 
\item 
Let $\bb D(V_i)$ be the space of negative $s_i$-planes in the hermitian space $(V_{i}, \mf h^{\tau_i})$. 
\end{enumerate}
\end{definition}
\noindent
We have the following proposition which is due to Shimura, see \cite[Theorem 2]{Shimura1963ONAF} or \cite[\S 1]{shimuratranscendental}. We give a different proof since it will imply Proposition \ref{prop:HcorrespondsNonSimple} below, whereas we did not know how to deduce Proposition \ref{prop:HcorrespondsNonSimple} from \textit{loc.cit.} We remark that Shimura assumes $\Lambda$ to be an $R$-module for any order $R \subset \OO_K$; our proof carries over, but we do not need this generalization.
%and hence the result would be unaltered if we assumed the same. 

\begin{proposition} \label{prop:canonicalbijection}
%Let $(\Psi, \eta)$ be a canonical tuple (Definition \ref{def:embeddingselement}) and 
%Let $T$ be a skew-hermitian form $\Lambda \times \Lambda \to \mf D_K^{-1}$. 
There is a canonical bijection 
$$
    \widetilde{\textnormal{Sh}}_K(\mf h) \cong {\bb D}(V_1) \times \cdots \times {\bb D}(V_g). 
$$
\end{proposition}

\begin{proof}
Let $(A, i, \lambda, j)$ be a representative of an isomorphism class in $\widetilde{\textnormal{Sh}}_{K}(\mf h)$. % Define a hermitian form $\mf h_A = \eta T_A: \rm H_1(A, \QQ)  \times \rm H_1(A, \QQ)  \to K$. 
Let $\rm H_1(A, \CC) = \rm H^{-1,0} \oplus \rm H^{0,-1}$ be the Hodge decomposition of $A$. For $1 \leq i \leq g$ there is a decomposition 
\begin{equation} \label{eq:posneg}
\rm H_1(A, \CC)_{\tau_i} = \rm H^{-1,0}_{\tau_i} \oplus \rm H^{0,-1}_{\tau_i},
\end{equation}
with $\dim \rm H^{-1,0}_{\tau_i} = r_i$ and $\dim \rm H^{0,-1}_{\tau_i} = s_i$. The latter holds because 
\[
\overline{ \rm H^{-1,0}_{\tau_i\sigma}} =  \rm H^{0,-1}_{\tau_i}.\]
By Lemma \ref{lemma:agree}, $\tau_i(\eta)E_{A, \CC}(x,\bar y)$ and ${\mf h}_{A, \CC}^{\tau_i}(x,y)$ agree as hermitian forms on the complex vector space $\rm H_1(A, \ZZ) \otimes_{\OO_K, \tau_i} \CC$. Since $\Im \tau_i(\eta) > 0$ for every $i$, %it follows that ${\mf h}_{A, \CC}^{\tau_i}(x,y)$ differs from $iE_{A, \CC}(x,\bar y)$ by a positive real scalar. But 
the decomposition of $\rm H_1(A, \CC)_{\tau_i}$ in (\ref{eq:posneg}) is a decomposition into a positive definite $r_i$-dimensional subspace and a negative definite $s_i$-dimensional subspace. %because $iE_{A, \CC}(x,\bar y)$ is positive definite on $\rm H^{-1,0}$ and negative definite on $\rm H^{0,-1}$. 
The isomorphism $j: \rm H_1(A, \QQ) \to \Lambda_\QQ$ induces an isometry $j_i: \rm H_1(A, \CC)_{\tau_i} \to V_i$ for every $i$, and so we obtain a negative $s_i$-plane $j ( \rm H^{0,-1}_{\tau_i})$ in the hermitian space $V_i$ for all $i$. 

Reversing the argument shows that given a negative $s_i$-plane $X_i \subset V_i$ for every $i$, there is a canonical polarized abelian variety $A = \rm H^{-1,0}/\Lambda$, acted upon by $\OO_K$ and inducing the planes $X_i \subset V_i$.
\end{proof}
\begin{definition}
\begin{enumerate}
\item 
Let $\textnormal{Sh}_{K}(\mf h)$ be the set of isomorphism classes of polarized $\OO_K$-linear abelian varieties $(A, i, \lambda)$, satisfying Conditions \ref{KRconditions}, such that $\rm H_1(A, \ZZ)$ is isometric to $\Lambda$ as hermitian $\OO_K$-modules. 
\item 
Let $\Gamma(\mf h) = \Aut_{\OO_K}(\Lambda, \mf h)$; this is the group of $\OO_K$-linear automorphisms of $\Lambda$ preserving our form $\mf h: \Lambda \times \Lambda \to \mf D_K^{-1}$.
\end{enumerate}
\end{definition}
The bijection in Proposition \ref{prop:canonicalbijection} being $\Gamma (\mf h)$-equivariant, we obtain the following:\begin{corollary}
There is a canonical bijection $$\textnormal{Sh}_{K}(\mf h) \cong \Gamma(\mf h) \setminus  {\bb D}(V_1) \times \cdots \times {\bb D}(V_g).$$ $\hfill \qed$
\end{corollary}
%is an isomorphism of $\OO_K$-modules with $T_A(x,y) = T(j(x),j(y))$ for every $x,y \in H_1(A, \ZZ)$. The following result is due to Shimura, see \cite[Theorem 2]{Shimura1963ONAF}, or \cite[\S 1]{shimuratranscendental}. 

\begin{comment}
\begin{corollary}
Suppose $n = 0$, i.e. $\Lambda$ is a free $\OO_K$-module of rank $1$. Let $E : \Lambda \times \Lambda \to \ZZ$ be a symplectic (i.e. non-degenerate alternating) form such that $|\Lambda^\vee/\Lambda| = d^2$ and $E(ax,y) = E(x,a^\sigma y)$. Up to isomorphism there is one and only one polarized abelian variety $A$ with $\OO_K$-multiplication of signature $\Psi$ having $(\Lambda, E)$ as underlying symplectic lattice. $\hfill \qed$
\end{corollary}
\end{comment}

\subsection{Abelian varieties with moduli in the hyperplane arrangement}

\noindent
%Hence $\Psi: K \to M_m(\CC)$ is a homomorphism such that $\Psi(a)$ is a diagonal matrix with $n$ entries equal $a^\tau = a^{\tau}$, one entry equal to $a^{\tau\sigma}$, and $n+1$ entries equal to $a^{\tau_i}$, for $2 \leq i \leq g$. Notice that $\Psi \cup \sigma \Psi = \Hom(K, \CC)$. 
%Recall that $N \in \ZZ_{\geq 0}$ was chosen to satisfy $\eta^{-1} = N \eta^{-1} \in \mf D_K^{-1}$. 
The set of embeddings $\Psi$ defined in (\ref{setofembeddings}) defines a map $\Psi: \OO_K \to \CC^g$, % by $\Psi(a) = (\tau_i(a))_{i = 1, \dotsc, g}$. This 
giving a complex torus $\CC^g/\Psi(\OO_K)$. %Note that any $\alpha \in K^\ast$ gives a map 
The map 
\[
Q: K \times K \to \QQ, \quad Q(x,y) = \text{Tr}_{K/\QQ}(\xi x \bar y)
\]
is a non-degenerate $\QQ$-bilinear form such that $Q(ax,y) = { Q}(x, a^\sigma y)$ for every $a,x,y \in K$. Moreover, $Q(\OO_K, \OO_K) \subset \ZZ$ because $\xi\in \mf D_K^{-1}$. %Observe that $\sigma(\eta^{-1}) = - \eta^{-1}$ and $\Im(\varphi(\eta^{-1})) < 0$ for every $\varphi \in \Psi$. Therefore, 
By \cite[Example 2.9 \& Footnote 16]{milneCM}, %we have the following: $ Q(x,y) = - Q(y,x)$, $ Q_\RR(J x, Jy) = { Q}_\RR(x,y)$ for every $x,y \in K \otimes_{\QQ} \RR \xrightarrow[\sim]{\Psi} \CC^g$ where $J$ is the induced complex structure on $K \otimes_{\QQ} \RR$, and $ Q_\RR(Jx,x) > 0$ for all non-zero $x \in K \otimes_{\QQ} \RR$. In other words, 
$ Q$ defines a Riemann form on the complex torus $\CC^g/\Psi(\OO_K)$. % if $\xi\in \mf D_K^{-1}$. %Let $T_B$ be the corresponding skew-hermitian form; then $T_B(\Psi(x), \Psi(y)) = \xix \bar y \in \OO_K$ for $x,y \in \OO_K$. 
%\begin{remark} \label{remark:N=1}
%If $\mf D_K = (\eta) = \eta\OO_K$, then we can take $N = 1$ and $\eta^{-1} = \eta^{-1} \in \mf D_K^{-1}$. Under this assumption, $\OO_K \times \OO_K \to \ZZ$, $(x,y) \mapsto \eta^{-1}  x \bar y$ 
%hence $ Q(\OO_K, \OO_K) = \text{Tr}_{K/\QQ}(\eta^{-1} \OO_K) \subset \ZZ$, so that $ Q: \OO_K \times \OO_K \to \ZZ$ 
%defines a polarization on $B$. 
%\end{remark}
%Let $N \in \ZZ$ be an integer such that $N \eta^{-1} \in \OO_K$. 
%Then $\tilde Q = N \cdot {^\eta^{-1} Q} : \OO_K \times \OO_K \to \ZZ, (x,y) \mapsto \textnormal{Tr}_{K/\QQ}(N \eta^{-1} x \bar y)$ defines a polarization $\mu_N: B \to \hat B$ on the complex torus $B = \CC^g/\Psi(\OO_K)$. 
%Write $V = V_1$ and $\tau = \tau_1$ for simplicity. 
%Suppose that the signature of ${\mf h}^{\tau} = \tau(\eta) T^{\tau}$ equals $(r_1,s_1) = (n,1)$. Let $\CC H^n$ be the set of negative lines in $V$. 
% = \eta T$ equals $(r_1,s_1) = (n,1)$
%: this we may do because the inclusion $\Lambda_\QQ \to V = \Lambda_\QQ \otimes_{K, \tau} \CC$ is compatible with $\mf h$ and $h^{\tau}$. 
%Define, for any $r \in \Lambda$ such that $\mf h(r,r) = 1$, a set $H_r$ as $H_r = \{[x] \in \CC H^n: \mf h(x,r) = 0\}$; moreover, let $\mr H = \cup_{\mf h(r,r) = 1} H_r \subset \CC H^n$. 
%Write $T(\Psi(x), \Psi(y)) = \text{Tr}_{K/\QQ}(N \eta^{-1} x \bar y) \in \ZZ$. 
%\\
%\\
%If we consider $B$ as polarized abelian variety in this way, then we get:
%and define $T' = NT$, $h' = \eta T' = \eta NT$, we get:
\\
\\
As in Section \ref{set-up}, let $\CC H^n$ be the set of negative lines in $\Lambda \otimes_{\OO_K, \tau_1} \CC$, and define $\mr H = \cup_{\mf h(r,r) = 1} \langle r_\CC \rangle ^\perp \subset \CC H^n$. 
\begin{proposition} \label{prop:HcorrespondsNonSimple}
%Suppose that the element $\beta$ can be chosen such that 
%Suppose that $\eta^{-1} \in \mf D_K^{-1}$. 
Suppose that $(r_1,s_1) = (n,1)$ and $(r_i,s_i) = (n+1,0)$ for $2 \leq i \leq g$. Then under the bijection $  \widetilde{\textnormal{Sh}}_{K}(\mf h) \cong \CC H^n$ of Proposition \ref{prop:canonicalbijection}, the subset $\mr H \subset {\bb C} H^n$ corresponds to the isomorphism classes of those polarized marked $\OO_K$-linear abelian varieties $A$ %$(A,i,\lambda,j)$ of signature $\Psi$ 
that admit a $\OO_K$-linear homomorphism $\CC^g/\Psi(\OO_K)  \to A$ 
%an abelian variety with complex multiplication 
of polarized abelian varieties. %, where $\CC^g/\Psi(\OO_K)$ is polarized by $\text{Tr}_{K/\QQ}(\eta^{-1} x \bar y)$ as above. 
%$\mr H \subset {\bb C} H^n(V)$ corresponds to isomorphism classes of polarized marked $\OO_K$-linear abelian varieties of signature $\Psi$ containing an abelian variety with complex multiplication by $K$ as a polarized marked $\OO_K$-linear abelian subvariety. 
\end{proposition}

\begin{proof}
Consider an isomorphism class $[(A,i,\lambda,y)] \in \widetilde{\textnormal{Sh}}_{K}(\mf h) $ corresponding to a point $[x] \in \CCH^n$. We may assume that $A = \rm H^{-1,0} / \Lambda$ with $\Lambda \otimes_\ZZ \CC = \rm H^{-1,0} \oplus \rm H^{0,-1}$, and that $T_A = T$. Let $$\phi: \CC^g/\Psi(\OO_K) \to A$$ be a homomorphism as in the proposition. We obtain a homomorphism $$\OO_K \to \Psi(\OO_K) \to \rm H_1(A, \ZZ) = \Lambda$$ which, for simplicity, we also denote by $\phi: \OO_K \to \Lambda$. Let $r \in \Lambda$ be the image of $1 \in \OO_K$. %We have $\mf h = \eta  T $ and we claim that $\mf h(r,r) = 1$. Indeed, %by Lemma \ref{lemma:EagreeT}, 
The fact that $Q = \phi^\ast E_A$ implies that $T_Q = \phi^\ast T_A = \phi^\ast T$. % by Lemma \ref{lemma:EagreeT}. 
Therefore, 
we have 
$$\eta^{-1} = T_Q(1,1)= T_A(\phi(1), \phi(1)) = T(\phi(1), \phi(1)) = T(r,r),$$ so that $\mf h(r,r) = \eta \cdot T(r,r) = 1$. We claim that $\mf h(x,r_\tau) = 0$, where the element $r_\tau \in  \left(\Lambda \otimes_\ZZ \CC \right)_\tau$ is the image of $r \in \Lambda$. To see this, write 
\[
\Psi(\OO_K) = L, \quad L \otimes \CC = W^{-1,0} \oplus W^{0,-1},
\]
and let $\alpha \in L$ correspond to $1 \in \OO_K$. Notice that $\left(L \otimes_\ZZ \CC \right)_\tau = W^{-1,0}_\tau$. Consequently, since the composition 
$$
W^{-1,0}_\tau = \left(L \otimes_\ZZ \CC \right)_\tau  \to \left(\Lambda \otimes_\ZZ \CC \right)_\tau = \rm H^{-1,0}_\tau \oplus \rm H^{0,-1}_\tau
$$
factors through the inclusion of $\rm H^{-1,0}_\tau$ into $\left(L \otimes_\ZZ \CC \right)_\tau$, we see that  
%and that the diagram
%\begin{equation}
%    \xymatrix{
%    \left(L \otimes_\ZZ \CC \right)_\tau \ar[r] & \left(\Lambda \otimes_\ZZ \CC \right)_\tau = \rm H^{-1,0}_\tau \oplus \rm H^{0,-1}_\tau \\
%    W^{-1,0}_\tau \ar[r]\ar@{=}[u]& \rm H^{-1,0}_\tau \ar@{^{(}->}[u]
%    }
%\end{equation}
%commutes. 
%$\alpha_\tau = \alpha^{-1,0}_\tau$ is mapped to 
$$r_\tau = r^{-1,0}_\tau \in \rm H^{-1,0}_\tau = \left( \rm H^{0,-1}_\tau \right)^\perp  = \langle x \rangle ^\perp,$$ and the claim follows. 
%It follows that $\mf h(x,r_\tau) = 0$. %hence $[x] \in H_r \subset \mr H$. 

Conversely, let $[x] \in \langle r_\CC \rangle ^\perp \subset \mr H$ with $r \in \Lambda$ such that $\mf h(r,r) = 1$ and consider the marked abelian variety $A = \rm H^{-1,0}/\Lambda$ corresponding to $[x]$. Define a homomorphism $\phi: \OO_K \to \Lambda$ by $\phi(1) = r$. Then $\phi$ can be shown to be a morphism of Hodge structures using the fact that its $\CC$-linear extension preserves the eigenspace decompositions.
%\begin{equation}
%    \oplus_{\varphi \in \Psi} \left( (\CC)_\varphi \oplus (\CC)_{\varphi\sigma} \right) \to      \oplus_{\varphi \in \Psi} \left( (\Lambda \otimes_\ZZ \CC)_\varphi \oplus (\Lambda \otimes_\ZZ \CC)_{\varphi\sigma} \right).
%\end{equation}
%Therefore, for every $\varphi \neq \tau$, we have $\phi \left( \CC_\varphi \right) \subset \rm H^{-1,0}_\varphi = (\Lambda \otimes_\ZZ \CC)_\varphi$ and $\phi \left( \CC_{\varphi\sigma} \right) \subset \rm H^{0,-1}_{\varphi\sigma} = (\Lambda \otimes_\ZZ \CC)_{\varphi\sigma}$ so all we need to prove is that if $\alpha \in L$ is the element corresponding to $1 \in \OO_K$, then $\phi(\alpha_\tau) \in \rm H^{-1,0}_\tau \subset H_\tau$ and $\phi(\alpha_{\tau\sigma}) \in \rm H^{0,-1}_{\tau\sigma} \subset H_{\tau\sigma}$. Note that the second equality follows from the first, % because $r_{\tau\sigma} = \sigma(r_\tau) = \phi(\sigma(\alpha_\tau))$. 
%and that $\phi(\alpha_\tau) = r_\tau$, $\mf h(x, r_\tau) = 0$ and $\langle x \rangle = \rm H^{0,-1}_\tau$. Therefore, $r_\tau \in \rm H^{-1,0}_\tau$ and the claim follows. 
We obtain an $\OO_K$-linear homorphism $\phi: \CC^g/\Psi(\OO_K) \to A$. The fact that $\mf h(r,r) = 1$ implies that $\phi$ preserves the polarizations on both sides.  
%$\phi^\ast E_A = Q$. 
%Since $\mf h(r,r) = 1$ we have $T(r,r)  = \eta^{-1}$. For $x,y \in \OO_K$, % then $\phi(x) = xr$ and $\phi(y) = yr \in \Lambda$. Therefore, %we have %$$ = %\textnormal{Tr}_{K/\QQ} \left( \eta^{-1} x \bar y \right) = $
%this gives $Q(\phi(x), \phi(y)) =  \textnormal{Tr}_{K/\QQ} \left( T(r,r) x \bar y  \right) =
%    \textnormal{Tr}_{K/\QQ} \left( T(xr,yr )  \right) = 
%    \textnormal{Tr}_{K/\QQ} \left( T(\phi(x), \phi(y))  \right) = E_A(\phi(x), \phi(y))$. 
%This fini\textnormal{Sh}es the proof. 
\end{proof}

%Recall also that $\eta \in \OO_K$ was chosen so that $K = F(\eta)$ and $\Im \tau_i (\eta) > 0$ for every $i$, and that $N \in \ZZ_{\geq 0}$ was chosen so that $\eta^{-1} = N \eta^{-1} \in \mf D_K^{-1}$. The latter condition ensured that the form $Q: \Psi(K) \times \Psi(K) \to \QQ$ defined as $Q(\Psi(x), \Psi(y)) = \textnormal{Tr}_{K/\QQ} \eta^{-1} x \bar y$ mapped takes integral values on $\Psi(\OO_K) \times \Psi(\OO_K)$. It is clear that if $\mf D_K = (\eta) = \eta \OO_K$, then $(\eta^{-1}) = \mf D_K^{-1}$ and $N$ can chosen to equal one - see also Remark \ref{remark:N=1}. 
\noindent
Observe that if the different $\mf D_K\subset \OO_K$ is a principal ideal $ (\eta) \subset \OO_K$, then we have 
\begin{align*}
\{x \in K : \textnormal{Tr}_{K/\QQ} \left(x \eta^{-1} \OO_K\right) \subset \ZZ \} &= \{x \in K: x \cdot \eta^{-1} \OO_K \subset \eta^{-1} \OO_K \} \\
&= \{x \in K: x \OO_K \subset \OO_K \} = \OO_K.
\end{align*}
Thus, $Q: \Psi(\OO_K) \times \Psi(\OO_K) \to \ZZ$ defines a \emph{principal} polarization on the torus $\CC^g/\Psi(\OO_K)$ in this case. In fact, for $\beta \in K$, the rational Riemann form 
\[
\Psi(K) \times \Psi(K) \to \QQ, \quad (\Psi(x), \Psi(y)) \mapsto \textnormal{Tr}_{K/\QQ} (\beta^{-1} x \bar y)
\]
defines a principal polarization on $\CC^g/\Psi(\OO_K)$ if and only if $(\rm i)$ we have that $\beta$ generates the different ideal $\mf D_K$, $(\rm{ii})$ we have that $\sigma(\beta) = -\beta$, and $(\rm{iii})$ we have that $\Im(\varphi(\beta)) > 0$ for every $\varphi \in \Psi$. This follows from the above; see also \cite{Wamelen99examplesof}.

Consider the following:
\begin{conditions} \label{crucialcondition}
\begin{enumerate}
\item \label{crucialone} The CM type $(K, \Psi)$ is primitive. 
%\textnormal{primitive} \cite[Definition 1.8]{milneCM}. 
\item \label{crucialtwo} We have $\mf D_K = (\eta)$ for some $\eta \in \OO_K$ such that $\sigma(\eta) = - \eta$. %Let $E: \Lambda \times \Lambda \to \ZZ$ be a symplectic form such that $E(ax, y) = E(x,a^\sigma y)$ for all $a \in \OO_K$, $x,y \in \Lambda$, let $T$ be the associated skew-hermitian form (Lemma \ref{lemma:equivalentforms}) and consider the $\OO_K$-valued hermitian form $\mf h = \eta \cdot T $ on $\Lambda$. 
\item \label{crucialthree} 
%Assume that 
The signature of $\mf h^{\tau_i}$ is $(n,1)$ for $i = 1$ and $(n+1,0)$ for $i \neq 1$.
\end{enumerate}
\end{conditions}

\begin{theorem} \label{th:conditionsimplyhypothesis}
Suppose that Conditions \ref{crucialcondition} hold. 
 %Then the following holds: if norm one vectors $r, t \in \Lambda$ %are such that $\mf h(r,r) = \mf h(t,t) = 1$,
Let $r_1,r_2 \in \Lambda$ satisfy $H_{r_1} \cap H_{r_2} \neq \emptyset$ and $H_{r_1} \neq H_{r_2} \subset {\bb C} H^n$ for $H_{r_i} = \langle r_{i,\CC} \rangle ^\perp \subset \CC H^n$. Then $\mf h(r_1,r_2) = 0$. 
%the signature condition in Proposition \ref{prop:HcorrespondsNonSimple}. %signature of $\mf h = \eta T$ equals $(r_1,s_1) = (n,1)$, $(r_i,s_i) = (n+1,0)$ for $2 \leq i \leq g$,
%Suppose that a canonical tuple $(\Psi, \eta)$ such that $\eta$ generates the different ideal $\mf D_K$ exists. %Let $T: \Lambda \to \Lambda \to \mf D_K^{-1}$ be 
%Under Conditions \ref{condition:one} and \ref{condition:two},
%For the hermitian form $\mf h: \Lambda \times \Lambda \to \OO_K$, the following holds: 
\end{theorem}
\begin{proof}
Let $[x] \in H_r \cap H_t \subset {\bb C} H^n(V)$, and let $A$ be an abelian variety whose isomorphism class gives $[x]$. Define $B$ to be the principally polarized abelian variety $\CC^g/\Psi(\OO_K)$. By Proposition \ref{prop:HcorrespondsNonSimple}, the roots $r$ and $t$ induce $\OO_K$-linear embeddings 
\[
\phi_1: B \hookrightarrow A \quad \tn{ and } \quad \phi_2 : B \hookrightarrow A
\]
of polarized abelian varieties. 
% (see Corollary \ref{cor:abeliansubvariety}). 
By Lemma \ref{lemma:abeliansplit} below, the $\phi_i$ induce decompositions 
$$ A \cong B \times C_1 \quad \tn{ and } \quad A \cong B \times C_2$$
as polarized abelian varieties. %, as one can deduce from the proof of the Poincar\'e Splitting Theorem as given in [GEM, Theorem (12.2)]. %There exists polarized abelian subvarieties $C_1, C_2 \subset A$ such that
%$$
%A \cong B \times C_1 \cong B \times C_2
%$$
%as polarized abelian varieties. 
Note that $B$ is non-decomposable as an abelian variety because $\End(B) \otimes_\ZZ \QQ = K$ is a field (here we use that the CM type $(K, \Psi)$ is primitive). %Let $\psi: B \times C_1 \to B \times C_2$ the isomorphism above. 
By \cite{debarreproduits}, the decomposition of $(A, \lambda)$ into non-decomposable polarized abelian subvarieties is unique, in the strong sense that if $(A_i, \lambda_i)$, $i\in \{1, \dotsc, r\}$ and $(B_j, \mu_j)$, $j \in \{1, \dotsc, m\}$ are polarized abelian subvarieties such that the natural homomorphisms $\prod_i(A_i, \lambda_i) \to (A, \lambda)$ and $\prod_j(B_j, \lambda_j) \to (A, \lambda)$ are isomorphisms, then $r = m$ and there exists a permutation $\sigma$ on $\{1, \dotsc, r\}$ such that $B_j$ and $A_{\sigma(j)}$ are \textit{equal} as polarized abelian subvarieties of $(A, \lambda)$, for every $j \in \{1, \dotsc, r\}$. Consequently, for the two abelian subvarieties
\[
B_i = \phi_i(B) \subset A,  \quad \tn{ we have either that } \quad B_1 = B_2 \subset A \quad \tn{ or that } \quad B_1 \cap B_2 = \{0\}.
\] 
Suppose first that $B_1 = B_2$. Then $$\OO_K\cdot r  = \phi_1(\OO_K) = \phi_2(\OO_K) = \OO_K \cdot t \subset \Lambda.$$ Therefore, $r = \lambda t$ for some $\lambda \in \OO_K^\ast$; but then $H_r = H_t$ which is absurd. Thus, we must have %$B_1 \cap B_2 = \{0\}$, i.e. %$B_1$ and $B_2$ are different direct factors of $A$ so that 
$$A \cong B_1 \times B_2 \times C$$ as polarized abelian varieties, for some polarized abelian subvariety $C$ of $A$. This implies that $$\rm H^{-1,0} = \text{Lie}(A) \cong \Lie(B_1) \times \Lie(B_2) \times \Lie(C),$$ which is orthogonal for the positive definite hermitian form $iE_\CC(x,\bar y)$ on $\rm H^{-1,0}$. 

Observe that $r_\tau = r^{-1,0}_\tau \in \rm H^{-1,0}_\tau$ and $t_\tau = t^{-1,0}_\tau \in \rm H^{-1,0}_\tau$: see the proof of Proposition \ref{prop:HcorrespondsNonSimple}. By Lemma \ref{lemma:agree}, we have
\begin{align*}
\mf h(r,t) &= \mf h^\tau(r_\tau, t_\tau) = \tau(\eta)\cdot T^\tau_\CC(r_\tau, t_\tau) \\
&= \tau(\eta) \cdot E_{\CC}(r_\tau, \bar{t_\tau}) = \tau(\eta) \cdot E_\CC (r_\tau^{-1,0}, \overline{t_\tau^{-1,0}}).
\end{align*}
Since $r_\tau^{-1,0} \in \Lie(B_1)$ and $t_\tau^{-1,0} \in \Lie(B_2)$, we have $i E_\CC(r_\tau^{-1,0}, \overline{t_\tau^{-1,0}}) = 0$. 
%Let $B \cong \prod_\nu B_\nu$ (resp. $C_i \cong \prod_{\mu_i} C_{\mu_i}$ be the decomposition of $B$ (resp. $C_i$) into non-decomposable polarized abelian subvarieties. This gives a decomposition
%$$
%A \cong \prod_\nu B_\nu \times \prod_{\mu_1} C_{\mu_1} \cong \prod_\nu B_\nu \times \prod_{\mu_2} C_{\mu_2}
%$$
%into indecomposable polarized abelian subvarieties. Such a decomposition is unique by [Debarre, ...]. Let $\psi: $
\end{proof}

\begin{lemma} \label{lemma:abeliansplit}
Let $A$ be an abelian variety over a field $k$, with polarization $\lambda: A \to \widehat A$. Let $B \subset A$ be an abelian subvariety such that the polarization $\mu = \lambda|_B$ is principal. There is a polarized abelian subvariety $Z \subset A$ such that $A \cong B \times Z$ as polarized abelian varieties. 
\end{lemma}

\begin{proof}
Let $W = \Ker(A \xrightarrow{\lambda} \widehat A \to \widehat B )$. Let $Z = W^0_{\text{red}}$. Then $Z$ is an abelian subvariety of dimension $\dim(A) - \dim(B)$ of $A$. The kernel of the natural homomorphism $B \times Z \to A$ is contained in $(B \cap Z) \times (B \cap Z)$. However, $B \cap Z \subset B \cap W = (0)$ because $\mu: B \to \widehat B$ is an isomorphism. Thus, $B \times Z \to A$ is an isomorphism. 
\end{proof}
\noindent
Finally, we remark that the condition on the different ideal $\mf D_K \subset \OO_K$ in Theorem \ref{th:conditionsimplyhypothesis} (see Conditions \ref{crucialcondition}) is satisfied in two interesting cases:

\begin{lemma} \label{lemma:discr}
Suppose that $K/\QQ$ is an imaginary quadratic extension, or that $K = \QQ(\zeta_n)$ is a cyclotomic field for some integer $n \geq 3$. % with $m$ a prime number, a power of $2$, a power of $3$, or the product of two distinct odd prime numbers. 
Then Condition \ref{crucialcondition}.\ref{crucialtwo} is satisfied. That is, we have $\mf D_K = (\eta) \subset \OO_K$ for some element $\eta \in \OO_K$ such that $\sigma(\eta) = -\eta$. 
%there exists a tuple ($\Psi, \eta$) as in Notation \ref{not:embeddingselementagain} such that %$\hfill \qed$
%$K = \QQ(\sqrt \Delta)$ with 
\end{lemma}
\begin{proof}
If $K/\QQ$ is imaginary quadratic with discriminant $\Delta$, %Let $\tau: K \to \CC$ be an embedding and require that $\tau(\sqrt \Delta)$ has positive imaginary part. 
then $\mf D_K  %\prod_i \mf p_i 
= (\sqrt \Delta)$ and the assertion is immediate. 
Let $n \geq 3$ be an integer, and consider the fields 
\[
K = \QQ(\zeta_n) \supset F = \QQ(\alpha), \quad \tn{ with } \quad  \alpha = \zeta_n + \zeta_n^{-1}. 
\]%Since $\zeta_n \in \OO_K$, its minimal polynomial over $F$ is a degree $2$ polynomial with coefficients in $\OO_F$; since $\zeta_n$ is a root of 
%Notice that $f(x) = x^2 - \alpha x + 1 \in \OO_F[x]$ is the minimal polynomial of $\zeta_n$ over $F$. %, it follows that $f_\alpha$ is the minimal polynomial of $\alpha$ over $F$. 
Since $\OO_K = \ZZ[\zeta_n]$ by \cite[I, Proposition 10.2]{Neukirch}, we have $ \OO_K = \OO_F[\zeta_n]$. Notice that $f(x) = x^2 - \alpha x + 1 \in \OO_F[x]$ is the minimal polynomial of $\zeta_n$ over $F$. We have $f'(\zeta_n) = 2 \zeta_n - \alpha \zeta_n = \zeta_n - \zeta_n^{-1}.$ Therefore, 
\[
\mf D_{K/F} = \left( f'(\zeta_n) \right) = \left(\zeta_n - \zeta_n^{-1}\right), \quad \textnormal{ see \cite[III, Proposition 2.4]{Neukirch}. }
\]
By \cite{Liang1976}, we know that $\OO_F = \ZZ[\alpha]$. Moreover, if $g(x) \in \ZZ[x]$ is the minimal polynomial of $\alpha$ over $\QQ$, then $\mf D_{F/\QQ} = (g'(\alpha))$. By \cite[III, Proposition 2.2]{Neukirch}, we have that $\mf D_{K/\QQ} = \mf D_{K/F} \mf D_{F/\QQ}$. Combining all this yields
$$
\mf D_{K/\QQ} = \mf D_{K/F} \mf D_{F/\QQ} = \left(\zeta_n - \zeta_n^{-1} \right) \cdot \left( g'(\alpha) \right) = \left( (\zeta_n - \zeta_n^{-1}) g'(\alpha) \right).
$$
\end{proof}

\begin{remark} \label{remark:avoidcondition}
It would be more natural to attach an orthogonal hyperplane arrangement $\ca H \subset \CC H^n$ to every primitive CM field $K$ and integral hermitian form $\mf h$ of hyperbolic signature, 
%signature $(r_1, s_1) = (n,1), r_i\cdot s_i = 0 \forall i \neq 1$, 
such that $\ca H = \mr H = \cup_{\mf h(r,r) = 1} \langle r_\CC\rangle^\perp$ if $\mf D_K = (\eta)$ for some $\eta \in \OO_K$ such that $\sigma(\eta) = - \eta$. This turns out to be possible. The idea is as follows. 

Consider our CM field $K$. Choose $\beta \in \OO_K - \OO_F$ such that $\beta^2 \in \OO_F$; then choose a CM type $\Psi = \{\tau_i: K \to \CC\}_{1 \leq i \leq g}$ such that $\Im(\tau_i(\beta)) > 0$ for all $i$. Let $\mf h$ be a non-degenerate hermitian form $\Lambda \times \Lambda \to \mf D_K^{-1}$ such that $\textnormal{sign}(\mf h^{\tau}) = (n,1)$ and $\textnormal{sign}(\mf h^{\tau_i}) = (n+1,0)$ for $i \neq 1$. %Suppose that the CM type $(K, \Psi)$ is simple. 
Let $\mr S$ be the set of fractional ideals $\mf a \subset K$ for which there exist an element $b \in \OO_F$ such that $\mf D_K \mf a \overline{\mf a} = (b\beta)$. By \cite[Theorem 4]{Wamelen99examplesof}, $\mr S$ is not empty. %there exists a fractional $\OO_K$-ideal $\mf a \subset K$ and an element $b \in \OO_F$ such that  $\mf D_K \mf a  \overline{\mf a} = (b\beta)$. 
For $\mf a \in \mr S$, define $\eta = b \beta \in \OO_K$ and consider the complex torus $B = \CC^g/\Psi(\mf a)$. It is equipped with the Riemann form $Q: \Psi(\mf a) \times \Psi(\mf a) \to \ZZ$, $(x,y) \mapsto \textnormal{Tr}_{K/\QQ}( \eta^{-1} x \bar y)$, %indeed integral valued since $\eta^{-1} \mf a \overline{\mf a} = \mf D_K^{-1}$ and in fact 
and $Q$ defines a principal polarization on $B$ \cite[Theorem 3]{Wamelen99examplesof}. Let $\mr R$ be the set of embeddings $\phi: \mf a \to \Lambda$, $\mf a \in \mr S$, such that $\mf h(\phi(x), \phi(y)) = x \bar y$ for all $x,y \in \mf a$. For $\phi \in \mr R$, one obtains a hyperplane $H_\phi = \{x \in \CC H^n: \mf h^{\tau}(x, \phi(\mf a)) = 0 \} \subset \CC H^n$. The sought-for hyperplane arrangement $\ca H \subset \CC H^n$ is defined as $\ca H = \cup_{\phi \in \mr R} H_\phi$. Indeed, if the CM type $(K, \Psi)$ is primitive, then $\ca H$ is an orthogonal arrangement by arguments similar to those used to prove Proposition \ref{prop:HcorrespondsNonSimple} and Theorem \ref{th:conditionsimplyhypothesis}. 
\end{remark}

\cleardoublepage
% Note that depending on your settings in the table of contents, subsections and subsubsections might appear virtually identical.
% Make sure to set the ToC depth and the numbering depth in the ToC the way you want.
\chapter{The moduli space of real binary quintics}\label{ch:binaryquintics}

\section{Introduction} \label{realbinaryintroduction}

In the previous Chapter \ref{ch:glueing} we saw that to any hermitian $\OO_K$-lattice $(\Lambda,h)$ over the ring of integers $\OO_K$ of a CM field $K$, one can attach a certain path metric space $P\Gamma \setminus Y$ in a canonical way. Under favourable circumstances, this glued space is a complete hyperbolic orbifold, thus a disjoint union of real ball quotients. %Even if the quotient spaces that were glued together were arithmetic, this may no longer hold for $P\Gamma \setminus Y$. 
In this Chapter \ref{ch:binaryquintics}, we work out the structure of the glued space in a specific example.

Indeed, describing moduli of real binary quintics was the main motivation behind setting up the glueing procedure in Chapter \ref{ch:glueing}. With that theory in place, we can prove that the moduli space of stable real binary quintics is isomorphic to the quotient of the hyperbolic plane by a non-arithmetic triangle group. 
%for a CM field $K$, a hermitian $\OO_K$-lattice $(\Lambda, h)$ of hyperbolic signature induces under some conditions a canonical complete real hyperbolic orbifold $P\Gamma \setminus Y$, the glued space.  

Let us explain this result in more detail. Let $X \cong \bb A^6_\RR$ be the affine space parametrizing homogeneous degree $5$ polynomials $F \in \RR[x,y]$. Let the varieties 
\[
X_0 \subset X, \quad X_s \subset X
\]
parametrize polynomials with distinct roots, respectively 
%and $X_s \subset X$ 
polynomials with roots of multiplicity at most two (i.e. stable in the sense of geometric invariant theory). The goal of Chapter \ref{ch:binaryquintics} is to study the moduli spaces
\begin{align*} 
 \ca M_s(\RR) &\coloneqq  \GL_2(\RR) \setminus X_s(\RR) \quad \quad \tn{ and } \\
 \ca M_0(\RR) &\coloneqq \GL_2(\RR) \setminus X_0(\RR),
 \end{align*}of stable and smooth \textit{real binary quintics}. If $P_s \subset (\PP^1_\RR)^5$ is the variety that parametrizes ordered five-tuples $(x_1, \dotsc, x_5)$ such that no three $x_i$ coincide (c.f. \cite{MR0437531}), and $P_0 \subset P_s$ the subvariety of five-tuples whose coordinates are pairwise distinct, then 
\begin{align*}
\ca M_s(\RR) &\cong \PGL_2(\RR) \setminus (P_s/\mf S_5)(\RR), \quad \quad \textnormal{and}\\
 \ca M_0(\RR) &\cong \PGL_2(\RR) \setminus (P_0/\mf S_5)(\RR).
\end{align*}
In other words, $\ca M_0(\RR)$ is the space of subsets $S \subset \PP^1(\CC)$ of cardinality $|S| = 5$ stable by complex conjugation modulo real projective transformations, and in $\ca M_s(\RR)$ one or two pairs of points are allowed to collapse. For $i = 0,1,2$, we define $\mr M_{i}$ to be the connected component of $\ca M_0(\RR)$ parametrizing five-tuples in $\PP^1(\CC)$ with $2i$ complex and $5 - 2i$ real points. 

There is a natural occult period map that defines an isomorphism \[
\ca M_s(\CC) = \GL_2(\CC) \setminus X_s(\CC) \xrightarrow{\sim} P\Gamma \setminus \CC H^2
\] 
%$\PGL_2(\CC) \setminus (P_s/\mf S_5)$ 
for a certain arithmetic ball quotient $P\Gamma \setminus \CC H^2$ (see Theorem \ref{th:delignemostow}, which depends on \cite{shimuratranscendental, DeligneMostow}). Moreover, one can prove that strictly stable quintics correspond to points in a hyperplane arrangement $\mr H \subset \CC H^2$ (Proposition \ref{prop:stableperiodshyperplane}). Investigating the equivariance of the period map with respect to a suitable set of anti-holomorphic involutions $\alpha_i : \CC H^2 \to \CC H^2,$ we obtain the following real analogue:
%Since for a suitable set of real structures $\{\alpha_i: \CC H^2 \to \CC H^2\}$, the period map is defined over $\RR$, we obtain: 
\begin{theorem} \label{th:theorem01}
For each $i \in \{0,1,2\}$, the period map induces an isomorphism of real analytic orbifolds 
\begin{align} \label{iso:smoothcase}
\mr M_i \cong P \Gamma_i \setminus \left(\RR H^2 - \mr H_i \right).
\end{align}
Here $\RR H^2$ is the real hyperbolic plane, $\mr H_i$ a union of geodesic subspaces in $ \RR H^2$ and $P\Gamma_i$ an arithmetic lattice in $\textnormal{PO}(2,1)$. Moreover, the $P\Gamma_i$ are projective orthogonal groups attached to explicit quadratic forms over $\ZZ[\zeta_5 + \zeta_5^{-1}]$, see (\ref{eq:explicitquadraticforms}).  
\end{theorem}
\noindent
In particular, Theorem \ref{th:theorem01} endows each component $\mr M_i$ with a hyperbolic metric. Since one can deform the topological type of a $\Gal(\CC/\RR)$-stable five-element subset of $\PP^1(\CC)$ by allowing two points to collide, the compactification $\ca M_s(\RR) \supset \ca M_0(\RR)$ is connected. One may wonder whether the metrics on the components $\mr M_i$ extend to a metric on the whole of $\va{\ca M_s(\RR)}$. If so, what the resulting space looks like at the boundary? The answer to this question is the main result of this chapter:
\begin{theorem} \label{th:theorem02}
There exists a complete hyperbolic metric on $\va{\ca M_s(\RR)}$ that restricts to the metrics on $\mr M_i$ induced by Theorem \ref{th:theorem01}. %In fact, $\va{\overline{\ca M}(\RR)}$ is a complete hyperbolic orbifold; 
If $\overline{\mr{M}}_\RR$ denotes the resulting metric space, and 
\begin{align}\label{PGAMMAR}
P\Gamma_{3,5,10} =  \langle \alpha_1, \alpha_2, \alpha_3 | \alpha_i^2 = (\alpha_1\alpha_2)^3 = (\alpha_1\alpha_3)^5 = (\alpha_2\alpha_3)^{10} = 1 \rangle,
\end{align} then there exist open embeddings $P \Gamma_i \setminus \left(\RR H^2 - \mr H_i \right) \subset P\Gamma_{3,5,10}\setminus \RR H^2$ and an isometry 
\begin{align} \label{isometry}
\overline{\mr{M}}_\RR \cong P\Gamma_{3,5,10} \setminus \RR H^2
\end{align} extending the orbifold isomorphisms (\ref{iso:smoothcase}) in Theorem \ref{th:theorem01}. %If $\overline{\mr{M}}_\RR$ denotes the resulting metric space, then 
In particular, $\overline{\mr{M}}_\RR$ is isometric to the hyperbolic triangle $\Delta_{3,5,10}$ of angles $\pi/3, \pi /5, \pi/10$, see Figure~\ref{fig:triangle} on the next page. 
%where $P\Gamma_\RR$ is isomorphic to the group $\langle 
\end{theorem}

%\newgeometry{margin=0.5in, top=0.01in, bottom=1in, footskip=0.5in}%, bottom=0.5in}

%\hspace*{-3cm}\includegraphics[scale=0.20]{triangle15transparent.png}

%\restoregeometry
%\newgeometry{margin=0.8in, top=1in, bottom=1in, footskip=0.5in}%, bottom=0.5in}

\begin{figure} 
\hspace*{-6.5cm}
    \includegraphics[scale=0.24]%\linewidth]
    {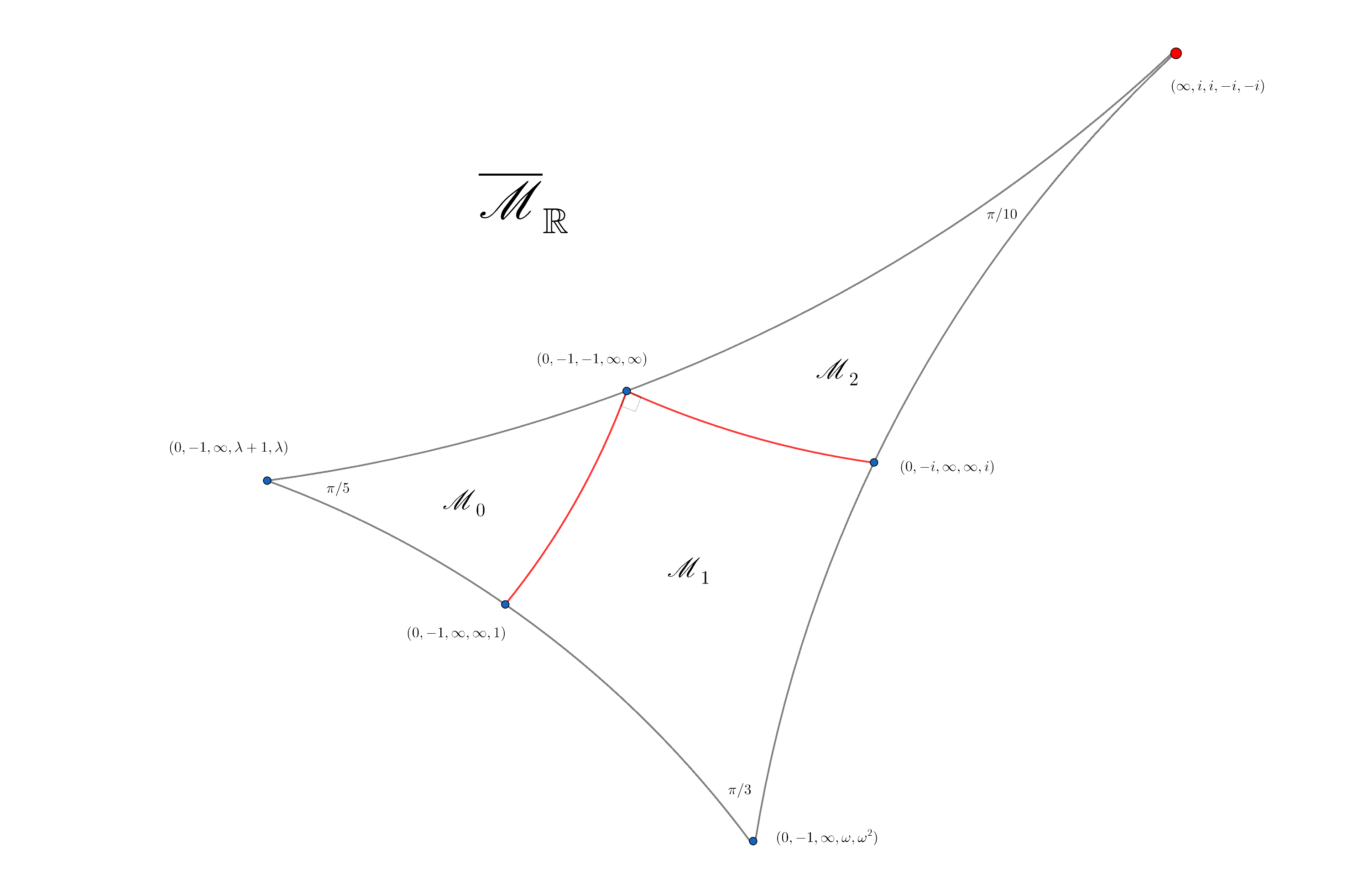}
    \caption{$\overline{\mr M}_\RR$ as the hyperbolic triangle $\Delta_{3,5,10} \subset \RR H^2$. Here $\lambda = \zeta_5 + \zeta_5^{-1}$ and $\omega = \zeta_3$.}
    \label{fig:triangle}
\end{figure}

%\textit{Figure 1: $\overline{\mr M}_\RR$ as the hyperbolic triangle $\Delta_{3,5,10} \subset \RR H^2$. Here $\lambda = \zeta_5 + \zeta_5^{-1}$ and $\omega = \zeta_3$.}
\begin{remark} \label{remark:takeuchi}
% As in \textit{loc. cit.}, 
The lattice $P\Gamma_{3,5,10} \subset \textnormal{PO}(2,1)$ is \textit{non-arithmetic}, as follows from \cite{takeuchi}. 
\end{remark}
\noindent
The isometry (\ref{isometry}) in Theorem \ref{th:theorem02} seems to provide $\overline{\mr M}_\RR$ with a hyperbolic orbifold structure. The proof of Theorem \ref{th:theorem02} actually goes in the other direction: we use the theory developed in the previous Chapter \ref{ch:glueing} to show that the pieces on the right hand side of (\ref{iso:smoothcase}) glue into the hyperbolic quotient on the right hand side of (\ref{isometry}), and afterwards, we prove that the period maps (\ref{iso:smoothcase}) glue to form (\ref{isometry}). 

By Theorem \ref{th:theorem02}, the topological space $\va{\ca M_s(\RR)}$ underlies two orbifold structures: the one of $\ca M_s(\RR)$ and the one of $\overline{\mr M}_\RR$. These structures only differ at one point of $\va{\ca M_s(\RR)}$, which is the point $(\infty, i, i, -i, -i)$ in Figure~\ref{fig:triangle}, see Proposition \ref{prop:conesreflectors}. 

%the metric space $\overline{\mr M}_\RR$ with a complete hyperbolic orbifold structure, and then use the Ehresmann--Thursten uniformization theorem [reference] to prove (\ref{isometry}). 

As explained in Section \ref{intro:sub:periodmaps}, equivariant (occult) period maps are often used in real algebraic geometry to uniformize components of a real moduli space of smooth varieties. Think of real moduli of abelian varieties \cite{grossharris}, algebraic curves \cite{seppalasilhol2}, K3 surfaces \cite{Nikulin1980} and quartic curves \cite{heckman2016hyperbolic}. Only recently, Allcock, Carlson and Toledo have shown that for moduli of cubic surfaces \cite{realACTsurfaces} and binary sextics \cite{realACTnonarithmetic, realACTbinarysextics}, the real ball quotient components can be glued along the hyperplane arrangement to form a larger ball quotient that parametrizes moduli of stable varieties. Binary quintics provide the first new example of this phenomenon. 

\begin{remark} Our glueing construction relies on Condition \ref{conditionast}, saying that $\mr H \subset \CC H^n$ is an \textit{orthogonal arrangement} in the sense of \cite{orthogonalarrangements}. Such arrangements are interesting in their own right. Indeed, %with the above notation, %suppose that $n > 1$ and that $h$ is definite at places different from $\tau$ (i.e. that $P\Gamma \subset \textnormal{PU}(n,1)$ is a lattice). Then 
if $n > 1$ then the orbifold fundamental $\pi_1^\textnormal{orb} \left( P\Gamma \setminus \left( \CC H^n - \mr H \right) \right)$ is not a lattice in any Lie group with finitely many connected components [\emph{loc.cit.}, Theorem 1.2]. In particular, neither $\pi_1 \left( P_0 / \mf S_5  \right)$ nor $\pi_1^\textnormal{orb}\left( \mr M_\CC \right)$ is a lattice in any Lie group with finitely many connected components. %The hyperplane arrangement $\mr H \subset \CC H^n$ is an orthogonal arrangement under the following condition, satisfied by quadratic and cyclotomic CM fields (c.f. Section \ref{unitaryshimura}): %Imaginary quadratic and cyclotomic number fields satisfy ,\ref{conditionastast}, 
\end{remark}

\section{Moduli of complex binary quintics} \label{complexball}

Recall from Section \ref{realbinaryintroduction} that $X \cong \bb A^6_\RR$ is the real affine space of homogeneous degree $5$ polynomials $F \in \RR[x,y]$, $X_0$ the subvariety of polynomials with distinct roots, and $X_s \subset X$ the subvariety of polynomials with roots of multiplicity at most two, i.e. non-zero polynomials whose class in the associated projective space is stable in the sense of geometric invariant theory \cite{GIT} for the action of $\SL_{2, \RR}$ on it. %In the remaining subsections, 
%\ref{complexperiodmaps}, \ref{realperiodmaps}, \ref{modulifivepoints} and \ref{sec:monodromygroups}, 
\begin{notation} \label{binarynotation}
In Chapter \ref{ch:binaryquintics}, the CM field $K$ is the cyclotomic field $ \QQ(\zeta)$, with $\zeta = \zeta_5 = e^{2 \pi i /5} \in \CC$. Recall that the ring of integers $\OO_K$ of $K$ is the ring $\ZZ[\zeta]$ \cite[Chapter I, Proposition (10.2)]{Neukirch}. We let $\mu_K \subset \OO_K^\ast$ denote the group of finite units as before; in this case $\mu_K$ is generated by $-\zeta$, and is therefore of order ten.  
\end{notation}
\noindent
The goal of this Section \ref{complexball} is to prove that there exists a hermitian $\OO_K$-lattice $\Lambda$ of rank three, such that, if $\Gamma = \Aut(\Lambda)$, $\GG(\CC) = \GL_2(\CC)/D$, and $\mr H \subset \CC H^n$ is the hyperplane arrangement defined by the norm one vectors in $\Lambda$, then there is an isomorphism of analytic spaces $\ca M_s(\CC) =  \GG(\CC) \setminus X_s(\CC) \cong P\Gamma \setminus \CC H^2$ restricting to an orbifold isomorphism $\ca M_0(\CC) =  \GG(\CC) \setminus X_0(\CC) \cong P\Gamma \setminus \left( \CC H^2 - \mr H \right)$. 
%identifying $\mr MG\setminus X_$

\subsection{The Jacobian of a cyclic quintic cover of the projective line} \label{jacofcyc}

We begin with the following:
\begin{lemma} \label{lemma:refinedhodge}
Let $Z \subset \bb P^1_\CC$ be a smooth quintic hypersurface. Let $ \bb P^2_\CC \supset C \to \PP^1_\CC$ be the quintic cover of $\bb P^1$ ramified along $Z$. Then $C$ has the following refined Hodge numbers:
\begin{align}
    h^{1,0}(C)_{\zeta} = 3, \white h^{1,0}(C)_{\zeta^2} = 2, \white 
    h^{1,0}(C)_{\zeta^3} = 1, \white 
    h^{1,0}(C)_{\zeta^{4}} = 0  \\
    h^{0,1}(C)_{\zeta} = 0, \white \rm H^{0,1}(C)_{\zeta^2} = 1, \white 
    h^{0,1}(C)_{\zeta^3} = 2, \white 
    h^{0,1}(C)_{\zeta^{4}} = 3. 
\end{align}
\end{lemma}
\begin{proof}
This follows from the Hurwitz-Chevalley-Weil formula, see \cite[Proposition 5.9]{Moonen2011TheTL}. Alternatively, see \cite[Section 5]{carlsontoledomonodromy}. 
\end{proof}
\noindent
Fix a point $F_0 \in X_0(\CC)$ and let 
\begin{equation} \label{eq:coverp1}
C = \{z^5 = F_0(x,y) \} \subset \PP^2_\CC
\end{equation}
be the corresponding cyclic cover of $\PP^1_\CC$. 
%and let $C$ be the corresponding curve. % and let $\widetilde C \to C$ be the normalization of $C$. 
Let 
\[
\left(A = J(C) = \text{Pic}^0(C), \quad \lambda \colon A \to \wh A, \quad \iota \colon \OO_K = \ZZ[\zeta] \to \End(A) \right)
\]
be the Jacobian of $C$, viewed as a principally polarized abelian variety of dimension six equipped with an $\OO_K$-action compatible with the polarization, see (\ref{KRconditions}). 

Write $\Lambda = \rm H_1(A(\CC), \ZZ)$. We have $\Lambda \otimes_\ZZ \CC = \rm H^{-1,0} \oplus \rm H^{0,-1}$, the Hodge decomposition of $\Lambda \otimes_\ZZ \CC$. 
%Let $A = J_{\textbf t}$ for some $\textbf t \in T(\CC)$, with action $\iota: \OO_K = \ZZ[\zeta] \to \End(A)$; write $A(\CC) = H/\Lambda$, where $H = \Lambda \otimes_\ZZ \CC = \rm H^{-1,0} \oplus H^{0,-1}$. 
\begin{comment}
For $i \in \ZZ/5$, define $$\rm H^{-1,0}_{(i)} = \{x \in \rm H^{-1,0}: \iota(\zeta)(x) = \zeta^ix \}.$$
Let $\rm H^{-1,0}_i = \dim_\CC \rm H^{-1,0}_{(i)}$. Since 
$
\rm H^{-1,0} = \Lie(A) = \rm H^1(C, \OO_C)$, we have $\rm H^{-1,0}_{(i)} = i-1$ for $1 \leq i \leq 4$.
\end{comment}
% by the above. 
%The dimensions $d_i:= \dim_\CC H^{0,-1}_{(i)}$ are given by $d_i = 0$ if $i \equiv 0 \mod 5$ and 
%$$
%d_1 = 2, \white d_2 = 0, \white d_3 = 3,\white d_4 = 1.
%$$
%\begin{equation}
%d_i = -1 + \sum_{i = 1}^5 \left\langle \frac{-na_i}{m} \right\rangle\white \white \textnormal{ if } \white n \not \equiv 0 \mod m,
%\end{equation}
%where $\langle x \rangle = x - \lfloor x \rfloor$ denotes the fractional part of a number $x$.
%Indeed, this follows from the Hurwitz-Chevalley-Weil formula, see \cite[Proposition 5.9]{Moonen2011TheTL}. 
%Hence if $h^{i,j}_{\zeta^k} = \dim_\CC H^{i,j}_{\zeta^k}$, then 
%\begin{equation} \label{eq:signaturefivepoints}
%\rm H^{-1,0}_\zeta = d_4 = 1, \white \rm H^{-1,0}_{\zeta^2} = d_3 = 3, \white \rm H^{-1,0}_{\zeta^3} = d_2 = 0, \white \rm H^{-1,0}_{\zeta^4} = d_1 = 2.  
%\end{equation}
%whose branch divisor is given by $\{F^2 = 0\} \subset \PP^1_\CC$. 
Define a CM-type $\Psi \subset \Hom(K, \CC)$ as follows:
\begin{align}\label{CMTYPE}
\tau_i: K \to \CC, \quad \tau_1(\zeta) = \zeta^3, \quad  \tau_2(\zeta) = \zeta^4; \quad \quad \Psi = \set{\tau_1, \tau_2}. 
\end{align}
Since 
$
\rm H^{-1,0} = \Lie(A) = \rm H^1(C, \OO_C) = \rm H^{0,1}(C)$, Lemma \ref{lemma:refinedhodge} implies that
\begin{equation} \label{eq:signaturefivepoints2}
\dim_\CC \rm H^{-1,0}_{\tau_1} = 2, \white \dim_\CC \rm H^{-1,0}_{\tau_1\sigma} = 1, \white \dim_\CC \rm H^{-1,0}_{\tau_2} = 3, \white \dim_\CC \rm H^{-1,0}_{\tau_2\sigma} = 0.  
\end{equation}
Define $\eta = 5/(\zeta - \zeta^{-1})$. Then $\mf D_K = (\eta)$ (see Lemma \ref{lemma:discr}). Let 
\[E: \Lambda \times \Lambda \to \ZZ\] 
be the alternating form corresponding to the polarization $\lambda$ of the abelian variety $A$. For $a \in \OO_K$ and $x,y \in \Lambda$, we have $E(\iota(a)x,y) = E(x, \iota(a^\sigma)y)$. Define 
\[
T \colon \Lambda \times \Lambda \to \mf D_K^{-1}, \quad T(x,y) = \frac{1}{5}\sum_{j = 0}^{4}\zeta^jE\left( x, \iota(\zeta)^j y \right).
\]
By Example \ref{examplesunitary}.\ref{ex:unittwo}, this is the skew-hermitian form corresponding to $E$ via Lemma \ref{lemma:equivalentforms}. We obtain a hermitian form on the free $\OO_K$-module $\Lambda$ as follows:
\begin{equation} \label{eq:hermitianformonbinaryquinticlattice}
    \mf h: \Lambda \times \Lambda \to \OO_K, \white
\mf h(x,y) = \eta T(x,y) = (\zeta - \zeta^{-1})^{-1}\sum_{j = 0}^{4}\zeta^jE\left( x, \iota(\zeta)^j y \right). 
\end{equation}
By Lemma \ref{lemma:equivalentforms}, the hermitian lattice $(\Lambda, \mf h)$ is unimodular, because $(\Lambda, E)$ is unimodular. 
%\begin{align*}
%\{x \in \Lambda_\QQ: \mf h(x, \Lambda) \subset \OO_K\}& = 
%\{x \in \Lambda_\QQ: T(x, \Lambda) \subset \mf D_K^{-1}\} \\
%&= 
%\{x \in \Lambda_\QQ: E(x, \Lambda) \subset \ZZ\} \\
%&= \Lambda.
%\end{align*}
%Let $\eta = -5/(\zeta^3- \zeta^{-3}) \in K$; then $\mf D_K = (\eta)$ by Example \ref{examples:someexx}.\ref{ex:unittwoex2}, and 
%is an hermitian form such that $\{x \in \Lambda_\QQ: h(x, \Lambda) \subset \OO_K\} = \Lambda$ by Lemma \ref{lemma:selfdual}, 
For each embedding $\varphi: K \to \CC$, the restriction of the hermitian form $\varphi(\eta)\cdot E_\CC(x, \bar y)$ on $\Lambda_\CC$ 
%\times \Lambda_\CC \to \CC$ 
to $(\Lambda_\CC)_{\varphi} \subset \Lambda_\CC$ coincides with 
$
\mf{h}^{\varphi}$  % : (\Lambda_\CC)_\varphi \times (\Lambda_\CC)_\varphi \to \CC
by Lemma \ref{lemma:agree}. %One checks that $\Im \left( \tau_1(\eta) \right) > 0$ and
%There is some real number $r \in \RR_{> 0}$ such that
%$$
%\Im \left( \tau_1(\eta) \right) = - \Im \left( \frac{p}{\zeta^{-3}- \zeta^{3}} \right) = - r \cdot ( - \Im (\zeta^{-3} - \zeta^{3}))  = r \cdot \Im (\zeta^2 - \zeta^{-2}) > 0. 
%$$
% $\Im \left( \tau_2(\eta) \right) > 0$. 
%there is some real number $t \in \RR_{> 0}$ such that
%$$
%\Im \left( \tau_2(\eta) \right) = - \Im \left( \frac{p}{\zeta^{6}- \zeta^{-6}} \right) = - t \cdot ( - \Im (\zeta - \zeta^{-1}))  = t \cdot \Im (\zeta - \zeta^{-1}) > 0. (\Lambda_\CC)_\varphi
%$$
Since $\Im (\tau_i(\zeta - \zeta^{-1})) < 0$ for $i = 1,2$,  %the hermitian form $\mf{h}^{\tau_i}$ %= \tau_i(\eta)T^{\tau_i}$ coincides on $H_{\tau_i}$ with $i E(x, \bar y)$ on $(\Lambda_\CC)_{\tau_i}$ up to some positive real scalar, hence 
%is positive definite on $\rm H^{-1,0}_{\tau_i}$ and negative definite on $H^{0,-1}_{\tau_i}$ for $i \in \{1,2\}$. %Since $\rm H^{-1,0}_{\tau_1} = h^{0,-1}_{\tau_1\sigma} = 2$, $\rm H^{-1,0}_{\tau_1\sigma}  = h^{0,-1}_{\tau_1} = 1$, $\rm H^{-1,0}_{\tau_2} = h^{0,-1}_{\tau_2\sigma} = 3$, $ \rm H^{-1,0}_{\tau_2\sigma} = h^{0,-1s}_{\tau_2} = 0$, 
%Therefore, the 
the signature of $\mf{h}^{\tau_i}$ on $V_i = \Lambda \otimes_{\OO_K, \tau_i}\CC $ is 
\begin{align} \label{quintic-cases}
\tn{sign}(\mf{h}^{\tau_i}) = 
\begin{cases}
(\rm h^{-1,0}_{\tau_1}, h^{0,-1}_{\tau_1}) = (2,1) & \textnormal{ for } i = 1, \quad \textnormal{ and } \\
(\rm h^{-1,0}_{\tau_2}, h^{0,-1}_{\tau_2}) = (3,0)& \textnormal{ for } i = 2. 
\end{cases}
\end{align}
%$(\rm h^{-1,0}_{\tau_1}, h^{0,-1}_{\tau_1}) = (2,1)$ for $i = 1$ and %and the signature of $\mf{h}^{\tau_2}$ is 
%$(\rm h^{-1,0}_{\tau_2}, h^{0,-1}_{\tau_2}) = (3,0)$ for $i = 2$.

\subsection{The monodromy representation} \label{sec:monodromy}

Consider the real algebraic variety $X_0$ introduced in Section \ref{realbinaryintroduction}. %It is a fine moduli space for smooth binary quintics. %Recall that $\mu_K \cong \ZZ/10$ is the group of finite units in $\OO_K^\ast$, see Notation \ref{binarynotation}, and observe that the subgroup $\langle \zeta \rangle \subset \GL_2
Let $D \subset \GL_2(\CC)$ be the subgroup $D = \set{\zeta^i\cdot I_2 } \subset \GL_2(\CC)$ of scalar matrices $\zeta^i \cdot I_2$, where $I_2 \in \GL_2(\CC)$ is the identity matrix of rank two, and define
\begin{align} \label{def:GC}
\bb G(\CC) = \GL_2(\CC)/D.  
\end{align}
The group $\GG(\CC)$ acts from the left on $X_0(\CC)$ in the following way: if $F(x,y) \in \CC[x,y]$ is a binary quintic, we may view $F$ as a function $\CC^2 \to \CC$, and define $g \cdot F = F(g^{-1})$ for $g \in \GG(\CC)$. This gives a canonical isomorphism of complex analytic orbifolds
\[
\ca M_0(\CC) = \GG(\CC) \setminus X_0(\CC),
\]
where $\ca M_0$ is the moduli stack of smooth binary quintics. 

Consider two families 
\[
\pi: \mr C \to X_0 \quad \tn{ and } \quad  \phi: J \to X_0,
\] 
defined as follows. We define $\pi$ as the universal family of cyclic covers $C \to \PP^1$ ramified along a smooth binary quintic $\{F = 0\} \subset \PP^1$. We let $\phi$ be the relative Jacobian of $\pi$. By Section \ref{jacofcyc}, $\phi$ is an abelian scheme of relative dimension six over $X_0$, with $\OO_K$-action of signature $\{(2,1), (3,0)\}$ with respect to $\Psi = \{\tau_1, \tau_2\}$. 

Let $\bb V = R^1\pi_\ast \ZZ$ be the local system of hermitian $\OO_K$-modules underlying the abelian scheme $J/X_0$. Attached to $\bb V$, we have a representation 
\[
\rho': \pi_1(X_0(\CC), F_0) \to \Gamma, \quad \Gamma = \Aut_{\OO_K}(\Lambda, \mf h),
\] whose composition with the quotient map $\Gamma \to P\Gamma = \Gamma / \mu_K$ defines a homomorphism 
\begin{equation} \label{eq:monodromy}
\rho: \pi_1(X_0(\CC), F_0) \to P\Gamma.
\end{equation}
We shall see that $\rho$ is surjective, see Corollary \ref{cor:surjectivemon} below. 

\subsection{Marked binary quintics} \label{sec:markedbinary}

For $F  \in X_0(\CC)$, define $Z_F$ as the hypersurface
\[
 Z_F = \{F = 0\} \subset \PP^1_\CC.
\] 
A \textit{marking} of $F$ is a ring isomorphism $m: \rm H^0(Z_F(\CC), \ZZ) \xrightarrow{\sim} \ZZ^5$. To give a marking is to give a labelling of the points $p \in Z_F(\CC)$. 
Let $\ca N_0$ be the space of marked binary quintics $(F, m)$; this is a manifold, equipped with a holomorphic map 
\begin{align}\label{markedquintics}
\ca N_0 \to X_0(\CC). 
\end{align}Let
\[
\psi: \ca Z \to X_0(\CC)
\]
be the universal complex binary quintic, and consider the local system $H = \psi_\ast \ZZ$ of stalk $H_F = \rm H^0(Z_F(\CC), \ZZ)$ for $F \in X_0(\CC)$. Then $H$ corresponds to a monodromy representation
\begin{equation}\label{eq:monodromy3}
\tau: \pi_1(X_0(\CC), F_0) \to \mf S_5. 
\end{equation}
It can be shown that $\tau$ is surjective using the results of \cite{beauvillemonodromie}. This implies that (\ref{markedquintics}) is covering space, i.e. that $\ca N_0$ is connected. 

If we choose a marking $m_0: \rm H^0(Z_{F_0}(\CC), \ZZ) \cong \ZZ^5$ lying over our base point $F_0 \in X_0(\CC)$, we obtain an embedding 
$
\pi_1 \left( \ca N_0, m_0 \right) \hookrightarrow  \pi_1(X_0(\CC), F_0)$, 
%Then $\ca M_0 \to X_0(\CC)$ is a covering space, 
%and let $\ca M_s \to X_s(\CC)$ be the Fox completion of the spread $\ca M_0 \to X_0(\CC) \to X_s(\CC)$. 
%By naturality of Fox completions, $G$ acts on $\ca M_s$. On the other hand, recall that $\theta = \zeta - \zeta^{-1}$; define $\Gamma_\theta \subset \Gamma$ to be the kernel of the representation $\Gamma \to \Aut(\Lambda/\theta\Lambda)$. 
%Define $PX_0 \subset PX_s \subset \PP(\rm H^0\PP^1_\CC, \OO_{\PP^1_\CC}(5))^\vee) \cong \PP^5_\CC$ to be the smooth, respectively stable locus in the projective space of binary quintics, $P\ca M_0$ the space of all markings of classes $[F] \in PX_0(\CC)$, and $P\ca M_s \to PX_s(\CC)$ the Fox completion of the spread $P\ca M_0 \to PX_0(\CC) \to PX_s(\CC)$. %There is a canonical $\PGL_2(\CC)$-equivariant isomorphism $P\ca M_0 \cong M$, and hence $\GG(\CC) \setminus \ca M_0 \cong \PGL_2(\CC) \setminus P\ca M_0$. %Moreover, $\PGL_2(\CC) \setminus P\ca M_s \cong Q_s$.
%
%Remark that we have $P\ca M_0 \cong P_0 = \{(x_1, \dotsc, x_5) \in (\PP^1(\CC))^5: x_i \neq x_j \forall i \neq j \}$, canonically. Choosing a base point $([F_0], m)$ lying over $[F_0]$ gives a representation 
whose composition with $\rho$ in (\ref{eq:monodromy}) defines a homomorphism 
\begin{equation}\label{eq:monodromy2}
\mu: \pi_1(\ca N_0, m_0) \to P\Gamma. 
\end{equation}
Define $\theta = \zeta - \zeta^{-1}$ and consider the $3$-dimensional $\bb F_5$ vector space $\Lambda/\theta\Lambda$ and the quadratic space 
$$
W \coloneqq \left(\Lambda/\theta\Lambda,q\right),
$$ where $q$ is the quadratic form obtained by reducing $\mf h$ modulo $\theta\Lambda$. Define two groups $\Gamma_\theta$ and $P\Gamma_\theta$ as follows:
\[
\Gamma_\theta = \Ker\left( \Gamma \to \Aut(W) \right), \quad P\Gamma_\theta = \Ker\left( P\Gamma \to P\Aut(W) \right) \subset \text{PU}(2,1). 
\]
Remark that the composition $\ca N_0 \to X_0(\CC) \to X_s(\CC)$ admits an essentially unique \textit{completion} $\ca N_s \to X_s(\CC)$, see \cite{fox} or \cite[\S8]{DeligneMostow}. Here $\ca N_s$ a manifold and $\ca M_s \to X_s(\CC)$ is a ramified covering space. 
\begin{proposition} \label{prop:commutativemonodromy}
The image of $\mu$ in (\ref{eq:monodromy2}) is the group $P\Gamma_\theta$, and the induced homomorphism $ \pi_1(X_0(\CC), F_0)/ \pi_1 \left( \ca N_0, m_0 \right)= \mf S_5 \to P\Gamma /P\Gamma_\theta$ is an isomorphism. In other words, we obtain the following commutative diagram with exact rows:
\begin{equation} \label{eq:commutativemonodromy}
    \xymatrix{
    0 \ar[r] & \pi_1(\ca N_0, m_0) \ar@{->>}[d]^\mu\ar[r] & \pi_1(X_0(\CC), F_0)\ar[d]^\rho \ar[r]^{\;\;\;\;\;\; \tau} & \mf S_5\ar[d]^*[@]{\sim}_\gamma \ar[r] & 0 \\
    0 \ar[r] & P\Gamma_\theta \ar[r] & P\Gamma \ar[r] & P\Aut(W) \ar[r] & 0.
    }
\end{equation}

\begin{proof}
Consider the quotient 
\[
Q = \GG(\CC) \setminus \ca N_0 = \PGL_2(\CC) \setminus P_0(\CC),\]
where $P_0 \subset (\PP^1_\RR)^5$ is the subvariety of distinct five-tupes, see Section \ref{ch:binaryquintics}. Let $0 \in Q$ be the image of $m_0 \in \ca N_0$. In \cite{DeligneMostow}, Deligne and Mostow define a hermitian space bundle $B_Q \to Q$ over $Q$ whose fiber over $0 \in Q$ is $\CC H^2$. Consequently, writing $V_1 = \Lambda \otimes_{\OO_K, \tau_1} \CC$, this gives a monodromy representation $$\pi_1(Q,0) \to \textnormal{PU}(V_1, \mf h^{\tau_1}) \cong \textnormal{PU}(2,1)$$ whose image we denote by $\Gamma_{\text{DM}}$. Kondō has shown that in fact, $\Gamma_{\text{DM}} = P\Gamma_\theta$ \cite[Theorem 7.1]{kondo5points}. Since $\ca N_0 \to Q$ is a covering space (the action of $\GG(\CC)$ on $\ca N_0$ being free) we have an embedding $\pi_1(\ca N_0, m_0) \hookrightarrow \pi_1(Q, 0)$ whose composition with $\pi_1(Q,0) \to\textnormal{PU}(2,1)$ is the map $\mu:\pi_1(\ca N_0, m_0) \to P\Gamma \subset  \textnormal{PU}(2,1)$. 

To prove that the image of $\mu$ is $P\Gamma_\theta$, it suffices to give a section of the map $\ca N_0 \to Q$. Indeed, such a section induces a retraction of $\pi_1(\ca N_0, m_0) \hookrightarrow \pi_1(Q, 0)$, so that the images of these two groups in $\textnormal{PU}(2,1)$ are the same. 
\newpage
\noindent
To define such a section, observe that if $\Delta \subset \PP^1(\CC)^5$ is the union of all hyperplanes $\{x_i = x_j\} \subset \PP^1(\CC)^5$ for $i \neq j$, then 
\begin{align*}
Q = \PGL_2(\CC) \setminus P_0(\CC) &= \PGL_2(\CC) \setminus \left( \PP^1(\CC)^5  - \Delta \right) \\
&\cong \{(x_4,x_5) \in \CC^2 : x_i \neq 0,1 \textnormal{ and } x_1 \neq x_2 \}.
\end{align*}
The section $Q \to \ca N_0$ may then be defined by sending $(x_4, x_5)$ to the binary quintic 
\[
F(x,y) = x(x-y)y(x-x_4\cdot y)(x-x_5\cdot y) \in X_0(\CC),
\]
marked by the labelling of its roots $\{[0:1],[1:1], [1:0], [x_4:1], [x_5:1]\}$. 

It remains to prove that the homomorphism $\gamma: \mf S_5 \to P\Gamma/P\Gamma_\theta$ appearing on the right in (\ref{eq:commutativemonodromy}) is an isomorphism. %In fact, $\Gamma/ \{\pm 1 \} \cdot \Gamma_\theta \cong \mf S_5$ by \cite[Propositions 4.2 \& 4.3]{yoshida}, hence $P\Gamma/P\Gamma_\theta \cong \mf S_5$. Alternatively, 
We use Theorem \ref{th:calculatemonodromyshimura}, proven by Shimura in \cite{shimuratranscendental}, which says that $$( \Lambda, \mf h ) \cong \left( \OO_K^3, \textnormal{diag}\left(1,1, \frac{1 - \sqrt 5}{2}\right)\right).$$ It follows that $$P\Gamma/P\Gamma_\theta = P\Aut(W) \cong \textnormal{PO}_3(\bb F_5) \cong \mf S_5.$$ Next, consider the manifold $\ca N_s$. Remark that $\mf S_5$ embeds into $\Aut(\GG(\CC) \setminus \ca N_s)$. Moreover, there is a natural isomorpism 
\[
p\colon \GG(\CC) \setminus \ca N_s \cong P\Gamma_\theta \setminus \CC H^2, \quad \quad \tn{see \cite{DeligneMostow, kondo5points}.}
\]
(See also (\ref{eq:periodframed}).) The two compositions $$\mf S_5 \subset \Aut(\GG(\CC) \setminus \ca N_s) \cong \Aut(P\Gamma_\theta \setminus \CC H^2) \; \tn{ and } \; \mf S_5 \to P\Gamma/P\Gamma_\theta \subset \Aut(P\Gamma_\theta \setminus \CC H^2)$$ agree, because of the equivariance of $p $ with respect to $\gamma$. Thus, $\gamma$ is injective.  
%(here one uses that all non-degenerate quadratic forms on $\bb F_5^3$ are equivalent) and $\textnormal{PO}_3(\bb F_5) \cong \mf S_5$.
\end{proof}

\end{proposition}
%. %Since the action of $\PGL_2(\CC)$ on $M$ is free, the map $M \to Q$ is \'etale; by uniqueness of the Fox completion, we see that $\PGL_2(\CC) \setminus P\ca M_s \cong \widetilde Q_{\text{st}}$. 
\begin{corollary} \label{cor:surjectivemon}
The monodromy representation $\rho$ in (\ref{eq:monodromy}) is surjective. $\hfill \qed$
\end{corollary}
%\begin{proof}
%Proposition \ref{prop:commutativemonodromy} implies that $\pi_1(\mr X) \to P\Gamma$ is surjective. 
%\end{proof}
%\begin{proof}
%This follows from \cite[Theorem 7.1]{kondo5points}. 
%\end{proof}

\subsection{Framed binary quintics}

By a \textit{framing} of a point $F \in X_0(\CC)$ we mean a projective equivalence class $[f]$, where \[
f \colon  \VV_F = \rm H^1(C_F(\CC), \ZZ) \to \Lambda\] is an $\OO_K$-linear isometry: two such isometries are in the same class if and only if they differ by an element in $\mu_K$. Let $\ca F_0$ be the collection of all framings of all points $x \in X_0(\CC)$. The set $\ca F_0$ is naturally a complex manifold, by arguments similar to those in \cite{ACTsurfaces}. 
 %in such a way that $\ca F_0 \to X_0(\CC)$ is a covering space. 
Note that Corollary \ref{cor:surjectivemon} implies that $\ca F_0$ is connected, hence 
\begin{align}\label{framedcovering}
\ca F_0 \to X_0(\CC)
\end{align} is a covering, with Galois group $P\Gamma$. %Let $P\ca F_0$ be the set of projective equivalence classes $[f]$ of framings $f$ of classes $[F] \in \mr X$. %Then $P\ca F_0 \to \mr X$ is \'etale: it is the quotient by $P\Gamma$. 
%By Proposition \ref{prop:commutativemonodromy}, %$P\Gamma$ is the monodromy group, i.e. the image of 
%$\pi_1(\mr X, [F_0]) \to P\Gamma$ is surjective, hence $P\ca F_0$ is connected and $P\ca F_0 \to \mr X$ is Galois. % cover with Galois group $P\Gamma$. 
\begin{lemma} \label{lemma:isomorphiccoveringspaces}
The spaces $P\Gamma_\theta \setminus \ca F_0$ and $ \ca N_0$ are isomorphic as covering spaces of $X_0(\CC)$. In particular, there is a covering map $\ca F_0 \to \ca N_0$ with Galois group $P\Gamma_\theta$. 
%If $\widetilde{\ca N_0} \to \ca N_0$ is the covering space corresponding to the surjection $\mu: \pi_1(\ca N_0, m_0) \to P\Gamma_\theta$, then there is an isomorphism $ \ca F_0 \cong \widetilde{\ca N_0}$ which is compatible with $P\Gamma_\theta \setminus \ca F_0 \cong \ca N_0$. 
\end{lemma}
\begin{proof}
We have $P\Gamma/P\Gamma_\theta \cong \mf S_5$ as quotients of $P\Gamma$, see Proposition \ref{prop:commutativemonodromy}. %As for the second statement, $\ca F_0$ and $\widetilde{\ca N_0}$ are both covering spaces of $P\Gamma_\theta \setminus \ca F_0 \cong \ca N_0$ with Galois group $P\Gamma_\theta$, 
%Since the Galois group of $\ca F_0 \to P\Gamma_\theta \setminus \ca F_0$ is $P\Gamma_\theta$, 
%$\widetilde{\ca N_0} \to \ca N_0$ is a Galois cover with Galois group $P\Gamma_\theta$ by Proposition \ref{prop:commutativemonodromy}
%$P\Gamma_\theta \setminus P\ca F_0 \to \mr X$ is the Galois cover with Galois group $P\Gamma/P\Gamma_\theta$. Since $P\ca N_0 \to \mr X$ has Galois group $\mf S_5$, and there exists an isomorphism $\mf S_5 \cong P\Gamma /P\Gamma_\theta$ making (\ref{eq:commutativemonodromy}) commute, $P\Gamma_\theta \setminus P\ca F_0 \cong P\ca N_0$ as covering spaces of $\mr X$.
%Now if we define $\widetilde{P\ca N_0} \to P\ca N_0$ to be the covering space corresponding to the image of $\pi_1(P\ca N_0) \to P\Gamma$, 
\end{proof}
%Consequently, $\widetilde{P\ca N_0} \cong P\ca F_0$ compatible with the isomorphism $P\Gamma_\theta \setminus P\ca F_0 \cong P\ca N_0$. 

%can be pulled back to a family of abelian varieties $\widetilde J \to \ca F_0$. Via 
%attached to the family of curves (\ref{eq:coverp1}). Let $\Gamma = \Aut(\Lambda)$. 
\begin{lemma} \label{lemma:irreduciblenormal}
$\Delta \coloneqq X_s(\CC) - X_0(\CC)$ is an irreducible normal crossings divisor of $X_s(\CC)$.
\end{lemma}
\begin{proof}
%Let $F \in \Delta =  X_s(\CC) - X_0(\CC)$ such that $F$ has two nodes. We assume that $F(x) = x^2(x-1)^2(x-2)$. Locally around $F$, the variety $X_s$ is isomorphic to 
%$$
%( = x(x-1)(x-2)(x-t_1)(x-t_2
%Let $\ca P_s \subset (\PP^1_\CC)^5$ be the open subvariety parametrizing five-tuples $(x_1, \dotsc, x_5)$ such that no three $x_i$ coincide. Then $\mf S_5$ acts on $\ca P_s$ and, writing $V = \rm H^0\PP^1_\CC, \OO_{\PP^1_\CC}(5))$, the quotient $\ca P_s/\mf S_5$ is canonically isomorphic to $\PP_s:=\PP(V)_s^\ast$, the space of stable homogeneous degree $5$ polynomials in two variables modulo $\bb G_m$. The quotient map $\ca P_s \to \PP_s$ is finite and flat, and \'etale over the locus of distinct five-tuples $\ca P_0 \subset \ca P_s$. On the other hand, $\ca P_s - \ca P_0$ is clearly a normal crossings divisor: locally on $\ca P_s$, it is the complement of $\cup_{k \geq 3}\{x_{i_1} = \dotsc = x_{i_k}\} \subset D \subset  \AA^5_\CC$ in the locus $D = \cup_{i \neq j} \{ x_i = x_j\} \subset \AA^5_\CC$. 
%Now a finite morphism whose target is smooth is flat if and only if the source of the morphism is Cohen-Macaulay. This implies that $(\PP^1)^5 \to (\PP^1)^5/\mf S_5$ is flat. % hence smooth since the fibers are smooth. 
The proof is similar to the proof of Proposition 6.7 in \cite{beauvillecubicsurfaces}. 
\end{proof}

\begin{lemma} \label{lemma:monodromy}
The local monodromy transformations of $\ca F_0 \to X_0(\CC)$ around every $x \in \Delta$ are of finite order. More precisely, if $x \in \Delta$ lies on the intersection of $k$ local components of $\Delta$, then the local monodromy group around $x$ is isomorphic to $(\ZZ/10)^k$. 
\end{lemma}

\begin{proof}
See \cite[Proposition 9.2]{DeligneMostow} or \cite[Proposition 6.1]{carlsontoledomonodromy} for the generic case, i.e. when a quintic $Z = \{F = 0\} \subset \PP^1_\CC$ aquires only one node. In this case, the local equation of the singularity is $x^2 = 0$, hence the curve $C_F$ acquires a singularity of the form $y^5 + x^2 = 0$. If the quintic acquires two nodes, then $C_F$ acquires two such singularities; the vanishing cohomology splits as an orthogonal direct sum, hence the local monodromy transformations commute. 
\end{proof}

\noindent
In the following corollary, we let $D = \set{z \in \CC \mid \va{z} < 1}$ denote the open unit disc, and $D^\ast = D - \{0\}$ the punctured open unit disc. 

\begin{corollary} \label{cor:framedquintics}
There is an essentially unique branched cover $\pi: \ca F_s \to X_s(\CC)$, with $\ca F_s$ a complex manifold, 
such that for any $x \in \Delta$, any open $x \in U \subset X_s(\CC)$ with $U \cong D^k \times D^{6-k}$ and $U \cap X_0(\CC) \cong (D^\ast)^k \times D^{6-k}$, % then any component of $\pi^{-1}\left(U \cap B \right)$ is isomorphic to $(D^\ast)^k \times D^{5-k}$, mapping to $U$ by $(z_1, \dotsc, z_n) \mapsto (z_1^{r_1}, \dotsc, z_k^{r_k}; z_{k+1}, \dotsc, z_n)$. 
and any connected component $U'$ of $\pi^{-1}(U) \subset \ca F_s$, there is an isomorphism $U' \cong D^k \times D^{6-k}$ such that the composition 
\[
D^k \times D^{6-k} \cong U' \to U \cong D^6 \; \text{ is the map } \; (z_1, \dotsc, z_6) \mapsto (z_1^{r_1}, \dotsc, z_k^{r_k}, z_{k+1}, \dotsc, z_6). 
\]
% to $U$. Finally, $\ca F_s$ is unique up to isomorphism. 
\end{corollary}
\begin{proof}
See \cite[Lemma 7.2]{beauvillecubicsurfaces}. See also \cite{fox} and \cite[Section 8]{DeligneMostow}.
\end{proof}
\noindent
The group $\GG(\CC) = \GL_2(\CC)/D$ (see (\ref{def:GC})) acts on $\ca F_0$ over its action on $X_0$. Explicitly, if $g \in \GG(\CC)$ and if $([\phi],  \phi: \VV_{F} \cong \Lambda)$ is a framing of $F \in X_0(\CC)$, then 
\[
\left([\phi \circ g^\ast], \phi \circ g^\ast: \VV_{g\cdot F} \to \Lambda\right)
\] is a framing of $g\cdot F \in X_0(\CC)$. This is a left action. % $(h \cdot g)\cdot [\phi] = [\phi \circ (h \cdot g)^\ast] = [\phi \circ g^\ast \circ h^\ast] = h \cdot \left(g \cdot [\phi]\right)$. 
The group $P\Gamma$ also acts on $\ca F_0$ from the left, and the actions of $P\Gamma$ and $\GG(\CC)$ on $\ca F_0$ commute. By functoriality of the Fox completion, the action of $\GG(\CC)$ on $\ca F_0$ extends to an action of $\GG(\CC)$ on $\ca F_s$.

\begin{lemma} \label{cor:freeaction}
The group $\GG(\CC) = \GL_2(\CC)/D$ acts freely on $\ca F_s$. 
\end{lemma}
\begin{proof}
%Let $\PP \ca F_0$ be the space of $\bb G_m$-equivalence classes of smooth quintics equipped with a framing $[f]$, let $\PP \ca F_s$ be the Fox completion of $\PP \ca F_0 \to PX_0(\CC) \to PX_s(\CC)$; let $\PP \ca N_0$ be the space of $\bb G_m$-equivalence classes of smooth marked quintics and let $\PP \ca M_s \to PX_s(\CC)$ be the corresponding Fox completion. 
%Similar to the construction of $\ca F_s$ in Lemma \ref{cor:framedquintics}, the composition $\ca N_0 \to X_0(\CC) \to X_s(\CC)$ admits an essentially unique \textit{completion} $\ca M_s \to X_s(\CC)$, see \cite{fox}. Again, $\ca N_s$ is a manifold, and 
The functoriality of the Fox completion gives an action of $\GG(\CC)$ on $\ca N_s$ such that, by Lemma \ref{lemma:isomorphiccoveringspaces}, there is a $\GG(\CC)$-equivariant commutative diagram 
\[
\xymatrixrowsep{0.8pc}
\xymatrix{
P\Gamma_\theta \setminus \ca F_s \ar[rr]^{\sim}\ar[dr] && \ca N_s \ar[dl] \\
& X_s(\CC).
}
\] 
%$P\Gamma_\theta \setminus \ca F_s$ and $\ca N_s$ are $\GG(\CC)$-equivariantly isomorphic as branched covering spaces of $X_s(\CC)$. 
In particular, it suffices to show that $\GG(\CC)$ acts freely on $\ca N_s$. Note that $\ca N_0$ admits a natural $\GG_m$-covering map $\ca N_0 \to P_0(\CC)$ where $P_0(\CC) \subset \PP^1(\CC)^5$ is the space of distinct ordered five-tuples in $\PP^1(\CC)$ introduced in Section \ref{realbinaryintroduction}. Consequently, there is a $\GG_m$-quotient map $\ca N_s \to P_s(\CC)$, where $P_s(\CC)$ is the space of stable ordered five-tuples, and this map is equivariant for the homomorphism $\GL_2(\CC) \to \PGL_2(\CC)$. 
%Now note that $\GG(\CC) \setminus \ca N_0 \cong \PGL_2(\CC) \setminus P_0$, the space of distinct ordered $5$-tuples in $\PP^1(\CC)$ moduli projective equivalence. 
%it suffices to prove that $G$ acts freely on $\ca N_s$. Now $\PP \ca N_0$ is canonically isomorphic to $P_0 \subset (\PP^1(\CC))^5$, the space of distinct five-tuples $(x_1, \dotsc, x_5) \in (\PP^1(\CC))^5$; consequently, $\PP \ca N_s$ canonically isomorphic to $P_s \subset (\PP^1(\CC))^5$, the space of stable five-tuples. 

Let $g \in \GL_2(\CC)$ and $x \in \ca N_s$ such that $gx = x$. It is clear that $\PGL_2(\CC)$ acts freely on $P_s(\CC)$. Therefore, $g = \lambda \in \CC^\ast$. Let $F \in X_s(\CC)$ be the image of $x \in \ca N_s$; then $$gF(x,y) = F(g^{-1}(x,y)) = F(\lambda^{-1}x, \lambda^{-1}y) = \lambda^{-5}F(x,y).$$ The equality $gF = F$ implies that $\lambda^{5} = 1 \in \CC$, from which we conclude that $\lambda \in \langle \zeta \rangle$. Therefore, $[g] = [\id] \in \GG(\CC) = \GL_2(\CC)/D$. 
%$(x_1, \dotsc, x_5) \in (\PP^1(\CC))^5$ such that no three $x_i$ coincide. 
\end{proof}

\subsection{Complex uniformization}

Consider the hermitian space $V_1 = \Lambda \otimes_{\OO_K, \tau_1} \CC$; define $\CC H^2$ to be the space of negative lines in $V_1$. Using Proposition \ref{prop:canonicalbijection} we see that the abelian scheme $J \to X_0$ induces a $\GG(\CC)$-equivariant morphism 
$
\ca P: \ca F_0 \to \CC H^2$. Explicitly, if $(F, [f]) \in \ca F_0$ is the framing $[f: \rm H^1(C_F(\CC), \ZZ) \xrightarrow{\sim}\Lambda]$ of the binary quintic $F \in X_0(\CC)$, and $A_F$ is the Jacobian of the curve $C_F$, then 
\begin{equation} \label{eq:periodframed}
\ca P: \ca F_0 \to \CC H^2, \quad
\ca P \left( F, [f] \right) = f \left( H^{0,-1}(A_F)_{\tau_1} \right) = f \left(H^{1,0}(C_F)_{\zeta^3} \right) \in \CC H^2. 
\end{equation}
%(in fact, $\ca F_0 \to X_0(\CC)$ trivializes $\bb V$ so that $\ca P$ lifts to a morphism $\ca P: \ca F_0 \to \CC H^2$). 
The map $\ca P$ is holomorphic, % The (\ref{eq:periodframed}) 
and descends to a morphism of complex analytic spaces
\begin{equation} \label{eq:period}
\ca M_0(\CC) =\GG(\CC) \setminus X_0(\CC) \to P\Gamma \setminus \CC H^2. 
\end{equation}
\noindent
By Riemann extension, (\ref{eq:periodframed}) extends to a $\GG(\CC)$-equivariant holomorphic map
\begin{align}\label{eq:stableperiodmapframed}
\overline{\ca P}: \ca F_s \to \CC H^2. 
\end{align}

\begin{theorem}[Deligne--Mostow] \label{th:delignemostow}
The period map (\ref{eq:stableperiodmapframed}) induces an isomorphism of complex manifolds
\begin{align} \label{eq:isomstablefivepoints-zero}
\ca M_s^f(\CC) \coloneqq \GG(\CC) \setminus \ca F_s \cong \CC H^2. 
\end{align}
Taking $P\Gamma$-quotients gives an isomorphism of complex analytic spaces
\begin{equation} \label{eq:isomstablefivepoints}
\ca M_s(\CC) = \GG(\CC) \setminus X_s(\CC) \cong P\Gamma \setminus \CC H^2. 
\end{equation} 
\end{theorem}
\begin{proof}
%Let us use the notation of \cite{DeligneMostow}. 
%Recall that $P_0  \subset \PP^1(\CC)^5$ is the set of $(x_1, \dotsc, x_5) \in \PP^1(\CC)^5$ such that all $x_i$ are distinct, and that $Q = \GG(\CC) \setminus \ca N_0 = \PGL_2(\CC) \setminus P_0$ with base point $0 \in Q$. Then $\ca F_0$ is a covering space of $\ca N_0$, with Galois group $P\Gamma_\theta$ (Lemma \ref{lemma:isomorphiccoveringspaces}). 
In \cite{DeligneMostow}, Deligne and Mostow define $\widetilde Q \to Q$ to be the covering space corresponding to the monodromy representation $\pi_1(Q,0) \to \textnormal{PU}(2,1)$; since the image of this homomomorphism is $P\Gamma_\theta$ (see the proof of Proposition \ref{prop:commutativemonodromy}), it follows that $\GG(\CC) \setminus \ca F_0 \cong \widetilde Q$ as covering spaces of $Q$. Consequently, if $\widetilde Q_{\textnormal{st}}$ is the Fox completion of the spread $$\widetilde Q \to Q \to  Q_{\textnormal{st}} := \GG(\CC) \setminus \ca N_s = \PGL_2(\CC) \setminus P_s(\CC),$$ then there is an isomorphism $\GG(\CC) \setminus \ca F_s \cong \widetilde  Q_{\textnormal{st}}$ of branched covering spaces of $ Q_{\textnormal{st}}$. We obtain commutative diagrams, where the lower right morphism uses (\ref{eq:commutativemonodromy}):
$$
\xymatrixcolsep{5pc}
\xymatrix{
\GG(\CC) \setminus \ca F_s \ar[r]^{\sim \;\;\;\;} \ar[d] & \widetilde{Q}_{\textnormal{st}}\ar[d] \ar[r] & \CC H^2  \ar[d] \\
\GG(\CC) \setminus \ca N_s \ar[r]^{\sim \;\;\;\;}  \ar[d] & Q_{\textnormal{st}}  \ar[r] \ar[d]& P\Gamma_{\theta} \setminus \CC H^2\ar[d]\\
\GG(\CC) \setminus X_s(\CC) \ar[r]^{\sim \;\;\;\;} &  Q_{\textnormal{st}}/\mf S_5 \ar[r] & P\Gamma \setminus \CC H^2.
}
$$
The map $\widetilde{Q}_{\textnormal{st}} \to \CC H^2$ is an isomorphism by \cite[(3.11)]{DeligneMostow}. Therefore, we are done if the composition $\GG(\CC) \setminus \ca F_0 \to \widetilde Q \to \CC H^2$ agrees with the period map $\ca P$ of equation (\ref{eq:periodframed}). This follows from \cite[(2.23) and (12.9)]{DeligneMostow}. 
\end{proof}

\begin{proposition} \label{prop:stableperiodshyperplane}
The isomorphism (\ref{eq:isomstablefivepoints}) induces an isomorphism of complex analytic spaces
\begin{equation}
\ca M_0(\CC) = \GG(\CC) \setminus X_0(\CC) \cong P\Gamma \setminus \left(\CC H^2 - \mr H \right). 
\end{equation}
\end{proposition}
\begin{proof}
%Consider the map $\overline{\ca P}: \ca F_s \to \CC H^2$. 
%We claim that $\overline{\ca P}( \ca F_s - \ca F_0) = \mr H$. 
%It suffices to show that the image of $\ca F_0$ under $\tilde{\ca P}$ is contained in $\CC H^2 - \mr H$. Indeed, if this holds true, then $\ca P^{-1} \left(\mr H/P\Gamma \right)$ is contained in $\partial \overline{\mr M}_{\CC}$; since both are divisors, we have $\ca P^{-1} \left(\mr H/P\Gamma \right) = \partial \overline{\mr M}_{\CC}$ and the proposition follows. But the image of $\ca F_0$ under $\tilde{\ca P}$ is contained in $\CC H^2 - \mr H$ 
We have $\overline{\ca P}(\ca F_0) \subset \CC H^2 - \mr H$ by Proposition \ref{prop:HcorrespondsNonSimple}, because the Jacobian of a smooth curve cannot contain an abelian subvariety whose induced polarization is principal. Therefore, we have $\overline{\ca P}^{-1}(\mr H) \subset \ca F_s - \ca F_0$. Since $\ca F_s$ is irreducible (it is smooth by Corollary \ref{cor:framedquintics} and connected by Corollary \ref{cor:surjectivemon}), the analytic space $\overline{\ca P}^{-1}(\mr H)$ is a divisor. Since $\ca F_s - \ca F_0$ is also a divisor by Corollary \ref{cor:framedquintics}, we have $$\overline{\ca P}^{-1}(\mr H) = \ca F_s - \ca F_0$$ and we are done. 

Alternatively, %there is a more geometric way to prove the proposition which goes as follows. 
%$\GG(\CC) \setminus X_s(\CC)$ is canonically isomorphic to $\overline{M}_{0,5}(\CC) / \mf S_5$, where $\overline{M}_{0,5}$ is the Deligne-Mumford compactification of the moduli space of five-pointed genus zero curves \cite{Knudsen1983}. 
let $H_{0,5}$ be the moduli space of degree $5$ covers of $\PP^1$ ramified along five distinct marked points \cite[\S 2.G]{Harris1998ModuliOC}. The period map $$H_{0,5}(\CC) \to P\Gamma \setminus \CC H^2,$$ that sends the moduli point of a curve $C \to \PP^1$ to the moduli point of the $\ZZ[\zeta]$-linear Jacobian $J(C)$, extends to the stable compactification $\overline{H}_{0,5}(\CC) \supset H_{0,5}(\CC)$ because the curves in the limit are of compact type. Since the divisor $\mr H \subset \CC H^2$ parametrizes abelian varieties that are products of lower dimensional ones by Proposition \ref{prop:HcorrespondsNonSimple}, the image of the boundary is exactly $P\Gamma \setminus \mr H$.
%$\overline{H}_{0,5}(\CC) - H_{0,5}(\CC)$ in $P\Gamma \setminus \CC H^2$ is exactly 
% and $\mr H \subset \overline{\ca P}( \ca F_s - \ca F_0) \subset \CC H^2$. 
\end{proof}

\section{Moduli of real binary quintics} \label{modrealbinquin}

Having understood the period map for complex binary quintics, we turn to the period map of real binary quintics in this Section \ref{modrealbinquin}. 

\subsection{The period map for stable real binary quintics}

Define $\kappa$ as the anti-holomorphic involution
\[
\kappa \colon X_0(\CC) \to X_0(\CC), \quad  F(x,y) = \sum_{i+j=  5}a_{ij}x^iy^j \mapsto \overline{F(x,y)}= \sum_{i+j =5}\overline{a_{ij}}x^iy^j. 
\]
Let $\mr A$ be the set of anti-unitary involutions $\alpha : \Lambda \to \Lambda$, and $P\mr A = \mu_K \setminus \mr A$, see Section \ref{unitaryshimura}. For each $\alpha \in P\mr A$, there is a natural anti-holomorphic involution $\alpha \colon \ca F_0 \to \ca F_0$ such that the following diagram commutes:
\[
\xymatrix{
\ca F_0 \ar[d]\ar[r]^\alpha & \ca F_0 \ar[d] \\
X_0(\CC) \ar[r]^\kappa & X_0(\CC). 
}
\]
It is defined as follows. Consider a framed binary quintic $(F, [f]) \in \ca F_0$, where $f: \VV_F \to \Lambda$ is an $\OO_K$-linear isometry. Let $C_F \to \PP^1_\CC$ be the quintic cover defined by a smooth binary quintic $F \in X_0(\CC)$. Complex conjugation $\tn{conj} \colon \PP^2(\CC) \to \PP^2(\CC)$ induces an anti-holomorphic map 
\[
\sigma_F: C_F(\CC) \to C_{\kappa(F)}(\CC). 
\] %which in turn induces an anti-holomorphic map $\sigma: J(C_F) \to J(C_{\kappa \cdot F})$. 
Consider the pull-back $\sigma_F^\ast: \bb V_{\kappa (F)} \to \bb V_F$ of $\sigma$. The composition 
\[
  \VV_{\kappa(F)} \xrightarrow{\sigma_F^\ast} \VV_F \xrightarrow{f} \Lambda \xrightarrow{\alpha} \Lambda
\] induces a framing of $\kappa(F) \in X_0(\CC)$, and we define
\[
\alpha(F, [f]) \coloneqq \left(\kappa(F), [\alpha \circ f \circ \sigma_F^\ast] \right) \in \ca F_0.
\] Although we have chosen a representative $\alpha \in \mr A$ of the class $\alpha \in P\mr A$, the element $\alpha(F, [f]) \in \ca F_0$ does not depend on this choice.  
\\
\\
Consider the covering map $\ca F_0 \to X_0(\CC)$ introduced in (\ref{modrealbinquin}), and define
\begin{align}\label{realpointsofFzero}
\ca F_0(\RR) = \bigsqcup_{\alpha \in P\mr A} \ca F_0^\alpha \subset \ca F_0
\end{align}
as the preimage of $X_0(\RR)$ in the space $\ca F_0$. To see why the union on the left hand side of (\ref{realpointsofFzero}) is disjoint, observe that 
\[
    \ca F_0^\alpha = \left\{ (F, [f]) \in \ca F_0 : \kappa(F) = F \textnormal{ and } [f \circ \sigma^\ast_{F} \circ f^{-1}] = \alpha \right\}.
    \]
    Thus, for $\alpha, \beta \in P\mr A$ and $(F, [f]) \in \ca F_0^\alpha \cap \ca F_0^\beta$, we have $\alpha = [f \circ \sigma \circ f^{-1}] = \beta$. 
    % for if $(F, [f])$ is in the intersection, then $[\alpha] = [f \circ F_{\infty, \kappa \cdot F} \circ f^{-1}] = [\beta]$. 

\begin{lemma}
The anti-holomorphic involution $\alpha \colon \ca F_0 \to \ca F_0$ defined by $\alpha \in P \mr A$ makes the period map $\ca P \colon \ca F_0 \to \CC H^2$ equivariant for the $\GG(\CC)$-actions on both sides.  
\end{lemma}
\begin{proof}
%Let $\left([\phi], f: \bb V_t \to \Lambda$ be a framing for $t \in X_0(\CC)$ and let $[\alpha] \in P\mr A$. Then 
%let $C_F \to \PP^1_\CC$ be the quintic cover defined by a smooth binary quintic $F \in X_0(\CC)$. Complex conjugation $\sigma: \{F = 0\} = Z \to \kappa Z = \{\kappa(F) = 0 \}$ extends to an anti-holomorphic map $\sigma: C_F \to C_{\kappa F}$ which in turn extends to an anti-holomorphic map $\sigma: J(C_F) \to J(C_{\kappa_F})$. Its pullback $\sigma^\ast: \bb V_{\kappa F} \to \bb V_F$ is linear; moreover, 
Indeed, if $\tn{conj} \colon \CC \to \CC$ is complex conjugation, then for any $F \in X_0(\CC)$, the induced map
$$\sigma^\ast_F \otimes \tn{conj} \colon \bb V_{\kappa (F)}\otimes_\ZZ \CC \to \bb V_F \otimes_\ZZ \CC$$ is anti-linear, preserves the Hodge decompositions \cite[Chapter I, Lemma (2.4)]{silholsurfaces} as well as the eigenspace decompositions. 
%$\ca V_{t, \tau_i}, \ca V_{t, \tau_i\sigma}$ for each $i$. 
\end{proof}
%%Therefore $\alpha$ maps the inclusion
%Given a framing $[f]$ of $t$, the inclusion
%$f:   \ca V_{t, \tau}^{0,-1} \to V_\sigma = \Lambda \otimes_{\OO_K, \sigma} \CC$ induces the inclusion
%\begin{equation}
 % \ca V_{\kappa t, \tau}^{0,-1} \xrightarrow{F_{dR}}  
 % \ca V_{\kappa t, \tau}^{0,-1}
 %M \xrightarrow{f}
 % V_\sigma 
 % \xrightarrow{\alpha}
%  V_\sigma
%\end{equation}
%which is the inclusion $\ca V_{\kappa t, \tau}^{0,-1} \to V_\sigma$ induced by the framing $[\alpha] \cdot [f]$ of $\kappa t$. Since $\alpha f \left(\ca V_{t, \tau}^{0,-1} \right)
%= \alpha f F_{dR} \left(\ca V_{\kappa t, \tau}^{0,-1}\right)$ we are done.
\noindent
We obtain a \textit{real period map} 
%\begin{equation} \label{eq:realperiodmap}
\begin{align}\label{therealperiodmap}
\xymatrix{
\ca P_\RR \colon \ca F_0(\RR)  \ar@{=}[r] &\coprod_{\alpha \in P\mr A} \ca F_0^\alpha  \ar[r] &\coprod_{\alpha \in P\mr A} \RR H^2_\alpha  \ar@{=}[r]& 
\widetilde Y.
}
\end{align}
%The real period map $\ca P_\RR: \ca F_0(\RR) \to \widetilde Y$ is 
Define $\GG(\RR) = \GL_2(\RR)$. The map (\ref{therealperiodmap}) is constant on $\GG(\RR)$-orbits, since the same is true for the complex period map $\ca P \colon \ca F_0 \to \CC H^2$. 

By abuse of notation, we for $\alpha \in P\mr A$, we write $\RRH^2_\alpha - \mr H = \RRH^2_\alpha - \left(\mr H \cap \RRH^2_\alpha\right)$.\begin{proposition} \label{prop:realsmoothperiods} %Assume that $\ca P( \ca F_0 ) = \CC H^2 - \mr H$ and $\overline{\ca P}: \GG(\CC) \setminus \ca F_s \to \CC H^2$ is a $P\Gamma$-invariant isomorphism of complex manifolds. 
%\item Let $\mu_K \to \GL_n(\CC)$ be the map $\zeta \mapsto \zeta \cdot I$. Then $\GG(\CC)$ is the quotient of $\GL_{n}(\CC)$ by $\mu_K$. 
%subgroup $\{\zeta^i \cdot I\} \subset \GL_n(\CC)$ of scalar matrices with diagonal entries an element of $\zeta \in \mu_K$. 
The real period map (\ref{therealperiodmap}) descends to a $P\Gamma$-equivariant diffeomorphism 
\begin{align} \label{firstperiod}
\ca M_0(\RR)^f \coloneqq \GG(\RR) \setminus \ca F_0(\RR) \cong \coprod_{\alpha \in P\mr A} \RR H^2_\alpha - \mr H.
 \end{align}
By $P\Gamma$-equivariance, the map (\ref{firstperiod}) induces an isomorphism of real-analytic orbifolds
\begin{equation} \label{smoothrealperiodiso}
\ca P_{\RR} \colon \ca M_0(\RR) = \GG(\RR) \setminus X_0(\RR) \cong  \coprod_{\alpha \in C \mr A}P\Gamma_\alpha \setminus \left( \RR H^2_\alpha - \mr H
    \right). 
\end{equation}
\end{proposition}

\begin{proof}
This follows from \cite[\textit{proof of Theorem 3.3}]{realACTsurfaces}. It is crucial that the actions of $G$ and $P\Gamma$ on $\ca F_0$ commute and are free, which is the case, see Corollary \ref{cor:freeaction}. 
\end{proof}

\subsection{The period map for smooth real binary quintics}

Our next goal will be to prove the real analogue of the isomorphisms (\ref{eq:isomstablefivepoints-zero}) and (\ref{eq:isomstablefivepoints-zero}) 
%$\overline{\mr M}_\CC = \GG(\CC) \setminus X_s(\CC) \cong P\Gamma \setminus \CC H^2$ 
in Theorem \ref{th:delignemostow}. We need a lemma, a definition, and then two more lemmas. %Observe that $\widetilde \Delta := \ca F_s - \ca F_0$ is a divisor with normal crossings, see Corollary \ref{cor:framedquintics}. %Let $\Delta'$ be the image of $\widetilde \Delta$ in $\GG(\CC) \setminus \ca F_s$. 

Consider the CM-type $\Psi = \set{\tau_1, \tau_2}$ defined in (\ref{CMTYPE}), the hermitian $\OO_K$-lattice $(\Lambda, \mf h)$ defined in (\ref{eq:hermitianformonbinaryquinticlattice}), and the sets (c.f. Definition \ref{fullset}): $$\ca H = \set{H_r \subset \CCH^2 \mid r \in \mr R}, \quad \tn{ and } \quad \mr H = \cup_{H\in \ca H}H \subset \CCH^2.$$ Here, $\mr R\subset \Lambda$ is the set of short roots (see Section \ref{set-up}). 

\begin{lemma}
The hyperplane arrangement $\mr H \subset \CCH^2$ satisfies Condition \ref{orthogonal}, that is: any two distinct $H_1, H_2\in\ca H$ either meet orthogonally, or not at all. 
\end{lemma}
\begin{proof}
Condition \ref{crucialcondition}.\ref{crucialone} holds because $K$ does not contain proper CM-subfields.  
By Lemma \ref{lemma:discr}, we have that Condition \ref{crucialcondition}.\ref{crucialtwo} is satisfied. 
By equation (\ref{quintic-cases}), Condition \ref{crucialcondition}.\ref{crucialthree} holds. 
By Theorem \ref{th:conditionsimplyhypothesis}, we obtain the desired result. 
\end{proof}

\begin{definition} 
\begin{enumerate}
\item
For $k = 1, 2$, define $\Delta_k \subset \Delta = X_s(\CC) - X_0(\CC)$ to be the locus of stable binary quintics with exactly $k$ nodes. Define $\widetilde \Delta = \ca F_s - \ca F_0$, and let  $\widetilde \Delta_k \subset \widetilde \Delta$ be the inverse image of $\Delta_k$ in $\widetilde \Delta$ under the map $\widetilde \Delta \to \Delta$. 
\item For $k = 1,2$, define $\mr H_k \subset \mr H$ as the set $\mr H_k = \set{x \in \CCH^2 \mid \va{\ca H(x)} = k}$. Thus, this is the locus of points in $\mr H$ where exactly $k$ hyperplanes meet. 
\end{enumerate}
\end{definition}
\begin{lemma} %Assume the hypotheses of Lemma \ref{lemma:realsmoothperiods}. 
\label{lemma:stabilizergroups}
\begin{enumerate}
\item 
The period map $\overline{\ca P}$ of (\ref{eq:stableperiodmapframed}) satisfies $\overline{\ca P}(\widetilde \Delta_k) \subset \mr H_k$. 
 %the points $f \in \widetilde \Delta \coloneqq \ca F_s - \ca F_0$ lying above binary quintics with $k$ nodes %for which the local equation for $\Delta$ around $f \in \ca F_s$ is $t_1 \cdots t_k = 0$ 
%to the locus of $\CC H^2$ where exactly $k$ of the hyperplanes of $\mr H$ meet. 
\item 
If $f \in \widetilde \Delta_k$, $x = \overline{\ca P(f)} \in \CCH^2$, and $\ca H(x) = \set{H_{r_1}, \dotsc, H_{r_k}}$ for $r_i \in \mr R$, %, $\textbf r = (r_1, \dotsc, r_k)$ a vector of short roots such that $\overline{\ca P}(f) = x \in \cap_i H_{r_i}$, 
%Moreover, 
then $\overline{\ca P}$ induces a group isomorphism $P\Gamma_f \cong G(x)$. \hfill \qed
\end{enumerate}
\end{lemma}
 \noindent
The naturality of the Fox completion implies that for $\alpha \in P\mr A$, the anti-holomorphic involution $\alpha: \ca F_0 \to \ca F_0$ extends to an anti-holomorphic involution $\alpha: \ca F_s \to \ca F_s$. 
%Define $\ca F_s(\RR) = \cup_{\alpha \in P\mr A} F_s^\alpha$. 

\begin{lemma} \label{lemma:alphaperiod}
%Assume the hypotheses of Lemma \ref{lemma:realsmoothperiods}. 
For every $\alpha \in P\mr A$, the restriction of $\overline{\ca P}: \ca F_s \to \CC H^2$ to $\ca F_s^\alpha$ defines a diffeomorphism $\GG(\RR) \setminus \ca F_s^\alpha \cong \RR H^2_\alpha$. 
\end{lemma}
\begin{proof}
See \cite[Lemma 11.3]{realACTsurfaces}. It is essential that $G$ acts freely on $\ca F_s$, which holds by Corollary \ref{cor:freeaction}. 
\end{proof}
\noindent
We arrive at the main theorem of Section \ref{modrealbinquin}. Define 
\[
\ca F_s(\RR) = \bigcup_{\alpha \in P\mr A} \ca F_s^\alpha = \pi^{-1}\left(X_s(\RR)\right).
\]
This is not a manifold because of the ramification of $\pi: \ca F_s \to X_s(\CC)$, but a union of embedded submanifolds. 

\begin{theorem} \label{th:realstableperiod}
There is a smooth map 
\begin{align}\label{therealstableperiodmap}
\overline{\ca P}_{\RR}: \coprod_{\alpha \in P\mr A} \ca F_s^\alpha \to  \coprod_{\alpha \in P\mr A}\RR H^2_\alpha =  \widetilde Y
\end{align} that extends the real period map (\ref{therealperiodmap}). The map (\ref{therealstableperiodmap}) induces the following commutative diagram of topological spaces, in which $\mr P_\RR$ and $\mf P_\RR$ are homeomorphisms:
\begin{equation*}
\xymatrixcolsep{5pc}
\xymatrix{
&\coprod_{\alpha \in P\mr A} \ca F_s^\alpha \ar[r]^{\overline{\ca P}_{\RR}}\ar[d] & \widetilde Y = \coprod_{\alpha \in P\mr A} \RR H^2_\alpha \ar[d]\\
&\ca F_s(\RR) \ar[r]^{\overline{\ca P}_\RR}\ar[d] & Y\ar@{=}[d] \\
\ca M_s(\RR)^f \ar@{=}[r]\ar[d] &\GG(\RR) \setminus \ca F_s(\RR) \ar[r]^{\mr P_\RR}_\sim\ar[d] & Y\ar[d] \\
\ca M_s(\RR) \ar@{=}[r] &\GG(\RR) \setminus X_s(\RR) \ar[r]^{{\mf P}_\RR}_\sim & P\Gamma \setminus Y. 
}
\end{equation*}
%
%and $\mr P_\RR: \GG(\RR) \setminus \ca F_s(\RR) \to Y$ and ${\mr P}_\RR: \GG(\RR) \setminus \ca F_s(\RR) / P\Gamma \to P\Gamma \setminus Y$ are homeomorphisms. 
%is a $P\Gamma$-equivariant isomorphism, hence
\end{theorem}

\begin{proof}
The existence of $\overline{\ca P}_{\RR}$ follows from the compatibility with the involutions $\alpha \in P\mr A$. We first show that the composition 
$$ \coprod_{\alpha \in P\mr A} \ca F_s^\alpha \xrightarrow{\overline{\ca P}_{\RR}} \widetilde Y \xrightarrow{p} Y$$ factors through $\ca F_s(\RR)$. Now $f_\alpha$ and $g_\beta\in \coprod_{\alpha \in P\mr A} \ca F_s^\alpha$ have the same image in $\ca F_s(\RR)$ if and only if $f = g \in \ca F_s^\alpha \cap \ca F_s^\beta$, in which case 
$$
x\coloneqq \overline{\ca P}(f) = \overline{\ca P}(g) \in \RR H^2_\alpha \cap \RRH^2_\beta,$$ so we need to show is that $x_\alpha \sim x_\beta \in \widetilde Y$. For this, note that $\alpha\beta \in P\Gamma_f \cong (\ZZ/10)^k$, %where the latter holds by by Lemma \ref{lemma:localdescription}.2, 
and $\overline{\ca P}$ induces an isomorphism 
$
P\Gamma_f \cong G(x)
$
by Lemma \ref{lemma:stabilizergroups}. Hence $\alpha \beta \in G(x)$ so that indeed, $x_\alpha \sim x_\beta$. 

Let us prove the $\GG(\RR)$-equivariance of $\overline{\ca P}_{\RR}$. Suppose that 
\[
f \in \ca F_s^\alpha, g \in \ca F_s^\beta \quad \mid \quad a \cdot f = g \in \ca F_s(\RR) \quad \tn{ for some } \quad a \in \GG(\RR).
\]
Then $x\coloneqq \overline{\ca P}(f) = \overline{\ca P}(g) \in \CC H^2$, so we need to show that $\alpha \beta \in G(x)$. The actions of $\GG(\CC)$ and $P\Gamma$ on $\CC H^2$ commute, and the same holds for the actions of $\GG(\RR)$ and $P\Gamma'$ on $\ca F_s^\RR$. 
% where $P\Gamma' = \Gamma'/\mu_K$, $\Gamma'$ being the group of all unitary and anti-unitary automorphisms of $\Lambda$. 
It follows that 
\[
\alpha(g) = \alpha (a \cdot f) = a \cdot \alpha(f) = a \cdot f = g,
\]
hence $g \in \ca F_s^\alpha \cap \ca F_s^\beta$. This implies in turn that $\alpha \beta (g) = g$, hence $\alpha \beta \in P\Gamma_g \cong G(x)$, so that indeed, $x_\alpha \sim x_\beta$. 

To prove that $\mr P_\RR$ is injective, let again 
$
f_\alpha, g_\beta\in \coprod_{\alpha \in P\mr A} \ca F_s^\alpha
$
and suppose that they have the same image in $Y$. This implies that 
\[
x\coloneqq \overline{\ca P}(f) = \overline{\ca P}(g) \in \RR H^2_\alpha \cap \RRH^2_\beta,
\]
and that $\beta = \phi \circ \alpha$ for some $\phi \in G(x)$. %We claim that there exists $a \in \GG(\RR)$ such that $a \cdot g \in \ca F_s^\alpha$. 
We have $\phi \in G(x) \cong P\Gamma_f$ (by Lemma \ref{lemma:stabilizergroups}) hence
\begin{equation}
    \beta(f) = \phi \left(\alpha (f)\right) = \phi(f) = f. 
\end{equation}
Therefore $f,g \in \ca F_s^\beta$; since $\overline{\ca P}(f) = \overline{\ca P}(g)$, it follows from Lemma \ref{lemma:alphaperiod} that there exists $a \in \GG(\RR)$ such that $a \cdot f = g$. This proves injectivity of $\mr P_\RR$, as desired.   
% \in \ca F_s^\alpha \cap \ca F_s^\beta$ 

The surjectivity of $\mr P_\RR: \GG(\RR) \setminus \ca F_s(\RR) \to Y$ is straightforward, using the surjectivity of $\overline{\ca P}_{\RR}$, which follows from Lemma \ref{lemma:alphaperiod}. 

Finally, we claim that $\mr P_\RR$ is open. Let $U \subset \GG(\RR) \setminus \ca F_s^\RR$, and write $U = \mr P_\RR^{-1} \mr P_\RR \left(U\right)$. Let $V$ be the preimage of $U$ in $\coprod_{\alpha \in P\mr A}\ca F_s^\alpha$. Then 
\[
V = \overline{\ca P}_{\RR}^{-1}\left( p^{-1}\left(\mr P_\RR(U)\right)\right)
\]
and hence 
$$
    \overline{\ca P}_{\RR}\left( V \right) = p^{-1}\left(\mr P_\RR(U)\right).
$$
The map $\overline{\ca P}_{\RR}$ is open, being the coproduct of the maps $\ca F_s^\alpha \to \RR H^2_\alpha$, which are open since they have surjective differential at each point. Thus $\mr P_\RR(U)$ is open in $Y$. 
\end{proof}

\begin{corollary} \label{cor:theorem2}
%Consider the real varieties $X_0 \subset X$ acted upon by $G$ with $\ca L \in \Pic(X)^G$ together with the Galois-equivariant hermitian $\ZZ$VHS of signature $(n,1)$ on $X_0(\CC)$ as above. We thus assume $X_0 \subset X_s$, $\Delta = X_s - X_0$ is an irreducible normal crossings divisor, the monodromy transformations of $\ca F_0 \to X_0(\CC)$ around general points on components of $\Delta$ are of finite order $m$, and the 
%Assume that the assumptions of Lemma \ref{lemma:realsmoothperiods} are satisfied. %The real period map gives an isomorphism  \begin{equation}
   % \GG(\RR) \setminus X_0(\RR) \cong  \coprod_{\alpha \in C \mr A}\left[P\Gamma_\alpha \setminus \left( \RR H^2_\alpha - \mr H
    %\right)\right]
    %\end{equation}
 %   Assume Hypothesis \ref{orthogonal} (e.g. this holds when $\eta$ generates $\mf D_K$). Then 
 There is a lattice $P\Gamma_\RR \subset \textnormal{PO}(2,1)$, an inclusion of orbifolds
 \begin{align}\label{inclusionofforbifodls}
 \coprod_{\alpha \in C \mr A}P\Gamma_\alpha \setminus \left( \RR H^2_\alpha - \mr H
    \right) \hookrightarrow P\Gamma_\RR \setminus \RR H^2, 
 \end{align}
 and a homeomorphism
\begin{equation} \label{stablehom}
\ca M_s(\RR) =  \GG(\RR) \setminus X_s(\RR) \cong P\Gamma_\RR \setminus \RR H^2 
\end{equation}
such that (\ref{stablehom}) restricts to (\ref{smoothrealperiodiso}) with respect to (\ref{inclusionofforbifodls}). 
%where $C$ ranges over the set of connected components of the semi-analytic space $ \GG(\RR) \setminus X_s(\RR)$, and 
 % for each $C \in \pi_0\left(\GG(\RR) \setminus X_s(\RR)\right)$. 
%In other words, for each connected component $C$ of $ \GG(\RR) \setminus X_s(\RR)$ there is a subset $I \subset C\mr A$, 
%Consider the real algebraic variety $X$ equipped with the $G$-action as above. 
\end{corollary}
\begin{proof}
This follows directly from Theorems \ref{glueingtheorem1} and \ref{th:realstableperiod}. 
%Indeed, we have $\Gamma \setminus Y \cong  P\Gamma_\RR \setminus \RR H^2$ by Theorem \ref{glueingtheorem1}, and $\GG(\RR) \setminus X_s(\RR) \cong P\Gamma \setminus Y$ by Theorem \ref{th:realstableperiod}. 
\end{proof}

\begin{remark}
The proof of Theorem \ref{th:realstableperiod} also shows that $\ca M_s(\RR)$ is homeomorphic to the glued space $P\Gamma \setminus Y$ (see Definition \ref{gluedspace}) if $\ca M_s$ is the stack of cubic surfaces or of binary sextics. This strategy to uniformize the real moduli space differs from the one used in \cite{realACTsurfaces, realACTnonarithmetic, realACTbinarysextics}, since we first glue together the real ball quotients, and then prove that our real moduli space is homeomorphic to the result. %, instead of performing the gluing directly inside the real moduli space. 
\end{remark}

%\textcolor{red}{Change the following text accordingly:}

\subsection{Automorphism groups of stable real binary quintics} \label{sec:1}

%The following is readily checked: 
%\begin{lemma} \label{lemma:polynomialzero}
%Mapping a polynomial to its zero locus defines isomorphisms of real orbifolds $ \GG(\RR) \setminus X_s(\RR) \cong \PGL_2(\RR) \setminus (P_s/\mf S_5)(\RR)$ and $\GG(\RR) \setminus X_0(\RR) \cong \PGL_2(\RR) \setminus (P_0/\mf S_5)(\RR)$. \hfill \qed
%\end{lemma}
%Our goal is to prove Theorem \ref{th:triangle}, which states that the moduli space $\ca M_s(\RR) = \GG(\RR) \setminus X_s(\RR)$ of stable binary quintics equipped with the metric of Corollary \ref{cor:theorem2} is isometric to the triangle $\Delta_{3,5,10} \subset H^2$ of angles $\pi/3, \pi/5$ and $\pi/10$ in the real hyperbolic plane $\RR H^2$. 

Before we can finish the proof of Theorem \ref{th:theorem02}, we need to understand the orbifold structure of $\ca M_s(\RR)$, and how this structure differs from the orbifold structure of the glued space $P\Gamma \setminus Y$. In the current Section \ref{sec:1} we start by analyzing the orbifold structure of $\ca M_s(\RR)$, by listing its stabilizer groups. %Recall that $P_s \subset (\PP^1_\RR)^5$ is the variety parametrizing ordered five-tuples $(x_1, \dotsc, x_5)$ such that no three $x_i$ coincide. 
There is a canonical orbifold isomorphism
$
\ca M_s(\RR) = \GG(\RR) \setminus X_s(\RR) = (P_s/\mf S_5)(\RR).
$
Therefore, to list the automorphism groups of binary quintics is to list the elements $x = [\alpha_1, \dotsc, \alpha_5] \in (P_s/\mf S_5)(\RR)$ whose stabilizer 
$ \PGL_2(\RR)_x
$ is non-trivial, and calculate $\PGL_2(\RR)_x$ in these cases. This will be our next goal. 

\begin{proposition} \label{prop:allstabilizergroups}
All stabilizer groups $\PGL_2(\RR)_x \subset \PGL_2(\RR)$ for points $x \in (P_s/ \mf S_5)(\RR)$ are among $\ZZ/2, D_3, D_5$. For $n \in \{3,5\}$, there is a unique $\PGL_2(\RR)$-orbit in $(P_s/ \mf S_5)(\RR)$ of points $x$ with stabilizer $D_n$.  
\end{proposition}
\begin{proof}
%First observe that since we are dealing with finite subgroups of $\PGL_2(\RR)$ this already restricts the possible stabilizer groups $\PGL_2(\RR)_x$ that can occur: %it is classical that all finite subgroups of $\PGL_2(\bar K)$ for an algebraically closed field $\bar K$ are isomorphic to $\ZZ/r$, $D_r$ (the dihedral group of order $2r$), $\mf A_4$, $\mf S_4$ or $\mf A_5$, and that there is only one conjugacy class for each of these groups \cite{beauvillePGL2}. Moreover, Beauville proves in [\textit{loc. cit.}, Prop. 1.1] which of these occur as subgroups of $\PGL_2(K)$ for a an arbitrary field $K$; it follows that 
%the only finite subgroups of $\PGL_2(\RR)$ are isomorphic to $\ZZ/r$ and $D_r$ for $r \in \ZZ_{\geq 2}$ \cite{beauvillePGL2}. %but there are no embeddings of $\mf A_4, \mf S_4$ or $\mf A_5$ into $\PGL_2(\RR)$ . 
%(Alternatively, %stabilizer groups of real $2$-dimensional orbifolds can only be of the form $\ZZ/n$ or $D_n$: 
%see \cite[Proposition 13.3.1]{Thurston80}). On the other hand, 
We have an injection 
$
(P_s/\mf S_5)(\RR) \hookrightarrow P_s/\mf S_5$ which is equivariant for the embedding $\PGL_2(\RR) \hookrightarrow \PGL_2(\CC)$. In particular, $\PGL_2(\RR)_x \subset \PGL_2(\CC)_x$ for every $x \in (P_s/\mf S_5)(\RR)$. The groups $\PGL_2(\CC)_x$ for equivalence classes of distinct points $x \in P_0/\mf S_5$ are calculated in \cite[Theorem 22]{wu2019moduli}, and such a group is isomorphic to $\ZZ/2, D_3, \ZZ/4$ or $D_5$. None of these have subgroups isomorphic to $D_2 = \ZZ/2 \rtimes \ZZ/2$ or $D_4 = \ZZ/2 \rtimes \ZZ/4$. %We first eliminate the possibility $\ZZ/4$. 
Define an involution 
\begin{equation*} %\label{eq:nu}
    \nu \coloneqq (z \mapsto 1/z) \in \PGL_2(\RR).
\end{equation*}
\noindent
The proof of Proposition \ref{prop:allstabilizergroups} will follow from the following Lemmas \ref{lemma:involution}, \ref{lem-two}, \ref{lemma:D3} and \ref{lemma:D5}. 

\begin{lemma} \label{lemma:involution}
Let $\tau \in \PGL_2(\RR)$. Consider a subset $S = \{x,y,z\} \subset \PP^1(\CC)$ stabilized by complex conjugation, such that $\tau(x) = x$, $\tau(y) = z$ and $\tau(z) = y$. There is a transformation $g \in \PGL_2(\RR)$ that maps $S$ to either $\{-1, 0, \infty\}$ or $\{-1, i, -i\}$, and that satisfies $g \tau g^{-1} = \nu = (z \mapsto 1/z) \in \PGL_2(\RR)$. In particular, $\tau^2 = \id$. 
%\\
%\\
%Let $\tau \in \PGL_2(\RR)$ be a transformation such that there exist $x = (x_1, x_2, x_3, x_4, x_5) \in (P_s/\mf S_5)(\RR)$ with $x_i \neq x_j$ for $i \neq j \leq 3$, and such that $\tau(x_1) = x_1 \in \PP^1(\RR)$, $\tau(x_2) = x_3$, $\tau(x_3) = x_2$, $\tau(x_4) = x_5 $ and $\tau(x_5) = x_4$. Then $x$ is equivalent to an element either of the form $(-1, 0, \infty, \beta, \beta^{-1})$ or of the form $(-1, i, -i, \beta, \beta^{-1})$, and $\tau$ becomes the transformation $\nu = (z \mapsto 1/z) \in \PGL_2(\RR)$ under this transformation.  
\end{lemma}
\begin{proof}
Indeed, two transformations $g,h \in \PGL_2(\CC)$ that satisfy $g(x_i) = h(x_i)$ for three different points $x_1, x_2, x_3 \in \PP^1(\CC)$ are necessarily equal. 
%Note that $\tau (\bar x) = \bar x$ hence $x = \bar x \in \PP^1(\RR)$. Suppose first that $y \in \PP^1(\RR)$; then also $z \in \PP^1(\RR)$. Let $g \in \PGL_2(\RR)$ be such that $g(x) = -1, g(y) = 0$ and $g(z) = \infty$ and consider $\phi = g \tau g^{-1}$. We have that $\phi(-1) = - 1$ and $\phi(0) = \infty$; consequently, we must have $\phi(\infty) = 0$. Since $\phi$ and $\nu$ coincide at three distinct points, they coincide. Now suppose that $y \in \PP^1(\CC) \setminus \PP^1(\RR)$; then also $z \in \PP^1(\CC) \setminus \PP^1(\RR)$. Let $g \in \PGL_2(\RR)$ be such that $g(x) = -1$ and $g(y) = i$. Then $g(z) = g( \bar y) = \overline{g(y)} = -i$. Consider $\phi = g \tau g^{-1}$. We have that $\phi(-1) = - 1$ and $\phi(i) = -i$; consequently, we must also have $\phi(-i) = i$. It follows that $\phi = \nu$ and we are done.
\end{proof}

\begin{lemma} \label{lem-two}
There is no $x \in (P_s/\mf S_5)(\RR)$ stabilized by some $\phi \in \PGL_2(\RR)$ of order $4$. 
%whose stabilizer $\PGL_2(\RR)_x \subset \PGL_2(\RR)$ is isomorphic to $\ZZ/4$. 
\end{lemma}
\begin{proof}
By \cite[Theorem 4.2]{beauvillePGL2}, all subgroups $G \subset \PGL_2(\RR)$ that are isomorphic to $\ZZ/4$ are conjugate to each other. Since the transformation $I: z \mapsto (z-1)/(z+1)$ is of order $4$, it gives a representative $G_I = \langle I \rangle$ of this conjugacy class. Hence, assuming there exists $x$ and $\phi$ as in the lemma, %we may a
possibly after replacing $x$ by $gx$ for some $g \in \PGL_2(\RR)$, we may and do assume that $\phi = I$. On the other hand, it is easily shown that $I$ cannot fix any $x \in (P_s/\mf S_5)(\RR)$. 
%we have $\langle \phi \rangle = G_I$. It is easy to show that this leads to a contradiction. 
%Necessarily $G_I$ must fix one of the $x_i \in x$ since otherwise it would contain an element of order at least $5$. So suppose $G_I$ fixes $z = x_1 \in x$. Note that $G_I$ fixes none of the $x \in (x_2, \dotsc, x_5)$ for if it did, it would permute freely a subset $S \subset \PP^1(\CC)$ of cardinality three which is not possible for an order four subgroup of $\PGL_2(\CC)$. By Lemma \ref{lemma:involution}, neither does $G_I$ divide $(x_2, \dotsc, x_5)$ in two orbits acted upon by involutions. Because $I \in \PGL_2(\RR)$, we have $I (\bar z) = \overline{ I(z)} = \bar z$; therefore $I$ fixes $\bar z$, hence $\bar z = z$. But by definition of $I$, we have, if $z \neq \infty$, that $$ I(z) = \frac{z-1}{z+1} = z \white \iff \white z-1 = z^2 +z \white \iff \white z^2 + 1 = 0 $$ contradicting the fact that $\bar z = z$. If $z = [1:0] \in \PP^1(\RR)$, then $I(z) = 1 $ which contradicts $I(z) = z$. Now if $x = (x_1, x_1, x_3, x_4, x_5)$ has one double point, then this point must be real but also be fixed by $I$ which is a contradiction by the above. Similarly, if $x = (x_1, x_1, x_3, x_3, x_5)$ has two double points, then $x_5$ must be real but fixed by $I$, contradiction. 
\end{proof}
\noindent
Define
$$
\rho \in \PGL_2(\RR), \white \rho(z) = \frac{-1}{z+1}.
$$

\begin{lemma} \label{lemma:D3}
Let $x = (x_1, \dotsc, x_5) \in (P_s/\mf S_5)(\RR)$. Suppose $\phi(x) = x$ for an element $\phi \in \PGL_2(\RR)$ of order $3$. There is a transformation $g \in \PGL_2(\RR)$ mapping $x$ to $z = (-1, \infty, 0, \omega, \omega^2)$ with $\omega$ a primitive third root of unity, and the stabilizer of $x$ to the subgroup of $\PGL_2(\RR)$ generated by $\rho$ and $\nu$. In particular, $\PGL_2(\RR)_x$ is isomorphic to $D_3$. 
%Suppose there exists $\phi \in \PGL_2(\RR)$ such that $\phi(x) = x$. 
\end{lemma}

\begin{proof}
It follows from Lemma \ref{lemma:involution} that there must be three elements $x_1, x_2, x_3$ which form an orbit under $\phi$. Since complex conjugation preserves this orbit, one element in it is real; since $g$ is defined over $\RR$, they are all real. Let $g \in \PGL_2(\RR)$ such that $g(x_1) = -1$, $g(x_2) = \infty$ and $g(x_3) = 0$. Define $\kappa = g \phi g^{-1}$. Then $\kappa^3 = \id$, and $\kappa$ preserves $\{-1, \infty, 0\}$ and sends $-1$ to $\infty$ and $\infty$ to $0$. Consequently, $\kappa(0) = -1$, and it follows that $\kappa = \rho$. Hence $x$ is equivalent to an element of the form 
$z = (-1, \infty, 0, \alpha, \beta)$. Moreover, $\beta = \bar \alpha$ and $\alpha^2 + \alpha + 1 = 0$. 
%where possibly, $\alpha = \beta$, and where $\beta = \bar \alpha$ if $\alpha \not \in \PP^1(\RR)$. Note that $\alpha$ and $\beta$ do not equal either $-1$, $\infty$ or $0$. If $\rho(\alpha) = \alpha$ then $\alpha^2 + \alpha + 1 = 0$, hence $\alpha = \omega$ for a primitive third root of unity $\omega \in \PP^1(\CC)$. In this case, $\beta  = \omega^{2}$ and so $\nu: z \mapsto 1/z$ gives a reflection proving that $H_z \cong D_3$. If $\rho(\alpha) = \beta = \bar \alpha \in \PP^1(\CC) \setminus \PP^1(\RR)$, then $-1 = (\alpha + 1) \bar \alpha = |\alpha|^2 + \bar \alpha$ which contradicts the fact that $\alpha$ is imaginary. Note finally that $\rho(\alpha) \neq \beta$ if $\alpha \neq \beta$ are both real, since then $\rho^2(\alpha) = -(\alpha + 1)/\alpha = \alpha$ so that $\alpha^2 + \alpha + 1 = 0$. 
%if $\phi$ is an involution on two points and fixes a third, it is an involution by Lemma .... 
\end{proof}
\noindent
Recall that $\zeta_5 = e^{2i\pi/5} \in \PP^1(\CC)$ and define
$$
\lambda = \zeta_5 + \zeta_5^{-1} \in \RR, \white\white \gamma(z) = \frac{(\lambda + 1)z - 1}{z + 1} \in \PGL_2(\RR). 
$$

\begin{lemma}\label{lemma:D5}
Let $x = (x_1, \dotsc, x_5) \in (P_s/\mf S_5)(\RR)$. Suppose $x$ is stabilized by a subgroup of $\PGL_2(\RR)$ of order $5$. There is a transformation $g \in \PGL_2(\RR)$ mapping $x$ to $z = (0, -1, \infty, \lambda+1, \lambda)$ and identifying the stabilizer of $x$ with the subgroup of $\PGL_2(\RR)$ generated by $\gamma$ and $\nu$. In particular, the stabilizer $\PGL_2(\RR)_x$ of $x$ is isomorphic to $D_5$. 
%Suppose there exists $\phi \in \PGL_2(\RR)$ such that $\phi(x) = x$. 
\end{lemma}

\begin{proof}
Let $\phi\in \PGL_2(\RR)_x$ be an element of order $5$. %If $x = (x_1, x_1, x_2, x_2, x_3)$ has two double points, then $\phi(x_3) = x_3$ and since $\phi(x_1) \neq x_2$ by Lemma \ref{lemma:involution}, we have $\phi(x_i) = x_i$ for every $i$, hence $\phi = \id$. If $x = (x_1, x_1, x_3, x_4, x_5)$ has one double point, then $\phi(x_1) = x_1$ and, since $\phi(x_j) \neq x_j$ for $j \neq 1$ by Lemma \ref{lemma:involution} we must have that $\phi$ permutes $(x_3, x_4, x_5)$ freely. But then $\phi^3 = \id$, contradiction. 
Using Lemma \ref{lemma:involution} one shows that $x$ must be smooth, i.e. all $x_i$ are distinct, and $x_i = \phi^{i-1}(x_1)$. Since there is one real $x_i$ and $\phi$ is defined over $\RR$, all $x_i$ are real. Now note that  $z = (0, -1, \infty, \lambda + 1, \lambda)$ is the orbit of $0$ under $\gamma: z \mapsto ((\lambda+1)z - 1)/(z+1)$. The reflection $\nu: z \mapsto 1/z$ preserves $z$ as well: if $\zeta = \zeta_5$ then $\lambda = \zeta + \zeta^{-1}$ hence $\lambda + 1 = - (\zeta^2 + \zeta^{-2}) = - \lambda^2 + 2$, so that $\lambda(\lambda + 1) = 1$. So we have $\PGL_2(\RR)_z \cong D_5$. By \cite[Theorem 22]{wu2019moduli}, the point $z$ with its stabilizer $\PGL_2(\RR)_z$ must be equivalent under $\PGL_2(\CC)$ to the point $(1, \zeta, \zeta^2, \zeta^3, \zeta^4)$ with its stabilizer $\langle x \mapsto \zeta x, x \mapsto 1/x \rangle$. %Their result also implies that $\PGL_2(\RR)_x \cong D_5$. 
Thus, there exists $g \in \PGL_2(\CC)$ such that $g(x_1) = 0$, $g(x_2) = -1$, $g(x_3) = \infty$, $g(x_4) = \lambda + 1$ and $g(x_5) = \lambda$, and such that $g\PGL_2(\RR)_xg^{-1} = \PGL_2(\RR)_z$. Since all $x_i$ and $z_i \in z$ are real, we see that $\bar g(x_i) = z_i$ for every $i$, hence $g$ and $\bar g$ coincide on more than $2$ points, hence $g = \bar g \in \PGL_2(\RR)$. 
%We conclude there is an element $g \in \PGL_2(\CC)$ such that 
\end{proof}
\noindent
Proposition \ref{prop:allstabilizergroups} follows. 
\end{proof}

\subsection{Binary quintics with automorphism group of order two} \label{sec:2}
%Points of $(P_s/\mf S_5)(\RR)$ with stabilizer $\ZZ/2 \hookrightarrow \PGL_2(\RR)$}  

%To finish our classification of orbifold points in $\PGL_2(\RR) \setminus (P_s/ \mf S_5)(\RR)$, all that remains is to calculate the possible configurations $x = (x_1, \dotsc, x_5) \in (P_s/ \mf S_5)(\RR)$ that have $\ZZ/2$ as stabilizer group. This will be the aim of this subsection. 
The goal of Section \ref{sec:2} is to prove that there are no cone points in the orbifold $\PGL_2(\RR) \setminus (P_s/ \mf S_5)(\RR)$, i.e. orbifold points whose stabilizer group is $\ZZ/n$ for some $n$ acting on the orbifold chart by rotations. By Proposition \ref{prop:allstabilizergroups}, this fact will follow from the following:
\begin{proposition} \label{prop:zmod2stabilizer}
Let $x=(x_1, \dotsc, x_5) \in (P_s/\mf S_5)(\RR)$ such that $\PGL_2(\RR)_x = \langle \tau \rangle $ has order two. There is a $\PGL_2(\RR)_x$-stable open neighborhood $U \subset (P_s/\mf S_5)(\RR)$ of $x$ such that $\PGL_2(\RR)_x \setminus U \to \ca M_s(\RR)$ is injective, and a homeomorphism $\phi: (U,x) \to (B,0)$ for $0 \in B \subset \RR^2$ an open ball, such that $\phi$ identifies $\PGL_2(\RR)_x$ with $\ZZ/2$ acting on $B$ by reflections in a line through $0$.  
\end{proposition}
\begin{proof}
%There must be at least one coordinate $x_i \in x$ with $\tau(x_i) = x_i$ and it is easy to show that there can be no more than one of such coordinates. 
%; if there are more of such coordinates, then there must be at least three; but then $\tau$ fixes three points hence is the identity. 
%First suppose $x\in (P_0/\mf S_5)(\RR)$. There is precisely one fixed point $x_i \in x$ - which we assume to be $x_1$ - and there are four points $x_2, \dotsc, x_5$ which satisfy $\tau(x_2) = x_3$, $\tau(x_4) = x_5$. Necessarily, $x_1$ is real because $\bar x_1$ is again a fixed point of $\tau$; indeed, since $\tau \in \PGL_2(\RR)$ it commutes with complex conjugation. Similarly, $x_2$ (resp. $x_4$) is real if and only if $x_3$ (resp. $x_5$) is real. 
Using Lemma \ref{lemma:involution}, one checks that the only possibilities for the element $x=(x_1, \dotsc, x_5) \in (P_s/\mf S_5)(\RR)$ are  $(-1, 0, \infty, \beta, \beta^{-1})$,  $(-1, i, -i, \beta, \beta^{-1}) $, $ (-1, -1, \beta, 0, \infty)$,\\
$ (-1, -1, \beta, i, -i) $, $(0,0, \infty, \infty, -1)$ and $(-1, i, i, - i, -i)$.  
\end{proof}

 \subsection{Comparing the orbifold structures} \label{sec:3}

Consider the moduli space $\ca M_s(\RR)$ of real stable binary quintics. 

\begin{definition} \label{neworbifold}
Let $\overline{\mr M}_\RR$ be the hyperbolic orbifold with $\va{\ca M_s(\RR)}$ as underlying space, whose orbifold structure is induced by the homeomorphism of Corollary \ref{cor:theorem2} and the natural orbifold structure of $P\Gamma_\RR \setminus \RRH^2$. 
\end{definition}
\noindent
There are two orbifold structures on the space $\va{\ca M_s(\RR)}$: the natural orbifold structure of $\ca M_s(\RR)$, see Proposition \ref{from-stack-to-orbifold} (i.e. the orbifold structure of the quotient $ \GG(\RR) \setminus X_s(\RR)$), and the orbifold structure $\overline{\mr M}_\RR$ introduced in Definition \ref{neworbifold}. 
%and one via the homeomorphism 
%one via the homeomorphism $\va{\ca M_s(\RR)} \cong P\Gamma \setminus Y$ of Theorem \ref{th:realstableperiod} and the orbifold structure on the glued space $P\Gamma \setminus Y$ given by Theorem \ref{glueingtheorem1}. 

\begin{proposition} \label{prop:conesreflectors}
\begin{enumerate}
\item The orbifold structures of $\ca M_s(\RR)$ and $\overline{\mr M}_\RR$ differ only at the moduli point attached to the five-tuple $(\infty, i, i, -i, -i)$. The stabilizer group of $\ca M_s(\RR)$ at that moduli point is $\ZZ/2$, whereas the stabilizer group of $\overline{\mr M}_\RR$ at that point is the dihedral group $D_{10}$ of order twenty. 
\item The orbifold $\overline{\mr M}_\RR$ has no cone points and three corner reflectors, whose orders are $\pi/3, \pi/5$ and $\pi/10$. 
\end{enumerate}
\end{proposition}

\begin{proof}
%We claim that these topological orbifold structures differ only at the moduli point $(\infty, i, i, -i, -i)$. 
%is the isomorphism class of the genus zero curve which is a chain of three ${\PP^1}'s$ equipped with marked points $\alpha \neq \beta \in \PP^1(\CC)\setminus \PP^1(\RR)$ on the first copy, $\infty \in \PP^1(\RR)$ on the second and $\bar \alpha, \bar \beta$ on the third $\PP^1$ (see the Figure ...). 
The statements can be deduced from Proposition \ref{localmodel}. The notation of that proposition was as follows: for $f \in Y \cong \GG(\RR) \setminus \ca F_s(\RR)$ (see Theorem \ref{th:realstableperiod}) the group $A_f \subset P\Gamma$ is the stabilizer of $f \in K$. Moreover, if $\tilde f \in \ca F_s(\RR)$ represents $f$ and if $F = [\tilde f] \in X_s(\RR)$ has $k = 2a + b$ nodes, then the image $x \in \CC H^2$ has $k = 2a + b$ nodes in the sense of Definition \ref{fullset}. 
%lies on $k$ orthogonal hyperplanes $H_r$, with $a$ pairs of complex conjugate hyperplanes and $b$ real hyperplanes. % where `complex conjugate' and `real' are to be understood with respect to the real structures $\alpha$ on $\CC H^2$ that satisfy $x \in \ca F_s^\alpha$. 
%Hence $G(x) \cong (\ZZ/10)^k$ where $\textbf r \in \mr R^k$ is a vector of short roots such that $x \in \cap_{i = 1}^k H_{r_i}$, $H_{r_i} \neq H_{r_j}$ for $i \neq j$ and $k$ maximal for this property. Note that $k \leq 2$. We have $B_f = \langle h_{r_i} \rangle_{i > 2a} \cong (\ZZ/10)^b$: this is the group generated by reflections in all the real hyperplanes. Finally, recall that $B_f \setminus K_f$ is a union of ${\bb B^2(\RR)}'s$ identified along common ${\bb B^2(\RR)}'s$ and ${\bb B^1(\RR)}'s$ (Proposition \ref{localmodel}.4). Let $\bb B$ be any one of them. Then $\Gamma_f$ is the stabilizer of $\bb B$ in the group $A_f / B_f$. Let $G_F \subset \GG(\RR)$ be the stabilizer of $F \in X_s(\RR)$; one observes that $G_F  = \Gamma_f / G(x)$. 
If $F$ has no nodes ($k = 0$), then $G(x)$ is trivial by Proposition \ref{localmodel}.\ref{caseone} and $G_F = A_f = \Gamma_f$. If $F$ has only real nodes, then $B_f = G(x)$ hence $G_F = A_f/G(x) = A_f / B_f = \Gamma_f$. %Hence the topological orbifold structures of $Y$ and $[\GG(\RR) \setminus X_s(\RR)]$ agree at those $[F]$ %\in [\GG(\RR) \setminus X_s(\RR)]$ 
%such that $F$ has either no nodes or only real nodes.
Now suppose that $a = 1$ and $b = 0$: the equation $F$ defines a pair of complex conjugate nodes. In other words, the zero set of $F$ defines a $5$-tuple $\underline{\alpha}= (\alpha_1, \dotsc, \alpha_5) \in \PP^1(\CC)$, well-defined up to the $\PGL_2(\RR) \times \mf S_5$ action on $\PP^1$, where $\alpha_1 \in \PP^1(\RR)$ and $\alpha_3 = \bar \alpha_2  = \alpha_5 = \bar \alpha_4 \in \PP^1(\CC) \setminus \PP^1(\RR)$. So we may write $\underline{\alpha} = (\rho, \alpha, \bar \alpha, \alpha, \bar \alpha)$ with $\rho \in \PP^1(\RR)$ and $\alpha \in \PP^1(\CC) \setminus \PP^1(\RR)$. Then there is a unique $T \in \PGL_2(\RR)$ such that $T(\rho) = \infty$ and $T(\alpha) = i$. But this gives $T(x) = (\infty, i , -i, i, -i)$ hence $F$ is unique up to isomorphism. As for the stabilizer $G_F = A_f / G(x)$, we have $G(x) \cong (\ZZ/10)^2$. Since there are no real nodes, $B_f$ is trivial. By Proposition \ref{localmodel}.\ref{casethree}, $K_f$ is the union of ten copies of $\bb B^2(\RR)$ meeting along a common point $\BB^0(\RR)$. In fact, in the local coordinates $(t_1, t_2)$ around $f$, the $\alpha_j: \BB^2(\CC) \to \BB^2(\CC)$ are defined by $(t_1, t_2) \mapsto (\bar t_2 \zeta^j , \bar t_1 \zeta^j)$, for $j \in \ZZ/10$, and so the fixed points sets are given by the equations $\RR H^2_j  = \{t_2 = \bar t_1 \zeta^j\}\subset \BB^2(\CC)$, $j \in \ZZ/10$. Notice that the subgroup $E \subset G(x)$ that stabilizes $\RR H^2_j$ is the cyclic group of order $10$ generated by the transformations $(t_1, t_2) \mapsto (\zeta t_1, \zeta^{-1} t_2)$. % (this is also the largest subgroup of $G(x)$ that commutes with the $\alpha_j$). Therefore, we obtain an exact sequence 
%$
%0 \to \ZZ/10 \to \Gamma_f \to G_F \to 0. 
%$
%Since the isomorphism class of $F$ corresponds to class of $x = (\infty, i, i,- i, -i) \in (P_s/\mf S_5)(\RR)$ via Lemma \ref{lemma:polynomialzero}, the stabilizer $G_F \subset \GG(\RR)$ is canonically isomorphic to the stabilizer $\PGL_2(\RR)_x$ of the point $x$ in $\PGL_2(\RR)$. 
There is only one non-trivial transformation $T \in \PGL_2(\RR)$ that fixes $\infty$ and sends the subset $\{i, -i\} \subset \PP^1(\CC)$ to itself, and $T$ is of order $2$. Hence $G_F = \ZZ/2$ so that we have an exact sequence $
0 \to \ZZ/10 \to \Gamma_f \to \ZZ/2 \to 0$ and this splits since $G_F$ is a subgroup of $\Gamma_f$. %Let $\phi \in G_F$ be the non-trivial automorphism of $x \in (P_s/\mf S_5)(\RR)$. Then there is an open neighborhood $x \in U \subset (P_s/\mf S_5)$ such that $(U,x)$ is homeomorphic to $(V,0)$ with $0 \in V \subset \BB^2(\RR)$ open, and $G = \langle \phi \rangle \cong \ZZ/2$ acting on $V$ by reflections in the $x$-axis. There is also an open $f \in W \subset B_f \setminus K_f$ with $(W,f)$ homeomorphic to $(V,0)$ (Proposition \ref{localmodel}\textcolor{blue}{.5}); on this open, $\Gamma_f \cong \ZZ/10$ acts by rotations of order $10$ around $0$, hence
%$
%\Gamma_f = \ZZ/10 \rtimes \ZZ/2 = D_{10}$. 
We are done by Propositions \ref{prop:allstabilizergroups} and \ref{prop:zmod2stabilizer}. % there is one orbifold point with stabilizer $D_3$, one with stabilizer $D_5$, and the other orbifold points have stabilizer $\ZZ/2$ acting locally by reflections. 
\end{proof}
%We have proved the following:

\subsection{The real moduli space as a hyperbolic triangle}
\label{sec:triangle2} 
%In this Section \ref{sec:triangle2}, we consider the orbifold $\overline{\mr M}_{\RR}$, see Definition \ref{neworbifold}. %= \GG(\RR) \setminus X_s(\RR)$ as a hyperbolic orbifold via the metric induced by the homeomorphism $\GG(\RR) \setminus X_s(\RR) \cong P\Gamma \setminus K$ of Theorem \ref{th:realstableperiod}, the metric and hyperbolic orbifold structure on $P\Gamma \setminus K$ of Theorem \ref{glueingtheorem1}.1. 
The goal of Section \ref{sec:triangle2} is to show that $\overline{\mr M}_{\RR}$ (see Definition \ref{neworbifold}) is isomorphic, as hyperbolic orbifolds, to the triangle $\Delta_{3,5,10}$ in the real hyperbolic plane $\RRH^2$ with angles $\pi/3, \pi/5$ and $\pi/10$. %In other words: there exists an isometry $\phi: \overline{\mr M}_{\RR} \xrightarrow{\sim} \Delta_{3,5,10}$ that preserves the orbifold structures. 
The results in the above Sections \ref{sec:1}, \ref{sec:2} and \ref{sec:3} give the orbifold singularities of $\overline{\mr M}_{\RR}$ together with their stabilizer groups. In order to complete determine the hyperbolic orbifold structure of $\overline{\mr M}_{\RR}$, however, we also need to know the underlying topological space $\va{\ca M_s(\RR)}$ of $\overline{\mr M}_{\RR}$. The first observation is that $\va{\ca M_s(\RR)}$ is compact. Indeed, it is classical that the topological space $\ca M_s(\CC) = \GG(\CC) \setminus X_s(\CC)$, parametrizing complex stable binary quintics, is compact. This follows from the fact that it is homeomorphic to $\overline{M}_{0,5}(\CC)/ \mf S_5$, and the stack of stable five-pointed curves $\overline{M}_{0,5}$ is proper \cite{Knudsen1983}, or from the fact that it is homeomorphic to a compact ball quotient \cite{shimuratranscendental}. Moreover, the map $\ca M_s(\RR) \to \ca M_s(\CC)$ is proper, which proves the compactness of $\ca M_s(\RR)$. 
 
The second observation is that $\ca M_s(\RR)$ is connected, since $X_s(\RR)$ is obtained from the euclidean space $ \{F \in \RR[x,y]: F \text{ homogeneous } \mid \deg(F) =  5\}$ by removing a subspace of codimension at least two. We can prove more:
\begin{lemma} \label{lem:simplyconnected}
The moduli space $\ca M_s(\RR)$ of real stable binary quintics is simply connected.
\end{lemma}

\begin{proof}
%Let $\Delta \subset \RR^2$ be a closed equilateral triangle and $\square\subset \RR^2$ a closed square whose sides have the same length as the sides of $\Delta$. We shall prove that $X_O$ has two closed subsets which are homeomorphic to $\$ is  can be constructed by gluing two copies of $\Delta$ two two adjacent sides of $\square$. 
%Let $\Delta \subset \RR^2$ be any closed equilateral triangle. 
The idea is to show that the following holds:
\begin{enumerate}
\item 
For each $i \in \{0,1,2\}$, the embedding $\mr M_i \hookrightarrow \overline{\mr M}_i \subset \ca M_s({\RR})$ of the connected component $\mr M_i$ of $\ca M_0(\RR)$ into its closure in $\ca M_s({\RR})$ is homeomorphic to the embedding $D \hookrightarrow \overline{D}$ of the open unit disc into the closed unit disc in $\RR^2$. 
\item We have $\ca M_s(\RR) = \overline{\mr M}_0  \cup \overline{\mr M}_1  \cup \overline{\mr M}_2$, and this glueing corresponds up to homeomorphism to the glueing of three closed discs $\overline{D}_i \subset \RR^2$ as in Figure~\ref{fig:triangle}. 
\end{enumerate}
%\ $\ca M_s({\RR})$ are homeomorphic to the closed disc $\bar D$ in $\RR^2$, and that they are glued together to form $\ca M_s({\RR})$ in the way that the components $\mr M_0$ and $\mr M_2$ are glued to $\mr M_1$ in Figure~\ref{fig:triangle}. 
%two discs are glued to a third by identifying closed connected intervals of the boundaries. 
%$X_O$ is the union of three closed subsets which are homeomorphic to $\Delta$, and that $X_O$ is in fact obtained by gluing one of them to another along one side and then the second to the third along another. Clearly, this will prove that $X_O$ is simply connected. 
To do this, one considers the moduli spaces of real smooth (resp. stable) genus zero curves with five real marked points \cite{Knudsen1983}, as well as twists of this space. %Recall that $P_0(\CC) \subset \PP^1(\CC)$ is the subset of distinct five-tuples, and $P_s(\CC) \subset \PP^1(\CC)$
%Let $M_{0,5}(\RR)$ (resp. $\overline{M}_{0,5}(\RR)$) be the  the subset of five-tuples where no three coordinates coincides. %If $P_0(\RR) \subset \PP^1(\RR)^5$ (resp. $P_s(\RR) \subset \PP^1(\RR)^5$) is the locus where no points (resp. no three points) coincide,
%We have $M_{0,5}(\RR) = \PGL_2(\RR) \setminus P_0(\RR)$ and $\overline{M}_{0,5}(\RR) = \PGL_2(\RR) \setminus P_s(\RR)$. %We shall use this observation to define other moduli spaces of real five-tuples in $\PP^1(\CC)$. 
Define two anti-holomorphic involutions $\sigma_i: \PP^1(\CC)^5 \to \PP^1(\CC)^5$ by
$
\sigma_1(x_1, x_2, x_3, x_4, x_5) = (\bar x_1, \bar x_2, \bar x_3, \bar x_5, \bar x_4), 
$
and %let $\sigma_2: \PP^1(\CC)^5 \to \PP^1(\CC)^5$ be the anti-holomorphic involution 
$
\sigma(x_1, x_2, x_3, x_4, x_5) = (\bar x_1, \bar x_3, \bar x_2, \bar x_5, \bar x_4). 
$
Then define 
$$
P_0^1(\RR) = P_0(\CC)^{\sigma_1}, \white P_s^1(\RR) = P_1(\CC)^{\sigma_1}, \white
P_0^2(\RR) = P_0(\CC)^{\sigma_2}, \white P_s^2(\RR) = P_1(\CC)^{\sigma_2}.
$$
%The action of $\mf S_3 \times \mf S_2$ on $\PP^1(\CC)^5$ commutes with $\sigma_1$ and hence $\mf S_3 \times \mf S_2$ preserves the sets $P_0^1(\RR)$ and $P_s^1(\RR)$. Similarly, $\mf S_2 \times \mf S_2$ preserves $P_0^2(\RR)$ and $P_s^2(\RR)$. On the other hand, $\PGL_2(\RR)$ acts on each $P_0^i(\RR)$ and each $P_s^i(\RR)$, and this commutes with the symmetric group actions. 
%$\mf S_3 \times \mf S_2$ and $\mf S_2 \times \mf S_2$ actions. 
%Recall that $\mr M_i \subset \ca M_0(\RR)$ was the connected component of unordered five-tuples $(x_1, \dotsc, x_5) \in \PGL_2(\RR) \setminus (P_s/\mf S_5)(\RR)$ that have $i$ pairs of complex conjugate points $x_i \neq \bar x_i \in x$. 
It is clear that $ \mr M_0 = \PGL_2(\RR) \setminus P_0(\RR) / \mf S_5$. Similarly, we have:
\begin{align*}
   \mr M_1 = \PGL_2(\RR) \setminus P_0^1(\RR) / \mf S_3 \times \mf S_2 \quad \tn{ and } \quad  \mr M_2 = \PGL_2(\RR) \setminus P_0^2(\RR) / \mf S_2 \times \mf S_2.
    \end{align*}
  Moreover, we have $\overline{\mr M}_0 = \PGL_2(\RR) \setminus P_s(\RR) / \mf S_5$. We define 
\begin{align*}
\overline{\mr M}_1 = \PGL_2(\RR) \setminus P_s^1(\RR) / \mf S_3 \times \mf S_2, \quad \tn{ and } \quad 
\overline{\mr M}_2 = \PGL_2(\RR) \setminus P_s^2(\RR) / \mf S_2 \times \mf S_2.  
\end{align*}
Each $\overline{\mr M}_i$ is simply connected. Moreover, the natural maps $\overline{\mr M}_i \to \ca M_s(\RR) $ are closed embeddings of topological spaces, and one can check that the images glue to form $\ca M_s(\RR)$ in the prescribed way. We leave the details to the reader. 
\end{proof}

\begin{proof}[Proof of Theorem \ref{th:theorem02}]
To any closed $2$-dimensional orbifold $O$ one can associate a set of natural numbers $S_O = \{n_1, \dotsc, n_k; m_1, \dotsc, m_l\}$ by letting $k$ be the number of cone points of $X_O$, $l$ the number of corner reflectors, $n_i$ the order of the $i$-th cone point and $2m_j$ the order of the $j$-th corner reflector. A closed $2$-dimensional orbifold $O$ is then determined, up to orbifold-structure preserving homeomorphism, by its underlying space $X_O$ and the set $S_O$ \cite{Thurston80}. By Lemma \ref{lem:simplyconnected}, $\overline{\mr M}_\RR$ is simply connected. By Proposition \ref{prop:conesreflectors}, $\overline{\mr M}_\RR$ has no cone points and three corner reflectors whose orders are $\pi/3, \pi/5$ and $\pi/10$. This implies $\overline{\mr M}_\RR$ and $\Delta_{3,5,10}$ are isomorphic as topological orbifolds. Consequently, the orbifold fundamental group of $\overline{\mr M}_\RR$ is abstractly isomorphic to the group $P\Gamma_{3,5,10}$ defined in (\ref{PGAMMAR}). 
%$\Gamma_{3,5,10} := \langle a,b,c \vert a^2 = b^2 = c^2 = (ab)^3 = (ac)^5 = (bc)^{10} = 1 \rangle$. Now 

Let $\phi: P\Gamma_{3,5,10} \hookrightarrow \text{PSL}_2(\RR)$ be \textit{any} embedding such that $X:=\phi\left(P\Gamma_{3,5,10}\right) \setminus \RR H^2$ is a hyperbolic orbifold; we claim that there is a fundamental domain $\Delta$ of $X$ isometric to $\Delta_{3,5,10}$. Consider the generator $a \in P\Gamma_{3,5,10}$. Since $\phi(a)^2 = 1$, there exists a geodesic $L_1 \subset \RR H^2$ such that $\phi(a) \in \text{PSL}_2(\RR) = \text{Isom}(\RR H^2)$ is the reflection across $L_1$. Next, consider the generator $b \in P\Gamma_{3,5,10}$. There exists a geodesic $L_2 \subset \RR H^2$ such that $\phi(b)$ is the reflection across $L_2$. One easily shows that $L_2 \cap L_1 \neq \emptyset$. 
%Suppose for contradiction that $L_2 \cap L_1 = \emptyset$. Then the composition $\phi(a)\phi(b)$ is the translation along the unique geodesic that intersects $L_1$ and $L_2$ orthogonally; henceforth, $\phi(a)\phi(b)$ is of infinite order which is absurd. So we must have $L_1 \cap L_2 \neq \emptyset$. 
Let $x \in L_1 \cap L_2$. Then $\phi(a)\phi(b)$ is an element of order three that fixes $x$, hence $\phi(a)\phi(b)$ is a rotation around $x$. Therefore, one of the angles between $L_1$ and $L_2$ must be $\pi/3$. Finally, we know that $\phi(c)$ is an element of order $2$ in $\PSL_2(\RR)$, hence a reflection across a line $L_3$. By the previous arguments, $L_3 \cap L_2 \neq \emptyset$ and $L_3 \cap L_1 \neq \emptyset$. It also follows that $x \in L_3 \cap L_2 \cap L_1 = \emptyset$. 
%Suppose for contradiction that there exists an element $x \in L_3 \cap L_2 \cap L_1$. Let $K \subset \RR H^2$ be a small closed disc around $x$. We know that $\psi:=\phi(a)\phi(b)\phi(c)$ has infinite order because of the given presentation of $\Gamma_{3,5,10}$; since $\psi$ fixes $x$ we have that $\psi^n\left(K\right) \cap K \neq \emptyset$ for all $n \in \bb N$. But $X$ is an orbifold hence $\phi\left(\Gamma_{3,5,10}\right)$ acts properly discontinuously on $\RR H^2$, thus the desired contradiction. 
Consequently, %$L_3$ intersects $L_1$ and $L_2$ in two different points, which implies that 
the three geodesics $L_i \subset \RR H^2$ enclose a hyperbolic triangle; the orders of $\phi(a)\phi(b)$, $\phi(a)\phi(c)$ and $\phi(b)\phi(c)$ imply that the three interior angles of the triangle are $\pi/3$, $\pi/5$ and $\pi/10$. 
\end{proof}

\section{The monodromy groups} \label{sec:monodromy}

In this section, we describe the monodromy group $P\Gamma$, as well as the groups $P\Gamma_\alpha$ appearing in Proposition \ref{prop:realsmoothperiods}. As for the lattice $(\Lambda, \mf h)$ (see (\ref{eq:hermitianformonbinaryquinticlattice})), we have:

%attached to the universal $\mr C \to X_0(\CC)$ family of cyclic quintic covers $C \to \PP^1_\CC$ ramified along a smooth binary quintic, as well as the groups $P\Gamma_\alpha$ appearing in Proposition \ref{prop:realsmoothperiods}, i.e. the groups that make $P\Gamma_\alpha \setminus \left( \RR H^2 - \mr H \right)$ isomorphic to a connected component $\mr M_i$ of the moduli space of smooth binary quintics $\ca M_0(\RR)$. 
%We start with an explicit description of $P\Gamma$. 
%Let $\mr C \to X_0(\CC)$ be the universal family of cyclic quintic covers of $\PP^1_\CC$ ramified along a smooth binary quintic. Let $K = \QQ(\zeta_5)$ and let $\Lambda$ be the $\OO_K$-module $\rm H^1(C(\CC), \ZZ)$ for the cyclic quintic $C \to \PP^1_\CC$ ramified along the smooth binary quintic defined by $F_0 \in X_0(\CC)$. 
%, define $\Gamma = \Aut_{\OO_K}(\Lambda)$, $P\Gamma = \Gamma/\mu_K$ and consider the monodromy representation $\rho: \pi_1(X_0(\CC), F_0) \to P\Gamma$ defined in Equation (\ref{eq:monodromy}). Recall that $\rho$ is surjective by Corollary \ref{cor:surjectivemon}. Moreover, we have:
\begin{theorem}[Shimura] \label{th:calculatemonodromyshimura}
%The monodromy representation $\rho: \pi_1(X_0(\CC), F_0) \to P\Gamma$ is surjective, and 
There is an isomorphism of hermitian $\OO_K$-lattices
 $$\left( \Lambda, \mf h \right) \cong \left( \OO_K^3, \textnormal{diag}\left(1,1, \frac{1 - \sqrt 5}{2}\right) \right).$$
\end{theorem}
\begin{proof}
See \cite[Section 6]{shimuratranscendental} as well as item (5) in the table on page 1. 
\end{proof}
%This is easy: by Corollary \ref{cor:surjectivemon}, the monodromy group of $\mr C \to X_0(\CC)$ is the group $P\Gamma = P \Aut_{\OO_K}(\Lambda,\mf h)$ for $\Lambda$ and $\mf h$ as in Section \ref{sec:markedbinary} (i.e. $\Lambda$ can be the $\OO_K$-module $\rm H^1(C(\CC), \ZZ)$ for any cyclic quintic covers $C \to \PP^1_\CC$ ramified along a smooth binary quintic, with $\mf h$ the hermitian form induced by cup product, see Equation (\ref{eq:hermitianformonbinaryquinticlattice})). Moreover, Shimura determines the hermitian lattice $(\Lambda,\mf h)$ in \cite[Section 6]{shimuratranscendental}, proving that 
%\begin{equation}
%    \left( \Lambda, \mf h \right) \cong \left( \OO_K^3, \textnormal{diag}(1,1, \frac{\sqrt 5 - 1}{2}) \right).
%\end{equation}
\noindent
Let us write $\Lambda = \OO_K^3$ and $\mf h = \textnormal{diag}(1,1, \frac{1 - \sqrt 5}{2})$ in the remaining part of Section \ref{sec:monodromy}. Write $\alpha = \zeta_5 + \zeta_5^{-1} =  \frac{\sqrt{5} - 1}{2}$. Recall that $\theta = \zeta_5 - \zeta_5^{-1}$ and observe that $|\theta|^2 = \frac{\sqrt{5} + 5}{2}$. Define three quadratic forms $q_0$, $q_1$ and $q_2$ on $\ZZ[\alpha]^3$ as follows:
\begin{equation} \label{eq:explicitquadraticforms}
\begin{split}
 q_0(x_0, x_1, x_2) & = x_0^2 + x_1^2 - \alpha x_2^2, %= - \alpha x_0^2 + x_1^2 + x_2^2 
\\
 q_1(x_0, x_1, x_2) & = |\theta|^2 x_0^2 + x_1^2 - \alpha x_2^2, %= -\alpha x_0^2 + x_1^2 + \left(3 + \alpha\right) x_2^2 = -\alpha x_0^2 + x_1^2 + |\theta|^2x_2^2
\\
 q_2(x_0, x_1, x_2) & = |\theta|^2 x_0^2 + |\theta|^2 x_1^2 - \alpha x_2^2. %= -\alpha x_0^2 + \left(3 + \alpha\right) x_1^2 + \left(3 + \alpha\right) x_2^2 \\
 %= -\alpha x_0^2 + |\theta|^2 x_1^2 + |\theta|^2 x_2^2.
 %\left(\frac{\sqrt 5 -1}{2}\right)
 \end{split}
\end{equation}
\noindent
We consider $\ZZ[\alpha]$ as a subring of $\RR$ via the standard embedding. 
%Recall that $\mr A$ is the set of anti-unitary involutions $\alpha: \Lambda \to \Lambda$, and that $C \mr A = P\Gamma \setminus P\mr A$, where $P\mr A = \mr A / \mu_K$. 
%that sends $\alpha$ to itself. 

\begin{theorem} \label{th:explicitquadratic}
Consider the quadratic forms $q_j$ defined in (\ref{eq:explicitquadraticforms}). There is a union of geodesic subspaces $\mr H_j \subset \RR H^2$ ($j \in \{0,1,2\}$) and an isomorphism of hyperbolic orbifolds
\begin{equation} 
    \ca M_0(\RR) \cong \coprod_{j = 0}^2 \textnormal{PO}(q_j,\ZZ[\alpha]) \setminus \left(\RR H^2 - \mr H_j \right). 
\end{equation}
\end{theorem}

\begin{proof}
Recall that $\theta = \zeta_5 - \zeta_5^{-1}$; we consider the $\bb F_5$-vector space $W$ equipped with the quadratic form $q = \mf h \mod \theta$. 
%Note that $\theta = \zeta^{-1}(\zeta - \zeta^2)$
Define three anti-isometric involutions as follows:
 %We consider three explicit anti-involutions $\alpha_j$ on the $\ca O_K$-lattice $\Lambda$:
\begin{equation} \label{chi}\begin{split} 
    \alpha_0: & (x_0,x_1,x_2) \mapsto (\;\;\;\bar x_0, \;\;\;\bar x_1,\bar x_2) \\
    \alpha_1: & (x_0,x_1,x_2) \mapsto (-\bar x_0, \;\;\;\bar x_1, \bar x_2) \\
    \alpha_2: & (x_0,x_1,x_2) \mapsto (-\bar x_0, -\bar x_1, \bar x_2).
    \end{split}
\end{equation}
For isometries $\alpha: W \to W$, the dimension and determinant of the fixed space $(W^\alpha, q|_{W^\alpha})$ are conjugacy-invariant. Using this, one easily shows that an anti-unitary involution of $\Lambda$ is $\Gamma$-conjugate to exactly one of the $\pm \alpha_j$, hence %Moreover, two anti-unitary involutions of $\Lambda$ are conjugate if and only if the restrictions of $q$ to the two fixed spaces in $V$ have the same dimension and determinant. 
$C\mr A$ has cardinality $3$ and is represented by $\alpha_0, \alpha_1, \alpha_2$ of (\ref{chi}). By Proposition \ref{prop:realsmoothperiods}, we obtain $\ca M_0(\RR) \cong \coprod_{j = 0}^2 P\Gamma_{\alpha_j} \setminus (\RR H^2_{\alpha_j} - \mr H)$ where each hyperbolic quotient $P\Gamma_{\alpha_j} \setminus (\RR H^2_{\alpha_j} - \mr H)$ is connected. Next, consider the fixed lattices
\begin{equation}\label{eq:fixedlattices} \begin{split} 
\Lambda_0 :=    \Lambda^{\alpha_0} = \ZZ[\alpha] \oplus \ZZ[\alpha] \oplus \ZZ[\alpha] \\
  \Lambda_1:=      \Lambda^{\alpha_1} = \theta \ZZ[\alpha] \oplus \ZZ[\alpha] \oplus \ZZ[\alpha] \\
\Lambda_2:=    \Lambda^{\alpha_2} = \theta \ZZ[\alpha] \oplus \theta\ZZ[\alpha] \oplus  \ZZ[\alpha]. 
\end{split}
\end{equation}
One easily shows that $P\Gamma_{\alpha_j}=N_{P\Gamma}(\alpha_j)$ for the normalizer $N_{P\Gamma}(\alpha_j)$ of $\alpha_j$ in $P\Gamma$. Moreover, if $h_j$ denotes the restriction of $\mf h$ to $\Lambda^{\alpha_j}$, then there is a natural embedding
\begin{equation}
\iota:    N_{P\Gamma}(\alpha_j) \hookrightarrow \textnormal{PO}(\Lambda_j,h_j, \ZZ[\alpha]).
\end{equation}
We claim that $\iota$ is actually an isomorphism. Indeed, this follows from the fact that the natural homomorphism
$
\pi: N_\Gamma(\alpha_j) \to O(\Lambda_j, h_j)
$
is surjective, where $N_\Gamma(\alpha_j) = \{g \in \Gamma: g \circ \alpha_j = \alpha_j \circ g \}$ is the normalizer of $\alpha_j$ in $\Gamma$. The surjectivity of $\pi$ follows in turn from the equality
$$
    \Lambda =  \ca O_K \cdot  \Lambda_j + \ca O_K \cdot \theta \Lambda_j^\vee \subset K^3
$$
which follows from (\ref{eq:fixedlattices}). Since $\textnormal{PO}(\Lambda_j,h_j, \ZZ[\alpha]) = \textnormal{PO}(q_j,\ZZ[\alpha])$, we are done. 
%C￿ʀ￿ʟʟ￿ʀʏ4.2.–The setCAhas cardinality5and is represented byχ0...,χ4of(4.1).We have an isomorphismMR0=￿4j=0PΓRj\(H4j−H)of real analytic orbifolds. For eachj,PΓRj\(H4j−H)is connected.
\end{proof}

\cleardoublepage
\part{Real algebraic cycles}
\cleardoublepage
% Note that depending on your settings in the table of contents, subsections and subsubsections might appear virtually identical.
% Make sure to set the ToC depth and the numbering depth in the ToC the way you want.
\chapter{Integral Fourier transforms}\label{ch:integralfourier}

\hfill This chapter is based on joint work with \textsc{Thorsten Beckmann}. 
\section{Introduction}

In the second part of this thesis, we focus on algebraic cycles on complex and real abelian varieties. A central role in this study -- which takes up Chapters \ref{ch:integralfourier}, \ref{ch:onecycles} and \ref{ch:realintegralhodge} -- is played by a certain correspondence. It has since long been known that for an abelian variety $A$ over a field $k$, with dual abelian variety $\wh A$, the Fourier transform
\begin{align} \label{intro-intro-Fourier}
\ca F_A \colon \CH(A)_{\QQ} \xrightarrow{\sim} \CH(\wh A)_{\QQ}
\end{align}
provides a powerful tool to study the $\QQ$-linear algebraic cycles of $A$. It is used to study the rational Chow ring of $A$, as well as the cycle class map to rational cohomology. Recently, Moonen and Polishchuk \cite{moonenpolishchuk} initiated the study of the integrality aspects of the Fourier transform (\ref{intro-intro-Fourier}). Indeed, it is natural to ask: 
\begin{question} \label{chow-interaction}
How does $\ca F_A$ interact with integral algebraic cycles?
%Chow ring $\CH(A)$? 
\end{question}
\noindent
The goal of the current Chapter \ref{ch:integralfourier} is to work on Question \ref{chow-interaction}, building on the results of Moonen--Polishchuk. The applications of \emph{loc.cit.} primarily concern the structure of the integral Chow rings themselves. We continue with their study, but we also address the compatibility of Fourier transforms with integral cycle class maps. Since Question \ref{chow-interaction} was phrased somewhat imprecisely, let us explain in some detail the steps that we take during our search for an answer: 
%The structure of the remaining three chapters will be as follows. 
%In this Chapter \ref{ch:integralfourier}, we develop on their study, but 
%he \emph{Fourier transform between the rational Chow group of an abelian variety and the rational Chow group of its dual 
%More precisely, we treat the following topics:
\begin{enumerate}
\item[Chapter 6:] We further develop the theory of \emph{integral Fourier transforms}, on Chow rings as well as on cohomology. The main result of this chapter (Theorem \ref{th:motivic}) will provide necessary and sufficient conditions for the Fourier transform (\ref{intro-intro-Fourier}) to lift to a motivic homomorphism between integral Chow groups. 
\item[Chapter 7:] We apply the theory of Chapter \ref{ch:integralfourier} to complex abelian varieties. The main outcome of this project is Theorem \ref{maintheorem}, which says that on a principally polarized complex abelian variety $A$ whose minimal cohomology class is algebraic, all integral Hodge classes of degree $2\dim(A) - 2$ are algebraic. 
\item[Chapter 8:] We apply the theory of Chapter \ref{ch:integralfourier} (which is developed for abelian varieties over an arbitrary field $k$) to the case of abelian varieties over $k = \RR$. The principal outcome is that modulo torsion, every real abelian threefold satisfies the real integral Hodge conjecture (Theorem \ref{theorem1}). Other applications of integral Fourier transforms include a detailed analysis of the Hochschild-Serre filtration on equivariant cohomology of a real abelian variety (Theorem \ref{fourierreduction}). 
\end{enumerate}
\noindent
Having lifted a veil of what to do with integral Fourier transforms, let us now make Question \ref{chow-interaction} more precise. Let $g$ be a positive integer and let $A$ be an abelian variety of dimension $g$ over a field $k$. \emph{Fourier transforms} are correspondences between the derived categories, rational Chow groups and cohomology of $A$ and $\wh A$ attached to the Poincar\'e bundle $\ca P_A$ on $A \times \wh A$ \cite{mukaiduality,beauvillefourier,huybrechtsfouriermukai}. On the level of cohomology, the Fourier transform preserves integral $\ell$-adic \'etale cohomology when $k = k_s$ (separable closure), and integral Betti cohomology when $k = \CC$. It is thus natural to ask whether the Fourier transform on rational Chow groups preserves integral cycles modulo torsion or, more generally, lifts to a homomorphism between integral Chow groups. This question was raised by Moonen--Polishchuk \cite{moonenpolishchuk} and Totaro \cite{totaroIHCthreefolds}. More precisely, Moonen and Polishchuk gave a counterexample for abelian varieties over non-closed fields and asked about the case of algebraically closed fields. 

The goal of Chapter \ref{ch:integralfourier} is to investigate this question further.

\section{Integral Fourier transforms} \label{intfourier-intone}

The main result of Section \ref{ch:integralfourier} gives necessary and sufficient conditions for the Fourier transform on rational Chow groups or cohomology to lift to a motivic homomorphism between integral Chow groups. To get there, we need a precise definition of "integral Fourier transform", which we introduce in this Section \ref{intfourier-intone}. 
%In the next Chapter ..., we will relate such lifts to the integral Hodge conjecture when $k = \CC$. In subsequent Chapter \ref{ch:onecycles} we will use the theory developed in this section to prove Theorem \ref{maintheorem}. 
%and also to the integral Tate conjecture when $k$ equals the separable closure of a finitely generated field.
%We shall see (see Corollary \ref{corollary:motivic}) that a homomorphism $\CH(A) \to \CH(\wh A)$ lifting the Fourier transform $\ca F_A \colon \CH(A)_{\QQ} \to \CH(\wh A)_{\QQ}$ exists provided that the one-cycle $c_1(\ca P_A)^{2g-1}/(2g-1)! \in \CH_1(A \times \wh A)_{\QQ}$ comes from a one-cycle in $\CH(A \times \wh A)$. 

\subsection{Notation and conventions} \label{sec:notation}

We let $k$ be a field with separable closure $k_s$ and $\ell$ a prime number different from the characteristic of $k$. For a smooth projective variety $X$ over $k$, we let $\CH(X)$ be the Chow group of $X$ and define $\CH(X)_\QQ = \CH(X) \otimes \QQ$, $\CH(X)_{\QQ_\ell} = \CH(X) \otimes {\QQ_\ell}$ and $\CH(X)_{\ZZ_\ell} = \CH(X) \otimes {\ZZ_\ell}$. We let $\rm H_{\textnormal{\'et}}^i(X_{k_s}, \ZZ_\ell(a))$ be the $i$-th degree \'etale cohomology group with coeffients in $\ZZ_\ell(a)$, $a \in \ZZ$.

Often, $A$ will denote an abelian variety of dimension $g$ over $k$, with dual abelian variety $\wh A$ and (normalized) Poincar\'e bundle on $\ca P_A$. The abelian group $\CH(A)$ will in that case be equipped with two ring structures: the usual intersection product $\cdotp$ as well as the Pontryagin product $\star$. Recall that the latter is defined as follows: 
\[\star \colon \CH(A) \times \CH(A) \to \CH(A), \quad x \star  y = m_{\ast} (\pi_1^{\ast}(x) \cdot \pi_2^{\ast}(y)). 
\]
Here, as well as in the rest of the paper, $\pi_i$ denotes the projection onto the $i$-th factor, $\Delta\colon A \to A \times A$ the diagonal morphism, and $m \colon A \times A \to A$ the group law morphism of $A$. 

For any abelian group $M$ and any element $x \in M$, we denote by $x_\QQ \in M \otimes_{\ZZ} \QQ$ the image of $x$ in $M \otimes_{\ZZ} \QQ$ under the canonical homomorphism $M \to M \otimes_\ZZ \QQ$.

\subsection{Integral Fourier transforms and integral Hodge classes}
\label{sec:intfourierintHodge}

For abelian varieties $A$ over $k = k_s$, it is unknown whether the Fourier transform $\ca F_A \colon \CH(A)_{\QQ} \to \CH(\wh A)_{\QQ}$ preserves the subgroups of integral cycles modulo torsion. A sufficient condition for this to hold is that $\ca F_A$ lifts to a homomorphism $\CH(A) \to \CH(\wh A)$. In this section we outline a second consequence of such a lift $\CH(A) \to \CH(\wh A)$ when $A$ is defined over the complex numbers: the existence of an integral lift of $\ca F_A$ implies the integral Hodge conjecture for one-cycles on $\wh A$. %Let us recall all these notions, and prove this implication. 
\\
\\
Let $A$ be an abelian variety over $k$. The Fourier transform on the level of Chow groups is the group homomorphism
\[
\ca F_A \colon \CH(A)_\QQ \to \CH(\wh A)_\QQ
\]
induced by the correspondence $\ch(\ca P_A) \in \CH(A \times \wh A)_\QQ$, where $\ch(\ca P_A)$ is the Chern character of $\ca P_A$. Similarly, %$cl\left(\ch(\ca P_A)\right) \in \rm H^\bullet( (A \times \wh A, \ZZ_\ell)$ 
one defines the Fourier transform on the level of \'etale cohomology: 
\[
\mr F_{A}\colon \rm H_{\textnormal{\'{e}t}}^\bullet(A_{k_s}, \QQ_\ell(\bullet)) \to \rm H_{\textnormal{\'{e}t}}^\bullet(\wh A_{k_s}, \QQ_\ell(\bullet)). 
\]
In fact, $\mr F_A$ preserves the integral cohomology classes and induces, for each integer $j$ with $0 \leq j \leq 2g$, an isomorphism \cite[Proposition 1]{beauvillefourier}, \cite[page 18]{totaroIHCthreefolds}:
\[
    \mr F_{A}\colon \rm H_{\textnormal{\'{e}t}}^j(A_{k_s}, \ZZ_\ell(a)) \to \rm H_{\textnormal{\'{e}t}}^{2g-j}(\wh A_{k_s}, \ZZ_\ell(a+g-j)).
\]
Similarly, if $k = \CC$, then $\ch(\ca P_A)$ induces, for each integer $i$ with $0 \leq i \leq 2g$, an isomorphism of Hodge structures 
\begin{equation} \label{eq:integralfouriercohomology}
    \mr F_A\colon \rm H^i(A, \ZZ) \to \rm H^{2g-i}(\wh A, \ZZ)(g-i).
\end{equation}
\noindent
In \cite{moonenpolishchuk}, Moonen and Polishchuk consider an isomorphism $\phi\colon A \xrightarrow{\sim} \wh A$, a positive integer $d$, and define the notion of motivic integral Fourier transform of $(A, \phi)$ up to factor $d$. The definition goes as follows. Let $\ca M(k)$ be the category of effective Chow motives over $k$ with respect to ungraded correspondences, and let $h(A)$ be the motive of $A$. Then a morphism $$\ca F \colon h(A) \to h(A)$$ in $\ca M(k)$ is a \textit{motivic integral Fourier transform of $(A, \phi)$ up to factor $d$} if the following three conditions are satisfied: (i) the induced morphism $h(A)_\QQ \to h(A)_\QQ$ is the composition of the usual Fourier transform with the isomorphism $\phi^\ast\colon h(\wh A)_\QQ \xrightarrow{\sim}h(A)_\QQ$, (ii) one has $d \cdot \ca F \circ \ca F = d \cdot (-1)^g \cdot [-1]_\ast$ as morphisms from $h(A)$ to $h(A)$, and (iii) $d\cdot \ca F \circ m_\ast = d \cdot \Delta^\ast \circ \ca F \otimes \ca F\colon h(A) \otimes h(A) \to h(A)$. 
%The reason that we call $\ca F$ a \textit{weak} integral Fourier transform is that Moonen and Polishchuk defined what it means to be a \textit{motivic integral Fourier transform up of $(A, \phi)$ to factor $d$} in the case $A$ that comes equipped with an isomorphism $\phi\colon A \xrightarrow{\sim} \wh A$, see Definition 3.4 in \cite{moonenpolishchuk}. Their notion is indeed stronger than ours: for a motivic integral Fourier transform $\ca F$ up to factor $d$, the induced morphism $d \cdot \ca F$ satisfies the usual properties $\ca F_A$.
%Let $X$ be an abelian variety over a field $k$ and $\phi$ an isomorphism $\phi\colon X \xrightarrow{\sim}\wh X$. Moonen and Polishchuk defined what it means to be a \textit{motivic integral Fourier transform up of $(A, \phi)$ to factor $d$}, see Definition 3.4 in \cite{moonenpolishchuk}. Their notion is indeed stronger than ours: for a motivic integral Fourier transform $\ca F$ up to factor $d$, the induced morphism $d \cdot \ca F$ satisfies the usual properties $\ca F_A$.

For our purposes, we consider similar homomorphisms $\CH(A) \to \CH(\wh A)$. To make their existence easier to verify (c.f.\ Theorem \ref{th:motivic}) we relax the above conditions:
 
\begin{definition} \label{def:weakintegralfourier}
Let $A_{/k}$ be an abelian variety and let $\ca F\colon \CH(A) \to \CH(\wh A)$ be a group homomorphism. We call $\ca F$ a \textit{weak integral Fourier transform} if the following diagram commutes:
\begin{equation}
\xymatrix{
\CH(A) \ar[d]\ar[r]^{\ca F} & \CH(\wh A) \ar[d] \\
\CH(A)_\QQ \ar[r]^{\ca F_A} & \CH(\wh A)_\QQ.
}
\end{equation}
We call a weak integral Fourier transform $\ca F$ \textit{motivic} if it is induced by a cycle $\Gamma$ in $\CH(A \times \wh A)$ that satisfies $\Gamma_\QQ = \ch(\ca P_A) \in \CH(A \times \wh A)_{\QQ}$. A group homomorphism $\ca F \colon \CH(A) \to \CH(\wh A)$ is an \textit{integral Fourier transform up to homology} if the following diagram commutes:
\begin{equation}\label{diagram:fouriercommutes}
\xymatrixcolsep{5pc}
    \xymatrix{
\CH(A) \ar[d]\ar[r]^{\ca F} & \CH(\wh A) \ar[d] \\
\oplus_{r \geq 0} \rm H_{\textnormal{\'{e}t}}^{2r} (A_{k_s}, \ZZ_\ell(r)) \ar[r]^{\mr F_{A}} & \oplus_{r \geq 0} \rm H_{\textnormal{\'{e}t}}^{2r} (\wh A_{k_s}, \ZZ_\ell(r)).
}
\end{equation}
Similarly, a $\ZZ_\ell$-module homomorphism $\ca F_\ell \colon \CH(A)_{\ZZ_\ell} \to \CH(\wh A)_{\ZZ_\ell}$ is called an \textit{$\ell$-adic integral Fourier transform up to homology} if $\ca F_\ell$ is compatible with $\mr F_A$ and the $\ell$-adic cycle class maps. % and $\mr F_A \colon \rm H_{\textnormal{\'{e}t}}^{2r} (A_{k_s}, \ZZ_\ell(r)) \to \rm H_{\textnormal{\'{e}t}}^{2r} (\wh A_{k_s}, \ZZ_\ell(r))$. 
%$cl\colon \CH(A)_{\ZZ_\ell} \to \rm H_{\textnormal{\'{e}t}}^{2r} (A_{k_s}, \ZZ_\ell(r))$ and 
%if $\ca F_\ell$ makes Diagram (\ref{diagram:fouriercommutes})$_\ell$ commute, where (\ref{diagram:fouriercommutes})$_\ell$ is the diagram obtained from (\ref{diagram:fouriercommutes}) by replacing integral Chow groups by Chow groups with $\ZZ_\ell$-coefficients and $\ca F$ by $\ca F_\ell$. 
Finally, an integral Fourier transform up to homology $\ca F$ (resp.\ an $\ell$-adic integral Fourier transform up to homology $\ca F_\ell$) is called \textit{motivic} if it is induced by a cycle $\Gamma \in \CH(A \times \wh A)$ (resp.\ $\Gamma_\ell \in \CH(A \times \wh A)_{\ZZ_\ell})$ such that $cl(\Gamma)$ (resp.\ $cl(\Gamma_\ell)$) equals $\ch(\ca P_A) \in \oplus_{r \geq 0}\rm H_{\textnormal{\'{e}t}}^{2r}((A \times \wh A)_{k_s}, \ZZ_\ell(r))$. 
\end{definition}

\begin{remark}
If $\ca F\colon \CH(A) \to \CH(\wh A)$ is a weak integral Fourier transform, then $\ca F$ is an integral Fourier transform up to homology, the $\ZZ_\ell$-module $\oplus_{r \geq 0}\rm H_{\textnormal{\'{e}t}}^{2r}(\wh A_{k_s}, \ZZ_\ell(r))$ being torsion-free. If $k = \CC$, then $\ca F\colon \CH(A) \to \CH(\wh A)$ is an integral Fourier transform up to homology if and only if $\ca F$ is compatible with the Fourier transform $\mr F_A\colon \rm H^\bullet(A, \ZZ) \to \rm H^\bullet(\wh A,\ZZ)$ on Betti cohomology. 
\end{remark} 
\noindent
The relation between integral Fourier transforms and Hodge classes is as follows:

\begin{lemma}  \label{lemma:trivial}
Let $A$ be a complex abelian variety and let $\ca F \colon \CH(A) \to \CH(\wh A)$ be an integral Fourier transform up to homology. 
\begin{enumerate}
\item 
For each $i \in \ZZ_{\geq 0}$, the integral Hodge conjecture for degree $2i$ classes on $A$ implies the integral Hodge conjecture for degree $2g-2i$ classes on $\wh A$. 
\item If $\ca F$ is motivic, then $\mr F_A$ induces a group isomorphism $\rm Z^{2i}(A) \xrightarrow{\sim} \rm Z^{2g-2i}(\wh A)$ and, therefore, the integral Hodge conjectures for degree $2i$ classes on $A$ and degree $2g-2i$ classes on $\wh A$ are equivalent for all $i$. 
\end{enumerate}
\end{lemma}

\begin{proof}
We can extend Diagram (\ref{diagram:fouriercommutes}) to the following commutative diagram:
\begin{equation*} 
    \xymatrix{
\CH^{i}(A) \ar[d]^{cl^i} \ar[r] & \CH(A) \ar[d]\ar[r]^{\ca F} & \CH(\wh A) \ar[d] \ar[r] & \CH_i(\wh A) \ar[d]^{cl_i} \\
\rm H^{2i}(A, \ZZ) \ar[r] &\rm H^\bullet(A, \ZZ) \ar[r]^{\mr F_A} & \rm H^\bullet(\wh A, \ZZ) \ar[r] & \rm H^{2g-2i}(\wh A, \ZZ). 
}
\end{equation*}
The composition $\rm H^{2i}(A, \ZZ)  \to \rm H^{2g-2i}(\wh A, \ZZ)$ appearing on the bottom line agrees up to a suitable Tate twist with the map $\mr F_A$ of equation (\ref{eq:integralfouriercohomology}). %, hence %becomes an isomorphism of Hodge structures . We 
Therefore, we obtain a commutative diagram:
\begin{equation} \label{diagram:cycleclassmaps}
\xymatrixcolsep{5pc}
\xymatrix{
\CH^i(A) \ar[d]^{cl^i}\ar[r] & \CH_i(\wh A) \ar[d]^{cl_i} \\
\Hdg^{2i}(A, \ZZ) \ar[r]^{\sim} & \Hdg^{2g-2i}(\wh A, \ZZ).
}
\end{equation}
Thus the surjectivity of $cl^i$ implies the surjectivity of $cl_i$. Moreover, if $\ca F$ is motivic, then replacing $A$ by $\wh A$ and $\wh A$ by $\skew{5.5}\widehat{\widehat{A}}$ in the argument above shows that the images of $cl^i$ and $cl_i$ are identified under the isomorphism $\mr F_A\colon \Hdg^{2i}(A, \ZZ) \xrightarrow{\sim} \Hdg^{2g-2i}(\wh A, \ZZ)$ in diagram (\ref{diagram:cycleclassmaps}). 
%\footnote{We may unfortunately not reverse roles: the existence of the integral Fourier transform for $A$ does not imply the existence of the integral Fourier transform for $\wh A$.} %This proves one implication. Since $A$ is isomorphic to the dual of $\wh A$, we can reverse roles to obtain the assertion.
\end{proof}

\section{Rational Fourier transforms}\label{sec:propertiesfourier}

The above suggests that to prove Theorem \ref{maintheorem}, one would need to show that for a complex abelian variety of dimension $g$ whose minimal Poincar\'e class $c_1(\ca P_A)^{2g-1}/(2g-1)! \in \rm H^{4g-2}(A \times \wh A, \ZZ)$ is algebraic, all classes of the form $c_1(\ca P_A)^{i}/i! \in \rm H^{2i}(A \times \wh A, \ZZ)$ are algebraic. With this goal in mind we shall study Fourier transforms on rational Chow groups in Section \ref{sec:propertiesfourier}, and investigate how these relate to $\ch(\ca P_A) \in \CH(A \times \wh A)_{\QQ}$. In turns out that the cycles $c_1(\ca P_A)^{i}/i!$ in $ \CH(A \times \wh A)_{\QQ}$ satisfy several relations that are very similar to those proved by Beauville for the cycles $\theta^{i}/i! \in \CH(A)_{\QQ}$ in case $A$ is principally polarized, see \cite{beauvillefourier}. Since we will need these results in any characteristic in order to prove Theorem \ref{introth:integraltate}, we will work over our general field $k$, see Section \ref{sec:notation}. 
%the images of the cycles $c_1(\ca P_A)^{i}/i! \in \CH^i(A \times \wh A)_{\QQ}$ are under the Fourier transform $\ca F_{A \times \wh A} \colon \CH(A)_{\QQ} \to \CH(\wh A)_{\QQ}$.
\\
\\
Let $A$ be an abelian variety over $k$. Define cycles 
 \begin{align*}
\ell &= c_1(\ca P_A) \in \CH^1(A \times \wh A)_{\QQ}, \\
 \wh \ell &= c_1(\ca P_{\wh A}) \in \CH^1(\wh A \times A)_{\QQ}, \\
 \mr R_A  &= c_1(\ca P_A)^{2g-1}/(2g-1)! \in \CH_1(A\times \wh A)_{\QQ}, \quad \tn{ and } \\
 \mr R_{\wh A} &= c_1(\ca P_{\wh A})^{2g-1}/(2g-1)! \in \CH_1(\wh A\times A)_{\QQ}.
 \end{align*} For $a \in \CH(A)_{\QQ}$, define $\mathrm{E}(a) \in \CH(A)_{\QQ}$ as the $\star$-exponential of $a$:
\[
\mathrm{E}(a)  \coloneqq \sum_{n \geq 0} \frac{a^{\star n}}{n!} \in \CH(A)_{\QQ}. 
\]
The key to Theorem \ref{maintheorem} will be the following:

\begin{lemma} \label{lemma:crucialprop} 
We have $\ch(\ca P_A) = e^\ell = (-1)^g\cdot \rm E((-1)^g\cdot \mr R_A) \in \CH( A \times \wh A)_{\QQ}$. 
%In fact, for any $p, q \in \ZZ_{\geq 0}$ such that $p + q = 2g$, one has $\ell^p/p! = \mr R_A^{\star q}/q! \in \CH^p(A \times \wh A)_\QQ$. 
\end{lemma}

\begin{proof} The most important ingredient in the proof is the following:

\hypertarget{claim1}{\textit{Claim $(\ast)$}}: Consider the Fourier transform $\ca F_{A \times \wh A} \colon \CH(A \times \wh A)_{\QQ} \to \CH(\wh A \times A)_{\QQ}$. One has \[
\ca F_{A \times \wh A}(e^\ell) = (-1)^g \cdot e^{-\wh \ell} \in \CH(\wh A \times A)_{\QQ}.\]
\noindent
To prove Claim (\hyperlink{claim1}{$\ast$}), we lift the desired equality in the rational Chow group of $\wh A \times A$ to an isomorphism in the derived category $\rm D^b(\wh A \times A)$ of $\wh A \times A$. For $X= A \times \wh A$ the Poincar\'e line bundle $\mathcal{P}_X$ on $X \times \wh X \cong A \times \wh A \times \wh A \times A$ is isomorphic to $\pi_{13}^\ast \mathcal{P}_A \otimes \pi_{24}^\ast \mathcal{P}_{\wh A}$. Consider
\begin{equation}
\label{eq:composition_fourier}
\Phi_{{\ca P}_{X}}(\ca P_A) \cong \pi_{34,\ast} \left(\pi_{13}^\ast\ca P_A \otimes \pi_{24}^\ast \ca P_{\wh A} \otimes \pi_{12}^\ast\ca P_A  \right) \in \rm D^b(\wh A\times A)
\end{equation}
whose Chern character is exactly $\mathcal{F}_X(e^\ell)$. Applying the pushforward along the permutation map
\[
(123) \colon A \times \wh A \times \wh A \times A \cong \wh A \times A \times \wh A \times A
\]
the object \eqref{eq:composition_fourier} becomes $\pi_{14,\ast}\left( \pi_{12}^\ast\mathcal{P}_{\wh A} \otimes \pi_{23}^\ast \mathcal{P}_A \otimes \pi_{34}^\ast \mathcal{P}_{\wh A} \right)$ which is isomorphic to the Fourier--Mukai kernel of the composition
\[
\Phi_{\mathcal{P}_{\wh A}} \circ \Phi_{\mathcal{P}_A} \circ \Phi_{\mathcal{P}_{\wh A}}.
\]
Since $\Phi_{\mathcal{P}_{A}}\circ \Phi_{\mathcal{P}_{\wh A}}$ is isomorphic to $[-1_{\wh A}]^\ast \circ [-g]$ by \cite[Theorem 2.2]{mukaiduality}, we have
\[
\Phi_{\mathcal{P}_{\wh A}} \circ \Phi_{\mathcal{P}_A} \circ \Phi_{\mathcal{P}_{\wh A}}\cong  \Phi_{\mathcal{P}_{\wh A}} \circ [-1_{\wh A}]^\ast \circ [-g].
\]
This is the Fourier--Mukai transform with kernel $\ca E = \mathcal{P}_{\wh A}^\vee[-g] \in \rm D^b(\wh A \times A)$. By uniqueness of the Fourier--Mukai kernel of an equivalence \cite[Theorem 2.2]{orlovfouriermukaikernel} and the fact that the Chern character of $\ca E$ equals $(-1)^{g}\cdot e^{-\wh \ell} \in \CH(\wh A \times A)_{\QQ}$, this finishes the proof of Claim (\hyperlink{claim1}{$\ast$}).

Next, we claim that $(-1)^g \cdot \ca F_{\wh A \times A}(-\wh \ell)= \mr R_A$. 
%\[
%\ca F_{A \times \wh A}(\ell) = - \wh \mr R_A, \quad \ca F_{A \times \wh  A}\left(\mr R_A\right) = -\wh \ell, \quad \ca F_{A \times \wh A}(-l) = \wh{\mr R_A}, \quad \ca F_{A \times \wh A}\left(-\mr R_A\right) = \wh \ell. 
%\]
To see this, recall that for each integer $i$ with $0 \leq i \leq g$, there is a canonical \textit{Beauville decomposition} $$\CH^i(A)_\QQ = \oplus_{j = i-g}^i\CH^{i,j}(A)_{\QQ}, \quad \quad \quad  \tn{see \cite{beauvilledecomposition}}.$$ Since the Poincar\'e bundle $\ca P_A$ is symmetric, we have $\ell \in \CH^{1,0}(A \times \wh A)_{\QQ}$ and hence $\ell^{i} \in \CH^{i,0}(A \times \wh A)_{\QQ}$. In particular, we have $\mr R_A \in \CH^{2g-1,0}(A \times \wh A)_\QQ$. 
%The maps 
%\[
%\CH^{1,0}(A \times \wh A)_\QQ \to \rm H^2(A \times \wh A, \QQ), \quad 
%\CH^{2g-1,0}(A \times \wh A)_\QQ \to \rm H^{4g-2}(A \times \wh A, \QQ)
%\]are injective, hence it suffices to prove Equalities \ref{item:fourierpoincareone}, \ref{item:fourierpoincaretwo} and \ref{item:landp} for the corresponding cycle classes in the cohomology rings $\rm H^\bullet (A \times \wh A, \QQ)$ and $\rm H^\bullet(\wh A \times A, \QQ)$. %it suffices to show that they agree in $\rm H^2(A \times \wh A, \QQ)$. 
The fact that $\ca P_{A}$ is symmetric also implies - via Claim (\hyperlink{claim1}{$\ast$}) - that we have
$
\ca F_{\wh A \times A}((-1)^g\cdot e^{-\wh \ell}) = e^{\ell}$. Indeed, 
\[
\ca F_{\wh A \times A} \circ \ca F_{A \times \wh A} = [-1]^\ast \cdot (-1)^{2g} = [-1]^\ast,\] see \cite[Corollary 2.22]{deningermurre}. Since $\ca F_{\wh A \times A}$ identifies the group $\CH^{i,0}(\wh A \times A)_{\QQ}$ with the group $\CH^{g-i,0}(A \times \wh A)$ (see \cite[Lemma 2.18]{deningermurre}), we must indeed have 
\begin{equation} \label{importantequation}
(-1)^g \cdot \ca F_{\wh A \times A}(-\wh \ell)= \ca F_{\wh A \times A}((-1)^{g+1}\cdot \wh \ell) =  \frac{{\ell}^{2g-1}}{(2g-1)!}  = \mr R_A. %\quad \textnormal{and} \quad \ca F_{A \times \wh A}((-1)^{2g-1} \mr R_A) = \ca F_{A \times \wh A}\left( (-1)^{2g-1} \frac{{\ell}^{2g-1}}{(2g-1)!} \right)  = \wh \ell. 
\end{equation}
%The right hand side gives
%\begin{equation*} 
%    -\mr R_A = (-1)^{2g-1}\mr R_A  = (-1)^{2g}[-1]^\ast (-1)^{2g-1}\mr R_A = \ca F_{\wh A \times A}(\wh \ell),
%\end{equation*}
%which implies that the left hand side gives
%\[
%-l =  (-1)^{2g}[-1]^{\ast} \left( -l \right) = 
%\frac{\left(\ca F_{\wh A \times A}(\wh \ell)\right)^{\star(2g-1)}}{(2g-1)!} 
%= \frac{\left(-\mr R_A\right)^{\star(2g-1)}}{(2g-1)!} = -1 \cdot \frac{\mr R_A^{\star(2g-1)}}{(2g-1)!}.
%\]
For a $g$-dimensional abelian variety $X$ and any $x,y \in \CH(X)_{\QQ}$, one has $\ca F_{X}(x \cdot y) = (-1)^g\cdot \ca F_{X}(x) \star \ca F_{X}(y)\in \CH(\wh X)_{\QQ}$, see \cite[Proposition 3]{beauvillefourier}. This implies (see also \cite[\S 3.7]{moonenpolishchuk}) that if $a$ is a cycle on $X$ such that $\ca F_{X}(a) \in \CH_{>0}(\wh X)_{\QQ}$, then $\ca F_{X}(e^a) =
%\sum_{i \geq 0}\ca F_{X}(a^i/i!) = \sum_{i \geq 0}\ca F_{X}(a)^{\star i}/i! = 
(-1)^g\cdot \rm E((-1)^g\cdot\ca F_{X}(a))$. This allows us to conclude that
\begin{align*}
e^\ell &= \ca F_{\wh A \times A}((-1)^{g} \cdot e^{-\wh \ell})=
(-1)^{g} \cdot\ca F_{\wh A \times A}( e^{-\wh \ell})
\\
&= (-1)^{g} \cdot\rm E(\ca F_{\wh A \times A}(-\wh \ell)) = (-1)^g\cdot \rm E((-1)^g\cdot \mr R_A),
\end{align*}
which finishes the proof. 
\end{proof}

\noindent
Next, assume that $A$ is equipped with a \textit{principal} polarization $\lambda \colon A \xrightarrow{\sim} \wh A$, define $\ell = c_1(\ca P_A)$, and let 
\begin{equation} \label{generaltheta}
    \Theta = \frac{1}{2}\cdot (\id, \lambda)^\ast \ell \in \CH^1(A)_\QQ
\end{equation}
be the symmetric ample class corresponding to the polarization. Here $(\id,\lambda)$ is the morphism 
$(\id,\lambda)\colon A \to A\times \wh A$. One can understand the relation between 
\[
\Gamma_\Theta \coloneqq \Theta^{g-1}/(g-1)! \in \CH_1(A)_{\QQ}
\]
and $\mr R_A = {\ell}^{2g-1}/(2g-1)! \in \CH_1(A \times \wh A)_{\QQ}$ in the following way. %Recall that $\pi_1\colon A \times \wh A \to A$ and $\pi_2\colon A \times \wh A \to \wh A$ were defined to be the two projections, and 
Define 
\begin{align*}
j_1 &\colon A \to A \times \wh A, \quad x \mapsto (x,0), \quad \tn{ and }  \\
j_{2} &\colon \wh A \to A \times \wh A, \quad y \mapsto (0,y).
\end{align*}
%Consider the morphism $(\id, \lambda)\colon A \to A \times \wh A$ be the morphism $x \mapsto (x, \lambda(x))$, 
Let $\wh \Theta \in \CH^1(\wh A)_\QQ$ be the dual of $\Theta$, and define a one-cycle $\tau$ on $A \times \wh A$ as follows:
\begin{align*} %\label{eq:tau}
\tau \coloneqq 
j_{1, \ast} (\Gamma_\Theta)  
+  j_{2, \ast}(\Gamma_{\wh \Theta}) - (\id, \lambda)_\ast (\Gamma_\Theta) 
\in \CH_1( A \times \wh A )_\QQ.
\end{align*}

\begin{lemma} \label{lemma:minimalclasspoincarecomparison}
%where we used that $\ca F_A(\theta) = (-1)^{g-1}/ \theta^{g-1}(g-1)! \in \CH_1(A)_\QQ$. 
%Define  = c_1(\ca P_A)$ and $\mr R_A = c_1(\ca P_A)/(2g-1)! \in \CH(A \times \wh A)_{\QQ}$. 
One has 
$
%    \ca F_{A \times A}(c_1(\ca P_A)) = (-1)\cdot\frac{c_1(\ca P_A)^{2g-1}}{(2g-1)!} = 
 \tau = (-1)^{g+1}\cdot \mr R_A \in \CH_1(A \times \wh A)_{\QQ}$. 
\end{lemma}

\begin{proof}

% To see this, note that $\ca F_A(e^\theta) = \ca F_A(e^{-\theta})$ by \cite[Lemme 1]{beauvillefourier} and that $\theta^{i}/i! \in \CH^{i,0}(A)_\QQ$ by \cite[Theorem 13.7]{geeredixmoon}. Therefore $\ca F_A\left(\theta^{i}/i!\right) \in \CH^{g-i,0}(A)_\QQ$ by \cite[Proposition 13.35]{geeredixmoon} hence indeed, $\ca F_A(\theta) = (-1)^g\cdot \theta^{g-1}/(g-1)!$. 

Identify $A$ and $\wh A$ via $\lambda$. This gives $\ell = m^\ast(\Theta) - \pi_1^\ast(\Theta) - \pi_2^\ast(\Theta)$, and the Fourier transform becomes an endomorphism $\ca F_A \colon \CH(A)_\QQ \to \CH(A)_\QQ$. 

We claim that \[
\tau = (-1)^g \cdot \left(  \Delta_\ast \ca F_A(\Theta) - j_{1, \ast} \ca F_A(\Theta) -  j_{2, \ast} \ca F_A(\Theta) \right).\] For this, it suffices to show that $\ca F_A(\Theta) = (-1)^{g-1}\cdot \Theta^{g-1}/(g-1)! \in \CH_1(A)_\QQ$. Now $\ca F_A(e^\Theta) = e^{-\Theta}$ by Lemma \ref{beauvillelemma} below. Moreover, since $\Theta$ is symmetric, we have $\Theta \in \CH^{1,0}(A)_{\QQ}$, hence $\Theta^{i}/i! \in \CH^{i,0}(A)_\QQ$ for each $i \geq 0$. Therefore, $\ca F_A\left(\Theta^{i}/i!\right) \in \CH^{g-i,0}(A)_\QQ$ by \cite[Lemma 2.18]{deningermurre}. This implies that, in fact, 
\[
\ca F_A\left(\Theta^i/i!\right) =(-1)^{g-i}\cdot \Theta^{g-i}/(g-i)! \in \CH^{g-i,0}(A)_\QQ\] for every $i$. In particular, the claim follows.

Next, recall that $\ca F_{A\times A}(\ell) = (-1)^{g+1}\cdot \mr R_A$, see Claim (\hyperlink{claim1}{$\ast$}). So at this point, it suffices to prove the identity $$\ca F_{A\times A}(\ell) = (-1)^g \cdot \left(  \Delta_\ast \ca F_A(\Theta) - j_{1, \ast} \ca F_A(\Theta) -  j_{2, \ast} \ca F_A(\Theta) \right).$$ To prove this, we use the following functoriality properties of the Fourier transform on the level of rational Chow groups. Let $X$ and $Y$ be abelian varieties and let $f\colon X \to Y$ be a homomorphism with dual homomorphism $\wh f\colon \wh Y \to \wh X$. We then have the following equalities \cite[(3.7.1)]{moonenpolishchuk}: 
\begin{equation}\label{eq:functoriality}
    (\wh f)^\ast \circ \ca F_X = \ca F_Y \circ f_\ast, \hspace{3mm}  \ca F_X \circ f^\ast = (-1)^{\dim X - \dim Y}\cdot (\wh f)_\ast \circ \ca F_Y. 
\end{equation}
Since $\ell = m^\ast \Theta - \pi_1^\ast \Theta - \pi_2^\ast \Theta$, it follows from Equation (\ref{eq:functoriality}) that
\begin{align*}
    \ca F_{A \times A}(\ell) &=  \ca F_{A \times A}\left( m^\ast \Theta\right) - \ca F_{A \times A}\left(\pi_1^\ast \Theta\right) - \ca F_{A \times A}\left(\pi_2^\ast \Theta\right) \\
 &= (-1)^g\cdot \left(\Delta_\ast \ca F_A(\Theta) - j_{1, \ast} \ca F_A(\Theta) -  j_{2, \ast} \ca F_A(\Theta) \right).
\end{align*}
\end{proof}

\begin{lemma}[Beauville] \label{beauvillelemma}
Let $A$ be an abelian variety over $k$, principally polarized by $\lambda \colon A \xrightarrow{\sim} \wh A$, and define $\Theta = \frac{1}{2}\cdot (\id, \lambda)^\ast c_1(\ca P_A) \in \CH^1(A)_{\QQ}$. Identify $A$ and $\wh A$ via $\lambda$. With respect to the Fourier transform \[
\ca F_A\colon \CH(A)_{\QQ} \xrightarrow{\sim} \CH(A)_{\QQ}, \quad \text{ one has }\quad \ca F_A(e^\Theta) = e^{-\Theta}.
\]
\end{lemma}

\begin{proof}
Our proof follows the proof of \cite[Lemme 1]{beauvillefourier}, but has to be adapted, since $\Theta$ does not necessarily come from a symmetric ample line bundle on $A$. Since one still has $\ell = m^\ast \Theta - \pi_1^\ast \Theta - \pi_2^\ast \Theta$, the argument can be made to work: one has  
\begin{align*}
\ca F_A(e^\Theta) &= 
\pi_{2,\ast}\left(
e^\ell \cdot \pi_1^\ast e^{\Theta} 
\right)\\
&=
\pi_{2,\ast}\left(e^{m^\ast \Theta - \pi_2^\ast\Theta} \right) = e^{-\Theta} \pi_{2,\ast}(m^\ast e^{\Theta}) \in \CH(A)_{\QQ}.
\end{align*}
For codimension reasons, one has $$\pi_{2,\ast}(m^\ast e^\Theta) = \pi_{2,\ast}m^\ast(\Theta^g/g!)= \deg(\Theta^g/g!) \in \CH^0(A)_{\QQ} \cong \QQ.$$ Pull back $\Theta^g/g!$ along $A_{k_s} \to A$ to see that $$\deg(\Theta^g/g!) = 1 \in \CH^0(A)_{\QQ} \cong \CH^0(A_{k_s})_{\QQ},$$ since over $k_s$ the cycle $\Theta$ becomes the cycle class attached to a symmetric ample line bundle.
\end{proof}

\section{Divided powers of algebraic cycles}

It was asked by Bruno Kahn whether there exists a PD-structure on the Chow ring of an abelian variety over any field with respect to its usual (intersection) product. There are counterexamples over non-closed fields: see \cite{esnaultelementarytheorems}, where Esnault constructs an abelian surface $X$ and a line bundle $\ca L$ on $X$ such that $c_1(\ca L) \cdot c_1(\ca L)$ is not divisible by $2$ in $\CH_0(X)$. However, the case of algebraically closed fields remains open \cite[Section 3.2]{moonenpolishchuk}. What we do know, is the following:
%On the other hand, Moonen and Polishchuk do show that there exists a PD-structure on $\CH(A)_{>0} \subset \CH(A)$ with respect to the Pontryagin product (see \cite[Corollary 1.7]{moonenpolishchuk} or Theorem \ref{th:PDstructure} above). 

\begin{theorem}[Moonen--Polishchuk] \label{th:PDstructure}
Let $A$ be an abelian variety over $k$. The ring $\left(\CH(A), \star \right)$ admits a canonical PD-structure $\gamma$ on the ideal $\CH_{>0}(A) \subset \CH(A)$. If $k = \bar k$, then $\gamma$ extends to a PD-structure on the ideal generated by $\CH_{>0}(A)$ and the zero cycles of degree zero. % $I \subset \CH(A)$ with respect to $\star$. $\hfill \qed$
\end{theorem}
\noindent
In particular, for each element $x \in \CH_{>0}(A)$ and each $n \in \ZZ_{\geq 1}$, there is a canonical element $x^{[n]} \in \CH_{>0}(A)$ such that $n!x^{[n]} = x^{\star n}$, see \cite[\href{https://stacks.math.columbia.edu/tag/07GM}{Tag 07GM}]{stacks-project}. For $x \in \CH_{>0}(A)$, we may then define 
\[
\mathrm{E}(x) = \sum_{n \geq 0} x^{[n]} \in \CH(A)
\]
 as the $\star$-exponential of $x$ in terms of its divided powers. %where the $a^{[n]}$ are the divided powers of $a$ as in Theorem \ref{th:PDstructure}.
 
Together with the results of Section \ref{sec:propertiesfourier}, Theorem \ref{th:PDstructure} enables us to provide criteria for the existence of a motivic weak integral Fourier transform. Recall that for an abelian variety $A$ over $k$, principally polarized by $\lambda \colon A \xrightarrow{\sim} \wh A$, we defined $\Theta \in \CH^1(A)_\QQ$ as the symmetric ample class attached to $\lambda$, see equation (\ref{generaltheta}). 
%For an abelian variety $A$ over $k$, let $\rm R_s(A) \subset \CH(A)$ be the subring (for the intersection product) generated by $\CH^0(A)$ and the group of isomorphism classes of symmetric line bundles $\Pic^{\textnormal{symm}}(A) \subset \Pic(A) = \CH^1(A)$, and let $\rm R_s^{>0}(A) \subset \rm R_s(A)$ be the ideal of cycles in $\rm R_s(A)$ of positive codimension.

\begin{theorem} \label{th:motivic}
Let $A_{/k}$ be an abelian variety of dimension $g$. The following are equivalent:
\begin{enumerate}
    \item \label{motivicone} The one-cycle $\mr R_A = c_1(\ca P_A)^{2g-1}/(2g-1)! \in \CH(A \times \wh A)_{\QQ}$ lifts to $\CH_1(A \times \wh A)$. %a one-cycle in $\CH(A \times \wh A)$. %In other words: the cycle $c_1(\ca P_A)^{2g-1} \in \CH_1(A \times \wh A)$ is divisible by $(2g-1)!$ in $\CH_1(A \times \wh A)$. 
    \item \label{motivictwo} The abelian variety $A$ admits a motivic weak integral Fourier transform. 
    \item \label{motivicthree} The abelian variety $A \times \wh A$ admits a motivic weak integral Fourier transform. 
\end{enumerate}
%Any of these statements implies that 
%classes of symmetric line bundles $\Pic^{\textnormal{symm}}(A) \subset \Pic(A) \subset \CH(A)$ admit divided powers in $\CH(A)$, in the sense that for each $\alpha \in $
%there exists a PD-structure on the ideal $\CH^{>0}(A)/(\textnormal{torsion}) \subset \CH(A)/(\textnormal{torsion})$ (for the intersection product). %$\rm R_s^{>0}(A) \subset \rm R(A)$. %on the ideal of positive codimensional cycles in the subring of $\left(\CH(A), \cdotp\right)$ generated by $\CH^0(A)$ and $\Pic^{\textnormal{sym}}(A)$. 
Moreover, if $A$ carries a symmetric ample line bundle that induces a principal polarization $\lambda \colon A \xrightarrow{\sim} \wh A$, then 
%if $A$ is principally polarized by $\lambda\colon A \xrightarrow{\sim} \wh A$, and if $\Theta = \frac{1}{2}\cdot (\id, \lambda)^\ast c_1(\ca P_A) \in \CH^1(A)_{\QQ}$, then 
the above statements are equivalent to the following equivalent statements:
%\ref{motiviccorollaryitem:one} is implied by 
\begin{enumerate}
\setcounter{enumi}{3}
    \item \label{motiviczero} The two-cycle $\mr S_A = c_1(\ca P_A)^{2g-2}/(2g-2)! \in \CH(A \times \wh A)_{\QQ}$ lifts to $\CH_2(A \times \wh A)$.
    \item \label{motivicfour} The one-cycle $\Gamma_\Theta  = \Theta^{g-1}/(g-1)! \in \CH(A)_\QQ$ lifts to a one-cycle in $\CH(A)$. 
%\end{enumerate}
%Finally, if the principal polarization $\lambda\colon A \xrightarrow{\sim} \wh A$ is induced by a symmetric ample line bundle on $A$, then any of the above statements $\ref{motivicone}- \ref{motivicfour}$ is equivalent to any one of the following statements:
%\begin{enumerate}
%\setcounter{enumi}{5}
    \item \label{motivicfive} The abelian variety $A$ admits a weak integral Fourier transform. 
    \item \label{motivicsix} The Fourier transform $\ca F_A$ % \colon \CH(A)_{\QQ} \to \CH(\wh A)_{\QQ}$
    satisfies $\ca F_A\left( \CH(A)/(\textnormal{torsion})\right) \subset \CH(\wh A)/(\textnormal{torsion})$. 
%    \item \label{motivicseven} The Fourier transform $\ca F_A$ % \colon \CH(A)_{\QQ} \to \CH(\wh A)_{\QQ}$
%    satisfies $\ca F_A\left( \CH(A)/(\textnormal{torsion})\right) = \CH(\wh A)/(\textnormal{torsion})$. 
%    \item \label{motiviccorollaryitem:onefive} The Fourier transform $\ca F_A$ % \colon \CH(A)_{\QQ} \to \CH(\wh A)_{\QQ}$
%    satisfies $\ca F_A\left( \Pic^{\textnormal{sym}}(A)\right) \subset \CH(\wh A)/(\textnormal{torsion})$. 
    \item \label{motiviceight} There exists a PD-structure on the ideal $\CH^{>0}(A)/(\textnormal{torsion}) \subset \CH(A)/(\textnormal{torsion})$. %of positive codimensional cycles in the subring of $\left(\CH(A), \cdotp\right)$ generated by $\CH^0(A)$ and $\Pic^{\textnormal{sym}}(A)$. 
\end{enumerate}
\end{theorem}

\begin{proof}
%We first prove the equivalence of \ref{motiviczero} and \ref{motivicone}. Define $\ell = c_1(\ca P_A) \in \CH^1(A \times \wh A)$ and $X = A \times \wh A$ as before and observe that 
%\[
%\mr S_A = \ell^{2g-2}/(2g-2)! = \ca F_{\wh X}\ca F_X\left(\ell^{2g-2}/(2g-2)!\right)
%\]
Suppose that \ref{motivicone} holds, and let $\Gamma \in \CH_1(A \times \wh A)$ be a cycle such that $\Gamma_\QQ = \mr R_A$. Then consider the cycle $(-1)^g\cdot \rm E((-1)^g\cdot\Gamma) \in \CH(A \times \wh A)$. By Lemma \ref{lemma:crucialprop}, we have
\begin{align*}
&(-1)^g\cdot \rm E((-1)^g\cdot \Gamma)_{\QQ} = (-1)^g\cdot \rm E((-1)^g\cdot \Gamma_{\QQ}) \\
&= (-1)^g\cdot \rm E((-1)^g\cdot \mr R_A) = \ch(\ca P_A) \in \CH(A \times \wh A)_{\QQ}.
\end{align*} Thus \ref{motivictwo} holds. We claim that \ref{motivicthree} holds as well. Indeed, consider the line bundle $\ca P_{A \times \wh A}$ on the abelian variety $X= A \times \wh A \times \wh A \times A$; one has that $\ca P_{A \times \wh A} \cong \pi_{13}^\ast\ca P_A \otimes \pi_{24}^\ast \ca P_{\wh A}$, which implies that we have the following equality in $\CH_1(X)_\QQ$:
{\footnotesize \begin{equation}\label{eq1}
\begin{split}
\mr R_{A \times \wh A} &= \frac{\left( 
\pi_{13}^\ast c_1(\ca P_A) + \pi_{24}^\ast c_1(\ca P_{\wh A})
\right)^{4g-1} }{(4g-1)!}
\\
&= 
\frac{
\pi_{13}^\ast c_1(\ca P_A)^{2g-1} \cdot \pi_{24}^\ast c_1(\ca P_{\wh A})^{2g}
+
\pi_{13}^\ast c_1(\ca P_A)^{2g} \cdot \pi_{24}^\ast c_1(\ca P_{\wh A})^{2g-1}
 }{(2g)!(2g-1)!}
\\
&= 
\frac{
\pi_{13}^\ast c_1(\ca P_A)^{2g-1} \cdot \pi_{24}^\ast ((2g)! \cdot [0]_{A \times \wh A})
+
\pi_{13}^\ast (
(2g)! \cdot [0]_{\wh A \times A}
)
\cdot \pi_{24}^\ast c_1(\ca P_{\wh A})^{2g-1}
}{(2g)!(2g-1)!}
\\
&= 
\pi_{13}^\ast ( \frac{c_1(\ca P_A)^{2g-1}}{(2g-1)!} )\cdot \pi_{24}^\ast ([0]_{A \times \wh A})
+
\pi_{13}^\ast (
[0]_{\wh A \times A}
)
\cdot \pi_{24}^\ast ( \frac{c_1(\ca P_{\wh A})^{2g-1}}{(2g-1)!} ). 
 \end{split}
\end{equation}}
\noindent
We conclude that $\mr R_{A \times \wh A}$ lifts to $\CH_1(X)$ which, by the implication $[\ref{motivicone} \implies \ref{motivictwo}]$ (that has already been proved), implies that $A \times \wh A$ admits a motivic weak integral Fourier transform. On the other hand, the implication $[\ref{motivicthree}\implies\ref{motivicone}]$ follows from the fact that $(-1)^g \cdot \ca F_{\wh A \times A}(- \wh \ell) = \mr R_A$ (see Equation (\ref{importantequation})) and the fact that an abelian variety admits a motivic weak integral Fourier transform if and only if its dual abelian variety does. Therefore, we have $[\ref{motivicone} \iff \ref{motivictwo} \iff \ref{motivicthree}]$. 

From now on, assume that $A$ is principally polarized by $\lambda \colon A \xrightarrow{\sim} A$, where $\lambda$ is the polarization attached to a symmetric ample line bundle $\ca L$ on $A$. Moreover, in what follows we shall identify $\wh A$ and $A$ via $\lambda$. 

Suppose that \ref{motiviczero} holds and let $S_{ A} \in \CH_2(A\times  A) = \CH^{2g-2}(A \times A)$ be such that $(S_{A})_\QQ = \mr S_{A} \in \CH_2(A \times A)_\QQ$. Define $\CH^{1,0}(A) \coloneqq \Pic^{\textnormal{sym}}(A)$ to be the group of isomorphism classes of symmetric line bundes on $A$. Then $S_{A}$ induces a homomorphism $\ca F \colon \CH^{1,0}(A) \to \CH_1( A)$ defined as the composition 
\begin{align*}
\CH^{1,0}(A) \xrightarrow{\pi_1^\ast} \CH^1(A \times  A) \xrightarrow{\cdot S_{A}} &\CH^{2g-1}(A \times A)  \\ & =\CH_1(A \times A) \xrightarrow{\pi_{2,\ast}} \CH_1(A).
\end{align*}
Since $\ca F_A \left(\CH^{1,0}(A)_\QQ\right) \subset \CH_1( A)_\QQ$ (see \cite[Lemma 2.18]{deningermurre}) we see that 
\begin{equation}
\label{commmuttttivitiy}
\xymatrix{
\CH^{1,0}(A)\ar[d] \ar[r]^{\ca F} & \CH_1( A)\ar[d] \\
\CH^{1,0}(A)_\QQ \ar[r]^{\ca F_A} & \CH_1( A)_{\QQ}
}
\end{equation}
commutes. Since the line bundle $\ca L$ is symmetric, we have 
\begin{align}\label{liftoftheta}
\begin{split}
\Theta &= \frac{1}{2}\cdot (\id, \lambda)^\ast c_1(\ca P_A) = \frac{1}{2}\cdot  c_1\left((\id, \lambda)^\ast\ca P_A\right)  \\ 
&= \frac{1}{2} \cdot c_1(\ca L \otimes \ca L) = c_1(\ca L) \in \CH^1(A)_{\QQ}. 
\end{split}
\end{align}
The class $c_1(\ca L) \in \CH^{1,0}(A)$ of the line bundle $\ca L$ thus lies above $\Theta \in \CH^1(A)_\QQ$. Therefore, $\ca F(c_1(\ca L)) \in \CH_1(A)$ lies above $\Gamma_\Theta = (-1)^{g-1}\ca F_{A}(\Theta)$ by the commutativity of (\ref{commmuttttivitiy}), and \ref{motivicfour} holds.

Suppose that \ref{motivicfour} holds. Then \ref{motivicone} follows readily from Lemma \ref{lemma:minimalclasspoincarecomparison}. Moreover, if \ref{motivictwo} holds, then $\ch(\ca P_A) \in \CH(A \times A)_\QQ$ lifts to $\CH(A \times A)$, hence in particular \ref{motiviczero} holds. Since we have already proved that \ref{motivicone} implies \ref{motivictwo}, we conclude that $[\ref{motiviczero} \implies  \ref{motivicfour} \implies \ref{motivicone} \implies \ref{motivictwo} \implies \ref{motiviczero}]$. 

%So let us assume that $\lambda$ is induced by a symmetric ample line bundle $\ca L$ on $A$ and prove the rest of the equivalences. 
%This implies that 
%The implications $[\ref{motivictwo}\implies\ref{motivicone}]$, $[\ref{motivictwo} \implies \ref{motivicfive}]$ and $[\ref{motivicfive} \implies \ref{motivicsix}]$ are trivial. 
%Next, assume that $A$ is principally polarized by $\lambda$, and suppose that \textcolor{blue}{4} holds. 
%Let $\alpha \in \CH_1(A)$ be a cycle such that $\alpha_\QQ = \Theta^{g-1}/(g-1)!$ and identify $A$ and $\wh A$ via $\lambda$. Moreover, define
%\[
%\tau(\alpha) = 
%j_{1, \ast} (\alpha)  
%+  j_{2, \ast}(\alpha) - \Delta_\ast (\alpha) 
%\in \CH_1( A \times \wh A ). 
%\]
%Then $\left(\tau(\alpha)\right)_\QQ = (-1)^{g+1}\cdot\mr R_A \in \CH_1(A \times \wh A)_{\QQ}$ by Lemma \ref{lemma:minimalclasspoincarecomparison}, which proves that \ref{motiviccorollaryitem:one} holds. 
The implications $[\ref{motivictwo}\implies\ref{motivicfive} \implies \ref{motivicsix}]$ are trivial. Assume that \ref{motivicsix} holds. By Equation (\ref{liftoftheta}), $\Theta \in \CH^1(A)_{\QQ}$ lifts to $\CH^1(A)$, hence $$\ca F_{A}(\Theta) = (-1)^{g-1}\cdot \Gamma_\Theta$$ lifts to $\CH_1(A)$, i.e.\ \ref{motivicfour} holds. 
%and that we have already proved the implications $[\ref{motivictwo}\implies\ref{motivicseven}\implies \ref{motiviceight}]$, see Equation (\ref{modulotorsionequality}) and the proof of \ref{motiviceight} that followed. If we can prove the implications $[\ref{motivicsix}\implies\ref{motivicfour}]$ and $[\ref{motiviceight}\implies\ref{motivicfour}]$ then we are done, so let us assume either \ref{motivicsix} or \ref{motiviceight}. 
%In order to finish the proof, we will prove the direction $[\ref{motiviccorollaryitem:onefour} \implies \ref{motiviccorollaryitem:onetwothreefourprincipal} ]$. %Let $\ca F$ be a weak integral Fourier transform $\CH(A) \to \CH(\wh A)$. 
\\
\\
Assume that \ref{motivicsix} holds. The fact that $    \ca F_{A}\left( 
    \CH(A)/(\textnormal{torsion}) 
    \right) \subset \CH(A)/(\textnormal{torsion})$ implies that 
\begin{align*}%\label{modulotorsionequality}
    \CH(A)/(\textnormal{torsion}) &= \ca F_{A}\left( \ca F_{ A} \left( 
    \CH( A)/(\textnormal{torsion}) 
    \right) \right) \\
    & \subset 
    \ca F_{A}\left( 
    \CH(A)/(\textnormal{torsion}) 
    \right) \subset \CH(A)/(\textnormal{torsion}).
\end{align*}
Thus the restriction of the Fourier transform $\ca F_{A}$ to $\CH(A)/(\textnormal{torsion})$ defines an isomorphism $$\ca F_{A}\colon \CH(A)/(\textnormal{torsion}) \xrightarrow{\sim} \CH( A)/(\textnormal{torsion}).$$ If $R$ is a ring and $\gamma$ a PD-structure on an ideal $I \subset R$, then $\gamma$ extends to a PD-structure on $I/(\textnormal{torsion}) \subset R/(\textnormal{torsion})$. Thus, the ideal $\CH_{>0}(A)/(\textnormal{torsion})$ of $\CH(A)/(\textnormal{torsion})$ admits a PD-structure for the Pontryagin product $\star$ by Theorem \ref{th:PDstructure}. Since $\ca F_A$ exchanges Pontryagin and intersection products up to sign \cite[Proposition~3(ii)]{beauvillefourier}, it follows that \ref{motiviceight} holds. Since \ref{motiviceight} trivially implies \ref{motivicfour}, we are done. 
%there exists a PD-structure on $\CH^{>0}(A)/(\textnormal{torsion}) \subset \CH(A)/(\textnormal{torsion})$, i.e.\
%Consider the cycle $c_1(\ca L) \in \CH^1(A)/(\textnormal{torsion})$. Recall that $\ca F_A(\Theta) = (-1)^{g-1}\cdot (\wh \Theta)^{g-1}/(g-1)!$ by Lemma \ref{beauvillelemma}. Now $(-1)^{g-1}\cdot(\wh \Theta)^{g-1}/(g-1)! =  \ca F_A(\Theta) = \ca F_A(c_1(\ca L)) \in \CH_1(A)/(\textnormal{torsion})$, hence \ref{motiviccorollaryitem:onetwothreefourprincipal} holds.
\end{proof}

\begin{question}[Moonen--Polishchuk \cite{moonenpolishchuk}, Totaro \cite{totaroIHCthreefolds}] \label{questionmoonenpolishchuk}
Let $A$ be any principally polarized abelian variety over $k = \bar k$. Are the equivalent conditions in Theorem \ref{th:motivic} satisfied for $A$? 
%for every principally polarized abelian variety over $k$? 
\end{question}

\begin{remark}
For the Jacobian $A = J(C)$ of a hyperelliptic curve $C$, the answer to Question \ref{questionmoonenpolishchuk} is "yes" \cite{moonenpolishchuk}. %The reason is that for an Abel-Jacobi embedding $\iota \colon C \to J(C)$, one has $cl(\iota(C)) = \Gamma_\Theta$. For non-hyperelliptic curves, this equality does not hold, see ... and ..., so it is unclear how to 
\end{remark}
\noindent
Similarly, there is a relation between integral Fourier transforms up to homology and the algebraicity of minimal cohomology classes induced by Poincar\'e line bundles and theta divisors. 

\begin{proposition} \label{prop:motivicetale-new}
Let $A$ be an abelian variety of dimension $g$ over $k$. 

The following are equivalent:
\begin{enumerate}
    \item \label{motivicone-new} The class \[\rho_A \coloneqq c_1(\ca P_A)^{2g-1}/(2g-1)! \in \rm H^{4g-2}_{\etale}((A \times \wh A)_{k_s}, \ZZ_\ell(2g-1))\] is the class of a cycle in $\CH_1(A \times \wh A)$. %In other words: the cycle $c_1(\ca P_A)^{2g-1} \in \CH_1(A \times \wh A)$ is divisible by $(2g-1)!$ in $\CH_1(A \times \wh A)$. 
    \item \label{motivictwo-new} The abelian variety $A$ admits a motivic integral Fourier transform up to homology.  
    \item \label{motivicthree-new} The abelian variety $A \times \wh A$ admits a motivic integral Fourier transform up to homology. 
\end{enumerate}
%Any of these statements implies that 
%classes of symmetric line bundles $\Pic^{\textnormal{symm}}(A) \subset \Pic(A) \subset \CH(A)$ admit divided powers in $\CH(A)$, in the sense that for each $\alpha \in $
%there exists a PD-structure on the ideal $\CH^{>0}(A)/(\textnormal{torsion}) \subset \CH(A)/(\textnormal{torsion})$ (for the intersection product). %$\rm R_s^{>0}(A) \subset \rm R(A)$. %on the ideal of positive codimensional cycles in the subring of $\left(\CH(A), \cdotp\right)$ generated by $\CH^0(A)$ and $\Pic^{\textnormal{sym}}(A)$. 
Moreover, if $A$ carries an ample line bundle that induces a principal polarization $\lambda \colon A \xrightarrow{\sim} \wh A$, then 
%if $A$ is principally polarized by $\lambda\colon A \xrightarrow{\sim} \wh A$, and if $\Theta = \frac{1}{2}\cdot (\id, \lambda)^\ast c_1(\ca P_A) \in \CH^1(A)_{\QQ}$, then 
the above statements are equivalent to the following equivalent statements:
%\ref{motiviccorollaryitem:one} is implied by 
\begin{enumerate}
\setcounter{enumi}{3}
    \item \label{motiviczero-new} The class \[\sigma_A \coloneqq c_1(\ca P_A)^{2g-2}/(2g-2)! \in \rm H_{\etale}^{4g-4}((A \times \wh A)_{k_s}, \ZZ_\ell(2g-2))\] is the class of a cycle in $\CH_2(A \times \wh A)$.
    \item \label{motivicfour-new} The class $\gamma_\theta  = \theta^{g-1}/(g-1)! \in \rm H^{2g-2}_{\etale}(A_{k_s}, \ZZ_\ell(g-1))$ lifts to a cycle in $\CH_1(A)$. 
%\end{enumerate}
%Finally, if the principal polarization $\lambda\colon A \xrightarrow{\sim} \wh A$ is induced by a symmetric ample line bundle on $A$, then any of the above statements $\ref{motivicone}- \ref{motivicfour}$ is equivalent to any one of the following statements:
%\begin{enumerate}
%\setcounter{enumi}{5}
    \item \label{motivicfive-new} The abelian variety $A$ admits an integral Fourier transform up to homology.  
%    \item \label{motivicsix-new} The Fourier transform $\ca F_A$ % \colon \CH(A)_{\QQ} \to \CH(\wh A)_{\QQ}$
%    satisfies $\ca F_A\left( \CH(A)/(\textnormal{torsion})\right) \subset \CH(\wh A)/(\textnormal{torsion})$. 
%    \item \label{motivicseven} The Fourier transform $\ca F_A$ % \colon \CH(A)_{\QQ} \to \CH(\wh A)_{\QQ}$
%    satisfies $\ca F_A\left( \CH(A)/(\textnormal{torsion})\right) = \CH(\wh A)/(\textnormal{torsion})$. 
%    \item \label{motiviccorollaryitem:onefive} The Fourier transform $\ca F_A$ % \colon \CH(A)_{\QQ} \to \CH(\wh A)_{\QQ}$
%    satisfies $\ca F_A\left( \Pic^{\textnormal{sym}}(A)\right) \subset \CH(\wh A)/(\textnormal{torsion})$. 
%    \item \label{motiviceight} There exists a PD-structure on the ideal $\CH^{>0}(A)/(\textnormal{torsion}) \subset \CH(A)/(\textnormal{torsion})$. %of positive codimensional cycles in the subring of $\left(\CH(A), \cdotp\right)$ generated by $\CH^0(A)$ and $\Pic^{\textnormal{sym}}(A)$. 
\end{enumerate}
\end{proposition}

\begin{proof}
%The $\ZZ_\ell$-module $\oplus_{r \geq 0}\rm H^{2r}_{\textnormal{\'et}}(A_{k_s}, \ZZ_\ell(r))$ admits a $\ZZ_\ell$-algebra structure given by the Pontryagin product $\star$. With this in mind, 
The proof of Theorem \ref{th:motivic} can easily be adapted to this situation.  
\end{proof}

\begin{proposition} \label{remarketalebetti}
\begin{enumerate}
\item \label{remark3.12.1}
If $k = \CC$, then each of the statements $\ref{motivicone-new} - \ref{motivicfive-new}$ in Proposition \ref{prop:motivicetale-new} is equivalent to the same statement with \'etale cohomology replaced by Betti cohomology. %Indeed, in this case $\ZZ_\ell(i) = \ZZ_\ell$ and the Artin comparison isomorphism $\rm H^{2i}_{\textnormal{\'et}}(A, \ZZ_\ell) \xrightarrow{\sim} \rm H^{2i}(A(\CC), \ZZ_\ell)$ \cite[III, Expos\'e XI]{SGA4} is compatible with the cycle class map. Moreover, the map $\rm H^{2i}(A(\CC), \ZZ) \to \rm H_{\textnormal{\'et}}^{2i}(A, \ZZ_\ell)$ is injective. Therefore, a class $\beta \in \rm H^{2i}(A(\CC), \ZZ)$ is in the image of $cl\colon \CH^i(A) \to \rm H^{2i}(A(\CC),\ZZ)$ if and only if its image $\beta_\ell \in \rm H^{2i}_{\textnormal{\'et}}(A, \ZZ_\ell(i))$ is in the image of $cl\colon \CH^i(A) \to \rm H^{2i}_{\textnormal{\'et}}(A, \ZZ_\ell(i))$. %In this way we will apply Proposition \ref{prop:motivicetale-new} in the next section. 
\item \label{remark3.12.2}
Proposition \ref{prop:motivicetale-new} remains valid if one replaces integral Chow groups in statements \ref{motivicone-new}, \ref{motiviczero-new} and \ref{motivicfour-new} by their tensor product with $\ZZ_\ell$ and `integral Fourier transform up to homology' by `$\ell$-adic integral Fourier transform up to homology' in statements \ref{motivictwo-new}, \ref{motivicthree-new} and \ref{motivicfive-new}. %Indeed, for an abelian variety $A_{/k}$, the PD-structure on $\CH_{>0}(A) \subset (\CH(A), \star)$ carries over to a PD-structure on $\CH_{>0}(A) \otimes \ZZ_\ell \subset (\CH(A)_{\ZZ_\ell}, \star)$ because $\CH(A) \to \CH(A)_{\ZZ_\ell}$ is flat (since the ring map $\ZZ \to \ZZ_\ell$ is flat), see \cite[\href{https://stacks.math.columbia.edu/tag/07H1}{Tag 07H1}]{stacks-project}.
\end{enumerate}
\end{proposition}
\begin{proof}
\begin{enumerate}
\item 
In this case $\ZZ_\ell(i) = \ZZ_\ell$ and the Artin comparison isomorphism \[\rm H^{2i}_{\textnormal{\'et}}(A, \ZZ_\ell) \xrightarrow{\sim} \rm H^{2i}(A(\CC), \ZZ_\ell)\] \cite[III, Expos\'e XI]{SGA4} is compatible with the cycle class map. Since the map $\rm H^{2i}(A(\CC), \ZZ) \to \rm H_{\textnormal{\'et}}^{2i}(A, \ZZ_\ell)$ is injective, a class $\beta \in \rm H^{2i}(A(\CC), \ZZ)$ is in the image of $cl\colon \CH^i(A) \to \rm H^{2i}(A(\CC),\ZZ)$ if and only if its image $\beta_\ell \in \rm H^{2i}_{\textnormal{\'et}}(A, \ZZ_\ell)$ is in the image of $cl\colon \CH^i(A) \to \rm H^{2i}_{\textnormal{\'et}}(A, \ZZ_\ell)$. %In this way we will apply Proposition \ref{prop:motivicetale-new} in the next section. 
\item 
Indeed, for an abelian variety $A$ over $k$, the PD-structure on $\CH_{>0}(A) \subset (\CH(A), \star)$ induces a PD-structure on $\CH_{>0}(A) \otimes \ZZ_\ell \subset (\CH(A)_{\ZZ_\ell}, \star)$ by \cite[\href{https://stacks.math.columbia.edu/tag/07H1}{Tag 07H1}]{stacks-project}, because the ring map $(\CH(A),\star) \to (\CH(A)_{\ZZ_\ell},\star)$ is flat. The latter follows from the flatness of $\ZZ \to \ZZ_\ell$.
\end{enumerate}
\end{proof}

\cleardoublepage
% Note that depending on your settings in the table of contents, subsections and subsubsections might appear virtually identical.
% Make sure to set the ToC depth and the numbering depth in the ToC the way you want.
\chapter{One-cycles on abelian varieties}\label{ch:onecycles}

\hfill This chapter is based on joint work with \textsc{Thorsten Beckmann}. 

\section{Introduction}

In this chapter we provide applications of the results developed in the previous Chapter \ref{ch:integralfourier}. These applications concern the cycle class map for curves on an abelian variety $A$. More precisely, we will consider the integral Hodge conjecture for one-cycles when $A$ is defined over $\CC$, and the integral Tate conjecture for one-cycles when $A$ is defined over the separable closure of a finitely generated field. 

To state the most important results of Chapter \ref{ch:onecycles}, let us recall how the complex cycle class map was defined (see also Section \ref{intro:sub:complexandrealalgebraiccycles}). Whenever $\iota \colon C \hookrightarrow A$ is a smooth curve, the image of the fundamental class under the pushforward map \[\iota_\ast \colon \rm H_{2}(C, \ZZ) \to \rm  H_{2}(A,\ZZ) \cong \rm  H^{2g-2}(A, \ZZ)\] defines a cohomology class $[C] \in \rm H^{2g-2}(A, \ZZ)$. This construction extends to one-cycles and factors modulo rational equivalence. The \textit{cycle class map for curves on $A$} is the canonical homomorphism defined in this way:
\[
cl\colon \CH_1(A) \to \textnormal{Hdg}^{2g-2}(A, \ZZ).
\]
It extends to a natural graded ring homomorphism $cl\colon \CH(A) \to \rm H^\bullet(A, \ZZ)$. 

The liftability of the Fourier transform that we studied in Chapter \ref{ch:integralfourier} turns out to have important consequences for the image of the cycle class map. An element $\alpha \in \rm H^{\bullet}(A, \ZZ)$ is called \textit{algebraic} if it is in the image of $cl$, and that $A$ satisfies the \textit{integral Hodge conjecture for $i$-cycles} if all Hodge classes in $\rm H^{2g-2i}(A,\ZZ)$ are algebraic. Although the integral Hodge conjecture fails in general \cite{atiyahintegralhodge, trento,totarocobordism}, it is an open question for abelian varieties. The main result of Chapter \ref{ch:onecycles} is as follows.
%that there are many cases for which the answer to this question is affirmative:
%The first non-trivial case to look at should be the integral Hodge conjecture for one-cycles, since its analogue with rational coefficients (the Hodge conjecture for one-cycles) holds by Lefschez (1,1) and hard Lefschetz. 

\begin{theorem}[with T. Beckmann] \label{maintheorem}
%Let $\wh A$ be the dual abelian variety and $\ca P_A$ the Poincar\'e line bundle. 
Let $A$ be a complex abelian variety of dimension $g$ with Poincar\'e bundle $\ca P_A$. The following three statements are equivalent:
\begin{enumerate}
    \item \label{introitem:minimalpoincare} The cohomology class $c_1(\ca P_A)^{2g-1}/(2g-1)! \in \rm H^{4g-2}(A \times \wh A, \ZZ)$ is algebraic. 
    \item \label{introitem:integralpoincare} The Chern character $\ch(\ca P_A) = \exp(c_1(\ca P_A)) \in \rm H^\bullet(A \times \wh A, \ZZ)$ is algebraic.
    \item \label{introitem:integralhodgeforproduct} The integral Hodge conjecture for one-cycles holds for $A \times \wh A$. 
\end{enumerate}
Any of these statements implies that
%Moreover, any of the statements  $
%\ref{introitem:minimalpoincare} - \ref{introitem:integralhodgeforproduct}$ implies that 
\begin{enumerate}
\setcounter{enumi}{3}
    \item \label{introitem:IHC} The integral Hodge conjecture for one-cycles holds for $A$ and $\wh A$. 
\end{enumerate}
Suppose that $A$ is principally polarized by $\theta \in \Hdg^2(A,\ZZ)$ %the following two statements are equivalent, and equivalent to 
%then statements $\ref{introitem:minimalpoincare} - \ref{introitem:IHC}$ are equivalent, and any one of them is equivalent to any one of the following statements:
%and any of them is equivalent to any of the following statements:
and consider the following statements:
\begin{enumerate}
\setcounter{enumi}{4}
\item \label{introitem:minimalclass} The minimal cohomology class $\gamma_\theta \coloneqq \theta^{g-1}/(g-1)! \in \rm H^{2g-2}(A, \ZZ)$ is algebraic. 
%    \item[6.] \label{introitem:minimalclass} For each $i$ with $1 \leq i \leq g-1$, the cohomology class $\theta^{i}/i! \in \rm H^{2i}(A, \ZZ)$ is algebraic. 
\item \label{introitem:minimalpoincare2} The cohomology class $c_1(\ca P_A)^{2g-2}/(2g-2)! \in \rm H^{4g-4}(A \times \wh A, \ZZ)$ is algebraic. 
\end{enumerate}
%the algebraicity of the minimal cohomology class $\Gamma_\Theta = \theta^{g-1}/(g-1)! \in \rm H^{2g-2}(A, \ZZ)$. 
Then statements $\ref{introitem:minimalpoincare} - \ref{introitem:minimalpoincare2}$ are equivalent. If they hold, then the class $\theta^{i}/i! \in \rm H^{2i}(A, \ZZ)$ is algebraic for every positive integer $i$.  
%$i \in \ZZ_{\geq 1}$. 
\end{theorem}
\noindent
Remark that Condition~\ref{introitem:minimalclass} is stable under products, so a product of principally polarized abelian varieties satisfies the integral Hodge conjecture for one-cycles if and only if each of the factors does. More importantly, if $J(C)$ is the Jacobian of a smooth projective curve $C$ of genus $g$, then %$\gamma_\theta \in \rm H^{2g-2}(J(C), \ZZ)$ is the class of the image of an Abel-Jacobi map $C \to J(C)$, hence 
every integral Hodge class of degree $2g-2$ on $J(C)$ is a $\ZZ$-linear combination of curves classes:

\begin{theorem} \label{introth:IHCforjacobians}
Let $C_1, \dotsc ,C_n$ be smooth projective curves over $\CC$. Then the integral Hodge conjecture for one-cycles holds for the product of Jacobians $J(C_1) \times \cdots \times J(C_n)$. %Let $C$ be a smooth projective curve over $\CC$, and let $J(C)$ be the Jacobian of $C$. Then $J(C)$ satisfies the integral Hodge conjecture for one-cycles. 
\end{theorem}

\noindent
See Remark \ref{symmetricpowerremark}.\ref{symmetricremarkone} for another approach towards Theorem \ref{introth:IHCforjacobians} in the case $n=1$. A second consequence of Theorem \ref{maintheorem} is that the integral Hodge conjecture for one-cycles on principally polarized complex abelian varieties is stable under specialization, see Corollary \ref{complexspecialization}. An application of somewhat different nature is the following density result, proven in Section \ref{subsec:density}:

\begin{theorem} \label{introth:density}
Let $\delta = (\delta_1, \dotsc, \delta_g)$ be positive integers such that $\delta_i | \delta_{i+1}$ and let $\msf A_{g,\delta}(\CC)$ be the coarse moduli space of polarized abelian varieties over $\CC$ with polarization type $\delta$. There is a countable union $X\subset \msf A_{g,\delta}(\CC)$ of closed algebraic subvarieties of dimension at least $g$, that satisfies the following property: $X$ is dense in the analytic topology, and the integral Hodge conjecture for one-cycles holds for those polarized abelian varieties whose isomorphism class defines a point in $X$. 
\end{theorem}
%In fact, the dimensions of the components of $X$ may be taken bigger, depending on the polarization type $\delta$, see Remark ... .
%by defining natural numbers $i_j$ %dividing $I = \{1, \dotsc, g\}$ into subsets $I_{i_1} \cup I_{i_1} \cup \dotsc \cup I_{i_k} = I$ 
%by the condition that $\delta_1 = \dotsc = \delta_{i_1}$ and $\delta_{i_1} \neq \delta_{i_1+1}$; $\delta_{i_1+1} = \dotsc = \delta_{i_2}$, $\delta_{i_2} \neq \delta_{i_2+1}$, etc. 
\begin{remark}
The lower bound that we obtain on the dimension of the components of $X$ actually depends on $\delta$ and is often greater than $g$. For instance, when $\delta = 1$ and $g\geq 2$, there is a set $X$ as in the theorem, whose elements are prime-power isogenous to products of Jacobians of curves. Therefore, the components of $X$ have dimension $3g-3$ in this case, c.f. Remark \ref{rem:dimensionimprovement}.
\end{remark}
%Theorem \ref{maintheorem} could be compared with a result of Grabowski, that says
\noindent
One could compare Theorem \ref{maintheorem} with the following statement, proven by Grabowski \cite{grabowski}: if $g$ is a positive integer such that the minimal class $\gamma_\theta = \theta^{g-1}/(g-1)!$ of every principally polarized abelian variety of dimension $g$ is algebraic, then every abelian variety of dimension $g$ satisfies the integral Hodge conjecture for one-cycles. In this way, he proved the integral Hodge conjecture for abelian threefolds, a result which has been extended to smooth projective threefolds $X$ with $K_X = 0$ by Voisin and Totaro \cite{voisinintegralhodge,totaroIHCthreefolds}. For abelian varieties of dimension greater than three, not many unconditional statements seem to have been known. 
%The idea behind the proof of Theorem \ref{maintheorem} is the following. Let $A$ be a complex abelian variety of dimension $g$ and let $i \geq 0$ be an integer. Then Poincar\'e duality induces a canonical isomorphism $\varphi\colon \rm H^{2i}(A, \ZZ) \cong \rm H^{2g-2i}(A, \ZZ)^\vee \cong \rm H^{2g-2i}(\wh A, \ZZ)$. The map $\varphi$ respects the Hodge structures and thus induces an isomorphism $\Hdg^{2i}(A, \ZZ) \cong \Hdg^{2g-2i}(\wh A, \ZZ)$. However, it is unclear a priori whether $\varphi$ sends algebraic classes to algebraic classes. We prove that the algebraicity of $c_1(\ca P_A)^{2g-1}/(2g-1)! $ forces $\varphi$ to be algebraic, i.e.\ to be induced by a correspondence $\Gamma \in \CH(A \times \wh A)$. In particular, one then has $\rm Z^{2i}(A) \coloneqq \Hdg^{2i}(A, \ZZ) / \rm H^{2i}(A, \ZZ)_{\textnormal{alg}} \cong \rm Z^{2g-2i}(\wh A)$. To prove this, we lift the cohomological Fourier transform to a homomorphism between integral Chow groups whenever $c_1(\ca P_A)^{2g-1}/(2g-1)!$ is algebraic. For this we use a theorem of Moonen--Polishchuk saying that the ideal of positive dimensional cycles in the Chow ring with Pontryagin product of an abelian variety admits a divided power structure \cite[Theorem 1.6]{moonenpolishchuk}.

In Section \ref{sec:integralhodgeuptofactorn}, we consider an abelian variety $A_{/\CC}$ of dimension $g$ and ask: if $n \in \ZZ_{\geq 1}$ is such that $n \cdot c_1(\ca P_A)^{2g-1}/(2g-1)! \in \rm H^{4g-2}(A \times \wh A, \ZZ)_{\textnormal{alg}}$, can we bound the order of $\rm Z^{2g-2}(A)$ in terms of $g$ and $n$? For a smooth complex projective $d$-dimensional variety $X$, $\rm Z^{2d-2}(X)$ is called the degree $2d-2$ \textit{Voisin group} of $X$ \cite{perry2020integral}, is a stably birational invariant \cite[Lemma 15]{voisin_someaspectsofthehodgeconjecture}, and related to the unramified cohomology groups by Colliot-Thélène--Voisin and Schreieder \cite{colliotthelenevoisin, schreieder2021refined}. %Here $\rm H^{2g-2}(A, \ZZ)_{\textnormal{alg}} = cl\left( \CH^{g-1}(A)\right) \subset \rm H^{2g-2}(A, \ZZ)$, and $\rm Z^{2g-2}(A)$ is finite because of the Hodge conjecture for one-cycles. 
We prove that if $n \cdot c_1(\ca P_A)^{2g-1}/(2g-1)!$ is algebraic, then $\gcd( n^2, (2g-2)!) \cdot \rm Z^{2g-2}(A) = (0)$. In particular, $(2g-2)! \cdot \rm Z^{2g-2}(A) = (0)$ for any $g$-dimensional complex abelian variety $A$.
%$m(g,n) \cdot \textnormal{Hdg}^{2g-2}(A, \ZZ) \subset \rm H^{2g-2}(A, \ZZ)_{\textnormal{alg}}$. 
Moreover, if $A$ is principally polarized by $\theta \in \textnormal{NS}(A)$ and if $n \cdot \gamma_\theta \in \rm H^{2g-2}(A,\ZZ)$ is algebraic, then $n \cdot c_1(\ca P_A)^{2g-1}/(2g-1)!$ is algebraic. Since it is well known that for Prym varieties, the Hodge class $2 \cdot \gamma_\theta$ is algebraic, these observations lead to the following result (see also Theorem~\ref{theorem:integralhodgeuptofactor}).
\begin{theorem} \label{introtheorem:prym}
Let $A$ be a $g$-dimensional Prym variety over $\CC$. Then $4 \cdot \rm Z^{2g-2}(A) = (0)$. 
%$4\cdot \Hdg^{2g-2}(A,\ZZ) \subset \rm H^{2g-2}(A,\ZZ)_{\textnormal{alg}}$. 
% and let $\alpha = \log_2 m(g,2)$. Then $2^\alpha \cdot \rm Z^{2g-2}(A)= 0$. In particular, if $A$ is any principally polarized complex abelian variety of dimension $4$ (respectively $5$), then $2^3 \cdot \rm Z^{2g-2}(A)= 0$ (respectively $2^4 \cdot \rm Z^{2g-2}(A)= 0$). 
\end{theorem}
\noindent
Recall that in our study of the liftability of the Fourier transform, carried out in the previous Chapter \ref{ch:integralfourier}, 
%which was initiated by Moonen and Polishchuk in \cite{moonenpolishchuk}, it is more natural to 
we considered abelian varieties defined over arbitrary fields. This generality allows us now to obtain the analogue of Theorem \ref{maintheorem} over the separable closure $k$ of a finitely generated field. 

A smooth projective variety $X$ of dimension $d$ over $k$ satisfies the \textit{integral Tate conjecture for one-cycles over $k$} if, for every prime number $\ell$ different from $\textnormal{char}(k)$ and for some finitely generated field of definition $k_0 \subset k$ of $X$, the cycle class map 
\begin{equation} \label{eq:integraltateconjecture}
    cl\colon \CH_1(X)_{\ZZ_\ell} = \CH_1(X)\otimes_\ZZ{\ZZ_\ell} \to \bigcup_U\rm H_{\textnormal{\'{e}t}}^{2d-2}(X, \ZZ_\ell(d-1))^U
\end{equation}
is surjective, where $U$ ranges over the open subgroups of $\Gal(k/k_0)$. 
%every element $u \in \rm H^{2g-2i}(X, \ZZ_\ell(g-i))$ fixed by some open subgroup of $\Gal(k/L)$ is the class of an algebraic $i$-cycle over $k$ with $\ZZ_\ell$-coefficients. 

\begin{theorem} \label{introth:integraltate} Let $A$ be an abelian variety of dimension $g$ over the separable closure $k$ of a finitely generated field. 
\begin{itemize}[wide, labelwidth=!, labelindent=0pt]
\item 
The abelian variety $A$ satisfies the integral Tate conjecture for one-cycles over $k$ if the cohomology class $$c_1(\ca P_A)^{2g-1}/(2g-1)! \in \rm H_{\textnormal{\'et}}^{4g-2}(A \times \wh A, \ZZ_\ell(2g-1))$$ is the class of a one-cycle with $\ZZ_\ell$-coefficients for every prime number $\ell < (2g-1)!$ unequal to $\textnormal{char}(k)$. 
    \item 
Suppose that $A$ is principally polarized and let $\theta_\ell \in \rm H^2_{\textnormal{\'et}}(A, \ZZ_\ell(1))$ be the class of the polarization. The map (\ref{eq:integraltateconjecture}) is surjective if 
$$
\gamma_{\theta_\ell} \coloneqq  \theta_\ell^{g-1}/(g-1)! \in \rm H_{\textnormal{\'{e}t}}^{2g-2}(A, \ZZ_\ell(g-1))
$$ is in its image. In particular, if $\ell > (g-1)!$ then this always holds. Thus $A$ satisfies the integral Tate conjecture for one-cycles if and only if $\gamma_{\theta_\ell}$ is in the image of (\ref{eq:integraltateconjecture}) for every prime number $\ell < (g-1)!$ unequal to $\textnormal{char}(k)$. 
%holds for any principally polarized abelian variety of dimension $g$. 
%A principally polarized abelian variety $A$ of dimension $g$ over $k$ satisfies the integral Tate conjecture for one-cycles over $k$ if $\theta^{g-1}/(g-1)! \in \rm H_{\textnormal{\'{e}t}}^{2g-2}(A, \ZZ_\ell(g-1))$ is algebraic (i.e.\ in the image of (\ref{eq:integraltateconjecture})). 

\end{itemize}
\end{theorem}
%\begin{corollary}
%Let $A$ be an abelian scheme over the ring of integers $\OO_K$ of a number field $K$. %, let $A$ be an abelian variety over $K$ and let $\ca A$ be . 
%Let $\eta = \Spec K$ be the generic point of $\Spec \OO_K$ and suppose that $A_{\bar K}$ satisfies the integral Tate conjecture for one-cycles over $\bar K = \bar \QQ$. Then $A_{k(\mf p)}$ satisfies the integral Tate conjecture for one-cycles over $k(\mf p) = \OO_K/\mf p$ for every prime $\mf p$ of $\OO_K$ at which $A_K$ has good reduction.  
%for some embedding $\sigma \colon K \to \CC$, the abelian variety $A_\CC = A \times_{K, \sigma} \CC$ satisfies the integral Hodge conjecture for one-cycles. 
%\end{corollary}
%which is in turn equivalent to the existence of a lift of the Fourier transform $\mr F_X\colon \rm H_{\textnormal{\'{e}t}}^{\bullet}(X, \ZZ_\ell(\bullet)) \to \rm H_{\textnormal{\'{e}t}}^{\bullet}(\wh X, \ZZ_\ell(\bullet))$ to a homomorphism $\CH(X)_{\ZZ_\ell} \to \CH(\wh X)_{\ZZ_\ell}$. These conditions are satisfied when $X$ is the Jacobian of a proper curve of compact type over $k$.
%In particular, if $\ell$ is any primer number greater than $(2g-1)!$ (respectively $(g-1)!$), then for any abelian variety (respectively principally polarized abelian variety) of dimension $A$ of dimension $g$, the cycle class map 
\noindent
Theorem \ref{introth:integraltate} implies in particular that products of Jacobians of smooth projective curves over $k$ satisfy the integral Tate conjecture for one-cycles over $k$. Moreover, for an abelian variety $A_K$ over a number field $K \subset \CC$, the integral Hodge conjecture for one-cycles on $A_\CC$ is equivalent to the integral Tate conjecture for one-cycles on $A_{\bar K}$ (Corollary \ref{hodgetatecomparison}), which in turn implies the integral Tate conjecture for one-cycles on the geometric special fiber $A_{\overline{k}(\mf p)}$ of the N\'eron model of $A_K$ over $\OO_K$ for any prime $\mf p \subset \OO_K$ at which $A_K$ has good reduction (Corollary \ref{cor:specialization}). 
%Another corollary is that the integral Tate conjecture for one-cycles on abelian varieties is stable under finitely generated base field extension (Corollary \ref{cor:basechange}). 
%Finally, %Theorem \ref{introth:integraltate} implies $-$ 
%just as in the complex case we have that 
%Theorem \ref{introth:integraltate} also implies that the integral Tate conjecture for one-cycles on a principally polarized abelian variety is a property which is stable under specialization (see e.g.\ Corollary \ref{cor:specialization}). 
\noindent
Finally, we obtain the analogue of Theorem \ref{introth:density} in positive characteristic as well. The definition for a smooth projective variety over the algebraic closure $k$ of a finitely generated field to satisfy the \textit{integral Tate conjecture for one-cycles over $k$} is analogous to the definition above (see e.g.\ \cite{charlespirutka}). 
%has an analogue over the algebraic closure $\bar k$ of $k$ (see e.g.\ ;;;). %Recall that, similarly to the above, a smooth projective variety $X$ over $\bar k$ satisfies the \textit{integral Tate conjecture for $i$-cycles over $\bar k$} if $cl \colon \CH^i(X)_{\ZZ_\ell} \to \cup _U\rm H_{\textnormal{\'{e}t}}^{2i}(X, \ZZ_\ell(i))^U$ is surjective, where $U$ ranges over the open subgroups of the absolute Galois group of a finitely generated field of definition of $X$. 
\begin{theorem} \label{th:chinglichai}
Let $k$ be the algebraic closure of a finitely generated field of characteristic $p>0$. %Let $\bar k$ be the algebraic closure of $k$ %let $\delta = (\delta_1, \dotsc, \delta_g)$ be positive integers such that $\delta_i | \delta_{i+1}$, 
Let $\msf A_{g}$ be the coarse moduli space over $k$ of principally polarized abelian varieties of dimension $g$ over $k$. Let $X \subset \msf A_{g}(k)$ be the subset of moduli points attached to principally polarized abelian varieties over $k$ that satisfy the integral Tate conjecture for one-cycles over $k$. Then $X$ is Zariski dense in $\msf A_{g}$. 
\end{theorem}

\section{The integral Hodge conjecture}

In this section we use the theory developed in Chapter \ref{ch:integralfourier} to prove Theorem \ref{maintheorem}. We also prove some applications of Theorem \ref{maintheorem}: the integral Hodge conjecture for one-cycles on products of Jacobians (Theorem \ref{introth:IHCforjacobians}), the fact that the integral Hodge conjecture for one-cycles on principally polarized complex abelian varieties is stable under specialization (Corollary \ref{complexspecialization}) and density of polarized abelian varieties satisfying the integral Hodge conjecture for one-cycles (Theorem \ref{introth:density}). 

\subsection{Proof of the main theorem} \label{sec:proofmaintheorem}

Let us prove Theorem \ref{maintheorem}. %We first prove that Condition \ref{introitem:minimalpoincare} in Theorem \ref{maintheorem} is stable under products.

\begin{proof}[Proof of Theorem \ref{maintheorem}]
Suppose that \ref{introitem:minimalpoincare} holds. Then \ref{introitem:integralpoincare} holds by Propositions \ref{prop:motivicetale-new} and \ref{remarketalebetti}.\ref{remark3.12.1}. Suppose that \ref{introitem:integralpoincare} holds. Then %$\ch(\ca P_A) \in \rm H^\bullet (A \times \wh A, \ZZ)$ and $\ch(\ca P_{\wh A}) \in \rm H^\bullet (\wh A \times A, \ZZ)$ lift to correspondences $\Gamma \in \CH(A \times \wh A)$ and $\wh \Gamma \in \CH(\wh A \times A)$ which are integral Fourier transforms up to homology, so 
\ref{introitem:IHC} follows from Lemma \ref{lemma:trivial}. So we have $[\ref{introitem:minimalpoincare} \iff \ref{introitem:integralpoincare}  \implies \ref{introitem:IHC}]$. If $\ref{introitem:minimalpoincare}$ holds, then $\rho_A = c_1(\ca P_A)^{2g-1}/(2g-1)! \in \rm H^{4g-2}(A \times \wh A, \ZZ)$ is algebraic, which implies that $\rho_{\wh A} \in \rm H^{4g-2}(\wh A \times A, \ZZ)$ is algebraic. Therefore, $\rho_{A \times \wh A} %= c_1(\ca P_{A \times \wh A})^{4g-1}/(4g-1)! 
\in \rm H^{8g-2}(A \times \wh A \times \wh A \times A, \ZZ)$ is algebraic by Equation (\ref{eq1}). We then apply the implication $[\ref{introitem:minimalpoincare}  \implies \ref{introitem:IHC}]$ to the abelian variety $A \times \wh A$, which shows that \ref{introitem:integralhodgeforproduct} holds. Since $[\ref{introitem:integralhodgeforproduct}  \implies \ref{introitem:minimalpoincare}]$ is trivial, we have proven $[\ref{introitem:minimalpoincare} \iff \ref{introitem:integralpoincare}  \iff \ref{introitem:integralhodgeforproduct} \implies  \ref{introitem:IHC}]$. 

Next, assume that $A$ is principally polarized by $\theta \in \textnormal{NS}(A) \subset \rm H^2(A, \ZZ)$. The directions $[\ref{introitem:IHC}\implies \ref{introitem:minimalclass}]$ and $[\ref{introitem:integralpoincare} \implies \ref{introitem:minimalpoincare2}]$ are trivial and $[\ref{introitem:minimalclass} \implies \ref{introitem:minimalpoincare}]$ follows from Propositions \ref{prop:motivicetale-new} and \ref{remarketalebetti}.\ref{remark3.12.1}. We claim that \ref{introitem:minimalpoincare2} implies \ref{introitem:IHC}. Define $$\sigma_A = c_1(\ca P_A)^{2g-2}/(2g-2)! \in \rm H^{4g-4}(A \times \wh A,\ZZ)$$ and let $S \in \CH_2(A \times \wh A)$ be such that $cl(S) =  \sigma_A$. The squares in the following diagram commute:
\begin{equation} \label{diagramwithonlysurfaceclasses}
%\xymatrixcolsep{2pc}
\xymatrix{
\CH^1(A) \ar[d]^{cl}\ar[r]^{\pi_1^\ast}& \CH^1(A \times \wh A) \ar[d]^{cl}\ar[r]^{\cdot S}& \CH^{2g-1}(A \times \wh A) \ar[d]^{cl} \ar[r]^{\pi_{2,\ast}} & \CH_1(\wh A) \ar[d]^{cl} \\
\rm H^2(A,\ZZ) \ar[r]^{\pi_1^\ast} & \rm H^{2}(A \times \wh A,\ZZ) \ar[r]^{\cdot \sigma_A} & \rm H^{4g-2}(A \times \wh A,\ZZ) \ar[r]^{\pi_{2,\ast}} & \rm H^{2g-2}(\wh A,\ZZ). 
}
\end{equation}
Since $\mr F_A = \pi_{2,\ast}\left( \ch(\ca P_A) \cdot \pi_1^\ast(-)\right)$ %\colon \rm H^{\bullet}(A,\ZZ) \to \rm H^\bullet(\wh A,\ZZ)$ 
restricts to an isomorphism $$\mr F_A\colon \rm H^2(A,\ZZ) \xrightarrow{\sim} \rm H^{2g-2}(\wh A,\ZZ)$$ by \cite[Proposition 1]{beauvillefourier}, the composition $\pi_{2,\ast} \circ (- \cdot \sigma_A) \circ \pi_1^\ast$ on the bottom row of (\ref{diagramwithonlysurfaceclasses}) is an isomorphism. By Lefschetz $(1,1)$, the map $cl\colon \CH_1(\wh A) \to \Hdg^{2g-2}(\wh A,\ZZ)$ is therefore surjective. 
%$\rm H^2(A,\ZZ) \xrightarrow{\sim}\rm H^{2g-2}(\wh A,\ZZ)$. 

It remains to prove the algebraicity of the classes $\theta^i/i! \in \rm H^{2i}(A,\ZZ)$. This follows from Theorem \ref{th:PDstructure} and the following equality, see \cite[Corollaire 2]{beauvillefourier}):
\[
\frac{\theta^{i}}{i!}= \frac{\gamma_\theta^{\star j}}{j!}, \quad \gamma_\theta = \frac{\theta^{g-1}}{(g-1)!} \in \rm H^{2g-2}(A, \ZZ), \quad i + j = g.
\]
Therefore, the proof is finished.  
\end{proof}

\begin{corollary} \label{cor:IHCforproducts}
Let $A$ and $B$ be complex abelian varieties of respective dimensions $g_A, g_B$.
\begin{itemize}
    \item The Hodge classes $\rho_A \in \rm H^{4g_A-2}(A \times \wh A, \ZZ)$ and $\rho_B \in \rm H^{4g_B-2}(B \times \wh B, \ZZ)$ are algebraic if and only if $A \times \wh A$, $B \times \wh B$, $A\times B$ and $\wh A \times \wh B$ satisfy the integral Hodge conjecture for one-cycles. 
    \item If $A$ and $B$ are principally polarized, then the integral Hodge conjecture for one-cycles holds for $A \times B$ if and only if it holds for $A$ and $B$. 
\end{itemize}
\end{corollary}
\begin{proof}
The first statement follows from Theorem \ref{maintheorem} and Equation (\ref{eq1}). The second statement follows from the fact that the minimal cohomology class of the product $A \times B$ is algebraic if and only if the minimal cohomology classes of the factors $A$ and $B$ are both algebraic.
\end{proof}

\begin{proof}[Proof of Theorem \ref{introth:IHCforjacobians}]
By Corollary \ref{cor:IHCforproducts} we may assume $n = 1$, so let $C$ be a smooth projective curve. Let $p \in C$ and consider the morphism $\iota\colon C \to J(C)$ defined by sending a point $q$ to the isomorphism class of the degree zero line bundle $\OO(p-q)$. %embedding $\iota \colon C \to J(C)$, $q \mapsto [\ca O_C(p-q)]$. 
Then $cl(\iota(C)) = \gamma_\theta \in \rm H^{2g-2}(J(C), \ZZ)$ by Poincar\'e's formula \cite{arbarello1985geometry}, so $\gamma_\theta$ is algebraic and the result follows from Theorem \ref{maintheorem}. 
\end{proof}

\begin{remarks} \label{symmetricpowerremark}
\begin{enumerate}[wide, labelwidth=!, labelindent=0pt]
    \item \label{symmetricremarkone} Let us give another proof of Theorem \ref{introth:IHCforjacobians} in the case $n = 1$, i.e.\ let $C$ be a smooth projective curve of genus $g$ and let us prove the integral Hodge conjecture for one-cycles on $J(C)$ in a way that does not use Fourier transforms. It is classical that any Abel-Jacobi map $C^{(g)} \to J(C)$ is birational. On the other hand, the integral Hodge conjecture for one-cycles is a birational invariant, see \cite[Lemma 15]{voisin_someaspectsofthehodgeconjecture}. Therefore, to prove it for $J(C)$ it suffices to prove it for $C^{(g)}$. One then uses \cite[Corollary 5]{delbanohodgecurve} which says that for each $n \in \ZZ_{\geq 1}$, there is a natural polarization $\eta$ on the $n$-fold symmetric product $C^{(n)}$ such that for any $i \in \ZZ_{\geq 0}$, the map $\eta^{n-i}\cup (-) \colon \rm H^i(C^{(n)}, \ZZ) \to \rm H^{2n-i}(C^{(n)}, \ZZ)$ is an isomorphism. In particular, the variety $C^{(n)}$ satisfies the integral Hodge conjecture for one-cycles for any positive integer $n$. 
    \item \label{symmetricremarktwo} Along these lines, observe that the integral Hodge conjecture for one-cycles holds not only for symmetric products of smooth projective complex curves but also for any product $C_1 \times \cdots \times C_n$ of smooth projective curves $C_i$ over $\CC$. Indeed, this follows readily from the K\"unneth formula. 
    \item Let $C$ be a smooth projective complex curve of genus $g$. Our proof of Theorem \ref{maintheorem} provides an explicit description of $\Hdg^{2g-2}(J(C),\ZZ)$ depending on $\Hdg^2(J(C),\ZZ)$. More generally, let $(A,\theta)$ be a principally polarized abelian variety of dimension $g$, identify $A$ and $\wh A$ via the polarization, and let $\ell = c_1(\ca P_A) \in \rm H^2(A \times \wh A,\ZZ)$. Then $\ell = m^\ast(\theta) - \pi_1^\ast(\theta) - \pi_2^\ast(\theta)$, which implies that 

    \begin{align*}
\sigma_A= 
\frac{\ell^{2g-2}}{(2g-2)!} &= \sum_{\substack{i,j,k \geq 0\\ i + j + k = 2g-2}}^{2g-2}
(-1)^{j+k}\cdot  m^\ast\left( \frac{\theta^i}{i!}\right)\cdot \pi_1^\ast\left(\frac{\theta^j}{j!}\right)\cdot \pi_2^\ast \left(\frac{\theta^{k}}{k!}\right). 
%\frac{\ell^{2g-2}}{(2g-2)!} &= \sum_{i = 0}^{2g-2}\sum_{j =0}^{2g-2-i}
%(-1)^{2g-2-i}\cdot  m^\ast\left( \frac{\theta^i}{i!}\right)\cdot \pi_1^\ast\left(\frac{\theta^j}{j!}\right)\cdot \pi_2^\ast \left(\frac{\theta^{2g-2-i-j}}{(2g-2-i-j)!}\right).  
    \end{align*}
    Any $\beta \in \Hdg^{2g-2}(A,\ZZ)$ is of the form $\pi_{2,\ast} \left( \sigma_A \cdot \pi_1^\ast [D] \right)$, where $[D] = cl(D)$ for a divisor $D$ on $A$, as follows from \eqref{diagramwithonlysurfaceclasses}. Therefore, any $\beta \in \Hdg^{2g-2}(A,\ZZ)$ may be written as
        \begin{equation} \label{beta}
\beta = \sum_{\substack{i,j,k \geq 0\\ i + j + k = 2g-2}}^{2g-2}
(-1)^{j+k}\cdot \pi_{2,\ast}\left( m^\ast\left( \frac{\theta^i}{i!}\right)\cdot \pi_1^\ast\left(\frac{\theta^j}{j!}\right)\cdot \pi_1^\ast[D]\right) \cdot \frac{\theta^{k}}{k!}. 
    \end{equation}
\noindent
Returning to the case of a Jacobian $J(C)$ of a smooth projective curve $C$ of genus $g$, the classes $\theta^i/i!$ appearing in (\ref{beta}) are effective algebraic cycle classes. Indeed, for $p \in C$ and $d \in \ZZ_{\geq 1}$, the image of the morphism $C^{d} \to J(C)$, $(x_i) \mapsto \ca O(\sum_ix_i-d\cdot p)$ defines a subvariety $W_d(C) \subset J(C)$ and by Poincar\'e's formula \cite[\S I.5]{arbarello1985geometry} one has $cl(W_d(C)) = \theta^{g-d}/(g-d)! \in \rm H^{2g-2d}(J(C), \ZZ)$. 
\end{enumerate}
\begin{comment}    if we define
    \begin{align*}
\alpha &= \sum_{i = 0}^{g-2}\pi_1^\ast \left(\frac{\theta^i}{i!}\right) \cdot \pi_2^\ast \left(\frac{\theta^{g-2-i}}{(g-2-i)!}\right) + m^\ast \left(\frac{\theta^{g-1}}{(g-1)!}\right)\cdot \left( \sum_{i = 0}^{g-1}\pi_1^\ast \left(\frac{\theta^i}{i!}\right) \cdot \pi_2^\ast\left(\frac{\theta^{g-1-i}}{(g-1-i)!}\right) \right) \\
&+ m^\ast \left( \frac{\theta^{g-2}}{(g-2)!} \right) \cdot \left( 
\sum_{i = 0}^g \pi_1^\ast \left( \frac{\theta^i}{i!} \right) \cdot \pi_2^\ast \left( \frac{\theta^{g-i}}{(g-i)!} \right)  
\right) \quad \in \quad \rm H^{4g-4}(A \times A,\ZZ),
    \end{align*}
    then $\alpha= \frac{\ell^{2g-2}}{(2g-2)!} = \sigma_A$. It follows then from (\ref{diagramwithonlysurfaceclasses}) that any $\beta \in \Hdg^{2g-2}(A,\ZZ)$ is of the form $\beta = \pi_{2,\ast} \left( \alpha \cdot \pi_1^\ast [D] \right)$ where $[D] = cl(D)$ for a divisor $D$ on $A$. By the projection formula,
\begin{align*}
\beta &=  %\sum_{i = 0}^{g-2}\pi_{2,\ast}\left(\pi_1^\ast \left(\frac{\theta^i}{i!} \cdot [D] \right)\right) \cdot  \frac{\theta^{g-2-i}}{(g-2-i)!}
%\\
%&+  
\sum_{i = 0}^{g-1}\pi_{2,\ast}\left(m^\ast \left(\frac{\theta^{g-1}}{(g-1)!}\right)\cdot \pi_1^\ast \left(\frac{\theta^i}{i!} \cdot [D]\right)\right) \cdot \frac{\theta^{g-i-1}}{(g-i-1)!}  \\
&+ 
\sum_{i = 0}^g \pi_{2,\ast} \left(m^\ast \left( \frac{\theta^{g-2}}{(g-2)!} \right) \cdot  \pi_1^\ast \left( \frac{\theta^i}{i!}\cdot [D] \right)\right) \cdot  \frac{\theta^{g-i}}{(g-i)!}
\quad \in \quad \rm H^{2g-2}(A,\ZZ).
    \end{align*}
    \end{comment}
\end{remarks}
\noindent
Apart from Theorem \ref{introth:IHCforjacobians}, we obtain the following corollary of Theorem \ref{maintheorem}:

\begin{corollary} \label{complexspecialization}
Let $A\to S$ be a principally polarized abelian scheme over a proper, smooth and connected variety $S$ over $\CC$. % principally polarized by an isomorphism $\lambda\colon A \to \wh A$ induced by a relatively ample line bundle on $A$ over $S$. 
Let $X \subset S(\CC)$ be the set of $x \in S(\CC)$ such that the abelian variety $A_x$ satisfies the integral Hodge conjecture for one-cycles. Then $X = \cup_iZ_i(\CC)$ for some countable union of closed algebraic subvarieties $Z_i \subset S$. %is a countable union $\cup_iZ_i(\CC) \subset S(\CC)$ of closed algebraic subvarieties $Z_i(\CC) \subset S(\CC)$. 
In particular, if the integral Hodge conjecture for one-cycles holds on $U(\CC)$ for a non-empty open subscheme $U$ of $S$, then it holds on all of $S(\CC)$. 
%Suppose that for every closed point $x$ in a non-empty open subscheme $U \subset S$, the abelian variety $A_x$ satisfies the integral Hodge conjecture for one-cycles. Then for every $x \in S(\CC)$, the abelian variety $A_x$ satisfies the integral Hodge conjecture for one-cycles. 
\end{corollary}
\begin{proof}
Write $\ca A = A(\CC)$ and $B = S(\CC)$ and let $\pi\colon \ca A \to B$ be the induced family of complex abelian varieties. Let $g \in \ZZ_{\geq 0}$ be the relative dimension of $\pi$ and define, for $t \in S(\CC)$, $$\theta_t \in \NS(\ca A_t)\subset \rm H^2(\ca A_t, \ZZ)$$ to be the polarization of $\ca A_t$. There is a global section $\gamma_\theta \in \rm R^{2g-2}\pi_\ast\ZZ$ such that for each $t \in B$, $$\gamma_{\theta_t} = \theta_t^{g-1}/(g-1)! \in \rm H^{2g-2}(\ca A_t, \ZZ).$$ %be the minimal cohomology class of $(A_t, \theta_t)$. These $\gamma_{\theta_t}$ are the restrictions of  since the polarized family $\pi$ is locally topologically trivial. 
Note that $\gamma_\theta$ is Hodge everywhere on $B$. For those $t \in B$ for which $\gamma_{\theta_t}$ is algebraic, write $\gamma_{\theta_t}$ as the difference of effective algebraic cycle classes on $\ca A_t$. This gives a countable disjoint union $$\phi \colon \sqcup_{ij}H_i\times_S H_j \to S$$ of products of relative Hilbert schemes $H_i/S$. By Lemma \ref{lemma_algebraicityminimalclass} below, $\gamma_{\theta_t}$ is algebraic precisely for closed points $t$ in the image $Y \subset S$ of $\phi$. By Theorem \ref{maintheorem}, $X = Y$. 
%It follows that $U \subset Y$, where $\widetilde U$ is the inverse image in $X$ of the open $U \subset S$ that exists by hypothesis. Since $\widetilde U \subset X$ is a Zariski open subset of $X$, and $X$ is smooth and connected, it follows that $Y = X$ by the Baire category theorem. Therefore the minimal class $\gamma_{\theta_x}$ of every principally polarized fiber $A_x$, $x \in S(\CC)$ of the original family $A/S$ is algebraic. By Theorem \ref{maintheorem} this is enough to prove the corollary. 
\end{proof}

\begin{lemma} \label{lemma_algebraicityminimalclass}
Let $S$ be an integral variety over $\CC$, let $\ca A \to S$ be a principally polarized abelian scheme of relative dimension $g$ over $S$ and let $\ca C_i \subset \ca A$ for $i= 1,\dotsc, k$ be relative curves in $\ca A$ over $S$. Let $n_1,\dotsc, n_k$ be integers and let $y \in S(\CC)$ be a point that satisfies $\sum_{i = 1}^k n_i\cdot cl(C_{i,y}) = \gamma_{\theta_y} \in \rm H^{2g-2}(A_y,\ZZ)$. Then for every $x \in S(\CC)$, one has the equality $\sum_{i = 1}^k n_i\cdot cl(C_{i,x}) = \gamma_{\theta_x} \in \rm H^{2g-2}(A_x,\ZZ)$. 
\end{lemma}

\begin{proof}
Since it suffices to prove the lemma for any open affine $U \subset S$ that contains $y$, we may assume that $S$ is quasi-projective. Fix $x \in S(\CC)$. After replacing $S$ by a suitable base change containing $x$ and $y$, we may assume that $S$ is a smooth connected curve. %Since $S$ can be covered by finitely many affine open subsets, we may assume that $S$ is affine. Then $S$ is affine open in a projective variety $V$ and applying a resolution of singularities to $V$ implies that we may assume that $S \subset V$ is an open subvariety of a smooth projective variety $V$. 
%\subset \PP^N_\CC$ and applying a resolution of singularities to $V$ gives a smooth projective variety $\tilde V$ and a surjective birational morphism $\tilde V \to V$; let $\tilde S \subset \tilde V$ be the inverse image of $S$ in $\tilde V$. 
%\\
%\\
%By Chow's lemma we may assume that $S$ is projective, and by Hironaka's resolution of singularities we may assume that $S$ is smooth. %Let $x \in S(\CC)$. 
%By Bertini's theorem, there is a smooth irreducible curve $Z \subset V$ containing any two given closed points of $V$, hence we may assume that $S$ is an open subset of a smooth irreducible curve. 
For $t \in S$, denote by $\theta_{\bar t} \in \rm H^{2}_{\etale}(A_{\bar t}, \ZZ_\ell)$ the class of the polarization and $\gamma_{\theta_{\bar t}} = \theta_{\bar t}^{g-1}/(g-1)!$. %the group $\rm H^{2g-2}(A_x,\ZZ)$ being torsion-free. 
Let $\eta = \Spec(K)$ be the generic point of $S$. The elements $\sum_i n_i \cdot cl(C_{i,\bar \eta})$ and $\gamma_{\theta_{\bar \eta}}$ in $\rm H^{2g-2}_{\etale}(A_{\bar \eta}, \ZZ_\ell)$ both map to $$\sum_i n_i \cdot cl(C_{i,y}) = \gamma_{\theta_{y}} \in  \rm H^{2g-2}_{\etale}(A_{y}, \ZZ_\ell)$$ under the specialization homomorphism 
$$
s\colon \rm H^{2g-2}_{\etale}(A_{\bar \eta}, \ZZ_\ell) \to \rm H^{2g-2}_{\etale}(A_{y}, \ZZ_\ell)
$$ %, because the latter lifts to a homomorphism $\CH_1(A_{\bar \eta}) \to \CH_1(A_y)$ by 
by \cite[Example 20.3.5]{fultonintersection}. Since $s$ is an isomorphism, we have $\sum_i n_i \cdot cl(C_{i,\bar \eta}) = \gamma_{\theta_{\bar \eta}}$, which implies that $\sum_{i}n_i\cdot cl(C_{x,i}) = \gamma_{\theta_x} \in \rm H_{\etale}^{2g-2}(A_x,\ZZ_\ell)$.
%in turn gives $\sum_i n_i \cdot cl(C_{x,i}) = \gamma_{\theta_{x}} \in \rm H_{\etale}^{2g-2}(A_x, \ZZ_\ell)$ as desired.
%$x, y \in Z$. 
\end{proof}

\subsection{Density of abelian varieties satisfying the integral Hodge conjecture for one-cycles} \label{subsec:density}

The goal of this section is to prove that Conditions $\ref{introitem:minimalpoincare}-\ref{introitem:integralhodgeforproduct}$ in Theorem \ref{maintheorem} are satisfied on a dense subset of the moduli space of complex abelian varieties. To do so, will we state yet another criterion that a complex abelian variety may satisfy. In some sense this criterion provides a bridge between abelian varieties outside the Torelli locus and those lying within, thereby implying the integral Hodge conjecture for one-cycles for the abelian variety under consideration. %In fact, for many principally polarized abelian varieties, the converse implication also holds. 

\begin{definition} \label{definition:primeisogenies}
Let $A$ and $B$ be a complex abelian varieties and let $p$ a prime number. We say that $A$ is \textit{prime-to-$p$ isogenous to a $B$} if there is an isogeny $\alpha \colon A \to B$ whose degree $\deg(\alpha)$ is not divisible by $p$. We say that $A$ is \textit{$p$-power isogenous to $B$} if $A$ is isogenous to $B$ for some isogeny $\alpha$ whose degree is a power of $p$. 
\end{definition}
\noindent
%Evidently, if $p \neq \ell$ are prime numbers and if $A$ is $\ell$-power isogenous to $B$, then $A$ is prime-to-$p$ isogenous to $B$. In this way we will use the following lemma:
The following proposition shows in particular that to prove the density part of the statement in Theorem~\ref{introth:density}, it suffices to prove that for any prime number $\ell$, those abelian varieties that are $\ell$-power isogenous to a product of elliptic curves are dense in their moduli space.

\begin{proposition} \label{proposition:IHCONELEMMA}
Let $A$ be a complex abelian variety of dimension $g$. Let $\wh A$ be the dual abelian variety and let $\ca P_A$ be the Poincar\'e bundle. Let $\kappa$ be a non-zero integer such that the cohomology class $\kappa \cdot c_1(\ca P_A)/(2g-1)! \in \rm H^{4g-2}(A \times \wh A, \ZZ)$ is algebraic. Consider the following statements:
\begin{enumerate}
    \item \label{itemlemma:one} The abelian variety $A$ satisfies the integral Hodge conjecture for one-cycles.
    \item \label{itemlemma:two} For every prime $p$, there is an abelian variety $B$ such that the abelian variety $A \times B$ is prime-to-$p$ isogenous to the Jacobian of a smooth projective curve. % there exists a smooth projective curve $C$ and an isogeny 
    %\[
    %\alpha \colon A \times B \to J(C)
    %\]
    %whose degree $\deg(\alpha)$ is not divisible by $p$. 
    \item \label{itemlemma:twosmall} For every prime $p$ that divides $\kappa$, there is an abelian variety $B$ such that the abelian variety $A \times B$ is prime-to-$p$ isogenous to a Jacobian of a smooth projective curve. 
    \item \label{itemlemma:three} For every prime $p$, there is an abelian variety $B$ such that the abelian variety $A \times B$ is prime-to-$p$ isogenous to a product of Jacobians of smooth projective curves.
     \item \label{itemlemma:threesmall} For every prime number $p$ dividing $\kappa$, there exists an abelian variety $B$ such that the abelian variety $A \times B$ is prime-to-$p$ isogenous to a product of Jacobians of smooth projective curves. 
\end{enumerate}
The $[\ref{itemlemma:two} \implies \ref{itemlemma:twosmall} \implies \ref{itemlemma:threesmall} \implies \ref{itemlemma:one}]$ and $[\ref{itemlemma:two} \implies \ref{itemlemma:three} \implies \ref{itemlemma:threesmall}]$. Moreover, if $A$ is principally polarized by $\theta_A \in \NS(A)$, then \ref{itemlemma:one}  is implied by 
\begin{enumerate}
\setcounter{enumi}{5}
    \item \label{itemlemma:zero} For any prime number $p | (g-1)!$ there exists a smooth projective curve $C$ and a morphism of abelian varieties $\phi\colon A \to J(C)$ such that $\phi^\ast \theta_{J(C)} = m\cdot \theta_A$ for $m \in \ZZ_{\geq 1}$ with $\gcd(m,p) = 1$. 
\end{enumerate}
Finally, if $A$ is principally polarized of Picard rank one, then the statements $\ref{itemlemma:one} - \ref{itemlemma:zero} $ are equivalent.  
\end{proposition}

\begin{proof}
\textbf{Step one}: $[\ref{itemlemma:two} \implies \ref{itemlemma:twosmall} \implies \ref{itemlemma:threesmall}]$ \textit{and} $[\ref{itemlemma:two} \implies \ref{itemlemma:three} \implies \ref{itemlemma:threesmall}]$. This is trivial. 
\\
\\
\textbf{Step two}: $[\ref{itemlemma:threesmall} \implies \ref{itemlemma:one}]$. %Suppose that \ref{itemlemma:threesmall}  holds. 
Let $g$ be the dimension of $A$. We want to prove that the class $c_1(\ca P_A)^{2g-1}/(2g-1)! \in \rm H^{4g-2}(A \times \wh A, \ZZ)$ is algebraic. Let $p$ be any prime number that divides $\kappa$. Then by Condition \ref{itemlemma:threesmall}, there exists an abelian variety $B$ and an isogeny $\alpha\colon A \times B \to Y$ to the product $Y = \prod_iJ(C_i)$ of Jacobians $J(C_i)$ of smooth projective curves $C_i$ such that $\gcd(\deg(\alpha), p) = 1$. Define $X = A \times B$. Let $g_B$ be the dimension of $B$, let $h = g + g_B = \dim(X) = \dim(Y)$, and let $m_p = \deg(\alpha)$. There exists an isogeny $\beta \colon Y \to X$ such that %$\alpha \circ \beta = [m_p]_{Y}$ and 
$\beta \circ \alpha = [m_p]_X$. If we define $n_p = \deg(\beta)$ then $$m_p \cdot n_p = \deg(\alpha) \cdot \deg(\beta) = \deg(\alpha \circ \beta) = m_p^{2h},$$ %In particular, $\gcd(p, n_p) = 1$, and we also have that $\wh \alpha \circ \wh \beta = [m_p]_{\wh X}$. 
%The composition $(\beta \times \wh \alpha) \circ (\alpha \times \wh \beta) \colon X \times \wh X \to Y \times \wh Y \to X \times \wh X$ 
%In particular, $\Ker(\beta) \subset \Ker([m_p]) \cong (\ZZ/m_p)^{2h}$. Let $d = \deg(\beta)$. It follows that $d | m_p^{2h}$; let $r \in \ZZ_{\geq 1}$ be such that $m_p^{2h} =  r \cdot d$. If $\wh \beta\colon \wh X \to \wh{Y}$ is the dual of $\beta$, then $\wh \beta$ is again an isogeny and the kernels of $\beta$ and $\wh \beta$ are Cartier dual to one another - in particular, $\deg(\wh \beta ) = \deg(\beta) = d$. Therefore, there exists an isogeny $\gamma\colon \wh{Y} \to \wh X$ such that $\gamma \circ \wh \beta = [d]_{\wh{X}}$. The composition
%\[
%\xymatrixcolsep{7pc}
%\xymatrix{
%X \times \wh X \ar[r]^{\alpha \times \wh \beta} & Y \times \wh{Y} \ar[r]^{\beta \times \wh \alpha} & X \times %\wh{X}
%}
%\]
hence $(\beta \circ \alpha) \times (\wh \alpha \circ \wh \beta) = [m_p]_{X \times \wh X}$. 
%is the map 
%\[
%\left(m_p^{2h-1} \cdot (\beta \circ \alpha) \right) \times \left( r \cdot (\gamma \circ \wh \beta) \right) = [m_p^{2h}]_X \times [ r\cdot d]_{\wh X} = [m_p^{2h}]_{X \times \wh X}.
%\]
For $N_p = 2h \cdot (4h-2)$, the homomorphism$$[m_p^{2h}]^{\ast} = (m_p^{N_p} \cdot (-) ) \colon \rm H^{4h-2}(X \times \wh X, \ZZ) \to \rm H^{4h-2}(X \times \wh X, \ZZ)$$ 
%\quad x \mapsto m_p^{N_p} \cdot x 
%\]
will therefore factor through $\rm H^{4h-2}(Y \times \wh{Y}, \ZZ)$. Since $Y \times \wh{Y}$ satisfies the integral Hodge conjecture by Theorem \ref{introth:IHCforjacobians}, the Hodge class $$m_p^{N_p} \cdot c_1(P_X)^{2h-1}/(2h-1)! \in \rm H^{4h-2}(X \times \wh X , \ZZ)$$ is algebraic. %Now consider the abelian variety
%Z\coloneqq A \times B \times \wh A \times \wh B = X \times \wh X.
%\]
Let $f\colon A \times B \times \wh A \times \wh B \to A \times \wh A$ and $g \colon A \times B \times \wh A \times \wh B \to B \times \wh B$ be the canonical projections. Then $\ca P_X \cong f^\ast \ca P_A \otimes g^\ast \ca P_B$. Using this and denoting $\mu = c_1(\ca P_A)$ and $\nu = c_1(\ca P_B)$ we have
\[
    \frac{c_1(\ca P_{X})^{2h-1}}{(2h-1)!} =f^{\ast}\left(\frac{\mu^{2g-1}}{(2g-1)!}\right) \cdot g^\ast \left( \frac{\nu^{2g_B}}{(2g_B)!} \right) + f^\ast \left( \frac{\mu^{2g}}{(2g)!} \right) \cdot  g^\ast\left(\frac{\nu^{2g_B-1}}{(2g_B-1)!}    \right).\] 
    This implies that 
    %With respect to the morphism $j\colon A \times \wh A \to X \times \wh X$ defined as $(a,b) \mapsto (a,0,b,0)$, %One has $p \circ j = \id$ and $q \circ j = 0$. 
$$f_\ast \left(c_1(\ca P_X)^{2h-1}/(2h-1)! \right) = (-1)^{g_b}\mu^{2g-1}/(2g-1)!.$$
%\begin{align*}
%j^\ast \left(\frac{c_1(\ca P_X)^{2h-1}}{(2h-1)!} \right) &=  j^\ast %\left(p^{\ast}\left(\frac{\mu^{2g-1}}{(2g-1)!}\right)\right) + j^\ast %\left(q^\ast\left(\frac{\nu^{2g_B-1}}{(2g_B-1)!}    \right)\right) = %\frac{\mu^{2g-1}}{(2g-1)!},
%\end{align*}
In particular, the class $m_p^{N_p} \cdot c_1(\ca P_A)^{2g-1}/(2g-1)! \in \rm H^{4g-2}(A \times \wh A, \ZZ)$ is algebraic. Let $p_1, \dotsc, p_n$ be all prime divisors of $\kappa$ and observe that $$\gcd(\kappa, m_{p_1}^{N_{p_1}},m_{p_2}^{N_{p_2}}, \dotsc, m_{p_n}^{N_{p_n}}) = 1.$$ 
%for if $\ell$ is a prime number that divides $k$, then $\ell = p_i$ for some $i$ hence $\ell$ does not divide $m_{p_i}^{N_{p_i}}$. 
%This means, by B\'ezout's identity, that 
Thus, there are integers $a, b_1, \dotsc, b_n$ such that $a \cdot \kappa + \sum_{i = 1}^n b_i \cdot m_{p_i}^{N_{p_i}} = 1$, and
\[
\frac{c_1(\ca P_A)^{2g-1}}{(2g-1)!} =
a \cdot \kappa \cdot \frac{c_1(\ca P_A)^{2g-1}}{(2g-1)!}  + \sum_{i = 1}^n b_i \cdot m_{p_i}^{N_{p_i}} \cdot \frac{c_1(\ca P_A)^{2g-1}}{(2g-1)!} \in \rm H^{4g-2}(A \times \wh A, \ZZ).
\]
This proves that $c_1(\ca P_A)^{2g-1}/(2g-1)!$ is a $\ZZ$-linear combination of algebraic classes, hence algebraic. Condition \ref{itemlemma:one} follows then from Theorem \ref{maintheorem}.
\\
\\
\textbf{Step three}: [$\ref{itemlemma:zero} \implies \ref{itemlemma:one}]$ \textit{for $A$ principally polarized by $\theta_A \in \NS(A)$}. Let $p_1,\dotsc, p_k$ be the prime factors of $(g-1)!$ and let $C_1,\dotsc, C_k$ be smooth proper curves for which there exist homomorphisms $\phi_i \colon A \to J(C_i)$ such that $\phi^\ast \theta_{J(C_i)} = m_i \cdot \theta_A$ for some $m_i \in \ZZ_{\geq 1}$ with $p_i \nmid m$. Since $\theta_{J(C_i)}^{g-1}/(g-1)! \in \rm H^{2g-2}(J(C_i),\ZZ)$ is algebraic for each $i$, the classes $$\phi_i^\ast(\theta_{J(C_i)}^{g-1}/(g-1)!) = m_i^{g-1} \cdot \theta_{A}^{g-1}/(g-1)! \in \rm H^{2g-2}(A,\ZZ)$$ are algebraic. Since $\gcd((g-1)!, m_1,\dotsc, m_k) = 1$, this implies that $\theta_A^{g-1}/(g-1)!$ is algebraic. Condition \ref{itemlemma:one} follows then from Theorem \ref{maintheorem}. 
\\
\\
\textbf{Step four}: [$\ref{itemlemma:zero} \impliedby \ref{itemlemma:one} \implies \ref{itemlemma:two}]$ \textit{for $(A,\theta_A)$ principally polarized with $\rho(A) =1$}. %Suppose that \ref{itemlemma:one} holds, that $A$ is principally polarized by $\theta \in \textnormal{NS}(A)$ and that $\rho(A) = 1$. 
Write $\theta = \theta_A$. Let $Z_1, \dotsc, Z_n$ be integral curves $Z_i \subset A$ and let $e_1, \dotsc, e_n \in \ZZ$ with $e_i \neq 0$ for all $i$ be such that 
$$
{\theta^{g-1}}/{(g - 1)!} = \sum_{i = 1}^n e_i\cdot [Z_i] \in \rm H^{2g - 2}(A, \ZZ).$$ Since $\rho(A) = 1$, the group $\Hdg^{2g-2}(A, \ZZ)$ is generated by $\theta^{g-1}/(g-1)!$. Consequently, we have $[Z_i] = f_i\cdot \left(\theta^{g-1}/(g-1)!\right)$ for some non-zero $f_i \in \ZZ$. Hence we can write 
$$
{\theta^{g-1}}/{(g - 1)!} = \sum_{i = 1}^n e_i\cdot [Z_i]  = \sum_{i = 1}^n e_i\cdot f_i \cdot \theta^{g-1}/(g - 1)!
$$ which implies $\sum_{i = 1}^n e_i\cdot f_i  = 1$. Now let $p$ be any prime number. Then there exists an integer $i$ with $1 \leq i \leq n$ such that $p$ does not divide $f_i$. Let $C_i \to Z_i$ be the normalization of $Z_i$ and let $\lambda_A = \varphi_{\theta} \colon A \to \wh A$ be the polarization corresponding to $\theta$. This gives a diagram
\begin{equation}
\label{diagram:embedintojacobinian}
\xymatrixcolsep{4pc}
\xymatrix{
C_i \ar[rr]^{\varphi} \ar[dr]^{\iota}&& A \ar@/_3pc/[rrr]^{\phi} \ar[r]_{\sim}^{\lambda_A} & \wh A \ar[r]^{\varphi^{\ast}} & \Pic^0(C_i) \ar[r]^{a}_{\sim} & J(C_i), \\
& J(C_i) \ar[ur]^{\psi} & & &  &
}
\end{equation}
where $$\iota\colon C_i \to J(C_i) = H^0(C, \Omega_C)^\ast/H_1(C,\ZZ)$$ is the Abel--Jacobi map (for some $p \in C$), and $\varphi^\ast\colon \wh A = \Pic^0(A) \to \Pic^0(C_i)$ is the pullback of line bundles along $\varphi\colon C_i \to A$. The natural homomorphism $a\colon \Pic^0(C_i) \to J(C_i)$ is an isomorphism by the Abel--Jacobi theorem. Since the triangle on the left in Diagram (\ref{diagram:embedintojacobinian}) commutes and $[Z_i] \in \rm H^{2g-2}(A, \ZZ)$ is non-zero, the morphism $\psi \colon J(C_i) \to A$ is non-zero. As $\rho(A) = 1$, the map $\psi\colon J(C_i) \to A$ must be surjective, the Picard rank of a non-simple abelian variety being greater than one. Dually, $\psi$ gives rise to a non-zero homomorphism $\wh \psi\colon \wh A \to \wh{J(C_i)}$, and the simpleness of $\wh A$ implies that $\wh \psi$ is finite onto its image. 

We claim that the same is true for $\phi$. To prove this, it suffices to show that the kernel of $\varphi^\ast \colon \wh A \to \Pic^0(C_i)$ is finite. Since the homomorphism $\iota^\ast\colon \wh{J(C_i)} \to \Pic^0(C_i)$ induced by the embedding $\iota\colon C_i \to J(C_i)$ is an isomorphism, dualizing the triangle on the left in Diagram (\ref{diagram:embedintojacobinian}) proves our claim. 
%\[
%\xymatrixcolsep{4pc}
%\xymatrix{
%\Pic^0(C_i)  && A\ar[ll]_{\varphi^\ast} \ar[dl]_{\wh \psi} \\
%& \wh{J(C_i)}. \ar[ul]_{\iota^\ast} & 
%}
%\]

By construction, we have 
$$\varphi_{\ast}[C_i] = [Z_i] = f_i \cdot \theta^{g-1}/(g-1)! \in \rm H^{2g-2}(A, \ZZ).$$
By a version of Welters' Criterion (see \cite[Lemma 12.2.3]{birkenhake}), this implies that $\phi^{\ast}\left(\theta_{J(C_i)}\right) = f_i \cdot \theta \in \rm H^2(A, \ZZ)$, where $\theta_{J(C_i)} \in \rm H^2(J(C_i), \ZZ)$ is the canonical principal polarization. In particular, \ref{itemlemma:zero} holds. 

We claim that also \ref{itemlemma:two} holds. Let $j\colon A_0 \hookrightarrow J(C_i)$ be the embedding of $A_0 = \phi(A)$ into $J(C_i)$ and let $\lambda_0 \colon A_0 \to \wh A_0$ be the polarization on $A_0$ induced by $j$. We have $\phi^{\ast}(\lambda) = \varphi_{f_i\cdot \theta} = f_i \cdot \varphi_\theta = f_i \cdot \lambda_A$. We obtain a commutative diagram
\[
\xymatrixcolsep{5pc}
\xymatrix{
&A\ar[dl]_{[f_i]_A} \ar[r]^{\pi}\ar[d]^{f_i\cdot \lambda_A} &A_0\ar[d]^{\lambda_0} \ar[r]^j & J(C_i) \ar[d]^{\lambda} \\
A & \wh A \ar[l]^{\lambda_{\wh A}} & \wh A_0 \ar[l]^{\wh \pi} & \widehat{J(C_i)}. \ar[l]\ar[l]^{\wh j} 
}
\]
Let $G$ be the kernel of $\pi$. Define $$K = \Ker([f_i]_A) = \Ker(f_i\cdot \lambda_A) \cong (\ZZ/f_i)^{2g} \subset A,$$ and $U = \Ker(\wh \pi \circ \lambda_0) \subset A_0$. Also define $H = \Ker(\lambda_0)$, and observe that $H \subset U$. The exact sequence $$0 \to G \to K \to U \to 0$$ shows that if $a, k, u$ and $h$ are the respective orders of $G$, $K$, $U$ and $H$, then one has
\begin{equation} \label{eq:dividingorders}
h | u | k | f_i \quad \textnormal{ and } \quad a | k | f_i.
\end{equation}
Then define $B = \Ker( \wh j \circ \lambda ) \subset J(C_i)$ with inclusion $i\colon B \hookrightarrow J(C_i)$. It is easy to see that $B$ is connected. %The exact sequence $0 \to A_0 \to J(C_i) \to J(C_i)/A_0 \to 0$ induces an exact sequence $0 \to \widehat{J(C_i)/A_0} \to \wh{J(C_i)} \to \wh{A_0} \to 0$ which proves that $B \cong \widehat{J(C_i)/A_0}$ is connected. 
Moreover, we have $A_0 \cap B = H$
%\[
%A_0 \cap B = \left\{ a \in A_0 \mid (\wh j \circ \lambda \circ j)(a) = 0 \right\} 
%= \left\{ a \in A_0 \mid \lambda_0(a) = 0 \right\} = H,
%\] 
and, therefore, an exact sequence of commutative group schemes $$0 \to H \to A_0 \times B \xrightarrow{\psi} J(C_i) \to 0.$$ 
The morphism $\alpha\colon A \times B \to J(C_i)$, defined as the composition
$$
\xymatrixcolsep{2pc}
\xymatrix{
A \times B \ar[r]^{\pi \times \id}  & A_0 \times B \ar[r]^{\psi} & J(C_i),}$$ is an isogeny. Since the degree of an isogeny is multiplicative in compositions, we have $$\deg(\alpha) = \deg\left(\psi \circ (\pi \times \id) \right) =  \deg(\psi) \cdot \deg(\pi \times \id)  = h \cdot \deg(\pi) = h \cdot a.$$ In particular, $p$ does not divide $\deg(\alpha)$ because $h$ and $a$ divide $f_i$ by Equation (\ref{eq:dividingorders}).
%We conclude that Condition \ref{itemlemma:two} is satisfied, and we are done.
\end{proof}

\begin{proof}[Proof of Theorem \ref{introth:density}]
According to Theorem \ref{maintheorem}, it suffices to show that the cohomology class 
$$
\frac{c_1(\ca P_A)^{2g-1}}{(2g-1)!} \in \rm H^{4g-2}(A \times \wh A, \ZZ)
$$ is algebraic for $[(A,\lambda)]$ in a dense subset $X$ of $\msf A_{g,\delta}(\CC)$ as in the statement. Define $D = \textnormal{diag}(\delta_1, \dotsc, \delta_g)$ and define, for each subring $R$ of $\CC$, a group
\begin{align}\label{spdeltagroup}
\textnormal{Sp}_{2g}^\delta(R) = \left\{M \in \GL_{2g}(R) \mid  M 
\begin{pmatrix} 
0 & D \\
-D & 0
\end{pmatrix} M^t = 
\begin{pmatrix} 
0 & D \\
-D & 0
\end{pmatrix}
\right\}. 
\end{align}
The isomorphism
\[
\textnormal{Sp}_{2g}^\delta(\RR)  \to \textnormal{Sp}_{2g}(\RR), \quad M \mapsto 
\begin{pmatrix} 
1_g & 0 \\
0 & D
\end{pmatrix}^{-1} M \begin{pmatrix} 
1_g & 0 \\
0 & D
\end{pmatrix}
\]
induces an action of $\textnormal{Sp}_{2g}^\delta(\ZZ)$ on the genus $g$ Siegel space $\bb H_g$, and the period map defines an isomorphism of complex analytic spaces $\msf A_{g,\delta}(\CC) \cong \textnormal{Sp}_{2g}^\delta(\ZZ) \setminus \bb H_g$ \cite[Theorem 8.2.6]{birkenhake}. Pick any prime number $\ell > (2g-1)!$ and consider, for a period matrix $x \in \bb H_g$, the orbit $\textnormal{Sp}_{2g}^\delta(\ZZ[1/\ell]) \cdot x \subset \bb H_g$. Let $(A, \lambda)$ be a polarized abelian variety admitting a period matrix equal to $x$. The image of $\textnormal{Sp}_{2g}^\delta(\ZZ[1/\ell]) \cdot x$ in $\msf A_{g,\delta}(\CC)$ is the \textit{Hecke-$\ell$-orbit} of $[(A, \lambda)] \in \msf A_{g,\delta}(\CC)$, i.e.\ the set of isomorphism classes of polarized abelian varieties $[(B, \mu)] \in \msf A_{g,\delta}(\CC)$ for which there exists integers $n,m\in \ZZ_{\geq 0}$ and an isomorphism of polarized rational Hodge structures $\phi\colon \rm H_1(B, \QQ) \xrightarrow{\sim} \rm H_1(A, \QQ)$ 
%$\psi\colon A \dashrightarrow B$ 
such that $\ell^n \cdot \phi$ and $\ell^m \cdot \phi^{-1}$ are morphisms of integral Hodge structures (Hecke orbits were studied in positive characteristic in e.g.\ \cite{chaiordinaryhecke, oorthecke}). The degree of the isogeny $\alpha = \ell^n\phi$ must be $\ell^k$ for some nonnegative integer $k$. %Therefore, all abelian varieties in one Hecke-$\ell$-orbit are $\ell$-power isogenous to one-another. 
In particular, if one abelian variety in a Hecke-$\ell$-orbit happens to be isomorphic to a Jacobian, then every abelian variety in that orbit is $\ell$-power isogenous to a Jacobian, see Definition \ref{definition:primeisogenies}. %A posteriori, these abelian varieties are then all prime-to-$p$ isogenous to a Jacobian for any prime number $p$ that divides $(2g-1)!$ and by Proposition \ref{proposition:IHCONELEMMA} this means that they all satisfy the integral Hodge conjecture for one-cycles. 
The decomposition of a polarized abelian variety into non-decomposable polarized abelian subvarieties is unique \cite[Corollaire 2]{debarreproduits}, which implies that the morphism
\begin{align*}
\pi \colon &\prod_{i = 1}^g \msf A_{1,1} \to \msf A_{g,\delta}, \\
&\left([(E_1, \lambda_1)], \dotsc, [(E_g, \lambda_g)] \right) \mapsto ([E_1 \times \cdots \times E_g, \delta_1\cdot \lambda_1 \times \cdots \times \delta_g\cdot \lambda_g)]
\end{align*}
is finite onto its image. Thus $\msf A_{g, \delta}$ contains a $g$-dimensional subvariety on which the integral Hodge conjecture for one-cycles holds. Let $V = \pi\left( \prod_{i = 1}^g \msf A_{1,1}\right) \subset \msf A_{g,\delta}$. Then 
\[
X' \coloneqq \Sp_{2g}^\delta(\ZZ[1/\ell]) \cdot V =  \cup_iZ_i \subset \msf A_{g, \delta}(\CC)
\]
is a countable union of closed analytic subsets $Z_i \subset \msf A_{g, \delta}(\CC)$ of dimension $\dim Z_i \geq g$ such that, by Lemma \ref{strongapproximationlemma} below, $X' \subset \msf A_{g, \delta}(\CC)$ is dense in the analytic topology, and that $c_1(\ca P_A)^{2g-1}/(2g-1)! \in \rm H^{4g-2}( A \times \wh A, \ZZ)$ is algebraic for every polarized abelian variety $(A, \lambda)$ of polarization type $\delta$ whose isomorphism class lies in $X'$. To prove the theorem, we are reduced to proving that there exists a similar countable union $X \subset \msf A_{g, \delta}(\CC)$ whose components are algebraic. For this, it suffices to prove:

\textit{Claim:} The locus of $[(A, \lambda)] \in \msf A_{g, \delta}(\CC)$ such that $$c_1(\ca P_A)^{2g-1}/(2g-1)! \in \rm H^{4g-2}( A \times \wh A, \ZZ)_{\textnormal{alg}}$$ is a countable union $$W = \cup_jY_j \subset \msf A_{g, \delta}(\CC)$$ of closed algebraic subsets $Y_j \subset \msf A_{g, \delta}(\CC)$. 

Indeed, let $U \to \ca A_{g, \delta}$ be a finite \'etale cover of the moduli stack $\ca A_{g, \delta}$ and let $\ca X \to U$ be the pullback of the universal family of abelian varieties along $U \to \ca A_{g, \delta}$. This gives an abelian scheme $\ca X \times \wh{\ca X} \to U$ carrying a relative Poincar\'e line bundle $\ca P_{\ca X/U}$ and arguments similar to those used to prove Lemma \ref{lemma_algebraicityminimalclass} show that indeed, for each irreducible component $U' \subset U$, the locus in $U'(\CC)$ where $c_1(\ca P_A)^{2g-1}/(2g-1)!$ is algebraic is a countable union of closed algebraic subvarieties of $U'(\CC)$. We have $X' \subset W$ and %since each %$g \cdot V$ for $g \in \Sp_{2g}^\delta(\ZZ[1/\ell])$ is irreducible, so that each 
since each $Z_i \subset X$ is irreducible, each $Z_i$ is contained in an irreducible component $Y_j \subset W$. We may then define $X$ as the union of those $Y_j \subset W$ that contain some $Z_i$. 

Finally, Theorem \ref{maintheorem} implies that for each $[(A, \lambda)] \in X$, the integral Hodge conjecture for one-cycles holds for the abelian variety $A$, so we are done. 
\end{proof}

\begin{lemma} \label{strongapproximationlemma}
Let $\ell$ be a prime number and let $\Sp_{2g}^\delta(\ZZ[1/\ell]) \subset \Sp_{2g}(\RR)$ be the group defined in (\ref{spdeltagroup}). Then $\Sp_{2g}^\delta(\ZZ[1/\ell])$ is analytically dense in $\Sp_{2g}(\RR)$. 
\end{lemma}
\begin{proof}
Since $\Sp_{2g}^{\delta}(\QQ)$ arises as the group of rational points of an algebraic subgroup $\Sp_{2g}^\delta$ of $\GL_{2g}$ over $\mathbb Q$ \cite[Chapter 2, \S 2.3.2]{PlatonovRapinchuk}, which is isomorphic to $\Sp_{2g}$ over $\QQ$, %symplectic vector spaces of the same dimension are isomorphic.  
the lemma follows from the well-known fact that for $S = \{\ell\} \subset \Spec(\ZZ)$, the algebraic group $\Sp_{2g}$ satisfies the strong approximation property with respect to $S$ \cite[Chapter 7, \S 7.1]{PlatonovRapinchuk}. The latter is classical and follows from the non-compactness of $\Sp_{2g}(\QQ_\ell)$, see \cite[Theorem 7.12]{PlatonovRapinchuk}. 
\end{proof}

\begin{remark} \label{rem:dimensionimprovement}
Using level structures one can show that whenever $\gcd(\prod_i\delta_i, (2g-1)!) = 1$ (or, more generally, $\gcd(\prod_i\delta_i, (2g-2)!) = 1$, see Section \ref{sec:integralhodgeuptofactorn} below), there is a countable union $X = \cup_iZ_i \subset \msf A_{g, \delta}(\CC)$ as in Theorem \ref{introth:density} such that $\dim Z_i \geq 3g-3$. Indeed, let $\msf A_{g, \delta_g}^\ast$ be the moduli space of principally polarized abelian varieties of dimension $g$ with $\delta_g$-level structure. Then there is a natural morphism $\phi \colon \msf A_{g, \delta_g}^\ast \to \msf A_{g, \delta}$ such that for any $ x= [(A, \lambda)] \in \msf A_{g, \delta_g}^\ast(\CC)$ with $[(B, \mu)] = \phi(x) \in \msf A_{g, \delta}(\CC)$, there exists an isogeny $\alpha\colon A \to B$ of degree $\prod_{i = 1}^g \delta_i$, see \cite{mumfordmoduliofabelianvarieties}.
\end{remark}

\begin{remark}
In the principally polarized case, the density in the moduli space of those abelian varieties that satisfy the integral Hodge conjecture for one-cycles admits another proof which might be interesting for comparison. Let $\msf A_g$ be the coarse moduli space of principally polarized complex abelian varieties of dimension $g$ and let $[(A,\theta)]$ be a closed point of $\msf A_g$. Then by \cite[Exercise 5.6.(10)]{birkenhake}, the following are equivalent: (i) $A$ is isogenous to the $g$-fold self-product $E^g$ for an elliptic curve $E$ with complex multiplication, (ii) $A$ has maximal Picard rank $\rho(A) = g^2$, (iii) $A$ is \textit{isomorphic} to the product $E_1 \times \cdots \times E_g$ of pairwise isogenous elliptic curves $E_i$ with complex multiplication. If any of these conditions is satisfied, then $A$ satisfies the integral Hodge conjecture for one-cycles by Theorem \ref{introth:IHCforjacobians}, and the set of isomorphism classes of principally polarized abelian varieties $(A,\theta)$ for which this holds is dense in $\msf A_g$ by \cite{lange_modulprobleme}. For an explicit example in dimension $g = 4$ of a principally polarized abelian variety $(A, \theta)$ that satisfies one of the equivalent conditions above, but is not isomorphic to a Jacobian, see \cite[\S 5]{debarre_annulation}. 
\end{remark}

\subsection{The integral Hodge conjecture for one-cycles up to factor $n$}\label{sec:integralhodgeuptofactorn}

In this section, we study a property of a smooth projective complex variety that lies somewhere in between the integral Hodge conjecture and the usual (i.e.\ rational) Hodge conjecture. The key will be the following:

\begin{definition}
Let $d,k,n \in \ZZ_{\geq 1}$ and let $X$ be a smooth projective variety over $\CC$ of dimension $d$. Recall the definition of the degree $2d-2k$ \textit{Voisin group} of $X$ \cite{voisinstablyrational, perry2020integral}:
    \begin{align*}
    \rm Z^{2d-2k}(X) &\coloneqq \textnormal{Hdg}^{2d-2k}(X, \ZZ)/ \rm H^{2d-2k}(X,\ZZ)_{\textnormal{alg}}\\
    & = \Coker \left( \CH_k(X) \to \textnormal{Hdg}^{2d-2k}(X, \ZZ) \right).
    \end{align*}
 We say that \textit{$X$ satisfies the integral Hodge conjecture for $k$-cycles up to factor $n$} if $\rm Z^{2d-2k}(X)$ is annihilated by $n$ (in other words, if $n \cdot x \in \rm H^{2d-2k}(X,\ZZ)_{\textnormal{alg}}$ for every $x \in \textnormal{Hdg}^{2d-2k}(X, \ZZ)$). 
\end{definition}

%We then have the following generalization of Theorem \ref{maintheorem}:

%Recall that the Fourier transform on cohomology $\mr F_A \colon \rm H^\bullet(A,\ZZ) \to \rm H^\bullet(\wh A,\ZZ)$ induces an isomorphism $\mr F_A \colon \rm H^2(A,\ZZ) \xrightarrow{\sim} \rm H^{2g-2}(\wh A,\ZZ)$ identifying $\Hdg^{2}(A,\ZZ)$ with $\Hdg^{2g-2}(\wh A,\ZZ)$ \cite{beauvillefourier}. 

\begin{lemma} \label{lemma:integralhodgeupton}
Let $A$ be a complex abelian variety of dimension $g$. Define $\sigma_A \in \rm H^{4g-4}(A \times \wh A,\ZZ)$ to be the class $c_1(\ca P_A)^{2g-2}/(2g-2)!$. 
\begin{enumerate}
    \item  \label{firstiem}
Let $n$ be a positive integer and let $\ca F_n\colon \CH^1(\wh A) \to \CH_1(A)$ be a group homomorphism such that the following diagram commutes:
%commuting with the cycle class map and the homomorphism
\[
\xymatrixcolsep{5pc}
\xymatrix{
\rm \CH^1(\wh A) \ar[d]\ar[r]^{\ca F_n} & \CH_1(A)\ar[d] \\
\rm H^2(\wh A,\ZZ)\ar[r]^{n \cdot \mr F_{\wh A} } & \rm H^{2g-2}( A, \ZZ).
}
\]
Then $A$ satisfies the integral Hodge conjecture for one-cycles up to factor $n$. 
\item \label{secondiem} Let $n \in \ZZ_{\geq 1}$ be such that $n \cdot \sigma_A$ is algebraic. A homomorphism $\ca F_n$ % \colon \CH(\wh A)\to \CH(A)$ 
as in \ref{firstiem} exists. 
%If $n = (2g-2)!$ then such a homomorphism $\ca F_n\colon \CH(\wh A) \to \CH(A)$ exists.
\end{enumerate}
\end{lemma}
\begin{proof}
Statement \ref{firstiem} follows immediately from the fact that $\CH^1(\wh A) \to \Hdg^2(\wh A,\ZZ)$ is surjective by Lefschetz $(1,1)$. To prove \ref{secondiem}, first observe that if $$\sigma_{\wh A} \coloneqq c_1(\ca P_{\wh A})^{2g-2}/(2g-2)! \in \rm H^{4g-4}(\wh A \times A,\ZZ),$$ then $n \cdot \sigma_{\wh A}$ is algebraic since $n \cdot \sigma_A$ is. Let $\Sigma_n \in \CH_2(\wh A \times A)$ be such that $cl(\Sigma_n) = n \cdot \sigma_{\wh A}$. This gives a commutative diagram:
%define $\sigma_A = c_1(\ca P_A)^{2g-2}/(2g-2)! \in \rm H^{4g-4}(A \times \wh A,\ZZ)$ as before. The following diagram commutes:
\begin{equation*} %\label{diagramwithonlysurfaceclasses-part2}
\xymatrixcolsep{1.5pc}
\xymatrix{
\CH^1(\wh A) \ar[d]^{cl}\ar[r]^{\pi_1^\ast}& \CH^1(\wh A \times A) \ar[d]^{cl}\ar[r]^{\cdot \Sigma_n}& \CH^{2g-1}(\wh A \times A) \ar[d]^{cl} \ar[r]^{\pi_{2,\ast}} & \CH_1( A) \ar[d]^{cl} \\
\rm H^2(\wh A,\ZZ)
%\ar@/_2pc/[rrr]^{n \cdot \mr F_{\wh A}} 
\ar[r]^{\pi_1^\ast} & \rm H^{2}(\wh A \times A,\ZZ) \ar[r]^{\cdot n \cdot \sigma_{\wh A}} & \rm H^{4g-2}(\wh A \times  A,\ZZ) \ar[r]^{\pi_{2,\ast}} & \rm H^{2g-2}( A,\ZZ). 
}
\end{equation*}
\noindent
Since $$ \pi_{2,\ast} \circ \left((-)\cdot n \cdot \sigma_{\wh A}\right) \circ \pi_1^\ast = n \cdot \mr F_{\wh A},$$ the homomorphism $\ca F_n \coloneqq \pi_{2,\ast} \circ \left((-)\cdot \Sigma_n\right) \cdot \pi_1^\ast$ has the required property.  \end{proof}
%Since $\mr F_{\wh A}$ is an isomorphism, 
%and consider Diagram (\ref{diagramwithonlysurfaceclasses}). The composition $\pi_{2,\ast} \circ (- \cdot \sigma_A) \circ \pi_1^\ast$ on its bottom row is an isomorphism since it coincides with $\mr F_A$. 
%of (\ref{diagramwithonlysurfaceclasses})
%by Lefschetz $(1,1)$. 
%Suppose that there exists a homomorphism $\ca F\colon \CH(A) \to \CH(A)$ as in Equations (\ref{eq:integralfourierdef}) and (\ref{eq:taualphadef}). It is clear\footnote{\textcolor{red}{Is that clear? I mean $\exp(nx)$ is not $n\exp(x)$.}} that its $\QQ$-linear extension
%rees with $n \cdot \ca F_A$. Consequently, 
%Denote by $p \colon \CH(A) \to \CH_1(A)$ the canonical projection. Restrict $\ca F_n$ to the subgroup $\CH^1(\wh A) \subset \CH(\wh A)$ to obtain the following commutative diagram:
%\[
%\xymatrixcolsep{5pc}
%\xymatrix{
%\CH^1(\wh A)\ar[d] \ar[r]^{\ca F_n} & \CH(A) \ar[r]^p& \CH_1(A)\ar[d] \\
%\textnormal{Hdg}^2(\wh A, \ZZ) \ar[r]^{\mr F_{\wh A}}_{\sim} & \textnormal{Hdg}^{2g-2}(A, \ZZ) \ar[r]^{\cdot n} & \textnormal{Hdg}^{2g-2}(A, \ZZ) . 
%}
%\]
%It follows that $n \cdot \textnormal{Hdg}^{2g-2}(\wh A, \ZZ)  \subset \rm{H}^{2g-2}(\wh A, \ZZ)_{\textnormal{alg}}$. For the second claim, remark that the cohomology class $(2g-1)! \cdot \ch(\ca P_A) \in \rm H^2(A \times \wh A, \ZZ)$ is algebraic. 
%can be lifted to $\CH(A \times \wh A)$. 

\begin{theorem} \label{theorem:integralhodgeuptofactor}
Consider a complex abelian variety $A$ of dimension $g$. Define the cycle $\sigma_A $ in $ \rm H^{4g-4}(A \times \wh A,\ZZ)$ as before and define $\rho_A = c_1(\ca P_A)^{2g-1}/(2g-1)! \in \rm H^{4g-2}(A \times \wh A,\ZZ)$. 
%Let $cl(\rho)\in \textnormal{Hdg}^{4g-2}(A \times \wh A, \ZZ)$ be the cohomology class of $\rho \in \CH^{2g-1}(A \times \wh A)_\QQ$. 
\begin{enumerate} 
    \item \label{item:nphodgealg} 
    Let $n \in \ZZ_{\geq 1}$ be such that 
    %Let $n,m \in \ZZ_{\geq 1}$ be such that 
    $n \cdot \rho_A$ %\in \textnormal{Hdg}^{4g-2}(A \times \wh A, \ZZ)$ 
    is algebraic. Then $n^2 \cdot \sigma_A$ is algebraic. In particular, $A$ satisfies the integral Hodge conjecture up to factor $\gcd(n^2, (2g-2)!)$ in this case. 
%    and $m \cdot \sigma_A \in \Hdg^{4g-4}(A \times \wh A,\ZZ)$ are algebraic. %Define \[m(g,n) = \gcd\left( n^2, (2g-2)!  \right).\]
%    Then $A$ satisfies the integral Hodge conjecture for one-cycles up to factor $\gcd(n^2, m)$.
    %$m(g,n)$. 
    \item \label{item:ppavnphodgealg} If $A$ is principally polarized, and $ n \in \ZZ_{\geq 1}$ is such that $n \cdot \gamma_\theta \in \textnormal{Hdg}^{2g-2}(A, \ZZ)$ is algebraic, then $n \cdot \rho_A\in \textnormal{Hdg}^{4g-2}(A \times \wh A, \ZZ)$ is algebraic. 
    \item \label{item:ppavnphodgealg-last}% Any $g$-dimensional abelian variety $A_{/\CC}$ 
    $A$ satisfies the integral Hodge conjecture for one-cycles up to factor $(2g-2)!$, and Prym varieties satisfy the integral Hodge conjecture for one-cycles up to factor $4$. 
    % and if $A$ is principally polarized, then $A$ satisfies the integral Hodge conjecture for one-cycles up to factor $\gcd(((g-1)!)^2, (2g-2)!$. 
\end{enumerate}
\end{theorem}
\begin{proof}
\ref{item:nphodgealg}. %Let $i, j \in \ZZ_{\geq 1}$ with $i + j = 2g$. 
By Lemma \ref{lemma:crucialprop}, one has \[
\sigma_A = c_1(\ca P_A)^{2g-2}/(2g-2)! = (-1)^g\cdot \left( \rho_A\right)^{\star 2}/2! \in \rm H^{4g-4}(A \times \wh A,\ZZ). 
\]
By Theorem \ref{th:PDstructure}, this implies that if $n \cdot \rho_A$ is algebraic, then $n^2\cdot \sigma_A$ is algebraic. Since $(2g-2)! \cdot \sigma_A$ is algebraic, it follows that $\gcd(n^2, (2g-2)!) \cdot \sigma_A$ is algebraic. Thus we are done by Lemma \ref{lemma:integralhodgeupton}.
%\[
%n^j \cdot \frac{c_1(\ca P_A)^i}{i!} = \pm n^j \cdot  \frac{\rho_A^{\star j}}{j!} = \pm \frac{(n \cdot \rho_A)^{\star j}}{j!}  \in \rm H^\bullet(A \times \wh A, \ZZ),
%\]
%and this lifts to $\CH(A \times \wh A)$ by Theorem \ref{th:PDstructure}. In particular, the element $m(g,n) \cdot \ch(\ca P_A)$ in $\rm H^\bullet (A \times \wh A, \ZZ)$ lifts to a correspondence in $\Gamma \in \CH(A \times \wh A)$. Since the assignment
%\[
%\Hdg^\bullet(A \times \wh A, \ZZ) \to \Hom_{\Hdg}\left( \rm H^\bullet(A, \ZZ),\rm H^\bullet(\wh A, \ZZ) \right), \quad \alpha \mapsto \left[ x \mapsto \pi_{2,\ast} \left(\alpha \cdot \pi_1^\ast(x)\right)  \right]
%\]
%is a group homomorphism, the group homomorphism $\ca F \colon \CH(A) \to \CH(\wh A)$ defined by $\Gamma$ commutes with the cycle class map and the group homomorphism
%\[
%m(g,n)\cdot \mr F_A\colon \rm H^\bullet(A, \ZZ) \to \rm H^\bullet(\wh A, \ZZ). 
%\]
%Therefore, we have $m(g,n)\cdot \mr F_A \left( \textnormal{Hdg}^2(A, \ZZ) \right) = m(g,n)\cdot  \textnormal{Hdg}^{2g-2}(A, \ZZ) \subset \rm H^{2g-2}(\wh A, \ZZ)_{\textnormal{alg}}$. 

\ref{item:ppavnphodgealg}. This follows from Lemma \ref{lemma:minimalclasspoincarecomparison}.

\ref{item:ppavnphodgealg-last}. This follows from Lemma \ref{lemma:integralhodgeupton}, parts \ref{item:nphodgealg} and \ref{item:ppavnphodgealg} and the fact that if $A$ is a $g$-dimensional Prym variety with principal polarization $\theta \in \Hdg^2(A,\ZZ)$, then $2 \cdot \gamma_\theta \in \rm H^{2g-2}(A,\ZZ)$ is algebraic.  
%This means that there exists a homomorphism $\ca F\colon \CH(A) \to \CH(\wh A)$ such that the induced
\end{proof}

\section{The integral Tate conjecture}

Let $X$ be a smooth projective variety over the separable closure $k$ of a finitely generated field. Let $k_0$ be a finitely generated field of definition of $X$. A class $u \in \rm H_{\textnormal{\'{e}t}}^{2i}(X, \ZZ_\ell(i))$ is an \textit{integral Tate class} if it is fixed by some open subgroup of $\Gal(k/k_0)$. Totaro has shown that for codimension-one cycles on $X$, the Tate conjecture over $k$ implies the integral Tate conjecture over $k$ \cite[Lemma 6.2]{totaroIHCthreefolds}. This means that every integral Tate class is the class of an algebraic cycle over $k$ with $\ZZ_\ell$-coefficients. 

Suppose that $A_{/k}$ is an abelian variety, defined over a finitely generated field $k_0 \subset k$ such that $k$ is the separable closure of $k_0$. Then the Tate conjecture for codimension-one cycles holds for $A$ over $k$ by results of Tate \cite{tate_endomorphisms}, Faltings \cite{Faltings1983, faltings_wustholz}, and Zarhin \cite{zarhin_one, zarhin_two}. By the above, $A$ satisfies the integral Tate conjecture for codimension-one cycles over $k$. On the other hand, the Fourier transform defines an isomorphism
\begin{equation} \label{eq:fourieretalecohomology}
\mr F_A \colon \rm H_{\textnormal{\'{e}t}}^2(A, \ZZ_\ell(1)) \xrightarrow{\sim} \rm H_{\textnormal{\'{e}t}}^{2g-2}(\wh A, \ZZ_\ell(g-1)),
\end{equation}
see \cite[Section 7]{totaroIHCthreefolds}. Since (\ref{eq:fourieretalecohomology}) is Galois-equivariant (the Poincar\'e bundle being defined over $k_0$) it sends integral Tate classes to integral Tate classes. Therefore, to prove the integral Tate conjecture for one-cycles on $A$, it suffices to lift (\ref{eq:fourieretalecohomology}) to a homomorphism $\CH^1(A)_{\ZZ_\ell} \to \CH_1(\wh A)_{\ZZ_\ell}$. 

\begin{proof}[Proof of Theorem \ref{introth:integraltate}]
This follows from the above and Proposition %\ref{prop:motivicetale-new} and
\ref{remarketalebetti}.\ref{remark3.12.2}.
\end{proof}

\begin{comment}
The proof of Theorem \ref{introth:integraltate} shows that similarly to the complex case there is an analogue for abelian varieties that are not necessarily principally polarized:

\begin{corollary} \label{corollary:poincareinttate}
For an abelian variety $A$ over the separable closure of a finitely generated field, the integral Tate conjecture for one-cycles holds for $A$ and $\wh A$ provided that $\rho_A = c_1(\ca P_A)^{2g-1}/(2g-1)! \in \rm H^{4g-2}_{\textnormal{\'et}}(A \times \wh A, \ZZ_\ell(2g-1))$ is algebraic. The latter holds if and only if $\ch(\ca P_A) \in \oplus_{r \geq 0} \rm H_{\textnormal{\'{e}t}}^{2r} (A \times \wh A, \ZZ_\ell(r))$ is algebraic if and only if $A \times \wh A$ satisfies the integral Tate conjecture for one-cycles.
\end{corollary}
\begin{proof}
In the course of the proof of Theorem \ref{introth:integraltate}, we showed that if $\mr R_A$ admits a lift $\beta \in \CH_1(A \times \wh A)_{\ZZ_\ell}$, then $(-1)^g\cdot \rm E((-1)^g\cdot\beta) \in \CH(A \times \wh A)_{\ZZ_\ell}$ defines a lift of $\ch(\ca P_A) \in \oplus_{r \geq 0} \rm H_{\textnormal{\'{e}t}}^{2r} (A \times \wh A, \ZZ_\ell(r))$. The rest of the proof is similar to the proof of Theorem \ref{maintheorem}, see Section \ref{sec:proofmaintheorem}. 
\end{proof}
\end{comment}

%As in the complex case, we obtain

\begin{corollary}Let $A$ and $B$ be abelian varieties defined over the separable closure $k$ of a finitely generated field, of respective dimensions $g_A$ and $g_B$. 
\begin{enumerate}
    \item \label{item:integraltateone} The classes 
    \[
    \rho_A \in \rm H_{\textnormal{\'et}}^{4g_A-2}(A \times \wh A, \ZZ_\ell(2g_A-1)), \quad \tn{ and } \quad \rho_B \in \rm H_{\textnormal{\'et}}^{4g_B-2}(B \times \wh B, \ZZ_\ell(2g_A-1))
    \]
    are algebraic if and only if $A \times \wh A$, $B \times \wh B$, $A\times B$ and $\wh A \times \wh B$ satisfy the integral Tate conjecture for one-cycles over $k$. 
    \item  \label{item:integraltatetwo} If $A$ and $B$ are principally polarized, then the integral Tate conjecture for one-cycles over $k$ holds for $A\times B$ if and only if it holds for both $A$ and $B$. 
    %and $\theta_A$ and $\theta_B$ are the first Chern classes of the respective polarizations, then 
    %The integral Tate conjecture for one-cycles holds for $A$, $\wh A$, $B$ and $\wh B$ if and only if it holds for $A \times B$ and $\wh A \times \wh B$. 
%    \item \label{item:integraltatetwo} If the abelian variety $A$ satisfies the integral Tate conjecture for one-cycles over $k$ (which is equivalent to $\mr R_A$ algebraic, see Corollary \ref{corollary:poincareinttate}), then $\ch(\ca P_A) \in \oplus_{r \geq 0} \rm H_{\textnormal{\'{e}t}}^{2r} (A, \ZZ_\ell(r))$ is algebraic. 
    \item \label{item:integraltatethree} If $\theta \in \rm H^2_{\textnormal{\'et}}(A, \ZZ_\ell(1))$ is the first Chern class of an ample line bundle that induces a principal polarization on $A$, and if $\gamma_\theta = \theta^{g-1}/(g-1)! \in  \rm H_{\textnormal{\'{e}t}}^{2g-2}(A, \ZZ_\ell(g-1))$ is algebraic, then $\theta^i/i!$ in $\rm H_{\textnormal{\'{e}t}}^{2i}(A, \ZZ_\ell(i))$ is algebraic for each $i \in \ZZ_{\geq 1}$. 
\end{enumerate}
\end{corollary}

%Define $j_1 \colon X \to X \times Y$ by $j_1(x) = (x,0)$, and $j_2 \colon Y \to X \times Y$ via $j_2(y) = (0,y)$. Let $\alpha \in \CH_1(X)$ and $\beta \in \CH_1(Y)$ be cycles that give the minimal classes on $X$ and $Y$ respectively. %as in Theorem \ref{maintheorem}.\ref{item:cycle} and 
%Define 
%\[
%\gamma = j_{1,\ast}(\alpha) + j_{2,\ast}(\beta) \in \CH_1(X \times Y). \]
%If $\theta_X$ and $\theta_Y$ are principal polarizations on $X$ and $Y$, then $\theta=p_X^\ast \theta_X + p_Y^\ast \theta_Y$ is a principal polarization on $X \times Y$. Denote by $g_X$ and $g_Y$ the respective dimensions. We claim that $[\gamma] = \theta^{g_X+g_Y-1}/(g_X+g_Y-1)! \in \rm H^{2g_X +2g_Y -2}(X \times Y, \ZZ)$. Indeed, if $g = g_X + g_Y$, then
%\begin{align*}
%    \theta^{g-1}&={g_X+g_Y-1 \choose g_X}p_X^\ast\theta_X^{g_X} p_Y^{\ast}\wh{l}{g_Y-1} + {g_X+g_Y-1 \choose g_Y}p_X^\ast\theta_X^{g_X-1} p_Y^{\ast}\wh{l}{g_Y}\\
%    &=(g_X+g_Y-1)!\left(\frac{p_Y^\ast\theta^{g_Y-1}}{(g_Y-1)!} + \frac{p_X^{\ast}\theta^{g_X-1}}{(g_X-1)!}\right)\\
%    &=(g-1)!\left( j_{1,\ast}(\alpha) + j_{2,\ast}(\beta) \right)=(g-1)![\gamma] \in \rm H_{\textnormal{\'{e}t}}^{2g - 2}(X \times Y, \ZZ_\ell(g-1)).
%\end{align*}

\begin{proof}
\ref{item:integraltateone}. See Equation (\ref{eq1}). 

\ref{item:integraltatetwo}. This is true because the minimal cohomology class of the product is algebraic if and only if the minimal cohomology classes of the factors are algebraic.

\ref{item:integraltatethree}. One has $\theta^i/i! = \gamma_\theta^{\star (g-i)}/(g-i)!$ by \cite[Corollaire 2]{beauvillefourier}.
\end{proof}

\noindent
Combining Theorems \ref{maintheorem} and \ref{introth:integraltate}, we obtain:

\begin{corollary} \label{hodgetatecomparison}
Let $A_K$ be a principally polarized abelian variety over a number field $K\subset \CC$. Its base change $A_\CC$ over $\CC$ satisfies the integral Hodge conjecture for one-cycles if and only if $A_{\bar K}$ satisfies the integral Tate conjecture for one-cycles over $\bar K = \bar \QQ$.
\end{corollary}

\begin{proof}
%Let $K \subset L$ be a finite field extension such that the principal polarization $\lambda_L \colon A_L \to \wh A_L$ is induced by an ample line bundle $\ca L = \ca O_{A_L}(\Theta)$ on $A_L$. 
We view $\bar \QQ$ as a subfield of $\CC$ in a way compatible with the inclusion $ K \hookrightarrow \CC$. For a prime number $\ell$, let $\theta_\ell = c_1(\ca L) \in \rm H_{\textnormal{\'et}}^{2}(A_{\bar \QQ}, \ZZ_\ell(1))$ be the $\ell$-adic \'etale cohomology class of $\ca L$. On the other hand, define $\theta_\CC \in \NS(A) \subset \rm H^{2}(A_\CC, \ZZ)$ to be the polarization of the complex abelian variety $A_\CC$. By Theorems \ref{maintheorem} and \ref{introth:integraltate}, it suffices to show that $\gamma_{\theta_\CC} \in \rm H^{2g-2}(A_\CC, \ZZ)$ is algebraic if and only if $\gamma_{\theta_\ell} \in \rm H_{\textnormal{\'et}}^{2g-2}(A_{\bar \QQ}, \ZZ_\ell(g-1))$ is in the image of (\ref{eq:integraltateconjecture}) for each prime number $\ell$. Define $\rm R^{2g-2}(A) = \Coker\left(\CH_1(A_\CC) \to \rm H^{2g-2}(A_\CC,\ZZ)\right)$. This implies that $$\rm R^{2g-2}(A)\otimes \ZZ_\ell = \Coker\left(\CH_1(A_\CC)_{\ZZ_\ell} \to \rm H^{2g-2}(A_\CC,\ZZ_\ell)\right).$$ %by flatness of $\ZZ_\ell$ over $\ZZ$.
\noindent
Suppose that $\gamma_{\theta_\ell}$ is in the image of (\ref{eq:integraltateconjecture}) for every prime number $\ell$. Then the image of $\gamma_{\theta_\CC}$ in $\rm R^{2g-2}(A)\otimes \ZZ_\ell$ is zero for each prime number $\ell$, which implies that the image of $\gamma_{\theta_\CC}$ in $\rm R^{2g-2}(A)$ is zero, i.e. $\gamma_{\theta_\CC}$ is algebraic. 
%then (\ref{eq:integraltateconjecture}) is surjective for each prime $\ell$ by Theorem \ref{introth:integraltate}, which means that 
%Then $\rm Z^{2g-2}(A)\otimes \ZZ_\ell = 0$ for each prime number $\ell$ by Theorem \ref{introth:integraltate}, which means that $\rm Z^{2g-2}(A) = 0$. 
Conversely, suppose that $\gamma_{\theta_\CC} = \sum_{i= 1}^k n_i \cdot cl(C_i)$ for some smooth projective curves $C_i$ over $\CC$. %Let $C = C_i$ be any one of them. 
The Hilbert scheme $\ca H = \textnormal{Hilb}_{A_K/K}$ is defined over $K$; %any $\bar \QQ$-point in the connected component of $C \subset A$ gives a curve $C' \subset A$ 
for each $i = 1, \dotsc, k$ we pick a $\bar \QQ$-point in the connected component of $\ca H$ containing $[C_i \subset A]$. This gives smooth projective curves $C_i' \subset A_{\bar \QQ}$ over $\bar \QQ$ and % such that $cl(C'_{\CC})= cl(C) \in \rm H^{2g-2}(A_\CC, \ZZ)$. 
% where $C'_{\CC}$ is the base change of $C'$ to $\CC$ for some embedding $K' \to \CC$ that extends $\sigma$. 
 %see that there exist % a finite field extension $K \subset L$ and
%curves $C_i' \subset A_{\bar \QQ}$ %such that 
we define $\Gamma = \sum_i n_i \cdot [C_i'] \in \CH_1(A_{\bar \QQ})$. On the one hand, $cl(\Gamma_\CC) = \gamma_{\theta_\CC}$ by Lemma \ref{lemma_algebraicityminimalclass}. On the other hand, the Artin comparison theorem gives an isomorphism of $\ZZ_\ell$-algebras
\[
\phi\colon \rm H^{\bullet}_{\textnormal{\'et}}(A_{\bar \QQ}, \ZZ_\ell) = \rm H^{\bullet}_{\textnormal{\'et}}(A_{\CC}, \ZZ_\ell) \cong \rm H^{\bullet}(A_\CC, \ZZ) \otimes_\ZZ \ZZ_\ell.
\]
Since $\phi$ is compatible with the cycle class maps $\CH(A_{\bar \QQ}) \to \rm H^{\bullet}_{\textnormal{\'et}}(A_{\bar \QQ}, \ZZ_\ell)$ and $\CH(A_\CC) \to \rm H^{\bullet}(A_\CC, \ZZ)$, we have 
\[
\phi(\gamma_{\theta_\ell}) = \gamma_{\theta_\CC} \quad \tn{ and } \quad \phi(cl(\Gamma)) = cl(\Gamma_\CC)= \gamma_{\theta_\CC}.
\]
 Therefore, $cl(\Gamma) = \gamma_{\theta_\ell}$. 
%To see this, note that there exists an integral finite type $K$-algebra $R$, a closed subscheme $\ca C \subset X_R$ such that the fibers of $\ca C \to \Spec R$ are curves and such that 
% for every prime number $\ell$ if and only if $\gamma_\theta \in \rm H^{2g-2}(A, \ZZ)$ is the class of a one-cycle $\Gamma_\CC \in \CH(A_\CC)$. %We may assume that $g \geq 4$ \cite{grabowski, totaroIHCthreefolds} and under this assumption the line bundle $\ca L^{\otimes(g-1)}$ is very ample \cite{MumfordAV}. By Bertini's theorem the linear system $|\ca L^{\otimes(g-1)}|$ contains a smooth projective curve $C \subset A_{\bar \QQ}$ over $\bar \QQ$ which gives an $[C] \in \$
\end{proof}
\noindent
Another corollary of Theorem \ref{introth:integraltate} is that the integral Tate conjecture for one-cycles on abelian varieties is stable under specialization. For example, one has the following (c.f. Corollary \ref{complexspecialization}):

\begin{corollary} \label{cor:specialization}
%Let $A$ be an abelian scheme over the ring of integers $\OO_K$ of a number field $K$. %, let $A$ be an abelian variety over $K$ and let $\ca A$ be . 
Let $A_K$ be a principally polarized abelian variety over a number field $K$. Suppose that $A_{\bar K}$ satisfies the integral Tate conjecture for one-cycles over $\bar K$. 
%and let $\OO_K$ be the ring of integers of $K$. 
%Let $\eta = \Spec K$ be the generic point of $\Spec \OO_K$ and suppose that $A_{\bar K}$ satisfies the integral Tate conjecture for one-cycles over $\bar K = \bar \QQ$. 
Let $\mf p$ be a prime ideal of the ring of integers $\OO_K$ of $K$ at which $A_K$ has good reduction and write $\kappa = \OO_K/\mf p$. The abelian variety $A_{\bar \kappa}$ over $\bar \kappa$ satisfies the integral Tate conjecture for one-cycles over $\bar \kappa$. 
%for some embedding $\sigma \colon K \to \CC$, the abelian variety $A_\CC = A \times_{K, \sigma} \CC$ satisfies the integral Hodge conjecture for one-cycles. 
\end{corollary}

\begin{proof}
Write $S = \Spec(\OO_K)$ and let $A \to S$ be the N\'eron model of $A_K$. Let $R$ (resp.\ $K_{\mf p}$) be the completion of $\OO_K$ (resp.\ $K$) at the prime $\mf p$. The natural composition $K \to K_{\mf p} \to \bar K_{\mf p}$ induces an embedding $\bar K \to \bar K_{\mf p}$, where $\bar K_{\mf p}$ is an algebraic closure of $K_{\mf p}$. This gives a commutative diagram, where the square on the right is provided in \cite[Example 20.3.5]{fultonintersection}:
\begin{equation} \label{specializationdiagram}
\begin{split}
    \xymatrix{
\xymatrixcolsep{3.2pc}
\CH(A_{\bar K})_{\ZZ_\ell} \ar[r] \ar[d] & 
\CH(A_{\bar K_{\mf p}})_{\ZZ_\ell} \ar[r] \ar[d] &
\CH(A_{\bar \kappa})_{\ZZ_\ell} \ar[d]\\ 
\oplus_{r \geq 0}\rm H^{2r}_{\textnormal{\'et}}(A_{\bar K}, \ZZ_\ell(r))\ar[r]^{\sim} &\oplus_{r \geq 0}\rm H^{2r}_{\textnormal{\'et}}(A_{\bar K_{\mf p}}, \ZZ_\ell(r))\ar[r]^{\sim}&
\oplus_{r \geq 0}\rm H^{2r}_{\textnormal{\'et}}(A_{\bar \kappa}, \ZZ_\ell(r)).
}
\end{split}
\end{equation}
Now the principal polarization $\lambda_K\colon A_K \xrightarrow{\sim} \wh A_K$ extends uniquely to a homomorphism $\lambda\colon A \to \wh A$ by the N\'eron mapping property \cite[Section 1.2, Definition 1]{BLR} and since the same is true for the inverse $\lambda_K^{-1}\colon \wh A_K \xrightarrow{\sim} A_K$ we find that $\lambda$ is an isomorphism. In particular, we see that $A_{\bar \kappa}$ is principally polarized and that the class in $\CH^1(A_{\bar K})_{\ZZ_\ell}$ of a theta divisor on $A_{\bar K}$ is sent to the class in $\CH^1(A_{\bar \kappa})_{\ZZ_\ell}$ of a theta divisor on $A_{\bar \kappa}$. Thus, the minimal class $\gamma_{\theta_{\bar K}} \in \rm H^{2g-2}_{\textnormal{\'et}}(A_{\bar K}, \ZZ_\ell(g-1))$ is sent to the minimal class $\gamma_{\theta_{\bar \kappa}} \in \rm H^{2g-2}_{\textnormal{\'et}}(A_{\bar \kappa}, \ZZ_\ell(g-1))$ by the isomorphism on the bottom of (\ref{specializationdiagram}). It follows that $\gamma_{\theta_{\bar \kappa}}$ is algebraic; by Theorem \ref{introth:integraltate}, we are done.  
\end{proof}
\noindent
Finally, let us prove Theorem \ref{th:chinglichai}. The theorem follows from Theorem \ref{introth:integraltate} together with a result of Chai on the density of an ordinary isogeny class in positive characteristic \cite{chaiordinaryhecke}.

\begin{proof}[Proof of Theorem \ref{th:chinglichai}]
For any $t \in A_g(k)$, let $(A_t, \lambda_t)$ be a principally polarized abelian variety such that $[(A_t, \lambda_t)] = t$. Let $$A = E_1 \times \cdots \times E_g$$ be the product of $g$ ordinary elliptic curves $E_i$ over $k$ and provide $A$ with its natural principal polarization. Let $x \in \msf A_{g}(k)$ be the point corresponding to the isomorphism class of $A$. Let $q > (g-1)!$ be a prime, different from $p$, and let $$\mr G_q(x) \subset \msf A_{g}( k)$$ be the set of isomorphism classes $y = [(A_y,\lambda_y)]$ that admit an isogeny $\phi \colon A_y \to A_x$ with $\phi^\ast\lambda_x = q^N \cdot \lambda_y$ for some nonnegative integer $N$. 

We claim that $A_y$ satisfies the integral Tate conjecture for one-cycles over $k$ for any $y \in \mr G_q(x)$. Indeed, for such $y$ there exists a nonnegative integer $N$ such that the isogeny $[q^N]\colon A_y \to A_y$ factors through $A_x$. Consequently, %$q^{(2g-2)\cdot N} \colon \rm H^{2g-2}_{\textnormal{\'et}}(A_y,\ZZ_q(g-1)) \to \rm H^{2g-2}_{\textnormal{\'et}}(A_y,\ZZ_q(g-1))$ factors through $\rm H^{2g-2}_{\textnormal{\'et}}(A_x,\ZZ_q(g-1))$ and 
$q^{(2g-2)\cdot N} \cdot \gamma_\theta$ is algebraic for the first Chern class $\theta$ of the principal polarization on $A_y$, which implies that $\gamma_\theta$ is algebraic (as $q > (g-1)!$). Thus, the claim follows from Theorem \ref{introth:integraltate}. 

Now $\mr G_q(z)$ is dense in $\msf A_g$ for any ordinary principally polarized abelian variety $(A_z,\lambda_z)$ by a result of Chai \cite[Theorem 2]{chaiordinaryhecke}. Therefore, $\mr G_q(x)$ is dense in $\msf A_g$ and the proof is finished.
%readily that $\deg(\phi)^2 = \deg([q^N]) = \ell^{2gN}$ for any such $\phi$ and $y$ and we claim that this implies that $A_y$ satisfies the integral Tate conjecture for one-cycles over $\bar k$. 
\end{proof}

\cleardoublepage
% Note that depending on your settings in the table of contents, subsections and subsubsections might appear virtually identical.
% Make sure to set the ToC depth and the numbering depth in the ToC the way you want.
\chapter{Curves on real abelian threefolds}\label{ch:realintegralhodge}

\section{Introduction}

%\subsection{The real integral Hodge conjecture}

%Let $k$ be any positive integer. 

%, for a smooth projective complex variety $X$, the Hodge structure on the singular integral cohomology group $\rm H^{2k}(X(\CC), \ZZ)

In the final Chapter \ref{ch:realintegralhodge} of our thesis, we will, as in the previous Chapter \ref{ch:onecycles}, provide applications of the results of Chapter \ref{ch:integralfourier}. In Chapter \ref{ch:onecycles}, we used integral Fourier transforms on complex abelian varieties $A$ of any dimension $g$ to establish an algebraic link between Hodge classes of degree two and Hodge classes of degree $2g-2$. We saw that this is possible if \emph{one} specific Hodge class is algebraic, namely the minimal Poincar\'e class $c_1(\ca P_A)^{2g-1}/(2g-1)!$ on $A \times \wh A$ (see Theorem \ref{maintheorem}). In Chapter \ref{ch:realintegralhodge} we likewise study one-cycles, but this time we study them over $\RR$. 

A natural question is whether there is again a preferred cohomology class, whose algebraicity would imply the real integral Hodge conjecture for one-cycles on the real abelian variety under consideration. It turns out that if one considers the real integral Hodge conjecture \emph{modulo torsion}, then this is indeed the case (see Theorem \ref{grabowski}). As in the complex case, the proof of this statement relies on integral Fourier transforms, i.e. on the theory developed in of Chapter \ref{ch:integralfourier}. 

Apart from this similarly with the complex situation, two differences stand out: 
\begin{enumerate}
\item \label{difficulty-one}Proving the algebraicity of minimal classes is more difficult over $\RR$ than over $\CC$. For example, for every complex Jacobian variety $(J(C), \theta)$, one has $\theta^{g-1}/(g-1)! = [C]$ for any Abel-Jacobi embedding of the curve $C$ into $J(C)$. If $C$ is defined over $\RR$, but $C(\RR) = \emptyset$, then it is not clear why the minimal class $\theta^{g-1}/(g-1)!$ in $ \Hdg^{2g-2}(J(C), \ZZ(g-1))^G$ should be algebraic. 
\item \label{difficulty-two}The $G$-equivariant cohomology group $\rm H^{2g-2}_G(A(\CC), \ZZ(g-1))$ of a real abelian variety $A$ %(see Section \ref{intro:sub:complexandrealalgebraiccycles}) 
often contains non-zero \emph{torsion} classes. For some of these torsion classes, there are no topological obstructions to be algebraic. 
%the real integral Hodge conjecture predicts it to be algebraic.  
\end{enumerate}
%Taking torsion classes in consideration, however, things become
It is because of these two subtleties \ref{difficulty-one} and \ref{difficulty-two} that when working with algebraic cycles on abelian varieties over the real numbers, new methods have to be employed, methods specific to real algebraic geometry. The goal of Chapter \ref{ch:realintegralhodge} is to introduce such techniques and provide some first positive results, mostly in dimension three. 

To explain the main results of Chapter \ref{ch:realintegralhodge}, let us recall and continue the discussion of Section \ref{intro:sub:complexandrealalgebraiccycles}. %Recently, the analogue of the integral Hodge conjecture for real algebraic varieties has been formulated by Benoist and Wittenberg \cite{BW20}. 
Let $X$ be a smooth projective variety over $\RR$, and remember that $G = \Gal(\CC/\RR)$ (Definition \ref{thegaloisgroup}). Building on work of Krasnov \cite{krasnovcharact,krasnovgroth} and Van Hamel \cite{vanhamel}, Benoist and Wittenberg define a subgroup $\Hdg^{2k}_G(X(\CC), \ZZ(k))_0$ of the $G$-equivariant cohomology group $\rm H^{2k}_G(X(\CC), \ZZ(k))$ in the sense of Borel, and study the cycle class map \cite{krasnovcharact}, \cite[\S 1.6.1]{BW20}
\begin{align}\label{realcycleclassmap}
\CH_i(X) \to \rm \Hdg^{2k}_G(X(\CC), \ZZ(k))_0 \quad \quad \left(i + k = \dim(X)\right).
\end{align}
The \emph{real integral Hodge conjecture for $i$-cycles} refers to the property that  (\ref{realcycleclassmap}) is surjective. Just as in the complex situation, this property holds for every smooth projective variety $X$ over $\RR$ if $i \in \{\dim(X), \dim(X)-1, 0\}$ \cite{krasnovcharact, mangoltehamel, vanhamel, BW20},
%= \dim(X)$ (trivial), $d = \dim(X)-1$ \cite{krasnovcharact, mangoltehamel, vanhamel}, or $d = \dim(X)$ \cite{BW20}, %However, %if $k = 0$ (trivial), $1$ \cite{krasnovcharact} or $\dim(X)$ \cite{BW20}, but 
but may fail for other values of $i$. 

Complex uniruled threefolds, as well as threefolds $X$ over $\CC$ with $K_X = 0$, satisfy the integral Hodge conjecture by work of Grabowski, Voisin and Totaro \cite{grabowski, voisinintegralhodge, totaroIHCthreefolds}. In \cite[Question 2.16]{BW20}, Benoist and Wittenberg ask whether the same is true over $\RR$. In fact, in \cite{BW21} they provide positive answers for various classes of uniruled threefolds. For real Calabi-Yau varieties, however, nothing seems to be known. In Chapter \ref{ch:realintegralhodge}, we address the following: %\blfootnote{\today} 

%In several instances, $\RR$IHC$_k$ holds for a real variety $X$ when its complex analogue holds for $X_\CC$ \cite{BW21}. There is an remarkable exception: unlike the positive situation over $\CC$ \cite{grabowski, voisinintegralhodge, totaroIHCthreefolds}, not much is known about whether real Calabi-Yau threefolds %and rationally connected varieties 
%satisfy the real integral Hodge conjecture. In particular, one may ask:
%Benoist and Wittenberg prove $\RR$IHC$_1$ it for several classes of uniruled threefolds \cite{BW21}. About the $K_X = 0$ case, not much is known. 

\begin{question} \label{question}
Does a real abelian threefold satisfy the real integral Hodge conjecture? 
\end{question}
\noindent
Our goal is to provide evidence towards a positive answer to Question \ref{question}. 

Let us start with explaining our results. If $A$ is a real abelian variety and $k = \dim(A)-1$, 
%Since the higher group cohomology of $\rm H^q(A(\CC), \ZZ(k)$ is two-torsion
%$\rm H^p(G, \rm H^q(A(\CC), \ZZ(k))$ for $p>0$ is two-torsion, 
\begin{comment}
the Hochschild-Serre spectral sequence degenerates \cite{krasnovharnackthom}.
One has $F^1\rm H^{2k}(A(\CC), \ZZ(k))  = \rm H^{2k}_G(A(\CC), \ZZ(k))[2]$, 
\end{comment}
there is an exact sequence (see Lemma \ref{lemma:important0}):
\begin{align*}
0 \to \Hdg^{2k}_G(A(\CC), \ZZ(k))_0[2] \to \Hdg^{2k}_G(A(\CC), \ZZ(k))_0 \to \Hdg^{2k}(A(\CC), \ZZ(k))^G \to 0.
\end{align*}\begin{definition}A real abelian variety $A$ of dimension $g$ satisfies the \emph{real integral Hodge conjecture for one-cycles modulo torsion} if the following map is surjective:
\begin{align} \label{eq:modtors}
\CH_1(A) \to \Hdg^{2g-2}(A(\CC), \ZZ(g-1))^G. 
%\Hdg^{2g-2}_G(A(\CC), \ZZ(g-1))_0/(\textnormal{torsion}) = \Hdg^{2g-2}(A(\CC), \ZZ(g-1))^G. 
\end{align}
%If $g = 3$, then $\CH^k(A) \to \Hdg^{2k}(A(\CC), \ZZ(k))^G$ is surjective for every $k \in \{0,1,3\}$. In that case, 
If $g = 3$, we say that $A$ satisfies the \emph{real integral Hodge conjecture modulo torsion} if the homomorphism (\ref{eq:modtors}) is surjective. 
\end{definition}
%If (\ref{eq:modtors}) is surjective for every non-negative integer $k$, then we say that $A$ satisfies the \emph{real integral Hodge conjecture modulo torsion}. 
%If $g = 3$, then (\ref{realcycleclassmap}) is surjective for $k \in \{0,1,3\}$, and $\Hdg_G^{2k}(A(\CC),\ZZ(k))_0/(\text{torsion}) = \Hdg^{2k}(A(\CC),\ZZ(k))^G$ for every $k$. Thus, if $A$ is an abelian threefold over $\RR$ and if (\ref{eq:modtors}) is surjective, we shall say that $A$ satisfies the \emph{real integral Hodge conjecture modulo torsion}. %if $\CH_1(A) \to \Hdg^{4}(A(\CC), \ZZ)^G$ is surjective. 
% and () is surjective for 
%$k \in \{0,1,2,3\}$. 
%If the ring homomorphism $\CH^\bullet(A) \to \Hdg^{2 \bullet}(A(\CC),\ZZ(\bullet))^G$ is surjective, then we simply say that $A$ satisfies the \emph{real integral Hodge conjecture modulo torsion}. 

%If (\ref{eq:modtors}) is surjective for every non-negative integer $k$, 

%If $A$ is an abelian threefold over $\RR$, then for every non-negative integer $k$, one has
%$
%\Hdg^{2k}_G(A(\CC), \ZZ(k))_0/(\textnormal{torsion})= \Hdg^{2k}(A(\CC), \ZZ(k))^G.
%$ 
\noindent
The first main result of Chapter \ref{ch:realintegralhodge} is as follows. 

\begin{theorem} \label{theorem1}
Every abelian threefold over $\RR$ satisfies the real integral Hodge conjecture modulo torsion. 
%Let $A$ be a principally polarized abelian threefold over $\RR$. Then $A$ satisfies the real integral Hodge conjecture modulo torsion. Equivalently: the homomorphism $\CH_1(A) \to \Hdg^{4}(A(\CC), \ZZ)^G$ is surjective. 
%Principally polarized abelian threefolds $A_{/\RR}$ satisfy the real integral Hodge conjecture modulo torsion. 
\end{theorem}
%\begin{remark}
%It follows from \ref{theorem1} that $\Hdg^{4}_G(A(\CC), \ZZ)_0/(\textnormal{torsion})= \Hdg^{4}(A(\CC), \ZZ)^G$. 
%\end{remark}
\noindent
By using the Hochschild-Serre spectral sequence to calculate the torsion rank of the equivariant cohomology of a real abelian threefold, we obtain:

\begin{corollary} \label{IHCforconnected}
Let $A$ be an abelian threefold over $\RR$ such that $A(\RR)$ is connected. Then $A$ satisfies the real integral Hodge conjecture. 
\end{corollary}
\noindent
Our proof of Theorem \ref{theorem1} is inspired by Grabowski's proof of the integral Hodge conjecture for complex abelian threefolds \cite{grabowski}. It consists of two steps:
%(see \cite[\S3.1]{grabowski}):
%What can we say about real abelian threefolds $A$ with a more complicated topology? 
%Our second result is conditional: to prove $\RR$IHC for principally polarized abelian varieties, it suffices to consider Jacobians of real algebraic curves with real points: %We define a \emph{real algebraic curve} to be a proper and geometrically connected curve over $\RR$. 
% under the Fourier transform 
%\[
%\mr F_A \colon \Hdg^2(A(\CC), \ZZ(1))^G \xrightarrow{\sim} \Hdg^{2g-2}(A(\CC), \ZZ(g-1))^G
%\] 
%of the first Chern class of an ample line bundle $L$ is algebraic. 
%in order to prove $\RR$IHC$_1$ modulo torsion for an abelian variety $A$ of dimension $g$, it suffices to show that the Fourier transform on integral Betti cohomology sends ample line bundle classes to one-cycle classes. To prove the latter in dimension $g = 3$, we proceed in two steps: 
\begin{enumerate}
\item \label{firstreduction} Reduce the real integral Hodge conjecture for one-cycles modulo torsion for abelian varieties of dimension $g$ to the algebraicity of the class $$\gamma_\theta =\theta^{g-1}/(g-1)! \in \Hdg^{2g-2}(A(\CC), \ZZ(g-1))^G$$ for every principally polarized real abelian variety $(A, \theta)$ of dimension $g$.  
%For a fixed positive integer $g$, reduce $\RR$IHC$_1$ modulo torsion for real abelian varieties of dimension $g$ to the algebraicity of the minimal class $\gamma_\theta = \theta^{g-1}/(g-1)!$ for principally polarized real abelian varieties of dimension $g$ (see Theorem \ref{grabowski}). 
\item \label{secondreduction} Reduce the algebraicity of $\gamma_\theta$ on a principally polarized real abelian threefold $A$ to the case where $A = J(C)$ is the Jacobian of a real algebraic curve $C$ with non-empty real locus, where this is clear. 
\end{enumerate}
\noindent
An essential ingredient for reduction step (\ref{firstreduction}) is the fact that any polarized abelian variety over $\RR$ is isogenous to a principally polarized one. Although the analogue of this statement over any algebraically closed field is classical \cite{MumfordAV}, it fails over general fields \cite{Howe2001}. We will prove this fact in Theorem \ref{principalisogeny}. %and use it to to prove the real analogue of [Grabowski, Proposition] combined with [Theorem.. CH..], see Theorem \ref{grabowski}. 

As for step (\ref{secondreduction}), our strategy is to use Hecke orbits.
% we use the density of suitable Hecke orbits in the moduli space of abelian varieties over $\RR$. 
%fact that each connected component $\ca A_3(\RR)^i$ of the moduli space of principally polarized 
%moduli space of principally polarized real abelian threefolds contains a non-empty 
%we prove that suitable Hecke orbits are dense in the moduli space of real abelian varieties. 
For integers $\alpha, \beta$, define the \emph{$(\alpha, \beta)$-Hecke orbit} of a moduli point $[(A, \lambda)] \in \ca A_g(\RR)$ as the set of $[(B, \mu)] \in \ca A_g(\RR)$ admitting an isogeny $A \to B$ preserving the polarizations up to a product of powers of $\alpha$ and $\beta$ (see Definition \ref{def:heckeorbits}). Hecke orbits are well-known to be dense in $\ca A_g(\CC)$ (see Lemma \ref{strongapproximationlemma}); we obtain the following real analogue. 
\begin{theorem} \label{density}
Let $p$ and $q$ be distinct odd prime numbers. %Let $\va{\ca A_g(\RR)}$ be the moduli space of principally polarized abelian varieties of dimension $g$ over $\RR$. 
The $(p,q)$-Hecke orbit of any $x \in \ca A_g(\RR)$ is analytically dense in the connected component of $\ca A_g(\RR)$ containing $x$.  
%Let $(A,\lambda)$ be a principally polarized abelian variety of dimension $g$ over $\RR$, and let $x = [(A, \lambda)] \in \va{\ca A_g(\RR)}$ be its isomorphism class. 
 %For any $x = [(A_x, \lambda_x)] \in \va{\ca A_g(\RR)}$, 
\end{theorem}
\begin{corollary} \label{usefulcorollary}
Let $p$ and $q$ be as above. Every principally polarized abelian threefold over $\RR$ is isogenous, via an isogeny that preserves the polarizations up to a product of powers of $p$ and $q$, to the Jacobian of a non-hyperelliptic curve with non-empty real locus. 
\end{corollary}

\noindent
Reduction step (\ref{secondreduction}) follows because if an odd multiple of $\theta^2/2$ is algebraic for a principally polarized real abelian threefold $(A, \theta)$, then $\theta^2/2$ is algebraic. Corollary \ref{usefulcorollary} turns out to be useful for the general principally polarized case as well:

\begin{theorem} \label{reduction}
Let $\ca A_3(\RR)^+$ be a component of the moduli space of principally polarized real abelian threefolds. Suppose that the real integral Hodge conjecture holds for every Jacobian $J(C)$ such that $[J(C)] \in \ca A_3(\RR)^+$ and the real locus $C(\RR)$ of $C$ is non-empty. Then the real integral Hodge conjecture holds for every real abelian variety in $\ca A_3(\RR)^+$. 
\end{theorem}
%The moduli space of principally polarized abelian varieties of dimension $g$ over $\RR$ is the space $\va{\ca A_g(\RR)}$ constructed by Silhol \cite{silholsurfaces}. The key to Theorem \ref{reduction} is as follows. 
%with $\phi^\ast (\lambda_x) = \alpha^n \beta^m \cdot \lambda_y$ ($n,m\in \ZZ_{\geq 0}$). %We prove:
%In this section, we prove the following:
%Observe that by Theorems \ref{theorem1} and \ref{reduction}, 
 %to prove $\RIHk$ for every principally polarized real abelian threefold, one is . 
\noindent
In view of Theorems \ref{theorem1} and \ref{reduction}, Question \ref{question} can (in the principally polarized case) be rephrased as follows. Let $C$ be a real algebraic curve of genus three with non-empty real locus. Is the torsion subgroup of $\Hdg^{4}_G(J(C)(\CC),\ZZ(2))_0$ algebraic? 

Our final result reduces this question further. The theorem concerns torsion cohomology classes of degree four on real abelian varieties of any dimension $g$. For a smooth projective variety $X$ over $\RR$, the Hochschild-Serre spectral sequence 
\[
E^{p,q}_2 = \rm H^p(G, \rm H^q(X(\CC), \ZZ(k)) \implies \rm H^{p+q}_G(X(\CC), \ZZ(k))
\]
induces a filtration $F^\bullet$ on $\Hdg^{2k}_G(X(\CC), \ZZ(k))_0 \subset \rm H^{2k}_G(X(\CC),\ZZ(k))$ (see Section \ref{sec:realintegralhodge}). %In trying  goal of proving the next theorem for abelian threefolds, it turned out that its proof generalized to abelian varieties of any dimension:

\begin{theorem} \label{fourierreduction}
Let $A$ be an abelian variety over $\RR$. The group $F^3\Hdg^{4}_G(A(\CC), \ZZ(2))_0$ is zero and the group $F^2\Hdg^{4}_G(A(\CC), \ZZ(2))_0$ is algebraic. 
%that satisfies the topological condition (\ref{topologicalcondition}) is algebraic. 
%the Jacobian of a real algebraic curve $C$ such that $C(\RR) \neq \emptyset$. \
\end{theorem}
\noindent
The proof of Theorem \ref{fourierreduction} relies on (a slightly more general version of) Proposition \ref{prop:motivicetale-new} in Chapter \ref{ch:integralfourier} and an analysis of the Abel-Jacobi map for zero-cycles. 

For an abelian variety $A$ over $\RR$, the Hochschild-Serre spectral sequence degenerates \cite{krasnovharnackthom}. For $p \in \ZZ_{\geq 0}$, define $\rm H^p(G, \rm H^{4-p}(A(\CC), \ZZ(2)))_0$ as the image of the canonical homomorphism \[
 F^p\Hdg^{4}_G(A(\CC), \ZZ(2))_0 \to \rm H^p(G, \rm H^{4-p}(A(\CC), \ZZ(2))).
\]
\noindent
Combining Theorems \ref{theorem1} and \ref{fourierreduction}, we obtain:

\begin{corollary} \label{questionreductioncorollary}
Let $A$ be an abelian threefold over $\RR$. Then $A$ satisfies the real integral Hodge conjecture 
%$A$ satisfies $\RR\textnormal{IHC}^2$ for torsion-cycles 
if and only if the canonical homomorphism
\[
\CH_1(A)_{\textnormal{hom}} \to \rm H^1(G, \rm H^3(A(\CC), \ZZ(2)))_0
\]
is surjective, where $\CH_1(A)_{\textnormal{hom}}$ denotes the kernel of $\CH_1(A) \to \rm H^4(A(\CC), \ZZ(2))$. 
\end{corollary}
\noindent
Using the K\"unneth formula and the Abel-Jacobi map for zero-cycles, we show that for a real abelian threefold $B$ and a real elliptic curve $E$ whose real locus $E(\RR)$ is connected, the homomorphism $\CH_1(B \times E)_{\hom} \to \rm H^1(G, \rm H^3(B(\CC) \times E(\CC), \ZZ(2)))$ is surjective (see Proposition \ref{prop:productsurfacecurve}). Together with Corollary \ref{questionreductioncorollary}, this implies: 
%the existence of abelian threefolds satisfying the real integral Hodge conjecture in many connected components of their moduli space. 
%that real abelian threefolds which are products of lower dimensional abelian varieties satisfy the real integral Hodge conjecture. 

\begin{proposition} \label{abeliansurface}
Let $B$ be a real abelian surface, and $E$ a real elliptic curve whose real locus $E(\RR)$ is connected. Then $A = B \times E$ satisfies the real integral Hodge conjecture. 
\end{proposition}
\noindent
Real abelian threefolds $A$ come in four different types, corresponding to the $G$-equivariant diffeomorphism type of the complex locus $A(\CC)$ (see \cite[\S1]{grossharris}). Proposition \ref{abeliansurface} shows that for abelian threefolds of three out of these four types, there are no topological obstructions to the real integral Hodge conjecture.

\section{The real integral Hodge conjecture} \label{sec:realintegralhodge}

%The goal of this paper is to study the analogous property for real varieties, and more specifically for abelian varieties over $\RR$. We shall prove the real integral Hodge conjecture for a large set of real principally polarized abelian threefolds. Let us recall the relevant notions.

%\begin{enumerate}[leftmargin=-0.08cm, rightmargin = -0.08cm]
%\item[2.1.] \textbf{Generalities}. 

\begin{comment}
\begin{convention} \label{convention}
A \emph{variety} over a field $k$ will be a geometrically integral $k$-scheme which is separated and of finite type over $k$. A \emph{curve} (resp. surface, resp. threefold) over $k$ is a variety over $k$ of dimension one (resp. dimension two, resp. dimension three). 
\end{convention}
\end{comment}

\subsection{Generalities} Let $X$ be a smooth projective variety over $\RR$. The group \[G = \Gal(\CC/\RR) = \{ \id, \sigma \}\] acts on $X(\CC)$ via the canonical anti-holomorphic involution $\sigma \colon X(\CC) \to X(\CC)$. For $k \in \ZZ$, we denote by $\ZZ(k)$ the $G$-module that has $\ZZ$ as underlying $\ZZ$-module, on which $G$ acts by $\sigma(1) = (-1)^k$. Thus, $\ZZ(k) = \ZZ(q)$ for every $q \in \ZZ$ with $k \equiv q \bmod 2$. By abuse of notation, we also denote by $\ZZ(k)$ the constant $G$-sheaf on $X(\CC)$ attached to the $G$-module $\ZZ(k)$. For $k,q \in \ZZ_{\geq 0}$, the $G$-action on the group $\rm H^{k}(X(\CC), \ZZ(q))$ is understood to be the one induced by the involution $\rm H^k(\sigma) \circ F_\infty$, where $F_\infty = \sigma^\ast$ is the pull-back of the anti-holomorphic involution $\sigma$ on $X(\CC)$, and $\rm H^k(\sigma)$ is the involution on cohomology induced by $\sigma \colon \ZZ(q) \to \ZZ(q)$. 
%composition the push-forward of the anti-holomorphic involution on $X(\CC)$ with $\rm H^{2k}(
%carries a $G$-action arising from the $G$-actions on $X(\CC)$ and $\ZZ(k)$. 
%that becomes a $G$-module by declaring that $\sigma(1) = (-1)^k$. 
%This action induces the structure of a $G$-module on the group $\rm H^{2k}(X(\CC), \ZZ(k))$. Here, $\ZZ(k)$ is the constant $G$-sheaf on $X(\CC)$ that arises from the action of $G$ on the constant sheaf $\ZZ$ 
%the constant sheaf $\ZZ$ on $X(\CC)$, turned into a $G$-sheaf by declaring that $\sigma(1) = (-1)^k$. %$G$-sheaf on $X(\CC)$ that has $\ZZ$ as underlying constant sheaf, on which $G$ acts by is the constant $G$-sheaf $\ZZ$ on $X(\CC)$ in such a way that $\sigma(1) = (-1)^k$. 
\\
\\
Let $k \in \ZZ_{\geq 0}$. Attached to $X$ is also the so-called degree $2k$ \emph{equivariant cohomology group} with coefficients in $\ZZ(k)$, see \cite{tohoku}. It is denoted by $\rm H_G^{2k}(X(\CC), \ZZ(k))$, and relates to singular cohomology via canonical homomorphisms
\begin{align}\label{equivariant-to-singular}
\varphi \colon \rm H^{2k}_G(X(\CC), \ZZ(k)) \to \rm H^{2k}(X(\CC), \ZZ(k))^G.
\end{align}
A real subvariety $Z \subset X$ of codimension $k$ induces a class $[Z] \in \rm H^{2k}_G(X(\CC), \ZZ(k))$, whose image $\varphi([Z])$ in $\rm H^{2k}(X(\CC), \ZZ(k))^G$ is the Hodge class $[Z_\CC]$. It turns out that such algebraic cycle classes satisfy an additional condition, discovered by Kahn and Krasnov \cite{kahn, krasnovgroth}. It depends only the structure of $X(\CC)$ as a topological $G$-space. For $i \in \{0, \dotsc, 2k\}$, define 
\begin{align*}
\phi_i \colon \rm H^{2k}_G(X(\CC),\ZZ(k)) \to \rm H^i(X(\RR),\ZZ/2)
\end{align*} 
as the composition 
\begin{align*}
\rm    H^{2k}_G(X(\CC), \ZZ(k)) &\xrightarrow{\bmod 2}  \rm H^{2k}_G(X(\CC), \ZZ / 2) \\
&\xrightarrow{\text{restriction}} \rm H^{2k}_G(X(\RR), \ZZ / 2)  =  \rm H^{2k}(X(\RR) \times BG , \ZZ/2) \\
  & \xrightarrow[\sim]{\text{K\"unneth}} \rm H^0(X(\RR), \ZZ/2) \oplus \cdots \oplus \rm H^{2k}(X(\RR), \ZZ/2) \\
&  \xrightarrow{\text{projection}} \rm H^i(X(\RR),\ZZ/2). 
\end{align*}
For $\alpha \in \equivcohom$, define $\alpha_i = \phi_i(\alpha) \in \rm H^i(X(\RR),\ZZ/2)$. 
 %the following group, that has all the potential of being spanned by algebraic cycles classes:
\begin{definition}[Benoist--Wittenberg] \label{def:BW}
The subgroup
\begin{align*}%\label{subgroup}
\Hdg^{2k}_G(X(\CC), \ZZ(k))_0 \subset \rm H^{2k}_G(X(\CC), \ZZ(k))
\end{align*}
is the group of classes $\alpha \in \rm H^{2k}_G(X(\CC), \ZZ(k))$ that satisfy the following conditions:
\begin{enumerate}
    \item \label{def:top} The class $\alpha$ lies in the subgroup $\rm H^{2k}_G(X(\CC),\ZZ(k))_0$, which means that
\begin{align}    \label{topologicalcondition}
(\alpha_{0}, \alpha_{1}, \dotsc, \alpha_{k}, \dotsc, \alpha_{2k}) = 
    \left(0, \dotsc, 0, \alpha_{k}, Sq^1(\alpha_{k}), Sq^2(\alpha_{k}), \dotsc, Sq^k(\alpha_{k})\right).
 \end{align}
 Here, the $Sq^i$ are the Steenrod operations
$$
Sq^i: \rm H^p(X(\RR), \ZZ/2) \to \rm H^{p+i}(X(\RR), \ZZ/2).
$$
    \item \label{def:hodge} The image of $\alpha$ in $\rm H^{2k}(X(\CC), \ZZ(k))$ under (\ref{equivariant-to-singular}) is a Hodge class. 
\end{enumerate}
\end{definition}
%The real cycle class map (\ref{realcycleclassmap}) 
%Similar to the complex case, one thus obtains a \emph{real cycle class map}
%there is a \emph{real cycle class map} [references]
%\[
%c\ell^k_\RR \colon \CH^k(X) \to \Hdg^{2k}_G(X(\CC), \ZZ(k))_0. 
%\]

By \cite[\S1.6.4]{BW20}, Definition \ref{def:BW} is functorial in $X$. %We say that $X$ satisfies the \emph{real integral Hodge conjecture for codimension $k$-cycles} ($\RIHk$) if the resulting cycle class map (\ref{realcycleclassmap}) is surjective. Moreover, $X$ satisfies the \emph{real integral Hodge conjecture} if it satisfies $\RIHk$ for every non-negative integer $k$ not greater than the dimension of $X$ \cite{BW20}. 

%\item[\hypertarget{twotwo}{2.2.}] 

\subsection{Hochschild-Serre} \label{subsec:hochschild}  For a smooth variety $X$ over $\RR$, the \emph{Hochschild-Serre} spectral sequence
\begin{align}\label{hochschild}
E^{p,q}_2 = \rm H^p(G, \rm H^q(X(\CC), \ZZ(k)) \implies \rm H^{p+q}_G(X(\CC), \ZZ(k))
\end{align}
is obtained by viewing $\rm H^{i}_G(X(\CC), -)$ as the right-derived functor of the composition of taking global sections and $G$-invariants on the category of $G$-sheaves on $X(\CC)$. 

Let $A$ be an abelian variety over $\RR$. Then (\ref{hochschild}) degenerates by \cite[\S5.7]{krasnovharnackthom}. Consequently, for every non-negative integer $k$, there are canonical identifications
\begin{align} \label{canonical-HS}
\begin{split}
\rm H^{2k}_G(A(\CC), \ZZ(k))_{\tors} &= \rm H^{2k}_G(A(\CC), \ZZ(k))[2] \\
&=  \Ker\left( \rm H^{2k}_G(A(\CC), \ZZ(k)) \to \rm H^{2k}(A(\CC), \ZZ(k))\right) \\
&=  F^1\rm H^{2k}_G(A(\CC), \ZZ(k)). 
\end{split}
\end{align}
Moreover, these $\ZZ/2$-modules are (non-canonically) isomorphic to the $\ZZ/2$-module $$\bigoplus_{\substack{p + q = k \\ p> 0}} \rm H^p(G, \rm H^q(A(\CC), \ZZ(k)).$$

%\item[\hypertarget{twothree}{2.3.}] \textbf{The topological condition}. 

\subsection{The topological condition} \label{subsec:topological}Let $X$ be a smooth projective variety of dimension $n$ over $\RR$. %Define 
%\begin{align*}
%\rm H^{2\bullet}_{\star}(X(\RR), \ZZ/2) = \bigoplus_{\substack{0 \leq p < n-2 \\ p \equiv n-1 \bmod 2 }} \rm H^p(X(\RR),\ZZ/2).
%\end{align*}
The following sequence is exact:
\begin{align}\label{fundamental}
\begin{split}
0 \to \rm H^{2n-2}_G(X(\CC), \ZZ(n-1))_0 &\to  \rm H^{2n-2}_G(X(\CC), \ZZ(n-1))   \to \\ %\rm H^{2\bullet}_{\star}(X(\RR), \ZZ/2) \to 0 \\
&\to \bigoplus_{\substack{0 \leq p < n-2 \\ p \equiv n-1 \bmod 2 }} \rm H^p(X(\RR),\ZZ/2) \to 0. 
\end{split}
\end{align}
This follows from \cite[Proposition 1.8, Equation (1.33) \& Remark 1.20.(i)]{BW20}. 
\\
\\
The following result will be useful for us. Let $n$ be a positive integer and let $k = n-1$. For a smooth projective variety $X$ of dimension $n$ over $\RR$, define $$\rm H^{2\bullet}_{\star}(X(\RR), \ZZ/2) = \bigoplus_{\substack{0 \leq p < k-1 \\ p \equiv k \bmod 2 }} \rm H^p(X(\RR),\ZZ/2).$$ 
\begin{lemma} \label{lemma:important0}
Let $X$ be a smooth projective variety of dimension $n$ over $\RR$, and let $k = n-1$. Suppose that $X(\CC)$ has torsion-free degree $2k$ integral singular cohomology and that the Hochschild-Serre spectral sequence (\ref{hochschild}) degenerates. % (these conditions hold, for instance, when $X$ is an abelian variety over $\RR$). 
Then each row and each column in the following commutative diagram is exact:
\begin{align} \label{important0}
\begin{split}
\xymatrixcolsep{0.7pc}
\xymatrix{
0 \ar[r]& \rm H^{2k}_G(X(\CC), \ZZ(k))_0[2]  \ar@{^{(}->}[d] \ar[r]& \rm H^{2k}_G(X(\CC), \ZZ(k))_0\ar@{^{(}->}[d] \ar[r]^{\varphi \hspace{1mm}} & \rm H^{2k}(X(\CC), \ZZ(k))^G \ar[r] \ar@{=}[d]& 0 \\
0 \ar[r]& \rm H^{2k}_G(X(\CC), \ZZ(k))[2] \ar@{->>}[d] \ar[r]& \rm H^{2k}_G(X(\CC), \ZZ(k)) \ar[r]^{\varphi\hspace{1mm}} \ar@{->>}[d]& \rm H^{2k}(X(\CC), \ZZ(k))^G \ar[r]& 0 \\
&\rm H^{2\bullet}_{\star}(X(\RR), \ZZ/2) \ar@{=}[r]&\rm H^{2\bullet}_{\star}(X(\RR), \ZZ/2). &&
%0 \ar[r]&\rm H^{2\bullet}_{\star}(X(\RR), \ZZ/2)  &\rm H^{2\bullet}_{\star}(X(\RR), \ZZ/2) &&
}
\end{split}
\end{align}
\end{lemma}
\begin{proof}
By the degeneration of the Hochschild-Serre spectral sequence (\ref{hochschild}), the map $$\varphi \colon \rm H^{2k}_G(X(\CC), \ZZ(k)) \to \rm H^{2k}(X(\CC), \ZZ(k))^G$$ is a surjective homomorphism between abelian groups of the same rank. The target of $\varphi$ is torsion-free, so its kernel is $ \rm H^{2k}(X(\CC), \ZZ(k))[2]$, which explains the horizontal exact sequence in the middle of diagram (\ref{important0}). By the proof of \cite[Proposition 1.8]{BW20}, the exact sequence (\ref{fundamental}) is \emph{split}. This implies that, in diagram (\ref{important0}), the vertical arrow on the bottom left and the horizontal map $\varphi$ on the top right are both surjective. 
%restriction of $\varphi$ to $\rm H^{2k}_G(X(\CC), \ZZ(k))_0$ are both surjective. 
%Thus, because (\ref{fundamental}) is split, $\varphi$ remains surjective after restricting to $\rm H^{2k}_G(X(\CC), \ZZ(k))_0$. 
\end{proof}

%This has the following notable consequence: 
\noindent
Note that the first and second horizontal sequence in diagram (\ref{important0}) remain exact after restricting to Hodge classes. 

%\item[\hypertarget{twofour}{2.4.}] 

\subsection{Threefolds}\label{subsec:threefolds} Now let $X$ be a smooth projective threefold over $\RR$. The topological condition (\ref{topologicalcondition}) on degree four classes takes a particularly simple form: a class $\alpha \in \rm H^4_G(X(\CC),\ZZ(2))$ lies in $\rm H^4_G(X(\CC),\ZZ(2))_0$ if and only if 
\[
\alpha|_x = 0 \in \rm H^4_G(\{x\}, \ZZ(2)) = \ZZ/2 \quad \textnormal{for any} \quad x  \in X(\RR).  
\]
The conditions $\alpha|_x = 0$ for $x$ in different connected components of $X(\RR)$ are \emph{linearly independent} over $\ZZ/2$: for $n = 3$, the sequence (\ref{fundamental}) is the split exact sequence
\begin{align} \label{important}
0 \to \rm H^4_G(X(\CC),\ZZ(2))_0 \to \rm H^4_G(X(\CC),\ZZ(2)) \xrightarrow{\phi_0} \rm H^0(X(\RR),\ZZ/2) \to 0. 
\end{align}
Finally, since $X$ satisfies the real integral Hodge conjecture for $d$-cycles whenever $d \in \{0,2,3\}$ (see \cite[\S 2.3.1 and \S 2.3.2]{BW20}), the real integral Hodge conjecture for $X$ is equivalent to the surjectivity of the homomorphism $$\CH_1(X) \to \Hdg^4_G(X(\CC),\ZZ(2))_0.$$

\section{Density of Hecke orbits}

%\item[\hypertarget{polarizedabelian}{3.1.}] 

\subsection{Polarized real abelian varieties} \label{sec:polarizedabelian}
Let $A$ be a real abelian variety. % by which we mean an abelian variety over $\RR$. Here, and in the sequel, 
As before, the dual abelian variety of $A$ is denoted by $\wh A$. Define $\Lambda = \rm H_1(A(\CC), \ZZ)$. Denote by $\sigma \colon A(\CC) \to A(\CC)$ the canonical anti-holomorphic involution, and by $F_{\infty} \colon \Lambda \to \Lambda$ its push-forward. 

There is a canonical bijection between: 

\begin{itemize}
\item Symmetric isogenies $\lambda \colon A \to \wh A$ such that $\lambda_\CC = \varphi_{\ca L}$ is the homomorphism $\varphi_{\ca L} \colon A_\CC \to \wh{A}_\CC$ induced by an ample line bundle $\ca L$ on $A_\CC$ as in \cite{MumfordAV}. 
\item Alternating forms $E \colon \Lambda \times \Lambda \to \ZZ$ such that $F_\infty^\ast(E) = - E$ and such that the following hermitian form is positive definite: 
\[
H \colon \Lambda_\RR \times \Lambda_\RR \to \CC, \quad H(x,y) = E(ix,y) + iE(x,y).
\]
%is positive definite. % (the complex structure on $\Lambda_\RR$ arises via the isomorphism $\Lambda_\RR = T_eA(\CC)$).  
%, for the complex structure on $\Lambda_\RR$ given by its identification with $T_eA(\CC)$. 
\item Classes of ample line bundles $\theta \in \text{NS}(A_\CC)^G = \Hdg^2(A(\CC),\ZZ(1))^G$. 
%is positive definite, where $I$ is the natural complex structure on $\Lambda_\RR$ 
\end{itemize}
\noindent
In the sequel, a \emph{polarization} on $A$ will be an element in either one of the three sets above; the context will make clear which structure is meant. 
%denote a \emph{polarized real abelian variety} by a tuple $(A, \lambda)$, $(A,E)$ or $(A, \theta)$ as above; the context will make clear which structure is meant. 

%\item[\hypertarget{threetwo}{3.2.}] 
\subsection{Moduli of real abelian varieties}  \label{sec:moduliofabelianRRR}
\begin{comment}
Recall that a \emph{real structure} on a complex torus $X$ is an anti-holomorphic involution $\sigma \colon X \to X$ such that $\sigma(0) = 0$. Let $A$ be a principally polarized abelian variety over $\RR$. Attached to $A$, is the triple
\[
\left(X = A(\CC), \Lambda = \rm H_1(X, \ZZ), E \colon \Lambda \times \Lambda \to \ZZ \right),
\]
where $E \colon \Lambda \times \Lambda \to \ZZ$ the symplectic form induced by the polarization $\theta \in \NS(A_{\CC})^G$. Let $\sigma \colon X \to X$ be the canonical real structure, and $F_\infty \colon \Lambda \to \Lambda$ the push-forward of $\sigma$. Then 
\begin{align} \label{eq:Einfinity}
F_{\infty}^\ast(E)(x,y) = E(F_\infty(x), F_\infty(y)) = -E(x,y), \quad \textnormal{ for all } \quad x,y \in \Lambda.
\end{align}
Indeed, this follows from the fact that $\NS(A_\CC)^G = \Hdg^{2}(A(\CC),\ZZ(1))^G \subset \rm H^2(A(\CC), \ZZ(1))^G$. In fact, to give a polarization on a real abelian variety $A$ is to give such a $G$-equivariant symplectic form $E$ on $\Lambda$. In the sequel, a \emph{polarized abelian variety over $\RR$} will either be denoted by $(A, \lambda)$ or by $(A, \theta)$, where $A$ is the real abelian variety under consideration, $\lambda \colon A \to \wh A$ its polarization, and $\theta \in \Hdg^2(A(\CC), \ZZ(1))^G$ the corresponding Hodge class. 
\end{comment}
%We say that $E$ is \emph{$\sigma$-real}. 
%\cite[Theorem IV.3.4]{silholsurfaces}. 
Let $(A, \lambda)$ be a principally polarized complex abelian variety of dimension $g$. By Galois descent (see Section \ref{stagesetting}), to give a model of $(A, \lambda)$ over $\RR$ is to give an anti-holomorphic involution 
\[\sigma \colon A(\CC) \to A(\CC) \quad \textnormal{ such that } \quad \sigma(0) = 0 \quad \textnormal{ and } \quad F_\infty^\ast(E) = -E.
\]
By \cite[Chapter IV, Theorem (4.1)]{silholsurfaces} (or \cite[Section 9]{grossharris}), such an anti-holomorphic involution $\sigma$ exists if and only if the complex principally polarized abelian variety $(A, \lambda)$ admits a period matrix of the form 
\begin{align}\label{chosenperiodmatrix}
(I_g, \frac{1}{2}M + \rm{i} N).
\end{align}
Here $N$ is a positive definite real matrix and $M$ is a symmetric $g \times g$-matrix with integral coefficients such that if $r = \textnormal{rank}(M) \leq g$, then $M$ is of the form
\begin{align}\tag{1} \label{typeone}
\begin{pmatrix} 
I_r & 0 \\ 
0 & 0 
    \end{pmatrix},
    \end{align}
    or of the form
    \begin{align} \tag{2} \label{typetwo}
    \begin{pmatrix} 
    0 & \dots & 1 & 0 \\
    \vdots&\iddots&\iddots &0\\
    1&0  & \iddots &\vdots\\
    0 &  0 & \dots &0 
    \end{pmatrix}.
\end{align}
%for some non-negative integer $n \leq g$. 
%We denote  $s = \textnormal{rank}(M)$ 
%$(\rm{i})$ or $(\rm{ii})$ in \emph{loc.cit.}
%\cite[IV, Theorem 3.4]{silholsurfaces}. 
%\end{theorem}

\begin{definition}[Silhol] \label{typedefinition}
The \emph{type} $(r, \alpha) \in \ZZ^2$ of a principally polarized real abelian variety $(A, \lambda)$ is defined as follows. 
%of $(X, \sigma\colon X \to X, \theta)$ as defined as follows. %$(X, \sigma\colon X \to X, \theta)$ is of \emph{type} $j= (\mu, \alpha)$ if it 
If $(I_g, \frac{1}{2}M + \rm{i} N)$ is a period matrix for $(A_\CC, \lambda_\CC)$ as above, then $r = \text{rank}(M)$. Define $\alpha \in \{0, 1,2\}$ in the following way:
\begin{itemize}
\item If $r$ is odd, then $\alpha = 1$. Thus the type of $(A, \lambda)$ is $(r, 1)$. 
\item If $r$ is zero, then $\alpha = 0$. Thus the type of $(A, \lambda)$ is $(0,1)$. 
\item If $r$ is even, but non-zero, then $\alpha = 1$ if $M$ is of the form (\ref{typeone}) and $\alpha = 2$ if $M$ is of the form (\ref{typetwo}). 
\end{itemize} 
%or $(r, (\ref{typetwo}))$ according to whether $M$ is of type (\ref{typeone}) or (\ref{typetwo}) above. 
%Then define 
%\[
%\mu = \textnormal{rank}(M(\tau)), \quad \left\{\alpha = (0)  \textnormal{ if $\mu = 0$,} \quad
%\alpha = (\rm{i})  \textnormal{ if $\mu$ is odd,} 
%\quad  \textnormal{$\alpha = (\rm{ii})$ else} \right\}. 
%\]
%rank$(M(\tau)) = \mu$ and $M(\tau)$ of type $\alpha \in \{(\rm{i}), (\rm{ii})\}$. If $\mu$ is odd or zero, we will always take $\alpha = (\rm{i})$. %and if $\mu = 0$, let $\alpha = 0$. 
%set $(\mu,\alpha) = (0,0)$. 
\end{definition}
\noindent
This definition makes sense, because of the following:

\begin{proposition}[Silhol]
The type $(r, \alpha)$ of a principally polarized real abelian variety $(A, \lambda)$ does not depend on the chosen period matrix (\ref{chosenperiodmatrix}) for $(A_\CC, \lambda_\CC)$ nor on the isomorphism class of $(A, \lambda)$. If $(A, \lambda)$ is of type $(r, \alpha)$, there exists a period matrix $(I_g, \frac{1}{2}M + \rm{i} N)$ for $(A, \lambda)$ such that $M$ is of the form (\ref{typeone}) or (\ref{typetwo}), according to whether $\alpha$ equals $1$ or $2$.\end{proposition}
\begin{proof}
See \cite[Chapter IV, Corollaries (4.3) and (4.5)]{silholsurfaces}.  
\end{proof}

%The type $(\mu, \alpha)$ of a principally polarized real abelian variety is then an invariant of its isomorphism class 
%$(X, E, \sigma)$ is an invariant of the isomorphism class of $(X, E, \sigma)$ . 

\begin{definition}
Let $\mr T(g)$ be the set of types $(r, \alpha)$ of principally polarized abelian varieties of dimension $g$ over $\RR$. For any type $\tau \in \mr T(g)$, define $M(\tau)$ to be the integral $g \times g$-matrix (\ref{typeone}) or (\ref{typetwo}) above, according to whether $\alpha$ equals $1$ or $2$. Then define $\GL_g^\tau(\ZZ)$ to be the subgroup of $\GL_g(\ZZ)$ of matrices $T \in \GL_g(\ZZ)$ that satisfy 
\[
T^t\cdot M(\tau) \cdot T \equiv M(\tau) \mod 2 \quad \quad (T^t = \textnormal{ transpose of } T). 
\]
Finally, let $\rm H_g$ be the set of symmetric positive definite real matrices of rank $g$. 
\end{definition}

\begin{theorem}[Silhol]
Let $g$ be a positive integer. For $\tau \in \mr T(g)$, define $\va{ \ca A_g(\RR)}^\tau$ to be the set of isomorphism classes of real principally polarized abelian varieties of type $\tau$. For each type $\tau \in \mr T(g)$, the period map induces a bijection
\[
\va{\ca A_g(\RR)}^\tau  = \GL_g^\tau(\ZZ) \setminus \rm H_g,
\]
where $\GL_g^\tau(\ZZ)$ acts on $\rm H_g$ by $N \mapsto {T^t}\cdot N \cdot T$. Therefore, %if $\va{\ca A_g(\RR)}$ denotes the set of isomorphism classes of principally polarized abelian varieties of dimension $g$ over $\RR$, then
\begin{align}\label{silholsbijection}
\va{\ca A_g(\RR)} = \bigsqcup_{\tau \in \mr T(g)}  \GL_g^\tau(\ZZ) \setminus \rm H_g. 
%\va{\ca A_g(\RR)}^\tau. 
\end{align}
\end{theorem}
\begin{proof}
See \cite[Chapter IV, Theorem (4.6)]{silholsurfaces}.
\end{proof}
%The group $\GL_g^\tau(\ZZ)$ acts on $\rm H_g$ via $T \mapsto {^tA}TA$. 
\noindent
%In this way, one obtains a topology on the moduli space $\va{\ca A_g(\RR)}$. 
See Theorem \ref{th:homeomorphismmoduli} for the fact that bijection (\ref{silholsbijection}) is actually a \emph{homeomorphism}, with respect to the real-analytic topology on $\va{\ca A_g(\RR)}$ (see Definition \ref{deffer}). 

%\begin{theorem}[Silhol]
%\end{theorem}
%For integers $\alpha, \beta \in \ZZ_{\geq 1}$, let $\ca G_{\alpha,\beta}(x) \subset \va{\ca A_g(\RR)}$ be the set of $y \in \va{\ca A_g(\RR)}$ such that there exists an isogeny $\phi \colon A_y \to A_x$ with $\phi^\ast (\mu_x) = \alpha^n \beta^m \cdot \mu_y$ for some $n,m\in \ZZ_{\geq 0}$. 
%In this section, we prove the following:
%\begin{theorem} \label{density}
%Let $p$ and $q$ be distinct odd prime numbers. Let $\va{\ca A_g(\RR)}$ be the moduli space of principally polarized abelian varieties of dimension $g$ over $\RR$. Let $x = [(A_x, \mu_x)] \in \va{\ca A_g(\RR)}$.  The subset $\ca G_{p,q}(x)$ is analytically dense in $\va{\ca A_g(\RR)}$. 
%\end{theorem}    %This is well-known. 
    %Indeed, for any positive definite symmetric matrix $T \in \rm M_g(\RR)$, and any $A \in \GL_g(\RR)$, the matrix ${^tA}TA \in \rm M_g(\RR)$ is again positive definite and symmetric. Moreover, by the spectral theorem, for such a $T$ there exists $A \in \rm O_g(\RR)$ such that $^tATA$ is a diagonal matrix $D$ with positive entries, say $D = \text{diag}(\mu_1,\dotsc, \mu_g)$. If we define $B = \text{diag}(\sqrt{\mu_1^{-1}}, \dotsc, \sqrt{\mu_g^{-1}}) \in \GL_g(\RR)$, then $B^{t}DB = I_g$. 
     %can be diagonalized over $\RR$; the diagonal entries will be positive. 

\subsection{Density of Hecke orbits over the real numbers} 

Before we prove Theorem \ref{density}, let us properly introduce the notion of Hecke orbits over the real numbers. 
%Then 
%$
%\GL_g^\tau(\ZZ) \subset \GL_g^\tau(R) \subset \GL_g(R) \subset \GL_g(\RR)$. 
\begin{definition} \label{def:heckeorbits}
Let $(A, \lambda)$ be a principally polarized abelian variety of dimension $g$ over $\RR$, and let $x = [(A, \lambda)] \in \va{\ca A_g(\RR)}$ the corresponding moduli point. For a tuple of integers $(\alpha, \beta )$, the \emph{$(\alpha,\beta)$-Hecke orbit of $x$} is the subset $\ca G_{\alpha,\beta}(x) \subset \va{\ca A_g(\RR)}$ of isomorphism classes $[(B, \nu)] \in \va{\ca A_g(\RR)}$ of principally polarized abelian varieties $(B, \nu)$ of dimension $g$ over $\RR$, for which there exist $n,m \in \ZZ_{\geq 0}$ and an isogeny 
\[
\phi \colon A \to B \quad \textnormal{ such that } \quad \phi^\ast(\nu) = \alpha^n\beta^m \cdot \lambda. 
\]
\end{definition}

 \begin{proof}[Proof of Theorem \ref{density}] 
 Let $p$ and $q$ be distinct odd prime numbers. 
\begin{enumerate}[leftmargin=-0.08cm, rightmargin = -0.08cm]
    \item[Step 1:] \emph{If $x = [(A, \lambda)] \in \va{\ca A_g(\RR)}$ and $\tau \in \mr T(g)$ is the type of $(A, \lambda)$, then $\ca G(x) \subset \va{\ca A_g(\RR)}^\tau$}. Indeed, for any $y = [(B, \mu)] \in \ca G(x)$, there exists an isogeny $\phi \colon A \to B$ such that $\phi^\ast(\mu) = n \cdot \lambda$ for some odd positive integer $n$. Such a map $\phi$ induces an isomorphism $\rm H_1(A(\CC),\ZZ/2) \cong \rm H_1(B(\CC),\ZZ/2)$ as symplectic spaces with involution. Since $x \in \va{\ca A_g(\RR)}^\tau$, this implies that $y \in \va{\ca A_g(\RR)}^\tau$ as well, see \cite[Section 9]{grossharris}. 
    \item[Step 2:] Define $R = \ZZ
\left[
\frac{1}{p}, \frac{1}{q}
\right]$. The ring homomorphism $R \to R / 2R$ induces a group homomorphism 
$
\GL_g(R) \to \GL_g(R/2R). 
$ For $\tau \in \mr T(g)$, we define
\[
\GL_g^\tau(R)= \{T \in \GL_g(R) \colon T^t \cdot M(\tau) \cdot T \equiv M(\tau) \bmod 2\}.
\]
 Fix one such $\tau \in \mr T(g)$. Observe that the action of $\GL_g^\tau(\ZZ)$ on $\rm H_g$ extends to a transitive action of $\GL_g(\RR)$ on $\rm H_g$. We claim:
 
     \emph{Let $x = [(A_x , \lambda_x)] \in \va{\ca A_g(\RR)}^\tau$, and lift $x$ to a point $y \in \rm H_g$. Consider the orbit $\GL_g^\tau(R) \cdot y \subset \rm H_g$ as well as its image $\GL_g^\tau(\ZZ)\setminus   \left(    \GL_g^\tau(R) \cdot y  \right) $ in $\va{\ca A_g(\RR)}^\tau = \GL_g^\tau(\ZZ) \setminus \rm H_g$. Then}
    \[
\GL_g^\tau(\ZZ) \setminus  \left(    \GL_g^\tau(R) \cdot y  \right) = \ca G_{p,q}(x). %\cap \va{\ca A_g(\RR)}^\tau. 
    \]
    Indeed, if $\bb H_g$ is the genus $g$ Siegel space of symmetric, complex $g\times g$ matrices $Z = X + \rm{i} Y$ whose imaginary part $Y$ is positive definite, then the inclusion 
    \[
    \rho_\tau \colon \rm H_g \hookrightarrow \bb H_g, \quad N \mapsto \frac{1}{2} \cdot M(\tau) + \rm{i}N\]
     is equivariant for the embedding
    \[
f_\tau\colon    \GL_g(\RR) \hookrightarrow \Sp_{2g}(\RR), \quad T \mapsto \begin{pmatrix} T^t & \frac{1}{2}\left(M(\tau) \cdot T^{-1}-T^t\cdot M(\tau)\right) \\ 0 & T^{-1} \end{pmatrix}. 
    \]
    Moreover, the action of the group $\Sp_{2g}(R)$ on $\bb H_g$ has the following geometric meaning: if we consider $\bb H_g$ as a moduli space of $g$-dimensional, principally polarized complex abelian varieties with symplectic basis, then two points $y = [A_y]$ and $z = [A_z] \in \bb H_g$ are in the same $\Sp_{2g}(R)$-orbit if and only if there exists an isogeny $\phi \colon A_y \to A_z$ that preserves the polarizations up to a product of powers of $p$ and $q$. Since the intersection
    %    two 
%    the $\Sp_{2g}(R)$-orbit of a point $y = [A_y] \in \bb H_g$ consists of those $z = [A_z] \in \bb H_g$ such that %for a point $y \in \bb H_g$ in the image of $\rho_i$,
    %observe that the intersection
     \[
    f_\tau\left(\GL_g(\RR) \right) \cap \Sp_{2g}(R) = f_\tau\left(\GL_g^\tau(R)\right)
    \]
    equals the subgroup of $\Sp_{2g}(R)$ that preserves the locus
    \[
    \rho_\tau(\rm H_g) = \left\{
    \frac{1}{2}M(\tau) + \rm{i} N\right\} \subset \bb H_g
    \]
 of real abelian varieties of type $\tau$, this concludes Step 2. 
%    \item[\bf{Step 2}:] \emph{ The action of $\GL_g^\tau(R)$ on $\rm H_g$ extends to a transitive action of $\GL_g(\RR)$ on $\rm H_g$.}
    \item[Step 3:] \emph{For any $\tau \in \mr T(g)$, the subgroup $\GL_g^\tau(R)\subset \GL_g(\RR)$ is dense in the analytic topology.}
    %\emph{When $M(\tau) = 0$, the subgroup $\GL_g^\tau(R)\subset \GL_g(\RR)$ is dense in the analytic topology. }
 % To prove this, let $G = \GL_g^\tau(R)\subset \GL_g(\RR)$; we need to show that $\overline G = \GL_g(\RR)$. Let us first assume that $H = 0$. 

%Write $G = \GL_g^\tau(R)$.      

Define  \[\SL_g^\tau(R)  = \SL_g(R) \cap \GL_g^\tau(R) =  \{T \in \SL_g(R) \colon T^t \cdot M(\tau) \cdot T \equiv M(\tau) \bmod 2\}.\] 
%; the assumption on $M(\tau)$ implies that $G=  \GL_g(R)$. 
We claim that 
\begin{align}\label{inclusionsgroups}
\SL_g(\RR)  = \overline{\SL_g(R)} =\overline{\SL_g^\tau(R)} \subset \overline{\GL_g^\tau(R)} \subset \GL_g(\RR).
\end{align}
Indeed, this follows from the following two statements:
\begin{enumerate}
\item \label{proofparttwo} The closure of $\SL_g(R)$ in $\GL_g(\RR)$ is $\SL_g(\RR)$. 
\item \label{proofpartone} The subgroup $\SL_g^\tau(R) \subset \SL_g(R)$ has finite index. 
%which implies that their closures in $\GL_g(\RR)$ are the same. 
\end{enumerate}
\noindent
To prove (\hyperlink{proofparttwo}{a}), observe that the subgroup $\SL_g(\RR) \subset \GL_g(\RR)$ is closed, which implies that the closure of $\SL_g(R)$ in $\GL_g(\RR)$ equals the closure of $\SL_g(R)$ in $\SL_g(\RR)$. Thus, (\hyperlink{proofparttwo}{a}) follows from the density of $\SL_g(R)$ in $\SL_g(\RR)$, which is true by strong approximation; see the proof of Lemma \ref{strongapproximationlemma} for the precise argument. As for (\hyperlink{proofpartone}{b}), the group $\SL_g(R)$ acts on $\rm M_g(R/2R)$ via $M \mapsto A^t \cdot M \cdot A \bmod 2$, so there is an injection
\[
\SL_g^\tau(R) \setminus \SL_g(R) \hookrightarrow \rm M_g(R/2R) = \rm M_g(\ZZ/2), \quad T \mapsto T^t \cdot M(\tau) \cdot T \mod 2. 
\]
Now (\hyperlink{proofpartone}{b}) %the index of $\SL_g^\tau(R) \subset \SL_g(R)$ is finite, implies that 
implies that the index of $\overline{\SL_g^\tau(R)} \subset \overline{\SL_g(R)}$ is finite. By (\hyperlink{proofparttwo}{a}), we have $\overline{\SL_g(R)} = \SL_g(\RR)$; thus $\overline{\SL_g^\tau(R)}  \subset \SL_g(\RR)$ is a closed subgroup of finite index, hence open. Therefore $\overline{\SL_g^\tau(R)}  = \SL_g(\RR)$ by connectivity of $\SL_g(\RR)$, proving Claim (\ref{inclusionsgroups}). 
%for $\SL_g(R)$ is dense in $\SL_g(\RR)$ by strong approximation. 

Write $G = \GL_g^\tau(R)$. If $H \subset \GL_g(\RR)$ is any Lie subgroup such that $\SL_g(\RR) \subset H$, then $H = \det^{-1}(\det(H))$. 
%Since the inclusion $H \subset \det^{-1}(\det(H))$ is clear, it suffices to show that any $g \in \det^{-1}(\det(H))$ is contained in $H$. But $\det(g) \in \det(H)$ so there exists $h \in H$ such that $\det(g h^{-1}) = 1$; but then $g h^{-1} \in \SL_g(\RR) \subset H$ which implies that $g \in H$ as claimed. 
Consequently, using (\ref{inclusionsgroups}), we obtain:
\begin{align}\label{determinantdetermins}
\overline G = {\det}^{-1}(\det(\overline G)) \subset \GL_g(\RR).
\end{align}
The equality (\ref{determinantdetermins}) implies that in order to prove Step 3, it suffices to show that ${\det}^{-1}(\det(\overline G)) = \GL_g(\RR)$; for this, it suffices to show that $\det(\overline G) = \RR^\ast$. Now the morphism $\det\colon \GL_g(\RR) \to \RR^\ast$ is open, since its differential at the identity matrix $I_g \in \GL_g(\RR)$ is the trace homomorphism $\rm M_g(\RR) \to \RR$. Writing
\[
\GL_g(\RR) = {\det}^{-1}(\RR^\ast) = {\det}^{-1}(\det(\overline G) \sqcup \det(\overline G)^c) =  \overline{G} \sqcup {\det}^{-1}\left(\det(\overline G)^c \right),
\]
it follows that %we see that ${\det}^{-1}\left(\det(\overline G)^c \right)$ is open in $\GL_g(\RR)$, hence $\det(\overline G)^c$ is open in $\RR^\ast$, hence 
$\det(\overline G)$ is closed in $\RR^\ast$. From this, we conclude that \[
\overline{\det(G)} \subset \overline{\det(\overline G)} = \det(\overline G). 
\]
Thus, to show that $\det(\overline G) = \RR^\ast$, it suffices to show that $\det(G)$ is dense in $\RR^\ast$. The homomorphism $\det \colon \GL_g^\tau(R) \to R^\ast$ is surjective because it admits the section 
\[
R^\ast \to \GL_g^\tau(R), \quad x \mapsto \begin{pmatrix} x & 0 \\ 0 & I_{g-1} \end{pmatrix}.
\]
Therefore, \[\det(G) = \det(\GL_g^\tau(R) ) = R^\ast = \set{\pm p^nq^m\colon n,m \in \ZZ},\] and it remains to prove that the latter is dense in $\RR^\ast$. To see this, note that $R^\ast_{>0} =  \set{p^nq^m\colon n,m \in \ZZ}$ is dense in $\RR_{>0}$ because $\log(R^\ast_{>0}) = \ZZ \log(p) + \ZZ\log(q)$ is dense in $\RR$; the latter holds since $\log(R^\ast_{>0})$ is not a cyclic subgroup of $\RR$. %(since $p \neq q$). 

\item[Step 4:] \emph{Finish the proof.} %We are now in the position to finish the proof of Theorem \ref{density}. 
Let $x = [(A, \lambda)] \in \va{\ca A_g(\RR)}$, and let $\tau \in \mr T(g)$ be the type of the principally polarized real abelian variety $(A,\lambda)$. Lift $x$ to a point $y \in \rm H_g$. By Step 3, we know that the orbit $\GL_g^\tau(R)\cdot y$ is dense in $\GL_g(\RR) \cdot y = \rm H_g$. Consequently, the image of $\GL_g^\tau(R)\cdot y$  under the projection $\rm H_g \to\va{\ca A_g(\RR)}^\tau$ is dense in $\va{\ca A_g(\RR)}^\tau$. By Step 2, this image is precisely $\ca G_{p,q}(x)$. Thus $\ca G_{p,q}(x)$ is dense in $\va{\ca A_g(\RR)}^\tau$ as desired. 
\end{enumerate}
\end{proof}
%Finally, observe that $\det(G) = R^\ast$ and that $R^\ast = \{\pm p^iq^j\colon i,j \in \ZZ\}$ is dense in $\RR^\ast$. Therefore, $\det(\overline G) = \RR^\ast$, thus $\overline G = \det^{-1}(\det(\overline G)) = \GL_g(\RR)$, which proves Step 2. 

%    \item[Step 3:] \emph{For any $\tau \in \mr T(g)$, the subgroup $\GL_g^\tau(R)\subset \GL_g(\RR)$ is dense in the analytic topology.} %
%Then $\SL_g^i(R) \subset \SL_g(R)$ is a subgroup of finite index, which implies that the closure of $\SL_g^i(R)$ in $\GL_g(\RR)$ is $\SL_g(\RR)$. We obtain the inclusions
%\[
%\SL_g(\RR) \subset \overline{\GL_g^\tau(R)} \subset \GL_g(\RR). 
%\]
%As before, this implies that 
%$
%\overline{\GL_g^\tau(R)} = {\det}^{-1}(\det(\overline{\GL_g^\tau(R)} )),
%$ and it follows by similar arguments that $\det(\overline{\GL_g^\tau(R)} )$ is closed in $\RR^\ast$. %hence $\det(\overline{\GL_g^\tau(R)} ) = \RR^\ast$, which proves Step 2. 

\section{Principal polarizations in real isogeny classes} \label{sec:principalpolarizations}

The goal of this section is to prove the following:

\begin{theorem} \label{principalisogeny}
Let $(A, \lambda_A)$ be a polarized abelian variety over $\RR$. Then there exists a principally polarized abelian variety $(B, \lambda_B)$ over $\RR$ and an isogeny 
\[
\phi \colon A \to B \quad \textnormal{ \emph{such that} } \quad \phi^\ast(\lambda_B) = \lambda_A. 
\]
\end{theorem}

\begin{proof}
Let $K \subset A(\CC)$ be the kernel of the analytified polarization $\lambda_A \colon A(\CC) \to \wh A(\CC)$. Then $K$ is a finite group of order $d^2$, where $d^2$ is the degree of $\lambda_A$, such that the real structure $\sigma \colon A(\CC) \to A(\CC)$ restricts to an involution
\[
\sigma \colon K \to K. 
\]
We may assume that $K \neq (0)$. Let $p$ be any prime number that divides the order of $K$. We claim that there exists a subgroup $K_1 \subset K$ of order $p$ such that $\sigma(K_1) = K_1$. To see this, let %$H \subset K$ be the largest $p$-group contained in $K$. It is clear that $\sigma(H) = H$. Let 
$H[p] \subset K$ be the $p$-torsion subgroup of $K$. Then $H[p]$ is preserved by $\sigma$, so that $H[p]$ is an $\FF_p$-vector space of finite rank equipped with a linear involution $\sigma$. Therefore, $H[p]$ contains a one-dimensional $\FF_p$-subspace $K_1$ preserved by $\sigma$, which proves our claim. 
%and the linear involution
%\[
%\sigma \colon H[p] \to H[p]
%\]
%is a root of the polynomial $f = X^2 -1 = (X-1)(X+1)$. This readily implies that $H[p]$ contains a $\sigma$-stable one-dimensional subspace $K_1$ over $\FF_p$, proving our claim. 
%hus the minimal polynomial of $\sigma$ divides $f$. 
%This implies that either $1$ or $-1(\mod p)$ is an eigenvalue $\lambda$ for $\sigma$. Let $v \in H[p]$ be any eigenvector with such an eigenvalue, and define $K_1 \subset H[p]$ to be the one-dimensional subspace generated by $v$.  

The group $K_1 \subset A(\CC)$ descends to a finite subgroup scheme $K_1 \subset A$ over $\RR$; define $A_1$ to be the abelian variety $A/K_1$ over $\RR$. Let $\Lambda = \rm H_1(A(\CC),\ZZ)$ and $M = \rm H_1(A_1(\CC), \ZZ)$; the projection $A \to A_1$ induces an exact sequence 
\[
0 \to \Lambda \to M \to K_1 \to 0.
\]
Let $E \colon \Lambda \times \Lambda \to \ZZ$ be the alternating form attached to the polarization $\lambda_A$ of $A$. Since $M / \Lambda = K_1 \subset K = \Lambda^\vee/\Lambda$, we have inclusions
\[
\Lambda \subset M \subset \Lambda^\vee, \quad\textnormal{ where } \quad \Lambda^\vee = \set{x \in \Lambda \otimes \QQ \mid E(x, \Lambda) \subset \Lambda}. 
\]
%Indeed, $ \Lambda^\vee / \Lambda = K \supset K_1 = M/\Lambda$ [refer to Mumford here]. 
%he subgroup $K_1 \subset K$ corresponds to the lattice $\Lambda \subset M \subset \Lambda^\vee$. 
Now the $\ZZ$-valued alternating form $E$ on the lattice $\Lambda$ gives rise to a bilinear form
\[
\overline{E} \colon \Lambda^\vee/\Lambda \times \Lambda^\vee/\Lambda \to \QQ/\ZZ,
\]which vanishes on $M/\Lambda$ because $M/\Lambda \cong \ZZ/p$ and $\overline{E}$ is alternating. This means that $E \colon \Lambda^\vee \times \Lambda^\vee \to \QQ$ restricts to an \emph{integer}-valued form $E_1$ on $M$. The latter induces a polarization $\lambda_{A_1} \colon A_1 \to \wh{A}_1$ that makes the following diagram commute:
\[
\xymatrix{
A \ar[d]^{\lambda_A} \ar[r]^{\pi}& A_1 \ar[d]^{\lambda_{A_1}} \\
\wh A & \wh A_1 \ar[l]_{\wh \pi}.
}
\]
Here $\pi$ is the quotient map $A \to A_1$ and $\wh \pi$ its dual. Since the degree of an isogeny is multiplicative in compositions, and $\deg(\pi) = p$, we have
\[
p^2 \cdot \deg(\lambda_{A_1}) = \deg(\pi)^2\cdot \deg(\lambda_{A_1}) = \deg\left( \lambda_A \right) = d^2. 
\]
%Therefore, the degree of the polarization $\lambda_{A_1}$ is $d^2/p^2$. 
If $d = p$, we are finished -- otherwise, we repeat the above procedure until the real abelian variety $A_n = A_{n-1}/K_{n-1}$ becomes principally polarized. \end{proof}

\section{Algebraic cycles on real abelian varieties}

\subsection{The Fourier transform} \label{subsec:fourier}

%\begin{enumerate}[leftmargin=-0.08cm, rightmargin = -0.08cm]

%\item[\hypertarget{fiveone}{5.1.}] 

Let $A$ be a real abelian variety of dimension $g$, and consider the Poincar\'e bundle $\ca P_A$ on $A \times \wh A$. Let $a_1, \dotsc, a_{2g}$ be integers. The Chern character
\[
\ch(\ca P_{A_\CC}) = \exp(c_1(\ca P_{A_\CC})) \in  \rm H^{2\bullet}(A(\CC) \times \wh A(\CC), \ZZ(\bullet))
\]
defines the \emph{Fourier transform}
\begin{align}\label{fourierone}
\mr F_A \colon  \bigoplus_{i \in \ZZ_{\geq 0}} \rm H^i(A(\CC),\ZZ(a_i)) \to \bigoplus_{i \in \ZZ_{\geq 0}}\rm H^{i}(\wh A(\CC),\ZZ(a_{2g-i}- g + i))
%\rm H^\bullet(A(\CC), \ZZ) \to \rm H^\bullet(\wh A(\CC),\ZZ). 
\end{align}
It is defined as $\mr F_A(x) = \pi_{2,\ast}( \ch(\ca P_A) \cdot \pi_1^\ast(x))$, where $\pi_i$ is the projection of $A \times \wh A$ onto the $i$-th factor. By \cite{beauvillefourier}, the map (\ref{fourierone}) is an isomorphism, inducing isomorphisms
%and the restriction of $\mr F_A$ to $\rm H^2(A(\CC), \ZZ(1))$ defines isomorphisms 
\begin{align}\label{eq:fourier}
\mr F_A \colon \rm H^i(A(\CC),\ZZ(a_i)) \xrightarrow{\sim} \rm H^{2g-i}(\wh A(\CC),\ZZ(a_i+g-i)). 
\end{align}
Since $\ch(\ca P_{A_\CC})$ is fixed by $G$, these maps are isomorphisms of $G$-modules. 

\subsection{Divisors} \label{subsec:divisors}

Let $A$ be an abelian variety over $\RR$. %Let $\theta \in \Hdg^2(A(\CC),\ZZ(1))^G$ be the cohomology class of $\lambda \colon A \to \wh A$. 
Both homomorphisms in the following composition are surjective:
\begin{align}
\label{divisorsurjectivity}
\CH^1(A) \to \Hdg_G^2(A(\CC),\ZZ(1)) \to \Hdg^2(A(\CC),\ZZ(1))^G. 
\end{align}
%is surjective because Hochschild-Serre degenerates. 
Indeed, the first map is surjective by the real integral Hodge conjecture for divisors (see \cite[Proposition 2.8]{BW20}), and the second by the degeneration of the Hochschild-Serre spectral sequence (see Section \ref{subsec:hochschild}). 
%Since $A$ satisfies the real integral Hodge conjecture for divisors, the element $\theta \in \Hdg^2(A(\CC),\ZZ(1))^G$ is the class of some $\ca L \in \Pic(A)$. 

\subsection{The real integral Hodge conjecture for one-cycles modulo torsion}

The goal of this section is to provide an application of Theorems \ref{density} and \ref{principalisogeny} combined: the real integral Hodge conjecture for one-cycles modulo torsion follows, in some cases, from the real integral Hodge conjecture for divisors modulo torsion (see Section \ref{subsec:divisors}). The proof uses Fourier transforms for real abelian varieties (see Section \ref{subsec:fourier}) in a way similar to the way in which we used Fourier transforms for complex abelian varieties in Section \ref{sec:proofmaintheorem}. Thus, we will need the results on integral Fourier transforms obtained in Chapter \ref{ch:integralfourier}. The theory in Chapter \ref{ch:integralfourier} was developed for abelian varieties over a general field $k$ -- to apply it, we take $k = \RR$. 
 \\
 \\
We will need the following:

%In fact, now that we have Theorem \ref{principalisogeny} to our disposal, not only the statement but also the proof of Proposition \ref{grabowski} below is a direct analogue of the proof of Proposition 3.1.8 in \cite{grabowski}. 

\begin{definition}
Let $A$ be a real abelian variety, and $k$ a non-negative integer. An element $\alpha \in \Hdg^{2k}(A(\CC),\ZZ(k))^G$ is called \emph{algebraic} if it is in the image of 
\[
\CH^k(A) \to  \Hdg^{2k}(A(\CC),\ZZ(k))^G. 
\]
\end{definition}
\noindent
Recall that by the main theorem of Chapter \ref{ch:onecycles}, the integral Hodge conjecture for one-cycles on a fixed principally polarized abelian variety $(A, \theta)$ over $\CC$ is equivalent to the algebraicity of the minimal class $\gamma_\theta = \theta^{g-1}/(g-1)!$ on $A$ (see Theorem \ref{maintheorem}). On the other hand, Grabowski reduced the integral Hodge conjecture for one-cycles for every complex abelian variety of dimension $g$ to the algebraicity of $\gamma_\theta$ for every principally polarized abelian variety of dimension $g$ \cite{grabowski}. 

We have the following real analogue of these results:

%and Proposition 3.1.8 in \cite{grabowski}: 
%\cite{grabowski} and \cite[Theorem 1.1]{beckmanndegaayfortman} combined:

\begin{theorem} \label{grabowski}
Fix a positive integer $g$. Let $(A, \theta)$ be a principally polarized abelian variety of dimension $g$ over $\RR$. The following are equivalent: 
\begin{enumerate}
\item \label{realmodtors} The real abelian variety $A$ satisfies the real integral Hodge conjecture for one-cycles modulo torsion. 
\item \label{realminalg} The minimal class 
\[
\gamma_\theta = \frac{\theta^{g-1}}{(g-1)!} \in \Hdg^{2g-2}(A(\CC), \ZZ(g-1))^G \quad \textnormal{ \emph{ is algebraic. }}
\]
\item \label{realchernalg} The Chern character 
\[
\ch(\ca P_{A_\CC}) = \exp(c_1(\ca P_{A_{\CC}})) \in \Hdg^{2\bullet}(A(\CC ) \times \wh A(\CC), \ZZ(\bullet))^G
\quad \textnormal{ \emph{ is algebraic. }}
\]
\end{enumerate}
Moreover, if the real integral Hodge conjecture for one-cycles modulo torsion holds for every principally polarized abelian variety of dimension $g$ over $\RR$, then it holds for every abelian variety of dimension $g$ over $\RR$. 
%if condition (b) holds for \emph{every} principally polarized abelian variety of dimension $g$ over $\RR$, then condition (a) holds for every abelian variety of dimension $g$ over $\RR$. 
%(b) holds for every principally polarized abelian variety of dimension $g$ over $\RR$, then also (a) holds for every abelian variety 
%\begin{enumerate}%[leftmargin=-0.05cm, rightmargin = -0.8cm]
%\item Every abelian variety of dimension $g$ satisfies the real integral Hodge conjecture for one-cycles modulo torsion. 
%\item For every principally polarized abelian variety $(A, \theta)$ of dimension $g$ over $\RR$, the minimal class 
%\[
%\gamma_\theta = \frac{\theta^{g-1}}{(g-1)!} \in \Hdg^{2g-2}(A(\CC), \ZZ(g-1))^G \quad \textnormal{ \emph{ is algebraic. }}
%\]
%(i.e. in the image of (\ref{eq:modtors})). 
%Moreover, if for some principally polarized abelian variety $(A, \theta)$ of dimension $g$ over $\RR$, the minimal class $\gamma_\theta$ is algebraic, then $A$ satisfies the real integral Hodge conjecture for one-cycles modulo torsion. 
\end{theorem}

\begin{proof}
The direction (\ref{realmodtors}) $\implies$(\ref{realminalg}) is trivial; let us assume that (\ref{realminalg}) holds. Let $(A, \theta)$ be a principally polarized abelian variety of dimension $g$ over $\RR$, and suppose that $\gamma_\theta \in \Hdg^{2g-2}(A(\CC), \ZZ(g-1))^G$ is algebraic. %By \S\hyperlink{fourfour}{4.4}, the polarization $\lambda \colon A \to \wh A$ is induced by an symmetric ample line bundle $\ca L$ on $A$. 
By the proof of Proposition \ref{prop:motivicetale-new}, the abelian variety $A$ admits a motivic integral Fourier transform up to homology, see Definition \ref{def:weakintegralfourier}. This means the following. Let $\ell$ be a prime number. %The proof of Proposition \ref{prop:motivicetale-new} shows that 
There exists a cycle %let $\Gamma \in \CH_1(A)$ be a cycle that represents it.
\[
\Gamma \in \CH(A \times \wh A) \quad \textnormal{ such that } \quad [\Gamma_\CC] = \ch(\ca P_{A_{\CC}}) \quad \in \quad \rm H^{2\bullet}_{\et}(A_{\CC} \times \wh A_{\CC}, \ZZ_\ell(\bullet)).
\]
As a consequence, $[\Gamma_\CC] = \ch(\ca P_{A_\CC}) \in \rm H^{2\bullet}(A(\CC) \times \wh A(\CC), \ZZ(\bullet))^G$, which proves (\ref{realchernalg}). 

Let us now assume that (\ref{realchernalg}) holds, and let $\Gamma \in \CH(A \times \wh A)$ be a cycle that induces $\ch(\ca P_{A_\CC})$ in Betti cohomology. The correspondence $\Gamma$ defines a group homomorphism $\Gamma_\ast \colon \CH^\bullet(A) \to \CH^\bullet(\wh A)$ such that the following diagram commutes:
%\begin{adjustwidth}{-1cm}{0pt}
{\footnotesize
\begin{align*}
\xymatrixcolsep{1pc}
\xymatrix{
\CH^1(A) \ar[d] \ar[r]  
& \CH^\bullet(A) \ar[d] \ar[r]^{\Gamma_\ast} &
 \CH^\bullet(\wh A) \ar[d] \ar[r] & 
 \CH_1(\wh A) \ar[d] \\
\Hdg^{2}(A(\CC),\ZZ(1))^G \ar[r] & \Hdg^{2\bullet}(A(\CC), \ZZ(\bullet))^G \ar[r]_\sim^{\mr F_A} & \Hdg^{2\bullet}(\wh A(\CC),\ZZ(\bullet))^G  \ar[r]& \Hdg^{2k}(\wh A(\CC),\ZZ(k))^G. 
}
\end{align*}
}
Here $k = g-1$, the composition on the bottom row is an isomorphism (see (\ref{eq:fourier}) in Section \ref{subsec:fourier}), and the left vertical arrow is surjective (see (\ref{divisorsurjectivity}) in Section \ref{subsec:divisors}). Therefore, the right vertical arrow is surjective, which implies (\ref{realmodtors}). 

Next, suppose that (\ref{realmodtors}) holds for every principally polarized real abelian variety of dimension $g$, and let $A$ be any real abelian variety of dimension $g$. We would like to show that $A$ satisfies the real integral Hodge conjecture for one-cycles modulo torsion. The isomorphism (\ref{eq:fourier}) induces an isomorphism 
\[
\mr F_A \colon \Hdg^2(A(\CC), \ZZ(1))^G \xrightarrow{\sim} \Hdg^{2g-2}(\wh A(\CC), \ZZ(g-1))^G.
\]
Since $\Hdg^2(A(\CC), \ZZ(1))^G$ is algebraic by Section \ref{subsec:divisors}, it suffices to show that $\mr F_A([L_\CC])$ is algebraic for every line bundle $L$ on $A$. By \cite[II, Exercice 7.5]{HAG}, there is an ample line bundle $M$ on $A$ such that $L \otimes M^{\otimes n}$ is ample for $n \gg 0$; we may thus assume that $L$ is ample. By Theorem \ref{principalisogeny}, there is a principally polarized abelian variety $(B, \lambda)$, and an isogeny 
\[
\phi \colon A \to B
\]
 such that, if $\theta \in \NS(B_\CC)^G = \Hdg^{2}(B(\CC), \ZZ(1))^G$ is the class corresponding to the principal polarization $\lambda \colon B \to \wh B$, then 
\[
\phi^\ast(\theta) = [L_\CC] \in \Hdg^2(A(\CC), \ZZ(1))^G. 
\]
On the other hand, the following diagram commutes by \cite[Proposition 3]{beauvillefourier}:
\[
\xymatrixcolsep{5pc}
\xymatrix{
\Hdg^2(A(\CC), \ZZ(1))^G \ar[r]^{\mr F_A}          & \Hdg^{2g-2}(\wh A(\CC), \ZZ(g-1))^G          \\
\Hdg^2(B(\CC), \ZZ(1))^G \ar[r]^{\mr F_B}\ar[u]_{\phi^\ast}  & \Hdg^{2g-2}(\wh B(\CC), \ZZ(g-1))^G \ar[u]_{\hat{\phi}_\ast}. 
}
\]
Moreover, by \cite[Proposition 5]{beauvillefourier}, we have
\[
\mr F_B(\theta) = (-1)^{g-1}\cdot \frac{\hat{\theta}^{g-1}}{(g-1)!} \in \rm H^{2g-2}(\wh B(\CC), \ZZ(g-1))^G,
\]
where $\hat{\theta} \in \Hdg^2(\wh B(\CC), \ZZ(1))^G$ denotes the dual polarization class. We conclude that $F_B(\theta)$ is algebraic, so that $\mr F_A([L_\CC]) = \wh{\phi}_\ast(\mr F_B(\theta))$ is algebraic as well. \end{proof}

%\item[\hypertarget{fivefour}{5.3.}] 

\begin{corollary} \label{IHCmodulotorsionforjacobians} Let $C_1, \dotsc, C_n$ be smooth projective geometrically integral curves over $\RR$ such that $C_i(\RR)\neq \emptyset$ for each $i$. The real abelian variety $A = J(C_1) \times \cdots \times J(C_n)$ satisfies the real integral Hodge conjecture for one-cycles modulo torsion. 
\end{corollary}
\begin{proof}
The minimal class on a product of principally polarized abelian varieties over $\RR$ is the sum of the pull-backs of the minimal classes on the factors, so by Theorem \ref{grabowski}, it suffices to treat the case $n = 1$. For a real algebraic curve $C$ whose real locus is non-empty, any Abel-Jacobi map gives an embedding of real varieties $\iota \colon C \hookrightarrow J(C)$. By Poincar\'e's formula, one has 
\[
[\iota(C)_\CC] = \frac{\theta^{g-1}}{(g-1)!} \in \Hdg^{2g-2}(J(C)(\CC), \ZZ(g-1))^G,
\]
where the class on the right hand side of the equality is the minimal cohomology class $\gamma_\theta$ of $J(C)$. Thus $\gamma_\theta$ is algebraic, so we are done by Theorem \ref{grabowski}. 
\end{proof}

\subsection{Integral Hodge classes modulo torsion on real abelian threefolds}

%\begin{proof}[Proof of Theorem \ref{theorem1}]

%\begin{enumerate}[leftmargin=-0.08cm, rightmargin = -0.08cm]

%\item[\hypertarget{sixone}{6.1.}]  

Recall (see Definition \ref{definition:realalgebraiccurve}) that a real algebraic curve is a smooth projective geometrically connected curve over $\RR$. Let $\ca M_3(\RR)$ be the moduli space of real algebraic curves of genus three, and consider the Torelli map 
\begin{align*}%\label{torelli}
t \colon \ca M_3(\RR) \to \ca A_3(\RR). 
\end{align*}
Let $\ca N_3(\RR) \subset \ca M_3(\RR)$ be the non-hyperelliptic locus. Let us first consider:

\emph{Claim:} The subset $t(\ca N_3(\RR))$ is open in $\ca A_3(\RR)$. 

Indeed, on the level of stacks, $\ca T \colon \ca M_3 \to \ca A_3$ is an open immersion when restricted to the non-hyperelliptic locus $\ca N_3 \subset \ca M_3$. Thus $\ca T(\ca N_3) \subset \ca A_3$ is an open substack. Recall that, for any algebraic stack $\mr X$ of finite type over $\RR$, the set $|\mr X(\RR)|$ of isomorphism classes of $\RR$-points of $\mr X$ admits a topology, called the real-analytic topology (see Definition \ref{deffer}). In this topology, the subset
\[
|\ca T(\ca N_3)(\RR)| \subset |\ca A_3(\RR)| \quad \quad \textnormal{is indeed open: this follows from Corollary \ref{stackycorollary}.\ref{stackyopen}.}
\]
%is open because $\ca T \colon \ca N_3 \to \ca A_3$ is an open immersion together with Corollary \ref{stackycorollary}.\ref{stackyopen}. 
%By Theorem \ref{th:homeomorphismmoduli}, the bijection ... is a homeomorphism, proving the claim. We can now prove the following:

\begin{lemma} \label{lemma:opensubsetofnonhyp}
Every connected component of the moduli space $\ca A_3(\RR)$ contains a non-empty open subset of non-hyperelliptic curves of genus three with non-empty real locus. 
\end{lemma}
\begin{proof}
%Since $g = 3$, the set $I$ of topological types has cardinality (?). 
%Indeed, let $I$ be the set of types of real principally polarized abelian varieties of dimension three (see Definition \ref{typedefinition}), and fix an element $\tau \in \mr T(g)$. 
Let $\ca A_3(\RR)^\tau$ be a connected component of $\ca A_3(\RR)$. By \cite[page 182]{grossharris}, there is a (unique) connected component $\ca M_3(\RR)^\tau$ of $\ca M_3(\RR)$ that satisfies the following two conditions: $C(\RR)$ is not empty for each $[C] \in \ca M_3(\RR)^\tau$, and
\[
t \left(\ca M_3(\RR)^\tau\right) \subset \ca A_3(\RR)^\tau. 
\]
Now $\ca M_3(\RR)^\tau$ contains a component $\ca N_3(\RR)^\tau$ of $\ca N_3(\RR)$ by \cite[Proposition 3.1]{grossharris} and the table on page 174 of \emph{loc.cit.}, and $\ca N_3(\RR)^\tau$ is open in $\ca A_3(\RR)$ by the above. 
\end{proof}
%\item[\hypertarget{fourfour}{4.4.}]  Actually, we shall need something stronger:

\begin{proof}[Proof of Corollary \ref{usefulcorollary}] 
This follows directly from Theorem \ref{density} and Lemma \ref{lemma:opensubsetofnonhyp}.\end{proof}
\noindent
We are now in the position to prove Theorem \ref{theorem1}. 
\begin{proof}[Proof of Theorem \ref{theorem1}]  By Theorem \ref{grabowski}, it suffices to show that for any principally polarized abelian threefold $(A, \theta)$ over $\RR$, the class $\gamma_\theta = \theta^2/2 \in \Hdg^4(A(\CC),\ZZ(2))^G$ is algebraic. Let us prove this. Let $p$ and $q$ be two distinct odd prime numbers. By Corollary \ref{usefulcorollary}, there exists a real algebraic curve $C$ of genus three such that $C(\RR) \neq \emptyset$ together with an isogeny 
\[
\phi \colon A \to J(C)
\]
such that $\phi^\ast(\lambda_C) = p^n q^m\cdot \lambda_A$ for some non-negative integers $n$ and $m$. Here $\lambda_C$ denotes the canonical polarization on $J(C)$. Let $\theta_C \in \Hdg^2(J(C)(\CC),\ZZ(1))^G$ be the corresponding cohomology class; then $\phi^\ast(\theta_C) = p^nq^m \cdot \theta$. Since $C(\RR) \neq \emptyset$, the minimal class $\theta^2_C/2$ on $J(C)$ is algebraic; see Corollary \ref{IHCmodulotorsionforjacobians}. Therefore, the class 
\[
p^{2n}q^{2m} \cdot \theta^2/2 = \phi^\ast(\theta^2_C/2) \in \Hdg^{4}(A(\CC), \ZZ(2))^G 
\] is algebraic. We are done because $\theta^2$ is algebraic and the integer $p^nq^m$ is odd. 
\end{proof}

\begin{proof}[Proof of Corollary \ref{IHCforconnected}] For each $i \in \{1, \dotsc, 6\}$, cup-product defines a canonical isomorphism of $G$-modules:
\[
\bigwedge^i\rm H^1(A(\CC), \ZZ) \xrightarrow{\sim} \rm H^i(A(\CC), \ZZ). 
\]
This allows us to calculate the $G$-module structure on $\rm H^i(A(\CC), \ZZ)$, for we know that $\rm H^1(A(\CC), \ZZ) \cong \ZZ[G]^3$ as $G$-modules \cite[Example 3.1]{vanhamel}. It turns out that the group $\rm H^p(G, \rm H^q(A(\CC), \ZZ))$ vanishes whenever $p + q = 4$ and $p > 0$. Therefore, \[\Hdg^4_G(A(\CC), \ZZ(2))_0 = \Hdg^4(A(\CC), \ZZ(2))^G\] by Section \ref{subsec:hochschild} and Lemma \ref{lemma:important0}. By Theorem \ref{theorem1}, we are done.
\end{proof}

\subsection{Isogenies and torsion cohomology classes}

Let $X$ be a smooth projective variety over $\RR$, and let $k$ be any non-negative integer. 

\begin{definition} \label{isogenydef:subgroups}
\begin{enumerate}
\item Let us call a class $\alpha \in \rm H^{2k}_G(X(\CC), \ZZ(k))$ \emph{topologically distinguished} if $\alpha \in \rm H^{2k}_G(X(\CC),\ZZ(k))_0$. This means that $\alpha$ is topologically distinguished if and only if $\alpha$ satisfies the topological condition (\ref{topologicalcondition}). %Then $\rm H^{2k}_G(X(\CC),\ZZ(k))_0$ is the group of distinguished classes in $\rm H^{2k}_G(X(\CC), \ZZ(k))$. 
\item 
For any subgroup $K \subset \rm H^{2k}_G(X(\CC), \ZZ(k))$, define
\[
K_0 = K \cap \rm H^{2k}_G(X(\CC),\ZZ(k))_0.
\]
\item In case the Hochschild-Serre spectral sequence (\ref{hochschild}) degenerates, we define, for $p > 0$,
\begin{align*}
&\rm H^p(G, \rm H^{2k-p}(X(\CC),\ZZ(k)))_0 \\
&= \Ima(F^p\rm H^{2k}_G(X(\CC),\ZZ(k))_0 \to \rm H^p(G, \rm H^{2k-p}(X(\CC),\ZZ(k)))). 
\end{align*}
%to be the image of $F^p\rm H^{2k}_G(X(\CC),\ZZ(k))_0$ in $\rm H^p(G, \rm H^{2k-p}(X(\CC),\ZZ(k)))$. 
\item Let us call a subgroup $K \subset \rm H^{2k}_G(X(\CC),\ZZ(k))$ \emph{algebraic} if every element of $K$ is in the image of the cycle class map (\ref{realcycleclassmap}). If the Hochschild-Serre spectral sequence (\ref{hochschild}) degenerates, we call a subgroup $L \subset \rm H^p(G, \rm H^{2k-p}(X(\CC), \ZZ(k)))$ \emph{algebraic} if $L$ is the image of an algebraic subgroup $K \subset F^p\rm H^{2k}_G(X(\CC),\ZZ(k))$ under the canonical map $F^p\rm H^{2k}_G(X(\CC),\ZZ(k)) \to  \rm H^p(G, \rm H^{2k-p}(X(\CC), \ZZ(k)))$. 
\end{enumerate}
\end{definition}
%Consider $\rm H^{2k}_G(X(\CC), \ZZ(k))_0[2]$, the two-torsion subgroup of the group of distinguished classes $\rm H^{2k}_G(X(\CC), \ZZ(k))_0$ of degree $2k$ on $X$. We shall need: %The first thing to show is the following: 
%For this we need the following:

The following lemma provides the key to our proof of Theorem \ref{reduction}. 

\begin{lemma}
Let $A$ and $B$ be real abelian varieties, and let $\phi$ be an isogeny $A \to B$ of odd degree. For each pair of non-negative integers $(k,i)$, the homomorphism 
\[
\phi^\ast \colon \rm H^{k}_G(B(\CC), \ZZ(i))[2]  \to \rm H^{k}_G(A(\CC), \ZZ(i))[2] 
\]
is an isomorphism. %restricting to an isomorphism $\rm H^{k}_G(B(\CC), \ZZ(i))_0[2] \cong \rm H^{k}_G(A(\CC), \ZZ(i))_0[2]$. 
If $k$ is even and $i = k/2$, then $\phi^\ast$ restricts to an isomorphism between the subgroups of topologically distinguished classes. % on both sides in case $k$ is even and $i = k/2$. 
%preserving the topologically distinguished cohomology classes. 
\end{lemma}
\begin{proof}
Let us first prove that for an odd integer $n$, possibly negative, the homomorphisms
\begin{align*}
&[n]^\ast \colon \rm H^{k}_G(A(\CC), \ZZ(i))[2]  \to \rm H^{k}_G(A(\CC), \ZZ(i))[2]  \quad \quad\quad \textnormal{ and } \\
&[n]^\ast \colon \rm H^{2k}_G(A(\CC), \ZZ(k))_0[2]  \to \rm H^{2k}_G(A(\CC), \ZZ(k))_0[2]  
\end{align*}
are isomorphisms. Write $\rm H = \rm H^{k}_G(A(\CC), \ZZ(i))$. The Hochschild-Serre spectral sequence (\ref{hochschild}) degenerates, so $F^1\rm H = \rm H[2]$ and for each $p \in \{0,\dotsc,k\}$ there is an exact sequence
\[
0 \to F^{p+1}\rm H \to F^p \rm H \to \rm H^p(G, \rm H^{k-p}(A(\CC),\ZZ(i))) \to 0. 
\]
We have $F^{k + 1}\rm H = (0)$, and for each $p \in \{0,\dotsc,k\}$, the morphism \[[n]^\ast \colon \rm H^{k-p}(A(\CC),\ZZ(i)) \to \rm H^{k-p}(A(\CC),\ZZ(i))\] induces the identity on $\rm H^p(G, \rm H^{k-p}(A(\CC),\ZZ(i)))$. By the snake lemma and descending induction on $p$, we find that $[n]^\ast$ restricts to an isomorphism on $F^p\rm H$ for each $p>0$, and in particular on $F^1\rm H = \rm H[2]$. 

Write $\rm M = \rm H^{2k}_G(A(\CC), \ZZ(k))$. By \cite[\S1.6.4]{BW20}, the class $[n]^\ast(\alpha)$ is topologically distinguished whenever $\alpha \in \rm M$ is topologically distinguished, thus $[n]^\ast$ induces an embedding $\rm M_0[2] \to \rm M_0[2]$, which is an isomorphism since $\rm M_0[2]$ is finite. 
\newpage
\noindent
Next, let $\psi \colon B \to A$ be an isogeny such that $\psi \circ \phi = [d]_A$ and $\phi \circ \psi = [d]_B$, where $d$ is the degree of the isogeny $\phi$. The resulting equalities
\[
\psi^\ast \circ \phi^\ast = [d]_A^\ast \quad \textnormal{ and } \quad \psi^\ast \circ \phi^\ast = [d]_B^\ast 
\]
together the fact $[d]_A^\ast$ and $[d]_B^\ast$ are isomorphisms that identify the subgroups of topologically distinguished classes if $k$ is even and $i  = k/2$, imply that 
\begin{align*}
%& \psi^\ast \colon \rm H^{2k}_G(A(\CC), \ZZ(k))[2] \to \rm H^{2k}_G(B(\CC), \ZZ(k))[2]  \quad \quad \textnormal{ and } \\
 & \phi^\ast \colon \rm H^{k}_G(B(\CC), \ZZ(i))[2] \to \rm H^{k}_G(A(\CC), \ZZ(i))[2] \quad \quad\quad \textnormal{ and } \\
&  \phi^\ast \colon \rm H^{2k}_G(B(\CC), \ZZ(k))_0[2] \to \rm H^{2k}_G(A(\CC), \ZZ(k))_0[2] 
\end{align*}
are isomorphisms as well.  
%Moreover, $\phi^\ast$ is compatible with the topological condition (\ref{topologicalcondition}) in case $k$ is even and $i = k/2$ by \cite[\S1.6.4]{BW20}. %the maps $\phi^\ast$ and $\psi^\ast$ preserve the topologically distinguished classes. 
%The endomorphism $[d]_A^\ast$ of $\rm H^{2k}_G(A(\CC), \ZZ(k))[2]$ factors as 
%\[
%\rm H^{2k}_G(A(\CC), \ZZ(k))[2] \xrightarrow{\psi^\ast} \rm H^{2k}_G(B(\CC), \ZZ(k))[2] \xrightarrow{\phi^\ast} \rm H^{2k}_G(A(\CC), \ZZ(k))[2]. 
%\]
%By \cite[\S1.6.4]{BW20}, the maps $\phi^\ast$ and $\psi^\ast$ preserve the topologically distinguished classes. By the above, $[d]_A^\ast$ and $[d]_B^\ast$ are isomorphisms, hence $\phi^\ast$ is surjective as well as injective. %. For the same reason, 
%$\psi^\ast \circ \phi^\ast = [d]_B^\ast$ is an isomorphism, the degree of $\psi$ being a divisor of a power of $d$. Therefore $\phi^\ast$ is injective as well, thus an isomorphism, and we are done.  
%This is a topological calculation, so we may assume that $A$ is a product of real elliptic curves $E_i$ ($i = 1,2,3$). The $G$-equivariant K\"unneth theorem for abelian varieties then says that 
%\[
%\rm H^{k}_G(A(\CC),\ZZ(i)) \cong \bigoplus_{a + b + c = k} \rm H_G^a(E_1(\CC),\ZZ
%\]
\end{proof}

\begin{corollary} \label{algebraicity}
Let $A$ and $B$ be real abelian varieties, and $\phi$ an isogeny
$
A \to B
$
of odd degree. For any pair of integers $(k,p)$ with $p>0$, the pull-back $\phi^\ast$ defines an isomorphism 
\begin{align} \label{filtrationisomorphism}
F^p\rm H^{2k}_G(B(\CC), \ZZ(k))_0 \xrightarrow{\sim} F^p\rm H^{2k}_G(A(\CC), \ZZ(k))_0. 
\end{align}
In particular, the left-hand side of (\ref{filtrationisomorphism}) is spanned by algebraic classes if and only if the right-hand side of (\ref{filtrationisomorphism}) is spanned by algebraic classes. $\hfill\qed$
%the distinguished classes in $\rm H^{2k}_G(A(\CC), \ZZ(k))[2]$ are algebraic if and only if the distinguished classes in $\rm H^{2k}_G(B(\CC), \ZZ(k))[2]$ are algebraic. 
\end{corollary}

%\begin{proof}
%Indeed, the pullback $\phi^\ast$ defines an isomorphism between $\rm H^{2k}_G(B(\CC), \ZZ(k))[2]$ and $\rm H^{2k}_G(A(\CC), \ZZ(k))[2]$, compatible with the algebraic as well as the distinguished classes. 
%\end{proof}

\subsection{Reduction to the Jacobian case}

\begin{proof}[Proof of Theorem \ref{reduction}]
\
Let $\ca A_3(\RR)^+$ be a connected component of $\ca A_3(\RR)$, % as in the statement of Theorem \ref{reduction} % assume that all Jacobians  in $\ca A_3(\RR)^i$ of curves 
and let $x = [(A, \lambda)] \in \ca A_3(\RR)^+$. By Corollary \ref{usefulcorollary}, there is a real algebraic curve $C$ with non-empty real locus together with an isogeny 
\[
\phi \colon A \to J(C)
\]
that preserves the polarizations up to an odd positive integer. By \cite[page 180]{grossharris}, we have $[J(C)] \in \ca A_3(\RR)^+$ for the isomorphism class $[J(C)]$ of the Jacobian $J(C)$ of $C$. The following sequence is exact (see Lemma \ref{lemma:important0}): 
\[
0 \to \rm H^4_G(A(\CC),\ZZ(2))_0[2] \to \Hdg^{4}_G(A(\CC),\ZZ(2))_0 \to \Hdg^4(A(\CC),\ZZ(2))^G \to 0.  
\]
By Theorem \ref{theorem1}, we know that $\Hdg^4(A(\CC),\ZZ(2))^G$ is generated by classes of one-cycles on $A$. Therefore, $A$ satisfies the real integral Hodge conjecture if and only if the group $\rm H^4_G(A(\CC),\ZZ(2))_0[2]$ is algebraic. By Corollary \ref{algebraicity}, this is equivalent to the algebraicity of $\rm H^4_G(J(C)(\CC),\ZZ(2))_0[2]$. 
\end{proof}
%of degree $p^nq^m$ for some odd primes $p \neq q$ and non-negative integers $n,m$. 

\subsection{Analysis of the Hochschild-Serre filtration} \label{subsec:analysis}
%\subsection{Codimension-two cycles on real abelian varieties}

Let $A$ be an abelian variety over $\RR$, define $\rm H = \rm H^4_G(A(\CC),\ZZ(2))$, and consider the Hochschild-Serre filtration (c.f. Section \ref{subsec:hochschild}):
\begin{align}\label{serrefiltration}
0 \subset \rm H^4(G,\rm H^0(A(\CC),\ZZ(2))) = F^4\rm H \subset F^3\rm H \subset F^2\rm H \subset F^1\rm H = \rm H[2] \subset \rm H. 
\end{align}
The pull-back $\pi^\ast$ along the structural morphism $\pi \colon A \to \Spec(\RR)$ defines a section 
\[
\ZZ/2 =  \rm H^4_G(\{x\}, \ZZ(2)) = \rm H^4(G, \rm H^0(\{x\}, \ZZ(2))) \to \rm H^4(G, \rm H^0(A(\CC),\ZZ(2))) \subset \rm H
\]
of the restriction $\rm H \to \rm H^4_G(\{x\},\ZZ(2)) = \ZZ/2$ to the equivariant cohomology of any $\RR$-point $x \in X(\RR)$. By Section \ref{subsec:threefolds}, this implies that $F^4\rm H_0 = (0)$, so that the intersection of (\ref{serrefiltration}) with the group of topologically distinguished classes becomes
\begin{align}\label{topologically-distinguished-filtration}
0 \subset F^3\rm H_0 \subset F^2\rm H_0 \subset F^1\rm H_0 \subset \rm H_0. 
\end{align}
%The surjectivity of $\CH^1(A) \to \Hdg^2(A(\CC),\ZZ(1))^G$ (see \S\hyperlink{fourthree}{4.3}) implies the surjectivity of $\CH_1(\wh A) \to \Hdg^{4}(\wh A(\CC),\ZZ)^G$; since $A$ and $\wh A$ are isomorphic, this proves Step 2, and thereby the theorem.
Continue to consider our abelian variety $A$ over $\RR$. %Our goal is to prove that $F^2\rm H_0$ is algebraic, where $\rm H = \rm H^4_G(A(\CC),\ZZ)$. %By Theorem \ref{density}, \S\hyperlink{fourtwo}{4.2}, and Corollary \ref{algebraicity}, we may assume that $A = J(C)$ is the Jacobian of a real algebraic curve $C$ with non-empty real locus. Therefore, 
%By \S\hyperlink{sixtwo}{6.2}, 
\emph{Assume that there exists a cycle}
\[
\Gamma \in \CH(A \times \wh A) \quad \textnormal{ \emph{such that} } \quad [\Gamma_{\CC}] = \ch(\ca P_{A_\CC})  \in  
\rm H^{2\bullet}(A(\CC) \times \wh A(\CC),\ZZ(\bullet))^G. 
\]
Under this assumtion, we may define a homomorphism $\Gamma_\ast$ as in the following diagram:
\begin{align} \label{fundamentalcorrespondence}
\begin{split}
\xymatrixcolsep{5pc}
\xymatrix{
\rm H^{2\bullet}_G(A(\CC),\ZZ(\bullet)) \ar[d]^{\Gamma_\ast}\ar[r]^{\pi_1^\ast} & \rm H^{2\bullet}_G(A(\CC) \times \wh A(\CC),\ZZ(\bullet))  \ar[d]^{\left([\Gamma] \cdot -\right)}  \\
\rm H^{2\bullet}_G(\wh A(\CC),\ZZ(\bullet)) & \rm H^{2\bullet}_G(A(\CC) \times \wh A(\CC),\ZZ(\bullet))  
\ar[l]^{\pi_{2,\ast}}. 
}
\end{split}
\end{align}
Then $\Gamma_\ast$ preserves the topologically distinguished classes by \cite[Theorem 1.21]{BW20}. Define 
\begin{align}\label{fundamental-ij}
\Gamma_{\ast}^{i,j} \colon \rm H^{2i}_G(A(\CC),\ZZ(i)) \to \rm H^{2j}_G(\wh A(\CC),\ZZ(j)) \quad (i,j \in \ZZ_{\geq 0})
\end{align}
as the composition of $\Gamma_\ast$ with the natural inclusion and projection morphisms. 

We are now in the position to prove:

\begin{proposition} \label{vanishing}
Let $A$  be an abelian variety over $\RR$. Then
\[
F^3\rm H^4_G(A(\CC),\ZZ(2))_0 = (0). 
\]
\end{proposition}
\begin{proof}
%First assume that $A$ is principally polarized. 
Let $g = \dim(A)$. By \cite[Chapter IV, Example 3.1]{vanhamel}, we have that $A(\CC)$ is $G$-equivariantly homeomorphic to the $g$-fold product of copies of the torus $S^1 \times S^1$, where $S^1 \subset \CC$ is the unit circle and where $G$ acts on each copy $S^1 \times S^1$ in one of the following ways: either $\sigma(x,y) = (x, \bar y)$ for $(x,y) \in S^1 \times S^1$, or $\sigma(x,y) = (y,x)$ for $(x,y) \in S^1 \times S^1$. %either by $S^1 \times S^1 \ni (x,y) \mapst$ to $(x,\bar y)$ or $(y,x)$. 
In particular, there exist $g$ elliptic curves $E_i$ over $\RR$ and a $G$-equivariant homeomorphism 
\begin{align}\label{topologicalstatement}
A(\CC) \cong E_1(\CC) \times \cdots \times E_g(\CC). 
\end{align}
Since the conclusion of Proposition \ref{vanishing} is a statement that only depends on the structure of $A(\CC)$ as a topological $G$-space, we may therefore assume that $A$ is principally polarized by $\theta \in \Hdg^{2}(A(\CC),\ZZ(1))^G$ and the following class is algebraic:
\[
\frac{\theta^{g-1}}{(g-1)!} \in \Hdg^{2g-2}(A(\CC), \ZZ(g-1))^G.
\]
By Theorem \ref{grabowski}, this means that the Chern character $\ch(\ca P_{A_\CC})$ is algebraic, which allows us to define a homomorphism $\Gamma_\ast \colon \rm H^{2\bullet}_G(A(\CC), \ZZ(\bullet)) \to \rm H^{2\bullet}_G(\wh A(\CC), \ZZ(\bullet))$ as in (\ref{fundamentalcorrespondence}), %Define $\rm H = \rm H^4_G(A(\CC),\ZZ(2))$. By Section \ref{subsec:analysis}, we know that $F^4\rm H_0 = (0)$. Therefore, 
%\[
%F^3\rm H_0 = \rm H^3(G, \rm H^1(A(\CC),\ZZ))_0 \quad\quad\quad (\textnormal{see Definition \ref{isogenydef:subgroups}}). 
%\]
and homomorphisms $\Gamma_\ast^{ij}$ as in (\ref{fundamental-ij}). We have a commutative diagram: 
{\footnotesize
\begin{align*}%\label{superdiagram}
\xymatrixcolsep{2.8pc}
\xymatrix{
\rm H^4_G(A(\CC),\ZZ(2)) \ar[r]^{\Gamma_\ast^{4,2g+2}} & \rm H^{2g+2}_G(\wh A(\CC),\ZZ(g+1)) \ar[r]^{\Gamma_\ast^{2g+2,4}} & \rm H^4_G(A(\CC),\ZZ(2))  \ \\ 
F^3\rm H^4_G(A(\CC),\ZZ(2))  \ar[r]^{\Gamma_{\ast}^{4,2g+2}}\ar@{^{(}->}[u] \ar[d] & F^3\rm H^{2g+2}_G(\wh A(\CC),\ZZ(g+1)) \ar@{^{(}->}[u] \ar[d] \ar[r]^{\Gamma_\ast^{2g+2,4}} &F^3\rm H^4_G(A(\CC),\ZZ(2))   \ar@{^{(}->}[u] \ar[d]  \\
\rm H^3(G,\rm H^1(A(\CC),\ZZ(2))) \ar[r]^{\rm H^3(G, \mr F_A)\quad\quad}& \rm H^3(G,\rm H^{2g-1}(\wh A(\CC),\ZZ(g+1)))
\ar[r]^{\quad\quad\rm H^3(G, \mr F_{\wh A})}
 & \rm H^3(G, \rm H^1(A(\CC),\ZZ(2))). 
}
\end{align*}
}
\noindent
Here $\mr F_A \colon \rm H^1(A(\CC),\ZZ(2)) \to \rm H^{2g-1}(\wh A(\CC),\ZZ(g+1))$ is the Fourier transform considered in Section \ref{subsec:fourier}. This map is an isomorphism, with inverse \[(-1)^{1+g}\cdot\mr F_{\wh A} \colon \rm H^{2g-1}(\wh A(\CC),\ZZ(g+1)) \to \rm H^1(A(\CC),\ZZ(2)),\]
see \cite[Corollary 9.24]{huybrechtsfouriermukai}. Consequently, $\rm H^3(G, \mr F_A)$ is an isomorphism, with inverse $\rm H^3(G, \mr F_{\wh A})$. By the compatibility of $\Gamma_\ast$ with the topological condition (\ref{topologicalcondition}), we obtain an isomorphism
\begin{align}\label{h1h5}
\rm H^3(G, \mr F_A) \colon \rm H^3(G,\rm H^1(A(\CC),\ZZ(2)))_0 \xrightarrow{\sim} \rm H^3(G,\rm H^{2g-1}(\wh A(\CC),\ZZ(g+1)))_0. 
\end{align}
Because $\rm H^{2g+2}_G(\wh A(\CC),\ZZ(g+1))_0 = (0)$ by \cite[\S2.3.2]{BW20}, the group on the right of equation (\ref{h1h5}) vanishes. Therefore, the group on the left must be zero as well. 
\end{proof}
\noindent
In fact, the argument used in the above proposition can be generalized to prove the following lemma, which we record for later use:

\begin{lemma} \label{topologicalcompatibility}
Let $A$ be an abelian variety of dimension $g$ over $\RR$. Let $p$ and $q$ be non-negative integers such that $p + q = 2k \in 2 \ZZ_{\geq 0}$. The isomorphism on group cohomology 
\[
\rm H^p(G, \mr F_A) \colon \rm H^p(G, \rm H^q(A(\CC), \ZZ(k))) \to \rm H^p(G, \rm H^{2g-q}(\wh A(\CC), \ZZ(g+k-q)))
\]
identifies $\rm H^p(G, \rm H^q(A(\CC), \ZZ(k)))_0$ with $\rm H^p(G, \rm H^{2g-q}(\wh A(\CC), \ZZ(g+k-q)))_0$. 
\end{lemma}
\begin{proof}
This is a topological statement, so we may and do assume that $A$ is a product of elliptic curves (see (\ref{topologicalstatement})). In this case, we may lift $\rm H^p(G, \mr F_A)$ to an algebraic homomorphism
\[
\Gamma_\ast^{2k, 2g+2k-2q} \colon F^p\rm H^{2k}_G(A(\CC), \ZZ(k)) \to F^p\rm H^{2g+2k-2q}_G(\wh A(\CC), \ZZ(g+k-q))
\]
as above (see Section \ref{subsec:analysis}). We can do the same thing for $\rm H^p(G, \mr F_{\wh A})$; the lemma follows from arguments similar to those used to prove Proposition \ref{vanishing}. 
%a diagram similar to (\ref{superdiagram}). 
\end{proof}

\subsection{Codimension-two cycles on real abelian varieties}

Let us now begin the proof of Theorem \ref{fourierreduction}. The proof consists of two steps: first we prove it for the Jacobian of a real algebraic curve with non-empty real locus, and then we reduce the general case to this particular case. 

\begin{proof}[Proof of Theorem \ref{fourierreduction} (Jacobian case)] Let us prove Theorem \ref{fourierreduction} in the case where $A$ is the Jacobian $J = J(C)$ of a real algebraic curve $C$ of genus $g \in \ZZ_{\geq 1}$ such that $C(\RR) \neq \emptyset$. By Proposition \ref{vanishing}, we have 
\[
F^2\rm H^4_G(J(\CC),\ZZ(2))_0 = \rm H^2(G, \rm H^2(J(\CC),\ZZ(2)))_0
\]
and it remains to prove that $F^2\rm H^4_G(J(\CC),\ZZ(2))_0$ is algebraic. 
By Theorem \ref{grabowski}, %Since $A$ is a principally polarized abelian threefold, 
the Chern character of the Poincar\'e bundle of $J$ in integral Betti cohomology $\rm H^{2\bullet}(J(\CC) \times \wh J(\CC), \ZZ(\bullet))^G$ is algebraic. % (see Corollary \ref{usefulcorollary} and Theorem \ref{grabowski}). 
By Section \ref{subsec:analysis} and \cite[Corollary 9.24]{huybrechtsfouriermukai}, 
%Using the same line of arguments as those used in Section..., one obtains an algebraic isomorphism 
%\[
%\rm H^2(G, \mr F_{J}) \colon \rm H^2(G, \rm H^2(J(\CC),\ZZ))_0 \xrightarrow{\sim} \rm H^2(G, \rm H^{2g-2}(\wh J(\CC),\ZZ(g)))_0
%\]
%with inverse $\rm H^2(G, \mr F_{\wh J})$, and a 
we obtain a commutative diagram
{\small 
\begin{align}\label{bottom}
\xymatrix{
F^2\rm H^4_G(J(\CC),\ZZ(2))_0\ar@{=}[d]  \ar@{^{(}->}[r]^{\Gamma_\ast^{4,6}} &  F^2\rm H^{2g}_G(\wh J(\CC),\ZZ(g))_0 \ar[d] \ar@{->>}[r]^{\Gamma_\ast^{6,4}} & F^2\rm H^4_G(J(\CC),\ZZ(2))_0\ar@{=}[d]  \\
\rm H^2(G, \rm H^2(J(\CC),\ZZ(2)))_0 \ar[r]^{\sim\hspace{1em}} &  \rm H^2(G, \rm H^{2g-2}(\wh J(\CC),\ZZ(g)))_0 \ar[r]^{\sim} & \rm H^2(G, \rm H^2(J(\CC),\ZZ(2)))_0
}
\end{align}
}
such that the composition on the bottom row of (\ref{bottom}) is the identity. Therefore, the composition on the top row of (\ref{bottom}) is the identity. Since $J$ satisfies the real integral Hodge conjecture for zero-cycles by \cite[Proposition 2.10]{BW20}, the group $F^2\rm H^{2g}_G(J(\CC),\ZZ(g))_0$ -- and hence also the group $F^2\rm H^4_G(J(\CC),\ZZ(2))_0$ -- is algebraic. 
\end{proof}

\begin{proof}[Proof of Theorem \ref{fourierreduction} (general case)]
%Note that by Theorem \ref{theorem1}, the second statement in Theorem \ref{fourierreduction} does indeed follow from the first. %Consider a real abelian ƒvariety $A$. 
To prove that $F^2\rm H^{4}_G(A(\CC),\ZZ(2))_0$ is algebraic for a real abelian variety $A$, we would like to reduce to the Jacobian case. The fact that this can be done rests on the following proposition in combination with Lemma \ref{gabber} below. 
\begin{proposition} \label{keyprop}
Let $0 \to B \to J \xrightarrow{f} A \to 0$ be an exact sequence of abelian varieties over $\RR$ such that $f_\RR \colon J(\RR) \to A(\RR)$ is surjective. Let $d = \dim(J)$ and $g = \dim(A)$. Then %the two maps
\begin{align*}
& f_\ast \colon F^2\rm H^{2d}_G(J(C)(\CC),\ZZ(d))_0 \to F^2\rm H^{2g}_G(A(\CC),\ZZ(g))_0 \quad\quad \text{ and } \\
 & \hat{f}^\ast \colon F^2\rm H^{4}_G(\wh J(C)(\CC),\ZZ(2))_0 \to F^2\rm H^{4}_G(\wh A(\CC),\ZZ(2))_0
 \end{align*}
are surjective, where $\wh f \colon \wh A \to \wh J$ is the dual homomorphism of $f \colon J \to A$. 
%\end{enumerate} 
\end{proposition}
\noindent
To prove Proposition \ref{keyprop}, we will need three lemmas (Lemmas \ref{lemone}, \ref{anotherlemma} and \ref{lemfour} below). We will first state and prove these lemmas, and then use them to prove Proposition \ref{keyprop}. After that, we will state and prove Lemma \ref{gabber}, and then use Proposition \ref{keyprop} and Lemma \ref{gabber} to prove Theorem \ref{fourierreduction}.  

Consider a projective variety $X$, smooth over $\RR$. Define $\CH_0(X)_{\textnormal{tors}}$ to be the group of torsion zero-cycles modulo rational equivalence on $X$, and define $\Lambda_X = \rm H^{2d-1}(X(\CC),\ZZ(d))/(\textnormal{torsion})$. If  %the Hochschild-Serre spectral sequence (\ref{hochschild}) degenerates, then in particular 
$X$ is endowed with a real point $x \in X(\RR)$, then there is an abelian variety $\Alb(X)$, the Albanese variety of $X$, equipped with a morphism $u \colon (X,x) \to \Alb(X)$ which is initial in the category of morphisms from $(X,x)$ to abelian varieties over $\RR$ \cite[Theorem A.1]{Wittenberg2008OnAT}. %The abelian variety $\Alb(X)$ is the 
%of $(X,x)$ is an abelian variety for which there is a morphism $u \colon (X,x) \to (\Alb(X),0)$
We have, as $G$-modules: $$\rm H_1(\Alb(X)(\CC), \ZZ) = \rm H_1(X(\CC), \ZZ)/(\textnormal{torsion}) = \Lambda_X.$$
%\rm H^{2d-1}(X(\CC),\ZZ(d)).$$
See \cite[Chapter IV, \S1]{silholsurfaces} for the former and \cite[Corollaire 3.1.9]{mangolte} for the latter isomorphism. The map $u$ induces a universal regular homomorphism $\CH_0(X_\CC)_{\hom} \to \Alb(X)(\CC)$, see \cite{murreapplications}, which is $G$-equivariant thus induces a homomorphism $\text{AJ} \colon \CH_0(X)_{\hom} \to \Alb(X)(\RR)$. See also \cite{vanhamelabeljacobi}. 
%then there exists an abelian variety $\Alb(X)$ over $\RR$ and 
%any $x \in X(\RR)$ defines an Albanese morphism $X \to \Alb(X)$ (see [....]) which induces an Abel-Jacobi morphism $\rm \CH_0(X) \to \Alb(X)(\RR)$ (see ...). 
\begin{lemma} \label{lemone}
Let $X$ be a smooth projective variety of dimension $d$ over $\RR$ with $x \in X(\RR)$. Suppose that the Hochschild-Serre spectral sequence (\ref{hochschild}) degenerates. Define $\CH_0(X)_{\textnormal{tors}}$ and $\Lambda_X$ as above. 
%Then $X(\RR) \neq \emptyset$, so there is an Albanese morphism $X \to \Alb(X)$. Let $\CH_0(X)_{\textnormal{tors}}$ be the group of torsion zero-cycles modulo rational equivalence, and define $\Lambda = \rm H^{2d-1}(X(\CC),\ZZ)/(\textnormal{torsion})$. 
Then the following diagram commutes, and its arrows are surjective:
\begin{align}\label{AJdiagram}
\xymatrixcolsep{5pc}
\xymatrix{
\CH_0(X)_{\textnormal{tors}}\ar@{->>}[d]^{\textnormal{AJ}} \ar@{->>}[rr] && F^1\rm H^{2d}_G(X(\CC), \ZZ(d))_0 \ar@{->>}[d] \\
%\ar@{=}[r] &{\rm H^{2d}_G(X(\CC), \ZZ(d))_0}_{\textnormal{tors}}\\
\Alb(X)(\RR)_{\textnormal{tors}} \ar@{->>}[r] &\pi_0\left(\Alb(X)(\RR)\right)\ar[r]^{\sim\;\;\;\;\;\;} & \rm H^1(G, \Lambda_X). 
}
\end{align}
Here, the horizontal isomorphism on the right of the bottom row is the map induced by the boundary map %$\delta \colon \Alb(X)(\RR) \to \rm H^1(G, \rm H^{2d-1}(X(\CC),\ZZ(d)))$ 
attached to the canonical exact sequence of $G$-modules 
\begin{align}\label{pizero}
0 \to \Lambda_X \to \rm H^{2d-1}(X(\CC), \RR(d)) \to \Alb(X)(\CC) \to 0.
\end{align}
\end{lemma}

\begin{proof}
We claim that we can complete diagram (\ref{AJdiagram}) in the following way: 
\begin{align}\label{AJdiagramcomplete}
\begin{split}
\xymatrix{
\CH_0(X)_{\textnormal{tors}}\ar[dr]^\sim\ar[rr]\ar[ddd]^{\textnormal{AJ}} && F^1\rm H^{2d}_G(X(\CC), \ZZ(d))_0 \ar[dd] \\
&\rm H^{2d-1}_G(X(\CC), \QQ/\ZZ(d))_0\ar[ur] \ar[d]& \\
&\rm H^{2d-1}(X(\CC), \QQ/\ZZ(d))^G \ar[r]& \rm H^1(G, \Lambda_X)  \\
\Alb(X)(\RR)_{\textnormal{tors}} \ar[ur]^\sim \ar[rr] &&\pi_0\left(\Alb(X)(\RR)\right). \ar@{=}[u]
}
\end{split}
\end{align}
Let us define the groups and arrows occurring in this diagram, prove that each arrow is surjective, and that each quadrilateral and triangle commutes. Observe that this is indeed enough to prove the lemma.  
%Define $\Lambda$ to be the $G$-module $\rm H^{2d-1}(X(\CC),\ZZ(d))$. 

Let us start with the diagram on the bottom of (\ref{AJdiagramcomplete}). The exact sequence of (sheaves of) $G$-modules
\begin{align}\label{divisbleexact}
0 \to \ZZ(d) \to \QQ(d) \to \QQ/\ZZ(d) \to 0
\end{align}
%The universal coefficient theorem and the torsion-freeness of $\rm H^{2d-1}(X(\CC), \ZZ(d))$ imply that 
induces an exact sequence of $G$-modules
\[
0 \to \Lambda_X \to \rm H^{2d-1}(X(\CC), \QQ(d)) \to \rm H^{2d-1}(X(\CC), \QQ/\ZZ(d)) \to 0.
\]
This gives a canonical isomorphism $\Alb(X)(\CC)_\textnormal{tors} \cong \rm H^{2d-1}(X(\CC), \QQ/\ZZ(d))$, and the boundary map 
\[
\Alb(X)(\RR)_\textnormal{tors} \cong \rm H^{2d-1}(X(\CC), \QQ/\ZZ(d))^G \to \rm H^1(G, \Lambda_X)
\]
is surjective because $\rm H^{2d-1}(X(\CC), \QQ(d)) $ has no higher group cohomology. Thus the diagram on the bottom of (\ref{AJdiagramcomplete}) exists, commutes and its arrows are surjective.  
%proves that they are surjective, and that the quadrilateral commutes. 

Let us continue with the diagram on the left of (\ref{AJdiagramcomplete}). A canonical isomorphism between $\CH_0(X_\CC)_{\textnormal{tors}}$ and $\rm H^{2d-1}(X(\CC), \QQ/\ZZ(d))$ is constructed in \cite{colliotcyclestorsion}. This map is $G$-equivariant, thus induces an isomorphism $$\CH_0(X_\CC)_{\textnormal{tors}}^G\cong \rm H^{2d-1}(X(\CC), \QQ/\ZZ(d))^G.$$ As shown by Van Hamel \cite{vanhamelabeljacobi}, the composition of the latter with the natural map $\CH_0(X)_{\textnormal{tors}} \to \CH_0(X_\CC)_{\textnormal{tors}}^G$ factors through a composition of morphisms
\begin{align}\label{comp}
\CH_0(X)_{\textnormal{tors}} \to \rm H^{2d-1}_G(X(\CC), \QQ/\ZZ(d)) \to \rm H^{2d-1}(X(\CC), \QQ/\ZZ(d))^G.
\end{align}
On the other hand, the composition 
\[
\CH_0(X)_{\textnormal{tors}} \to \rm H^{2d-1}_G(X(\CC), \QQ/\ZZ(d))  \xrightarrow{\delta} \rm H^{2d}_G(X(\CC),\ZZ(d))
\]
coincides with the cycle class map by \cite[Corollaire 1]{torsiondanslegroupe}, where the map $\delta$ is the boundary map induced by the exact sequence (\ref{divisbleexact}). Therefore, the triangle on the top of (\ref{AJdiagramcomplete}) commutes, where we define $\rm H^{2d-1}_G(X(\CC), \QQ/\ZZ(d))_0$ to be the inverse image of $\rm H^{2d}_G(X(\CC),\ZZ(d))_0$ under this boundary map $\delta$. Note that the image of $\delta$ lies indeed in $F^1\rm H^{2d}_G(X(\CC),\ZZ(d))$, because the composition of $\delta$ with the map $\rm H^{2d}_G(X(\CC),\ZZ(d)) \to \rm H^{2d}(X(\CC),\ZZ(d))$ factors through the following composition, which is zero:
\begin{align*}
 \rm H^{2d-1}_G(X(\CC), \QQ/\ZZ(d)) \to \rm H^{2d-1}(X(\CC), \QQ/\ZZ(d)) \to \rm H^{2d}(X(\CC),\ZZ(d)) = \ZZ. 
\end{align*}
We claim that the thus-obtained map
\[
\rm H^{2d-1}_G(X(\CC),\QQ/\ZZ(d))_0 \to F^1\rm H^{2d}_G(X(\CC),\ZZ(d))_0
\]
is surjective. Indeed, this follows since $\rm H^{2d-1}_G(X(\CC),\QQ/\ZZ(d)) \to F^1\rm H^{2d}_G(X(\CC),\ZZ(d))$ is surjective, which is in turn a consequence of the commutativity of the diagram
\[
\xymatrix{
&\rm H^{2d}(X(\CC) , \ZZ(d))^G \ar[r] & \rm H^{2d}(X(\CC),\QQ(d))^G \\
\rm H^{2d-1}_G(X(\CC),\QQ/\ZZ(d)) \ar[dr]^\delta \ar[r]& \rm H^{2d}_G(X(\CC),\ZZ(d)) \ar[r]\ar[u]& \rm H^{2d}_G(X(\CC), \QQ(d)),  \ar[u]_{\wr} \\
&F^1\rm H^{2d}_G(X(\CC),\ZZ(d)) \ar@{^{(}->}[u]  &
}
\]
and the fact that its middle row and column are exact. 

The image of the first arrow in (\ref{comp}) equals $ \rm H^{2d-1}_G(X(\CC), \QQ/\ZZ(d))_0$ because of the following diagram with exact rows, see \cite[Theorem 3.2]{vanhamelabeljacobi}, \cite[Proposition 1.8]{BW20}:
{\footnotesize
\begin{align*}\label{fundamentalzerosequence}
\xymatrixcolsep{1.5pc}
\xymatrix{
0 \ar[r]& 
\CH_0(X)_{\textnormal{tors}}  \ar@{->>}[d]\ar[r]& 
\rm H^{2d-1}_G(X(\CC),\QQ/\ZZ(d)) \ar[r]\ar@{->>}[d]& 
\bigoplus_{i > 0}\rm H^{d - 2i}(X(\RR), \ZZ/2) \ar[r] \ar@{=}[d]& 0 \\
0 \ar[r]& F^1\rm H^{2d}_G(X(\CC), \ZZ(d))_0 \ar[r]&  F^1\rm H^{2d}_G(X(\CC), \ZZ(d))  \ar[r]& \bigoplus_{\substack{0 \leq p < d \\ p \equiv d \bmod 2 }} \rm H^p(X(\RR),\ZZ/2) \ar[r] & 0.  
}
\end{align*}
}
The Abel-Jacobi map $\text{AJ}$ on the left of (\ref{AJdiagramcomplete}) is surjective by \cite[Corollary 4.5]{vanhamelabeljacobi} %The second map above is surjective in our case, because of \cite[Lemma 4.3]{vanhamelabeljacobi} 
and our assumption that the Hochschild-Serre spectral sequence (\ref{hochschild}) degenerates. Moreover, the fact that the quadrilateral on the left of (\ref{AJdiagramcomplete}) commutes follows from the commutativity of 
\[
\xymatrix{
\CH_0(X_\CC)_{\textnormal{tors}} \ar[r]\ar[d] &  \rm H^{2d-1}(X(\CC), \QQ/\ZZ(d)) \ar[dl]_\sim \\
\Alb(X)(\CC)_{\textnormal{tors}}&
}
\]
which is \cite[Proposition 3.9]{blochtheoremroitman}. 

To finish the proof, it remains to show that the quadrilateral on the top right of (\ref{AJdiagramcomplete}) commutes. This follows from the fact that the morphism
\[
\rm H \coloneqq \rm H^{2d-1}_G(X(\CC), \QQ/\ZZ(d)) \to \rm H^{2d}_G(X(\CC), \ZZ(d)) \eqqcolon \tilde{\rm H}
\]
shifts the filtrations induced by the Hochschild-Serre spectral sequences by one degree, thereby inducing a commutative diagram
\[
\xymatrix{
0 \ar[r] & F^1\rm H \ar[r] \ar[d]  & F^0\rm H \ar[r] \ar[d]& \rm H^0(G, \rm  H^{2d}(X(\CC), \QQ/\ZZ(d))) \ar[r] \ar[d]& 0 \\
0 \ar[r] & F^2\tilde{\rm H} \ar[r] & F^1 \tilde{\rm H} \ar[r] & \rm H^1(G, \rm  H^{2d-1}(X(\CC), \ZZ(d))) \ar[r] & 0.
}
\]
\end{proof}

\begin{lemma} \label{anotherlemma} Let $f \colon J \to A$ be a surjective homomorphism of real abelian varieties. 
\begin{enumerate}
\item \label{anotherlemma:first}
The induced homomorphism $f_\RR^0 \colon J(\RR)^0 \to A(\RR)^0$ is surjective. 
\item 
Let $B = \Ker(f)$ and assume that $B$ is an abelian variety (equivalently: that $B(\CC)$ is connected). 
%Let $0 \to B \to J \xrightarrow{f} A \to 0$ be an exact sequence of real abelian varieties. 
The induced homomorphism $J(\RR)^0_{\tors} \to A(\RR)^0_\tors$ is also surjective. 
%Then the induced homomorphisms
%\begin{align*}
%J(\RR)^0 \to A(\RR)^0, \\
%J(\RR)^0_{\tors} \to A(\RR)^0_\tors
%\end{align*}
%are also surjective. If $f \colon J(\RR) \to A(\RR)$ is surjective, then  
\end{enumerate}
\end{lemma}
\begin{proof}
\begin{enumerate}
\item 
The diagram
\[
\xymatrix{
J(\CC) \ar[d]^f \ar[r]^{\rm{Nm}} & J(\RR)^0 \ar[d]^{f_\RR^0} \\ 
A(\CC) \ar[r]^{\rm{Nm}} & A(\RR)^0
}
\]
commutes, where $\rm{Nm} \colon X(\CC) \to X(\RR)^0, x \mapsto x + \sigma(x)$ denotes the norm homomorphism of a real abelian variety $X$. By \cite[Proposition 1.1]{grossharris}, this map $\rm{Nm}$ is surjective. Thus $f_\RR^0$ is surjective.
\item Let $x \in A(\RR)^0_\tors$ and suppose that $n \cdot x = 0$ for some $n \in \ZZ_{\geq 2}$. By part \ref{anotherlemma:first}, there exists $y \in J(\RR)^0$ such that $f(y) = x$. We have that $n \cdot y \in B(\RR)$, hence $2n \cdot y \in B(\RR)^0$. Since the abelian group $B(\RR)^0$ is divisible, there exists $z \in B(\RR)^0$ such that $2n \cdot z = 2n \cdot y$. Define $\alpha = y - z \in J(\RR)^0$. Then $f(\alpha) = f(y) = x$. Moreover, $2n \cdot \alpha = 0$. We conclude that $J(\RR)^0_{\tors} \to A(\RR)^0_\tors$ is surjective. 
\end{enumerate}
\end{proof}

\begin{lemma} \label{lemfour}
Consider an exact sequence of real abelian varieties
$
 0 \to B \to J \xrightarrow{f} A \to 0. 
$
If the restricted homomorphism $f \colon J(\RR) \to A(\RR)$ is surjective, then the push-forward
% following map is also surjective:
\[
f_\ast \colon \CH_0(J)_{\textnormal{tors}} \to \CH_0(A)_{\textnormal{tors}} \quad\quad \text{ is also surjective.} %\quad \text{and } \quad f \colon \pi_0\left(J(\RR)\right) \to \pi_0\left(A(\RR)\right)
\]
%\begin{enumerate}
%\item 
%For every $n \in \ZZ_{\geq 1}$, the induced homomorphism $J(\RR)[n] \to A(\RR)[n]$ is surjective. 
%Q\item 
%If $\alpha_1, \dotsc, \alpha_n \in A(\RR)_{\text{tors}}$ are such that $\suM(\tau)\alpha_i \in A(\RR)^0$, then there exist $\beta_1, \dotsc, \beta_n \in J(\RR)_{\text{tors}}$ such that $\pi(\beta_i) = \alpha_i$ for each $i$ and such that $\suM(\tau)\beta_i \in J(\RR)^0$. 
%\end{enumerate}
\end{lemma}

\begin{proof}
For an abelian group $G$, let $G_{\text{tors}}$ be the torsion subgroup and $G_{\text{tors},\text{div}}$ the maximal divisible torsion subgroup of $G$. Following van Hamel \cite{vanhamelabeljacobi}, we may define, for an algebraic variety $X$ over $\RR$, a group $\rm A_0(X)^{\text{top}}$ as the quotient 
\[
\rm A_0(X)^{\text{top}} = \rm A_0(X)_{\text{tors}} / \rm A_0(X)_{\textnormal{tors}, \textnormal{div}} = \CH_0(X)_{\text{tors}} / \CH_0(X)_{\textnormal{tors}, \textnormal{div}},
\]
where $\rm A_0(X)$ is the group of zero cycles of degree zero modulo rational equivalence on $X$. 

In our case, we obtain a commutative diagram with exact rows:
\[
\xymatrix{
0 \ar[r] & \CH_0(J)_{\textnormal{tors}, \textnormal{div}} \ar[d] \ar[r] & \CH_0(J)_{\text{tors}} \ar[d] \ar[r] & \rm A_0(J)^{\text{top}} \ar[r] \ar[d] & 0 \\
0 \ar[r] & \CH_0(A)_{\textnormal{tors}, \textnormal{div}} \ar[r] & \CH_0(A)_{\text{tors}} \ar[r] & \rm A_0(A)^{\text{top}} \ar[r] & 0. 
}
\]
To prove that the vertical arrow in the middle is surjective, it suffices to prove the surjectivity of the outer two vertical arrows. To prove the surjectivity of the left vertical arrow, one uses the commutativity of the diagram
\[
\xymatrixcolsep{3pc}
\xymatrix{
\CH_0(J)_{\textnormal{tors}, \textnormal{div}} \ar[d] \ar[r]_{\sim}^{\textnormal{AJ}} & J(\RR)^0_{\textnormal{tors}}\ar[d] \\
 \CH_0(A)_{\textnormal{tors}, \textnormal{div}} \ar[r]_{\sim}^{\textnormal{AJ}} & A(\RR)^0_{\textnormal{tors}}
}
\]
and Lemma \ref{anotherlemma}. The two Abel-Jacobi maps appearing in this diagram are isomorphisms by \cite[Corollary 4.2]{vanhamelabeljacobi}. 
\\
\\
It remains to prove that $A_0(J)^{\text{top}} \to A_0(A)^{\text{top}}$ is surjective. Now for any smooth projective variety $X$ over $\RR$, one has by \cite[\S3.2]{vanhamelabeljacobi} that \[
\rm A_0(X)^{\text{top}} = \rm A_0(X) / \rm A_0(X)_{\text{div}}.\]
Thus the desired surjectivity follows from the surjectivity of $\rm A_0(J) \to \rm A_0(A)$. The latter is a consequence of the surjectivity of $J(\CC) \to A(\CC)$ and $J(\RR) \to A(\RR)$, and the fact that for an abelian variety $X$ over a field $k$, the group $\rm A_0(X)$ is generated by zero-cycles of the form $[x] - \deg(k(x_i)/k)\cdot [0]$ for closed points $x$ on $X$. \end{proof}
\noindent
Let us now prove Proposition \ref{keyprop}. We will need some notation. 
\begin{definition}  \label{F2def}
For an abelian variety $A$ of dimension $g$ over $\RR$, define
\begin{align*}
F^2\CH_0(A)_{\textnormal{tors}} &= \Ker \left(  \CH_0(A)_{\textnormal{tors}} \to \rm H^1(G, \rm H^{2g-1}(A(\CC), \ZZ(g)))\right) \\
&= \left\{ \alpha \in \CH_0(A)_{\textnormal{tors}} \mid \textnormal{AJ}(\alpha) \in A(\RR)^0 \right\}.
\end{align*}
\end{definition}
\noindent
For the second equality in Definition \ref{F2def}, see Lemma \ref{lemone}. 
%Of course, one may proceed to define a decreasing filtration $F^\bullet$ on $\CH_0(A)_{\textnormal{tors}}$, but we will not need this. 
%\begin{comment}
%Inductively define a decreasing filtration $F^\bullet$ on the Chow group of zero-cycles $\CH_0(A)$ together with homomorphisms 
%\begin{align}\label{induction}
%F^p\CH_0(A) \to F^p\rm H^{2g}_G(A(\CC),\ZZ(g)), \quad p \in \ZZ_{\geq 0}
%\end{align}
%in the following way. Define $F^0\CH_0(A) = \CH_0(A)$. Having defined $F^p\CH_0(A)$ and the homomorphism (\ref{induction}), define $F^p\CH_0(A)$ to be the kernel of the induced map 
%\[
%F^p\CH_0(A) \to \rm H^p(G, \rm H^{2g-p}(A(\CC),  \ZZ(g))). 
%\] 
%For example, 
%\[
%F^1\CH_0(A) = \CH_0(A)_{\text{hom}} = \Ker\left(  \CH_0(A) \to \rm H^{2g}(A(\CC), \ZZ(g))^G\right)
%\]
%is the group of zero-cycles of degree zero, and 
%\[
%F^2\CH_0(A) = \Ker\left(  \CH_0(A)_{\text{hom}} \to \rm H^1(G, \rm H^{2g-1}(A(\CC), \ZZ(g))) \right). 
%\]
%\end{comment}

\begin{proof}[Proof of Proposition \ref{keyprop}]
Let $d = \dim(J)$ and $g  = \dim(A)$. Consider the following commutative diagram:
\[
\xymatrix{
F^2\CH_0(J)_{\tors} \ar[r] \ar[d] & F^2\CH_0(A)_{\tors} \ar[d] \\
F^2\rm H^{2d}_G(J(\CC), \ZZ(d))_0 \ar[r] & F^2\rm H^{2g}_G(A(\CC), \ZZ(g))_0. 
}
\]
Its vertical arrows are surjective by Lemma \ref{lemone} and Definition \ref{F2def}. To prove the surjectivity of the lower horizontal arrow, it thus suffices to prove the surjectivity of the upper horizontal arrow. Let this be our first goal.
\\
\\
For a real abelian variety $X$ of dimension $n$ over $\RR$, define $\Lambda_X = \rm H^{2n-1}(X(\CC), \ZZ(n))$. Recall the canonical identification $\pi_0(X(\RR)) = \rm H^1(G, \Lambda_X)$, see the sequence (\ref{pizero}). By Lemma \ref{lemone}, the rows in the following commutative diagram are exact:
\begin{align} \label{snake}
\xymatrix{
0 \ar[r] & F^2\CH_0(J)_{\tors} \ar[r] \ar[d] & \CH_0(J)_{\tors} \ar[r] \ar[d]^\phi& \rm H^1(G, \Lambda_J) \ar[r] \ar[d]^f & 0  \\
0 \ar[r] & F^2\CH_0(A)_{\tors} \ar[r] & \CH_0(A)_{\tors} \ar[r]& \rm H^1(G, \Lambda_A) \ar[r] & 0.  
}
\end{align}
The middle vertical arrow of diagram (\ref{snake}) is surjective by Lemma \ref{lemfour}, and its right vertical arrow is surjective by the surjectivity of $\pi_0(J(\RR)) \to \pi_0(A(\RR))$; the latter follows from the surjectivity of $J(\RR) \to A(\RR)$. %If we denote by $\phi$ the homomorphism $\CH_0(J)_{\tors} \to \CH_0(A)_{\tors}$ induced by $f$, then 
\\
\\
\emph{Claim \hypertarget{claimone}{1}:} The kernel of $\phi \colon \CH_0(J)_{\tors} \to \CH_0(A)_{\tors}$ surjects onto the kernel of $f \colon \pi_0(J(\RR)) \to \pi_0(A(\RR))$. 

Indeed, this follows from the following commutative diagram, in which the two-headed arrows are surjective:
\begin{align}\label{anotherdiagram}
\xymatrix{
\Ker(\phi) \ar[dd] \ar[rrr] & && \Ker(f)  \ar[dd]  \\
& \CH_0(B)_{\tors}\ar[dl] \ar[ul] \ar@{->>}[r] &  \pi_0(B(\RR)) \ar[dr] \ar@{->>}[ur]&\\
\CH_0(J)_{\tors} \ar[d]^{\phi} \ar[rrr] && & \pi_0(J(\RR))  \ar[d]^f \\
\CH_0(A)_{\tors} \ar[rrr] &&& \pi_0(A(\RR)).
}
\end{align}
The fact that $\CH_0(B)_{\tors} \to \pi_0(B(\RR))$ is surjective follows from Lemma \ref{lemone}. The fact that $\pi_0(B(\RR))$ surjects onto the kernel of $f \colon \pi_0(J(\RR)) \to \pi_0(A(\RR))$ follows from the exact sequence in group cohomology
\begin{align} \label{groupcohomologysequence}
0 \to \Lambda_B^G \to \Lambda_J^G \to \Lambda_A^G \to \rm H^1(G, \Lambda_B) \to \rm H^1(G, \Lambda_J) \to \rm H^1(G, \Lambda_A) \to 0 
\end{align}
arising from the short exact sequence of $G$-modules $0 \to \Lambda_B \to \Lambda_J \to \Lambda_A \to 0$. 
%We conclude from diagram (\ref{anotherdiagram}) that the homomorphism 
%\[
%\Ker(\phi) \to \Ker(f)
%\]
%is surjective as desired. 
\\
\\
Claim \hyperlink{claimone}{1} implies the surjectivity of the left vertical arrow in (\ref{snake}) by the snake lemma. It remains to prove the surjectivity of the map $\hat{f}^\ast$ in the proposition. \\
\\
\emph{Claim \hypertarget{claimtwo}{2}:} The maps $f_\ast$ and $\hat{f}^\ast$ fit into the following commutative diagram, where the arrows $\twoheadrightarrow$ are surjective and the arrows $\xrightarrow{\sim}$ are isomorphisms:
{\footnotesize
\begin{align} \label{hugediagram}
%\xymatrixrowsep{1.2pc}
\xymatrixcolsep{0.3pc}
\xymatrix{
F^2\rm H^{2d}_G(J(\CC), \ZZ(d))_0 \ar@{->>}[dd]^{f_\ast}\ar@{->>}[dr]& & F^2\rm H^4_G(\wh J(\CC), \ZZ(2))_0\ar[dd]^(.35){\hat{f}^\ast}
\ar[dr]^{\sim}& \\
& \rm H^2(G, \rm H^{2d-2}(J(\CC), \ZZ(d)))_0\ar[dd]^{\rm H^2(G, f_\ast)}\ar[rr]_{\rm H^2(G, \mr F_J)\quad\quad\quad}^{\sim\quad\quad\quad\quad} &&\rm H^2(G, \rm H^2(\wh J(\CC), \ZZ(2)))_0\ar[dd]^{\rm H^2(G, \hat{f}^\ast)} \\
F^2\rm H^{2g}_G(A(\CC), \ZZ(g))_0 \ar@{->>}[dr]& & F^2\rm H^4(\wh A(\CC), \ZZ(2))_0 \ar[dr]^{\sim}& \\
& \rm H^2(G, \rm H^{2g-2}(A(\CC), \ZZ(g)))_0 \ar[rr]^\sim_{\rm H^2(G, \mr F_A)}&&\rm H^2(G, \rm H^2(\wh A(\CC), \ZZ(2)))_0. 
}
\end{align}
}
Since the surjectivity of $f_\ast$ on the left of diagram (\ref{hugediagram}) has already been proved, claim \hyperlink{claimtwo}{2} follows from the functoriality of Fourier transforms (see \cite[(3.7.1)]{moonenpolishchuk}), Proposition \ref{vanishing} and Lemma \ref{topologicalcompatibility}. 

We can now finish the proof. Consider the diagram (\ref{hugediagram}). By the commutativity on the left, the morphism 
$
\rm H^2(G, f_\ast)$ %\colon \rm H^2(G, \rm H^{2d-2}(J(\CC), \ZZ(d)))_0 \to \rm H^2(G, \rm H^{2g}(A, \ZZ(g)))_0
is surjective. By the commutativity of the square on the front, this implies that $\rm H^2(G, \hat{f}^\ast)$ %\colon \rm H^2(G, \rm H^2(\wh J(\CC), \ZZ))_0 \to \rm H^2(G, \rm H^2(\wh A, \ZZ))_0
is surjective which, by the commutativity on the right hand side of the diagram, implies that the morphism $\hat{f}^\ast \colon F^2\rm H^4_G(\wh J(\CC), \ZZ(2))_0 \to F^2\rm H^4(\wh A(\CC), \ZZ(2))_0$ is surjective. 
\end{proof}
%the fact that for an abelian variety $X$ of dimension $n$ over $\RR$, the Fourier transform 
%\[
%\mr F_X \colon \rm H^{2n-2}(X(\CC), \ZZ(n)) \xrightarrow{\sim} \rm H^{2}(\wh X(\CC), \ZZ) 
%\]
%identifies $\rm H^{2n-2}(X(\CC), \ZZ(n))_0$ with $\rm H^{2}(\wh X(\CC), \ZZ)_0$. To prove the latter, we may assume $X$ to be a product of elliptic curves, in which case it follows from \S\ref{eighttwo} and diagram (\ref{bottom}). 
\begin{lemma} \label{gabber}
Let $A$ be an abelian variety over $\RR$. Then $A$ contains a smooth, proper, geometrically connected curve $C$ over $\RR$ that passes through $0 \in A(\RR)$ in such a way that the following two conditions hold: the induced homomorphism $f \colon J(C) \to A$ is surjective with connected kernel, and the homomorphism $f_\RR \colon J(C)(\RR) \to A(\RR)$ is surjective. 
\end{lemma}
\begin{proof}
Let $S \subset A(\RR)$ be a finite set of points containing $0 \in A(\RR)$ and at least one point of each connected component of $A(\RR)$. Since $\RR$ is infinite, Bertini's theorem can be applied: there exists a smooth and geometrically connected hyperplane section $Z \subset A$ passing through $S$ \cite[II, Theorem 8.18]{HAG}, \cite[Footnote 12, page 32]{debarrebook}. Let $g = \dim(A)$. By taking $g-1$ general such hyperplane sections, we get a smooth curve $C$ in $A$ that contains $S$. It remains to prove that $C$ satisfies the requirements stated in the proposition. 

Write $J = J(C)$ and consider the map $f \colon J \to A$ arising from the inclusion $(C,0) \hookrightarrow (A,0)$. By the proof of Theorem 10.1 in \cite{milnejacobian}, the homomorphism $f$ is surjective. We claim that the kernel of $f \colon J_\CC \to A_\CC$ is connected. This follows from the fact that $\rm H_1(C(\CC), \ZZ) \to \rm H_1(A(\CC), \ZZ)$ is surjective by the Lefschetz hyperplane theorem. Alternatively, see \cite[Proposition 2.4]{gabberspacefilling} for an algebraic argument.

We have that $f_\RR^0 \colon J(\RR)^0 \to A(\RR)^0$ is surjective by Lemma \ref{anotherlemma}. Because $S$ is contained in the image of $f_\RR$, we conclude that $f_\RR$ is surjective. 
\end{proof}
\noindent
Let us finish the proof of Theorem \ref{fourierreduction}. Let $A$ be an abelian variety over $\RR$. The group $F^3\rm H^4_G(A(\CC), \ZZ(2))_0$ is trivial by Proposition \ref{vanishing}, so it remains to prove that $F^2\rm H^4_G(A(\CC),\ZZ(2))_0$ is algebraic. Let $C \subset \wh A$ be a real algebraic curve that satisfies the conditions of Lemma \ref{gabber}. By Proposition \ref{keyprop}, the pull-back homomorphism
\[
\hat{f}^\ast \colon F^2\rm H^{4}_G(\wh J(C)(\CC),\ZZ(2))_0 \to F^2\rm H^{4}_G(A(\CC),\ZZ(2))_0
\]
is surjective, where $f \colon J(C) \to \wh A$ is the homomorphism induced by the inclusion of $(C,0)$ in $(\wh A, 0)$. By the proof of Theorem \ref{fourierreduction} in the Jacobian case, we know that $ F^2\rm H^{4}_G(\wh J(C)(\CC),\ZZ(2))_0$ is algebraic. Therefore $F^2\rm H^{4}_G(A(\CC),\ZZ(2))_0$ is algebraic. 
\end{proof}
%By construction, we have that $N \cdot \xi = 0$, and also that 
%$
%\pi(\xi) = \pi(y) = \AJ(\alpha_\CC). 
%$
%^we claim that $N \cdot g \in B(\RR)^0$. 

\subsection{Reduction to the Abel-Jacobi map}
Let us prove Corollary \ref{questionreductioncorollary} and Proposition \ref{abeliansurface}. 

\begin{proof}[Proof of Corollary \ref{questionreductioncorollary}]
This follows Lemma \ref{lemma:important0}, Theorem \ref{theorem1}, equalities (\ref{canonical-HS}),  %$$F^1\rm H^4_G(A(\CC), \ZZ(2))_0 = \rm H^4_G(A(\CC), \ZZ(2))_0[2] \quad \quad (\textnormal{see } ,$$ 
filtration (\ref{topologically-distinguished-filtration}) and Theorem \ref{fourierreduction}. 
%\[
%0 \to F^1\rm H^4_G(A(\CC), \ZZ)_0 \to \Hdg^4_G(A(\CC), \ZZ)_0 \to \Hdg^4(A(\CC), \ZZ)^G \to 0. 
%\]
\end{proof}

%and Proposition \ref{vanishing}, 
%because if this holds, then $A$ satisfies $\RR$IHC$^2$ if and only if $A$ satisfies $\RR$IHC$^2$ modulo torsion and $\CH^2(A)_{\textnormal{hom}} \to \Hdg^4(A(\CC),\ZZ)^G$ is surjective. Indeed, this follows from the exactness of the following sequences:
%\begin{align*}
%& 0  \to F^2\rm H^4_G(A(\CC),\ZZ)_0 \to F^1\rm H^4_G(A(\CC),\ZZ)_0  \to \rm H^1(G,\rm H^3(A(\CC),\ZZ))_0, \to 0 \\
%& 0 \to F^1\rm H^4_G(A(\CC),\ZZ)_0  \to \Hdg^4_G(A(\CC),\ZZ) \to \Hdg^4(A(\CC),\ZZ)^G \to 0.
%\end{align*}

%\item[\hypertarget{eightsix}{8.6.}] \emph{Proof of the 

%Corollaries \ref{questionreductioncorollary} and \ref{abeliansurface}. 

\begin{proposition} \label{prop:productsurfacecurve}
Let $B$ be a real abelian surface, and let $E$ be a real elliptic curve whose real locus $E(\RR)$ is connected. The following homomorphism is surjective: 
\begin{align}\label{surjectivityforproducts}
\CH^2(B \times E)_{\hom} \to \rm H^1(G, \rm H^3(B(\CC) \times E(\CC), \ZZ(2))).
\end{align}
In particular, $B \times E$ satisfies the real integral Hodge conjecture, i.e. Proposition \ref{abeliansurface} holds. 
\end{proposition}

\begin{proof}
% In fact, we shall prove that the group $\rm H^1(G, \rm H^3(A(\CC), \ZZ(2)))$ is algebraic. Note that this does not only prove Corollary \ref{abeliansurface}, but also implies that $\rm H^1(G, \rm H^3(A(\CC), \ZZ(2)))_0 = \rm H^1(G, \rm H^3(A(\CC), \ZZ(2)))$. 
First observe that 
\[
\rm H^1(G, \rm H^2(B(\CC), \ZZ) \otimes \rm H^1(E(\CC), \ZZ)) = (0)
\]
because $\rm H^1(E(\CC), \ZZ) \cong \ZZ[G]$ by our hypothesis that $E(\RR)$ is connected \cite[Chapter IV, Example 3.1]{vanhamel}. By the K\"unneth formula, the canonical morphism 
\begin{align} \label{equiv:kunneth}
\begin{split}
\rm H^1(G, \rm H^3(B(\CC), \ZZ(2)) \;\; &\bigoplus \;\; \rm H^1(G, \rm H^1(B(\CC), \ZZ(1)) \otimes \rm H^2(E(\CC), \ZZ(1))) \\
  & \to \;\; \rm H^1(G, \rm H^3(B(\CC) \times E(\CC), \ZZ(2)))
\end{split}
\end{align}
is therefore an isomorphism. Since $\rm H^2(E(\CC), \ZZ(1)) \cong \ZZ$ as $G$-modules, the canonical map
%\[
%\rm H^1(G, \rm H^1(B(\CC), \ZZ(1))) \;\; \bigotimes \rm H^0(G, \rm H^2(E(\CC), \ZZ(1))) \to \rm H^1(G, \rm H^1(B(\CC), \ZZ(1)) \otimes \rm H^2(E(\CC), \ZZ(1)))
%\]
\begin{align} \label{equiv:gmodstructure}
\begin{split}
\rm H^1(G, \rm H^1(B(\CC), \ZZ(1))) \;\; &\bigotimes \;\;  \rm H^0(G, \rm H^2(E(\CC), \ZZ(1))) \\
  & \to \;\;  \rm H^1(G, \rm H^1(B(\CC), \ZZ(1)) \otimes \rm H^2(E(\CC), \ZZ(1)))
  \end{split}
\end{align}
is an isomorphism as well. To simplify notation, for any abelian variety $X$ over $\RR$, define $$\rm H^i_X(j) = \rm H^i(X(\CC), \ZZ(j)).$$ %Define $K = \rm H^3(B(\CC), \ZZ(2))$, $L = \rm H^1(B(\CC), \ZZ(1)), M = \rm H^2(E(\CC), \ZZ(1))$ and $N = \rm H^3(A(\CC), \ZZ(2))$, and 
Consider the following commutative diagram: 
{\footnotesize
\begin{align} \label{equiv:important}
%\xymatrixcolsep{0.6pc}
\xymatrix{
\CH^2(B)_{\tors} \bigoplus  \left(\CH^1(B)_{\hom} \otimes \CH^1(E) \right) \ar[r] \ar[d] & \CH^2(B \times E)_{\hom} \ar[d] \\
F^1\rm H^4_G(B(\CC), \ZZ(2)) \bigoplus  \left(F^1\rm H^2_G(B(\CC), \ZZ(1)) \otimes \rm H^2_G(E(\CC), \ZZ(1)) \right) \ar[r]  \ar[d]& F^1\rm H^4_G(B(\CC) \times E(\CC), \ZZ(2)) \ar[dd] \\
\rm H^1(G, \rm H^3_B(2)) \bigoplus \left( \rm H^1(G, \rm H^1_B(1)) \otimes \rm H^0(G, \rm H^2_E(1)) \right) \ar[d]^{\wr} &  \\
\rm H^1(G, \rm H^3_B(2)) \bigoplus \rm H^1(G, \rm H^1_B(1) \otimes \rm H^2_E(1)) \ar[r]^{\hspace{10mm} \sim} & \rm H^1(G, \rm H^3_{B \times E}(2)). 
}
\end{align}
}
The indicated isomorphisms $\xrightarrow{\sim}$ in diagram (\ref{equiv:important}) arise from the isomorphisms (\ref{equiv:kunneth}) and (\ref{equiv:gmodstructure}) above. The map $\CH^2(B)_{\tors} \to \rm H^1(G, \rm H^3_B(2))$ is surjective by Lemma \ref{lemone}, and the map $\CH^1(B)_{\hom} \to \rm H^1(G, \rm H^1_B(1))$ is surjective by the real integral Hodge conjecture for divisors (see Section \ref{subsec:divisors}) and the fact that the topological condition for degree two cohomology classes is trivial \cite[\S2.3.1]{BW20}. Therefore, the vertical composition on the left of (\ref{equiv:important}) is surjective, which implies that the vertical composition on the right -- which is  (\ref{surjectivityforproducts}) -- is surjective. By Corollary \ref{questionreductioncorollary}, this implies that $B \times E$ satisfies the real integral Hodge conjecture. 
\end{proof}
%\[
%\CH_1(A) = \CH^2(B \times E)_{\hom} \to \rm H^1(G, \rm H^3_A(2)) = \rm H^1(G, \rm H^3(A(\CC), \ZZ(2)))
%\]
%Let $A = B \times E$ for a real abelian surface $B$ and a real elliptic curve $E$ with connected real locus $E(\RR)$. 
%as in the statement of the corollary. %To simplify notation, let us write $X = B(\CC)$ and $Y = E(\CC)$. 
%By Corollary \ref{questionreductioncorollary}, it suffices to show that the group 
%$
%rm H^1(G, \rm H^3(A(\CC), \ZZ(2)))_0$ is algebraic, in the sense of Definition \ref{isogenydef:subgroups}.
%\end{proof}
%The proof of Corollary \ref{abeliansurface} has the following corollary. 
\noindent
The topological condition (\ref{topologicalcondition}) does not appear on the right hand side of (\ref{surjectivityforproducts}). %Proposition \ref{prop:productsurfacecurve} 
This has the following corollary. To state it, recall the following fact. %For a real abelian variety $A$, denote by $n(A) = \va{\pi_0(A(\RR))}$ the number of connected components of its real locus. 
For a real abelian variety $A$ of dimension $g$, one has $\va{\pi_0(A(\RR))} \leq 2^g$ (see \cite[\S1]{grossharris}) and for every non-negative pair of integers $(i,g)$ with $i \leq g$, there is a real abelian variety $A$ of dimension $g$ such that $\va{\pi_0(A(\RR))} = 2^i$ (take a suitable product of elliptic curves). In particular, if $g = \dim(A) = 3$, then $\va{\pi_0(A(\RR))} \in \{1,2,4,8\}$. 
%Seeor \cite[Chapter IV, Example 3.1]{vanhamel}. 

\begin{corollary}
Let $A$ be a real abelian threefold, and suppose that $\va{\pi_0(A(\RR))}  \neq 8$. Then the canonical map 
\begin{align} \label{AJsurjectivity}
F^1\rm H^4_G(A(\CC), \ZZ(2))_0 \to \rm H^1(G, \rm H^3(A(\CC), \ZZ(2)))
\end{align}
is surjective. Thus, $A$ satisfies the real integral Hodge conjecture if and only if the Abel-Jacobi map $$\CH_1(A)_{\hom} \to \rm H^1(G, \rm H^3(A(\CC), \ZZ(2)))$$ is surjective. 
%$A(\RR)$ does not have the maximal number of connected components (i.e.
\end{corollary}
\begin{proof}
To prove the surjectivity of (\ref{AJsurjectivity}), we may replace $A(\CC)$ by any differentiable $G$-manifold which is $G$-equivariantly diffeomorphic to $A(\CC)$. In particular, we may assume that %Since the differentiable $G$-manifold structure of $A(\CC)$ can be realized as the complex locus of a product of elliptic curves, we may assume that 
$A = B \times E$, where $B$ is an abelian surface and $E$ an elliptic curve whose real locus is connected (see (\ref{topologicalstatement})). The surjectivity of (\ref{AJsurjectivity}) follows then from the fact that the group $\rm H^1(G, \rm H^3(B(\CC) \times E(\CC), \ZZ(2)))$ is algebraic by Proposition \ref{prop:productsurfacecurve}. The second statement follows from the first, due to Corollary \ref{questionreductioncorollary}. 
\end{proof}
%as was shown in the proof of Corollary \ref{abeliansurface}. 
%

\cleardoublepage
\part{Summary in French}
\cleardoublepage
\chapter{R\'esum\'e en fran\c{c}ais}
\label{ch:francais}

\addtocontents{toc}{\protect\setcounter{tocdepth}{0}}
\section{Introduction} 
\addtocontents{toc}{\protect\setcounter{tocdepth}{1}}

Cette thèse porte sur l'étude des cycles alg\'ebriques et des espaces de modules en géométrie algébrique réelle. Nous esp\'erons que le pr\'esent manuscrit puisse apporter une contribution à la compréhension de ces concepts, tant pris individuellement que dans leurs interactions mutuelles. Afin de mener ce projet \`a bien, nous \'etudions au fil de ce texte les ph\'enom\`enes suivants : 

\begin{itemize}
\item La structure hyperbolique de l'espace de modules des quintiques binaires r\'eelles.
\item La distribution des variétés abéliennes non-simples dans une famille de variétés abéliennes réelles.
\item Courbes algébriques sur des variétés abéliennes complexes et réelles. 
\end{itemize}

\noindent
Comme nous allons le voir, les résultats ainsi que les méthodes de ces diff\'erents sujets sont étroitement liés. Un outil omnipr\'esente tout au long de ce projet est la \emph{théorie de Hodge équivariante} : l'utilisation de la théorie de Hodge pour étudier la cohomologie, les modules et les cycles des variétés algébriques réelles. 

\addtocontents{toc}{\protect\setcounter{tocdepth}{0}}
\section{R\'esum\'e du Chapitre \ref{ch:realmodulispaces}: Espaces de modules r\'eelles} \label{realmoduli-fr}
\addtocontents{toc}{\protect\setcounter{tocdepth}{1}}

Notre premier résultat fournit une méthode générale de construction des espaces de modules en géométrie algébrique réelle. Dans le cadre classique de la géométrie complexe, l'espace de modules d'une catégorie donn\'ee d'objets algébriques peut souvent être construit en appliquant le théorème de Keel-Mori au champ de modules correspondant. Par exemple, cela fonctionne lorsqu'un tel champ est séparé et de Deligne-Mumford sur $\CC$. Pour un champ de modules défini sur les réels, il est \emph{a priori} moins ais\'e d'obtenir une structure analytique sur l'ensemble des classes d'isomorphisme de son lieu réel, cet ensemble n'\'etant en g\'en\'eral pas en bijection avec le lieu réel de l'espace de modules grossier. Pourtant, nous parvenons \`a \'etablir le résultat suivant.

\begin{theoreme} \label{fromstacktoorbifold-fr}
Soit $\mr X$ un champ algébrique localement de type fini sur $\RR$. Il existe une topologie sur l'ensemble des classes d'isomorphisme $\va{\mr X(\RR)}$ du lieu réel de $\mr X$, fonctorielle en $\mr X$ et qui coincide avec la topologie analytique réelle lorsque $\mr X$ est un schéma. En outre, si $\mr X$ est lisse et de Deligne-Mumford sur $\RR$, l'espace $\va{\mr X(\RR)}$ porte une structure d'orbifold fonctorielle en $\mr X$. 
\end{theoreme}
\noindent
Comment cette approche se compare-t-elle aux m\'ethodes plus classiques pour construire des espaces de modules en géométrie algébrique réelle ? Consid\'erons deux cas particuliers : soit $\ca A_g$ le champ des variétés abéliennes principalement polarisées de dimension $g$, et soit $\ca M_g$ le champ des courbes propres, lisses et géométriquement connexes de genre $g$.  %Dans \cite{grossharris, seppalasilhol2}, Gross, Harris, Sepp\"al\"a et Silhol dotent ces espaces d'une structure semi-analytique réelle. 
Il est bien connu qu'il existe des uniformisations analytiques complexes 
\[
\ca A_g(\CC) \cong \Sp_{2g}(\ZZ) \setminus \bb H_g, \quad \textnormal{ et } \quad \ca M_g(\CC) \cong \Gamma_g \setminus \ca T_g.\] Ici, $\bb H_g$ (resp. $\ca T_g$) est une vari\'et\'e complexe param\'etrant les vari\'et\'es ab\'eliennes complexes (resp. des courbes complexes) dot\'ees d'une certaine structure suppl\'ementaire; $\Sp_{2g}(\ZZ)$ (resp. $\Gamma_g$) est un groupe discret agissant holomorphiquement et proprement discontinûment sur $\bb H_g$ (resp. $\ca T_g$) en permutant lesdites structures suppl\'ementaires; enfin, l'application 
%et les rev\^etements 
%covers % un espace moduli pour les variétés abéliennes (resp. les courbes) dotées d'une certaine structure supplémentaire, 
$\bb H_g \to \ca A_g(\CC)$ (resp. $\ca T_g \to \ca M_g(\CC)$) est obtenue en oubliant ces structures. 
%En rigidifiant les objets dans $ \ca A_g(\CC)$ (resp. $\ca M_g(\CC)$) d'une mani\`ere appropri\'ee. 

Pour obtenir des espaces de modules r\'eels pour ces probl\`emes de modules, on d\'efinit des ensembles d'involutions anti-holomorphes \[\{\sigma \colon \bb H_g \to \bb H_g\} \quad \textnormal{ et } \quad \{\tau \colon \ca T_g \to \ca T_g\}\] de sorte que toute vari\'et\'e ab\'elienne r\'eelle principalement polari\'ee de dimension $g$ (resp. toute courbe r\'eelle de genre $g$) se trouve dans $\bb H_g^\sigma$ (resp. $\ca T_g^\tau$) pour un certain $\sigma$ (resp. $\tau$). Pour des sous-groupes $\Sp_{2g}(\ZZ)(\sigma) \subset \Sp_{2g}(\ZZ)$ et $ \Gamma_g(\tau) \subset \Gamma_g$ d\'efinis de mani\`ere appropri\'ee, on obtient des bijections (voir \cite{grossharris, silholsurfaces, seppalasilhol2}):
\begin{align}\label{analytictopology-fr}
\ca A_g(\RR) \cong \bigsqcup_\sigma \Sp_{2g}(\ZZ)(\sigma) \setminus \bb H_g^\sigma \quad \textnormal{ and } \quad \ca M_g(\RR) \cong \bigsqcup_\tau \Gamma_g(\tau) \setminus \ca T_g^\tau.
\end{align}
Comme dans le cas complexe, il se trouve que ces applications des p\'eriodes sont des hom\'eomorphismes. 
\begin{theoreme}
Les bijections dans (\ref{analytictopology-fr}) sont des homéomorphismes, o\`u les ensembles $\va{\ca A_g(\RR)}$ et $\va{\ca M_g(\RR)}$ sont munis de la topologie donn\'ee par le Théorème \ref{fromstacktoorbifold-fr}. 
\end{theoreme}

\addtocontents{toc}{\protect\setcounter{tocdepth}{0}}
\section{R\'esum\'e du Chapitre \ref{ch:density}: Lieux de Noether-Lefschetz r\'eels}\label{intro:sub:realabeliansubvarietiesinfamily-fr}
\addtocontents{toc}{\protect\setcounter{tocdepth}{1}}

Dans le Chapitre \ref{ch:density}, nous d\'eveloppons un certain crit\`ere de densit\'e, analogue au crit\`ere de Green--Voisin pour les lieux de Noether-Lefschetz d'une variation de structure de Hodge de poids deux (voir \cite[\S17.3.4]{voisin}). Pour formuler ce r\'esultat et le comparer au crit\`ere bien connu de Green--Voisin, on consid\`ere une famille holomorphe de variétés abéliennes complexes polarisées
\begin{align} \label{NLfamily-fr}
\left( \psi \colon A \to B, \quad s \colon B \to A, \quad E \in R^2\psi_\ast\ZZ \right).
\end{align} 
\noindent
C'est-à-dire que $\psi$ est une submersion propre holomorphe, $s$ une section holomorphe de $\psi$ et on suppose que pour chaque $t \in B$, la fibre $A_t = \psi^{-1}(t)$ est une variété abélienne complexe polaris\'ee d'origine $s(t)$ et de polarisation $E_t \in \rm H^2(A_t, \ZZ)$. On suppose que $B$ est connexe. Pour chaque entier positif $k < g$, notons $S_k \subset B$ l'ensemble des $t \in B$ tels que $A_t$ contient une sous-vari\'eti\'e ab\'elienne complexe de dimension $k$. Le résultat principal de \cite{Colombo1990} est le suivant : 

\begin{theoreme-non}[Colombo--Pirola] \label{CPth-fr}
Si la Condition \ref{criterion} du Chapitre \ref{ch:density} est vérifiée, alors $S_k$ est dense pour la topologie euclidienne dans $B$. 
\end{theoreme-non}

\noindent
Supposons \`a notre tour qu'il existe des involutions anti-holomorphes $\sigma \colon B \to B$ et $\tau \colon A \to A$, qui commutent avec $\psi$ et $s$ et qui sont compatibles \`a la polarisation $E$. %Let $A \to B$ be a family of polarized real abelian varieties. If $B$ is connected and
En posant
\[
R_k= \set{t \in B(\RR) \mid A_t \textnormal{ contient une sous-vari\'et\'e ab\'elienne r\'eelle de dimension $k$}},
\]
nous obtenons le r\'esultat suivant : 

\begin{theoreme}\label{intro:NLtheorem-fr}
Si la Condition \ref{criterion} du Chapitre \ref{ch:density} est v\'erif\'ee, alors $R_k$ est dense pour la topologie euclidienne dans $B(\RR)$. 
\end{theoreme}

\noindent
Il est intéressant de comparer ces th\'eor\`emes avec des résultats analogues pour les lieux de Noether-Lefschetz des hypersurfaces dans $\PP^3$. Fixons $d \geq 4$ un entier et consid\'erons l'hypersurface lisse universelle $\PP^3 \times \mr B \supset \mr S \to \mr B$ de degr\'e $d$. Classiquement, le \emph{lieu de Noether-Lefschetz} est d\'efini comme l'ensemble des $t \in \mr B(\CC)$ tel que $\mr S_t$ contient une courbe qui n'est pas une intersection complète. Dans ce cas, le résultat principal de \cite{NLlocus} assure que ce lieu est dense pour la topologie euclidenne dans $\mr B(\CC)$, bien que son intérieur soit vide \cite{MR0033557}. La situation sur $\RR$ s'avère être plus délicate que la situation sur $\CC$. Par analogie avec ce qui précède, on peut définir le \textit{lieu de Noether-Lefschetz r\'eel} comme l'ensemble des $t \in \mr B(\RR)$ tel que $\Pic(\mr S_t) \not \cong \ZZ$. En contraste avec la situation complexe, il existe pour $d=4$ une composante connexe $K$ de $\mr B(\RR)$ dont l'intersection avec le lieu de Noether-Lefschetz réel est vide. %En contraste avec la situation complexe, \textcolor{pour $d = 4$ il existe une composante connexe $K$ de $\mr B(\RR)$ telle que} l'intersection de $K$ avec le lieu de Noether-Lefschetz r\'eel est vide. 
Il existe un critère de Green--Voisin sur les réels \cite [Proposition 1.1]{benoistttt}, mais son hypothèse est plus compliquée et il ne s'applique qu'à une composante connexe de $\mr B(\RR)$ à la fois. 

Il est remarquable que pour le probl\`eme analogue pour les vari\'et\'es ab\'eliennes, aucune de ces difficult\'es n'apparaît. %La condition \ref{criterion} n'implique pas du tout la structure réelle de (\ref{NLfamily-fr}), et implique une densité dans chaque composante connexe du lieu réel $B(\RR)$. 
Les applications du Théorème \ref{intro:NLtheorem-fr} sont donc pléthores : 

%The situation over $\RR$ turns out to be more delicate than the situation over $\CC$. Note that the above universal surface $\mr S \to \mr B$ is naturally defined over $\RR$. In analogy with the above, one may thus define the \textit{real Noether-Lefschetz locus} as the set of real surfaces $S$ in $\bb P^3_{{\bb R}}$ with $\Pic(S) \not \cong {\bb Z}$. In contrast with the complex situation, there exists a connected component $K$ of the real moduli space of degree four surfaces in $\PP^3_\RR$, such that every $S$ in $K$ satisfies $\Pic(S) \cong \ZZ$. There is a Green--Voisin criterion over the reals \cite[Proposition 1.1]{benoistttt}, but its hypothesis is more complicated and only applies to one component of $\mr B(\RR)$ at a time. 

%It is therefore remarkable that for loci of non-simple abelian varieties, none of these problems occur. Condition \ref{criterion} does not involve the real structure of (\ref{NLfamily-fr}) at all, and implies density in every connected component of the real locus $B(\RR)$. 
%Accordingly, the applications of Theorem \ref{intro:NLtheorem-fr} are generous:

\begin{theoreme} \label{abeliancyclecorollari-fr} Soient $k$ et $g$ deux entiers positifs avec $k < g$. 
\begin{enumerate}
\item
%Let $k<g$ be positive integers, and consider $S_k(\ca A_g(\RR)) \subset \ca A_g(\RR)$, the set of moduli points of principally polarized abelian varieties of dimension $g$ over $\RR$ containing a abelian subvariety of dimension $k$ over $\RR$. Then $S_k(\ca A_g(\RR))$ is dense in $\ca A_g(\RR)$. 
Les points dans $\va{\ca A_g(\RR)}$ correspondant \`a des variétés abéliennes r\'eelles contenant une sous-variété abélienne r\'eelle de dimension $k$ sont denses dans $\va{\ca A_g(\RR)}$. 
\item Si $ k  \leq 3 \leq g$, les points dans $\va{\ca M_g(\RR)}$ correspondant \`a des courbes alg\'ebriques r\'eelles $C$ admettant un morphisme $f \colon C \to A$ vers une vari\'et\'e ab\'elienne $A$ de dimension $k$ tel que $f(C(\CC))$ engendre le groupe $A(\CC)$ sont denses dans $\va{\ca M_g(\RR)}$. 
%Let $k \leq 3 \leq g$ be positive integers. Let $\ca T_g(\RR) \subset \ca A_g(\RR)$ be the Torelli locus, consisting of points attached to Jacobians $J(C)$ of real genus three curves $C$. Then $S_k(\ca A_g(\RR)) \cap \ca T_g(\RR)$ is dense in $\ca T_g(\RR)$. 
%Let $S_k(\ca M_g(\RR)) \subset \ca M_g(\RR)$ be the set of real algebraic curves $C$ admitting a map $\varphi: C \to A$ with $A$ a $k$-dimensional abelian variety over ${\bb R}$ such that $\varphi(C(\CC))$ generates $A(\CC)$. If $g \geq 3$ and $k \in \{1,2,3\}$, then $S_k(\ca M_g(\RR))$ is dense in $\ca M_g(\RR)$. 
\item Soit $V \subset \bb P\rm H^0(\bb P^2_{{\bb R}}, \OO_{\bb P^2_{{\bb R}}}(d))$ l'espace des courbes planes r\'eelles lisses de degr\'e $d$. Les points $t \in V$ tels que la courbe correspondante $C_t$ admet un morphisme non constant $C_t \to E$ vers une courbe elliptique réelle $E$ sont denses dans $V$. 
%smooth degree $d \geq 3$ real plane curves. 
%Let $S_k(V)$ be the set of $t \in V$ such that the corresponding curve $C_t$ admits a non-constant map $C_t \to E$ to a real elliptic curve $E$. Then $S_k(V)$ is dense in $V$. 
\end{enumerate}
\end{theoreme}

\addtocontents{toc}{\protect\setcounter{tocdepth}{0}}
\section{R\'esum\'e du Chapitre \ref{ch:glueing}:  Recoller des quotients de la boule} \label{recollement-fr}
\addtocontents{toc}{\protect\setcounter{tocdepth}{1}}

%It is relatively easy to prove that each connected component of the real moduli space of smooth binary quintics is isomorphic to an open subset in a quotient of the real hyperbolic plane by a discrete group of isometries, see the next Section \ref{intro:sub:fivepoints-fr}. Much harder is to prove that these isomorphisms extend to an isomorphism between the the real moduli space of stable binary quintics, which is compact and connected, and the quotient of the real hyperbolic plane by a non-arithmetic triangle group (see Theorem ... below). To prove such a statement, we develop a general method of glueing real hyperbolic quotient spaces in Chapter \ref{ch:glueing}. This builds on and generalizes work of Allcock, Carlson and Toledo \cite{realACTnonarithmetic, realACTsurfaces}. 
On verra plus tard (Section \ref{intro:sub:fivepoints-fr}) que chaque composante connexe de l'espace de modules des quintiques binaires r\'eelles lisses est isomorphe à un sous-ensemble ouvert d'un quotient du plan hyperbolique réel par un groupe discret d'isométries. En fait, il se trouve que ces isomorphismes s'étendent en un isomorphisme entre l'espace de modules réel des quintiques binaires stables et le quotient du plan hyperbolique réel par un groupe de triangles non arithmétique (voir le Théorème \ref{th:theorem02-fr} ci-dessous). Pour \'etablir un tel r\'esultat, nous développons une méthode générale de recollement des espaces quotients hyperboliques réels dans le Chapitre \ref{ch:glueing}, s'appuyant sur les travaux de Allcock, Carlson et Toledo \cite{realACTnonarithmetic, realACTsurfaces}. 
%Unitary Shimura varieties provide a suitable framework for these glueing techniques. Let $K$ be a CM field of degree $2g$ over $\QQ$ with ring of integers $\OO_K$, and let $\Lambda$ be a finite free $\OO_K$-module equipped with a hermitian form $h: \Lambda \times \Lambda \to \OO_K$. Suppose that $h$ has signature $(n,1)$ with respect to an embedding $\tau: K \to \CC$ and is definite other infinite places of $K$. Let $\CC H^n$ be the space of negative lines in $\Lambda \otimes_{\OO_K, \tau} \CC$ and $P\Gamma = \Aut(\Lambda, h) / \mu_K$ where $\mu_K \subset \OO_K^\ast$ is the group of finite units in $\OO_K$. Let $P\mr A$ be the quotient of the set of anti-unitary involutions $\alpha: \Lambda \to \Lambda$ by $\mu_K$. 
Les variétés unitaires de Shimura fournissent un cadre approprié pour ces techniques de collage ; nous expliquons ici sommairement et dans un langage plus \'el\'ementaire comment cela fonctionne. 
%comme nous allons l'expliquer maintenant. 

Soit $K$ un corps CM de degré $2g$ sur $\QQ$ d'anneau des entiers $\OO_K$. Soit $\Lambda$ un $\OO_K$-module libre de type fini, muni d'une forme hermitienne $h : \Lambda \times \Lambda \to \OO_K$. Supposons que la signature de $h$ soit $(n,1)$ par rapport à un plongement $\tau \colon K \to \CC$, et que $h$ soit définie par rapport aux autres places infinies de $K$. Soit $\CC H^n$ l'espace des droites négatives dans $\Lambda \otimes_{\OO_K, \tau} \CC$ et $P\Gamma = \Aut(\Lambda, h) / \mu_K$ où $\mu_K \subset \OO_K^\ast$ le groupe des unités finies de $K$. Notons $P\mr A$ le quotient de l'ensemble des involutions anti-unitaires $\Lambda \to \Lambda$ par $\mu_K$, et $\mr H = \cup_{h(r,r) = 1} \langle r_\CC \rangle ^\perp\subset \CC H^n$, o\`u $\langle r_\CC \rangle ^\perp$ est l'ensemble $\set{x \in \CCH^n \mid h(x,r) = 0}$. Supposons que la condition suivante soit satisfaite : \begin{condition-non}[voir \cite{orthogonalarrangements}] 
% $(\textcolor{blue}{\ast})$: 
Si $r, t \in \Lambda$ sont des \'el\'ements de norme un, et si $h(r,t) \neq 0$, alors soit $\langle r_\CC \rangle ^\perp  = \langle t_\CC \rangle ^\perp \subset \CCH^n$, soit $\langle r_\CC \rangle ^\perp   \cap  \langle t_\CC \rangle ^\perp = \emptyset \subset \CCH^n$. 
\end{condition-non}
\noindent
%For example, this holds under a condition on the CM field $K$ (see Theorem \ref{th:conditionsimplyhypothesis}) satisfied when $K$ is cyclotomic or quadratic (see Lemma \ref{lemma:discr}). (In fact, condition \ref{conditionast} is \textit{always} satisfied if one is willing to adapt the definition of $\mr H$, see Remark \ref{remark:avoidcondition}.)
Cette condition est toujours v\'erifi\'ee lorsque le corps $K$ v\'erifie certaines hypoth\`eses (voir les Conditions \ref{crucialcondition}), %Par exemple, la cette condition tient sous une condition sur le corps $K$, 
satisfaites lorsque $K$ est quadratique imaginaire ou \'egal \`a $\QQ(\zeta_p)$ pour un nombre premier $p>2$ (voir le Théorème \ref{th:conditionsimplyhypothesis} et le Lemme \ref{lemma:discr}). %(En fait, la condition \ref{conditionast} est \textit{toujours} satisfaite si l'on veut bien adapter la définition de $\mr H$, voir Remark \ref{remark:avoidcondition}).
%We claim that there is a canonical way to glue the different copies 

Nous fournissons un moyen canonique de recoller les différentes copies $\RR H^n_\alpha \coloneqq \left(\CC H^n\right)^\alpha \subset \CC H^n, \alpha \in P\mr A$ de l'espace hyperbolique réel $\RR H^n$ le long de l'arrangement d'hyperplans $\mr H$. Cela donne un espace topologique que nous désignons par $Y$, sur lequel agit $P\Gamma$. Définissons $P\Gamma_\alpha \subset P\Gamma$ comme étant le stabilisateur de $\RR H^n_\alpha$. 

Le but du Chapitre \ref{ch:glueing} est de d\'emontrer le théorème suivant.\begin{theoreme} \label{th:theorem03-fr}
L'espace topologique $P\Gamma \setminus Y$ porte une structure canonique d'orbifold hyperbolique réelle, pour laquelle %il existe un encastrement ouvert d'orbitales X -> Y. 
%L'espace topologique $P\Gamma \setminus Y$ porte une structure canonique d'orbifold hyperbolique réelle, %De plus, muni de cette structure, 
il existe un plongement ouvert d'orbifolds $$\coprod_{\alpha \in P\Gamma \setminus P\mr A}  [P\Gamma_\alpha \setminus \left(\RR H^n_\alpha - \mr H \right)] \hookrightarrow P\Gamma \setminus Y.$$ De plus, pour chaque composante connexe $C \subset P\Gamma \setminus Y$ il existe un r\'eseau $P\Gamma_C \subset \textnormal{PO}(n, 1)$ ainsi qu'un isomorphisme d'orbifolds hyperboliques réels $C \cong [P\Gamma_C \setminus \RR H^n]$. 
\end{theoreme}

\addtocontents{toc}{\protect\setcounter{tocdepth}{0}}
\section{R\'esum\'e du Chapitre \ref{ch:binaryquintics}: Modules des quintiques binaires r\'eelles}\label{intro:sub:fivepoints-fr}
\addtocontents{toc}{\protect\setcounter{tocdepth}{1}}
Soit $X \cong \bb A^6_\RR$ l'espace affine paramétrant les polynômes homogènes $F \in \RR[x,y]$ de degré cinq. Consid\'erons les sous-variétés $X_0 \subset X_s \subset X$ paramétrant les polynômes à racines distinctes, resp. les polynômes à racines de multiplicité au plus deux (c'est-à-dire stables au sens de la théorie des invariants géométriques). 

Le but du Chapitre \ref{ch:binaryquintics} est d'étudier les espaces de modules des \textit{quintiques binaires r\'eelles} stables et lisses 
$$
\ca M_s(\RR) = \GL_2(\RR) \setminus X_s(\RR) \supset \GL_2(\RR) \setminus X_0(\RR) = \ca M_0(\RR). 
$$
Si $P_s \subset \PP^1(\CC)^5$ est l'espace des cinq-uples $(x_1, \dotsc, x_5)$ dans $\PP^1(\CC)$ tels que pour $i,j,k$ deux \`a deux distincts on n'ait pas $x_i = x_j = x_k$ (c.f. \cite{MR0437531}), et $P_0 \subset P_s$ le sous-espace des cinq-uples dont toutes les coordonnées sont distinctes, alors 
$$
\ca M_0(\RR) \cong \PGL_2(\RR) \setminus (P_0/\mf S_5)(\RR) \quad \quad \text{ et } \quad \quad \ca M_s(\RR) \cong \PGL_2(\RR) \setminus (P_s/\mf S_5)(\RR). 
$$ 
%\textcolor{ $\ca M_s(\RR)$ s'identifie avec l'espace des $\PGL_2(\RR)$-classes d'\'equivalence des sous-ensembles $S \subset \PP^1(\CC)$ de cardinalité $3 \leq \va{S} \leq 5$, stables par la conjugaison complexe et munis d'un poids $f \colon S \to \set{1,2}$ tel que $\sum_{s \in S} f(s) = 5$; sous cette identification, les points dans $\ca M_0(\RR)$ correspondent \`a des sous-ensembles $S \subset \PP^1(\CC)$ de cardinalité $\va{S} = 5$. 
Pour $i = 0,1,2$, définissons $\mr M_{i}$ comme la composante connexe de $\ca M_0(\RR)$ paramétrant les $S \subset \PP^1(\CC)$ avec exactement $5 - 2i$ points réels. 

Il existe une application de p\'eriodes naturelle qui définit un isomorphisme entre $\ca M_s(\CC) = \GL_2(\CC) \setminus X_s(\CC)$ et un certain quotient de boule arithmétique $P\Gamma \setminus \CC H^2$. Sous cet isomorphisme, les quintiques strictement stables correspondent à des points dans un arrangement d'hyperplans $\mr H \subset \CC H^2$ (Th\'eor\`eme \ref{th:delignemostow} et Proposition \ref{prop:stableperiodshyperplane}). En étudiant l'équivariance de l'application des périodes sous des involutions anti-holomorphes $\CC H^2 \to \CC H^2$ ad\'equates, nous obtenons le r\'esultat suivant :

%There is a natural period map that defines an isomorphism between $\ca M_s(\CC) = \GL_2(\CC) \setminus X_s(\CC)$ and a certain arithmetic ball quotient $P\Gamma \setminus \CC H^2$. Moreover, one can prove that strictly stable quintics correspond to points in a hyperplane arrangement $\mr H \subset \CC H^2$. Investigating the equivariance of the period map with respect to a suitable set of anti-holomorphic involutions $\alpha_i : \CC H^2 \to \CC H^2$, we obtain :
\begin{theoreme} \label{th:theorem01-fr}
%For each $i \in \{0,1,2\}$, the period map induces an isomorphism of real analytic orbifolds $\mr M_i \cong P \Gamma_i \setminus \left(\RR H^2 - \mr H_i \right)$. Here $\RR H^2$ is the real hyperbolic plane, $\mr H_i$ a union of geodesic subspaces in $ \RR H^2$ and $P\Gamma_i$ an arithmetic lattice in $\textnormal{PO}(2,1)$ attached to a quadratic form over $\ZZ[\zeta_5 + \zeta_5^{-1}]$. 
Pour chaque $i \in \{0,1,2\}$, l'application des périodes induit un isomorphisme d'orbifolds analytiques réels $\mr M_i \cong P \Gamma_i \setminus \left(\RR H^2 - \mr H_i \right)$. Ici, $\RR H^2$ est le plan hyperbolique réel, $\mr H_i$ une union de sous-espaces géodésiques dans $\RR H^2$ et $P\Gamma_i$ un r\'eseau arithmétique dans $\textnormal{PO}(2,1)$ attaché à une forme quadratique sur $\ZZ[\zeta_5 + \zeta_5^{-1}]$. 
\end{theoreme}

\noindent
Le Théorème \ref{th:theorem01-fr} dote chaque composante $\mr M_i$ d'une métrique hyperbolique naturelle. Puisque l'on peut déformer le type topologique d'une partie \`a cinq \'el\'ements de $\PP^1(\CC)$, stable sous $\Gal(\CC/\RR)$, en permettant à deux points d'\^{e}tre recoll\'es, la compactification $\ca M_s(\RR) \supset \ca M_0(\RR)$ est connexe. On peut donc se demander si les métriques sur les $\mr M_i$ s'étendent en une métrique sur $\ca M_s(\RR)$ et, le cas \'echeant, à quoi ressemble l'espace résultant à la frontière. Le résultat principal du Chapitre \ref{ch:glueing} est le suivant, qui est une application des Théorèmes \ref{th:theorem03-fr} et \ref{th:theorem01-fr} ci-dessus.\begin{theoreme} \label{th:theorem02-fr}
Il existe une métrique hyperbolique complète sur $\ca M_s(\RR)$ dont la restriction \`a $\mr M_i$ est la m\'etrique induite par le Théorème \ref{th:theorem01-fr}. Muni de cette m\'etrique, $\ca M_s(\RR)$ est isométrique au triangle hyperbolique d'angles $\pi/3, \pi /5, \pi/10$ (voir Figure \ref{fig:triangle}). De l\`a, si $$P\Gamma_{3,5,10} = \langle \alpha_1, \alpha_2, \alpha_3 | \alpha_i^2 = (\alpha_1\alpha_2)^3 = (\alpha_1\alpha_3)^5 = (\alpha_2\alpha_3)^{10} = 1 \rangle, $$ il existe un homéomorphisme $$\ca M_s(\RR) \cong P\Gamma_{3,5,10} \setminus \RR H^2
\; \text{ et un plongement } \; \coprod_i P \Gamma_i \setminus \left(\RR H^2 - \mr H_i \right) \subset P\Gamma_{3,5,10}\setminus \RRH^2
$$ compatibles avec les isomorphismes $\mr M_i \cong P \Gamma_i \setminus \left(\RR H^2 - \mr H_i \right)$ du Théorème \ref{th:theorem01-fr}. 
%There exists a complete hyperbolic metric on $\ca M_s(\RR)$ that restricts to the metrics on $\mr M_i$ induced by Theorem \ref{th:theorem01-fr}. With respect to it, $\ca M_s(\RR)$ is isometric to the hyperbolic triangle of angles $\pi/3, \pi /5, \pi/10$. Thus, if $$P\Gamma_\RR =  \langle \alpha_1, \alpha_2, \alpha_3 | \alpha_i^2 = (\alpha_1\alpha_2)^3 = (\alpha_1\alpha_3)^5 = (\alpha_2\alpha_3)^{10} = 1 \rangle,$$ then there is a homeomorphism $$\ca M_s(\RR) \cong P\Gamma_\RR \setminus \RR H^2\quad \text{and an embedding} \quad \coprod_i P \Gamma_i \setminus \left(\RR H^2 - \mr H_i \right) \subset P\Gamma_\RR \setminus \RRH^2$$ that extend the isomorphisms in Theorem \ref{th:theorem01-fr}. 
\end{theoreme}
%where $P\Gamma_\RR$ is isomorphic to the group $\langle \alpha_1, \alpha_2, \alpha_3 | \alpha_i^2 = (\alpha_1\alpha_2)^3 = (\alpha_1\alpha_3)^5 = (\alpha_2\alpha_3)^{10} = e \rangle$. 
%$\overline{\mr{M}}_\RR \cong P\Gamma_\RR \setminus \RR H^2$ for a certain lattice $P\Gamma_\RR \subset \textnormal{PO}(2,1)$. 
%\begin{enumerate}
%\item There is an isomorphism of real-analytic orbifolds $$[M_{0,5}/\mathfrak S_5](\RR) \cong \coprod_{i = 0}^2 P\Gamma_i \setminus \left(\RR H^2 - \mr H_i \right).$$ Here $P\Gamma_i$ is a lattice in $PO(2,1)$ and $\mr H_i$ is a union of geodesic subspaces. 
%\item There is a lattice $P\Gamma_\RR \subset \textnormal{PO}(2,1)$ such that $\coprod_{i = 0}^2 [P\Gamma_i \setminus \left(\RR H^2 - \mr H \right)]$ is a real-analytic suborbifold of $[P\Gamma_\RR \setminus \RR H^2]$ and for which there exists a homeomorphism $$|[\bar{M}_{0,5}/\mathfrak S_5](\RR)| \cong P\Gamma_\RR \setminus \RR H^2$$ that restricts to the isomorphism in part 1. 
%\item $[P\Gamma_\RR \setminus \RR H^2]$ 
%
%With the modified orbifold structure, $[\bar{M}_{0,5}/\mathfrak S_5](\RR)$ 
%is a real-hyperbolic triangle of angles $2\pi/3, 2 \pi /5$ and $2\pi/10$.
%\end{enumerate}
%An important part of the proof of Theorem ... consists of 

%The fact that $\ca M_s(\RR) \cong H \setminus \RRH^2$ for \emph{some} lattice $H \subset \tn{PO}(2,1}$ follows from Theorems ... and ... above. To prove that $\Gamma = P\Gamma_\RR$, we calculate the orbifold structures of $\ca M_s(\RR)$ and 
%Theorem ... can be deduced from Theorem ... by applying Theorem ... above, 

\addtocontents{toc}{\protect\setcounter{tocdepth}{0}}
\section{R\'esum\'e du Chapitre \ref{ch:integralfourier}: Transformations de Fourier enti\`eres} \label{fourier-fr}
\addtocontents{toc}{\protect\setcounter{tocdepth}{1}}
\begin{flushright}
(d'apr\`es un travail en commun avec Thorsten Beckmann)
\end{flushright}
%In the second part of this thesis, we focus on algebraic cycles on complex and real abelian varieties. A central role in this study -- which takes up Chapters \ref{ch:integralfourier}, \ref{ch:onecycles} and \ref{ch:realintegralhodge} -- is played by a certain correspondence. 
\noindent
%Let $g$ be a positive integer and let $A$ be an abelian variety of dimension $g$ over a field $k$. \emph{Fourier transforms} are correspondences between the derived categories, rational Chow groups and cohomology of $A$ and $\wh A$ attached to the Poincar\'e bundle $\ca P_A$ on $A \times \wh A$ \cite{mukaiduality,beauvillefourier,huybrechtsfouriermukai}. The Fourier transform on Chow groups
%\begin{align} \label{intro-intro-Fourier-fr}
%\ca F_A \colon \CH(A)_{\QQ} \xrightarrow{\sim} \CH(\wh A)_{\QQ}
%\end{align}
%provides a powerful tool to study algebraic cycles with rational coefficients on $A$. % and the cycle class map to rational \'etale or Betti cohomology. 
%It is used to study the rational Chow ring of $A$, as well as the cycle class map to rational cohomology. 
%Recently \cite{moonenpolishchuk}, Moonen and Polishchuk initiated the study of the integrality aspects of the Fourier transform (\ref{intro-intro-Fourier-fr}). On the level of cohomology, the Fourier transform preserves integral $\ell$-adic \'etale cohomology when $k$ is separably closed, and integral Betti cohomology when $k = \CC$. This makes it natural to ask whether the Fourier transform on rational Chow groups lifts to a homomorphism between integral Chow groups. This question was raised by Moonen--Polishchuk in \emph{loc.cit} and by Totaro in \cite{totaroIHCthreefolds}. Consider the following result.
Soit $g$ un entier positif et $A$ une variété abélienne de dimension $g$ sur un corps $k$. \emph{Les transformations de Fourier} sont des correspondances attachées au faisceau de Poincar\'e $\ca P_A$ sur $A \times \wh A$ entre les catégories dérivées, les groupes de Chow rationnels et la cohomologie de $A$ et de $\wh A$ \cite{mukaiduality,beauvillefourier,huybrechtsfouriermukai}. La transformation de Fourier sur les groupes de Chow rationnels
\begin{align} \label{intro-intro-Fourier-fr}
\ca F_A \colon \CH(A)_{\QQ} \xrightarrow{\sim} \CH(\wh A)_{\QQ}
\end{align}
fournit un outil puissant d'\'etude des cycles algébriques sur $A$. Au niveau de la cohomologie, la transformation de Fourier préserve la cohomologie \'etale $\ell$-adique enti\`ere lorsque $k$ est séparablement clos, et la cohomologie de Betti enti\`ere lorsque $k = \CC$. Il est donc naturel de se demander si la transformation de Fourier (\ref{intro-intro-Fourier-fr}) se rel\`eve en un homomorphisme entre groupes de Chow entiers. %Cette question a \'et\'e soulevée par Moonen-Polishchuk dans \cite{moonenpolishchuk} et par Totaro dans [Tot21], sera étudiée au chapitre 6. 
Cette question a \'et\'e soulevée par Moonen--Polishchuk dans \cite{moonenpolishchuk} et par Totaro dans \cite{totaroIHCthreefolds}, et sera explorée dans le Chapitre \ref{ch:integralfourier}. Considérons le r\'esultat suivant :
%le th\'eor\`eme suivant :
%Cette question a été soulevée par Moonen--Polishchuk dans \cite{moonenpolishchuk} et par Totaro 
%dans \cite{totaroIHCthreefolds}. % \textcolor{et continue d'être explorée dans le Chapitre \ref{ch:integralfourier}. } 
\begin{theoreme-non}[Moonen--Polishchuk]
\begin{enumerate}
\item Soit $k$ un corps et $C$ une courbe hyperelliptique sur $k$ telle que l'un des points de Weierstrass de $C$ est défini sur $k$. Pour la jacobienne $A = J(C)$ de $C$, la transformation de Fourier (\ref{intro-intro-Fourier-fr}) se rel\`eve en un homomorphisme motivique $\ca F \colon \CH(A) \to \CH(\wh A)$. 
\item %Let $k$ be the function field and $(C, p_0)$ the generic point of the moduli stack $\ca M_{g,1}$. Let $A = J(C)$ be the Jacobian of $C$. The Fourier transform (\ref{intro-intro-Fourier-fr}) does \emph{not} lift to a homomorphism $\ca F \colon \CH(A) \to \CH(\wh A)$. 
%The image of $\CH(A)/(\tn{torsion})$ under the Fourier transform (\ref{intro-intro-Fourier-fr}) is not contained in $\CH(\wh A)/(\tn{torsion})$. 
Soit $g \in \ZZ_{\geq 2}$ et consid\'erons le champ de modules $\ca M_{g,1}$. Soit $K$ son corps de fonctions et $(C, p_0)$ sont point générique. Pour la jacobienne $A = J(C)$ de $C$, la transformation de Fourier (\ref{intro-intro-Fourier-fr}) se ne se rel\`eve pas en un homomorphisme $\ca F \colon \CH(A) \to \CH(\wh A)$. 
%Let $k$ be any field, let $C$ be a hyperelliptic curve over $k$ such that one of the Weierstrass points of $C$ is defined over $k$. Let $A = J(C)$ be the Jacobian of $C$. The Fourier transform (\ref{intro-intro-Fourier-fr}) lifts to a motivic homomorphism $\ca F \colon \CH(A) \to \CH(\wh A)$. 
%$(C, p)$ be a hyperelliptic curve over $k$ with a rational point $p \in C(k)$. 
\end{enumerate}
\end{theoreme-non}
%More precisely, Moonen and Polishchuk gave a counterexample for abelian varieties over non-closed fields and asked about the case of algebraically closed fields. 
%Indeed, it is natural to ask: How does $\ca F_A$ interact with integral algebraic cycles?
%Chow ring $\CH(A)$? 
%The goal of the current Chapter \ref{ch:integralfourier} is to work on Question \ref{chow-interaction}, 
%Building on the techniques of Moonen--Polishchuk, 
\noindent
Nous \'etablissons le r\'esultat suivant : 

\begin{theoreme}
Soit $k$ un corps, $A$ une vari\'et\'e ab\'elienne sur $k$ de dimension $g \geq 1$, et $\Theta \in \CH^1(A)$ la classe d'un diviseur induisant une polarisation principale sur $A$. 
%Let $k$ be a field, let $A$ be an abelian variety over $k$, and let $\Theta \in \CH^1(A)$ be the class of a divisor that induces a principal polarization on $A$. 
\begin{enumerate}
\item \label{rat-fr} %Soit $k$ un corps, $A$ une variété abélienne sur $k$, et $\Theta \in \CH^1(A)$ la classe d'un diviseur qui induit une polarisation principale sur $A$. 
Si le cycle rationnel $\Theta^{g-1}/(g-1) ! \in \CH_1(A)_\QQ$ se rel\`eve en un cycle entier $\Gamma \in \CH_1(A)$, alors la transformation de Fourier (\ref{intro-intro-Fourier-fr}) se rel\`eve en un homomorphisme motivique $\ca F \colon \CH(A) \to \CH(\wh A)$. 
%If the rational cycle $\Theta^{g-1}/(g-1)! \in \CH_1(A)_\QQ$ lifts to an cycle in $\CH_1(A)$, then the Fourier transform (\ref{intro-intro-Fourier-fr}) lifts to a motivic homomorphism $\ca F \colon \CH(A) \to \CH(\wh A)$. 
\item Si $(-1)^\ast\Theta = \Theta \in \CH^1(A)$, l'inverse de l'assertion \ref{rat-fr} est également vrai. 
%If $(-1)^\ast\Theta = \Theta \in \CH^1(A)$, then the converse of statement \ref{rat-fr} is true as well. 
\end{enumerate} 
\end{theoreme}

\addtocontents{toc}{\protect\setcounter{tocdepth}{0}}
\section{R\'esum\'e du Chapitre \ref{ch:onecycles}: Cycles alg\'ebriques sur les jacobiennes}\label{intro:sub:theintegralhodgeconjectureforjacobianvarieties-fr}
\addtocontents{toc}{\protect\setcounter{tocdepth}{1}}
\begin{flushright}
(d'apr\`es un travail en commun avec Thorsten Beckmann)
\end{flushright}

%In our next result, we somewhat 
%change directions, and consider curves on complex abelian varieties. %Our third result concerns curves on complex abelian varieties. 
\noindent
Savoir si la transformation de Fourier (\ref{intro-intro-Fourier-fr}) se rel\`eve aux groupes de Chow entiers a des conséquences importantes pour l'image de l'application classe de cycle. Rappelons que, bien que la conjecture de Hodge enti\`ere échoue en général \cite{atiyahintegralhodge, trento,totarocobordism}, elle reste une question ouverte pour les variétés abéliennes. En nous appuyant sur les résultats du Chapitre \ref{ch:integralfourier}, nous prouvons ce qui suit au Chapitre \ref{ch:onecycles} : 
%Le résultat principal du Chapitre \ref{ch:onecycles} est le suivant :
%that there are many cases for which the answer to this question is affirmative:
%The first non-trivial case to look at should be the integral Hodge conjecture for one-cycles, since its analogue with rational coefficients (the Hodge conjecture for one-cycles) holds by Lefschez (1,1) and hard Lefschetz. 

\begin{theoreme} \label{cycle-maintheorem-fr}
%Let $\wh A$ be the dual abelian variety and $\ca P_A$ the Poincar\'e line bundle. 
Soit $A$ une variété abélienne complexe de dimension $g \geq 1$ avec fibr\'e de Poincar\'e $\ca P_A$. Les trois assertions suivantes sont équivalentes :
\begin{enumerate}
    \item \label{introitem:minimalpoincare-fr} La classe de cohomologie $c_1(\ca P_A)^{2g-1}/(2g-1) ! \in \rm H^{4g-2}(A \times \wh A, \ZZ)$ est algébrique. 
    \item \label{introitem:integralpoincare-fr} Le caractère de Chern $\ch(\ca P_A) = \exp(c_1(\ca P_A)) \in \rm H^\bullet(A \times \wh A, \ZZ)$ est algébrique.
    \item \label{introitem:integralhodgeforproduct-fr} La conjecture de Hodge enti\`ere pour les courbes est valable pour $A \times \wh A$. 
\end{enumerate}
L'une quelconque de ces assertions implique la conjecture enti\`ere de Hodge pour les courbes sur $A$. De plus, si $A$ est principalement polarisée par $\theta \in \Hdg^2(A, \ZZ)$, les énoncés \ref{introitem:minimalpoincare-fr} - \ref{introitem:integralhodgeforproduct-fr} sont équivalents à l'algébricité de la classe $\theta^{g-1}/(g-1) !  \in \rm H^{2g-2}(A, \ZZ)$.
\end{theoreme}
%Remark that Condition~\ref{introitem:minimalclass} is stable under products, so a product of principally polarized abelian varieties satisfies the integral Hodge conjecture for one-cycles if and only if each of the factors does. More importantly, if $J(C)$ is the Jacobian of a smooth projective curve $C$ of genus $g$, then %$\gamma_\theta \in \rm H^{2g-2}(J(C), \ZZ)$ is the class of the image of an Abel-Jacobi map $C \to J(C)$, hence 
%every integral Hodge class of degree $2g-2$ on $J(C)$ is a $\ZZ$-linear combination of curves classes:
%As a corollary, we obtain:
\noindent
Nous obtenons comme corollaire le r\'esultat suivant :
%Comme corollaire, on obtient :
\begin{theoreme} \label{introth:IHCforjacobians-fr}
Soient $C_1, \dotsc ,C_n$ des courbes projectives lisses sur $\CC$. La conjecture de Hodge enti\`ere pour les courbes vaut pour le produit des jacobiennes $J(C_1) \times \cdots \times J(C_n)$. %Let $C$ be a smooth projective curve over $\CC$, and let $J(C)$ be the Jacobian of $C$. Then $J(C)$ satisfies the integral Hodge conjecture for one-cycles. 
\end{theoreme}

%\noindent
%Another consequence of Theorem \ref{cycle-maintheorem-fr} is the following:
%that the integral Hodge conjecture for one-cycles on principally polarized complex abelian varieties is stable under specialization. 
\noindent
De plus, \`a l'aide du Th\'eor\`eme \ref{cycle-maintheorem-fr}, nous d\'emontrons : 

\begin{theoreme} \label{introth:density-fr}
Soit $g \in \ZZ_{\geq 2}$. Il existe une union dénombrable $X \subset \ca A_{g}(\CC)$ de sous-variétés algébriques de dimension au moins $3g-3$, satisfaisant les conditions suivantes: $X$ est dense pour la topologie euclidienne dans $\ca A_g(\CC)$ et la conjecture de Hodge enti\`ere pour les courbes vaut pour les variétés abéliennes définissant un point dans $X$. 
\end{theoreme}

\noindent
La généralité de la théorie des transformations de Fourier enti\`eres que nous avons développée (voir la Section \ref{fourier-fr}) permet d'obtenir des résultats similaires en caractéristique non nulle. \'Etant donn\'ee $k$, la clôture séparable d'un corps finiment engendré, une variété projective lisse $X$ de dimension $n$ sur $k$ satisfait la \emph{conjecture de Tate enti\`ere pour les courbes sur $k$} si pour tout nombre premier $\ell \neq \tn{char}(k)$ et tout corps de définition finiment engendré $k_0 \subset k$ de $X$, l'application classe de cycle 
\begin{equation} \label{eq:integraltateconjecture-f}
    cl \colon \CH_1(X)_{\ZZ_\ell} = \CH_1(X)\otimes_\ZZ{\ZZ_\ell} \to \bigcup_U \rm H_{\textnormal{\'et}}^{2n-2}(X, \ZZ_\ell(n-1))^U
\end{equation}
est surjective, où $U$ parcourt les sous-groupes ouverts de $\Gal(k/k_0)$. 

\begin{theoreme} \label{introth:integraltate-fr} Soit $(A, \lambda)$ une variété abélienne principalement polarisée sur $k$, de dimension $g$ et de classe de polarisation $\theta_\ell \in \rm H^2_{\textnormal{\'et}}(A, \ZZ_\ell(1))$. L'application classe de cycle (\ref{eq:integraltateconjecture-f}) est surjective si $ \theta_\ell^{g-1}/(g-1) ! \in \rm H_{\textnormal{\'et}}^{2g-2}(A, \ZZ_\ell(g-1))$ est dans son image. 
\end{theoreme}
\noindent
Observer que la condition dans le Th\'eor\`eme \ref{introth:integraltate-fr} est toujours vérifiée si $\ell > (g-1)!$. Nous obtenons comme corollaire le r\'esultat suivant :
\begin{theoreme}
Soient $C_1, \dotsc ,C_n$ des courbes projectives lisses sur $k$. La conjecture de Tate enti\`ere pour les courbes sur $k$ vaut pour le produit des jacobiennes $J(C_1) \times \cdots \times J(C_n)$. 
%Le théorème \ref{introth:integraltate-fr} implique que 
%\textcolor{Tout produit des jacobiennes des courbes projectives lisses sur $k$} satisfait la conjecture de Tate enti\`ere pour les courbes sur $k$. 
\end{theoreme}
\noindent
Soit $k$ la clôture algébrique d'un corps finiment engendré de caractéristique non nulle, et soit $X$ une vari\'et\'e projective lisse sur $k$. La \emph{conjecture enti\`ere de Tate pour les courbes sur $k$} est l'analogue de la propriété introduite ci-dessus. 
\begin{theoreme} \label{th:chinglichai-fr}
Soit $\msf A_{g}$ l'espace de modules grossier des variétés abéliennes principalement polarisées de dimension $g$ sur $k$. Les points de $\msf A_{g}(k)$ correspondant \`a des variétés abéliennes satisfaisant la conjecture de Tate enti\`ere pour les courbes sur $k$ sont Zariski denses dans $\msf A_{g}$. 
\end{theoreme}

\addtocontents{toc}{\protect\setcounter{tocdepth}{0}}
\section{R\'esum\'e du Chapitre \ref{ch:realintegralhodge}: Courbes sur les solides ab\'eliens r\'eels}\label{intro:sub:therealintegralhodgeconjectureforabelianthreefolds-fr}
\addtocontents{toc}{\protect\setcounter{tocdepth}{1}}

Posons $G = \Gal(\CC/\RR)$. Pour toute variété projective lisse $X$ sur $\RR$, Benoist et Wittenberg \cite{BW20} ont défini un sous-groupe $\Hdg^{2k}_G(X(\CC), \ZZ(k))_0$ du groupe de cohomologie équivariant $\rm H^{2k}_G(X(\CC), \ZZ(k))$ et ont étudié l'application classe de cycle %\cite{krasnovcharact}, \cite[\S 1.6.1]{BW20}
\begin{align}\label{realcycleclassmap-fr}
\CH_i(X) \to \rm \Hdg^{2k}_G(X(\CC), \ZZ(k))_0 \quad \quad \left(i + k = \dim(X)\right).
\end{align}
La \emph{conjecture de Hodge enti\`ere réelle pour les $i$-cycles} questionne la surjectivit\'e de (\ref{realcycleclassmap-fr}). Comme dans la situation complexe, la surjectivit\'e de (\ref{realcycleclassmap-fr}) est valable pour toute variété projective lisse $X$ sur $\RR$ si $i \in \{\dim(X), \dim(X)-1, 0\}$ \cite{krasnovcharact, mangoltehamel, vanhamel, BW20}, mais peut échouer pour d'autres valeurs de $i \in \set{0, \dotsc, \dim(X)}$. 

En utilisant la théorie du Chapitre \ref{ch:integralfourier}, nous fournissons dans le Chapitre \ref{ch:realintegralhodge} des résultats soutenant la conjecture de Hodge enti\`ere réelle pour les solides abéliens. 

%Le premier résultat principal du Chapitre \ref{ch:realintegralhodge} est le suivant : 
\begin{theoreme} \label{theorem1-fr}
Tout solide abélien $A$ sur $\RR$ satisfait la conjecture de Hodge entière réelle modulo la torsion, au sens o\`u l’homomorphisme suivant est surjectif :
$$
\CH_1(A) \to \Hdg^{4}_G(A(\CC), \ZZ(2))_0/(\textnormal{torsion})= \Hdg^{4}(A(\CC), \ZZ(2))^G.
$$
\end{theoreme}

\noindent
Nous obtenons comme corollaire le r\'esultat suivant :

\begin{theoreme} \label{IHCforconnected-fr}
Soit $A$ un solide abélien sur $\RR$ tel que $A(\RR)$ est connexe. Alors $A$ satisfait la conjecture de Hodge enti\`ere réelle. 
\end{theoreme}

\noindent
En ce qui concerne la conjecture de Hodge enti\`ere réelle pour les solides ab\'eliens principalement polarisés, nous d\'emontrons le théorème de réduction suivant :

\begin{theoreme} \label{reduction-fr}
Soit $\ca A_3(\RR)^+$ une composante connexe de $\ca A_3(\RR)$, l'espace de modules des solides abéliens réels principalement polarisés. Supposons que la conjecture de Hodge enti\`ere réelle vaut pour toute jacobienne $J(C)$ tel que $[J(C)] \in \ca A_3(\RR)^+$ et $C(\RR) \neq \emptyset$. Alors la conjecture de Hodge enti\`ere réelle vaut pour toute variété abélienne dans $\ca A_3(\RR)^+$. 
\end{theoreme}
\noindent
%Envoyer un solide ab\'elien principalement polaris\'ee $A$ vers la classe de diff\'eomorphisme \'equivariant de son lieu r\'eel $A(\CC)$ d\'efinit une bijection entre $\pi_0\left(\ca A_3(\RR)\right)$ et 
%\textcolor{Le \emph{type topologique} d'un solide ab\'elien r\'eel est la classe d'isomorphisme de la vari\'et\'e diff\'erentiable $G$-\'equivariante de son lieu complexe $A(\CC)$. % r\'ealis\'ees par le lieu complexe d'un solide ab\'elien principalement polaris\'e. %Les solides abéliens réels $A$ sont de quatre types différents, correspondant aux classes de difféomorphisme équivariant des lieux complexes $A(\CC)$. 
%£Pour trois des quatres types topologiques des 
L'espace de modules $\ca A_3(\RR)$ a quatre composantes connexes, correspondant aux types topologiques des lieux r\'eels des solides ab\'eliens r\'eels principalement polaris\'es. Dans trois de ces quatre composantes, il n'y a pas d'obstruction topologique à la conjecture de Hodge enti\`ere réelle, comme suit de notre dernier r\'esultat : 
\begin{theoreme} \label{questionreductioncorollary-fr} Soit $B$ une surface abélienne réelle, et $E$ une courbe elliptique avec $E(\RR)$ connexe. Le solide abélien réel $A = B \times E$ satisfait la conjecture de Hodge enti\`ere réelle. 
\end{theoreme}

\cleardoubleevenpage
\cleardoublepage
% Probably you won't need to edit this file. Instead, edit the bib file itself.
\bibliographystyle{amsalpha}

\bibliography{references}
% The following command will make sure that the colophon is on an even page (the ``back'' of your thesis)
\cleardoubleevenpage
\thispagestyle{plain}
% Uncomment the addmargin environment when using Option 2 in the main file
% \begin{addmargin}[-0.6cm]{-3cm}

%\hfill

%\vfill

%\pdfbookmark[0]{Colophon}{colophon}
%\section*{Colophon}
%\noindent This thesis was typeset using the typographical look-and-feel\\
%\texttt{classicthesis} developed by Andr\'e Miede and Ivo Pletikosić.\bigskip
%\noindent The style was inspired by Robert Bringhurst's seminal book\\
%on typography ``\emph{The Elements of Typographic Style}''.\bigskip

%\noindent Here you can insert things like ``Figures were created with...''\bigskip

%\noindent [ Insert version number/description, if you want ]
% \end{addmargin}

\includepdf[pages={4}]{couverture-anglais.pdf}

%\includepdf[pages={5,6}]{myfile.pdf}

\end{document}